\documentclass[9pt,oneside]{extbook}
\usepackage[utf8]{inputenc}
\DeclareUnicodeCharacter{1EF3}{\`y}
\usepackage{CGFW-survey}

\title{Conditional Gradient Methods}

\hypersetup{
  pdftitle={Conditional Gradient Methods},
  pdfauthor={Gábor Braun, Alejandro Carderera, Cyrille W. Combettes,
    Hamed Hassani, Amin Karbasi, Aryan Mokhtari, Sebastian Pokutta},
  pdfkeywords={Frank–Wolfe algorithm, conditional gradient,
    gradient descent, linear optimization oracle, convex optimization},
  pdfsubject={MSC2000 Primary 68Q32; 
    Secondary
    90C52 
  }}

\makeindex

\begin{document}
\pagenumbering{alph}
\thispagestyle{empty}
\renewcommand{\thepage}{}
\phantomsection
\pdfbookmark[1]{Cover}{cover}
\pagecolor[rgb]{.349,.627,.49}
\begin{picture}(0,0)
  \put(-22mm,-180mm){\includegraphics[width=156mm, alt=]{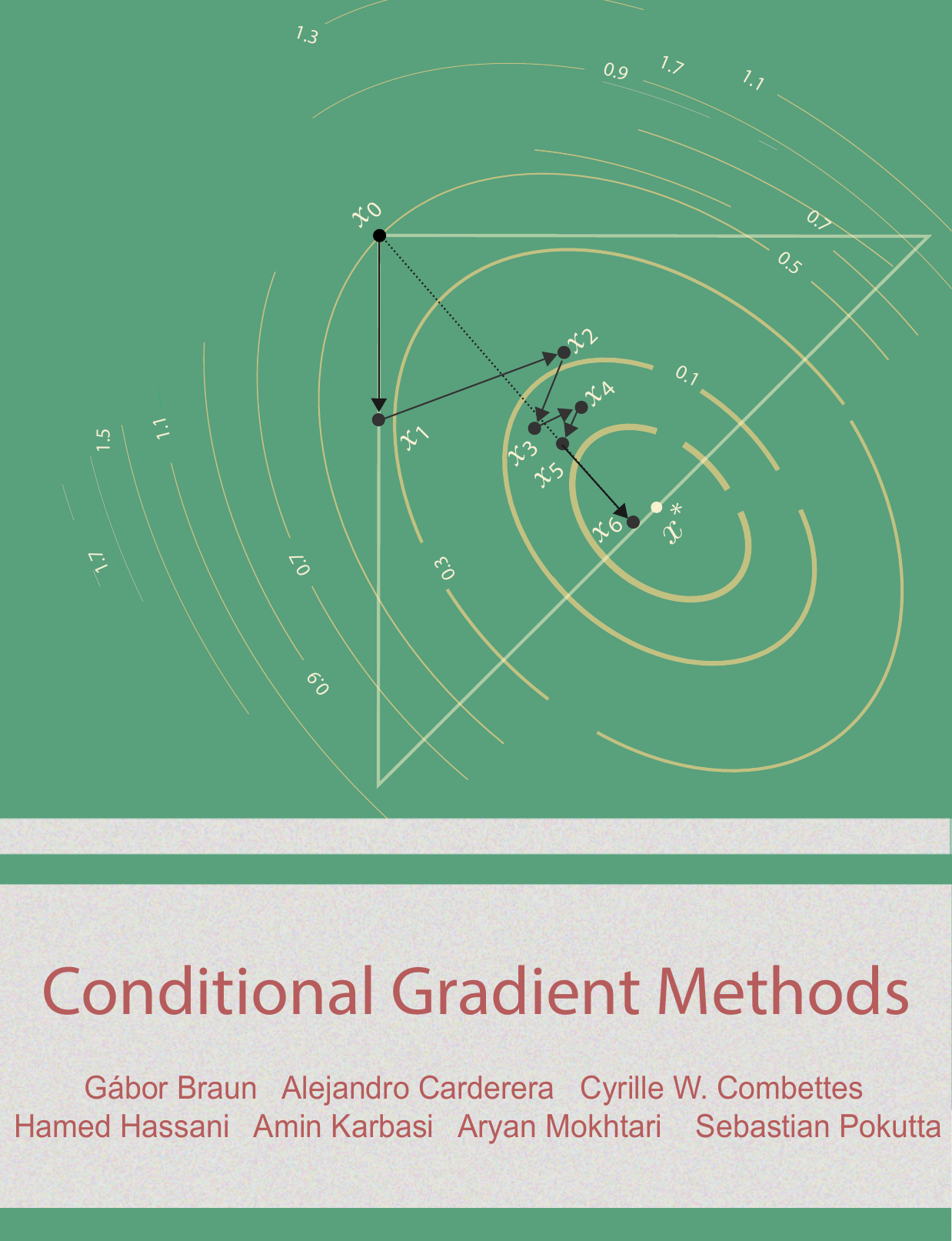}}
\end{picture}

\newpage
\thispagestyle{empty}

\begin{center}
\fontfamily{phv}\selectfont

\begin{minipage}{.8\textwidth}

Conditional gradient algorithms have become an essential part of the
algorithmic toolbox in machine learning, signal processing, and related fields.
This monograph offers a comprehensive review of both classical results and
recent generalizations, including extensions to large-scale settings. The
presentation is notably clear, featuring illustrations, detailed proofs, and
application examples. It will serve as an important reference for graduate
students and researchers in data science.

\rightline{\emph{Francis Bach}}

\vspace{3\baselineskip}

\emph{Conditional Gradient Methods} is a thorough and accessible guide to one
of the most versatile families of optimization algorithms. The book traces the
rich history of the conditional gradient algorithm and explores its modern
advancements, offering a valuable resource for both experts and newcomers. With
clear explanations of the algorithms, their analysis, and practical
applications, the authors provide a go-to reference for anyone tackling
constrained optimization problems. This book is sure to inspire fresh ideas and
drive advancements in the field.

\rightline{\emph{Elad Hazan}}
\end{minipage}
\end{center}

\newpage

\pagecolor{white}

\frontmatter

\begin{titlepage}
\phantomsection
\pdfbookmark[1]{Title}{Title sheet}
\centering  

\vspace*{\stretch{2}}

\makeatletter
{\fontsize{22pt}{30pt}\selectfont
\bfseries
\@title
} 
\makeatother

\vspace{\stretch{1}}

\large
\mbox{Gábor Braun}\textsuperscript{a,\,b}
\qquad
\mbox{Alejandro Carderera}\textsuperscript{c}
\qquad
\mbox{Cyrille W. Combettes}\textsuperscript{d}
\\
\mbox{Hamed Hassani}\textsuperscript{e}
\qquad
\mbox{Amin Karbasi}\textsuperscript{f}
\qquad
\mbox{Aryan Mokhtari}\textsuperscript{g}
\\
\mbox{Sebastian Pokutta}\textsuperscript{a,\,b}

\vspace{\stretch{1}}

\begin{enumerate}[label=\textsuperscript{\alph*},labelsep=1pt,first=\centering\footnotesize]
\item Department for AI in Society, Science, and Technology\\ Zuse Institute Berlin, Germany\\
\texttt{\href{mailto:braun@zib.de}{\{braun},\href{mailto:pokutta@zib.de}{pokutta\}@zib.de}}

\item {Institute of Mathematics\\ Technische Universität Berlin, Germany}

\item {School of Industrial and Systems Engineering\\ Georgia Institute of Technology, USA\\
\texttt{\href{mailto:alejandro.carderera@gatech.edu}{alejandro.carderera@gatech.edu}}}

\item {Toulouse School of Economics\\ Université Toulouse 1 Capitole, France\\
\texttt{\href{mailto:cyrille.combettes@tse-fr.eu}{cyrille.combettes@tse-fr.eu}}}

\item {Department of Electrical and Systems Engineering\\ University of Pennsylvania, USA\\
\texttt{\href{mailto:hassani@seas.upenn.edu}{hassani@seas.upenn.edu}}}

\item {Department of Electrical Engineering, Computer Science, Statistics \& Data Science\\ Yale University, USA\\
  \texttt{\href{mailto:amin.karbasi@yale.edu}{amin.karbasi@yale.edu}}}

\item {Department of Electrical and Computer Engineering\\ University of Texas at Austin, USA\\
\texttt{\href{mailto:mokhtari@austin.utexas.edu}{mokhtari@austin.utexas.edu}}}
\end{enumerate}

\vspace{\stretch{1}}

    Please send comments/suggestions to \texttt{frank.wolfe.book@gmail.com}.

\vspace{\stretch{2}}

\makeatletter
\small
\@date
\makeatother

\end{titlepage}

\chapter*{Preface}
\addcontentsline{toc}{chapter}{Preface}

\emph{Optimization} is a cornerstone of scientific and engineering disciplines, where the ability to improve system performance and make efficient decisions is essential. With the rapid advancements in machine learning and artificial intelligence, optimization has become even more critical. Modern AI systems—ranging from deep learning architectures to reinforcement learning agents—rely on sophisticated optimization techniques to train models, fine-tune hyperparameters, and ensure robustness in decision-making. 

\emph{Constrained optimization} plays a fundamental role in many AI applications, including structured prediction, resource allocation, and fairness in machine learning. As models grow in size and complexity and are increasingly deployed in real-world settings, the need to optimize under constraints—whether due to data limitations, interpretability requirements, or ethical considerations—has become more pressing. A prominent example is the challenge of training large language models (LLMs), which power cutting-edge natural language processing systems. Optimizing billions of parameters within the computational, memory, and energy constraints requires sophisticated techniques such as low-rank adaptations, structured sparsity, and constrained fine-tuning. 

Beyond AI, constrained optimization is essential in numerous fields, including signal processing, computational biology, finance, and operations research. It underpins key tasks such as sparse recovery in compressed sensing, low-rank approximations in recommender systems, optimal transport in economics, and portfolio optimization in finance.

This book provides a detailed exploration of constrained optimization, with a primary focus on \emph{Frank-Wolfe methods and Conditional Gradients}—a family of first-order algorithms known for their efficiency, scalability, and ability to handle structured constraints. These methods, which leverage gradient information to navigate complex optimization landscapes, try to strike a balance between computational cost and convergence speed. Our goal in this book is twofold: to offer a rigorous yet accessible foundation for understanding these methods and to equip readers with the practical insights necessary to apply them effectively. By integrating both theoretical developments and real-world applications, we aim to create a resource that is valuable to researchers advancing the field of optimization as well as to practitioners deploying these algorithms in large-scale computational systems.

\section*{A guide to the reader}

The purpose of this survey is to serve both as a gentle introduction
and a coherent overview of state-of-the-art algorithms. More
specifically, we will cover \emph{what} is this class of optimization
methods, \textit{why} you should care about them, and \emph{how} you
can run an appropriate algorithm, from a plethora of variants, for
your specific
optimization problem.
The selection of the material has been guided by the principle of
highlighting crucial ideas as well as presenting new approaches that
we believe might become important in the future,
with ample citations even of old works imperative in the development
of newer methods.
Yet, our selection is sometimes biased, and need not reflect
consensus of the research community,
and we have certainly missed recent important contributions.
After all the research area of Frank–Wolfe algorithms is very active,
making it a
moving target.  We apologize sincerely in advance for any such
distortions and we fully acknowledge: We stand on the shoulder of
giants.

This survey contains three main parts. Chapter~\ref{cha:FW-basics}
addresses the basics of the Frank--Wolfe algorithm and provides an
overview of the basic methodology, lower bounds, and some variants
that are to be considered standard by now. In
Chapter~\ref{cha:FW-improved} we consider more recent improvements of
conditional gradients, mostly better convergence rates and (more
generally) faster algorithms. In Chapter~\ref{cha:FW-large} we address
the large-scale settings, where we cover the stochastic case,
decentralized optimization, and online learning. Finally,
Chapter~\ref{cha:misc} contains a few related methods that are
substantially different from core variants of Frank–Wolfe algorithms
and a number of applications.

For readers interested in understanding
the basics, we  highly recommend
Chapter~\ref{cha:FW-basics} as it contains much of the core algorithmic ideas
and is a prerequisite for the later ones.  Chapters~\ref{cha:FW-improved}
and~\ref{cha:FW-large} are quite independent of each other and contain more
advanced topics, including very recent results.  We have also tried to include
computational comparisons wherever reasonable and provide some practical advice
on how to run different algorithms. 

A valuable tool for implementing Frank–Wolfe methods is the FrankWolfe.jl Julia package, which is distributed under the MIT license and can be found at \url{https://github.com/ZIB-IOL/FrankWolfe.jl}. Although not utilized in our work for legacy reasons, this package offers high-performance, state-of-the-art implementations of various Frank–Wolfe algorithms. It serves not only as an excellent benchmark and baseline but also as a robust framework for developing new Frank–Wolfe variants, similar to frameworks like TensorFlow, PyTorch, or scikit-learn within the machine learning community. The code used to generate the figures in the book as well as additional information can be found at the companion website at \url{https://conditional-gradients.org/}.

\section*{Acknowledgments}

We thank David Martínez-Rubio for helpful discussion on oracle complexity
bounds. We thank Shpresim Sadiku for generating the adversarial example in
Figure~\ref{fig:adv}. We thank Simon Lacoste-Julien for helpful discussions and we would also like to thank Francis Bach, Jelena Diakonikolas, Dan Garber, and Marc Pfetsch for providing comments on an early version. 
We thank Zev Woodstock for suggesting improvements to the text.
We thank Ð.Khuê Lê-Huu for drawing our attention to an early
appearance of Generalized Conditional Gradients
(Algorithm~\ref{gencgmirror}).

\tableofcontents
\listofalgorithms
\chapter*{Glossary}
\markboth{Glossary}{Glossary}
\addcontentsline{toc}{chapter}{Glossary}
\label{sec:glossary}

\begin{small}

  \begin{xltabular}{\linewidth}{@{}lX@{}}
\(\defeq\) &  defining equality
\\
\(\mathcal{X}\) &  compact convex set,
  the feasible region for optimization problems
\\
\(\interior( \mathcal{X})\) &  interior of the convex set \(\mathcal{X}\)
\\
\(\relint{ \mathcal{X}}\) &  relative interior of the convex set
  \(\mathcal{X}\)
\\
\(B^{n}\), \(S^{n}\) &  \(n\)-dimensional
  closed Euclidean ball and sphere
  Note: the boundary of  \(B^{n}\) is \(S^{n-1}\)
  of dimension one less
\\
\(B(x, r)\) & closed ball of radius \(r\) around \(x\) in the
  norm \(\norm{\cdot}\)
\\
\(P\) &  polytope (as feasible region)
\\
\(\vertex{P}\) &  set of vertices of \(P\)
\\
\(\face_{P} (x)\) &  minimal face of polytope \(P\) containing
  point \(x\)
\\
\(\conv{S}\) &  convex hull of \(S\)
\\
\(I^n\) &  Identity matrix of dimension $n$
\\
\(f\) &  objective function
\\
\(\nabla f(x)\) &  gradient of \(f\) at \(x\)
  (standard notation)
\\
\(\partial f (x)\) &  set of subgradients of \(f\) at \(x\)
  (rarely used in this survey)
\\
\(A\) &  linear map
\\
\(A^{\top}\) &  adjoint of the linear map \(A\)
  (the common notation is \(A^{*}\) but we use it for
  convex conjugate)
  \(\innp{A^{\top} x}{y} = \innp{x}{A y}\)
\\
\(\widetilde{\nabla} f(x)\) & 
  gradient estimator of \(f\) at \(x\)
\\
\(f^{*}\) &  convex conjugate of \(f\)
\\
\(x^{*}\) &  a (not necessarily unique)
  minimum of \(f\) on \(\mathcal{X}\)
\\
\(\Omega^{*}\) & set of optimal solutions, hardly used, likely
  unnecessary
\\
\(\mu\) &  strong convexity parameter (of \(f\))
\\
\(L\) &  smoothness parameter (of \(f\))
\\
\(G\) &  gradient bound, \(\dualnorm{\nabla f(x)} \leq G\)
\\
\(M\) &  function value bound, \(\abs{f(x)} \leq M\)
\\
\(t\) &  iteration index
\\
\(h(x)\) &  primal gap \(f(x) - f(x^{*})\)
\\
\(\innp{\nabla f(x)}{x - x^{*}}\) &  dual gap
\\
\(g(x)\) &  Frank–Wolfe gap
  \(\max_{v \in \mathcal X} \innp{\nabla f(x)}{x - v}\)
\\
\(x_{t}\) &  solution sequence,
  usually produced by an algorithm
\\
\(h_{t} = h(x_{t})\) &  primal gap at \(x_{t}\)
\\
\(g_{t} = g(x_{t})\) &  Frank–Wolfe gap at \(x_{t}\)
\\
\(\mathcal{S}_{t}\) &  active set of vertices
  after iteration \(t\)
\\
intervals &  \([0,1)\), \([1,2]\), \((1,T)\) with round
  parentheses denoting open end (not containing the end point)
\\
\(\mathbb{Z}\) & set of integers
\\
\(e_{i}\) & coordinate vector \((0, \dotsc, 0, 1, 0, \dotsc, 0)\)
  with coordinate \(i\) being \(1\) and all other coordinates being \(0\)
\\
\(\Delta_{n}\) & probability simplex, convex hull of
  coordinate vectors \(e_{1}\), \dots, \(e_{n}\)
\\
\(\ln\) &  natural logarithm, i.e., logarithm with base \(e\)
\\
\(\log\) &  logarithm with arbitrary base, but the base should be
  the same inside a formula
\\
\(\size{A}\) & size (number of elements) of set \(A\)
\\
\(\norm{\cdot}\) &  any fixed norm depending on the vector space;
  operator norm for linear and multilinear operators
  (the dual norm has a different notation)
\\
\(\dualnorm{\cdot}\) &  dual norm
\\
\(\distance{X}{Y}\) &  distance of sets \(X\) and \(Y\)
  in norm \(\norm{\cdot}\),
  i.e.,
  \(\min_{x \in X,\ y \in Y} \norm{x-y}\)
\\
\(\mathcal{O}(t), \Omega(t), \Theta(t)\) & 
  instances of big O notation
\end{xltabular}

\end{small}

\mainmatter
\chapter{Introduction}
\label{cha:introduction}

You can't improve what you don't measure.\footnote{Often attributed to Peter
Drucker.} Indeed, many scientific and engineering methods rely on the simple
idea of measuring a phenomenon, such as the capacity of a communication
network, the position of a robotic arm, etc, and then making small improvements
towards a pre-defined goal, such as decreasing the communication error or
changing the position of the robotic arm to grab an object. More often than
not, in order to update  such systems reliably we need to be mindful of various
constraints along the path of improvements, for example, the energy consumption
of the communication system should never be above a threshold, or  to meet the
safety measures for the robotic arm some configurations cannot be reached.
Maximizing the utility (or minimizing the cost) subject to feasibility
constraints is the essence of constrained optimization problems. 

With the emergence of machine learning and artificial intelligence
applications, constrained optimization  has become an integral part of the
entire training and inference pipeline. Examples include signal recovery with a
sparsity constraint \citep{donoho2006compressed, candes2006near}, matrix
completion with a rank constraint \citep{candes2009exact}, exploration with
safety constraints \citep{sui2015safe}, adversarial attacks with noise
constraints \citep{goodfellow15}, experimental design with a budget
constraint \citep{pukelsheim2006optimal}, optimal transport with a probability
polytope constraint \citep{villani2009optimal}, and submodular optimization
with a matroid constraint \citep{calinescu2011maximizing}, to just name a
few.\footnote{Of course, the number of machine learning examples is ever
increasing as witnessed by the rapid growth of papers put on arXiv.}

\begin{figure}[ht]
  \centering
  \footnotesize
  \begin{minipage}{.5\linewidth}
    \raisebox{-.3\height}{\includegraphics[width=1em, alt=]{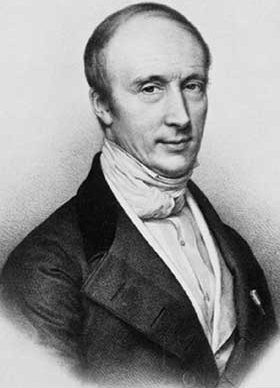}}
    \textbf{Augustin-Louis Cauchy}

    \smallskip

    \sffamily
    I’ll restrict myself here to outlining the principles underlying
    [my method], with the intention to come again over the same
    subject, in a paper to follow.
  \end{minipage}

  \medskip

  \footnotesize
  English translation, \citet{CauchyGD}
  \caption{When proposing gradient descent, Cauchy was apparently
    unhappy with the lack of rigor, as the excerpt shows.
    As often the case, there is no follow-up paper!
    \citep{CauchyGDrecall}.}
  \label{fig:Cauchy}
\end{figure}

At a bird's eye view, constrained optimization problems involve minimizing an
objective function $f$ (or equivalently maximizing a utility function~$-f$)
over a feasible region $\mathcal{X}$.  In most machine learning applications
(including the ones just mentioned), $f$ measures the discrepancy between the
model we aim to develop and the underlying data.  Typically, optimization
methods start from an initial point and aim to improve the solution via local
steps/measurements. What kind of information they use determines how fast they
might converge to a desirable solution. For instance, at the lowest level, an
optimization method may only have access to a limited number of function
evaluations at each step (so called zeroth-order information).
Typically,
it is feasible to obtain higher order information such as the rate of change,
i.e., gradient, which leads to a slew of algorithms that exploit this
information (in particular, the direction of largest change of
$f$) in an intelligent fashion. These algorithms are called first-order
methods,
like the gradient descent method attributed to Augustin-Louis Cauchy,
see Figure~\ref{fig:Cauchy}.
Of course, we can think of developing methods that use not only the
rate of change, but also the rate of the rate of change (i.e., so called
second-order methods), or the rate of the rate of the rate of change (so called
third-order methods), etc. Nothing prevents us from doing that except the
computational cost. As the data points become high-dimensional, the cost of
computing higher derivatives becomes prohibitive. This is why the first-order
algorithms, which only utilize the gradient of the objective function to build
each iterate, provide a good compromise between progress and cost per function
access and have thus become the workhorse of most machine learning
applications.  This survey is mainly focused on first-order methods.

\begin{figure}[b]
  \centering
  \footnotesize
  \begin{minipage}{.5\linewidth}
    \includegraphics[height=2.5em, alt=]{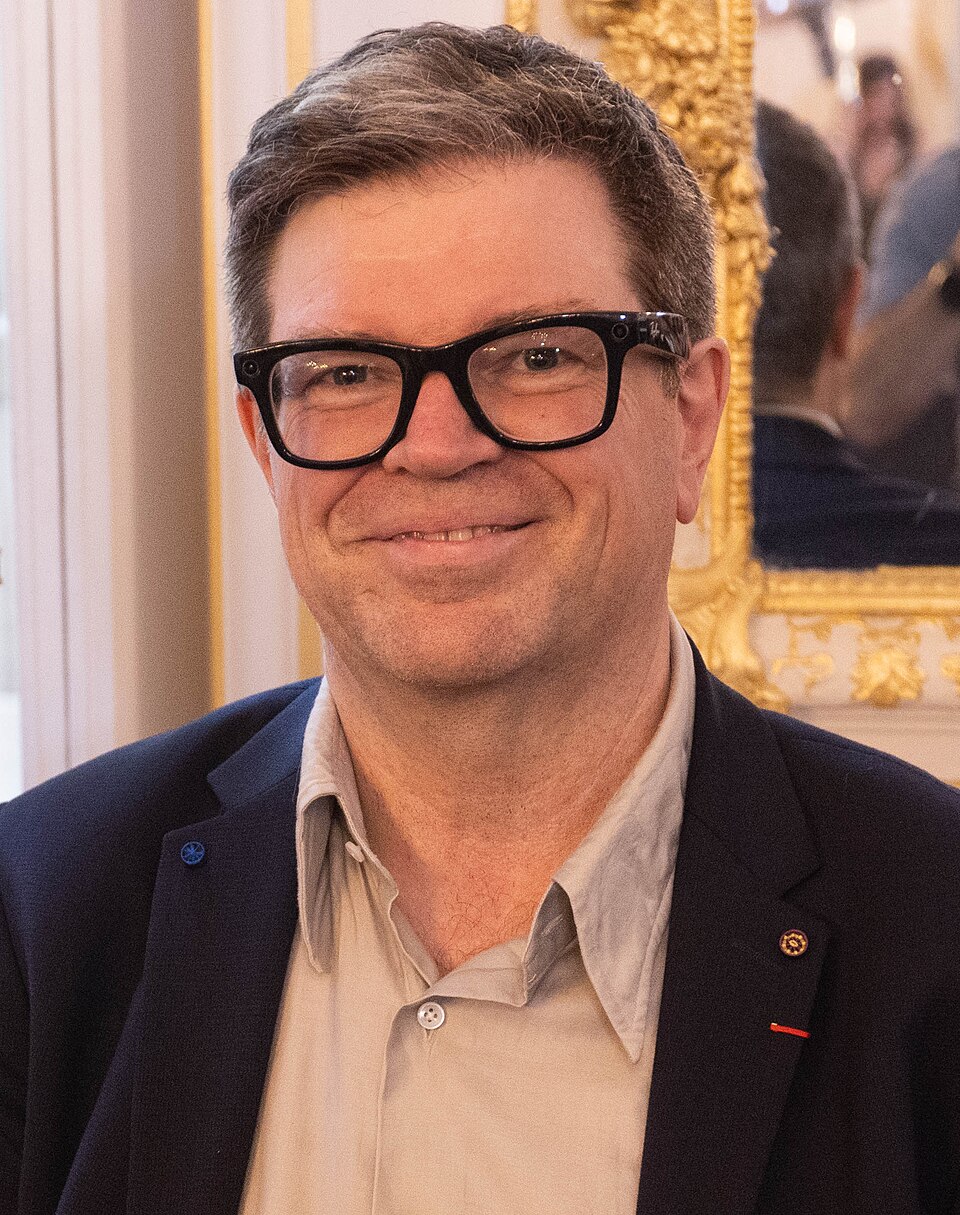}
    \begin{tabular}[b]{ll}
      Yann LeCun \\ @ylecun
    \end{tabular}
    \smallskip

    \sffamily
    I've been trying to convince many of my more theory oriented
    colleagues of the unbelievable power of gradient descent for close
    to 4 decades.
  \end{minipage}
\medskip

\footnotesize
Yann LeCun [@ylecun] (2022, June 5) [Tweet] Twitter (X since 2023,
July 23)\\
\url{https://x.com/ylecun/status/1533451405167669249}\\
Photo courtesy of Ecole polytechnique (Paris, France).
\caption{Recently, academics have started using Twitter (with each message
containing up to 280 characters) to settle differences. Ironically, the medium
is hardly optimized for such debates.}
\label{fig:yann}
\end{figure}

As we mentioned earlier,  most first-order algorithms move along adequately
chosen directions that typically attempt to decrease the value of the objective
function. However note that such methods may easily leave the feasible region,
leading to an infeasible solution. Traditionally, \emph{projected} gradient
descent algorithms have been the method of choice for constrained optimization.
Some even attributed the huge success of machine learning methods to the
gradient descent algorithm (see Figure~\ref{fig:yann} for a post starting a
comprehensive debate).

However, in recent years, with the surge of large data and complicated
constraints, it has become evident that such methods may hugely suffer from the
cost of the projection operation. Indeed, it is known that for many practical
problems, it is practically infeasible to compute the projection operator. As a
result, there has been a lot of interest in optimization methods that do not
require any projection, so-called \emph{projection-free methods}, which resort
to simpler operations with the aim of never leaving the feasible domain.

\section{Why Frank–Wolfe algorithms?}
\label{sec:whyFW}

Recall that the projection operator maps a point \(y\) to a closest
point in some set \(\mathcal{X}\):
\begin{equation}\label{eq: projection}
  \argmin_{x \in \mathcal{X}} \norm[2]{x - y} \quad \text{(projection)},
\end{equation}
which can be implemented via a quadratic program.  To avoid projections, we
should consider simpler  mathematical programs that are arguably easier to
implement. This survey is about methods that replace the projection with  a
Linear Minimization Oracle (LMO), i.e., 
\begin{equation}\label{eq: linear-minimization}
\argmin_{x \in \mathcal{X}} \innp{x}{y} \quad \text{(linear minimization)}.
\end{equation}
When the feasible region $\mathcal{X}$ admits fast linear optimization, the method of choice
becomes the Frank–Wolfe algorithm \citep{fw56}, also known as Conditional
Gradient algorithm \citep{polyak66cg}, a simple projection-free first-order
algorithm for constrained convex optimization problems, requiring a comparable
number of linear minimizations as the number of projections needed by
projection-based first-order methods.  Thus the actual cost of these
projection-free methods is much lower. Frank–Wolfe methods avoid projection by
moving towards but not beyond an extreme point obtained via linear
minimization, which obviously ensures staying within the feasible region
$\mathcal{X}$.  One of the most striking examples highlighting the advantage of
projection-free methods is probably the spectrahedron, where projections
require the computation of a full SVD (singular value decomposition), whereas
for the LMO the computation of the largest eigenvector is enough, e.g., via 
the Lanczos method.  Some other practical problems for which projections are costly
are matrix completion \citep{freund2017extended}, network routing
\citep{leblanc1985improved}, and finding a maximum matching
\citep{GraphMatching16}. 

\begin{table}
 \caption{Complexities of linear minimizations and projections on some
    convex sets up to an additive error \(\varepsilon\) in the Euclidean norm.
    When \(\varepsilon\) is missing, there is no additive error.
		The $\tilde{\mathcal{O}}$ hides polylogarithmic factors in
		the dimensions and polynomial factors in constants related to the
		distance to the optimum.
    For the nuclear norm ball\index{nuclear norm ball|see {spectrahedron}},
    i.e., the \myindex{spectrahedron},
		$\nu$ denotes the number of non-zero entries
		and $\sigma_1$ denotes the top singular value
		of the projected matrix.}
\label{tab:lmo}

\small
\centering

\begin{tabularx}{\linewidth}{@{}>{\hangindent=1em\hangafter=1}Xll@{}} 
\toprule
Set & Linear minimization & Projection\\
\midrule
\(n\)-dimensional $\ell_p$-ball, \(p \neq 1, 2, \infty\)
&$\mathcal{O}(n)$&$\tilde{\mathcal{O}}(n/\varepsilon^2)$\\
Nuclear norm ball of \(n \times m\) matrices
&$\mathcal{O}(\nu\ln(m+n)\sqrt{\sigma_1}/\sqrt{\varepsilon})$&$\mathcal{O}(mn\min\{m,n\})$\\
  Flow polytope on a graph with \(m\) vertices and \(n\)~edges
  with capacity bound on edges
    &$\mathcal{O}((n \log m) (n + m \log m))$
                          &$\mathcal{O}(n^4 \log n)$\\
Birkhoff polytope
(\(n \times n\) doubly stochastic matrices)
&$\mathcal{O}(n^3)$&$\tilde{\mathcal{O}}(n^2/\varepsilon^2)$\\
\bottomrule
\end{tabularx}
	
\end{table}
To provide an (incomplete) comparison, in Table~\ref{tab:lmo} we present
complexities of linear minimization and projection onto the
$\ell_p$-ball,
the nuclear norm ball (also called spectrahedron),
the flow polytope, and the Birkhoff polytope
(see Example~\ref{ex:Birkhoff}). The matching polytope is not included in the table,
as no (either theoretical or practical)
efficient projection algorithm is known,
however it has a practical, efficient polynomial time
linear minimization algorithm
\citep{edmonds1965maximum}
despite all LP formulations having
an exponential number of inequalities
(in the number of nodes of the underlying graph).
The table is from
\citet{combettes21complexity}
except the complexities for the flow polytope,
which are from original research papers
\citet{vegh2016flow,orlin1983flow}.
There is also an old algorithm by Wolfe for projections to
an arbitrary polytope in the Euclidean norm,
somewhat similar to the Away-step Frank–Wolfe algorithm,
see \citet{WolfeNearest0976},
which unfortunately has an exponential worst-case bound,
see \citet{WolfeMinNorm2020}.

Another important property of conditional gradient algorithms is the
way the iterates are usually built, namely as convex combinations of
extremal points (e.g., vertices in the case of polytopes). In fact,
most Frank–Wolfe algorithms add (at most) one extremal point per
iteration to the representation. This typically leads to sparse
iterates, which is important in many applications. A classical example
is optimization over the nuclear norm ball, where in each iteration a
single rank-$1$ matrix is added, so that after $t$ iterations, we
obtain an approximation or solution with rank at most $t+1$
(see Examples~\ref{example:lazy_comparison}, \ref{example:CGS_cvx} and
\ref{ex:online-matrix}). Another classical example is the approximate
Carathéodory problem, where a point contained in a compact convex set
is to be approximated as a convex combination of a small number of
extremal points and the object of study here is the sparsity
vs.~approximation error trade-off,
see Example~\ref{ex:approx_Caratheodory}.

\section{A bit of history}

Frank–Wolfe algorithms (a.k.a.~conditional gradients) have a very long
history and we only mention some of the key works along the way. In
particular, recreating the timeline is not without difficulty given
sometimes prolonged publication delays, hence minor inaccuracies are
inevitable.

The original algorithm is due to \citet{fw56} with later
generalizations due to \citet{polyak66cg}. Initial convergence
guarantees for the vanilla method were established back then. A first
lower bound to the achievable convergence rate was then obtained in
\citet{Canon_FWbound68}. In particular, the zigzagging phenomenon
that fundamentally limits the rate of convergence of the vanilla
Frank–Wolfe algorithm became apparent and various extensions of the
base method were explored. One particular suggestion was the
Away-step Frank–Wolfe algorithm by \citet{wolfe70}, albeit at this
point in time without any established convergence rate. Another
important variant, which came to be known as the Fully-Corrective
Frank–Wolfe algorithm, was presented by \citet{Holloway74FW}. Other
important works considered specific convergence rate regimes
\citep{dunn1979rates} and in particular open-loop step-sizes
\citep{dunn78}. Almost ten years later \citet{gm86} presented a very
nice argument demonstrating that when the optimal solution is
contained in the interior, and the function is strongly convex 
then the
vanilla Frank–Wolfe algorithm with line search converges
linearly. This result also implied that after a certain number of
iterations (effectively after having identified the optimal face), for
strongly convex functions Wolfe's Away-step Frank–Wolfe algorithm
converges linearly, indicating that global linear convergence for
strongly convex functions might be achievable; however it would take
another 30~years to actually prove this.

Then it stayed relatively quiet around conditional gradient methods for about
27~years, interspersed with isolated but no less important results such as,
e.g., coreset constructions via the Frank–Wolfe algorithm in
\citet{clarkson08,clarkson03,clarkson08opti}. Around 2013, there was a
sudden surge of new results for Frank–Wolfe methods
\citep{jaggi13fw,lacoste13,lan2013complexity,garber2013playing} re-analyzing
some of the former results and providing unifying perspectives. The stream of
new results in that area has been uninterrupted since then. Some of the
highlights include faster convergence rates over structured sets \citep[see,
e.g.,][]{garber2015faster} and global linear convergence proofs for
(modifications of) conditional gradients \citep{garber2016linearly}, including
for Wolfe's Away-step Frank–Wolfe algorithm \citep{lacoste15}. In the following
we will present many more results, up to today, as well as put the classical
ones into context.

\section{Overview of conditional gradient algorithms}

As mentioned earlier conditional algorithms are iterative algorithms:
i.e., they generate a sequence of feasible solutions called iterates,
with each iterate intended to improve on the previous one.
The core building block is
the Frank–Wolfe step: at each iterate
move a bit towards an extreme point that
minimizes the linear approximation of the objective function
given by its gradient.
This move leads to the next iterate.
The extensions and variants of the original Frank–Wolfe algorithm fall broadly into the following four categories.

\begin{description}
\item[Alternative descent directions.]
A descent direction is the direction between successive iterates.
The first category combines the Frank–Wolfe step with other ways of
finding an improving solution or descent directions typically
maintaining projection-freeness, in particular for structured feasible
regions (such as, e.g., polytopes). The Away-step Frank–Wolfe
algorithm and the Pairwise-step Frank–Wolfe algorithm are quite unique
in this category by adding simple new steps designed for the
Frank–Wolfe setting overcoming known convergence obstacles (e.g., the
zigzagging phenomenon). An extreme variant is for example the
Fully-Corrective Frank–Wolfe algorithm which optimizes over a local
subproblem to determine the best direction.

\item[Adaptivity.]
The second category improves parts of the Frank–Wolfe algorithm by,
e.g., deliberately avoiding problem parameter estimations or
arbitrary fixed quantities, e.g., for the step size, i.e.,
how far to go into the descent direction.
Such parameters are a source of
suboptimal performance in practice for all kinds of algorithms,
not just conditional gradients.
The aim is to derive adaptive algorithms that, e.g., adjust to the
\emph{local behavior}, rather than using some loose global estimates
for these parameters. This includes adaptive step-size strategies that
approximate the local smoothness constant but also variants that
automatically adapt to additional problem structure, both in the
objective function (e.g., sharpness) or the feasible region (e.g.,
uniform convexity).

\item[Other Oracles.]
A third category's goal is to replace even the most basic building
blocks of Frank–Wolfe algorithms
(function access, linear minimization) treated as oracles
with cheaper or stronger variants.
Cheaper variants include, e.g., lazy algorithms allowing a large error
in linear minimization, stochastic algorithms using stochastic
gradients (cheap gradient approximations)
instead of the true gradients,
or distributed algorithms where function access has significant
communication cost and hence amortization schemes are devised.
On the other hand, stronger oracles include the nearest extreme
point oracle, which solves a more complicated problem but delivers
finer sparsity and convergence guarantees.

\item[Applicability.]
Finally, there has been a significant
amount of work for making Frank–Wolfe methods more broadly
applicable, e.g., for a wider class of objective functions
(e.g., generalized self-concordant, non-smooth,
or composite functions) and feasible regions
(e.g., infinite dimensional domains).
\end{description}
As a note on naming conventions,
while in the literature often ``Frank–Wolfe'' and ``conditional gradients''
are used interchangeably, in this survey a ``Frank–Wolfe algorithm'' is a core variant of the
original algorithm, while a ``conditional gradient algorithm''
differs significantly from the original version.  This classification
reflects the opinion of the authors in some corner cases.
However, we use the original algorithm names for consistency with the
literature, even if the classification requires otherwise.

\section{Basics and notation}
\label{sec:basics}

The basic setup is the following, which
some later sections augment or replace.
Given a differentiable function $f \colon \mathcal{X} \to
\mathbb{R}$ over a
compact (i.e., closed and bounded)
convex domain \(\mathcal{X}\), our goal is to find a point in
\(\mathcal{X}\) that minimizes the value of \(f\)
up to a given additive error.
(The definition of convexity is recalled below.)
We do this armed with the following two oracles as the only methods
directly accessing the objective function \(f\)
and the feasible region \(\mathcal{X}\).
As mentioned above,
these oracles provide a balanced choice between cost and useful
information in many settings.
Note that here and in the following, as our algorithms don't need
access to function values outside the feasible region,
unless otherwise stated,
we implicitly make the simplifying assumption
that the domain of the function is the feasible region $\mathcal{X}$.

The first oracle, called the \emph{First-Order Oracle}\index{First-Order Oracle} (denoted
as FOO\index{FOO|see {First-Order Oracle}} and shown in Oracle~\ref{ora:FO}), gives information about the
function \(f\): when queried with a point \(x \in \mathcal{X}\),
FOO returns the function value \(f(x)\)
(i.e., zero-order information)
and the gradient \(\nabla f(x)\) of \(f\) at \(x\).

\begin{oracle}[H]
  \caption{First-Order Oracle for \(f\) (FOO)}
  \label{ora:FO}
  \begin{algorithmic}
    \REQUIRE Point \(x \in \mathcal X\)
    \ENSURE \(\nabla f(x)\) and \(f(x)\)
\end{algorithmic}
\end{oracle}

The second oracle, called a \emph{Linear Minimization Oracle}\index{Linear Minimization Oracle} (denoted
as LMO\index{LMO|see {Linear Minimization Oracle}} and shown in Oracle~\ref{ora:LP}), gives information
about the domain \(\mathcal{X}\): when queried with a linear function
\(c\), the LMO returns an extreme point \(v \in \mathcal{X}\)
minimizing \(c\), in particular,
\(v = \argmin_{x\in\mathcal{X}} \innp{c}{x}\).
Note that \(v\) is not necessarily unique,
however as here, we slightly abuse notation for simplicity.
(Extreme points are often called
\emph{atoms} in the context of conditional gradient algorithms).

\begin{oracle}[H]
  \caption{Linear Minimization Oracle over \(\mathcal{X}\) (LMO)}
  \label{ora:LP}
  \begin{algorithmic}
    \REQUIRE Linear objective \(c\)
    \ENSURE \(v = \argmin_{x\in \mathcal{X}}\innp{c}{x}\)
\end{algorithmic}
\end{oracle}

\enlargethispage{1\baselineskip}

We now review convexity and related properties
allowing for better algorithms
than evaluating the objective function at every feasible point
by brute force.

\begin{definition}[Convex set]\index{convex set}
 A set \(\mathcal{X} \subseteq \mathbb{R}^{n}\) is \emph{convex} if for every
 points \(x, y \in \mathcal{X}\) the segment between \(x\) and \(y\) is
 contained in \(\mathcal{X}\), that is,
\begin{equation}
\gamma x + (1 - \gamma) y \in \mathcal{X}
\quad \text{for all $0 \leq \gamma \leq 1$.}
\label{convexset}
\end{equation}
A convex set \(\mathcal{X}\) is \emph{full dimensional} if it has a
non-empty interior, or equivalently
if it is not contained in any proper affine subspace of \(\mathbb{R}^{n}\).
\end{definition}

\begin{definition}[Convex function]\index{convex function}
A function \(f \colon \mathcal{X} \to \mathbb{R}\) is \emph{convex} if its domain \(\mathcal{X}\) is convex and
\begin{equation}
f\bigl(\gamma x + (1 - \gamma) y\bigr) \leq \gamma f(x) + (1 - \gamma) f(y)
  \quad \text{for all } x, y \in \mathcal{X},\ 0 \leq \gamma \leq 1
  .
\label{convexfunction}
\end{equation}
If \(f\) is differentiable then it is convex if and only if
\begin{equation}
  \label{convex}
  f(y) - f(x) \geq \innp{\nabla f(x)}{y - x}
  \quad \text{for all } x, y \in \mathcal{X},
\end{equation}
or equivalently $f(y) - f(x) \leq \innp{\nabla f(y)}{y-x}$.
\end{definition}

Convex functions have the nice property that local minima are
always global minima.  Optimization algorithms often
rely on good local approximation of the objective function, and
therefore the following two properties have been useful for finding a
quadratic approximation. For simplicity we formulate the properties
only for differentiable functions.
See Section~\ref{sec:smooth-strong-conv-ineq} below for further
inequalities for these properties.
See \citet{gutman2020condition} for generalization of these properties,
which in most results can replace the above ones.

\begin{definition}[Strongly convex function]\index{strongly convex function} \label{DefStrCvx}
  A differentiable function
  \(f \colon \mathcal{X} \to \mathbb{R}\)
  is \emph{$\mu$-strongly convex}
  if
  \begin{equation}
    f(y) - f(x) \geq \innp{\nabla f(x)}{y-x}
    + \frac{\mu}{2} \norm{y-x}^{2}
     \quad \text{for all } x, y \in \mathcal{X}.
\label{strconvex}
\end{equation}
\end{definition}

\begin{definition}[Smooth function]\label{DefSmooth}\index{smooth function}\index{smoothness}
  A differentiable function
  \(f \colon \mathcal{X} \to \mathbb{R}\)
  is \emph{$L$-smooth}
  if for all $x,y\in\mathcal{X}$
  \begin{equation}
    f(y) - f(x) \leq \innp{\nabla f(x)}{y-x}
    + \frac{L}{2} \norm{y-x}^{2}
    \quad \text{for all } x, y \in \mathcal{X}.
\label{smooth}
\end{equation}
\end{definition}

Note that in optimization theory smoothness has a different meaning
than in topology or differential geometry, where a smooth function
means an infinitely many times differentiable function; older papers
also use the term bounded convexity for smoothness.

An important application of smoothness is estimating the decrease in
function value of \(f\) when improving a solution \(x\) via
moving in an arbitrary direction \(d\),
i.e., estimating \(f(x) - f(x - \gamma d)\),
where \(\gamma\) is called the \emph{step size}, and \(x - \gamma d\)
is the improved solution.
This is formulated in the following lemma,
which is mostly useful when moving in a descent direction, i.e.,
when \(\innp{\nabla f(x)}{d} > 0\).
In applications,
the upper bound on the step size \(\gamma\) required by the lemma
will usually be ensured by force, i.e., choosing
\(\gamma\) as the minimum of the upper bound and a desired value.
Conditional gradient algorithms typically choose \(d = v - x\)
for some feasible point \(v\).

\begin{lemma}[Progress Lemma from smoothness]\index{Progress Lemma}
  \label{lemma:progress}
  Let \(f\) be an \(L\)-smooth function.
  Define \(y = x - \gamma d\),
  where \(d\) is an arbitrary vector (i.e., a direction).
  Then if \(y\) is in the domain of \(f\)
  \begin{equation}
    \label{eq:primalProgressSmall}
    f(x) - f(y)
    \geq
    \frac{\innp{\nabla f(x)}{d}}{2}
    \cdot
    \gamma \quad \text{for} \quad 
    0 \leq \gamma \leq
    \frac{\innp{\nabla f(x)}{d}}{L \norm{d}^2}.
  \end{equation}
\begin{proof}
Equation~\eqref{eq:primalProgressSmall} follows easily from Definition~\ref{DefSmooth}:
\begin{equation}
   \label{eq:primalProgressDerivation}
 \begin{split}
  f(x) - f(y) \geq \gamma\innp{\nabla f(x)}{d}
  - \gamma^{2} \frac{L \norm{d}^{2}}{2}
  & \geq
  \gamma \innp{\nabla f(x)}{d}
  - \gamma \frac{\innp{\nabla f(x)}{d}}{L \norm{d}^2}
  \cdot
  \frac{L \norm{d}^{2}}{2}
  \\
  & =
  \frac{\innp{\nabla f(x)}{d}}{2} \cdot \gamma
  ,
 \end{split}
\end{equation}
where we have used the definition of $y$
and the bound on the step size $\gamma$.
\end{proof}
\end{lemma}

\begin{remark}[Typical application of Lemma~\ref{lemma:progress}]
  \label{rem:progress}
  In applications of Lemma~\ref{lemma:progress}
  there will usually be an upper bound \(\gamma_{\max}\) on
  the step size \(\gamma\) to ensure that \(y\) lies in the feasible
  region.
  The step size will be chosen to optimize the progress bound:
  \(\gamma \defeq \min\{\innp{\nabla f(x)}{d} \mathbin{/}
  (L \norm{d}^{2}), \gamma_{\max}\}\),
  so that the progress lower bound becomes
  \begin{equation*}
    f(x) - f(y)
    \geq
    \frac{\innp{\nabla f(x)}{d}}{2}
    \cdot \min\left\{\gamma_{\max},
    \frac{\innp{\nabla f(x)}{d}}{L \norm{d}^2}\right\}.
  \end{equation*}
In fact for most steps (usually except for a few early iterations), we have $\frac{\innp{\nabla f(x)}{d}}{L \norm{d}^2} \leq \gamma_{\max}$, so that the above reduces to
  \begin{equation*}
    f(x) - f(y)
    \geq
    \frac{\innp{\nabla f(x)}{d}^2}{2L \norm{d}^2}.
  \end{equation*}
Sometimes we will also work with a lower bound $\phi$ with $\innp{\nabla f(x)}{d} \geq \phi$ and in this case, the progress estimation becomes 
  \begin{equation*}
    f(x) - f(y)
    \geq
    \frac{\phi}{2}
    \cdot \min\left\{\gamma_{\max},
    \frac{\phi}{L \norm{d}^2}\right\}.
  \end{equation*}
\end{remark}

\subsection{Notation}
\label{sec:notation}

Now we introduce notation
used throughout this survey.
Whenever the domain is a polytope, i.e., the
convex hull of finitely many points, we will use $P$ instead of
\(\mathcal{X}\).
The domain \(\mathcal{X}\) will always be a subset of
a finite dimensional vector space \(\mathbb{R}^{n}\).
For convenience, we equip \(\mathbb{R}^{n}\) with a
non-degenerate bilinear form
\(\innp{\cdot}{\cdot} \colon
\mathbb{R}^{n} \times \mathbb{R}^{n} \to \mathbb{R}\),
so that every linear function \(\mathbb{R}^{n} \to \mathbb{R}\)
is uniquely represented by a vector \(c \in \mathbb{R}^{n}\)
as \(\innp{c}{\cdot}\).
In particular, the \emph{gradient} \(\nabla f(x)\)
of a function at \(x\) represents the derivative \(f'(x)\)
at \(x\) as a linear function, i.e., \(f'(x) z = \innp{\nabla f(x)}{z}\).
We shall further equip \(\mathbb{R}^{n}\) with a norm \(\norm{\cdot}\)
but for vectors representing linear functions we will use the dual
norm \(\dualnorm{\cdot}\).
Thus \(\abs{\innp{c}{x}} \leq \dualnorm{c} \cdot \norm{x}\)
for example.
As a consequence,
even though \(\nabla f(x)\) depends on the choice of
the inner product \(\innp{\cdot}{\cdot}\),
expressions like \(\innp{\nabla f(x)}{z}\) and
\(\dualnorm{\nabla f(x)}\)
are independent of it.

Readers unfamiliar with norms can restrict to the most common choice:
the Euclidean norm \(\norm[2]{x} = \sqrt{\sum_{i=1}^{n} x_{i}^{2}}\),
which is its own dual norm,
for which we always use the standard scalar product
\(\innp{x}{y} = \sum_{i=1}^{n} x_{i} y_{i}\)
as inner product.
In this case our definition of gradient reduces to the usual one.

\looseness=1
When \(\innp{\cdot}{\cdot}\) is the standard scalar product,
for any symmetric positive definite matrix $H$,
we use $\norm[H]{x}$ to define the matrix norm defined by $H$, that is,
$\norm[H]{x} = \sqrt{\innp{x}{Hx}}$.
In this survey we will also use the \(\ell_{p}\)-norm
(also called \(p\)-norm):
\(\norm[p]{x} = \sqrt[p]{\sum_{i=1}^{n} \abs{x_{i}}^{p}}\)
for \(1 \leq p \leq \infty\).
Recall that
\(\norm[\infty]{x} = \max_{i=1}^{n} \abs{x_{i}}\).
An \(\ell_{p}\)-ball is a ball in the \(\ell_{p}\)-norm.
The \emph{diameter} $D$ of a set
$\mathcal{X}$ is \smash{$D \defeq \sup_{x,y \in \mathcal{X}} \norm{x - y}$},
which is finite only for bounded sets.

Given a linear map \(A \colon \mathbb{R}^{m} \to \mathbb{R}^{n}\)
we shall denote by
\(A^{\top} \colon \mathbb{R}^{n} \to \mathbb{R}^{m}\)
its adjoint linear map
defined via \(\innp{A^{\top} y}{x} = \innp{y}{A x}\)
for all \(x \in \mathbb{R}^{m}\) and \(y \in \mathbb{R}^{n}\).
For a matrix \(M\), we denote its conjugate transpose by \(M^{\top}\),
for real matrices this is simply the transpose.
Here, there are two non-degenerate bilinear forms:
one over \(\mathbb{R}^{m}\) and another one over \(\mathbb{R}^{n}\).
If both are the standard scalar products, then the matrix representing
\(A^{\top}\) in the standard basis is the transpose of the matrix
representing \(A\) in the standard basis.
The adjoint is mainly used for the gradient of composite functions
\(f = g \circ A\), namely, \(\nabla f(x) = A^{\top} \nabla g(A x)\).

We shall also use the common big \(\mathcal{O}\) notation.
Given nonnegative functions \(f\) and \(g\),
recall that \(f(t) = \mathcal{O}(g(t))\)
if there is a positive number \(C\)
with \(f(t) \leq C g(t)\) for all \(t\).
Similarly \(f(t) = \Omega(g(t))\)
if there is a positive number \(C\)
with \(f(t) \geq C g(t)\) for all \(t\),
i.e., \(g(t) = \mathcal{O}(f(t))\).
Finally, \(f(t) = \Theta(g(t))\)
if there are positive numbers \(C_{1}\) and \(C_{2}\)
with \(C_{1} g(t) \leq f(t) \leq C_{2} g(t)\) for all \(t\),
i.e., \(f(t) = \mathcal{O}(g(t))\) and \(f(t) = \Omega(g(t))\).
Note that \(f(t) \leq \mathcal{O}(g(t))\)
and \(f(t) \leq \Theta(g(t))\)
are alternate notation for \(f(t) = \mathcal{O}(t)\),
while \(f(t) \geq \Omega(g(t))\) and \(f(t) \geq \Theta(g(t))\)
are equivalent to \(f(t) = \Omega(g(t))\).

\subsection{Smoothness and strong convexity inequalities}
\label{sec:smooth-strong-conv-ineq}

We collect here several inequalities regarding
smoothness and strong convexity, most of which provide
equivalent characterizations for these properties.
These are probably well known but we do not have direct references
to them.
First, for a differentiable function \(f\),
smoothness and strong convexity can be formulated using only the
gradients but no function values.
Namely, \(f\) is \(\mu\)-strongly convex if
\(\innp{\nabla f(y) - \nabla f(x)}{y - x} \geq \mu \norm{y - x}^{2}\),
and \(L\)-smooth if
\(\innp{\nabla f(y) - \nabla f(x)}{y - x} \leq L \norm{y - x}^{2}\).

Instead of smoothness,
some papers assume a seemingly stronger property:
\(L\)-Lipschitz continuous gradient \(\nabla f\).
Here we prove that it
is equivalent to \(L\)-smoothness if \(f\) is convex
and the domain \(\mathcal{X}\)
is full dimensional
(see \citet[Theorem~2.1.5]{nesterov18} for \(\mathcal{X} = \mathbb{R}^{n}\)).
Thus for twice differentiable convex functions, \(L\)-smoothness is
equivalent to the operator norm of the Hessian being bounded:
\(\norm{\nabla^{2} f} \leq L\), i.e., the rate of change of the
gradients is bounded, which is commonly used in practice
for estimating the smoothness parameter \(L\).

\begin{lemma}[Equivalence of smoothness and Lipschitz-continuous
  gradients]
  \label{lem:smooth-Lipschitz}
  Let \(f \colon \mathcal{X} \to \mathbb{R}\) be a differentiable
  convex function on
  a full dimensional convex domain \(\mathcal{X}\).
  Then \(f\) is \(L\)-smooth if and only if its gradient \(\nabla f\)
  is \(L\)-Lipschitz continuous.
\begin{proof}
It is well known that for a (not necessarily convex)
function \(f\) with
\(L\)-Lipschitz continuous gradients
one has
\(\abs{f(y) - f(x) - \innp{\nabla f(x)}{y - x}}
\leq L \norm{y - x}^{2} \mathbin{/} 2\) for every segment \([x, y]\)
in the domain of \(f\).
In particular, if \(f\) is also convex, then it is \(L\)-smooth.

For the other direction, assume \(f\) is \(L\)-smooth.
Recall that for convex domains Lipschitz continuity is a local
property, so it is enough to show Lipschitz continuity
in a small neighborhood of every point of \(\mathcal{X}\).
In particular,
for \(L\)-Lipschitz continuity in the interior of \(\mathcal{X}\)
it is enough to show
\(\dualnorm{\nabla f(y) - \nabla f(x)} \leq L \norm{y - x}\),
whenever \(x\), \(y\) are inner points of \(\mathcal{X}\)
such that \(\mathcal{X}\) contains the
\(\norm{y-x} / 2\)-neighborhood of \(x\) and \(y\).
By uniform continuity, \(L\)-Lipschitz continuity on the whole of
\(\mathcal{X}\) follows.

Therefore to prove local Lipschitz continuity,
let \(x\) and \(y\) be points of \(\mathcal{X}\)
such that their closed \(\norm{y-x} / 2\)-neighborhoods
are contained in \(\mathcal{X}\).
Let \(z\) be a vector of length \(\norm{y-x} / 2\).
Repeatedly using convexity and smoothness of \(f\),
we obtain
the following.
(For a motivation, the goal is to replace \(f\) with a quadratic
function with a very tight estimation.
In particular,
all inequalities should hold
with equality in the very tight special case of
\(f(w) = L \norm[2]{w}^{2} / 2\)
in the \(\ell_{2}\)-norm and \(z\) some multiple of \(y-x\), which turns out
to be \(z = (y-x) / 2\), for the convexity inequalities to be tight.)
\begin{equation}
  \label{eq:23}
 \begin{split}
  \innp{\nabla f(y) - \nabla f(x)}{z}
  &
  \leq
  f(y) - f(y - z) + \frac{L \norm{z}^{2}}{2}
  +
  f(x) - f(x + z) + \frac{L \norm{z}^{2}}{2}
  \\
  &
  \leq
  f(y) + f(x) - 2 f\left( \frac{x+y}{2} \right) + L \norm{z}^{2}
  \\
  &
  \leq
  \begin{multlined}[t]
    \innp*{\nabla f\left( \frac{x+y}{2} \right)}{y - \frac{x+y}{2}}
    + \frac{L \norm*{y - \frac{x+y}{2}}^{2}}{2}
    \\
    + \innp*{\nabla f\left( \frac{x+y}{2} \right)}{x - \frac{x+y}{2}}
    + \frac{L \norm*{x - \frac{x+y}{2}}^{2}}{2}
    + L \norm{z}^{2}
  \end{multlined}
  \\
  &
  =
  \frac{L \norm{y-x}^{2}}{4} + L \norm{z}^{2}
  .
 \end{split}
\end{equation}
(Here the last step is a special case of the non-differential form of
smoothness presented in Equation~\eqref{eq:smooth-non-diff} later.)

We conclude that
\(\innp{\nabla f(y) - \nabla f(x)}{z}
\leq L \cdot \norm{y-x} \cdot \norm{z}\)
for all \(z\) with \(\norm{z} = \norm{y - x} / 2\).
In particular,
\(\dualnorm{\nabla f(y) - \nabla f(x)} \leq L \norm{y - x}\)
follows.
\end{proof}
\end{lemma}

Next we bound the difference of gradients
for smooth convex functions, which is useful for stochastic
algorithms, see Section~\ref{sec:FW_Sto}.
For the \(\ell_{2}\)-norm and functions defined on the whole \(\mathbb{R}^{n}\),
the statement appeared in \citet[Lemma~1]{svrf16},
and it was shown in \citet{PEP_Taylor_2016} to be the only relationship
between function values and gradients:
given finitely many points with prescribed function values and gradients,
there is an \(L\)-smooth function defined on \(\mathbb{R}^{n}\)
having these function values and gradients
if and only if the prescribed values satisfy the inequality.
In particular, for differentiable convex functions on
\(\mathbb{R}^{n}\),
the inequality is equivalent
to \(L\)-smoothness \citep[Theorem~2.1.5]{nesterov18}.
\begin{lemma}
  \label{lem:smooth-grad-diff}
  Assume that \(f\) is an \(L\)-smooth convex function
  on the \(D\)-neighborhood of a convex domain \(\mathcal{X}\),
  where \(D\) is the diameter of \(\mathcal{X}\).
  Then for any points \(x, y \in \mathcal{X}\) it holds
  \begin{equation}
    \label{eq:smooth-grad-diff}
    \dualnorm{\nabla f(y) - \nabla f(x)}^{2}
    \leq
    2L \bigl(f(y) - f(x) - \innp{\nabla f(x)}{y - x}\bigr).
  \end{equation}
\end{lemma}
Here the assumption that \(f\) is smooth on a large neighborhood of
\(\mathcal{X}\) is necessary: e.g.,
for sufficiently small \(\varepsilon > 0\)
the function
\(f(x_{1}, x_{2}) = (1 + \varepsilon) (x_{1}^{2} + x_{2}^{2})
- (1/2 + x_{2} + x_{2}^{2}) \cos 2x_{1}\)
is \((4 + o(1))\)-smooth in a neighborhood of \([x, y]\)
where \(y = (0, 0)\) and \(x = (\pi /2, 0)\)
but Equation~\eqref{eq:smooth-grad-diff} fails.

\begin{proof}
The proof is similar to that of the previous lemma.
By assumption, for any vector \(z\) with \(\norm{z} \leq D\),
we have by convexity and smoothness
\begin{equation*}
  f(x) + \innp{\nabla f(x)}{y - z - x}
  \leq
  f(y - z)
  \leq
  f(y) - \innp{\nabla f(y)}{z} + \frac{L \norm{z}^{2}}{2}
  .
\end{equation*}
Rearranging provides
\begin{equation*}
  \innp{\nabla f(y) - \nabla f(x)}{z}
  \leq
  f(y) - f(x) - \innp{\nabla f(x)}{y  - x}
  + \frac{L \norm{z}^{2}}{2}
  .
\end{equation*}
This holds for all \(z\) with
\(\norm{z} = \dualnorm{\nabla f(y) - \nabla f(x)} / L\)
as
\(\dualnorm{\nabla f(y) - \nabla f(x)} / L \leq \norm{y-x} \leq D\)
by Lemma~\ref{lem:smooth-Lipschitz}.
Hence, by taking the supremum over all such \(z\)
\begin{multline*}
  \dualnorm{\nabla f(y) - \nabla f(x)} \cdot
  \frac{\dualnorm{\nabla f(y) - \nabla f(x)}}{L}
  \\
  \leq
  f(y) - f(x) - \innp{\nabla f(x)}{y  - x}
  + \frac{L}{2} \cdot
  \left(
    \frac{\dualnorm{\nabla f(y) - \nabla f(x)}}{L}
  \right)^{2}
  ,
\end{multline*}
which yields the claim by rearranging.
\end{proof}

\begin{remark}[Smoothness and strong convexity for non-differentiable
  functions]
  While we shall not use them in this survey,
  for information we recall characterizations of
  strong convexity and smoothness
  of a function without using the gradient,
  which readily generalize
  to non-differentiable functions.
  For the \(\ell_{2}\)-norm, these inequalities
  (like most other characterizations we have seen earlier)
  express that
  \(f(x) - \mu \norm[2]{x}^{2} / 2\) and
  \(L \norm[2]{x}^{2} / 2 - f(x)\)
  are convex, respectively.
  A function \(f \colon \mathcal{X} \to \mathbb{R}\)
  is \emph{\(\mu\)-strongly convex}
  if for all \(x, y \in \mathcal{X}\)
  and \(0 \leq \gamma \leq 1\)
  \begin{equation}
    \label{eq:smooth-non-diff}
    \gamma f(x) + (1 - \gamma) f(y)
    \geq
    f\bigl(\gamma x + (1 - \gamma) y\bigr)
    + \gamma (1 - \gamma) \cdot \frac{\mu \norm{y - x}^{2}}{2}
    .
  \end{equation}
  A convex function \(f \colon \mathcal{X} \to \mathbb{R}\)
  is \emph{\(L\)-smooth} if for all \(x, y \in \mathcal{X}\)
  and \(0 \leq \gamma \leq 1\)
  \begin{equation}
    \label{eq:strongly-convex-non-diff}
    \gamma f(x) + (1 - \gamma) f(y)
    \leq
    f\bigl(\gamma x + (1 - \gamma) y\bigr)
    + \gamma (1 - \gamma) \cdot \frac{L \norm{y - x}^{2}}{2}
    .
  \end{equation}
\end{remark}

\subsection{Measure of efficiency}
\label{sec:measure-efficiency}

We let \(\Omega^{*}\) denote the set of minima of a function \(f\) over a
convex domain \(\mathcal{X}\).  For ease of presentation, we choose an
arbitrary minimizer $x^{*}$ even if not unique, which will be mostly used for
the suggestive notation \(f(x^{*})\) for minimal function value.

There are several measures for the quality of a proposed solution \(x\) to the
minimization problem.  A natural measure is \(\distance{x}{\Omega^{*}}\), the
\emph{distance to the set of optimal solutions}.  However, for conditional
gradient algorithms no \emph{direct} convergence bound is known for the
distance to the optimal solution set.  More precisely, all known convergence
bounds in distance originate from convergence bounds in function value via an
additional property relating the two, like strong convexity.  The reason is
presumably partly that conditional gradient algorithms at their core are
independent of the norm used for measuring the distance.  (Even though some
parts of the algorithm, like step size, may depend on the norm.) Note that even
if an algorithm generates iterates with the distance to the solution set
converging to~\(0\), this does not mean that the iterates actually converge to
a point, see \citet{combettes22} for counterexamples for the vanilla
Frank–Wolfe algorithm (Algorithm~\ref{fw}).

Another common measure is the \emph{primal gap}
\(h(x) \defeq f(x)-f(x^{*})\), difference in function value,
which is the most commonly used measure  for conditional gradient
algorithms.
\begin{definition}[Primal gap]\index{primal gap}
  \label{def:primal-gap}
  The \emph{primal gap} of a function
  \(f \colon \mathcal{X} \to \mathbb{R}\)
  at a point \(x\) is
  \begin{equation}
    \label{eq:primal-gap}
    h(x) \defeq \max_{y \in \mathcal{X}} f(x) - f(y)
    = f(x) - f(x^{*})
    .
  \end{equation}
\end{definition}

In general, neither distance to solutions
nor difference in function value
is directly computable without knowledge
of the optimum \(x^{*}\) or the optimal function value \(f(x^{*})\).
To remedy this, one can use the \emph{Frank–Wolfe gap}
(Definition~\ref{FrankWolfeGap}), which is an upper bound on the
primal gap that is computable without the knowledge of \(x^{*}\).
This quantity is often used in stopping criteria,
in particular, because it is usually computed as an integral part of
conditional gradient algorithms anyway:

\begin{definition}[Frank–Wolfe gap]\index{Frank-Wolfe gap} \label{FrankWolfeGap}
  The \emph{Frank–Wolfe gap}  of a function
  \(f \colon \mathcal{X} \to \mathbb{R}\)
  at a point \(x\) is
\begin{equation}
  g(x) \defeq \max_{v \in \mathcal{X}} \innp{\nabla f(x)}{x-v}.
\end{equation}
\end{definition}
Clearly \(g(x) \geq 0\), as \(x \in \mathcal X\). Note further that by convexity
\begin{equation}
  \label{eq:gaps}
  0 \leq h(x) \leq \innp{\nabla f(x)}{x - x^{*}}
  \leq \max_{v \in \mathcal{X}} \innp{\nabla f(x)}{x -v} = g(x).
\end{equation}

Note that some works use ``dual gap'' for what we call ``Frank–Wolfe
gap''. The quantity \(\innp{\nabla f(x)}{x - x^{*}}\), which will be used in
the proofs to come, is called the \emph{dual gap}
(which may depend on the choice of \(x^{*}\) if it is not unique). For smooth objective functions,
there is a partial converse
\citep[Theorem~2]{lacoste15}
that allows us to upper bound the
Frank–Wolfe gap using the primal gap
(Proposition~\ref{prop:dual-by-primal-smooth}). The bound is a special
case of Lemma~\ref{lemma:progress}, using the direction \(d = v - x\),
with \(v = \argmax_{v \in \mathcal{X}} \innp{\nabla f(x)}{x-v}\), so
that \(g(x) = \innp{\nabla f(x)}{x-v}\), and applying the obvious
bound \(\norm{x-v} \leq D\).

We shall use the shorthand
\(h_{t}\defeq h(x_{t})\) and \(g_{t} \defeq g(x_{t})\).

\begin{proposition}
  \label{prop:dual-by-primal-smooth}
  For any \(L\)-smooth, not necessarily convex function \(f\)
  over a compact convex set,
  the primal gap provides the following upper bound on
  the Frank–Wolfe gap:
  \begin{equation}
    \label{eq:dual-by-primal-smooth}
    g(x) \leq
    \begin{cases}
      h(x) + \frac{L D^{2}}{2} & h(x) \geq \frac{L D^{2}}{2},
      \\[.5ex]
      D \sqrt{2 L h(x)} &  h(x) \leq \frac{L D^{2}}{2}.
    \end{cases}
  \end{equation}
\end{proposition}

\begin{remark}[First-order optimality condition]\index{First-order optimality condition}
  \label{rem:fo}
  Closely related to the dual gap and the Frank–Wolfe gap, is the
  \emph{first-order optimality condition} for the problem \(\min_{x
    \in \mathcal X} f(x)\) for a differentiable convex function \(f\)
  defined on a convex set \(\mathcal{X}\).
  The point \(x^* \in \mathcal X\) is
  optimal if and only if
  \begin{equation}
    \label{eq:fo-opt}
    \innp{\nabla f(x^*)}{x^* - x} \leq 0,
  \end{equation}
  for all \(x \in \mathcal X\). This in particular holds for the
  maximum, so that we obtain for the Frank–Wolfe gap \(0 \leq g(x^*)
  \leq 0\). Therefore, for any \(x \in \mathcal X\)
  \begin{equation}
    \label{eq:gOptMeasure}
    g(x) = 0 \iff x = \argmin_{y \in \mathcal{X}} f(y).
  \end{equation}
  For non-convex functions, the Frank–Wolfe gap and the dual gap are
  less useful.  E.g., the Frank–Wolfe gap being $0$
  is only a necessary but not sufficient condition
  even for local optimality.
\end{remark}

\chapter{Basics of Frank–Wolfe algorithms}
\label{cha:FW-basics}

In this chapter, we introduce the simplest variant of the family of Frank–Wolfe algorithms,
which serves as a foundation for all later versions.  The final section of the
chapter provides a minimal number of improvements, which are the most influential for further
variants.

\section{Fundamentals of the Frank–Wolfe algorithm}
\label{sec:CG-algorithms}

In the following,
we will present the base variant
of all Frank–Wolfe algorithms.
Already this basic variant is widely used in practice,
e.g., for solving sparse linear equations for ranking webpages
\citep{anikin2022efficient}.
We include the adjective ``vanilla'' in the algorithm's name
to distinguish it from later algorithms.
Throughout, the Frank–Wolfe algorithms are analyzed via the sequence
of function values $f(x_t)$, and we will be interested in deriving
upper and lower bounds on the rate at which $f(x_t)$ converges to
$f(x^*)$, where $x^*$ denotes an optimal solution.
Note that nothing much is known about the convergence behavior of the
iterates $x_t$: even for the vanilla Frank–Wolfe algorithm, the
general assumptions on $\mathcal{X}$ (compact and convex) and $f$
(smooth and convex), which ensure that $f(x_t)$ converges to $f(x^*)$
(Theorem~\ref{fw_sub}), do not ensure that $x_t$ converges
\citep{combettes22}. However, the convergence of function values does
imply that the distance \(\distance{x_{t}}{\Omega^{*}}\) to the set of
optimal solution converges to \(0\).

\subsection{The vanilla Frank–Wolfe algorithm}
\label{sec:classic}
  
The vanilla Frank–Wolfe algorithm \citep{fw56},
specified in Algorithm~\ref{fw},
generates a sequence of feasible points
\(x_{t}\)
by optimizing linear functions, namely, \emph{first-order approximations}
of the objective function \(f\).
At each iteration, the linear optimization provides an extreme point $v_t$,
which together with the current iterate $x_t$
is then used to form the \emph{descent direction} $v_t - x_t$,
a substitute of the negative gradient direction in the Euclidean norm.
The next iterate \(x_{t+1}\) is obtained by moving in this direction.
This can be
schematically seen in Figure~\ref{fig:schematic_CG_step},
showing the level sets of the objective function $f(x)$
and its linear approximation using the gradient at $x_t$,
namely, $f(x_t) + \innp{\nabla f(x_t)}{x - x_t}$.

\begin{algorithm}\index{Frank-Wolfe algorithm, original variant}
  \caption{(vanilla) Frank–Wolfe (FW) \citep{fw56}
    (a.k.a. Conditional Gradient
    \citep{polyak66cg})}
  \label{fw}
  \begin{algorithmic}[1]
    \REQUIRE Start atom $x_0\in \mathcal{X}$,
      objective function \(f\),
      smoothness \(L\)
    \ENSURE Iterates $x_1, \dotsc \in \mathcal{X}$
    \FOR{$t=0$ \TO \dots}
      \STATE\label{fw_extract}
        $v_{t} \gets
        \argmin_{v \in \mathcal{X}} \innp{\nabla f(x_{t})}{v}$
      \STATE\label{line:basic-step-size}
        \(\gamma_{t} \gets
        \min\left\{
          \frac{\innp{\nabla f(x_{t})}{x_{t}-v_{t}}}{L
            \norm{x_{t} - v_{t}}^{2}},
          1
        \right\}\)
        \COMMENT{approximate \(\argmin_{0 \leq \gamma \leq 1}
          f(x_{t} + \gamma (v_{t} - x_{t}))\)}
      \STATE $x_{t+1} \gets x_{t} + \gamma_{t} (v_{t} - x_{t})$
    \ENDFOR
  \end{algorithmic}
\end{algorithm}

\begin{figure}[t]
  \centering
  \includegraphics[width=0.5\linewidth,alt={A Frank–Wolfe step
    with value of objective function and its linear approximation via
    gradient.}]{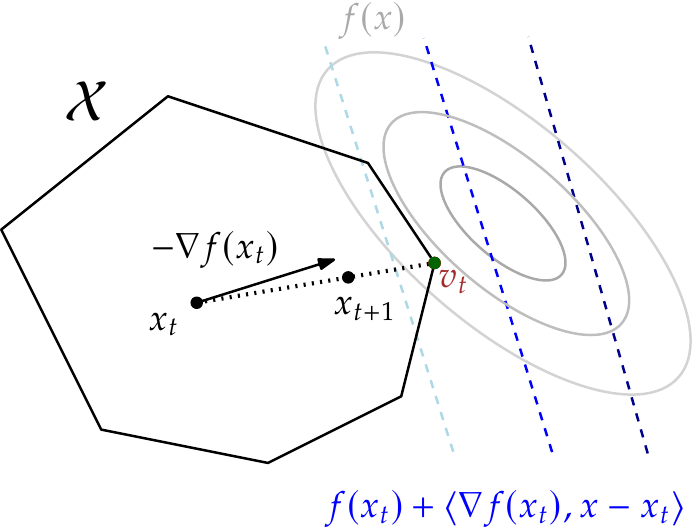}

\caption{Schematic of a Frank–Wolfe step
for minimizing a convex function \(f\)
over a polytope
$\mathcal{X}$, i.e., a single step to improve on a solution
\(x_{t}\) to a new solution \(x_{t+1}\).
The step builds a linear approximation of \(f\)
using the gradient at $x_t$,
namely,
\(f(x_t) + \innp{\nabla f(x_t)}{x - x_t}\).
The Frank–Wolfe vertex \(v_{t}\)
is a vertex minimizing the linear approximation.
The step moves from \(x_{t}\) towards \(v_{t}\)
by an amount specified by some step size rule
to obtain the new solution \(x_{t+1}\).
The contour lines of the function \(f\) being minimized
and the linear approximation are depicted in shades of gray and
blue, respectively.}
\label{fig:schematic_CG_step}
\end{figure}

An important choice as with many optimization algorithms is the step
size rule, i.e., the choice of \emph{step size} \(\gamma_{t}\) in the
algorithm.
We now briefly discuss the various choices, all of which aim to
greedily minimize the function value.
We will mostly use the \emph{short step rule} in algorithms,
being easy to compute and providing good performance.

\begin{description}
\item[Line search.]\index{line search} Line search takes the point with minimum function value
  along the descent direction, i.e., \(x_{t+1} = \argmin_{x \in [x_{t}, v_{t}]}
  f(x)\) optimizing \(f\) over the line segment between the current iterate
  \(x_{t}\) and the Frank–Wolfe vertex \(v_{t}\).  This promises the most immediate
  progress, and has as good if not better theoretical convergence rate than the
  other rules below.  In particular, it guarantees monotonically decreasing function
  values.  However line search is often expensive, and therefore some heuristic
  \emph{step size rules} are employed instead.
  The other step size rules are also useful for line search
  itself: as a starting point or to limit search
  as a trade-off between accuracy and computational cost.

\item[Short step rule.]\index{short step rule} The short step rule, stated in Algorithm~\ref{fw},
  minimizes a quadratic upper bound of \(f\) provided by smoothness (this
  approximation has been implicitly used in Lemma~\ref{lemma:progress}).  The
  rule has an easily computable exact formula, and guarantees monotonically
  decreasing function values.  The downside of the short step rule is that it
  requires a good approximation of the smoothness constant \(L\) of \(f\).  In
  Section~\ref{sec:adaptive-step-size} we will present an adaptive (short) step
  size rule that dynamically approximates $L$, basically providing the
  performance of line search but being as cheap as the short step rule.  In
  particular, this rule adapts to local changes in smoothness whereas the short
  step rule uses a global upper bound estimate.
  The convergence rate in Theorem~\ref{fw_sub} also holds for
  the variant \(\gamma_{t} = \min
  \{\innp{\nabla f(x_{t})}{x_{t}-v_{t}} \mathbin{/} (L D^{2}), 1\}\),
  which is also useful for line search criteria and theory,
  but direct use is less practical
  as it involves an additional global parameter:
  the diameter \(D\) of the feasible region.

\item[Function-agnostic step size rules.]\index{function-agnostic step size rule} Function-agnostic step size rules
  choose the step size \(\gamma_{t}\) independent of the objective function
  \(f\), solely as a function of the iterate number \(t\).  They were
  proposed in \citet{dunn78} to avoid the need for knowledge of the
  smoothness
  constant $L$, but they also turned out to be especially useful when even
  obtaining function values of \(f\) is expensive, see
  Chapter~\ref{cha:FW-large} for some examples.
  A major drawback of any
  function-agnostic step size rule is that it limits the algorithm to a
  predetermined convergence rate, removing potential adaptivity from the
  algorithm, see Proposition~\ref{prop:FW-fixed-lower}.  The investigated
  function-agnostic step size rules are similar to those of accelerated
  projected gradient descent, having a similar (worst-case) convergence rate
  for the vanilla Frank–Wolfe algorithm. In particular, the simple rule
  \(\gamma_{t} = 2/(t+2)\), which was popularized in \citet{jaggi13fw}, provides
  the best currently known convergence rate up to a constant factor. Note 
  that the \(2\) in the numerator is essential. With function-agnostic rules,
  the primal gap often increases in some iterations in practice, i.e., we do
  not have a descent algorithm in the classical sense;
  see for example the first
  step for the rule \(\gamma_{t} = 2/(t+2)\) in Figure~\ref{fig:stepsizeL2} and
  non-monotonicity is also prominently visible on the right plot of
  Figure~\ref{fig:NEP}.
\end{description}
As a rule of thumb, the less a step size rule depends on arbitrary or
global parameters, the better convergence it provides in practice.
Non-agnostic rules often provide better theoretical convergence
guarantees than agnostic ones
for advanced algorithms or favorable settings,
e.g., when the function is sharp (see,
Section~\ref{sec:adaptive_rates})
or when the geometry has favorable properties (see
Section~\ref{sec:improved-convergence}).
However, recently it has been shown that there are (very rare) cases
where function-agnostic step size rules yield quadratically higher
convergence rates than the short step rules and its variants
(see \citet{bach2012equivalence} and \citet{wirth2022acceleration}).

The following example
and Figure~\ref{fig:stepsize}
highlight most of the discussed points from above.

\begin{example}[A simple run of the vanilla Frank–Wolfe algorithm]
	\label{exa:convergence_simple_fw}
  We provide a simple example where the run of the vanilla Frank–Wolfe
  algorithm (Algorithm~\ref{fw}) can be computed by hand.  Let
  \(\mathcal{X} = [-1, +1]\) be the unit ball in the real line and
  \(f(x) = x^{2}\) be the objective.  The optimum is clearly
  \(x^{*} = 0\). Obviously, with line search the algorithm will reach
  the optimum in one iteration \(x_{1} = x^{*}\). The short step rule
  will do the same using the exact value of the smoothness constant
  \(L=2\). However, using a looser smoothness constant \(L > 2\), the
  iterates will be \(x_{t} = (1 - 2 / L)^{t}\) (starting at
  \(x_{0} = 1\)),
  hence \(h_{t} \defeq h(x_{t}) = f(x_{t}) - f(x^{*}) = (1 - 2 / L)^{2t}\).
  For the
  function-agnostic step size rule \(\gamma_{t} = 2 / (t+2)\)
  it is
  easy to verify that \(x_{2t} = 1 / (2 t + 1)\) and
  \(x_{2t+1} = - 1 / (2 t + 1)\) for all \(t \geq 0\), i.e.,
  \(\abs{x_{t} - x^{*}} = \Theta(1/t)\) and
  \(h_{t} = \Theta(1 / t^{2})\). This very simple example already
  demonstrates that the choice of step size rule influences
  the convergence rate.
\end{example}

\begin{figure}
\centering
\begin{minipage}[b]{0.4\textwidth}
  \centering
  \includegraphics[width=\linewidth, alt={Example of exponentially
    decreasing primal gap over the \(\ell_{2}\)-ball.}]{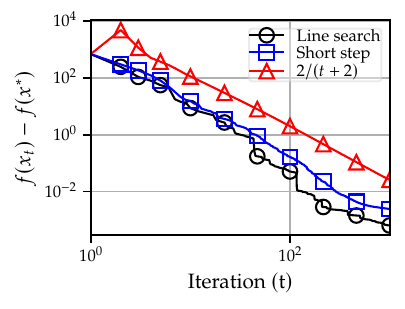}
  \subcaption{\(100\)-dimensional unit $\ell_{2}$-ball,
    condition number
    \(L/\mu = \num{10000}\).}
  \label{fig:stepsizeL2}
\end{minipage}
\qquad
\begin{minipage}[b]{0.4\textwidth}
  \centering
  \includegraphics[width=\linewidth, alt={Example of modest primal gap
    convergence over the hypercube (\(\ell_{1}\)-ball).}]{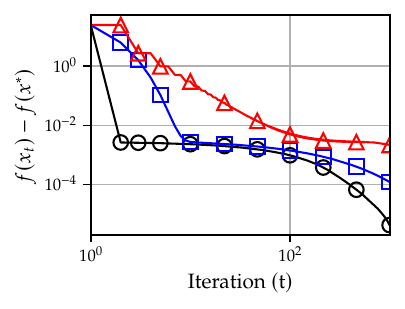}
  \subcaption{\(100\)-dimensional unit $\ell_{1}$-ball,
    condition number
    \(L/\mu = 100\).}
  \label{fig:stepsizeL1}
\end{minipage}

\caption{Primal gap evolution of the vanilla Frank–Wolfe algorithm
  (Algorithm~\ref{fw})
  for a \(\mu\)-strongly convex, \(L\)-smooth quadratic
  function under various step size rules.}
\label{fig:stepsize}
\end{figure}

The vanilla Frank–Wolfe algorithm has a similar convergence rate in
primal gap as the projected gradient descent algorithm, namely,
$f(x_t) - f(x^*) \leq \frac{2 L D^{2}}{t + 3}$.
By Proposition~\ref{prop:dual-by-primal-smooth}, this also implies a
conservative bound on the convergence rate of the Frank–Wolfe gap:
$g_t \defeq g(x_{t}) \leq \frac{2 L D^{2}}{\sqrt{t + 3}}$
(which will be generalized to non-convex functions
in Theorem~\ref{thm:fw-nonconvex}).
However, the Frank–Wolfe gap is not monotonically decreasing in $t$,
and therefore the \emph{running minimum} of
the Frank–Wolfe gap up to iteration $t$ is more suitable
for convergence results, and in fact has a convergence rate similar to the
primal gap.
Thus it is sufficient to justify a stopping criterion
\(g_{t} \leq \varepsilon\) for an additive error of
\(\varepsilon > 0\) in primal gap:
it will be fulfilled at some point
within \(\mathcal{O}(1 / \varepsilon)\)
iterations.

\begin{theorem}
\label{fw_sub}
Let \(f \colon \mathcal{X} \to \mathbb{R}\) be an \(L\)-smooth
convex function on a compact convex set \(\mathcal{X}\)
with diameter \(D\).
The vanilla Frank–Wolfe algorithm (Algorithm~\ref{fw})
converges as follows:
\begin{equation}
  \label{eq:vanilla-rate}
  f(x_t) - f(x^*) \leq \frac{2 L D^{2}}{t + 3}
  \quad \text{for } t \geq 1.
\end{equation}
The running minimum of the Frank–Wolfe gaps up to iteration $t$
has a similar convergence guarantee:
\begin{equation}
  \min_{0 \leq \tau \leq t} g_{\tau} \leq \frac{4 L D^{2}}{t + 3}
  \quad \text{for } t \geq 1.
\end{equation}
In other words, the algorithm produces a solution with a primal gap
and Frank–Wolfe gap
smaller than \(\varepsilon > 0\)
with at most \(2 L D^{2} / \varepsilon\)
linear optimizations and gradient computations.
\end{theorem}

Before we prove the theorem, the following remark is useful.
Note that here we deliberately chose a proof
that is directly compatible
both with the short step rule and function-agnostic step size rules by not
establishing a contraction of the form $h_{t+1} \leq (1 - \alpha_t) h_t$
(recall that \(h_{t} \defeq h(x_{t})\))
but
rather using an induction argument; contraction-based arguments will be used
later in Section~\ref{sec:improved-convergence}. 

\begin{remark}[Modified agnostic step size $\gamma_t = 2/(t+3)$ vs.~(standard)
  agnostic step size $\gamma_t = 2/(t+2)$]
In essence, the proof below uses the
\emph{modified} function-agnostic step size rule
\(\gamma_{0} = 1\) and
\(\gamma_{t} = 2/(t+3)\) for \(t \geq 1\),
and the rate for line search and the short
step rule follows simply by dominating
the modified function-agnostic step size rule.
The modified function-agnostic rule is a shift of the standard
function agnostic rule \(\gamma_{t} = 2 / (t+2)\)
employed here solely for the purpose of constant factors;
the same proof can be done directly with the function agnostic rule
\(\gamma_{t} = 2 / (t+2)\) at the cost of slightly worse constant factors.

In fact, with the agnostic step size rule
$\gamma_t = 2/(t+2)$ in Algorithm~\ref{fw}, the rates in
Theorem~\ref{fw_sub} become $f(x_t) - f(x^*) \leq \frac{2 L D^{2}}{t + 2}$ and 
$\min_{0 \leq \tau \leq t} g_{\tau} < \frac{6.75 L D^{2}}{t + 2}$,
respectively.
As
demonstrated in \citet{PEP_Taylor_2015},
the best worst-case convergence rate can be
found by numerically solving a \emph{semidefinite program} (SDP), which for the
primal gap provides an improvement
\(h_{t} \leq c_{t} L D^{2} / (t+2)\) with \(2/3 \leq c_{t} \leq 1\)
for \(t \leq 100\), still under the function-agnostic step size rule
\(\gamma_{t} = 2 / (t+2)\).
In Section~\ref{sec:lowerbound} we will see an
example with \(h_{t} = \Omega(1/t)\), at least when \(t\) is smaller than the
dimension of the domain \(\mathcal{X}\), so that the rate here, up to constant
factors independent of the domain \(\mathcal{X}\), is optimal.  
\end{remark}

\begin{proof}[Proof of Theorem~\ref{fw_sub}]
The proof follows closely \citet{jaggi13fw},
and is a prototypical convergence proof.
The main estimation is based on the smoothness of \(f\)
and the choice of \(x_{t+1}\) to estimate the primal progress:
\begin{equation}
  \begin{split}
  f(x_{t+1}) - f(x_{t})
  &
  \leq \innp{\nabla f(x_t)}{x_{t+1}- x_t}
  + \frac{L}{2} \norm{x_{t+1}- x_t}^2
  \\
  &
  = \gamma_t \innp{\nabla f(x_t)}{v_t- x_t}
  + \frac{L\gamma_t^2}{2} \norm{v_t- x_t}^2
  \label{eq:FW-step-progress}
  \\
  &
  \leq
  \gamma_{t} \innp{\nabla f(x_t)}{x^{*}- x_t}
  + \gamma_{t}^{2} \frac{L D^{2} }{2}
  \\
  &
  \leq
  \gamma_{t} \bigl(f(x^{*}) - f(x_{t})\bigr)
  +  \gamma_{t}^{2} \frac{L D^{2}}{2}.
 \end{split}
\end{equation}
Here we used the minimality of \(v_{t}\) to
replace it with the minimum \(x^{*}\) of \(f\), so that we can
directly relate the primal progress $f(x_{t}) - f(x_{t+1})$ to the
dual gap $\innp{\nabla f(x_t)}{x_t - x^{*}}$ and
the primal gap $f(x_{t}) - f(x^{*})$.
For the quadratic term this substitution is not possible
(the desire of such a substitution has inspired the variant in
Section~\ref{sec:near-extreme-point}),
therefore we use a more conservative estimation \(\norm{v_{t} - x_{t}}
\leq D\).
The last inequality uses convexity to replace the gradient term with
function values,
effectively bounding the primal gap by the dual gap. By rearranging,
we obtain
\begin{equation}
  h_{t+1} \leq (1 - \gamma_{t}) h_{t} + \gamma_{t}^{2} \frac{LD^{2}}{2}.
\end{equation}

So far we have used no assumption on \(\gamma_{t}\).
The considered step size rules all minimize an intermediate bound
in the inequality chain:
line search minimizes the left-hand side \(f(x_{t+1}) - f(x_{t})\),
the short step rule minimizes
the second line in \eqref{eq:FW-step-progress}
(and its variant \(\min \innp{\nabla f(x_{t})}{x_{t}-v_{t}}
\mathbin{/} (L D^{2}), 1\}\) mentioned above minimizes the third line).
Thus the final inequality actually holds
for any number between \(0\) and \(1\) instead of \(\gamma_{t}\):
\begin{equation}
  \label{eq:classic-FW-primal}
  h_{t+1} \leq (1 - \gamma) h_{t} + \gamma^{2} \frac{LD^{2}}{2},
  \qquad \text{for all } 0 \leq \gamma \leq 1.
\end{equation}

The last inequality reduces to
\(h_{t+1} - h_{t} \leq - h_{t}^{2} / (2 L D^{2})\)
for \(h_{t} \leq L D^{2}\)
with optimal step size \(\gamma = h_{t} / (LD^{2})\).
As a heuristic for the optimal bound from this inequality,
we solve the continuous analogue with equality:
\(h'(t) = - h(t)^{2} / (2 L D^{2})\),
whose solution is
\(h(t) = 2 L D^{2} / (t + t_{0})\)
for a constant \(t_{0}\)
with step size \(2 / (t + t_{0})\) at point \(t\).

We will use induction to prove the claim
\(h_{t} \leq 2 L D^{2} / (t+3)\),
which is an analogue of the heuristic with \(t_{0} = 3\) chosen
for the initial bound on \(h_{1}\).
The initial case $t=1$ follows with the choice $\gamma = 1$.
Assuming that the bound in Equation~\eqref{eq:vanilla-rate} holds
for some $t$,
we will use \eqref{eq:classic-FW-primal} directly
to avoid solving a quadratic inequality.
We choose \(\gamma = 2/(t+3)\) via the above heuristic:
\begin{equation}
 \begin{split}
  h_{t+1}
  &
  \leq
  \left(
    1 - \frac{2}{t+3}
  \right)
  \frac{2 L D^{2}}{t+3}
  +
  \frac{4}{(t+3)^{2}}
  \cdot
  \frac{L D^2}{2}
  \\
  &
  = \frac{2 L D^{2} (t + 2)}{(t + 3)^{2}}
  \\
  &
  \leq \frac{2 L D^{2}}{t+4},
 \end{split}
\end{equation}
by \((t+2) (t+4) \leq (t+3)^{2}\).
Therefore, Equation~\eqref{eq:vanilla-rate} holds for $t+1$.

Next we will establish the convergence rate of the Frank–Wolfe gap, largely
following \citet{jaggi13fw}.
The main idea is that the Frank–Wolfe gap cannot
be large in too many consecutive iterations, as it would decrease the function
value below the optimal one.  The proof is based on Progress
Lemma~\ref{lemma:progress}, relating primal progress and the Frank–Wolfe gap:
\begin{equation}
  \label{eq:fw-dual-progress}
 \begin{split}
  h_{\tau} - h_{\tau+1}
  &
  =
  f(x_{\tau}) - f(x_{\tau+1})
  \\
  &
  \geq
  \frac{g_{\tau}}{2} \cdot \min
  \left\{
    \frac{g_{\tau}}{L  \norm{v_{\tau} - x_{\tau}}^{2}}, 1
  \right\}
  \\
  &
  \geq
  \frac{g_{\tau}^{2}}{2 L D^{2}}
  .
 \end{split}
\end{equation}
For the last inequality
recall that we have already proven \(h_{t} \leq L D^{2} / 2\),
and that \(g_{t} \leq h_{t} + L D^{2}\)
by Proposition~\ref{prop:dual-by-primal-smooth},
hence \(g_{t} \leq h_{t} + L D^{2} / 2 \leq LD^{2}\).
From this and \(\norm{v_{\tau} - x_{\tau}} \leq D\)
the last inequality of \eqref{eq:fw-dual-progress} follows.

Next we sum up the inequalities, omitting early iterations, where the inequality is likely loose. In other words,
we sum up Equation~\eqref{eq:fw-dual-progress}
for \(\tau=t_{0}, t_{0} + 1, \dotsc, t+1\)
where \(t_{0}\) will be specified later.
This provides the second inequality below,
preceded by the primal gap bound from Equation~\eqref{eq:vanilla-rate},
which is valid for \(t_{0} \geq 1\):
\begin{equation}
  \label{eq:6}
  \frac{2 L D^{2}}{t_{0} + 3}
  \geq
  h_{t_{0}} - h_{t + 1}
  \geq
  (t - t_{0} + 1)
  \cdot
  \frac{\min_{0 \leq \tau \leq t} g_{\tau}^{2}}{2 L D^{2}}
  .
\end{equation}
Rearranging provides
\begin{equation}
  \label{eq:7}
  \frac{2 L D^{2}}{\sqrt{(t_{0} + 3) (t - t_{0} + 1)}}
  \geq
  \min_{0 \leq \tau \leq t} g_{\tau}
  .
\end{equation}
Now the claim
\(\min_{0 \leq \tau \leq t} g_{\tau} \leq 4 L D^{2} / (t + 3)\)
easily follows by an appropriate choice of \(t_{0}\).
We choose \(t_{0}\) to roughly minimize the left-hand side.
For \(t = 1\) and \(t = 2\), we choose \(t_{0} = 1\),
and the claim readily follows.
For \(t \geq 3\) we choose \(t_{0} = \lceil t/2 \rceil - 1\)
which ensures \(t_{0} \geq 1\).
Now, by \(t_{0} + 3 \geq (t + 4) /2\) and
\(t - t_{0} + 1 \geq (t + 3) / 2\),
we have
\begin{equation}
  \label{eq:19}
  \min_{0 \leq \tau \leq t} g_{\tau}
  \leq
  \frac{4 L D^{2}}{\sqrt{(t + 4) (t + 3)}}
  < \frac{4 L D^{2}}{t + 3}
  ,
\end{equation}
as claimed.
\end{proof}

Actually, the proof for the convergence rate of Frank–Wolfe gap shows a little
more: the \emph{quadratic mean} (and hence also the \emph{average}) of the dual
gaps over the last half iterations converges also at the postulated rate.
Moreover, from the last line of Equation~\eqref{eq:fw-dual-progress} we obtain
\smash{$h_t - h_{t+1} \geq \frac{g_t^2}{2LD^2} \geq \frac{h_t^2}{2LD^2}$} and hence the
contraction $h_{t+1} \leq \bigl(1 - \frac{h_t}{2LD^2}\bigr) h_t$, i.e., the
primal progress is monotone, when employing the short step rule.

In addition to being simple, the vanilla Frank–Wolfe algorithm
(Algorithm~\ref{fw}) is also robust to the linear minimization oracle: an
oracle returning an approximate optimum, up to additive errors, can be used
with little overhead in convergence.  With minimal modifications to the
argument above, the simplest result was obtained in \citet{jaggi13fw} using a
fixed diminishing approximation error, which we recall at the end of this
section.  Later so-called lazy algorithms set the approximation error based on
progress, and further relax the requirements of the linear minimization oracle,
see Section~\ref{sec:lazification}.

We finish with two remarks helpful for later discussions and a convergence rate
for inexact linear minimization oracles.

\begin{remark}[Initial primal gap estimation]
\label{rem:initial-bound}
Suppose we consider an optimal solution \(x^* \in \interior \mathcal
X\) (i.e., in the interior of \(\mathcal{X}\)).
In this case we obtain a bound on the initial primal gap directly
from the smoothness Equation~\eqref{smooth}
\begin{equation}
f(x)-f(x^*) \leq \underbrace{\innp{\nabla f(x^*)}{x-x^*}}_{\geq
        0, \text{ by Remark~\ref{rem:fo}}} + \frac{L}{2}\norm{x-x^*}^2.
\end{equation}
As \(x^* \in \interior \mathcal X\), we have \(\nabla f(x^*) = 0\), so
that 
\begin{equation}
  f(x)-f(x^*) \leq \frac{L D^2}{2}.
\end{equation}
The above does not necessarily hold if \(x^*\) lies on
the boundary of \(\mathcal X\). However, as we have seen in the
proof of Theorem~\ref{fw_sub}, a single Frank–Wolfe step using
step size~\(1\), the short step rule, or line search
suffices to ensure such a bound:
i.e., \(h_{1} \leq L D^{2}\)
(via Equation~\eqref{eq:classic-FW-primal}).
In particular,
this also holds for the function agnostic rule as the first
step size is \(1\).
\end{remark}

\begin{remark}[Burn-in phase]\index{burn-in phase}
\label{rem:burn-in}
In the proof above the first iteration was special by establishing an initial
bound \(h_{1} \leq L D^{2} / 2\), after which the main estimation holds.  Such
early iterations where the algorithm might behave differently are sometimes
referred to as \emph{burn-in phase}.  Thus the burn-in phase for the vanilla
Frank–Wolfe algorithm consists of at most \(1\) iteration.  Note that in this
survey we define the burn-in phase separately for each algorithm in an ad-hoc
manner.

For some later Frank–Wolfe variants the burn-in phase can be longer but
typically we have linear convergence in this phase (in the above \(h_{t+1} \leq
h_{t} / 2\)), so that it usually only consists of a logarithmic number of
steps.
\end{remark}

Finally, as promised, we include a convergence rate for the Frank–Wolfe
algorithm with a linear minimization oracle returning only an approximate
optimal solution to a linear problem.  This is a small extension of the
argument above.  In the rest of the survey, we assume that the LMO is exact to
simplify exposition, although most results allow for inexact oracles.  The only
exceptions are the lazy algorithms in Section~\ref{sec:lazification},
especially designed for much more inexact oracles, than the ones discussed
here.

Note that the result here requires a diminishing additive error
\(\mathcal{O}(1/t)\) for linear minimization in iteration \(t\).

\begin{theorem}
\label{thm:fw-approx}
The Frank–Wolfe algorithm (Algorithm~\ref{fw})
allows
approximate implementation of the linear minimization oracle
(Line~\ref{fw_extract}): If the LMO returns approximate solutions
$v_0, \dotsc, v_{T-1}$ such that
\begin{equation}
\label{inexactlo}
\innp{\nabla f(x_t)}{v_t} \leq \min_{v\in\mathcal{X}}\innp{\nabla
f(x_t)}{v} + \frac{2 L D^{2}}{t + 3} \delta
\end{equation}
for all \(t\) where $\delta>0$, then the
Frank–Wolfe algorithm satisfies
\begin{equation}
\label{fw:rateApprox}
f(x_t)-f(x^*)\leq\frac{2LD^2}{t + 3}(1+\delta)
\end{equation}
for all \(t \geq 1\).
Equivalently, $f(x_t)-f(x^*)\leq \varepsilon$
for
\begin{equation}
t \geq \frac{2LD^2}{\varepsilon} (1+ \delta).
\end{equation}
For the dual convergence we have
\begin{equation}
\label{fw:dualrate-inexact}
\min_{0 \leq \tau \leq t} g_{\tau}
\leq \frac{6.75 L D^{2}}{t+2} (1 + \delta).
\end{equation}
\end{theorem}
\begin{proof}
The proof is almost identical to that of Theorem~\ref{fw_sub},
therefore we highlight only the differences here.
The main estimation is along the lines of Equation~\eqref{eq:FW-step-progress}
(substituting \(\gamma = 2 / (t+3)\) for simplicity)
\begin{equation}
\label{eq:fw_sub_step}
\begin{split}
f(x_{t+1}) - f(x_{t})
&
\leq
\frac{2}{t + 3} \innp{\nabla f(x_{t})}{v_{t} - x_{t} }
+
\left( \frac{2}{t + 3} \right)^{2}
\cdot \frac{L}{2} \norm{v_{t} - x_{t}}^{2}
\\
&
\leq
\frac{2}{t + 3}
\left(
  \innp{\nabla f(x_{t})}{x^{*} - x_{t} }
  +
  \frac{2}{t + 3} L D^{2} \delta
\right)
+
\left( \frac{2}{t + 3} \right)^{2}
\cdot \frac{L}{2} \norm{v_{t} - x_{t}}^{2}
\\
&
\leq
\frac{2}{t + 3} \bigl( f(x^{*}) - f(x_{t}) \bigr)
+
\left( \frac{2}{t + 3} \right)^{2}
\cdot \frac{L D^{2} (1 + \delta)}{2}
  ,
\end{split}
\end{equation}
where the first inequality follows from smoothness,
the second one uses approximate minimality of \(v_{t}\)
and the third one uses convexity and \(\norm{x_{t} - v_{t}} \leq D\).
The essential difference here is the extra term
\(2 L D^{2} \delta / (t+3)\) in the middle
arising from approximation.
The rest of the proof is completely analogous to that of
Theorem~\ref{fw_sub},
replacing \(L D^{2}\) by \(L D^{2} (1 + \delta)\).

The convergence rate of the Frank–Wolfe gap follows similarly,
but due to the fixed approximation rates it is more technical,
see~\citet[Theorem~2]{jaggi13fw}.
\end{proof}

\subsection{Lower bound on convergence rate}
\label{sec:lowerbound}

In the following we provide two lower bounds
on the convergence rate of the vanilla Frank–Wolfe algorithm,
of which the first one is a fundamental barrier
of linear programming based methods in general.

\subsubsection{Lower bound on convergence rate due to sparsity}
\label{sec:limit-sparse}

We will now consider an important example that provides natural lower
bounds on the convergence rate of any convex optimization method
accessing the feasible region only
through an LMO (linear minimization oracle).
The presented example, which also provides an
inherent sparsity~vs.~optimality tradeoff, reveals that
$\mathcal{O}(1/t)$ primal gap error after \(t\) LMO calls
of Theorem~\ref{fw_sub} cannot be
improved in general \citep[see][]{jaggi13fw,lan2013complexity}
\emph{without the use of parameters of the feasible region}
(besides the diameter).
However, in the example the feasible region depends on \(t\),
in particular, it does not claim anything on the convergence rate
after an initial number of iterations
\emph{depending on the feasible region}.
In fact, there exist algorithms with
even linear convergence rates, i.e., rates of the form
\(\mathcal{O}(e^{- \Omega(t)})\),
where the constant factor in the exponent involves
a small parameter of the feasible region,
as we will see later in, e.g.,
Sections~\ref{sec:line-conv-gener} and~\ref{sec:acceleration}.

\begin{example}[Primal Gap Lower Bound]
\label{example:lowerbound}
We provide an example where \(\Omega(L D^{2} / \varepsilon)\)
linear minimizations are needed to achieve a primal gap additive
error at most \(\varepsilon\) for an \(L\)-smooth convex function
over a feasible region with diameter \(D\) for
any positive numbers \(L\), \(D\) and \(\varepsilon\),
however the example depends even on \(\varepsilon\).
(The same example provides a lower bound
\(\Omega(G^{2} D^{2} / \varepsilon^{2})\)
for \(G\)-Lipschitz objective functions,
using the square root of the objective function in this example,
see \citet[Theorem~2]{lan2013complexity}.)
We consider the problem
\begin{equation}
\min_{x \in \conv{e_{1}, \dots, e_{n}}} \norm[2]{x}^{2},
\end{equation}
of minimizing the quadratic objective function
\(f(x) = \norm[2]{x}^{2}\)
over the \myindex{probability simplex} 
\smash{\(P = \Delta_{n} \defeq
  \conv{\setb{e_{1}, \dots, e_{n}}}$, where the $e_{i}\)}
are the coordinate vectors,
i.e., the vectors in the standard basis of $\R^{n}$.
The unique optimal solution to the problem is $x^* =
\allOne/n$, the point whose coordinates are all \(1/n\).
Starting the algorithm from any vertex of the
probability simplex, after $t < n$ LMO calls,
the only information available from the feasible region is
$t+1$ of the vertices \(e_{1}, \dotsc, e_{n}\).
Thus the only feasible points \(x_{t}\) the algorithm can produce
are convex combination of these points, therefore
\begin{equation}
  f(x_t) \geq \min_{\substack{x \in \conv{\mathcal{S}} \\
      \mathcal{S} \subseteq \setb{e_1, \dots, e_n} \\
      \size{\mathcal{S}} \leq t + 1}}
  f(x) = \frac{1}{t + 1},
\end{equation}
leading to the primal gap lower bound
\(f(x_{t}) - f(x^{*}) \geq 1 / (t+1) - 1/n\).
This implies that with a choice \(n \gg 1 / \varepsilon\)
one needs \(\Omega(1 / \varepsilon)\) linear minimization
to achieve a primal gap of at most \(\varepsilon\),
matching the rate in Theorem~\ref{fw_sub} up to a constant factor.
Here \(L = 2\) is the smoothness parameter of \(f\),
and \(D = \sqrt{2}\) is the diameter of \(\Delta_{n}\)
in the \(\ell_{2}\)-norm.
Note that the objective function is also \(2\)-strongly convex.
We leave it to the reader to scale the example
to achieve an \(\Omega(L D^{2} / \varepsilon)\) lower bound on linear
minimizations for arbitrary \(L\) and \(D\).

Moreover, this argument also provides an inherent
sparsity~vs.~optimality tradeoff. Here \emph{sparsity} refers to the
number of vertices used to write \(x_t\) as a convex combination and the
example shows that if we seek an approximate solution with sparsity \(t\) the
primal gap can be as large as $f(x_t)-f(x^*)\geq 1 / (t+1) - 1/n$. 
\end{example}

This example shows that no improvement in LMO calls
is expected, which is independent of additional problem parameters.
However, with mild additional assumptions, late iterates (i.e., those after
some problem-dependent number of initial iterates) do converge with a rate
dependent only on the minimal face
containing the optimal solution,
and the objective function in a neighborhood of the face
\citep[see][]{garber2020sparseFW}.
This mild assumption roughly states that the objective function grows fast
away from the minimal face containing \(x^{*}\).
See Section~\ref{sec:linConvInterior} for the special case
when \(x^{*}\) is an interior point of \(P\),
where the distance of \(x^{*}\) to the boundary of \(P\)
appears in the convergence rate.
In particular, the vanilla Frank–Wolfe algorithm
(with the short step rule or line search)
for Example~\ref{example:lowerbound}
converges in a finite number of steps:
one has \(x_{t} = (e_{1} + \dots + e_{t+1}) / (t + 1)\)
and Frank–Wolfe vertices \(v_{t} = e_{t+1}\)
for \(0 \leq t \leq n-1\), and hence \(x_{n-1} = x^{*}\).

In Figure~\ref{fig:lowerbound} we depict the primal gap convergence and
the lower bound \(1 / (t+1) - 1/n\) from Example~\ref{example:lowerbound} in
$\mathbb{R}^n$ with $n = 1000$ and the function-agnostic step size rule
for the vanilla Frank–Wolfe algorithm.

\begin{figure}
  \centering
  \includegraphics[width=.4\linewidth, alt={Primal gap under the
    agnostic step size rule, essentially agreeing with the lower bound
    until the lower bound starts sharply
    dropping to \(0\).}]{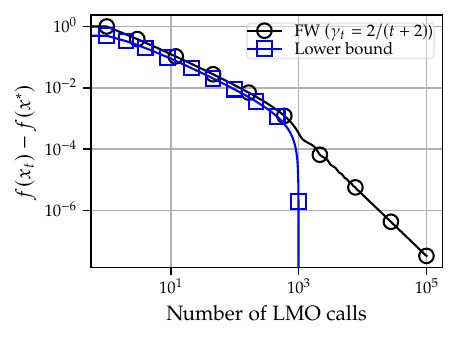}

\caption{Minimizing the function $f(x) = \norm[2]{x}^2$ over the
  \myindex{probability simplex} for dimension $n = 1000$.
  Primal gap convergence for the
  vanilla Frank–Wolfe algorithm (FW, Algorithm~\ref{fw}) using the agnostic
  step size rule $\gamma_t = 2/(t+2)$ in \textcolor{black}{black}.  The lower
  bound on the primal gap from Example~\ref{example:lowerbound} is shown in
  \textcolor{blue}{blue}.  The lower bound applies to any algorithm that
  accesses the feasible region $\mathcal{X}$ (in this case the probability
simplex) only through linear minimization.}
\label{fig:lowerbound}
\end{figure}

\subsubsection{Zigzagging -- Lower bound on convergence rate due to moving towards vertices}
\label{sec:limit-move-vertex}

The example from Section~\ref{sec:limit-sparse}
has the drawback that convergence is slow
only in initial iterations of the vanilla
Frank–Wolfe algorithm (Algorithm~\ref{fw}).
Here we complement it with
the so called \emph{zigzagging} phenomenon intrinsic to the vanilla
Frank–Wolfe algorithm, which
is not only observed in practice, but also justified by a lower bound
on convergence rate under mild assumptions, slightly worse than the
\(\mathcal{O}(1 / \varepsilon)\) convergence rate from Theorem~\ref{fw_sub}.  The
prototypical example is a polytope as feasible region with optimum
lying on a face of the polytope, see Figure~\ref{fig:zigzag} for a
simple example.

\begin{figure}
\centering
\includegraphics[width=0.45\linewidth, alt={Trajectory of Frank–Wolfe
  algorithm when zigzagging: moving alternately in the direction of
  the vertices of the side of a triangle domain which contains
  the optimum.}]{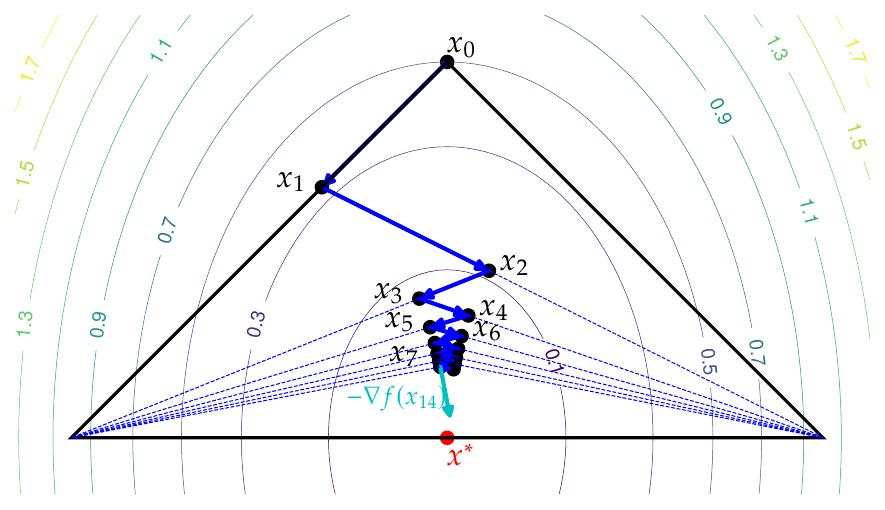}

\caption{Zigzagging of the Frank–Wolfe algorithm
(Algorithm~\ref{fw}),
i.e., back-and-forth sideways movements that 
result in slow progress. Here the
algorithm uses line search.
The feasible region is the triangle
$P =\conv{\{(-1,0),(1,0),(0,1)\}}$,
the objective function is \(f(x,y) = 2 x^{2} + y^{2}\).
The starting vertex is $x_{0} = (0,1)$,
and the optimum is $x^{*} = (0,0)$.}
\label{fig:zigzag}
\end{figure}

The informal description of zigzagging is the following.
When the optimum lies well inside a face, as the algorithm approaches the face,
there are no vertices available
in the approximate direction to the optimum,
so the Frank–Wolfe algorithm has to move in progressively worse
directions.
In order to not overshoot and worsen the function
value the algorithm has to compensate with progressively smaller step
sizes and frequent changes of direction,
which is commonly called \emph{zigzagging},
so that the average movement of
the iterates points towards the optimum.

Another slightly informal way to think about the lower bound is as
follows: Suppose the optimal solution \(x^*\) lies on a face and we
have picked up an off-face vertex early on. Then we need to \lq{}wash
out\rq{} this off-face vertex from the convex combinations that we
form in order to reach the optimal face. 

The lower bound below is based on \citet{Canon_FWbound68}, which essentially assumes a
strongly convex objective function (it is a slight improvement to
\citet[§\,7]{wolfe70}, which assumes a quadratic objective function).

In our formulation we have several assumptions, which we now justify.
The condition that the optimum \(x^{*}\) is an interior point of a
face \(F\) and that a late iterate lie outside \(F\) are necessary
preconditions to start zigzagging.
If late iterates all lie in \(F\) then the algorithm converges
linearly, as we will see in Section~\ref{sec:linConvInterior}.
The technical upper bound on the step size
at a high-level disallows escaping from the zigzagging behaviour
by moving to a vertex of \(F\).

\begin{theorem}
\label{thm:FW-slow}
Let \(\mathcal{X} = P\) be a polytope,
\(f\) an \(L\)-smooth convex function
over \(P\),
with minimum set \(\Omega^{*}\)
in the relative interior of a (at least \(1\)-dimensional) face \(F\) of \(P\).
If the Frank–Wolfe algorithm
uses a step size rule satisfying for some constant \(\nu\)
\begin{align}
  \label{eq:15}
  \gamma_{t}
  &
  \leq
  \nu \cdot
  \frac{\innp{\nabla f(x_{t})}{x_{t} - v_{t}}}
  {\norm{x_{t} - v_{t}}^{2}},
  \\[.5ex]
  \label{eq:FW-slow-progress}
  f(x_{t}) - f(x_{t+1})
  &
  \geq
  \frac{\innp{\nabla f(x_{t})}{x_{t} - v_{t}}^{2}}
  {2 L \norm{x_{t} - v_{t}}^{2}},
\end{align}
and
reaches an iterate \(x_{t_{0}}\),
which is not on the face \(F\)
but \(f(x_{t_{0}}) < \min_{v \in \vertex{F}} f(v)\),
then
for all \(\delta > 0\)
we have
\(f(x_{t}) - f(x^{*}) \geq L D^{2} \mathbin{/}
(t \log^{2 + \delta} t)\)
for infinitely many \(t\).
\end{theorem}

Note that the step size upper bound obviously holds for
the short step rule with \(\nu = 1/L\).
It also holds for a \(\mu\)-strongly convex objective function \(f\) and
line search,
as using monotonicity of \(h_{t}\) we obtain
\begin{equation*}
 \begin{split}
  0
  & 
  \leq
  f(x_{t}) - f(x_{t+1})
  \leq
  \innp{\nabla f(x_{t})}{x_{t} - x_{t+1}}
  -
  \frac{\mu \norm{x_{t} - x_{t+1}}^{2}}{2}
  \\
  &
  =
  \gamma_{t} \innp{\nabla f(x_{t})}{x_{t} - v_{t}}
  -
  \gamma_{t}^{2} \frac{\mu \norm{x_{t} - v_{t}}^{2}}{2}
  ,
 \end{split}
\end{equation*}
which by rearranging proves Equation~\eqref{eq:15}
with \(\nu = 2 / \mu\).

As a comparison and warm-up we provide a simple lower bound in Proposition~\ref{prop:FW-fixed-lower} illustrating that
small step sizes and the fact that we take convex combinations naturally induce
a lower bound.
In particular, for function-agnostic step size rules, such as,
e.g., \(\gamma_{t} = 2/(t+2)\), it manifests in an explicit lower bound on
convergence rate under mild assumptions.  Compared to the theorem from above,
the lower bound is worse, but it is optimal up to a constant factor under
additional mild assumptions
by \citet[Proposition~2]{Bach2020RichardsonML}
for the rules \(\gamma_{t} = 1/(t+1)\) and \(\gamma_{t} = 2/(t+2)\).
This also means that the rule \(\gamma_{t} = 2/(t+2)\)
is necessarily less greedy in progress compared to the short step rule
and line search and that, in particular, Theorem~\ref{thm:FW-slow}
does not generalize to the function-agnostic step size rule
\(\gamma_{t} = 2/(t+2)\); see also \citet{wirth2022acceleration} for situations where agnostic step size rules achieve convergence rates faster than $\mathcal{O}(1/t)$.  

\begin{proposition}
  \label{prop:FW-fixed-lower}
  Let \(f \colon \mathcal{X} \to \mathbb{R}\) be a convex function
  over a convex set \(\mathcal{X}\).
  For every \(T \geq 1\) there exists \(y_{T} \in \mathcal{X}\)
  such that the iterate \(x_{T}\)
  of the Frank–Wolfe algorithm with any step sizes \(\gamma_{t}\)
  satisfy
  \begin{align}
    \label{eq:FW-fixed-lower-point}
    x_{T}
    &
    =
    \prod_{t=1}^{T-1} (1 - \gamma_{t}) \cdot x_{1}
    +
    \left(
      1 - \prod_{t=1}^{T-1} (1 - \gamma_{t})
    \right)
    y_{T}, \\
    \label{eq:FW-fixed-lower-value}
    f(x_{T}) - f(x^{*})
    &
    \geq
    \prod_{t=1}^{T-1} (1 - \gamma_{t})
    \cdot
    \innp{\nabla f(x^{*})}{x_{1} - x^{*}}.
  \end{align}
  In particular,
  for the step size rule \(\gamma_{t} = 2 / (t+2)\)
  we have
  \begin{equation}
    \label{eq:FW-fixed-lower-example}
    f(x_{T}) - f(x^{*})
    \geq
    \frac{2}{T (T+1)}
    \cdot
    \innp{\nabla f(x^{*})}{x_{1} - x^{*}}
    \qquad T \geq 1
    .
  \end{equation}
\end{proposition}

Obviously, in Equations~\eqref{eq:FW-fixed-lower-point} and
\eqref{eq:FW-fixed-lower-value} the indices of the products might start
from \(t=0\).  The choice \(t=1\) is convenient for
the last claim only.

\begin{proof}
Equation~\eqref{eq:FW-fixed-lower-point}
is a simplified variant of the following rule,
which follows by an easy induction on \(T\)
using the update rule \(x_{T+1} = x_{T} + \gamma_{T} (v_{T} - x_{T})\):
\begin{equation}
  \label{eq:FW-combined-point}
  x_{T} = \prod_{t=1}^{T-1} (1 - \gamma_{t}) x_{1}
  + \sum_{t=1}^{T-1} \gamma_{t}
  \prod_{\tau = t+1}^{T-1} (1 - \gamma_{\tau}) v_{t}
  .
\end{equation}
Equation~\eqref{eq:FW-fixed-lower-point}
arises by defining \(y_{T}\) as a convex combination of the \(v_{t}\)
for \(T > 1\)
(the initial point \(y_{1} \in \mathcal{X}\) can be arbitrary):
\begin{equation*}
  y_{T} \defeq
  \frac{\sum_{t=1}^{T-1} \gamma_{t}
    \prod_{\tau = t+1}^{T-1} (1 - \gamma_{\tau}) v_{t}}%
  {1 - \prod_{t=1}^{T-1} (1 - \gamma_{t})}
  .
\end{equation*}
Now, Equation~\eqref{eq:FW-fixed-lower-value} easily follows from
convexity and Equation~\eqref{eq:FW-fixed-lower-point}:
\begin{equation*}
 \begin{split}
  f(x_{T}) - f(x^{*})
  &
  \geq
  \innp{\nabla f(x^{*})}{x_{T} - x^{*}}
  \\
  &
  =
  \prod_{t=1}^{T-1} (1 - \gamma_{t})
  \cdot
  \innp{\nabla f(x^{*})}{x_{1} - x^{*}}
  +
  \left(
    1 - \prod_{t=1}^{T-1} (1 - \gamma_{t})
  \right)
  \innp{\nabla f(x^{*})}{y_{T} - x^{*}}
  \\
  &
  \geq
  \prod_{t=1}^{T-1} (1 - \gamma_{t})
  \cdot
  \innp{\nabla f(x^{*})}{x_{1} - x^{*}}
  .
 \end{split}
\end{equation*}
Equation~\eqref{eq:FW-fixed-lower-example}
is a special case of Equation~\eqref{eq:FW-fixed-lower-value}.
\end{proof}

We also need a calculus result on convergent sequences
for the lower bound from \citet[Lemma~1]{Canon_FWbound68}.

\begin{lemma}
  \label{lem:divergence-slows-square}
  Let \(\gamma_{t}\) be a sequence of positive numbers
  and \(C, \delta > 0\) such that
  \begin{equation}
    \label{eq:12}
    \sum_{i \geq t} \gamma_{i}^{2}
    \leq \frac{C}{t \log^{2 + \delta} t}
    \qquad \text{for } t \geq t_{0}
    .
  \end{equation}
  Then the series \(\sum_{i=0}^{\infty} \gamma_{i}\) is convergent.
\begin{proof}
This is a consequence of the generalized mean inequality between the
arithmetic and quadratic mean:
for \(t \geq t_{0}\) we have
\begin{equation}
  \sum_{i=t+1}^{2t} \gamma_{i}
  \leq
  \sqrt{t \sum_{i=t+1}^{2t} \gamma_{i}^{2}}
  \leq
  \sqrt{t \cdot \frac{C}{t \log^{2 + \delta} t}}
  =
  \frac{\sqrt{C}}{ \log^{1 + \delta/2} t}
  .
\end{equation}
Summing up we obtain the claim:
\begin{equation*}
  \sum_{i = t_{0} + 1}^{\infty} \gamma_{i}
  =
  \sum_{k=0}^{\infty} \sum_{i = 2^{k} t_{0} + 1}^{2^{k+1} t_{0}} \gamma_{i}
  \leq
  \sum_{k=0}^{\infty}
  \frac{\sqrt{C}}{(k \log 2)^{1 + \delta/2}}
  <
  \infty
  .
  \qedhere
\end{equation*}
\end{proof}
\end{lemma}

We are ready to prove Theorem~\ref{thm:FW-slow}. 

\begin{proof}[Proof of Theorem~\ref{thm:FW-slow}]
The following equation and inequality are the two main claims on the convergence rate of
the step sizes \(\gamma_{t}\) chosen by the algorithm,
provided that the algorithm actually converges,
i.e., \(\lim_{t \to \infty} h_{t} = 0\). We will prove that 
\begin{alignat}{2}
  \label{eq:5}
  \sum_{t=0}^{\infty} \gamma_{t} &= + \infty,
  \\
  \label{eq:FW-slow-steps-small}
  \sum_{t \geq T} \gamma_{t}^{2}
  &\leq
  \frac{2 L \nu^{2}}{c^{2}} \cdot
  h_{T}
  & \quad\text{for all \(T\) large enough},
\end{alignat}
where the distance \(c\) will be chosen explicitly later;
it will be roughly
the minimal distance between \(\Omega^{*}\) and any vertex of \(P\).

For Equation~\eqref{eq:5},
our first claim is that \(x_{t} \notin F\) for all \(t \geq t_{0}\),
which we prove by induction on \(t\).
The start case \(t = t_{0}\) holds by assumption.
For the induction step, assume \(x_{t} \notin F\).
By the update rule
\(x_{t+1} = (1 - \gamma_{t}) x_{t} + \gamma_{t} v_{t}\),
we clearly have \(x_{t+1} \notin F\) unless \(\gamma_{t} = 1\) and
\(x_{t+1} = v_{t}\).
In the latter case, \(x_{t+1}\) is a vertex, and
\(f(x_{t+1}) \leq f(x_{t_{0}}) < \min_{v \in \vertex{F}} f(v)\)
by monotonicity and assumption.
Thus \(x_{t+1} \notin \vertex{F}\),
i.e., \(x_{t+1}\) is a vertex of \(P\) but not of \(F\),
and hence \(x_{t+1} \notin F\).

Now as \(\Omega^{*}\) does not contain a vertex,
and the algorithm converges,
there is a~\(t_{1}\) such that for all \(t \geq t_{1}\)
we have \(f(x_{t}) < \min_{v \in \vertex{P}} f(v)\).
By an argument similar to above, we have
that \(x_{t}\) is not a vertex and \(\gamma_{t} < 1\)
for all \(t \geq t_{1}\).

Now we employ an argument analogous to
Proposition~\ref{prop:FW-fixed-lower}.
Let \(x^{*} \in \Omega^{*}\) be any minimal point of \(f\).
As \(F\) is a face, it is of the form
\(F = \{ x \in P \mid \innp{a}{x} = b\}\)
where \(\innp{a}{x} \geq b\) holds for all \(x \in P\).
From the update rule
\(x_{t+1} = (1 - \gamma_{t}) x_{t} + \gamma_{t} v_{t}\),
as \(\innp{a}{v_{t}} \geq b = \innp{a}{x^{*}}\),
we obtain
\(\innp{a}{x_{t+1} - x^{*}} \geq
(1 - \gamma_{t}) \innp{a}{x_{t} - x^{*}}\),
and therefore
\(\innp{a}{x_{t} - x^{*}} \geq
\innp{a}{x_{t_{1}} - x^{*}} \prod_{i=t_{1}}^{t-1} (1 - \gamma_{i})\)
for \(t \geq t_{1}\).
Now, as \(\lim_{t \to \infty} h_{t} = 0\),
it follows \(\lim_{t \to \infty} \distance{x_{t}}{\Omega^{*}} = 0\),
and hence \(\lim_{t \to \infty} \innp{a}{x_{t} - x^{*}} = 0\).
As \(\innp{a}{x_{t_{1}} - x^{*}} > 0\),
we conclude \(\prod_{i = t_{0}}^{\infty} (1 - \gamma_{i}) = 0\),
and hence
\(\sum_{i=t_{0}}^{\infty} \gamma_{i} = + \infty\) as claimed,
since all the \(\gamma_{i} < 1\), as shown above.

To prove Inequality~\eqref{eq:FW-slow-steps-small},
the key is the assumed upper bound
in Equation~\eqref{eq:15} on the step size.
Now, as \(x_{t}\) is not a vertex for \(t \geq t_{1}\),
and \(\distance{x_{t}}{\Omega^{*}} \to 0\), where \(\Omega^{*}\) does
not contain a vertex either,
there is a positive number \(c\) such that \(\norm{x_{t} - v} \geq c\)
for \(t \geq t_{1}\) and every vertex \(v\) of \(P\).
In particular,  \(\norm{x_{t} - v_{t}} \geq c\)
for \(t \geq t_{1}\).

We combine Equation~\eqref{eq:15} with Lemma~\ref{lemma:progress},
using that \(f(x_{t}) < f(v_{t})\) and hence
\begin{equation}
  \label{eq:13}
  h_{t} - h_{t+1}
  =
  f(x_{t}) - f(x_{t+1})
  \geq
  \frac{\innp{\nabla f(x)}{x_{t} - v_{t}}^{2}}
  {2 L \norm{x_{t} - v_{t}}^{2}}
  \geq
  \frac{\norm{x_{t} - v_{t}}^{2}} {2 L \nu^{2}}
  \cdot \gamma_{t}^{2}
  \geq
  \frac{c^{2}} {2 L \nu^{2}}
  \cdot \gamma_{t}^{2}
  .
\end{equation}
For any \(T \geq t_{1}\),
summing this up for all \(t \geq T\) proves the claim
\begin{equation}
  \label{eq:16}
  \sum_{t \geq T} \gamma_{t}^{2}
  \leq
  \frac{2 L \nu^{2}}{c^{2}}
  h_{T}
  .
\end{equation}

The theorem now follows directly from Equations~\eqref{eq:5}
and~\eqref{eq:FW-slow-steps-small} via
Lemma~\ref{lem:divergence-slows-square} applied to the sequence
$\gamma_t$.
\end{proof}

\subsubsection{First-order oracle complexity}
\label{sec:oracle-complexity}

Finally, we recall lower bounds for the number of first-order oracle
calls, which hold in general for any algorithm,
as long as the oracle is the only access to the objective function.
The results are similar to Example~\ref{example:lowerbound}, namely,
the worst-case problem may depend on the primal gap error,
are summarized in Table~\ref{tab:FOO-complexity}.
The lower bounds presented here
have been established in \citet{nemirovsky1983problem} and \citet{nesterov04}.

\begin{table}[t]
\caption{Lower bound on worst-case number of
first-order oracle calls
to minimize a convex objective function
over a feasible region of diameter \(D\)
up to a primal gap of at most \(\varepsilon > 0\).
Both the objective function and feasible region may depend on
\(\varepsilon\).
\citep[Theorems~2.1.7, 2.1.13, 3.2.1 and~3.2.5]{nesterov18}}  

\small
\centering
\renewcommand{\arraystretch}{1.2}
\begin{tabular*}{\linewidth}{@{\extracolsep{\fill}}ll@{}}
\toprule
Function & Minimum FOO calls \\
\midrule
\(L\)-smooth & \(\Omega\bigl(\sqrt{L D^{2} / \varepsilon}\bigr)\) \\
\(L\)-smooth function, \(\mu\)-strongly convex &
\(\Omega\bigl(\sqrt{L / \mu} \log (\mu D^{2} / \varepsilon)\bigr)\) \\
\(G\)-Lipschitz, \(n\)-dimensional domain
& \(\Omega\bigl(\min\{(G D / \varepsilon)^{2}, n \log
(G D / \varepsilon)\}\bigr)\) \\
\(G\)-Lipschitz, \(\mu\)-strongly convex &
\(\Omega\bigl(G^{2} / (\mu \varepsilon)\bigr)\) \\
\bottomrule
\end{tabular*}

\label{tab:FOO-complexity}
\end{table}

\subsection{Nonconvex objectives}
\label{sec:nonconvex}

In this section, we consider the vanilla Frank–Wolfe algorithm
with a non-convex objective function \(f\),
as studied in \citet{lj16nonconvex}.
Convergence is measured in the Frank–Wolfe gap
\smash{$g(x) \defeq \max_{v\in\mathcal{X}}\innp{\nabla f(x)}{x - v}$},
see Definition~\ref{FrankWolfeGap},
but it is a local property unrelated to distance
to \emph{global minimum} in general.
In practice, algorithms tend to converge to a \emph{local minimum}.
Often the objective function is convex in a neighborhood of local
minima,
so that the dual gap for late iterations is likely
an upper bound to the difference in function value to
the local minimum the algorithm converges.
The difference from the convex case is
that the objective function can have multiple local minima besides
the global minimum, however, the Frank–Wolfe gap is always \(0\)
at all local minima.

We shall use the short step rule as usual,
but note that via an analogous but technically slightly more involved
argument,
the function-agnostic step size \(\gamma_{t} = 1 / \sqrt{t + 1}\)
provides a similar bound.

\begin{theorem}
  \label{thm:fw-nonconvex}
  Let
  \(f \colon \mathcal{X} \to \mathbb{R}\)
  \index{convergence for non-convex objective}
  be an \(L\)-smooth but not necessarily convex function over a compact
  convex set
  \(\mathcal{X}\) with
  diameter \(D\).
  The Frank–Wolfe gap of
  the vanilla Frank–Wolfe algorithm
  with line search and the short step rule
  converges as follows
  \begin{equation}
    \min_{0 \leq \tau \leq t} g_{\tau}
    \leq\frac{\max \{2 h_{0}, L D^{2}\}}{\sqrt{t+1}}
  \end{equation}
for \(t \geq 1\).
  In other words,
  the algorithm finds a solution with Frank–Wolfe gap smaller than
  \(\varepsilon > 0\) with at most
  \((\max\{h_{0}, L D^{2}\} / \varepsilon)^{2}\)
  linear optimizations and gradient computations.
\end{theorem}

\begin{proof}
The proof is similar to that of Theorem~\ref{fw_sub},
however, we do not have the luxury of a primal gap bound.
Nevertheless up until
the first inequality of Equation~\eqref{eq:fw-dual-progress}
the proof is the same
(requiring only smoothness but not convexity of \(f\)),
which is our starting point:
\begin{equation}
  f(x_{\tau}) - f(x_{\tau+1})
  \geq
  \frac{g_{\tau}}{2} \cdot \min
  \left\{
    \frac{g_{\tau}}{L  \norm{v_{\tau} - x_{\tau}}^{2}}, 1
  \right\}
  \geq
  \frac{g_{\tau}}{2} \cdot \min
  \left\{
    \frac{g_{\tau}}{L D^{2}}, 1
  \right\}
  .
\end{equation}
Let \(G_{t} \defeq \min_{0 \leq \tau \leq t} g_{\tau}\).
Summing up the above inequality for \(\tau=0, 1, \dotsc, t\)
we obtain
\begin{equation}
  h_{0} \geq  f(x_{0}) - f(x_{t+1})
  \geq
  (t + 1) \min\left\{
    \frac{G_{t}^{2}}{2LD^{2}},
    \frac{G_{t}}{2}
  \right\}.
\end{equation}
We conclude
\begin{equation*}
  G_{t}
  \leq
  \max \left\{
    \sqrt{\frac{2h_{0} L D^{2}}{t + 1}},
    \frac{2 h_{0}}{t + 1}
  \right\} \leq
  \frac{\max\{2h_{0}, L D^{2}\}}{\sqrt{t + 1}}
  .
  \qedhere
\end{equation*}
\end{proof}

\subsection{Notes}
\label{sec:basic-notes}

Here we briefly mention some variations of the problem setup from the
literature, not necessary to understand the rest of the survey.

\subsubsection{Affine invariance and norm independence}
\label{sec:affine-invariance}

Most Frank–Wolfe-style algorithms are affine invariant, i.e.,
invariant under translations and linear transformations,
like rescaling or changing basis for the coordinates.
They are also often independent of the norm,
which some papers implicitly include in affine invariance,
as the norm is the \(\ell_{2}\)-norm.
The major norm-dependent part of the Frank–Wolfe algorithm is the
short step size rule, while
line search and the function-agnostic step size rule
\(\gamma_t = \frac{2}{t+2}\) are independent of any norm.

As such there has been recent interest
in removing also the dependency on external choices such as norms.
For example, the norm-dependent term \(LD^{2}\) can be replaced with
a norm-independent constant \(C\) called curvature,
which is defined to satisfy
\begin{equation}
  \label{eq:curvature}
  f\bigl((1 - \gamma) y + \gamma x\bigr) \leq
  f(y) + \gamma \innp{\nabla f(y)}{x - y} + \frac{C \gamma^{2}}{2},
  \quad x,y \in \mathcal{X}, \ \ 0 \leq \gamma \leq 1.
\end{equation}
While such invariant formulations are important from a
theoretical perspective and have a certain beauty and purity,
the invariant parameters we are aware of are properties of the
domain \(\mathcal{X}\) and the objective function \(f\)
\emph{together},
and hence both estimation and efficient use
of the parameters are quite challenging.
Thus, estimated invariant parameters
often degrade performance in practice, as observed, e.g., in
\citet{pedregosa2018step}. This can also be seen when using the above
notion of curvature to estimate primal progress from smoothness. In
the affine-variant case we have \(f(x_{t}) - f(x_{t+1}) \geq
\innp{\nabla f(x_{t})}{x_t - v_{t}}^2 \mathbin{/}
(2 L \norm{x_t - v_{t}}^{2})\), whereas in the
affine-invariant case we have \(f(x_{t}) - f(x_{t+1}) \geq
\innp{\nabla f(x_{t})}{x_t - v_{t}}^2 \mathbin{/}
(2C)\). Note that the former progress guarantee might
improve when \(x_t\) and \(v_t\) are close (e.g., when the
feasible region is strongly convex), whereas the
latter is constant.

Affine invariance should not be confused with invariance under extended
formulations, i.e., that
given an affine surjection \(\pi \colon Q \to P\)
the approximate solutions \(x\) to \(\min_{x\in P} f(x)\) and
\(y\) to \(\min_{y \in Q} f(\pi(y))\) produced by the same algorithm
satisfy \(x = \pi(y)\).
Affine invariance is the special case when \(\pi\) is a bijection.
For non-injective \(\pi\) in general,
invariance is rather an exception than the rule,
e.g., the vanilla Frank–Wolfe algorithm is invariant under extended
formulation with the function-agnostic step size rule or line search.

\subsubsection{Self-concordant objective functions}
\label{sec:self-concordant}

As Frank–Wolfe algorithms optimize linear functions as a subroutine,
they usually require a bounded feasible region.
However, see \citet{SelfConcordantFW_2020} for a variant
optimizing self-concordant functions \(f\) over an unbounded feasible
region with compact level sets \(\{x \in \dom f : f(x) \leq t\}\)
for all real number \(t\)
(and where \(f\)  might take the value \(+ \infty\)). The algorithm internally restricts
to an initial level set for some \(t = f(x_{0})\). More recently, in
\citet{carderera2021simple} it was shown that a very simple
modification of the original Frank–Wolfe algorithm is sufficient to
handle (generalized) self-concordant functions. See also
\citet{fwpoisson16} for earlier work regarding the Frank–Wolfe
algorithm for a special class of non-smooth objectives arising in the
context of the Poisson phase retrieval problem; their approach is
subsumed by the aforementioned one for (generalized) self-concordant
functions though.

\section{Improved convergence rates}
\label{sec:improved-convergence}

In this section, we will present improved convergence rates for
the vanilla Frank–Wolfe algorithm (Algorithm~\ref{fw})
and some of its core variants
in important special cases, the best rate being the
\emph{linear convergence rate}
(called geometric convergence rate in old literature):
\(\mathcal{O}(\log (1/\varepsilon))\)
resources (linear optimizations, function evaluations, gradient
computations) are sufficient to obtain a solution
with an additive error at most \(\varepsilon\)
in primal gap. These special cases include problems in which
the optimum lies in the relative interior of the feasible region
\citep{gm86, beck2004conditional},
see Section~\ref{sec:linConvInterior}, 
or when the feasible region is uniformly or
strongly convex \citep{polyak66cg}, see 
Section~\ref{sec:impr-conv-strongly}. 
Table~\ref{tab:fw-faster} summarizes the improved convergence rates
for the vanilla Frank–Wolfe algorithm
under additional assumptions.

\begin{table}[t]
\caption{Convergence rate of the vanilla Frank–Wolfe algorithm
   (Algorithm~\ref{fw}) under additional assumptions
   for a smooth objective function \(f\)
   over a compact convex set \(\mathcal{X}\).
   (See Definition~\ref{def:grad-dominate} for the gradient dominated
   property.)
   As elsewhere, $\Omega^*=\argmin_{\mathcal{X}}f$ is the set of
   minimizers of $f$.
   The last column contains the upper bound on primal gap after
   \(t\) linear minimizations, where the constant factors
   include parameters of both the feasible region \(\mathcal{X}\)
   and objective function \(f\).
   Uniform convexity instead of strong convexity also leads
   to faster than \(\mathcal{O}(1/t)\) rates.}
\label{tab:fw-faster}

\setlength{\tabcolsep}{2pt}

\small
\begin{tabular*}{\linewidth}{@{\extracolsep{\fill}}ccccl@{}}
\toprule
\multicolumn{4}{c}{Additional assumptions} &Primal gap\\
\cmidrule(lr){1-4}
$\mathcal{X}$ strongly convex&$f$ gradient
dominated&$\Omega^*\cap\relint{\mathcal{X}}\neq\varnothing$
&$\dualnorm{\nabla f} \geq c > 0$\\
\midrule
\textcolor{gray!60}{\ding{55}}&\textcolor{gray!60}{\ding{55}}&\textcolor{gray!60}{\ding{55}}&\textcolor{gray!60}{\ding{55}}&$\mathcal{O}(1/t)$\\
\textcolor{gray!60}{\ding{55}}&\ding{51}&\ding{51}&\textcolor{gray!60}{\ding{55}}&$\exp(- \Omega(t))$\\
\ding{51}&\textcolor{gray!60}{\ding{55}}&\textcolor{gray!60}{\ding{55}}&\ding{51}
&$\exp(- \Omega(c t))$\\
\ding{51}&\ding{51}&\textcolor{gray!60}{\ding{55}}&\textcolor{gray!60}{\ding{55}}&$\mathcal{O}(1/t^2)$\\
\bottomrule
\end{tabular*}
 
\end{table}

As we shall see, for Frank–Wolfe algorithms the asymptotically better
rates involve additional problem parameters, some of which are hard to
estimate, like distance of the optimal solution to the boundary of the
feasible region or strong convexity parameters. It is also important
to note that in several cases the linear rates that we obtain contain
dimension-dependent factors and in \citet{garber2020sparseFW} it was
shown that this is unavoidable in the worst-case. This is in contrast
to projection-based methods that do not necessarily (and usually do
not) involve dimension-dependent terms in the rate. However, these
projection-based methods solve a harder problem in each iteration, the
projection problem, whose complexity might again depend on the
dimension. In this context, the Conditional Gradient Sliding
algorithm (Algorithm~\ref{cgs-smooth}) \citep{lan2016conditional} is
very insightful, where the projection problem itself is solved with
Frank-Wolfe algorithms; for LMO-based algorithms optimal rates and
trade-offs are obtained.

While early results assumed strong convexity of the objective
function, in many cases a weaker property of being gradient dominated
suffices, see Lemma~\ref{lem:SCprimal} relating the two properties.
Here we introduce only the version closest to strong convexity,
see Remark~\ref{rem:gradientDom} for the general definition.
\begin{definition} \index{gradient dominated function}
  \label{def:grad-dominate}
  A differentiable function \(f\) is \emph{\(c\)-gradient dominated}
  for a constant \(c > 0\)
  if for all \(x\) and \(y\) in its domain
  \begin{equation}
    \label{eq:grad-dominate}
      \frac{f(x) - f(y)}{c^{2}} \leq \dualnorm{\nabla f(x)}^{2}
      .
  \end{equation}
\end{definition}

Whenever the objective function is smooth and gradient dominated,
we expect to obtain \emph{linear convergence},
which roughly states that the primal optimality gap contracts as
\(h(x_{t}) \leq (1-r)^{t} h(x_0)\) for some constant \(0 < r < 1\). 
This is because being gradient dominated ensures
large gradients and hence large Frank–Wolfe gaps,
so that the objective function~\(f\) is likely rapidly decreasing
everywhere towards the optimum, making it easier to find a good
descent direction.
These results also justify the assumptions on
the lower bounds in Section~\ref{sec:lowerbound} by
showing better convergence when the assumptions are violated.

Note that in ``linear convergence'' linear refers to
per-iteration contraction
(like comparing \(\norm{x_{t+1} - x^{*}}\) with
\(\norm{x_{t} - x^{*}}\))
rather than the overall convergence
(magnitude of \(\norm{x_{t} - x^{*}}\)).
Recall that a sequence \(x_{t}\)
converges linearly to \(x^{*}\)
if
\(\norm{x_{t+1} - x^{*}} \leq (1 - \delta)
\norm{x_{t} - x^{*}}\)
for some \(\delta > 0\)
and all \(t \geq 0\).
As a consequence
\(\norm{x_{t} - x^{*}} = \mathcal{O}((1 - \delta)^{t})\).
In a similar vein, \(x_{t}\)
converges \(p\)-adically
for some \(p > 1\)
if \(x_{t}\)
converges linearly to \(x^{*}\)
and \(\norm{x_{t+1} - x^{*}} \leq C \norm{x_{t} - x^{*}}^{p}\)
for some \(C > 0\)
and all \(t \geq 0\)
(here an upper bound on \(C\) is not necessary
for fast convergence).
Convergence slower than linear convergence, e.g.,
\(\norm{x_{t+1} - x^{*}} \leq (1 - \frac{1}{t}) \norm{x_{t} - x^{*}}\)
is referred to as ``sublinear convergence''.

As we will see, for some optimization algorithms the linear
convergence update rule
\(\norm{x_{t+1} - x^{*}} \leq (1 - \delta) \norm{x_{t} - x^{*}}\)
holds for sufficiently many iterations to ensure
\(\norm{x_{t} - x^{*}} = \mathcal{O}((1 - \delta)^{t})\),
but not necessarily for
all iterations.  Hence while occasionally violating the linear
convergence update rule, this is not considered to significantly alter
the convergence rate, therefore they are still called ``linear
convergent'' in the literature, at least informally.

\subsection{Linear convergence: a template}
\label{sec:templ-line-con}

Linear convergence\index{linear convergence} is usually realised via
\(h_{t} - h_{t+1} = f(x_{t}) - f(x_{t+1}) = \Omega(h_{t})\),
a per iteration progress proportional to the primal gap,
and strong convexity ensures exactly that.
Recall from Progress Lemma~\ref{lemma:progress}
the per iteration progress
(assuming the step size is small enough
for the iterate to remain in the feasible region
for simplicity of discussion):
\begin{equation}
  \label{eq:primal-progress-simple}
  f(x_t) - f(x_{t+1}) \geq
  \frac{\innp{\nabla f(x_t)}{d_{t}}^2}{2L \norm{d_{t}}^{2}}.
\end{equation}
Assuming an appropriate choice of \(d_{t}\), the right-hand side
is roughly quadratic in the dual gap, which is lower bounded by the
primal gap for linear convergence via the following scaling
inequality, the primary use of strong convexity:
\begin{lemma}[Primal gap bound via strong convexity]
  \label{lem:SCprimal} Let \(f\) be a \(\mu\)-strongly convex
  function over a convex set \(\mathcal{X}\) with minimum \(x^{*}\),
  then it is also \(1/\sqrt{2 \mu}\)-gradient dominated,
  more precisely,
  \begin{equation}
    \label{eq:SCprimal}
    f(x) - f(x^*) \leq \frac{\innp{\nabla f(x)}{x - x^*}^2}{2\mu
    \norm{x - x^{*}}^{2}} \leq \frac{\dualnorm{\nabla f(x)}^{2}}{2\mu}.
  \end{equation}
  \begin{proof}
By definition of strong convexity (Definition~\ref{strconvex}),
between points $y = x + \eta(x^* - x)$ and $x$
for $0 \leq \eta \leq 1$ provides
    \begin{equation}
    f\bigl(x + \eta(x^* - x)\bigr) - f(x) \geq \eta \innp{\nabla f(x)}{x^* - x} + \eta^2 \frac{\mu}{2} \norm{x^* - x}^2.
    \end{equation}
We lower bound the right-hand side by minimizing it in \(\eta\)
disregarding the restriction \(0 \leq \eta \leq 1\)
    \begin{equation}
    f(x) -  f\bigl(x + \eta(x^* - x)\bigr) \leq \frac{\innp{\nabla f(x)}{x^* - x}^2}{2\mu \norm{x^* - x}^2}.
    \end{equation}
Now choosing $\eta = 1$ leads to
    \begin{equation*}
      f(x) -  f(x^{*})
      \leq
      \frac{\innp{\nabla f(x)}{x^{*} - x}^{2}}{2\mu
        \norm{x^{*} - x}^{2}}
      \leq
      \frac{\dualnorm{\nabla f(x)}^{2}}{2\mu}.
    \end{equation*}
    Finally, as \(x^{*}\) is a minimum, obviously
    \(f(x) - f(y) \leq f(x) - f(x^{*}) \leq
    \dualnorm{\nabla f(x)}^{2} / (2\mu)\)
    for all \(x, y \in \mathcal{X}\),
    hence \(f\) is gradient dominated, too, as claimed.
  \end{proof}
\end{lemma}

Thus the following is our linear convergence template.
Let \(d_t\) denote the negative of the direction
taken by the algorithm.
To obtain explicit convergence rates,
we assume the following \emph{scaling condition} \index{scaling condition}
for some constant \(\alpha > 0\) for all \(t \geq 0\):
\begin{equation}
  \label{eq:scalingCondition}
  \alpha \frac{\innp{\nabla f(x)}{x - x^*}}{\norm{x - x^*}} \leq \frac{\innp{\nabla f(x)}{d_t}}{\norm{d_t}},
\end{equation}
which relates the direction $d_t$ to the idealized direction $x-x^*$ and will allow us to directly relate the progress from the direction $d_t$ with that of the idealized direction~$x-x^*$ that points towards to the optimum $x^*$. 
The scaling condition is usually ensured by some additional structure either of the objective function or the feasible region (or their combination in some cases) as we will see in what follows. Assuming the scaling condition we have the progress estimate:
\begin{equation*}
 \begin{split}
f(x_t) - f(x_{t+1}) & \geq \frac{\innp{\nabla
                      f(x)}{d_t}^2}{2L  \norm{d_t}^2} \\
  & \geq \alpha^2 \frac{\innp{\nabla
                      f(x)}{x - x^*}^2}{2L  \norm{x - x^*}^2} \\
  & \geq \alpha^2 \frac{\mu}{L} h(x_t), 
 \end{split}
\end{equation*}
where the first inequality is by Lemma~\ref{lemma:progress} as before,
the second inequality is the scaling condition, and the third
inequality is obtained from Lemma~\ref{lem:SCprimal}. Equivalently, we
obtain linear convergence through the contraction
\[
  h(x_{t+1}) \leq \left(1 - \alpha^2 \frac{\mu}{L} \right) h(x_{t})
  .
\]

\begin{example}[Gradient descent algorithm]
For illustration, we recall the usual linear convergence proof of
the gradient descent algorithm.
The gradient descent algorithm is intended for unconstrained
problems,
when \(f\) is minimized over the whole \(\mathbb{R}^{n}\),
and uses the Euclidean norm \(\norm[2]{\cdot}\).

For an \(L\)-smooth convex \(f\),
the gradient descent algorithm makes steps maximizing
the right-hand side in Equation~\eqref{eq:primal-progress-simple},
i.e., \(x_{t+1} = x_t - \frac{1}{L} \nabla f(x_t)\) with progress
\(f(x_{t}) - f(x_{t+1}) \geq \norm[2]{\nabla f(x_{t})}^{2} / (2L)\).
When \(f\) is additionally \(\mu\)-strongly convex,
using the maximality of \(d_{t} = \nabla f(x_{t})\) instead
implies Equation~\eqref{eq:scalingCondition} with \(\alpha = 1\),
and hence a linear convergence rate
\[
  f(x_{t+1}) - f(x^{*}) \leq
  \left(1-\frac{\mu}{L}\right) \bigl(f(x_{t}) - f(x^{*})\bigr)
  .
\]

\begin{remark}[Optimal linear rate]
The worst-case linear convergence rate for the gradient descent
algorithm under an arbitrary step size rule is
\begin{equation*}
  f(x_{t+1}) - f(x^{*}) \leq
  \left(
    \frac{L - \mu}{L + \mu}
  \right)^{2}
  \bigl(f(x_{t}) - f(x^{*})\bigr)
\end{equation*}
which is achieved with either line search or
\(x_{t+1} = x_{t} - \frac{2}{L + \mu} \nabla f(x_{t})\),
see \citet[Theorem~4.2]{GD-opt_Taylor2017}.
\end{remark}
\end{example}

\subsection{Inner optima}
\label{sec:linConvInterior}

\enlargethispage{1\baselineskip}

One of the first, if not \emph{the} first linear convergence result for the
Frank–Wolfe algorithm is due to \citet[§\,8]{wolfe70} and \citet{gm86}, who showed
that whenever the optimal solution \(x^*\) is contained in the interior of
\(\mathcal{X}\), denoted by \( \interior(\mathcal{X}) \), then the (vanilla)
Frank–Wolfe algorithm (employing either line search or the short step rule)
converges linearly.
Compared to the lower bound in Theorem~\ref{thm:FW-slow},
this is the case (assuming a polytope feasible region \(P\))
when the assumption is violated
that the optimum solution set \(\Omega^{*}\) is contained
in the relative interior of a face.
(The assumption can also be violated if a vertex of \(P\) is
an optimal solution, however in this case the vanilla Frank–Wolfe algorithm
with line search
converges in finitely many steps,
as some optimal vertex
necessarily occurs as a Frank–Wolfe vertex.)
However, we do not assume a polytope domain here,
neither do we assume a unique optimal solution.

The original reasoning implicitly contains the scaling
condition, which we make explicit here.
Recall that \(B(x, r)
\defeq \{ z : \norm{z - x} \leq r\}\) is the ball of radius \(r\) around \(x\).

\begin{proposition}[Scaling condition when \(x^* \in
  \interior(\mathcal{X})\)] \label{prop:scaleGM} Let \(\mathcal{X} \subseteq \R^n\) be a compact convex set of
  diameter \(D\) and \(f\) be a convex function. If there
  exists \(r > 0\) so that \(B(x^*, r) \subseteq \mathcal{X}\)
  for a minimizer \(x^{*}\) of \(f\),
  then for all
  \(x \in \mathcal{X}\) we have
  \[
    \innp{\nabla f(x)}{x - v} \geq r \dualnorm{\nabla f(x)} \geq
    \frac{r \innp{\nabla f(x)}{x - x^{*}}}{\norm{x - x^{*}}},
\]
where $v = \argmax_{u \in \mathcal{X}} \innp{\nabla f(x)}{x - u}$.
\begin{proof}
To take advantage of the ball \(B(x^*, r)\) being in \(\mathcal{X}\),
we compare the Frank–Wolfe vertex \(v\) with the point of the ball
with minimal product with \(\nabla f(x)\), instead of comparing
with \(x^{*}\).
I.e., we consider \(x^{*} - r z\),
where \(z\) is a point with \(\norm{z} = 1\)
and \(\innp{\nabla f(x)}{z} = \dualnorm{\nabla f(x)}\).
Therefore,
\begin{equation*}
  \innp{\nabla f(x)}{v} \leq
  \innp{\nabla f(x)}{x^{*} - r z}
  =
  \innp{\nabla f(x)}{x^{*}} - r \dualnorm{\nabla f(x)}
  .
\end{equation*}
By rearranging and using \(\innp{\nabla f(x)}{x - x^{*}} \geq 0\)
we obtain the claim
\begin{equation*}
  \innp{\nabla f(x)}{x - v} \geq
  \innp{\nabla f(x)}{x - x^{*}}
  + r \dualnorm{\nabla f(x)}
  \geq
  r \dualnorm{\nabla f(x)}
  .
  \qedhere
\end{equation*}
\end{proof}
\end{proposition}

Having established the scaling condition,
we immediately obtain the following theorem.
Similar results apply to sharp objective functions \(f\)
\citep[see][]{kerdreux2018restarting}, which we will
discuss in Section~\ref{sec:adaptive_rates}.
See also \citet[Proposition~8]{gutman2020condition}
for using generalizations of smoothness and sharpness.
The following theorem also has a straightforward generalization when the domain
\(\mathcal{X}\) is not full dimensional, i.e.,
the optimum \(x^{*}\) lies in the \emph{relative} interior of \(\mathcal{X}\); essentially we restrict to the affine subspace spanned by
\(\mathcal{X}\), including the norm of the gradient
in the definition of gradient dominance.

\begin{theorem}
  \label{fw_linInt}
  Let \(\mathcal{X}\) be a compact convex set
  with diameter \(D\)
  and \(f\) an \(L\)-smooth and \(c\)-gradient dominated function.
  Assume further there exists a minimizer
  \(x^* \in \interior(\mathcal{X})\) of \(f\)
  in the interior of \(\mathcal{X}\),
  i.e., there exists an \(r > 0\)
  with \(B(x^*, r) \subseteq \mathcal{X}\).
  Then the (vanilla) Frank–Wolfe algorithm's
  iterates \(x_{t}\) satisfy
  \begin{equation*}
    f(x_t) - f(x^*) \leq
    \left(
      1 - \frac{r^{2}}{2 L c^{2} D^{2}}
    \right)^{t - 1}
    \frac{L D^{2}}{2}
  \end{equation*}
  for all \(t \geq 1\).
  Equivalently,
  the primal gap is at most \(\varepsilon > 0\)
  after the following number of
  linear optimizations and gradient computations:
  \begin{equation}
    \label{fw:rateLinInt}
    1 +
    \frac{2 L c^{2} D^{2}}{r^{2}} \ln \frac{L D^{2}}{2 \varepsilon}
    .
  \end{equation}
  In particular, if \(f\) is \(\mu\)-strongly convex
  (and hence \(c = 1/\sqrt{2 \mu}\)-gradient dominated)
  then for a primal gap at most \(\varepsilon > 0\),
  at most the following number of
  linear optimizations and gradient computations are needed:
  \begin{equation}
    1 +
    \frac{L D^{2}}{\mu r^{2}} \ln \frac{L D^{2}}{2 \varepsilon}
    .
  \end{equation}

\begin{proof}
We prove only the first claim,
since the second claim clearly follows from it, as usual.
We apply Proposition~\ref{prop:scaleGM}
with \(x = x_{t}\) and the Frank–Wolfe vertex \(v = v_{t}\)
and repeat the argumentation from above using
Lemma~\ref{lemma:progress} (as in Remark~\ref{rem:progress}).
\begin{equation*}
 \begin{split}
  h_{t} - h_{t+1}
  =
  f(x_{t}) - f(x_{t+1})
  &
  \geq
  \frac{\innp{\nabla f(x_{t})}{x_{t} - v_{t}}}{2}
  \min\left\{
    1,
    \frac{\innp{\nabla f(x_{t})}{x_{t} - v_{t}}}{L
      \norm{x_{t} - v_{t}}^{2}}
  \right\}
  \\
  &
  \geq
  \frac{\innp{\nabla f(x_{t})}{x_{t} - v_{t}}^{2}}{2L
    D^{2}}
  \geq
  \frac{r^{2} \dualnorm{\nabla f(x_{t})}^{2}}{2 L D^{2}}
  \geq
  \frac{r^{2}}{2 L c^{2} D^{2}} h_{t}
  .
 \end{split}
\end{equation*}
For the second inequality,
note that \(\nabla f(x^{*}) = 0\) as \(x^{*}\) is a minimum of \(f\)
that is an inner point, hence
\(\innp{\nabla f(x_{t})}{x_{t} - v_{t}}
= \innp{\nabla f(x_{t}) - \nabla f(x^{*})}{x_{t} - v_{t}}
\leq L \norm{x_{t} - x^{*}} \norm{x_{t} - v_{t}} \leq L D^{2}\),
which we have used to lower bound the minimum.

By rearranging we obtain
\[h_{t+1} \leq h_{t} \left(1 - \frac{r^{2}}{2 L c^{2} D^{2}}
  \right)
\]
for all $t \geq 0$.
Now the claim follows by an obvious induction on \(t\).
The initial bound \(h_{1} \leq L D^{2} / 2\)
holds generally for the Frank–Wolfe algorithm
(even without assuming strong convexity or anything about the position
on \(x^{*}\)), see Theorem~\ref{fw_sub}.
\end{proof}
\end{theorem}

\subsection{Strongly convex and uniformly convex sets}
\label{sec:impr-conv-strongly}

In this section, we will violate the assumption of the lower bound
in Theorem~\ref{thm:FW-slow} that there are no extreme points near the
optimum \(x^{*}\), when it lies on the boundary of the feasible
region.  Therefore we assume that every boundary point is
extreme in a strong sense.

Specifically,
we will consider the case where the feasible
region is uniformly convex which includes the strongly convex case.
Intuitively uniform convex sets have a curved, ball-like boundary, leaving little room for zigzagging compared to the flat boundary of polytopes. 
In fact for uniformly convex domains
we obtain improved convergence rates
if either
\begin{enumerate*}[after=., itemjoin={{, }}, itemjoin*={{, or }}]
\item \label{it:gradBoundedBelow}the gradient of \(f\)
  is bounded away from \(0\)
  (typically~\(f\) is defined on a neighborhood of
  \(\mathcal{X}\), where its \emph{unconstrained} optimum lies
  \emph{outside} of~\(\mathcal{X}\))
\item \label{it:fIsSc}\(f\) is gradient dominated (e.g., strongly convex)
\end{enumerate*}

In the particular case of strongly convex sets, we obtain even a linear rate in
case~\ref{it:gradBoundedBelow} and a quadratic improvement in case~\ref{it:fIsSc}.
It is an open question, whether a
Frank–Wolfe style method has linear rate for strongly convex functions
on strongly convex sets.
In contrast, if the feasible region is a polytope,
strong convexity (or gradient dominance) of the objective function is
sufficient for linear rates for some Frank–Wolfe algorithm variants,
as we shall see in Section~\ref{sec:line-conv-gener}.

Finally, \citet{UniformConvexFW_2020} also provide improved convergence for sharp
objective functions and uniformly convex domains, which we omit to simplify the
exposition. We will discuss sharpness in Section~\ref{sec:adaptive_rates},
which generalizes strong convexity and often provides similar improved
convergence rates.
See also \citet{garber_spectrahedron}
for similar results over the spectrahedron,
which is not uniformly convex.

We recall the notion of uniform convexity, taken from
\citet{UniformConvexFW_2020}:

\begin{definition}[\((\alpha, q)\)-uniformly convex set]\index{uniformly convex set}
	\label{def:stronglyUCset}
  Let \(\alpha\) and \(q\) be positive numbers.
  The set $\mathcal{X} \subseteq \R^{n}$ is
  \emph{$(\alpha, q)$-uniformly convex}
  with respect to the norm $\norm{\cdot}$
  if for any $x, y \in \mathcal{X}$,
  $\gamma \in [0,1]$, and $z \in \R^n$ with $\norm{z} \leq 1$
  the following holds:
  \begin{equation}
    \label{strcvxset}
    y + \gamma (x - y)
    + \gamma (1 - \gamma) \cdot \alpha \norm{x - y}^{q} z
    \in \mathcal{X}.
  \end{equation}
\end{definition}
Examples of uniformly convex sets include the \(\ell_{p}\)-balls,
which are \(((p-1)/2, 2 )\)-uniformly convex for $1 < p \leq 2$
and \((1/p, p)\)-uniformly convex for $p > 2$, with respect to $\norm{\cdot}_p$.
An \emph{\(\alpha\)-strongly convex} set
is an \((\alpha/2, 2)\)-uniformly convex set.
In particular, the \(\ell_{p}\)-ball is
\((p-1)\)-strongly
convex for \(1 < p \leq 2\) with respect to $\norm{\cdot}_p$.

In the special case where \(\norm{\cdot} = \norm[2]{\cdot}\),
there are various equivalent characterisations of
a closed set being \(\alpha\)-strongly convex,
illustrated in Figure~\ref{fig:strongSCSet}:
\begin{figure}[t]
\centering
\includegraphics[width=0.8\linewidth, alt={(1) intersection of \(3\)
  circles as an example of a strongly convex set;
  (2) circle as an example of a strongly convex set,
  contained in another circle, the two circles sharing a tangent;
  (definition) strong convexity of a circle:
  the circle should contain the (red) circle around the midpoint
  of a segment by definition, among others.}]{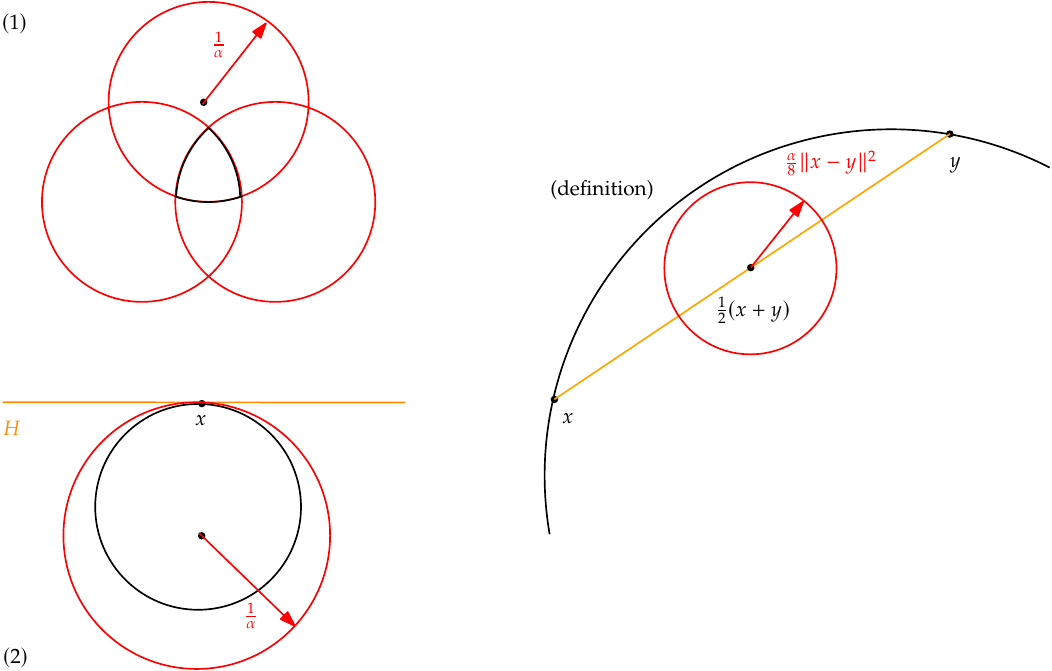}
\caption[Equivalent definitions of an \(\alpha\)-strongly
convex set $\mathcal{X}$]{Equivalent definitions of an \(\alpha\)-strongly
convex set $\mathcal{X}$.
Right:
Definition~\ref{def:stronglyUCset}, a generalization of
convexity, illustrated here for the midpoint: the red ball should
belong to the set \(\mathcal{X}\).
(Recall the difference by a factor of \(2\)
compared to Definition~\ref{def:stronglyUCset}.)
Left: (1) Definition as intersection of balls of radius \(1 / \alpha\).
(2) The set $\mathcal{X}$ contained in
the ball of radius \(1 / \alpha\) touching
a tangent hyperplane \(H\) at any boundary point \(x\).}
	\label{fig:strongSCSet}
\end{figure}%
\begin{enumerate*}[after=., itemjoin={{, }}, itemjoin*={{, or }}]
\item the set being the intersection of balls of radius \(1/\alpha\)
\item given any boundary point \(x\) and a supporting hyperplane \(H\) at
  \(x\), the set \(\mathcal{X}\) is contained in the ball of radius
  \(1/\alpha\) going through \(x\) and tangent to a supporting hyperplane at
  \(x\), and lying on the same side of \(H\) as \(\mathcal{X}\)
\end{enumerate*}
See \citet{Weber2012stronglyConvex} for an overview,
or \citet[Theorem~1]{Vial_StrongConvex82}.
In particular, an \(\alpha\)-strongly convex compact set has diameter
at most \(2 / \alpha\).

The latter geometric condition in its algebraic form is
a modified scaling condition as found in \citet{garber2015faster}
(for any norm);
see also \citet{polyak66cg, dunn1979rates, demyanov1970approximate}
for the result in Theorem~\ref{fw_strcvxset} for the strongly convex
case.  We first present this algebraic form, which is the key
to all convergence proofs for uniformly convex sets.

\begin{proposition}[Modified scaling condition for uniformly convex
  sets]\label{prop:scalescb} 
Let $\mathcal{X}$ be a full dimensional
compact
$(\alpha, q)$-uniformly convex set, \(\psi\) any non-zero vector
and \(v = \argmin_{y \in \mathcal{X}} \innp{\psi}{y}\).
Then for all \(x \in \mathcal{X}\)
\begin{equation*}
\frac{\innp{\psi}{x-v}}{\norm{x-v}^{q}}
\geq
\alpha \dualnorm{\psi}
.
\end{equation*}
\begin{proof}
Let
\(z\) be a unit vector (\(\norm{z} = 1\)) with
\(\innp{\psi}{z} = - \dualnorm{\psi}\) and
for some \(0 < \gamma < 1\)
\begin{equation}
  \label{eq:3}
  m \defeq x + \gamma (v - x)
  + \gamma (1-\gamma) \cdot \alpha \norm{x-v}^{q} z.
\end{equation}
By uniform convexity of $\mathcal{X}$, we have $m\in\mathcal{X}$.
Thus, by minimality of $v$,
\begin{equation*}
 \begin{split}
  \innp{\psi}{x - v} & \geq \innp{\psi}{x - m} \\
  & =
  \gamma \innp{\psi}{x - v}
  - \gamma (1 - \gamma) \cdot \alpha \norm{x - v}^{q}
  \innp{\psi}{z} \\
  & =
  \gamma \innp{\psi}{x - v}
  + \gamma (1 - \gamma) \cdot \alpha \norm{x - v}^{q}
  \dualnorm{\psi}.
 \end{split}
\end{equation*}
Rearranging and dividing by \(1 - \gamma\) provides
\begin{equation*}
  \innp{\psi}{x - v}
  \geq
  \gamma \cdot \alpha \norm{x - v}^{q}
  \dualnorm{\psi}.
\end{equation*}
Taking the limit as \(\gamma\) tends to \(1\) completes the proof.
\end{proof}
\end{proposition}

The intention is to apply the proposition for gradients
\(\psi = \nabla f(x)\) of a convex objective function \(f\),
to obtain a scaling condition like the one shown in Equation~\eqref{eq:scalingCondition},
similarly to Proposition~\ref{prop:scaleGM} for
$x^*\in \interior(\mathcal{X})$.
There is an important difference however: the
(modified) scaling condition has
a different power of \(\norm{x - v}\)
in the denominator, which affects the proof and also the achievable
rates.  

We present the convergence results from \citet{UniformConvexFW_2020} for the
case of a uniformly convex feasible region subsuming previous results for the
case of a strongly convex feasible region from \citet{garber2015faster}
\citep[see also][]{polyak66cg, dunn1979rates, demyanov1970approximate}. Note that the
results in \citet{UniformConvexFW_2020} hold more generally allowing to combine
uniformly convex feasible regions with functions satisfying sharpness (also
known as the \emph{Hölder Error Bound condition}), a condition that captures
how fast \(f\) is curving around its optima and that subsumes strong convexity
of \(f\). 

We first consider the case where the gradient of the objective
function \(f\) is bounded away from \(0\), which typically arises
when \(f\) extends to a neighborhood of \(\mathcal{X}\),
and its minima in the neighborhood lie outside of~\(\mathcal{X}\).

\begin{theorem}
\label{fw_strcvxset}
Let $\mathcal{X}$ be a compact,
\((\alpha, q)\)-uniformly convex set with \(q \geq 2\)
and \(f\) an
\(L\)-smooth and convex function with gradient bounded away from
\(0\): i.e.,
$\dualnorm{\nabla f(x)} \geq c >0$ for all $x \in \mathcal{X}$.
Let \(D\) be the diameter of \(\mathcal{X}\).
Then running the
(vanilla) Frank–Wolfe algorithm
from a starting point
$x_0 = \argmin_{v \in \mathcal{X}} \innp{\nabla f(x)}{v}$ with
$x \in \mathcal{X}$ ensures that:
\begin{equation*}
  f(x_{t}) - f(x^{*}) \leq
  \begin{cases}
    \max \left\{ \frac{1}{2}, 1 - \frac{ \alpha c}{2 L} \right\}^{t-1}
    \frac{L D^{2}}{2}
    & q=2, \\[1ex]
    \frac{L D^{2}}{2^{t}}
    & q>2, 1 \leq t \leq t_{0}, \\[1ex]
    \frac{L (L / \alpha c)^{2/(q-2)}}
    {\bigl(1 + (1/2 - 1/q) (t - t_{0})\bigr)^{q / (q-2)}}
    =
    \mathcal{O}(1 / t^{q/(q-2)})
    & q>2, t \geq t_{0}
  \end{cases}
\end{equation*}
for all \(t \geq 1\) where
\begin{equation*}
  t_{0} \defeq \max
  \left\{
    \left\lfloor
      \log_{2} \left(
        \frac{D^{2}}{(L / \alpha c)^{2/(q-2)}}
      \right)
    \right\rfloor + 1,
    1
  \right\}
  .
\end{equation*}
Equivalently, the primal gap is at most \(\varepsilon > 0\)
after the following number of
linear optimizations and gradient computations:
\begin{equation}
   \label{fw:rateLinstrcvxset}
   \begin{cases}
     1 + 2
     \max \left\{ 1, \frac{L}{\alpha c} \right\}
     \cdot
     \ln \frac{L D^{2}}{2\varepsilon}
     & q=2, \\[1ex]
     t_{0}
     + \frac{q}{(q - 2) \cdot L \cdot (\alpha c)^{2/q} \varepsilon^{1 - 2/q}}
     - \frac{2 q}{q - 2}
     = \mathcal{O}
     \left(
       \frac{1}{\varepsilon^{1 - 2/q}}
     \right)
     & q>2,\ \varepsilon \leq L / (2^{q / (q-2)} (\alpha c)^{2 / (q-2)})
     .
   \end{cases}
\end{equation}
\end{theorem}

Here and in the following a recurring scheme will be to turn a
contraction (via a single step progress) into a convergence rate. To
this end the following lemma is very helpful for
conditional gradient algorithms where we typically have an initial
burn-in phase with often linear convergence and then the asymptotic
rate dominates; various special cases of such and similar lemmas have
appeared in numerous cases before
\citep[e.g.,][]{temlyakov2011greedy,garber2015faster,nguyen2017greedy,xu18heb}.

\begin{lemma}[From contractions to convergence rates]
  \label{lem:ConvergenceRate}
  Let \(\{h_{t}\}_{t}\) be a sequence of positive numbers
  and \(c_{0}, c_{1}, c_{2}, \alpha\) be positive numbers
  with \(c_{1} < 1\)
  such that $ h_{1} \leq c_{0}$ and
  $h_{t} - h_{t+1}\geq h_{t} \min\{c_{1}, c_{2} h_{t}^{\alpha}\}$
  for $t \geq 1$, then
  \begin{equation}
    \label{eq:ConvergenceRate}
    h_{t} \leq
    \begin{cases}
      c_{0} (1 - c_{1})^{t-1} & 1 \leq t \leq t_{0}, \\[1ex]
      \frac{(c_{1} / c_{2})^{1 / \alpha}}
      {\bigl(1 + c_{1} \alpha (t - t_{0})\bigr)^{1 / \alpha}}
      = \mathcal{O}(1/t^{1 / \alpha}) & t \geq t_{0},
    \end{cases}
  \end{equation}
  where
  \begin{equation}
    \label{eq:ConvergenceRateCutoff}
    t_{0} \defeq
    \max
    \left\{
      \left\lfloor
        \log_{1 - c_{1}}
        \left(
          \frac{(c_{1} / c_{2})^{1 / \alpha}}{c_{0}}
        \right)
      \right\rfloor
      + 2,
      1
    \right\}
    .
  \end{equation}
  In particular, we have \(h_{t} \leq \varepsilon\) if
  \begin{equation}
    t
    \geq
    t_{0}
    +
    \frac{1}{\alpha c_{2} \varepsilon^{\alpha}}
    -
    \frac{1}{\alpha c_{1}}
    \quad
    \text{and}
    \quad
    \varepsilon \leq (c_{1} / c_{2})^{1 / \alpha}
    .
  \end{equation}
\end{lemma}

It is instructive to break up the minimization in the progress condition
\(h_{t} - h_{t+1} \geq h_{t} \min \{c_{1}, c_{2} h_{t}^{\alpha}\}\),
which would also help presentation of the proof.

\begin{lemma}
\label{lem:ConvergenceRateLinear}
Let \(\{h_{t}\}_{t}\) be a sequence of nonnegative numbers
and \(c_{0}, c_{1}\) be positive numbers with \(c_{1} < 1\)
such that $ h_{1} \leq c_{0}$
and $ h_{t} - h_{t+1}\geq h_{t} c_{1}$ for $t \geq 1$, then
for all \(t\geq 1\)
\begin{equation}
\label{eq:ConvergenceRateLinear}
h_{t} \leq c_{0} (1 - c_{1})^{t-1}.
\end{equation}
\begin{proof}
Obvious by induction, as \(h_{1} \leq c_{0}\)
and \(h_{t+1} \leq (1 - c_{1}) h_{t}\) by assumption.
\end{proof}
\end{lemma}

\begin{lemma}
  \label{lem:ConvergenceRatePower}
  Let \(\{h_{t}\}_{t}\) be a sequence of positive numbers and
  \(c_{0}, c_{2}, \alpha\) be positive numbers such that $ h_{1} \leq
  c_{0}$ and $ h_{t} - h_{t+1}\geq c_{2} h_{t}^{1 + \alpha}$
  for $t \geq 1$, then for all \(t \geq 1\)
  \begin{equation}
    \label{eq:ConvergenceRatePower}
    h_{t}
    \leq
    \frac{c_{0}}{\bigl(1 +
      c_{2} c_{0}^{\alpha} \alpha (t - 1) \bigr)^{1 / \alpha}}
    .
  \end{equation}
\begin{proof}
As \(h_{t} - h_{t+1}\geq c_{2} h_{t}^{1 + \alpha} > 0\),
we have \(h_{t} > h_{t+1}\) by assumption.  Using convexity of
the function \(x \mapsto 1 / x^{\alpha}\)
and \(h_{t} - h_{t+1}\geq c_{2} h_{t}^{1 + \alpha}\) again,
we obtain
\begin{equation}
  \frac{1}{h_{t}^{\alpha}} - \frac{1}{h_{t+1}^{\alpha}}
  \leq
  - \frac{\alpha}{h_{t}^{\alpha + 1}} \cdot
  (h_{t} - h_{t+1})
  \leq
  - c_{2} \alpha
  ,
\end{equation}
Summing up we get a telescoping sum on the left, hence
\begin{equation}
  \frac{1}{c_{0}^{\alpha}}
  -
  \frac{1}{h_{t}^{\alpha}}
  \leq
  \frac{1}{h_{1}^{\alpha}}
  -
  \frac{1}{h_{t}^{\alpha}}
  \leq
  - c_{2} \alpha (t - 1)
  .
\end{equation}
The claim follows by rearranging.
\end{proof}
\end{lemma}

We now prove the original lemma.

\begin{proof}[Proof of Lemma~\ref{lem:ConvergenceRate}]
The last inequality is just a reformulation of the previous one
for the case \(t \geq t_{0}\), hence its proof is omitted.

We combine the two preceding lemmas:
one applying to the initial part of the sequence
when \(c_{1} < c_{2} h_{t}^{\alpha}\)
and the other to the rest when \(c_{1} > c_{2} h_{t}^{\alpha}\),
with \(t_{0}\) being an upper bound on the index of the boundary of
the two parts.

Let \(t'\) be the smallest index with \(t' \geq t_{0}\)
or \(h_{t'} \leq (c_{1}/c_{2})^{1/\alpha}\).
(The index \(t'\) is meant to be the boundary between the parts of the
sequence \(h_{t}\).)
The initial part of the sequence is
\(h_{1}, \dotsc, h_{t'}\).
In particular, \(c_{1} < c_{2} h_{t}^{\alpha}\) for \(1 \leq t < t'\)
and hence \(h_{t} \leq c_{0} (1 - c_{1})^{t-1}\)
for \(1 \leq t \leq t'\) by Lemma~\ref{lem:ConvergenceRateLinear}.

By the choice of \(t'\) we have \(t' \leq t_{0}\)
and \(h_{t'} \leq (c_{1}/c_{2})^{1/\alpha}\) unless possibly for
\(t'=t_{0}\).
Actually, the inequality holds even for \(t'=t_{0}\),
as \(h_{t'} \leq c_{0} (1 - c_{1})^{t' - 1}
= c_{0} (1 - c_{1})^{t_{0} - 1} \leq (c_{1}/c_{2})^{1/\alpha}\)
by the choice of \(t_{0}\).

Now Lemma~\ref{lem:ConvergenceRatePower} applies
to \(h_{t'}, h_{t'+1}, \dots\) with \(h_{t'} \leq
(c_{1}/c_{2})^{1/\alpha}\),
using again that \(h_{t}\) is monotonically decreasing
so that \(c_{2} h_{t}^{\alpha} \leq c_{2} ((c_{1} /
c_{2})^{1/\alpha})^{\alpha} = c_{1}\).
Thus by the lemma, for \(t \geq t'\)
\begin{equation}
  h_{t}
  \leq
  \frac{(c_{1}/c_{2})^{1/\alpha}}{\bigl(1 +
    c_{2} (c_{1}/c_{2}) \alpha (t - t') \bigr)^{1 / \alpha}}
  ,
\end{equation}
from which the claim follows via \(t' \geq t_{0}\).
\end{proof}

We are ready to establish a convergence rate for
Frank–Wolfe algorithms over uniformly convex sets.

\begin{proof}[Proof of Theorem~\ref{fw:rateLinstrcvxset}]
The proof follows from considering the progress shown in Lemma~\ref{lemma:progress} and applying Proposition~\ref{prop:scalescb} as follows:
\begin{spreadlines}{2ex}
\begin{equation*}
 \begin{split}
  f(x_{t}) - f(x_{t+1})
  &
  \geq
  \frac{g_{t}}{2} \min \left\{
    \frac{g_{t}}{L\norm{x_t - v_t}^2}, 1
  \right\}
  \\
  &
  \geq
  \frac{g_{t}}{2} \min \left\{
    \frac{g_{t}^{1 - 2/q} \cdot
      \alpha^{2/q} \dualnorm{\nabla f(x_{t})}^{2/q}}{L}, 1
  \right\}
  \\
  &
  \geq
  \frac{h_{t}}{2} \min \left\{
    \frac{h_{t}^{1 - 2/q} \cdot
      (\alpha c)^{2/q}}{L}, 1
  \right\}
  .
 \end{split}
\end{equation*}
\end{spreadlines}
For \(q > 2\) the claim follows from Lemma~\ref{lem:ConvergenceRate}
with the choice \(c_{0} = L D^{2} / 2\), \(c_{1} = 1/2\),
\(c_{2} = (\alpha c)^{2/q} / L\) and \(\alpha = 1 - 2/q\).
For \(q=2\), reordering provides
\begin{equation*}
  h_{t+1}
  \leq
  \max \left\{ \frac{1}{2}, 1 - \frac{ \alpha c}{2 L} \right\}
  h_{t}.
\end{equation*}
Together with the initial gap \(h_{1} \leq L D^{2} / 2\)
this proves the claim.
\end{proof}

Theorems~\ref{fw_strcvxset} and~\ref{fw_linInt} together
cover the case when the global minimum of the objective function
lies away from the boundary of the domain \(\mathcal{X}\).
Naturally the question arises
what happens when the minimum lies on the boundary.
We are not aware of a definitive answer, however,
additional assumptions guarantee improved convergence
even if not a linear rate.
We provide a rate for strongly convex objective functions; we refer
the reader to \citet{UniformConvexFW_2020} for sharp functions
in general. The theorem establishes a convergence rate interpolating
between $\mathcal{O}(1/t)$ and $\mathcal{O}(1/t^{2})$ depending on the
uniform convexity parameter of the set $\mathcal{X}$.

\begin{theorem}
\label{fw_strcvxset2}
Let $\mathcal{X}$ be a compact, $(\alpha, q)$-uniformly convex set
and \(f\) an
\(L\)-smooth and \(c\)-gradient dominated convex function over
$\mathcal{X}$.
Let \(D\) be the diameter of \(\mathcal{X}\).
Then the (vanilla) Frank–Wolfe algorithm
from a starting point
$x_0 = \argmin_{v \in \mathcal{X}} \innp{\nabla f(x)}{v}$ with
$x \in \mathcal{X}$ ensures that
\begin{equation}
  f(x_{t}) - f(x^{*}) \leq
  \begin{cases}
    \frac{L D^{2}}{2^{t}}
    & q \geq 2, 1 \leq t \leq t_{0}, \\
    \frac{(4 L c^{2}/ \alpha^{2})^{1/(q-1)}}
    {(1 + (1/2) \cdot (1 - 1/q) \cdot (t - t_{0}) )^{q / (q-1)}}
    =
    \mathcal{O}(1 / t^{q/(q-1)})
    & q \geq 2, t \geq t_{0}
  \end{cases}
\end{equation}
for all \(t \geq 1\) where
\begin{equation*}
  t_{0} \defeq \max
  \left\{
    \left\lfloor
      \log_{2} \left(
        \frac{D^{2}}
        {(4 c^{2} / \alpha^{2})^{1/(q-1)}}
      \right)
    \right\rfloor + 1,
    1
  \right\}
  .
\end{equation*}
Equivalently, the primal gap is at most \(\varepsilon\)
for small enough positive \(\varepsilon\)
after the following number of
linear optimizations and gradient computations
\begin{multline}
  \label{fw:rateLinstrcvxset2}
  t_{0}
  + \frac{q}{(q - 1) \cdot L \cdot
    (\alpha c^{-1})^{2/q} \varepsilon^{1 - 1/q}}
  - \frac{2 q}{q - 1}
  = \mathcal{O}
  \left(
    \frac{1}{\varepsilon^{1 - 1/q}}
  \right)
  \\
  q>2, \ \
  \varepsilon \leq L / \bigl(2^{q / (q-1)} (\alpha c^{-1})^{2 / (q-1)}\bigr)
  .
\end{multline}
\begin{proof}
We use an argument similar to Theorem~\ref{fw_strcvxset}, but use
gradient dominance of \(f\) to lower bound the norm of its gradient.
We obtain
\begin{spreadlines}{2ex}
\begin{equation*}
 \begin{split}
  f(x_{t}) - f(x_{t+1})
  &
  \geq
  \frac{g_{t}}{2} \min \left\{
    \frac{g_{t}^{1 - 2/q} \cdot
      \alpha^{2/q} \dualnorm{\nabla f(x_{t})}^{2/q}}{L}, 1
  \right\}
  \\
  &
  \geq
  \frac{h_{t}}{2} \min \left\{
    \frac{h_{t}^{1 - 2/q} \cdot
      \alpha^{2/q} (h_{t} / c^{2})^{1/q}}{L}, 1
  \right\}
  \\
  &
  =
  \frac{h_{t}}{2} \min \left\{
    \frac{h_{t}^{1 - 1/q} \cdot
      (\alpha c^{-1})^{2/q}}{L}, 1
  \right\}
  .
 \end{split}
\end{equation*}
\end{spreadlines}
The rest of the proof is analogous to Theorem~\ref{fw_strcvxset},
but note that the exponent of \(h_{t}\) is \(1 - 1/q\) here
instead of \(1 - 2/q\) slowing down the convergence rate.
\end{proof}
\end{theorem}

\subsection{Polytopes and the Away-step Frank–Wolfe algorithm}
\label{sec:line-conv-gener}

Here we consider the case where the feasible region
\(\mathcal{X} = P\) is a polytope but the algorithm is no longer
the vanilla Frank–Wolfe algorithm so that
the lower bound in Theorem~\ref{thm:FW-slow} does not apply.
Linear convergence with smooth and strongly convex objectives
has been long conjectured, but was first proven quite
recently.
The first formal proof of linear convergence in the number of linear
minimization oracle calls for a modification
of the (vanilla) Frank–Wolfe algorithm
appeared in \citet{garber2013playing,garber2016linearly}.
The main idea here is optimizing a linear
objective only in a small neighborhood instead of the whole polytope
with oracle complexity of
a single linear minimization over the whole polytope.
Later on
\citet{lacoste15} showed that the well-known Away-step Frank–Wolfe
algorithm (Algorithm~\ref{away}) proposed by \citet{wolfe70} and
related variants converge linearly, too. The implementation of these
algorithms is relatively straightforward with very good real-world
performance.  In \citet{beck2017linearly} linear convergence was
extended to
objective functions which are only almost strongly convex,
being the sum of a linear function \(b\) and
the composition of a $\mu$-strongly
convex function \(g\) with another linear function \(A\),
i.e., functions of the form $f(x) \defeq g(A x) + \innp{b}{x}$.
This is now understood as a special case of sharp objective functions,
see Corollary~\ref{cor:AFWconvergence-sharp}.
In this section, we present the Away-step Frank–Wolfe
algorithm (Algorithm~\ref{away}) focusing on strongly convex objective
functions.

\looseness=1
Recall from
Section~\ref{sec:limit-move-vertex} that the vanilla Frank–Wolfe
algorithm's convergence speed is limited by vertices picked
up early on as they stay in the convex combination of all future
iterates.  To overcome this difficulty, \citet{wolfe70}
introduced steps that move \emph{away} from vertices in a given
convex combination of \(x_t\), as opposed to steps that move
\emph{towards} vertices of \( P\). These steps are suggestively
called \emph{\myindex{away step}s}.  The key idea is to remove
weight in the convex combination from undesirable vertices, i.e.,
those that reduce convergence speed. We present a slightly modified variant of the
original algorithm in \citet[§\,8]{wolfe70}, by choosing
away vertices only from active sets (as done in \citet{gm86,lacoste15}) and using the short step rule rather than line search \citep[as, e.g., done in][]{pedregosa2018step}.
To this end, we formally define the \emph{\myindex{away vertex}} as
\(v_{t}^{\text{A}} = \argmax_{v \in \mathcal{S}}
\innp{\nabla f(x_{t})}{v}\),
where \(\mathcal{S} \subseteq \vertex{P}\)
so that \(x_t\) is a (strict) convex combination of
the elements in \(\mathcal{S}\), i.e.,
\(x_t = \sum_{v \in \mathcal{S}} \lambda_v v\) with
\(\lambda_v > 0\) for all \(v \in \mathcal{S}\) and
\(\sum_{v \in \mathcal{S}} \lambda_v = 1\).
Such a set \(\mathcal{S}\)
is called an \emph{active set}\index{active set} as the vertices actively participate in
the convex combination. The algorithm is presented in
Algorithm~\ref{away}.
To get an overview, one might ignore
Lines~\ref{line:AFW-update-start}–\ref{line:AFW-update-end}
detailing the updates to coefficients in a convex combination.
Figure~\ref{fig:away} illustrates the benefit of this method, resolving
the zigzagging issue of Figure~\ref{fig:zigzag}.

\begin{figure}[b]
\centering
  \includegraphics[width=0.5\linewidth, alt={Example trajectory of
    Away-step Frank–Wolfe algorithm on a triangle: after a few
    zigzagging, an away step moves to the side containing
    the optimum.}]{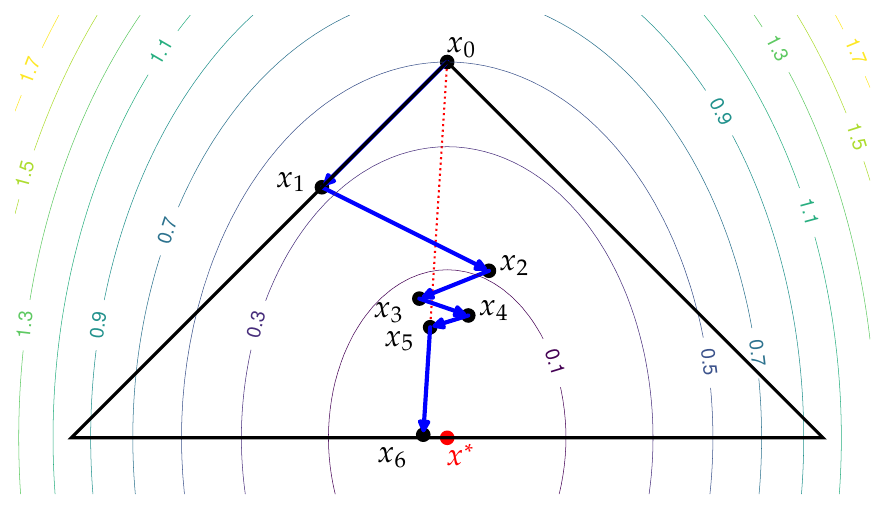}

\caption{The Away-step Frank–Wolfe algorithm (Algorithm~\ref{away})
   breaking zigzagging behavior via away steps
   (Line~\ref{line:AFW-away-step-size}),
   improving upon the convergence rate of
   the vanilla Frank–Wolfe algorithm (Algorithm~\ref{fw}.
   The example is the same as in
   Figure~\ref{fig:zigzag} except the algorithm:
   The algorithm uses line search. The
   feasible region is the triangle
   $P =\conv{\{(-1,0),(1,0),(0,1)\}}$,
   the objective function is \(f(x,y) = 2 x^{2} + y^{2}\).
   The starting vertex is $x_{0} = (0,1)$, and the optimum is $x^{*}
   = (0,0)$.  The dashed \textcolor{red}{red} line indicates the
   direction of update for the away step performed at $x_5$, which
   constitutes the main difference from the Frank–Wolfe algorithm
   (Algorithm~\ref{fw}).
   Here $x_6$ is obtained using an away step from $x_5$,
   speeding up the algorithm considerably by removing most of the weight from
   $x_0$. Note that $x_6$ does not lie on the lower edge of the triangle, 
   but lies instead right above it, as the objective function is not
 overall monotonically decreasing towards the lower edge.
    \vspace{-1ex}}
  \label{fig:away}
\end{figure}

As all algorithms, the Away-step Frank–Wolfe algorithm also has
 trade-offs: for some polytopes,
like the spectrahedron or the Birkhoff polytope
(see Example~\ref{ex:Birkhoff}),
the active set $\mathcal{S}$ can grow very large.
This is not only very memory intensive, but also makes
finding the away vertex \(v_{t}^{\text{A}}\) expensive,
as it often requires going through the whole active set $\mathcal{S}$; this can be a real issue in implementations.
These can be somewhat mitigated by periodically reducing the size of
\(\mathcal{S}\) to roughly at most the dimension of the polytope
(Carathéodory's theorem).
In Section~\ref{sec:decomposition-invariant}
we will present variants that don't use an active set.

\begin{example}[AFW convergence]
  \label{example:AFW}
  Consider the first iterates (see Figure~\ref{fig:away}) produced by
  the Away-step Frank–Wolfe algorithm with line search, for
  $P =\conv{\{(-1,0),\allowbreak (1,0),(0,1)\}}$
  and objective $f(x, y) \defeq  2x^{2} + y^{2}$,
  with starting point $x_0=(0,1)$.
  Note that the optimum is $x^{*} = (0,0)$.
\end{example}

\begin{algorithm}\index{Away-step Frank-Wolfe algorithm}
  \caption{Away-step Frank–Wolfe (AFW) \citep{gm86}}
\label{away}
\begin{algorithmic}[1]
  \REQUIRE Start atom \(x_{0} \in
    \argmin_{v \in P}\innp{\nabla f(x)}{v}\) for $x\in P$
  \ENSURE Iterates $x_1, \dotsc \in P$
  \STATE$\mathcal{S}_{0} \leftarrow \{x_0\}$
  \STATE$\lambda_{x_0, 0} \leftarrow 1$
  \FOR{$t=0$ \TO \dots}
    \STATE \label{FWVertex}
      \(v_{t}^{\text{FW}} \leftarrow \argmin_{v \in P}\innp{\nabla f(x_{t})}{v}\)
    \STATE \label{AwayVertex}
      $v_{t}^{\text{A}} \leftarrow \argmax_{v \in \mathcal{S}_{t}} \innp{\nabla f(x_{t})}{v}$
    \IF[Frank–Wolfe step]
        {\(\innp{\nabla f(x_{t})}{x_{t} - v_{t}^{\text{FW}}} \geq
          \innp{\nabla f(x_{t})}{v_{t}^{\text{A}} - x_{t}}\)}
        \label{va}
      \STATE \(d_{t} \gets x_{t} - v_{t}^{\text{FW}}\),
        \(\gamma_{t,\max} \gets 1\)
    \ELSE[away step]
      \STATE \label{line:AFW-away-step-size}
        \(d_{t} \gets v_{t}^{\text{A}} - x_{t}\),
        \(\gamma_{t,\max} \gets \frac{\lambda_{v_{t}^{\text{A}}, t}}{1 -
          \lambda_{v_{t}^{\text{A}}, t}}\)
    \ENDIF
    \STATE \(\gamma_{t} \leftarrow \min \left\{
        \frac{\innp{\nabla f(x_{t})}{d_{t}}}{L \norm{d_{t}}^{2}},
        \gamma_{t,\max} \right\}\)
    \STATE \(x_{t+1} \leftarrow x_{t} - \gamma_{t} d_{t}\)
    \IF[Update coefficients and active set]
        {\(\innp{\nabla f(x_{t})}{x_{t} - v_{t}^{\text{FW}}} \geq
          \innp{\nabla f(x_{t})}{v_{t}^{\text{A}} - x_{t}}\)}
        \label{line:AFW-update-start}
      \STATE
        \(\lambda_{v, t+1} \leftarrow (1 - \gamma_{t}) \lambda_{v, t}\)
        for all
        \(v \in \mathcal{S}_{t} \setminus \{v_{t}^{\text{FW}}\}\)
      \STATE
        \(\lambda_{v_{t}^{\text{FW}}, t+1} \leftarrow
        \begin{cases}
          \gamma_{t} &
          \text{if } v_{t}^{\text{FW}} \notin \mathcal{S}_{t} \\
          (1 - \gamma_{t}) \lambda_{v_{t}^{\text{FW}}, t}
          + \gamma_{t} &
          \text{if } v_{t}^{\text{FW}} \in \mathcal{S}_{t}
        \end{cases}\)
      \STATE
        \(\mathcal{S}_{t+1} \leftarrow
        \begin{cases}
          \mathcal{S}_{t} \cup \{v_{t}^{\text{FW}}\}
          & \text{if } \gamma_{t} < 1 \\
          \{v_{t}^{\text{FW}}\}
          & \text{if } \gamma_{t} = 1
        \end{cases}\)
    \ELSE
      \STATE
        \(\lambda_{v, t+1} \leftarrow
        (1 + \gamma_{t}) \lambda_{v, t}\)
        for all \(v \in \mathcal{S}_{t} \setminus \{v_{t}^{\text{A}}\}\)
      \STATE
        \(\lambda_{v_{t}^{\text{A}}, t+1} \leftarrow
        (1 + \gamma_{t}) \lambda_{v_{t}^{\text{A}}, t}
        - \gamma_{t}\)
      \STATE\label{line:AFW-update-end}
        \(\mathcal{S}_{t+1} \leftarrow
        \begin{cases}
          \mathcal{S}_{t} \setminus \{v_{t}^{\text{A}}\}
          & \text{if } \lambda_{v_{t}^{\text{A}}, t+1} =
          0 \\
          \mathcal{S}_{t} & \text{if } \lambda_{v_{t}^{\text{A}}, t+1} >
          0
        \end{cases}\)
    \ENDIF
  \ENDFOR
\end{algorithmic}
\end{algorithm}

For away steps, the step size is conservatively limited in
Line~\ref{line:AFW-away-step-size} to ensure that the new iterate lies
in $P$, and can be easily written as a convex combination of vertices already present in the convex combination. At the limit of the step size we usually cannot lower bound primal progress sufficiently, however the away vertex is \emph{removed}
from the convex combination, providing improved sparsity.
This sparsity improvement will play a crucial role in the convergence proof: combined with the fact that the active set
\(\mathcal{S}_{t}\) can only grow in Frank–Wolfe steps, and then also by at most one new vertex, the number of these so-called \emph{drop steps}, where we remove an (away) vertex from the convex combination, is well controlled.

\looseness=1
Whether to take a traditional Frank–Wolfe step or an away step will
depend on which one promises more progress (see Line~\ref{va}). To
further highlight the importance of vertex removal through drop steps,
it was conjectured that after a finite number of initial iterations, the iterates of the algorithm would permanently remain on the minimal face containing the optimum \(x^{*}\), at which point the linear convergence for inner optima would apply (see Section~\ref{sec:linConvInterior}) \citep[§\,8]{wolfe70}. Later, in
\citet[Theorem~5]{gm86} it was shown that this is indeed the case in
certain situations, however, the arising linear rate via this approach
heavily depends on the position of the minimum \(x^{*}\),
and is especially bad when \(x^{*}\) is close to a face but is not
on the face.
Very recently in \citet{garber2020sparseFW}, using a strict
complementarity assumption, it was shown that the AFW algorithm reaches
 the face that contains \(x^{*}\) in a well-characterized 
number of iterations, exhibiting afterwards a linear 
convergence rate that depends favourably on geometric properties
of the aforementioned face and the optimal solution, such as, e.g., the face's diameter and the solution's sparsity.

To obtain linear convergence
without additional assumption and
independent of the position of the minimum,
\citet{lacoste15} directly replaced the distance of the optimum
to the boundary in the scaling
inequality in Proposition~\ref{prop:scaleGM} with a geometric
distance-like constant, so called the \emph{pyramidal width} \(\delta\)
of the polytope~\(P\).
It is the same as the \emph{facial distance} used in
\citet{pena2018polytope},
which is the minimum distance of a face of \(P\)
from the rest of the polytope (convex hull of the vertices not lying
on the face),
and refines the \emph{vertex-facet distance} proposed in
\citet{beck2017linearly}.

The pyramidal width\index{pyramidal width} is
the minimal positive number \(\delta\) satisfying the scaling
inequality shown in Lemma~\ref{lemma:pyraScaling} \citep[see][Theorem
3]{lacoste15}. More details about the pyramidal width can be found
at the \hyperref[sec:pyramidal-width]{end of this section}.
It is easy to see that
\(\delta \leq D\).

\begin{lemma}[Scaling Inequality via Pyramidal Width]
\label{lemma:pyraScaling}
  Let \(P\) be a polytope and let \(\delta > 0\) denote the pyramidal
  width of \(P\).
  Let \(x \in P\) and let \(\mathcal{S}\) denote any set of vertices of \(P\)
  with \(x \in \conv{\mathcal{S}}\).
  Let \(\psi\) be any vector, so that we define
  \(v^{\text{FW}} = \argmin_{v \in P} \innp{\psi}{v}\)
  and
  \(v^{\text{A}}
  = \argmax_{v \in \mathcal{S}} \innp{\psi}{v}\).
  Then for any \(y \in P\)
  \begin{equation}
    \label{eq:pyraScale}
    \innp{\psi}{v^{\text{A}} - v^{\text{FW}}}
    \geq \delta
    \frac{\innp{\psi}{x -  y}}{\norm{x - y}}.
  \end{equation}
\end{lemma}

\pagebreak

For non-polytope domains, no generalization of the pyramidal width is
known.  In particular, the straightforward generalization of allowing
all extreme points in the set \(\mathcal{S}\)
would lead to a pyramidal width of \(0\).

As done in the introduction to this section, the scaling inequality together with strong convexity allows us to derive an upper bound on the primal gap. We refer to this bound as \emph{geometric strong convexity}\index{geometric strong convexity}

\begin{lemma}[Geometric Strong Convexity]
\label{lem:geoSC}
  Let \(P\) be a polytope with pyramidal
  width \(\delta > 0\) and let \(f\) be a \(\mu\)-strongly convex
  function, then using the notation of Lemma~\ref{lemma:pyraScaling},
  i.e.,
  with \(v^{\text{FW}} = \argmin_{v \in P} \innp{\nabla f(x)}{v}\) and \(v^{\text{A}}
  = \argmax_{v \in \mathcal{S}} \innp{\nabla f(x)}{v}\) with $S
  \subseteq \vertex{P}$, so that $x \in \conv{S}$,
  we have
  \begin{equation}
    \label{eq:geoSCprimal}
    f(x) - f(x^{*})
    \leq
    \frac{\innp{\nabla f(x)}{v^{\text{A}} -
        v^{\text{FW}}}^{2}}{2 \mu \delta^{2}}
    .
\end{equation}
\end{lemma}

For the Away-step Frank–Wolfe algorithm, we will use Lemma~\ref{lem:geoSC} with
\(x=x_{t}\), where it takes the form
\begin{equation}
\label{eq:scaling-pyramid}
f(x_{t}) - f(x^{*})
\leq
\frac{\innp{\nabla f(x_t)}{v_{t}^{\text{A}} -
v^{\text{FW}}_{t}}^{2}}{2 \mu \delta^{2}}, 
\end{equation}
where \(v^{\text{FW}}_{t}\) denotes the Frank–Wolfe vertex and
\(v_{t}^{\text{A}}\) denotes the away vertex in iteration \(t\)
defined as above. 

\begin{proof}
The statement follows from combining Equation~\eqref{eq:pyraScale}
with Lemma~\ref{lem:SCprimal} for \(\psi = \nabla f(x)\).
We have
\begin{equation*}
  f(x) - f(x^{*})
  \leq
  \frac{\innp{\nabla f(x)}{x - x^{*}}^{2}}
  {2 \mu \norm{x - x^{*}}^{2}}
  \leq
  \frac{\innp{\nabla f(x)}{v^{\text{A}}
      - v^{\text{FW}}}^{2}}
  {2 \mu \delta^{2} }.
  \qedhere
\end{equation*}
\end{proof}

Note that the upper bound in Lemma~\ref{lem:geoSC} involves
a new constant \(\mu \delta^{2}\) combining
a property of the objective function \(f\)
(the strong convexity constant \(\mu\))
with a property of the feasible region \(P\)
(the pyramidal width \(\delta\)).

\begin{remark}[Strong Frank–Wolfe gap]\index{strong Frank-Wolfe gap}
It is worthwhile to pause to appreciate the quantity
\[
s(x) \defeq \min_{\mathcal{S} \subseteq \vertex{P}:
  v \in \conv{\mathcal{S}}}
\innp{\nabla f(x)}{v^{\text{A}}_{\mathcal{S}}
  - v^{\text{FW}}},
\]
which is the \emph{strong Frank–Wolfe gap}.  Here we take the minimum over all
subsets of vertices of $P$ whose convex hull contains \(x\), i.e., all sets of
vertices that can represent $x$ as convex combination. We do this in order to
define the strong Frank–Wolfe gap independent of any external quantity,
following the approach in \citet{kerdreux2018restarting}.  However, for most
purposes any convex combination of \(x\) and associated active set \(\mathcal
S\) will do.  Clearly \(g(x) \leq s(x)\). In fact, \(s(x) = 0\) if and only if
\(x\) is an optimal solution, i.e., \(f(x) - f(x^*) = 0\), even if we consider
a specific active set \(\mathcal S\). This can be seen easily, as 
\[
\innp{\nabla f(x)}{v^{\text{A}}_{\mathcal{S}}
  - v^{\text{FW}}} = \underbrace{\innp{\nabla f(x)}{x
- v^{\text{FW}}}\vphantom{v^{\text{A}}_{\mathcal{S}}}}_{\geq 0} + \underbrace{\innp{\nabla
    f(x)}{v^{\text{A}}_{\mathcal{S}} - x}}_{\geq 0} \geq 0,
\]
where the term $\innp{\nabla f(x)}{x - v^{\text{FW}}}$ is the
Frank–Wolfe gap and the term $\innp{\nabla
f(x)}{v^{\text{A}}_{\mathcal{S}} - x}$ is the so-called
\emph{away gap}. If either of the two is strictly positive, then we
can perform a Frank–Wolfe step or an away step, respectively, leading to
positive primal progress via Lemma~\ref{lemma:progress} and hence
\(x\) cannot have been optimal (we will see this later in the
convergence proof). The other direction is trivial as \(f(x) - f(x^*)
\leq g(x) \leq s(x) \leq 0\). From the strong Frank–Wolfe gap, we can also see that
dropping a vertex from the convex combination can significantly
tighten the gap. In fact the contrapositive of this property can lead
to several implementation challenges: away-vertices with tiny weight
in the convex combination can artificially enlarge the away gap, and
it is imperative in implementations to correctly drop vertices. We
will also see soon that the strong Frank–Wolfe gap converges at a rate
identical to the primal gap.
\end{remark}

We are ready to prove a linear convergence rate for the Away-step Frank–Wolfe
algorithm for strongly convex functions.  Note that the properties of the
objective function and the feasible region appear separately: namely, the
conditional number~\(L / \mu\) of the objective function and a geometric
parameter \(D / \delta\) of the feasible region, which often depends on its
dimension.  It is possible to replace their combination \smash{$\frac{\mu}{L} \bigl(
\frac{\delta}{D} \bigr)^2$}, and in fact more precisely, replace $\mu \delta^2$
by a joint \emph{geometric strong convexity} that depends on the feasible
region and the function simultaneously, as  pioneered in \citet{lacoste15}
or with the relative condition number
from \citet{gutman2020condition}.  Both
approaches however lead essentially to the same results in convergence
rate.  The only thing that changes is the interpretation of the
constant and both upper bound the same quantity $\mu \delta^2$.

In Algorithm~\ref{away}, whenever \(\gamma_{t} = \gamma_{t, \max}\) in an away
step, the active set becomes smaller, as the away vertex \(v_{t}^{\text{A}}\)
is dropped and we call such steps \emph{\myindex{drop step}s}.
However, this cannot
happen in more than half the steps, as you can only drop a vertex that has been
added in an earlier iteration, as we will see in the proof below. 

\begin{theorem}
\label{th:AFWConvergence}
Let $P \subset\R^n$ be a polytope and
$f$ be an $L$-smooth and $\mu$-strongly convex function over \(P\).
The convergence rate of the Away-step Frank–Wolfe algorithm
(AFW, Algorithm~\ref{away}) with objective \(f\)
is linear: for all $t \geq 1$
\begin{equation}
  \label{eq:AFWConvergence}
  f(x_t) - f(x^*)
  \leq
  \left(
    1 - \frac{\mu}{4L} \frac{\delta^{2}}{D^{2}}
  \right)^{\lceil (t-1)/2 \rceil}
  \frac{LD^2}{2}
  ,
\end{equation}
where $D$ and $\delta$ are the diameter and the pyramidal width of the
polytope $P$ respectively.
Equivalently, the primal gap is at most \(\varepsilon\)
after at most the following number of
linear optimizations and gradient computations:
\begin{equation}
  1 + \frac{8L}{\mu} \frac{D^{2}}{\delta^{2}}
  \ln \frac{LD^2}{2\varepsilon}.
\end{equation}
Moreover, for the strong Frank–Wolfe gap we obtain
\begin{equation}
  \label{eq:sfw-gap-convergence}
  s(x_{t})
  \leq
  \left(
    1 - \frac{\mu}{4L} \frac{\delta^{2}}{D^{2}}
  \right)^{\lceil (t-1) / 2 \rceil / 2}
  2 L D^{2},
\end{equation}
whenever iteration \(t\) was not a drop step. 
Equivalently,
the strong Frank–Wolfe gap is at most \(\varepsilon > 0\)
after at most the following number of
linear optimizations and gradient computations:
\begin{equation}
  1 + \frac{16 L}{\mu} \frac{D^{2}}{\delta^{2}}
  \ln \frac{2 L D^{2}}{\varepsilon}
  .
\end{equation}
\end{theorem}
When the objective function is smooth but not necessarily
strongly convex,
an \(\mathcal{O}(1/\varepsilon)\) convergence rate follows
similar to \citet{fw56,polyak66cg}; i.e., Theorem~\ref{fw_sub} here
(even though drop steps slightly complicate the proof).

\begin{proof}
We follow the proof in \citet{lacoste15}.
To treat the two branches of the conditional statement
in Line~\ref{va} and \ref{line:AFW-away-step-size} of 
Algorithm~\ref{away} uniformly,
note that \(x_{t+1} = x_{t} - \gamma_{t} d_{t}\),
where \(d_{t}\) is either
\(x_{t} - v_{t}^{\text{FW}}\) or \(v_{t}^{\text{A}} - x_{t}\),
with \(\innp{\nabla f(x_{t})}{d_{t}} \geq
\innp{\nabla f(x_{t})}{v_{t}^{\text{A}} - v_{t}^{\text{FW}}} / 2\).
By geometric strong convexity (Lemma~\ref{lem:geoSC}), we have
\begin{equation}
  \label{eq:4}
  h_{t} = f(x_{t}) - f(x^{*})
  \leq
  \frac{\innp{\nabla f(x_{t})}{v_{t}^{\text{A}} -
      v_{t}^{\text{FW}}}^{2}}{2 \mu \delta^{2} }
  \leq
  \frac{2 \innp{\nabla f(x_{t})}{d_{t}}^{2}}{\mu \delta^{2} }
  .
\end{equation}
Thus, by using Equation~\eqref{eq:4} and
Progress Lemma~\ref{lemma:progress}
with \(\gamma_{t, \max} = 1\) for Frank–Wolfe steps,
and \(\gamma_{t, \max} = \smash[b]{\frac{\lambda_{v_{t}^{\text{A}}, t}}
{1 - \lambda_{v_{t}^{\text{A}}, t}}}\) for away steps, and the fact \(\innp{\nabla f(x_{t})}{d_{t}} \geq
\innp{\nabla f(x_{t})}{x_{t} - v_{t}} \geq h_{t}\)
we obtain 
\begin{spreadlines}{1ex}
\begin{equation}
  \label{eq:AFWstep}
 \begin{split}
  h_{t} - h_{t+1}
  &
  \geq
  \frac{\innp{\nabla f(x_{t})}{d_{t}}}{2}
  \min \left\{
    \gamma_{t, \max},
    \frac{\innp{\nabla f(x_{t})}{d_{t}}}{L \norm{d_{t}}^{2}}
  \right\}
  \\
  &
  =
  \min \left\{\gamma_{t, \max} \frac{\innp{\nabla f(x_{t})}{d_{t}}}{2},
    \frac{\innp{\nabla f(x_{t})}{d_{t}}^{2}}{2 L \norm{d_{t}}^{2}}
  \right\}
  \\
  &
  \geq
  \min
  \left\{
    \gamma_{t, \max} \frac{h_{t}}{2},
    \frac{\mu \delta^{2} h_{t}}{4 L D^{2}}
  \right\}
  = \min\left\{
    \frac{\gamma_{t, \max}}{2},
    \frac{\mu\delta^{2}}{4LD^{2}}
  \right\}
  \cdot h_{t}.
 \end{split}
\end{equation}
\end{spreadlines}
Therefore
\begin{equation}
  \label{contraction:fw}
  h_{t+1}
  \leq
  \left( 1 -
  \min \left\{
    \frac{\gamma_{t, \max}}{2},
    \frac{\mu\delta^{2}}{4 L D^{2}}
  \right\}
  \right)
  h_{t}
  .
\end{equation}
This is close to a linear convergence, as long as
\(\gamma_{t, \max}\) is bounded away from \(0\).
For Frank–Wolfe steps, this is easy
as clearly \(\gamma_{t, \max} = 1 \geq \mu \delta^{2} / (L D^{2})\).
For away steps,
there seems to be no easy way to lower bound \(\gamma_{t, \max}\),
so we only obtain a monotone progress \(h_{t+1} < h_{t}\)
in general.

However, \(\gamma_{t, \max}\) cannot be small too often: When
\(\gamma_{t} = \gamma_{t, \max}\) in an away step, i.e., a drop step, recall that the
active set becomes smaller, as the away vertex \(v_{t}^{\text{A}}\) is
removed. Moreover, the
active set can only grow by one new vertex per iteration (namely, by
the Frank–Wolfe vertex in Frank–Wolfe steps).  As it is impossible to
remove more vertices from the active set than have been added with Frank–Wolfe steps, it follows that
up to any iteration \(t\) only at most half of the iterations could be
drop steps (here we count every vertex with multiplicity: the
number of times it has been added or removed from the active set).
Thus, in all other steps we have \(\gamma_{t} = \innp{\nabla
  f(x_{t})}{d_{t}} \mathbin{/} (L \norm{d_{t}}^{2})
< \gamma_{t, \max}\),
ensuring \(h_{t+1} \leq (1 - \mu \delta^{2} / (4 L D^{2})) h_{t}\) in
those steps. 

All in all, we obtain
\(h_{t+1} \leq (1 - \mu \delta^{2} / (4 L D^{2})) h_{t}\)
for at least half of the iterations (those that are not drop steps),
and \(h_{t+1} \leq h_{t}\) for the rest (the drop steps).
By Remark~\ref{rem:initial-bound} we have that \(h_1 \leq \frac{LD^2}{2}\)  as the very first
step has to be a Frank–Wolfe step. The convergence rate follows.

The convergence of the strong Frank–Wolfe gap
easily follows from Equation~\eqref{eq:AFWstep} for non-drop steps
(replacing \(\gamma_{t, \max}\) with \(1\), as either
\(\gamma_{t} < \gamma_{t, \max} \leq 1\) or \(\gamma_{t, \max} = 1\)
as discussed above)
via
\begin{spreadlines}{1ex}
  \begin{equation*}
   \begin{split}
\left( 1 - \frac{\mu}{4L} \frac{\delta^{2}}{D^{2}} \right)^{(t-1)/2}
 \frac{LD^2}{2}
 & \geq
 h_{t}
 \geq
 h_{t} - h_{t+1}
 \\
 & \geq
 \min \left\{\frac{\innp{\nabla f(x_{t})}{d_{t}}}{2},
   \frac{\innp{\nabla f(x_{t})}{d_{t}}^{2}}{2 L \norm{d_{t}}^{2}}
 \right\}
 \\
 & \geq
 \min \left\{
   \frac{s(x_{t})}{4},
   \frac{s(x_{t})^{2}}{8 L D^{2}}
 \right\}
 ,
 \end{split}
\end{equation*}
\end{spreadlines}
where we have used \(\innp{\nabla f(x_{t})}{d_{t}} \geq
\innp{\nabla f(x_{t})}{v_{t}^{\text{A}} - v_{t}^{\text{FW}}} \mathbin{/}
2 \geq s(x_{t}) / 2\).
This leads to
\begin{spreadlines}{1ex}
  \begin{equation*}
   \begin{split}
  s(x_{t}) 
  & \leq
  \left( 1 - \frac{\mu}{4L} \frac{\delta^{2}}{D^{2}}
  \right)^{\min\{\lceil (t-1) / 2 \rceil,  \lceil (t-1) / 2 \rceil / 2\}}
  \cdot
  2 L D^{2}
  \\
  & =
  \left( 1 - \frac{\mu}{4L} \frac{\delta^{2}}{D^{2}}
  \right)^{\lceil (t-1) / 2 \rceil / 2}
  \cdot
  2 L D^{2}
  .
  \qedhere
   \end{split}
  \end{equation*}
\end{spreadlines}
\end{proof}

\subsubsection{Pyramidal width}
\label{sec:pyramidal-width}

In this section, we provide examples and lower bounds on the pyramidal
width of a polytope \(P\).
Recall that we define \emph{pyramidal width} as the smallest number
satisfying the scaling inequality (Lemma~\ref{lemma:pyraScaling}):
\begin{definition}[{Pyramidal width\index{pyramidal width},
    cf.~\citet[§\,3]{lacoste15}}]
  \label{def:pyramidal}
  The pyramidal width of a polytope \(P\)
  is the largest number \(\delta\)
  such that for
  any set \(\mathcal{S}\) of some vertices of \(P\),
  any point \(x \in \conv{\mathcal{S}}\) in the convex hull
  of \(\mathcal{S}\),
  any \(y \in P\),
  and any vector \(\psi\),
  we have
  \begin{equation}
    \label{eq:pyraScale2}
    \innp{\psi}{v^{\text{A}} - v^{\text{FW}}}
    \geq \delta
    \frac{\innp{\psi}{x -  y}}{\norm{x - y}}
    ,
  \end{equation}
  where
  \(v^{\text{A}} \defeq \argmax_{v \in \mathcal{S}} \innp{\psi}{v}\)
  and
  \(v^{\text{FW}} \defeq \argmin_{v \in P} \innp{\psi}{v}\).
  (The maximum and minimum may not be unique, however the choice
  does not affect the definition.)
\end{definition}
This is obviously the intention behind the original definition
in \citet[§\,3]{lacoste15}, where there are additional technical
restrictions on the vector \(\psi\), so that the right-hand side
of Equation~\eqref{eq:pyraScale2} simplifies to
\(\delta \dualnorm{\psi}\)
when maximizing over \(y\).

Instead of the original definition we recall an equivalent
characterization due to \citet[Theorems~1 and 2]{pena2018polytope}.
Historically the proof of this characterization has been
established as a chain of equivalences between various width notions; we provide a direct proof here.

\begin{proposition}[Pyramidal width as facial distance]\index{facial distance}
  \label{prop:pyramidal-faces}
  For any polytope \(P\), its pyramidal width \(\delta\) is the
  minimum distance of any of its faces \(F\) from the convex hull of all
  vertices of \(P\) not lying on \(F\), i.e.,
  \begin{equation}
    \label{eq:pyramidal-faces}
    \delta = \min_{F \text{ face of } P}
    \distance{F}{\conv{\vertex{P} \setminus F}}
    ,
  \end{equation}
where we restrict to proper faces \(F\),
i.e., \(F \notin \{\emptyset, P\}\);
alternatively, one may treat distance between the empty set and \(P\)
as infinite.
\begin{proof}
Let \(\delta_{\text{face}}\) denote
the right-hand side of Equation~\eqref{eq:pyramidal-faces}:
\begin{equation}
  \label{eq:8}
  \delta_{\text{face}} \defeq
  \min_{F \text{ face of } P}
  \distance{F}{\conv{\vertex{P} \setminus F}}
  .
\end{equation}
We first prove \(\delta \geq \delta_{\text{face}}\).
Let \(\psi\) be a vector,
\(x, y \in P\) be points of \(P\)
with \(x\) a convex combination of a vertex set
\(\mathcal{S}\).
We need to show the scaling inequality
\begin{equation}
  \innp{\psi}{v^{\text{A}} - v^{\text{FW}}}
  \geq \delta_{\text{face}}
  \frac{\innp{\psi}{x -  y}}{\norm{x - y}}
  .
\end{equation}
Let \(F\) be the minimal face of \(P\) containing \(y\).
(If \(y\) is in the relative interior of \(P\)
then \(F = P\).)
We proceed by induction on the number of vertices
\(F\) has from \(\mathcal{S}\).

If \(F\) contains no vertices from \(\mathcal{S}\)
then
\(x \in \conv{\mathcal{S}} \subseteq \conv{\vertex{P} \setminus F}\),
hence \(\norm{x-y} \geq \distance{F}{\vertex{P} \setminus F}\).
Together with
\(\innp{\psi}{v^{\text{A}} - v^{\text{FW}}} \geq \innp{\psi}{x -  y}\),
the scaling inequality obviously follows.

The other case is when \(F\) contains
some vertex \(v\) from \(\mathcal{S}\).
We project \(x\) and \(y\) simultaneously from \(v\) to remove the
common vertex \(v\) for induction to apply.
To this end, consider \(y_{\gamma} = (1 + \gamma) y - \gamma v\)
and \(x_{\gamma} = (1 + \gamma) x - \gamma v\).
Choose \(\gamma \geq 0\) maximal such that \(y_{\gamma} \in P\)
and \(x_{\gamma} \in \conv{\mathcal{S}}\).
Maximality ensures that either \(y_{\gamma}\) or \(x_{\gamma}\)
will lie on a face of \(P\) or \(\conv{\mathcal{S}}\), respectively,
which does not contain \(v\).
In particular,
either
\(y_{\gamma}\) is contained in a face \(F'\) of \(F\) not
containing \(v\),
when induction applies with
\(x_{\gamma} \in \conv{\mathcal{S}}\) and \(y_{\gamma} \in F'\),
or \(x_{\gamma}\) is contained in the convex hull
of \(\mathcal{S} \setminus \{v\}\),
when induction applies with
\(x_{\gamma} \in \conv{\mathcal{S} \setminus \{v\}}\)
and \(y_{\gamma} \in F\).
In both cases the scaling inequality applies
to \(x_{\gamma}, y_{\gamma}\) by induction.
As \(x_{\gamma} - y_{\gamma} = (1 + \gamma) (x - y)\),
we conclude
\begin{equation}
  \innp{\psi}{v^{\text{A}} - v^{\text{FW}}}
  \geq \delta_{\text{face}}
  \frac{\innp{\psi}{x_{\gamma} - y_{\gamma}}}
  {\norm{x_{\gamma} - y_{\gamma}}}
  = \delta_{\text{face}}
  \frac{\innp{\psi}{x - y}}
  {\norm{x - y}}
  .
\end{equation}
(Note that in the case
\(x_{\gamma} \in \conv{\mathcal{S} \setminus \{v\}}\),
the new away vertex \smash{\(v^{\text{A}}_{\mathcal{S} \setminus \{v\}}\)}
might differ from \(v^{\text{A}}\) due to being chosen from a smaller
set, but obviously \(  \innp{\psi}{v^{\text{A}}} \geq
\innp{\psi}{v^{\text{A}}_{\mathcal{S} \setminus \{v\}}}\).)

Next we prove \(\delta \leq \delta_{\text{face}}\).
Let \(F_{0}\) be a face of \(P\) realizing the minimum for
\(\delta_{\text{face}}\), i.e.,
\(\delta_{\text{face}} =
\distance{F_{0}}{\conv{\vertex{P} \setminus F_{0}}}\).
Let this distance be realized by
\(y \in F_{0}\) and \(x \in \conv{\vertex{P} \setminus F_{0}}\),
i.e., \(\norm{x-y} =
\distance{F_{0}}{\conv{\vertex{P} \setminus F_{0}}}\).

We now replace \(F_{0}\) with a face \(F\) having all these
properties, but in addition having \(y\) in its relative interior.
To this end,
let \(F\) be the minimal face of \(P\) containing \(y\),
which obviously contains \(y\) in its relative interior.
We also have \(F \subseteq F_{0}\), and therefore
\(x \in \conv{\vertex{P} \setminus F_{0}} \subseteq
\conv{\vertex{P} \setminus F}\).
Finally,
\(\norm{x-y} =
\distance{F_{0}}{\conv{\vertex{P} \setminus F_{0}}}
= \delta_{\text{face}}
\leq
\distance{F}{\conv{\vertex{P} \setminus F}}\),
so that \(\norm{x-y}\) realizes the distance between
\(F\) and \(\conv{\vertex{P} \setminus F}\), i.e.,
\(\norm{x-y}
= \delta_{\text{face}}
=
\distance{F}{\conv{\vertex{P} \setminus F}}\).

We choose the vector \(\psi\) to separate
the convex sets \(F\) and \(\conv{\vertex{P} \setminus F}\),
and at the same time justify the distance between them.
Let \smash{\(B^{\circ}(r) \defeq \{ x \mid \norm{x} < r\}\)} denote
the open ball of radius \(r\) around origin.

Then the Minkowski sum \(F + B^{\circ}(\delta_{\text{face}})\)
is an open convex set disjoint from the closed convex set
\(\conv{\vertex{P} \setminus F}\), hence there is a non-zero
\(\psi\) with \(\innp{\psi}{z} < \innp{\psi}{w}\)
for all \(z \in F + B^{\circ}(\delta_{\text{face}})\)
and \(w \in \conv{\vertex{P} \setminus F}\),
i.e.,
\(\innp{\psi}{z} \leq
\innp{\psi}{w} - \dualnorm{\psi} \delta_{\text{face}}\)
for all \(z \in F\) and \(w \in \conv{\vertex{P} \setminus F}\).
As obviously \(\innp{\psi}{y} \geq \innp{\psi}{x} - \dualnorm{\psi}
\norm{x - y} = \innp{\psi}{x} - \dualnorm{\psi} \delta_{\text{face}}\),
this implies \(y \in \argmax_{v \in F} \innp{\psi}{v}\)
and \(x \in \argmin_{v \in \conv{\vertex{P} \setminus F}}\)
with
\(\innp{\psi}{y} = \innp{\psi}{x}
- \dualnorm{\psi} \delta_{\text{face}}\).

As \(y\) lies in the \emph{relative interior} of \(F\), this implies
that actually the product with \(\psi\) is constant on \(F\).
In particular, \(v^{\text{FW}} \in F\)
and \(\innp{\psi}{v^{\text{FW}}} = \innp{\psi}{y}\).
Similarly, \(x\) lies on the face
of
\(\conv{\vertex{P} \setminus F}\)
defined by \(\{z : \innp{\psi}{z} = \innp{\psi}{x}\}\).
Let \(\mathcal{S}\) be the set of vertices of this face,
which ensures \(x \in \conv{\mathcal{S}}\)
and \(\innp{\psi}{v^{\text{A}}} = \innp{\psi}{x}\).
We conclude
\begin{equation}
  \label{eq:9}
  \innp{\psi}{v^{\text{A}} - v^{\text{FW}}}
  = \innp{\psi}{x -  y}
  =
  \delta_{\text{face}}
  \frac{\innp{\psi}{x -  y}}{\norm{x - y}}
  ,
\end{equation}
showing \(\delta \leq \delta_{\text{face}}\).
\end{proof}

\end{proposition}

The pyramidal width is hard to compute in general,
nevertheless it has been determined for simple cases
in \citet[§\,B.1]{lacoste15},
which the above characterization greatly simplifies,
as noticed in \citet[Examples~1 and 2]{pena2018polytope}.
For example, for \(1 \leq p < \infty\) the \(n\)-dimensional
\(0/1\)-hypercube has pyramidal width $1 / \sqrt[1-1/p]{n}$
in the \(\ell_{p}\)-norm and \(1 / n\) in the \(\ell_{\infty}\)-norm.
Here the minimum is attained for the point
\(x = (1/n, 1/n, \dotsc, 1/n)\),
vector \(\psi = (1, \dotsc, 1)\),
and the set of coordinate vectors \(\mathcal{S}
= \{(1, 0, \dotsc, 0), \dotsc, (0, 0, \dotsc, 1)\}\).
See Figure~\ref{fig:PyrWidth} for a graphical representation.
\begin{figure}
  \centering
  \footnotesize
  \begin{tikzpicture}[scale=.7, baseline={(0,0)}]
    \draw (0,0) node[point, label={[left]{\(v^{\text{FW}}\)}}](v-FW){}
    -- (2, 0) node[point, label={[right]{\(v^{\text{A}}\)}}](v-A){}
    |- (0, 2)
    edge[blue] node[pos=0.33, left]{\(\mathcal{S}_{1}\)} (v-A)
    --cycle;
    \draw (1, 1) node[point, fill=white, draw, thin,
    label={[below left]\(x\)}] {}
    edge[vector] node[right, inner sep=1ex]{\(\psi\)} +(.5,.5);
    \draw[alignment line] (v-A) -- (2.75, -0.75) coordinate(end)
    (v-FW) -- (1.75, -1.75) edge[<->, solid]
    node[below right]
    {\(\innp{\psi}{v^{\text{A}} - v^{\text{FW}}}\)} (end);
  \end{tikzpicture}
  \begin{tikzpicture}[scale=.7, baseline={(0,0)}]
    \draw (0,0) node[point, label={[left]{\(v^{\text{FW}}\)}}](v-FW){}
    -| (2, 2) node[point, label={[right]{\(v^{\text{A}}\)}}](v-A){}
    edge[blue] node[pos=0.7, above left]{\(\mathcal{S}_{2}\)} (v-FW)
    -| cycle;
    \draw (1, 1) node[point, fill=white, draw, thin,
    label={[left]\(x\)}] {}
    edge[vector] node[right, inner sep=1ex]{\(\psi\)} +(.5,.5);
    \draw[alignment line] (v-A) -- (3.75, 0.25) coordinate(end)
    (v-FW) -- (1.75, -1.75) edge[<->, solid]
    node[below right]
    {\(\innp{\psi}{v^{\text{A}} - v^{\text{FW}}}\)} (end);
  \end{tikzpicture}
  \caption{Pyramidal width for the hypercube
    per Definition~\ref{def:pyramidal},
    showing the effect of convex decomposition of a point \(x\)
    for a direction (unit vector) \(\psi\)
    (omitting decompositions with superfluous vertices).
    The convex hull of the vertices in the decomposition
    is drawn in blue.}
\label{fig:PyrWidth}
\end{figure}
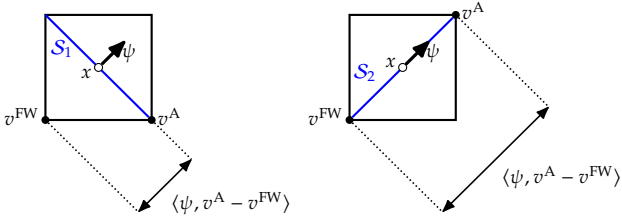
For \(n \geq 2\) in the \(p\)-norm,
the standard simplex (convex hull of \(n\) coordinate vectors)
has pyramidal width
\(\sqrt[p]{\lfloor n/2 \rfloor^{1 - p} +
  \lfloor (n+1)/2 \rfloor^{1 - p}}\)
(and \(1 / \lfloor n/2 \rfloor\) in the \(\ell_{\infty}\)-norm),
and
the \(n\)-dimensional unit \(\ell_{1}\)-ball
(convex hull of \(n\) coordinate vectors and their negatives)
has pyramidal width \((n-1)^{1/p - 1}\)
(and \(1 / (n-1)\) in the \(\ell_{\infty}\)-norm).

As a consequence, the lower bounds on the pyramidal width
via an explicit system of defining inequalities
from \citet{garber2016linearly} and \citet{garber2016linear} follow.
They are similar in style to the following corollary,
using the \(\ell_{1}\)-norm in the numerator, and various estimations
of the norms involved.
\begin{corollary}[Pyramidal width lower bound]
  \label{cor:pyramidal-bound}
  Let \(P\) be a polytope defined by a system \(A x \leq b\)
  of linear inequalities.
  Then its pyramidal width has a lower bound
  \begin{equation}
    \label{eq:pyramidal-bound}
    \delta \geq
    \min_{\substack{z \in P, I: A_{I} z = b_{I} \\
        v \in \vertex{P}: A_{I} v \neq b_{I}}}
    \frac{\norm{b_{I} - A_{I} v}}
    {\norm{A_{I}}}
    ,
  \end{equation}
  where \(I\) ranges over all subsets of rows of \(A\),
  and \(A_{I}, b_{I}\)
  denote submatrices of \(A, b\) restricted to the rows in \(I\).
  The norm of \(A_{I}\)
  is the operator norm.
\end{corollary}
Note that the fraction \(\norm{b_{I} - A_{I} v} / \norm{A_{I}}\)
is a lower bound on
\(\distance{F}{\conv{\vertex{P} \setminus F}}\)
for the face \(F \defeq \{z \in P : A_{I} z = b_{I}\}\).
The norm in the numerator can be arbitrarily chosen,
but the operator norm in the denominator depends on it.

\subsection{Fully-Corrective Frank–Wolfe algorithm}
\label{sec:corrective}

In this section, we present another linearly convergent algorithm,
which is simpler and more natural from a theoretical point than
the Away-step Frank–Wolfe algorithm (Algorithm~\ref{away})
from the previous section.
The \emph{Fully-Corrective Frank–Wolfe algorithm} (FCFW)
(Algorithm~\ref{fcfw}) fully utilizes the answers from the linear
minimization oracle by optimizing the objective function over the
convex hull of all previously encountered feasible points
before each call of the oracle.
Therefore, out of all the Frank–Wolfe variants discussed so far,
this is the one that makes the most progress per linear minimization,
however at the cost of solving a potentially hard subproblem.
Thus, while the improved progress makes the
Fully-Corrective Frank–Wolfe algorithm produce very sparse solutions,
it is often not competitive in wall-clock time in practice.
We will present later an advanced Frank–Wolfe algorithm variant, the
\emph{Blended Conditional Gradient} algorithm in
Section~\ref{sec:bcg}, that optimizes only partially over the convex
hull to account for the cost of minimization over it,
achieving a similar (or not smaller) cost
as the Away-step Frank–Wolfe algorithm (Algorithm~\ref{away}).

\begin{algorithm}\index{Fully-Corrective Frank-Wolfe algorithm}
\caption{Fully-Corrective Frank–Wolfe (FCFW) \citep{Holloway74FW}}
\label{fcfw}
\begin{algorithmic}[1]
  \REQUIRE Start atom $x_0\in P$
  \ENSURE Iterates $x_1, \dotsc \in P$
\STATE$\mathcal{S}_0\leftarrow\{x_0\}$
\FOR{$t=0$ \TO \dots}
\STATE$v_t\leftarrow\argmin_{v\in P}\innp{\nabla f(x_t)}{v}$\label{fcfw_extract}
\STATE$\mathcal{S}_{t+1}\leftarrow\mathcal{S}_t\cup\{v_t\}$
\STATE\label{fcfw_minconv}
  \(x_{t+1} \leftarrow \argmin_{x \in \conv{\mathcal{S}_{t+1}}} f(x)\)
\ENDFOR
\end{algorithmic}
\end{algorithm}

\begin{example}[FCFW convergence]
\label{example:FCFW}
Performance of the Fully-Corrective Frank–Wolfe algorithm
for the problem in Example~\ref{example:AFW}
is illustrated in Figure~\ref{fig:fcfw}.
Recall that the feasible region is
$P =\conv{\{(-1,0),(1,0),(0,1)\}}$,
the objective is $f(x, y) \defeq  2x^{2} + y^{2}$,
and the starting point is $x_{0} = (0,1)$.
The optimal solution is $x^{*} = (0,0)$.

\begin{figure}[b]
\centering
\includegraphics[width=0.5\linewidth, alt={Trajectory of
  Fully-Corrective Frank–Wolfe algorithm on a triangle:
  converging in \(3\) steps to the optimum as expected.}]{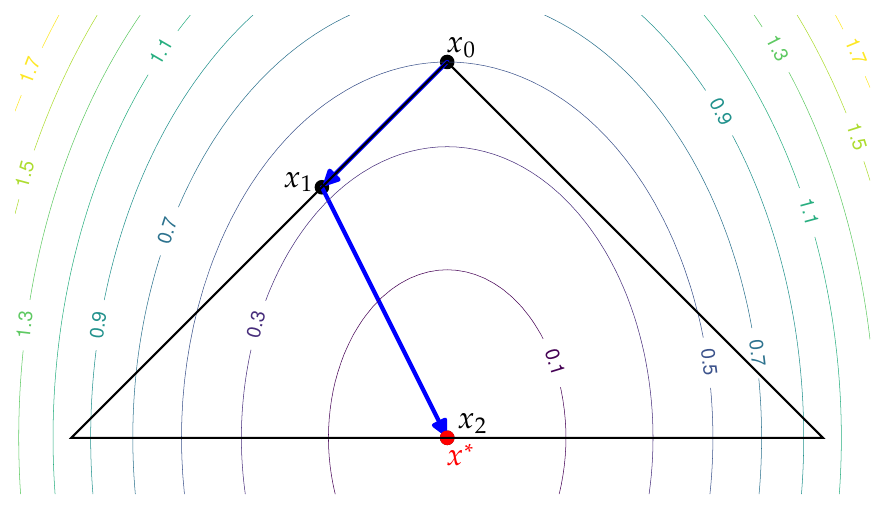}

\caption{The Fully-Corrective Frank–Wolfe algorithm
(Algorithm~\ref{fcfw}) optimizing
\(f(x,y) = 2 x^{2} + y^{2}\)
over the triangle
$P =\conv{\{(-1,0),(1,0),(0,1)\}}$
with starting vertex is $x_{0} = (0,1)$.
The optimum is $x^{*} = (0,0)$.
As the algorithm completely minimizes \(f\) over
all previously seen vertices
at every iteration,
in this simple example the algorithm reaches the optimum
after just two iterations, having seen all vertices.
For comparison, the performance of
the vanilla Frank–Wolfe algorithm (Algorithm~\ref{fw})
and the Away-step Frank–Wolfe algorithm (Algorithm~\ref{away})
minimizing the same function over the same
feasible region can be seen in Figures~\ref{fig:zigzag}
and~\ref{fig:away}, respectively.
\vspace{-1ex}
}
\label{fig:fcfw}
\end{figure}

\end{example}

The algorithm has appeared first in \citet{Holloway74FW} without proof
of linear convergence. The name \emph{fully-corrective} stems from the
full minimization over the convex hull of the current active set and
the first rigorous proof of linear convergence, to the best of
our knowledge, is due to \citet{lacoste15}, via comparison
to the Away-step Frank–Wolfe algorithm (Algorithm~\ref{away}),
as outlined below.
Moreover, note that at least with an exact linear minimization oracle,
the algorithm converges in finitely many steps,
as the optimum \(x^{*}\) will be necessarily in the convex hull
of the vertices returned by the linear minimization oracle at some
point.

\begin{theorem}
\label{th:FCFWConvergence}
Suppose $P$ is a polytope and
$f$ is $L$-smooth and $\mu$-strongly convex over \(P\).
Then the Fully-Corrective Frank–Wolfe algorithm (Algorithm~\ref{fcfw})
satisfies
\begin{equation}
  f(x_t) - f(x^*)
  \leq
  \left(
    1 - \min
    \left\{
      \frac{1}{2},
       \frac{\mu \delta^{2}}{L D^{2}}
    \right\}
  \right)^{t-1}
  \frac{LD^2}{2}
\end{equation}
for all $t \geq 1$, where $D$ and $\delta$ are the diameter
and the pyramidal width of the polytope $P$.
In other words, the algorithm finds a solution of primal gap at most
\(\varepsilon > 0\) after at most the following number of linear
minimizations, gradient computations and \emph{convex optimizations} over
convex hull of some sets:
\begin{equation}
  \label{eq:FCFW}
  1 + \max \left\{
    2,
    \frac{L D^{2}}{\mu \delta^{2}}
  \right\}
  \ln
  \left(
    \frac{L D^{2}}{2 \varepsilon}
  \right).
\end{equation}
\end{theorem}

The bound follows similarly to that of
Theorems~\ref{th:AFWConvergence} or \ref{th:PFWConvergence},
and therefore the proof is omitted.
The main observation is that every iteration makes as much progress
as AFW would with a single Frank–Wolfe step and any number of
immediately following away steps, and hence there are no iterations
(drop steps) where the algorithm does not make provably exponential
progress.

\chapter{Improvements and generalizations of Frank–Wolfe algorithms}
\label{cha:FW-improved}

In the following, we will present several advanced variants of
Frank–Wolfe algorithms that go
beyond the basic ones, offering significant performance improvements in
specific situations.
We start with the easily presented and important ones in
Section~\ref{sec:advanced}.
Further, less important or more complex variants will be presented in
Section~\ref{sec:further-FW}. Algorithms, which use underlying ideas of
conditional gradients, but significantly differ from the basic conditional
gradient template will be relegated to Section~\ref{sec:related}.

\section{Advanced variants}
\label{sec:advanced}

In this section, we present the simplest improvements to Frank–Wolfe
algorithms.  While usually geared towards special cases, e.g.,
strongly convex functions, these variants often (empirically) perform
very well  in other settings. Nonetheless, their worst-case convergence
rates reduce to those of the vanilla Frank–Wolfe algorithm.

\subsection{Pairwise Frank–Wolfe algorithm: a natural
  improvement of Away-step Frank–Wolfe algorithm}
\label{sec:pairwise}

In the Away-step Frank–Wolfe algorithm (Algorithm~\ref{away}) mentioned
before, we  either take an
away-step or a Frank–Wolfe step. Another natural approach is to
combine the two steps, performing a so-called \emph{\myindex{pairwise step}},
which will lead us to the \emph{Pairwise Frank–Wolfe algorithm}
(PFW) (Algorithm~\ref{pairwise}). The Pairwise Frank–Wolfe
algorithm, introduced in \citet{lacoste15} and drawing inspiration from
\citet{mitchell1974finding} and \citet{nanculef2014novel}, also requires
the explicit maintenance of iterates as convex combinations of active
atoms.  The underlying idea is to move weight in the convex
decomposition $x_t = \sum_{v \in \mathcal{S}_{t}} \lambda_{v, t} v$
directly from a \emph{bad} active atom to an atom provided by the
linear optimization oracle.  (For brevity we have omitted the explicit
update rules in Line~\ref{alg:pfwUpdate} of Algorithm~\ref{pairwise},
which are similar in spirit to Algorithm~\ref{away}.)  In contrast to
this, the Away-step Frank–Wolfe algorithm spreads the weight of this
\emph{bad} atom uniformly among all remaining active atoms, before
adding weight to a potentially new atom in the following iteration.

\begin{example}[PFW convergence]
  \label{example:PFW}
  The performance of the Pairwise Frank–Wolfe algorithm with
  line search,
  for the right triangle $P =\conv{\{(-1,0),(1,0),(0,1)\}}$,
  is illustrated in Figure~\ref{fig:pairwise} for
  \(f(x, y) = 2x^{2} + y^{2} \)
  with optimum at \(x^* = (0,0)\)
  and start vertex at $x_0=(0,1)$.

  \begin{figure}[b]
\centering
  \includegraphics[width=0.5\linewidth, alt={Trajectory of Pairwise
    Frank–Wolfe algorithm over a triangle: converging fast even though
    ragged due to the small number of vertices.}]{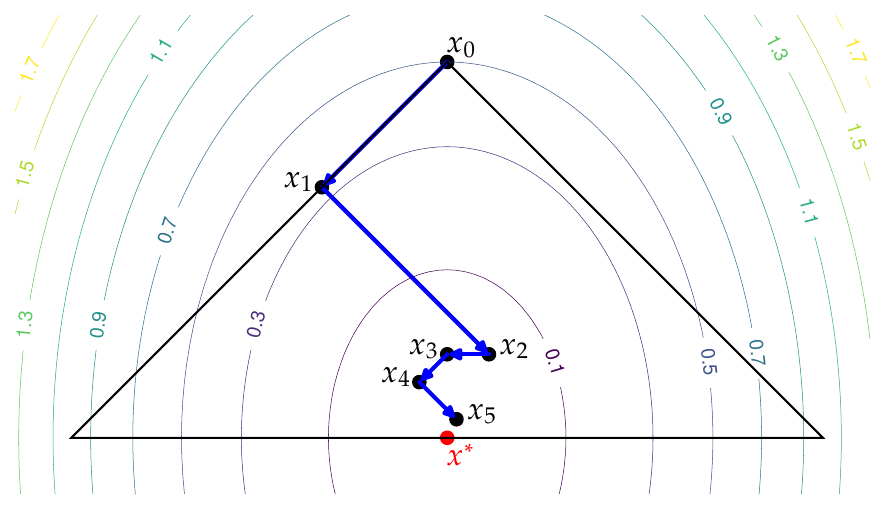}
  \caption{The Pairwise Frank–Wolfe algorithm (Algorithm~\ref{pairwise}) moves only along directions
    determined by two vertices in contrast to, e.g., the vanilla
    Frank–Wolfe algorithm where the direction is towards a vertex,
    and hence also involves the current iterate.
    In the small example here this allows only a
    limited number of directions,
    which is clearly visible in the trajectory.
    The
    feasible region is the triangle
    $P =\conv{\{(-1,0),(1,0),(0,1)\}}$,
    the objective function is \(f(x,y) = 2 x^{2} + y^{2}\).
    The starting vertex is $x_{0} = (0,1)$,
    with following iterates \(x_{1}, x_{2}, \dots\)
    and the optimum is $x^{*} = (0,0)$.
    For comparison, the performance 
    of the vanilla Frank–Wolfe algorithm (Algorithm~\ref{fw}),
    the Away-step
    Frank–Wolfe algorithm (Algorithm~\ref{away}) and the
    Fully-Corrective Frank–Wolfe algorithm (Algorithm~\ref{fcfw})
    minimizing
    the same function over the same feasible region can be seen 
    in Figures~\ref{fig:zigzag}, \ref{fig:away} and~\ref{fig:fcfw}
    respectively.}
  \label{fig:pairwise}
\end{figure}

\end{example}

\begin{algorithm}
  \caption{Pairwise Frank–Wolfe (PFW) \citep{lacoste15}}
\label{pairwise}
\begin{algorithmic}[1]
  \REQUIRE
    Start atom $x_{0}\in P$
  \ENSURE Iterates $x_{1}, \dotsc \in P$
  \STATE $\mathcal{S}_{0} \leftarrow \{x_{0}\}$,
    $\lambda_{x_{0}, 0} \leftarrow 1$
  \FOR{$t=0$ \TO \dots}
    \STATE $v_{t}^{\text{FW}} \leftarrow
      \argmin_{v\in P} \innp{\nabla f(x_{t})}{v}$
    \STATE\label{AwayVertex2}
      $v_{t}^{\text{A}} \leftarrow
      \argmax_{v\in\mathcal{S}_{t}} \innp{\nabla f(x_{t})}{v}$
    \STATE
      \(\gamma_{t} \leftarrow \min \left\{
        \frac{%
          \innp{\nabla f(x_{t})}{v_{t}^{\text{A}} - v_{t}^{\text{FW}}}}
          {L \norm{v_{t}^{\text{A}} - v_{t}^{\text{FW}}}^{2}},
          \lambda_{v_{t}^{\text{A}}, t}
        \right\}\)
      \STATE\label{pairwise_step}
        $x_{t+1} \leftarrow
        x_{t} + \gamma_{t}(v_{t}^{\text{FW}}-v_{t}^{\text{A}})$
        \COMMENT{pairwise step}
      \STATE
        \label{alg:pfwUpdate} Find convex decomposition
        \(x_{t+1} = \sum_{v \in \mathcal{S}_{t+1}} \lambda_{v, t+1}
        v\)
        with \(\mathcal{S}_{t+1} \subseteq
        \mathcal{S}_{t} \cup \{v_{t}^{\text{FW}}\}\).
    \ENDFOR
  \end{algorithmic}
\end{algorithm}

\begin{theorem}
  \label{th:PFWConvergence}
  Suppose $P$ is a polytope and
  $f$ is $L$-smooth and $\mu$-strongly convex over \(P\).
  Then the Pairwise Frank–Wolfe algorithm
  (Algorithm~\ref{pairwise}) satisfies
\begin{equation}
  f(x_t) - f(x^*)
  \leq
  \left(
    1 - \min
    \left\{
      \frac{1}{2},
       \frac{\mu \delta^{2}}{L D^{2}}
    \right\}
  \right)^{(t-1)/(3 \abs{\vertex{P}}! + 1)}
  \frac{LD^2}{2}
\end{equation}
for all $t \geq 0$, where $D$ and $\delta$ are the
diameter and the pyramidal width of the polytope $P$, respectively.
Therefore the primal gap is at most \(\varepsilon\) after at most the
following number of gradient computations and linear optimizations:
\begin{equation}
  \label{eq:PFWconvergence}
  1 +
  \bigl(3 \abs{\vertex{P}}! + 1\bigr)
  \cdot
  \max
    \left\{
      2,
       \frac{L D^{2}}{\mu \delta^{2}}
    \right\}
  \cdot
  \ln \frac{LD^2}{2 \varepsilon}
  .
\end{equation}
\end{theorem}

The proof is similar to that of Theorem~\ref{th:AFWConvergence} for
the Away-step Frank–Wolfe algorithm and is therefore omitted here.  We
only mention that beside drop steps, there is no proof for linear
improvement also for so-called \emph{\myindex{swap step}s},
where one replaces
the away vertex with the Frank–Wolfe vertex, leaving all other
coefficients of the active vertices unchanged. The term
\(3 \abs{\vertex{P}}!\) arises as a conservative bound on the
number of consecutive drop steps and swap steps,
similar to the factor $2$ in
the proof of the Away-step Frank–Wolfe algorithm.
To eliminate the large constant, one might use modified versions of
the algorithm like the one in \citet{rinaldi2020unifyfree},
using the same gradient across successive swap steps
 (see Algorithm~\ref{alg:PFW-SSC} in the next section). In order to get a better understanding of the relative performance of
FW, AFW, PFW, and FCFW, it is helpful to consider a computational
example.

\begin{example}[Performance of the FW, AFW, FCFW and PFW algorithms]
Consider minimizing an $L$-smooth and $\mu$-strongly convex function
$f(x) = \frac{1}{2} \norm[2]{M x}^{2} + \innp{b}{x}$ such that $L/\mu
\approx 10^{6}$ over the \myindex{probability simplex} \(\Delta_{100}\).
The matrix $M$ and the vector $b$ are generated by choosing the entries
uniformly at random between $0$ and $1$. By solving the optimization problem 
to high accuracy we find that the resulting value of $x^*$ for this problem 
is a convex combination of $10$ vertices of the polytope. 
Figure~\ref{fig:comparisonvariants} shows the evolution 
of the primal gap for the different algorithms 
in the number of iterations with both axes in a logarithmic scale.

\begin{figure}
  \centering
  \includegraphics[width=.45\linewidth, alt={An example with primal
    gap decreasing much faster per iteration
    for advanced Frank–Wolfe algorithms
    than for the vanilla algorithm.}]{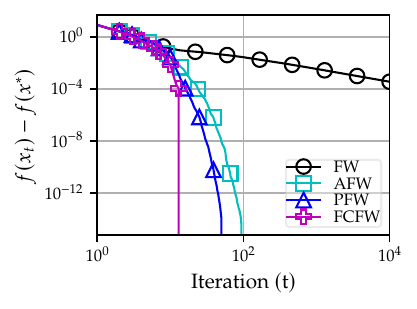}
  \qquad
  \includegraphics[width=.45\linewidth, alt={The previous example
    with primal gap decreasing much faster in wall-clock time, too,
    for advanced Frank–Wolfe algorithms, except the Fully-Corrective
    Frank–Wolfe algorithm,
    than for the vanilla algorithm.}]{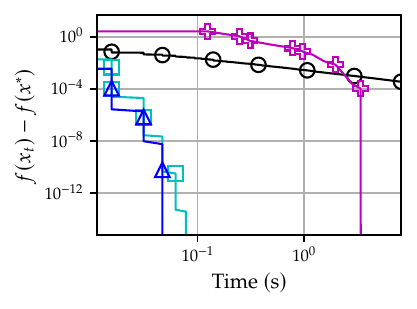}

   \caption{Primal gap convergence for the vanilla
   Frank–Wolfe algorithm
    (FW, Algorithm~\ref{fw}), the Away-step Frank–Wolfe algorithm
    (AFW, Algorithm~\ref{away}),
    the Fully-Corrective Frank–Wolfe algorithm
    (FCFW, Algorithm~\ref{fcfw})
    and the Pairwise Frank–Wolfe algorithm
    (PFW, Algorithm~\ref{pairwise}) in number of iterations and
    wall-clock time for
    minimizing a strongly convex and smooth quadratic function
    over the \(99\)-dimensional probability simplex
    using line search.
    The optimum is a convex combination
    of $10$ vertices of the polytope,
    explaining FCFW reaching the optimum in a small number of
    (possibly costly) iterations.
    Note the linear convergence of AFW and PFW,
  in contrast to the sublinear convergence of FW.}
  \label{fig:comparisonvariants}
\end{figure}

The use of the AFW, PFW, and FCFW algorithms brings a clear advantage
in convergence for this class of functions over the FW
algorithm. Note however that the AFW and PFW algorithms require that
we maintain and update the active set $\mathcal{S}$ (this can become
expensive if the size of the active set gets larger), the FCFW
algorithm, on the other hand requires solving the subproblem in
Line~\ref{fcfw_minconv} of Algorithm~\ref{fcfw} at each iteration,
which can be as computationally expensive as the original problem we
are trying to solve. Finally, note that as the optimum is not contained
in the interior of \(\mathcal{X}\), we cannot expect the FW algorithm
to have linear convergence in primal gap (as we have seen in
Section~\ref{sec:linConvInterior}), whereas the AFW, FCFW, and PFW
algorithms do.
\end{example}

\subsubsection{Improving Pairwise Frank–Wolfe via the
  \myindex{Sufficient Slope Condition}}
\label{sec:SSC}

Some conditional gradient variants (e.g., AFW, PFW, and also
BoostFW as introduced in Section~\ref{sec:boostfw} a little later) might
suffer from steps for which we cannot guarantee sufficient primal
progress within the convergence analysis; these steps are usually
referred to as \emph{bad steps}. For example in the case of the Away-step Frank–Wolfe Algorithm
and the Pairwise Frank–Wolfe Algorithm we 
have no primal progress guarantee for \emph{drop steps} and \emph{swap
  steps}. Usually convergence analyses ignore these bad steps, provided
they ensure at least non-negative (but possibly \(0\)) primal
progress and then argue that they cannot happen ``too often'', so that
in ``most'' steps we make enough progress.  

In this section we present an adaptation of a general enhancement via
the \emph{Sufficient Slope Condition (SSC)} for first-order methods
due to \citet{rinaldi2020unifyfree}. The basic idea of using  SSC is to
chain together bad steps without recomputing the gradient until the SSC is
met; we make this precise below. One might want to think of SSC as a
generalization of the scaling inequality, see e.g.,
Lemma~\ref{lemma:pyraScaling}. As \citet{rinaldi2020unifyfree} do not
provide explicit convergence rates, it is not obvious how to compare
the convergence rate of the enhancement with the original
methods. Nevertheless, as we shall see, in the specific case of the
Pairwise Frank–Wolfe Algorithm the convergence rate improves to
essentially that of the Away-step Frank–Wolfe Algorithm, as the SSC
enhancement eliminates the very loose theoretical bound on the number
of swap steps; for the expert, the key here is that the Frank–Wolfe
vertex remains the same for consecutive swap steps. For the
Away-step Frank–Wolfe Algorithm we are not aware of an improved bound
for the SSC-enhanced algorithm, neither are we for, e.g., the boosted
FW algorithm from Section~\ref{sec:boostfw}. As such, for the sake of
exposition, we simplify the presentation of the SSC enhancement from
\citet{rinaldi2020unifyfree} by specializing it to the Pairwise Frank–Wolfe
Algorithm only, however in a form suggestive of how the SSC
enhancement would work for arbitrary first-order methods. More
recently, it has been shown in \citet{tsuji2021sparser} that a simple
modification of the Pairwise Frank–Wolfe Algorithm is sufficient to
remove swap steps by incorporating ideas from the Blended Conditional
Gradient Algorithm (see Section~\ref{sec:bcg}) leading to an algorithm
with improved convergence and sparsity compared to the Pairwise
Frank–Wolfe Algorithm.

The main idea of SSC can be summarized as follows: In a given iteration either we made enough
primal progress so that it is warranted to compute a new gradient etc
or we made little to no primal progress, which (under the respective
assumptions) implies that we also did not move far from the last
iterate, so that we can reuse the gradient and compute a new step. As
tiny steps remain close to the previous iterate, the gradient should
not change much, so by reusing the previous gradient one exchanges a
gradient computation for a sufficiently small error. A gradient is
reused until we moved far enough from the iterate from where
we computed the same gradient; then we have also made enough primal
progress even if it has been achieved with a myriad of tiny steps.
The important feature of this modification is that it provides good
convergence rate in the \emph{number of gradient
  computations}. However, for an overall convergence rate in
the number of \emph{steps} we also have to bound the number of 
tiny steps taken as well, which we do here for the Pairwise Frank–Wolfe variant with SSC but remains challenging in other scenarios.

The algorithm is presented in Algorithm~\ref{alg:PFW-SSC}.
Intuitively, the pairwise direction \(d_{t, i}\)
is the opposite direction where
the Pairwise Frank–Wolfe Algorithm
would go to from the point \(y_{t, i}\)
provided the objective function has gradient \(\nabla f(x_{t})\)
there.
The maximal step size \(\gamma_{t, i}\)
is the maximum of step sizes the algorithm would choose
for any objective function and only depends on the active set.
More precisely this means that,
the pairwise direction is \(d_{t, i} = v^{\text{A}} - v^{\text{FW}}\)
where
\(v^{\text{A}}\) and \(v^{\text{FW}}\)
are the away vertex and Frank–Wolfe vertex at \(y_{t, i}\)
but using the gradient \(\nabla f(x_{t})\) at \(x_{t}\)
instead of \(y_{t, i}\).
The maximal pairwise step size \(\gamma_{t, i, \max}\)
is the coefficient of the away vertex \(v^{\text{A}}\)
in the convex decomposition of \(y_{t, i}\)
into vertices of the polytope \(P\).
This convex combination should be maintained as in
the original Pairwise Frank–Wolfe Algorithm
(Algorithm~\ref{pairwise}) or the Away-step Frank–Wolfe Algorithm (Algorithm~\ref{away}) and is not repeated here.
The convex set \(\Omega_{t, i}\) is a small ball around \(x_{t}\)
of the points not far away from \(x_{t}\)
where the last gradient is reused.
The original algorithm uses only the last iterate for the radius of
the ball
\(\norm{x_{t} - z}
\leq \innp{\nabla f(x_{t})}{d_{t,i}} \mathbin{/} (L \norm{d_{t, i}})\)
and cuts it off with another restriction
\smash{\(\innp{\nabla f(x_{t})}{x_{t} - z} \geq L \norm{x_{t} - z}^{2}\)}.
To streamline the presentation
we have decreased the radius of the ball instead so that the last
inequality is implied.
This simplification also makes computing the step size bound
\(\alpha_{t, i}\) less expensive.
This change is the only difference from the original algorithm
in \citet{rinaldi2020unifyfree}. 

\begin{algorithm}
  \caption{Pairwise Frank–Wolfe with SSC \citep{rinaldi2020unifyfree}}
  \label{alg:PFW-SSC}
  \begin{algorithmic}[1]
    \REQUIRE
      Start atom $x_{0}\in P$
    \ENSURE Iterates $x_{1}, \dotsc \in P$
    \FOR{$t=0$ \TO \dots}
      \STATE \(y_{t, 0} \leftarrow x_{t}\)
      \FOR{\(i=0\) \TO \dots}
        \STATE \label{line:pfw-ssc-direction} \(d_{t, i} \leftarrow\) pairwise direction
          at \(y_{t, i}\) for \(\nabla f(x_{t})\)
        \STATE \(\gamma_{t, i, \max} \leftarrow\)
          maximal pairwise step size
        \STATE \(\Omega_{t, i} \leftarrow
          \{ z \in P \mid
          \norm{x_{t} - z}
          \leq
          \min_{0 \leq j \leq i}
          \innp{\nabla f(x_{t})}{d_{t,j}}
          \mathbin{/} (L \norm{d_{t,j}})
          \}\)
        \IF{\(y_{t, i} \notin \Omega_{t, i}\) \OR
              \(d_{t, i} = \allZero\)}
          \STATE \(x_{t+1} \leftarrow y_{t, i}\)
          \BREAK
        \ELSE
	\STATE \label{line:alphaSearch} \(\alpha_{t,i} \leftarrow
            \max \{ \gamma : \gamma > 0, \
            y_{t, i} - \gamma d_{t,i} \in \Omega_{t,i}\}\)
          \STATE
            \(\gamma_{t, i} \leftarrow \min \left\{
              \gamma_{t, i, \max},
              \alpha_{t. i}
            \right\}\)
          \STATE
            $y_{t, i+1} \leftarrow y_{t, i} - \gamma_{t, i} d_{t, i}$
          \IF{\(\gamma_{t, i, \max} \geq \alpha_{t,i}\)}
            \STATE \(x_{t+1} \leftarrow y_{t, i+1}\)
            \BREAK
          \ENDIF
        \ENDIF
      \ENDFOR
    \ENDFOR
  \end{algorithmic}
\end{algorithm}

Note that in Line~\ref{line:pfw-ssc-direction}
the Frank–Wolfe vertex remains
the same and only the away vertex changes via
the inner iterations.  Moreover, the additional step size
bound \(\alpha_{t, i}\) in Line~\ref{line:alphaSearch} in
Algorithm~\ref{alg:PFW-SSC} is hard to compute for
some norms like \(\ell_{p}\)-norms in general (easy however for
\(p=1, 2, \infty\)).  Regardless, it is easy to verify that, as
\(y_{t,0} = x_{t}\), the first iteration of the inner loop makes
exactly a pairwise step with the short step rule, since
\(\alpha_{t, 0} = \innp{\nabla f(x_{t})}{d_{t, 0}} / (L \norm{d_{t,
    0}}^{2})\).  In particular, the algorithm deviates from the standard
Pairwise Frank–Wolfe Algorithm only where
theoretical guarantee for primal progress is insufficient,
i.e., after a swap step, replacing a
sequence of swap and drop steps with a sequence of a more efficient
variant of drop steps.

\begin{theorem}
  \label{thm:PFW-SSC-rate}
  Let \(f\) be a \(\mu\)-strongly convex,
  \(L\)-smooth function over a polytope \(P\)
  with diameter \(D\) and pyramidal width \(\delta\).
  Then for every outer iteration \(t \geq 1\) of the
  Pairwise Frank–Wolfe Algorithm with SSC
  (Algorithm~\ref{alg:PFW-SSC}) we have:
  \begin{equation}
    \label{eq:PFW-SSC-rate-iter}
    h_{t}
    \leq
    \left(
      \frac{1}{1 + \frac{\mu}{L} \cdot
        \frac{\delta^{2}}{(\delta + D)^{2}}}
    \right)^{t-1}
    \cdot \frac{L D^{2}}{2}
    .
  \end{equation}
  Thus the primal gap is at most \(\varepsilon\) after at most
  the following number of gradient computations and
  linear optimizations
  and at most twice as many inner iterations, or more precisely, step size limit computations
  (i.e., computations of the \(\alpha_{t, i}\)):
  \begin{equation}
    \label{eq:PFW-SSC-rate}
    1 +
    \left(
      1 + \frac{L}{\mu} \cdot
      \frac{2 D^{2} + \delta^{2}}{\delta^{2}}
    \right)
    \cdot \ln \frac{L D^{2}}{2 \varepsilon}
    .
  \end{equation}
\end{theorem}
The proof largely follows \citet{locatello17mpfw}.
The only major deviation here is the use of primal gap as
progress measure instead of distance of iterates,
as customary for conditional gradient algorithms.
We start with a series of lemmas showing
increasingly more progress in function value.
Recall that \(x_{t+1} = y_{t, i}\) for the largest \(i\)
for which the algorithm generates \(y_{t, i}\).
\begin{lemma}[Monotonicity]
  \label{lem:PFW-SSC-monotone}
  The iterates of Algorithm~\ref{alg:PFW-SSC}
  are monotonically decreasing in function value:
  \(f(y_{t, i}) \geq f(y_{t, i+1})\) for all outer iteration \(t\)
  and inner iteration \(i\).
  In particular, \(f(y_{t, i}) \geq f(x_{t+1})\).
\end{lemma}
\begin{proof}
Using \(L\)-Lipschitz continuity of gradients and
\(y_{t, i+1} \in \Omega_{t, i}\) we obtain
\begin{equation}
  \label{eq:10}
  \innp{\nabla f(y_{t, i+1})}{d_{t, i}}
  \geq
  \innp{\nabla f(x_{t})}{d_{t, i}}
  -
  L \norm{x_{t} - y_{t, i+1}} \cdot \norm{d_{t, i}}
  \geq
  0
  .
\end{equation}
Therefore
\begin{equation*}
  f(y_{t, i}) - f(y_{t, i+1})
  \geq
  \innp{\nabla f(y_{t, i+1})}{y_{t, i} - y_{t, i+1}}
  =
  \gamma_{t, i} \innp{\nabla f(y_{t, i+1})}{d_{t, i}}
  \geq
  0
  .
  \qedhere
\end{equation*}
\end{proof}

For progress compared to the start of an inner loop,
we show more progress than just monotonicity.
Note that the bound also holds even for the last \(y_{t, i}\)
even though there might not be an inner iteration \(i\).
\begin{lemma}[Progress]
  \label{lem:PFW-SSC-progress}
  For every outer iteration \(t\) and all \(i\)
  for which \(y_{t, i}\) is defined
  we have
  the following lower bound on primal and dual progress:
  \begin{align}
    \label{eq:PFW-SSC-progress}
    f(x_{t}) - f(y_{t, i})
    &
    \geq
    \frac{\innp{\nabla f(x_{t})}{x_{t} - y_{t, i}}}{2}
    \geq
    \frac{L \norm{x_{t} - y_{t, i}}^{2}}{2}
    ,
    \\
    \label{eq:PFW-SSC-FW-gap}
    \innp{\nabla f(x_{t})}{x_{t} - y_{t, i}}
    &
    \geq
    \norm{x_{t} - y_{t, i}}
    \cdot
    \min_{0 \leq j < i}
    \frac{\innp{\nabla f(x_{t})}{d_{t, j}}}{\norm{d_{t, j}}}
    .
  \end{align}
\end{lemma}
\begin{proof}
The core of the proof is showing Equation~\eqref{eq:PFW-SSC-FW-gap},
using \(x_{t} = y_{t, 0}\) and
\(y_{t, j} - y_{t, j+1} = \gamma_{t, j} d_{t, j}\).
\begin{equation}
 \begin{split}
  \innp{\nabla f(x_{t})}{x_{t} - y_{t, i}}
  &
  =
  \innp*{\nabla f(x_{t})}{\sum_{j=0}^{i-1} \gamma_{t,j} d_{t, j}}
  =
  \sum_{j=0}^{i-1} \gamma_{t,j} \innp{\nabla f(x_{t})}{d_{t, j}}
  \\
  &
  \geq
  \sum_{j=0}^{i} \gamma_{t,j} \norm{d_{t, j}}
  \cdot
  \min_{0 \leq j < i}
  \frac{\innp{\nabla f(x_{t})}{d_{t, j}}}{\norm{d_{t,j}}}
  \\
  &
  \geq
  \norm*{\sum_{j=0}^{i-1} \gamma_{t,j} d_{t, j}}
  \cdot
  \min_{0 \leq j < i}
  \frac{\innp{\nabla f(x_{t})}{d_{t, j}}}{\norm{d_{t,j}}}
  \\
  &
  =
  \norm{x_{t} - y_{t, i}}
  \cdot
  \min_{0 \leq j < i}
  \frac{\innp{\nabla f(x_{t})}{d_{t, j}}}{\norm{d_{t,j}}}
  .
 \end{split}
\end{equation}
Using \(\norm{x_{t} - y_{t, i}} \leq
\min_{0 \leq j < i}
\innp{\nabla f(x_{t})}{d_{t,j}} \mathbin{/} (L \norm{d_{t,j}})\)
also ensured by the algorithm,
we obtain
\begin{equation}
  \innp{\nabla f(x_{t})}{x_{t} - y_{t, i}}
  \geq
  \norm{x_{t} - y_{t, i}}
  \cdot
  \min_{0 \leq j < i}
  \frac{\innp{\nabla f(x_{t})}{d_{t, j}}}{\norm{d_{t,j}}}
  \geq
  L \norm{x_{t} - y_{t, i}}^{2}
  .
\end{equation}
Together with \(L\)-smoothness of \(f\) we conclude
\begin{equation}
  \label{eq:22}
 \begin{split}
  f(x_{t}) - f(y_{t, i})
  &
  \geq
  \innp{\nabla f(x_{t})}{x_{t} - y_{t,i}}
  -
  \frac{L \norm{x_{t} - y_{t, i}}^{2}}{2}
  \\
  &
  \geq
  \frac{\innp{\nabla f(x_{t})}{x_{t} - y_{t,i}}}{2}
  \\
  &
  \geq
  \frac{L \norm{x_{t} - y_{t, i}}^{2}}{2}
  .
  \qedhere
 \end{split}
\end{equation}
\end{proof}

Now we combine these to estimate the total progress over a single inner loop.
\begin{lemma}[Main estimation]
  \label{lem:PFW_SSC-main}
  For every outer iteration \(t\)
  if the inner loop terminates,
  there is an inner iteration \(i\)
  satisfying either
  \(x_{t+1} = \argmin_{z \in P} \innp{\nabla f(x_{t})}{z}\) or
  \begin{equation}
    \label{eq:18}
    f(x_{t}) - f(x_{t+1})
    \geq
    \frac{\innp{\nabla f(x_{t})}{d_{t, i}}^2}{2 L \norm{d_{t, i}}^{2}}
    .
  \end{equation}
\end{lemma}
\begin{proof}
For a fixed outer iteration \(t\),
let \(i\) be the inner iteration during which the inner loop
terminates.
There are several cases depending on how the termination occurs.

If \(d_{t, i} = 0\) then
\(x_{t+1} = y_{t, i} = \argmin_{z \in P} \innp{\nabla f(x_{t})}{z}\),
and the lemma clearly follows. 

For the other cases we first show that
\(\norm{x_{t} - x_{t+1}}
\geq
\innp{\nabla f(x_{t})}{d_{t, j}} \mathbin{/} (L \norm{d_{t, j}})\)
for some \(0 \leq j \leq i\).
If \(y_{t, i} \notin \Omega_{t, i}\)
then \(x_{t+1} = y_{t,i}\) and \(\norm{x_{t} - x_{t+1}}
= \norm{x_{t} - y_{t,i}} \geq
\min_{0 \leq j \leq i}
\innp{\nabla f(x_{t})}{d_{t,j}} \mathbin{/} (L \norm{d_{t,j}})\),
and the claim follows.
If \(\alpha_{t, i} \leq \gamma_{t, i, \max}\)
then \(\gamma_{t, i} = \alpha_{t, i}\)
and \(x_{t+1} = y_{t, i+1}\) lies on the boundary of \(\Omega_{t,i}\),
i.e.,
\(\norm{x_{t} - x_{t+1}}
= \min_{0 \leq j \leq i}
\innp{\nabla f(x_{t})}{d_{t, j}} \mathbin{/} (L \norm{d_{t, j}})\),
showing the claim.

Thus in all cases \(\norm{x_{t} - x_{t+1}} \geq
\innp{\nabla f(x_{t})}{d_{t, j}} \mathbin{/} (L \norm{d_{t, j}})\)
for some \(j\),
hence by Lemma~\ref{lem:PFW-SSC-progress}
\begin{equation*}
  f(x_{t}) - f(x_{t+1})
  \geq
  \frac{L \norm{x_{t} - x_{t+1}}^{2}}{2}
  \geq
  \frac{\innp{\nabla f(x_{t})}{d_{t, j}}^{2}}{2 L \norm{d_{t, j}}^{2}}
  .
  \qedhere
\end{equation*}
\end{proof}

Finally, we combine these estimates with
ideas of the Pairwise Frank–Wolfe Algorithm,
like the scaling inequality (Lemma~\ref{lemma:pyraScaling}),
to obtain a directly interpretable progress estimate,
i.e., using only parameters independent of the algorithm.
\begin{proof}[Proof of Theorem~\ref{thm:PFW-SSC-rate}]
First we show that
at any point in the algorithm, the total number of inner
iterations is at most twice the total number of outer iterations,
counting also the iterations in progress.  To this end we investigate
the number of active vertices, i.e., the vertices of \(P\)
appearing in the convex decomposition of \(y_{t,i}\).

During an outer iteration \(t\),
there is at most one new active vertex
(the Frank–Wolfe vertex for \(\nabla f(x_{t})\)),
while every inner iteration \(i\) except the last one
removes an active vertex (the away vertex)
by the choice of the step size \(\gamma_{t, i} = \gamma_{t, i, \max}\)
as the maximal pairwise step size.
Denoting by \(N\) the total number of inner iterations
(including the ones from past outer iterations) up to some point in
outer iteration \(t\), the number of active vertices is at most
\(1 + t - (N - t)\).
As there is always at least one active vertex, we have
\(1 + t - (N - t) \geq 1\), i.e., \(N \leq 2 t\) as claimed.

Now let \(t\) be a fixed outer iteration.
We shall show the following,
from which the claimed convergence rate follows.
The initial bound \(h_{1} \leq L D^{2} / 2\) follows from
the very first iteration being identical to that of
the Pairwise Frank–Wolfe Algorithm (Theorem~\ref{th:PFWConvergence}),
as mentioned above just before Theorem~\ref{thm:PFW-SSC-rate},
hence
\(f(x_{1}) - f(x^{*}) \leq f(y_{0, 1}) - f(x^{*}) \leq L D^{2} / 2\).

We will establish the following contraction for the outer iterations:
\begin{equation}
  \label{eq:PFW-SSC-rate-linear}
  h_{t+1}
  \leq
  \frac{1}{1 + \frac{\mu}{L} \cdot
    \frac{\delta^{2}}{(\delta + D)^{2}}}
  \cdot h_{t}
  .
\end{equation}
By the above, the inner loop terminates
in some iteration \(i\).
We apply Lemma~\ref{lem:PFW_SSC-main}.
If \(x_{t+1} = \argmin_{z \in P} \innp{\nabla f(x_{t})}{z}\)
then by Lemma~\ref{lem:PFW-SSC-progress}
\begin{equation}
  \label{eq:21}
  f(x_{t}) - f(x_{t+1})
  \geq
  \frac{\innp{\nabla f(x_{t})}{x_{t} - x_{t+1}}}{2}
  \geq
  \frac{\innp{\nabla f(x_{t})}{x_{t} - x^{*}}}{2}
  \geq
  \frac{h_{t}}{2}
  ,
\end{equation}
i.e., \(h_{t+1} \leq h_{t} / 2\),
and Equation~\eqref{eq:PFW-SSC-rate-linear} follows
via \(\mu \leq L\).

Otherwise \(f(x_{t}) - f(x_{t+1}) \geq
\frac{\innp{\nabla f(x_{t})}{d_{t, i}}^2}{2 L \norm{d_{t, i}}^{2}}\)
for some inner iteration \(i\).
Using the scaling inequality (Lemma~\ref{lemma:pyraScaling}),
\(L\)-Lipschitz continuous gradients (via \(L\)-smoothness,
see Lemma~\ref{lem:smooth-Lipschitz})
and strong convexity (Lemma~\ref{lem:SCprimal})
we obtain
\begin{equation}
  \label{eq:26}
 \begin{split}
  \innp{\nabla f(x_{t})}{d_{t, i}}
  &
  \geq
  \delta
  \frac{\innp{\nabla f(x_{t})}{y_{t, i} - x^{*}}}{\norm{y_{t, i} - x^{*}}} \\
  &
  \geq
  \delta
  \frac{\innp{\nabla f(y_{t, i})}{y_{t, i} - x^{*}}}{\norm{y_{t, i} - x^{*}}}
  - L \delta \norm{x_{t} - y_{t,i}}
  \\
  &
  \geq
  \delta \sqrt{2 \mu (f(y_{t, i}) - f(x^{*}))}
  - L \delta \norm{x_{t} - y_{t,i}}
  .
 \end{split}
\end{equation}
Note that the second term above is the error we incur due to reusing
the gradient $\nabla f(x_t)$ rather than
computing $\nabla f(y_{t, i})$.
By rearranging and using
\begin{equation*}
  f(x_{t}) - f(x_{t+1}) \geq
  \frac{\innp{\nabla f(x_{t})}{d_{t, i}}^2}{2 L \norm{d_{t, i}}^{2}}
  \geq \frac{\innp{\nabla f(x_{t})}{d_{t, i}}^2}{2 L D^{2}}
\end{equation*}
(as \(\norm{d_{t, i}} \leq D\))
and Equation~\eqref{eq:PFW-SSC-progress} from
Lemma~\ref{lem:PFW-SSC-progress},
together with monoticity
\(f(y_{t, i}) \geq f(x_{t+1})\) (Lemma~\ref{lem:PFW-SSC-monotone})
we obtain 
\begin{equation}
  \label{eq:27}
 \begin{split}
  \delta \sqrt{2 \mu \cdot \bigl(f(x_{t+1}) - f(x^{*})\bigr)}
  &
  \leq
  \delta \sqrt{2 \mu \cdot \bigl(f(y_{t, i}) - f(x^{*})\bigr)} \\
  &
  \leq
  L \delta \norm{x_{t} - y_{t,i}} + \innp{\nabla f(x_{t})}{d_{t, i}}
  \\
  &
  \leq
  \delta \sqrt{2 L \bigl(f(x_{t}) - f(y_{t, i})\bigr)}
  + D \sqrt{2 L \bigl(f(x_{t}) - f(x_{t+1})\bigr)}
  \\
  &
  \leq
  (\delta + D) \sqrt{2 L \bigl(f(x_{t}) - f(x_{t+1})\bigr)}
  .
 \end{split}
\end{equation}
Hence Equation~\eqref{eq:PFW-SSC-rate-linear} follows.
\end{proof}

\subsubsection{Decomposition-Invariant Conditional Gradient algorithm}
\label{sec:decomposition-invariant}

Recall that the Away-step and Pairwise Frank–Wolfe algorithms maintain an
explicit convex decomposition of iterates, which is sometimes a disadvantage as
it might introduce a significant computational overhead.  To avoid the
decomposition, \citet{garber2016linear} and the follow-up work
\citet{bashiri2017decomposition} proposed searching for the away vertex in the
minimal face \(\face_{\mathcal{X}}(x)\) containing the current iterate
\(x_{t}\). More precisely, rather than solving the away-step linear optimization
problem over $\conv{\mathcal S}$, where $\mathcal S$ is the maintained active
set of $x_t$, it is solved over \(\face_{\mathcal{X}}(x)\); see Line~\ref{alg:Line:optimal_face_problem} of Algorithm~\ref{DIPFW}. The so-obtained
away vertex can be easily shown to be an away vertex for \emph{some}
active
set; this is where the term \emph{decomposition-invariant} stems from. This
class of algorithms usually requires two additional oracles: an away-step
oracle solving the optimization problem over \(\face_{\mathcal{X}}(x)\) and a
line-search oracle providing the maximum step size for away-vertices to remain
feasible as we do not have an active set anymore where we can read-off the
weight of the away vertex.
In other words, the main difficulty is ensuring that
away steps and pairwise steps remain in the
feasible region, while still making significant progress,
as the convex decomposition is no longer available to compute
a maximal step size.
No general solution is known to this difficulty.
One approach is doing line search,
but it requires feasibility testing,
which is not completely trivial for some polytopes.

We will follow the approach of \citet{garber2016linear} here, restricting ourselves to the important class of simplex-like integral polytopes
and special step sizes, in order to help us maintain feasibility.
The result is the Decomposition-invariant Pairwise
Conditional Gradient algorithm of \citet{garber2016linear}
presented in Algorithm~\ref{DIPFW} 
\citep[using the form proposed by][]{bashiri2017decomposition},
and a similar Decomposition-invariant Away-step
Frank–Wolfe algorithm of \citet{bashiri2017decomposition} (which we omit).
A simplex-like polytope $\mathcal{X} $ has the following form
\begin{equation*}
  \mathcal{X} = \{x \in \mathbb{R}^n \mid x \geq 0, Ax = b \}
\end{equation*}
with only \(0/1\)-vertices, i.e., every coordinate of a vertex is
either \(0\) or \(1\).
The experts note that this does \emph{not} cover all polytopes
by projection as we require the vertices being in $\{0,1\}^{n}$.
A few important examples are the probability simplex, the
Birkhoff polytope (see Example~\ref{ex:Birkhoff})
and the bipartite matching polytope.
We briefly recall the argument ensuring feasibility of the iterates.
Step sizes \(\gamma_{t}\) are chosen to be non-increasing and of the form \(2^{-k}\)
for some natural number \(k\), so that the coordinates of the iterate
\(x_{t}\) are integral multiples of \(\gamma_{t}\). In particular, via an 
easy induction, the non-zero
coordinates of \(x_{t}\) are at least \(\gamma_{t}\).
As all vertices are \(0/1\)-vectors, we have
\begin{equation*}
  x_{t+1} = x_{t} + \gamma_{t} \bigl(v_{t}^{\text{FW}} - v_{t}^{\text{A}}\bigr)
  \geq 0,
\end{equation*}
noting that for coordinates \(x_{t, i} = 0\)
we have \(v_{t, i}^{\text{A}} = 0\)
as \(v_{t}^{\text{A}}\) lies on the minimal face containing \(x_{t}\) by construction.
This argument explicitly uses the condition \(x \geq 0\) in the definition of 
simplex-like polytope \(\mathcal{X}\).

A novel and important feature of Algorithm~\ref{DIPFW} is that the linear convergence rate it achieves for smooth and strongly convex functions depends on the number of non-zero
coordinates of the optimal solution $x^{*}$ (see Theorem~\ref{th:diPFW})
instead of the pyramidal width.
Thus the convergence bound is better for \emph{sparse} optima,
which likely occurs only when the simplex like presentation of
\(\mathcal{X}\)
is given by a set of natural constraints,
like for the examples mentioned above.

To loop back to what we have seen before, sparsity is actually an
upper bound on a \emph{local} variant of pyramidal width,
replacing the pyramidal width
in the scaling inequality \eqref{eq:pyraScale},
namely, by \citet[Lemma~2]{garber2016linear}:
\begin{equation}
  \innp{\psi}{v^{\text{A}} - v^{\text{FW}}}
  \geq
  \frac{1}{\sqrt{\sparsity{y}}}
  \cdot
  \frac{\innp{\psi}{x -  y}}{\norm[2]{x - y}}
  ,
\end{equation}
where \(\sparsity{y}\) denotes the number of non-zero coordinates of \(y\).
(Note that if \(y = 0\) then \(\mathcal{X} = \{0\}\).
We exclude this trivial case to avoid dividing by zero.)
We remark that sparsity as treated here is not symmetric under
coordinate-flips, i.e., under transformations $x_i \mapsto 1 - x_i$.
It is an open question whether the sparsity notion can be improved to,
e.g., the number of fractional variables.

\begin{algorithm}
  \caption{Decomposition-invariant Pairwise Frank–Wolfe
    (DI-PFW) \citep{garber2016linear, bashiri2017decomposition}}
\label{DIPFW}
\begin{algorithmic}[1]
  \REQUIRE Vertex $x_0$ of $\mathcal{X}$,
    a sequence of step sizes
    \(\gamma_{t}\) for \(t \geq 1\) with
    \(1 / \gamma_{1}\) and \(\gamma_{t} / \gamma_{t+1}\)
    being integers for \(t \geq 1\)
  \ENSURE Iterates $x_1, \dotsc \in \mathcal{X}$
\FOR{$t=0$ \TO \dots}
\STATE$v_t^\text{FW}\leftarrow\argmin_{v\in \mathcal{X}}\innp{\nabla f(x_t)}{v}$ \label{alg:Line:normal_problem}
  \STATE \(v_{t}^{\text{A}} \leftarrow
    \argmax_{v \in \face_{\mathcal{X}}(x_{t})} \innp{\nabla f(x_{t})}{v}\)
    \COMMENT{\(\face_{\mathcal{X}}(x)\) is the minimal face containing
      \(x\).} \label{alg:Line:optimal_face_problem}
  \STATE $x_{t+1} \leftarrow x_{t} + \gamma_{t} (v_{t}^{\text{FW}}
    - v_{t}^{\text{A}})$
\ENDFOR
\end{algorithmic}
\end{algorithm}

We recall the convergence rate from
\citet[Theorem~1]{garber2016linear}.

\begin{theorem}
  \label{th:diPFW}
  Let \(f\) be a \(\mu\)-strongly convex, \(L\)-smooth convex function
  in the Euclidean norm over a simplex-like polytope \(\mathcal{X}\).
  Let $\sparsity{x^*}$ denote the number of non-zero elements of $x^*$.
  Running the Decomposition-Invariant PFW algorithm \ref{DIPFW}
  with step sizes
  \begin{equation}
    \gamma_{t} = \max_{k=0, 1, \dotsc} \left\{ 2^{-k} \,\middle|\,
      2^{-k} \leq
      \sqrt{\frac{\mu}{16 L D^{2} \cdot \sparsity{x^{*}}}}
      \left( 1 -
        \frac{\mu}{16 L D^{2} \cdot \sparsity{x^{*}}}
      \right)^{\frac{t-1}{2}}
    \right\}
    ,
  \end{equation}
  with \(\norm[0]{x}\) denoting the number of non-zero coordinates of \(x\),
  the primal gap satisfies for all $t\geq 1$:
  \begin{equation}
    f(x_{t}) - f(x^{*})
    \leq
    \frac{L D^{2}}{2}
    \exp \left( - \frac{\mu}{8 L D^{2} \sparsity{x^*}} t \right).
  \end{equation}
  Equivalently, the primal gap is at most \(\varepsilon\)
  with at most the following number of linear minimizations
  (including minimizations over minimal faces \(\face_{\mathcal{X}}(x_{t})\)):
  \begin{equation}
    \frac{8 L D^{2} \sparsity{x^*}}{\mu}
    \ln \frac{L D^{2}}{2 \varepsilon}
    .
  \end{equation}
\end{theorem}

This algorithm is mostly beneficial for problems with a small simplex-like
representation and it requires the knowledge of $L$, $\mu$, and $D$ as well as the
sparsity of the optimal solution \(\sparsity{x^{*}}\) in order to compute the step
sizes. While knowing  $L$, $\mu$, and $D$ can be realistic in some cases,
the knowledge of the sparsity is more problematic.
A good estimate of sparsity is crucial though to
realize convergence in sparsity as
the algorithm effectively provides guarantees in this estimate only and not the
true parameter. In many important cases one can get around these parameter-estimation 
issues by substituting the step size rule for a line search. Nonetheless, for more general problems we face several challenges. For
problems with large simplex-like representation,
as is the case for the (non-bipartite)
matching polytope \citep{rothvoss14} and its approximations \citep{braun2014matching} needing an exponential number of extra coordinates, sparsity
is hard to interpret for the original problem, and the simplex-like
representation is impractical for solving the LP subproblems. Computing the
away vertex has the same challenges as computing the in-face directions for the In-Face
Extended Frank–Wolfe algorithm (Algorithm~\ref{extendedInFace}), as we
will see later on.  In addition, as
the vertices are assumed to be integral, often the away-step optimization
problem becomes an integer program. Then we run into a complexity-theoretic
issue: linear optimization problems with integrality constraints tend to be
NP-hard. As such, unless co-NP is equal to NP (which is generally believed to
be not the case), the away step oracle cannot be implemented as we will not
have concise access to the describing constraints required to confine the
optimization to the face of the current iterate
(see~\citet{diakonikolas2020locally} for a discussion of this problem).

One the other hand, if we have a problem that has a favorable structure --~and
there are many important ones, e.g., in machine learning~-- then the performance
of DI-PFW can be quite remarkable and usually making it the fastest
conditional gradient algorithms for this problem class. This is not surprising:
we essentially have an algorithm with per-iteration convergence higher than
that of PFW \emph{and} we do not have to maintain an active set. It remains an
open question whether one can find a conditional gradient
algorithm without active set
that circumvents these issues and still guarantees linear convergence
for sharp or strongly convex functions.

\begin{example}[A numerical run of the PFW and DI-PFW algorithms]
  \label{ex:Birkhoff}
  In Figure~\ref{fig:DIPFWPrimalComparison} we compare 
  the primal gap evolution of the pairwise-step FW (PFW) and the
  Decomposition-invariant pairwise Conditional Gradient (DI-PFW)
  algorithm when both algorithms are run using \emph{line search},
  so that there is no need to estimate sparsity \(\norm[0]{x^{*}}\).
  The feasible region $\mathcal X$
  is the \myindex{Birkhoff polytope} in $\mathbb{R}^{n \times n}$,
  the set of doubly stochastic matrices\index{doubly stochastic
    matrix|see {Birkhoff polytope}},
  i.e., nonnegative matrices, where for each row and column
  the entries sum up to \(1\).
  Here however it will be convenient to disregard the matrix structure,
  therefore we treat every matrix in $\mathbb{R}^{n \times n}$
  as a vector of its entries,
  i.e., an element of \(\mathbb{R}^{n^{2}}\).
  The objective
    function is a quadratic $f(x) = 
    \frac{1}{2} \norm[2]{Q x}^{2} + \innp{b}{x}$ where
    $Q\in \mathbb{R}^{n^2 \times n^2}$ and
    $b\in \mathbb{R}^{n^2}$ with $n = 70$. The entries
    of $Q$ and $b$ are selected uniformly
    at random from the interval between $0$ and $1$.
    This results in an objective function with $L \approx 10^{6}$ and
    $\mu \approx 10^{-5}$. Note that for simplex-like polytopes with
    \(0/1\) vertices, computing the away vertex in the DI-PFW algorithm,
    that is, solving the linear optimization problem in
    Line~\ref{alg:Line:optimal_face_problem} of Algorithm~\ref{DIPFW} has
    the same computational cost as solving the linear optimization problem
    in Line~\ref{alg:Line:normal_problem}, which is solved over
    $\mathcal{X}$ and has complexity $n^3$ using the Hungarian algorithm
    (see \citet{combettes21complexity} for an overview). More concretely,
    and following from \citet{garber2016linear},
    Line~\ref{alg:Line:optimal_face_problem} of Algorithm~\ref{DIPFW} is
    computed as:
    \begin{equation*}
      v_{t}^{\text{A}} \defeq
    \argmax_{v \in \mathcal{X}} \innp{d(x_t)}{v},
  \end{equation*}
where \(d(x_t)\in\mathbb{R}^{n^{2}}\) is defined as
\begin{equation*}
[d(x)]_i \defeq 
        \begin{cases}
          [\nabla f(x_t)]_i & \text{if \ } [x_t]_i >0,
        \\
        -\infty & \text{if \ } [x_t]_i = 0
        \end{cases}
\end{equation*}
for $0 \leq i \leq n^{2}$. When considering the PFW
algorithm, we compute the away vertex using an active set, 
and loop through the vertices in $\mathcal{S}_t$ to find the 
vertex that has greatest inner product with the gradient. 
This has a complexity of
$\mathcal{O}(\size{\mathcal{S}_t} n^2)$, where $\size{\mathcal{S}_t}$
is the cardinality of the active set $\mathcal{S}_t$. 
This means that finding the direction towards which PFW
moves has complexity $\mathcal{O}(n^3 + \size{\mathcal{S}_t} n^2)$,
and finding the direction towards which DI-PFW moves has complexity
$\mathcal{O}(n^3)$, so if $n < \size{\mathcal{S}_t}$, we can
expect the cost per iteration of the DI-PFW to be lower than that of
the PFW algorithm, which would widen the difference between the primal
gap convergence in wall-clock time of the two algorithms.

\begin{figure}
  \centering
  \includegraphics[width=.45\linewidth, alt={Graph of primal gap in
    iteration,
    Decomposition-invariant Pairwise Frank–Wolfe algorithm eventually
    outperforming Pairwise Frank–Wolfe algorithm.}]{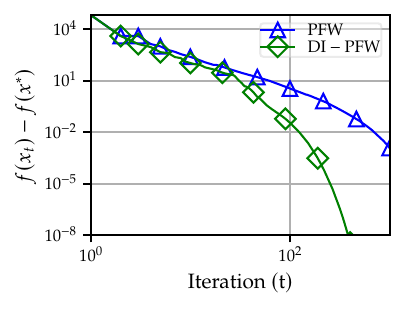}
  \qquad
  \includegraphics[width=.45\linewidth, alt={Graph of primal gap in
    time, similar to the graph in iteration,
    Decomposition-invariant Pairwise Frank–Wolfe algorithm eventually
    outperforming Pairwise Frank–Wolfe algorithm.}]{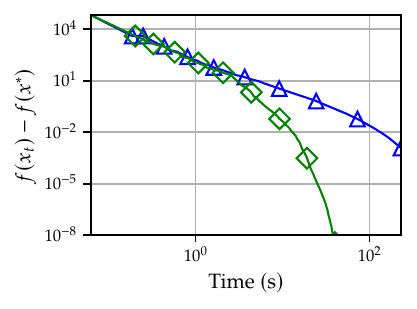}

    \caption{\label{fig:DIPFWPrimalComparison}
    Primal gap evolution of
    the Pairwise Frank–Wolfe algorithm
    (PFW, Algorithm~\ref{pairwise}) and
    the Decomposition-invariant Pairwise Frank–Wolfe algorithm
    (DI-PFW, Algorithm~\ref{DIPFW})
    for a random quadratic function in dimension \(4900\)
    over the Birkhoff polytope (i.e., doubly stochastic matrices),
    see Example~\ref{ex:Birkhoff}.
    Note that these are log-log plots and
    we can clearly see that the convergence of DI-PFW
    is not just faster in time due to not having to maintain an
    explicit active set at each iteration but also in iterations due
    to the use of better away-vertices, not being confined to
    an arbitrary active set but considering all possible active sets.}
\end{figure}

For this specific instance, we can see from Figure~\ref{fig:DIPFWPrimalComparison} that
 the DI-PFW algorithm provides an advantage
in both the number of iterations and wall-clock time over the PFW algorithm.
The advantage is explained by DI-PFW potentially using better
away-vertices, searching them from a possibly larger set.
For this instance the optimum is very likely sparse as
the high accuracy solution we have computed is sparse with \(\sparsity{x^{*}} = 205\). 
More generally, we will see also in later comparisons that the DI-PFW algorithm
will have a computational advantage in wall-clock time over the PFW
algorithm when solving the linear minimization problem in
Line~\ref{alg:Line:optimal_face_problem} of Algorithm~\ref{DIPFW} is
computationally cheaper than the looping through the active set that the PFW
algorithm typically performs to find an away vertex.

\end{example}

\subsection{Adaptive step size control for conditional gradient algorithms}
\label{sec:adaptive-step-size}

Recall that to avoid the cost of  performing a line search, the short step
rule is typically employed, which has the disadvantage of requiring a good
upper bound \(L\) on the smoothness of the objective function. 
\citet{pedregosa2018step} instead proposed an adaptive step size strategy
(Algorithm~\ref{alg:AdaptiveStepSize}) where the initial
smoothness estimate is continuously adjusted
as new data becomes available,
creating
a step size algorithm that is a blend of a backtracking line search strategy
and the short step rule. The overhead from this adaptive strategy is relatively
low compared to running a line search in each iteration and simply requires having access to the objective function value (often termed a zeroth-order oracle call), so that it can be often used as \emph{the}
go-to step size strategy.

The adaptive step size strategy of \citet{pedregosa2018step} has a tendency to
progressively \emph{decrease} the smoothness constant estimate to potentially
take advantage of a locally possibly smoother behavior, and then backtracks to
ensure the smoothness inequality holds; as smoothness induces primal progress,
it is referred to as the \emph{sufficient decrease condition} in the paper.  The
progressive tightening and backtracking is controlled by two parameters: the
tightening parameter \(\eta\) and the backtracking parameter \(\tau\).
In a nutshell, the adaptive step size strategy maintains an
estimate $\widetilde{L}$ of the smoothness parameter, then performs a
short step with that estimate $\widetilde{L}$ and verifies that the
smoothness inequality is satisfied.  If so, the step is accepted,
otherwise $\widetilde{L}$ is updated and the procedure is repeated.

In conditional gradient algorithms,
like the vanilla Frank–Wolfe algorithm, the
Away-step Frank–Wolfe algorithm, and
Pairwise Frank–Wolfe algorithm (Algorithms~\ref{fw}, \ref{away}, and
\ref{pairwise}),
the adaptive step size strategy provides convergence rates that are essentially
identical (up to constants) to the short step rule, however requiring
additionally a constant number of function evaluations per iteration depending
on \(\eta\) and \(\tau\) in the worst case.  A theoretical model to leverage
the progressive behavior in order to obtain strictly better rates is an open
question.
In particular, there is little guidance on the choice of \(\tau\)
and \(\eta\). 
However, \citet{pedregosa2018step} claims the values
\(\tau = 2\) and \(\eta = 0.9\) perform well in practice; the same values have been also used as default values in the \href{https://github.com/ZIB-IOL/FrankWolfe.jl}{\texttt{FrankWolfe.jl}} Julia package \citep{BCP2021} with very good results.

\begin{algorithm}[b]
  \caption[]{Adaptive step size strategy (\Ada{f}{x}{v}{\widetilde{L})} \citep{pedregosa2018step}}
  \label{alg:AdaptiveStepSize}
  \begin{algorithmic}[1]
    \REQUIRE Objective function \(f\),
      smoothness estimate \(\widetilde{L}\),
      feasible points \(x\), \(v\)
      with \(\innp{\nabla f(x)}{x - v} \geq 0\),
      progress parameters \(\eta \leq 1 < \tau\)
    \ENSURE Updated estimate \(\widetilde{L}^{*}\),
      step size \(\gamma\)
    \STATE\(M \leftarrow \eta \widetilde{L}\)
    \LOOP
      \STATE
        \(\gamma \leftarrow
        \min \{\innp{\nabla f(x)}{x - v}
        \mathbin{/} (M \norm{x - v}^{2}), 1 \}\) \COMMENT{compute short step for $M$}
      \IF{\(f(x + \gamma (v - x)) - f(x)
            \leq \gamma \innp{\nabla f(x)}{v - x}
            + \frac{\gamma^{2} M}{2} \norm{x - v}^{2}\)} 
          \label{SufficientDecrease}
        \STATE \(\widetilde{L}^{*} \leftarrow M\)
        \RETURN \(\widetilde{L}^{*}\), \(\gamma\)
      \ENDIF
      \STATE\(M \leftarrow \tau M\)
    \ENDLOOP
  \end{algorithmic}
\end{algorithm}

In our presentation of the adaptive step size strategy
(Algorithm~\ref{alg:AdaptiveStepSize}),
the previous smoothness estimate
\(\widetilde{L}\) is provided as an input to the algorithm, which is then
continuously updated until the smoothness condition is ensured between the
current point, and the new point the algorithm intends to move towards. For
simplicity, the update direction is represented by the furthest point \(v\) the
algorithm is willing to go in that direction. We present a modification of the
vanilla Frank–Wolfe algorithm endowed with the adaptive step size
strategy in Algorithm~\ref{alg:adaFW} as an example; all algorithms can be
modified similarly by simply replacing the short step rule with this strategy.
In order to choose the initial Lipschitz estimate \(L_{-1}\) in
Algorithm~\ref{alg:adaFW} or in any adaptive variant, one can simply guess (as
the strategy is going to fix it anyway), e.g., via the heuristic suggested in
\citet{pedregosa2018step}: \(L_{-1} \defeq \norm{\nabla f(x_0) - \nabla f(x_0 +
\gamma (v_0 - x_0))} /(\gamma\norm{v_0 - x_0})\), where \(\gamma>0\) is a small
constant, and \( v_0 = \argmin_{x \in \mathcal{X}}
\innp{\nabla f(x_0)}{x}\).
The original work \citet{pedregosa2018step} claims the value
\(\gamma = 0.001\) works well in practice,
which we empirically confirmed, too.

\begin{algorithm}[t]
  \caption{Adaptive vanilla Frank–Wolfe (AdaFW)
    \citep{pedregosa2018step}}
  \label{alg:adaFW}
  \begin{algorithmic}
    \STATE As Algorithm~\ref{fw}, but choosing an initial \(L_{-1}\)
      and replacing Line~\ref{line:basic-step-size} with
    \STATE \(L_{t}, \gamma_{t} \leftarrow
      \Ada{f}{x_{t}}{v_{t}}{L_{t-1}}\)
  \end{algorithmic}
\end{algorithm}

\begin{example}[A numerical example]
  \label{ex:ada}
  In Figure~\ref{fig:adaFW} (left) we compare
  \begin{figure}[b]
       \centering
       \includegraphics[width=.45\linewidth, alt={Graph of primal gap
         in iteration, Adaptive Pairwise Frank–Wolfe algorithm
         eventually outperforming the other algorithms, which have a
         similar slow convergence.}]{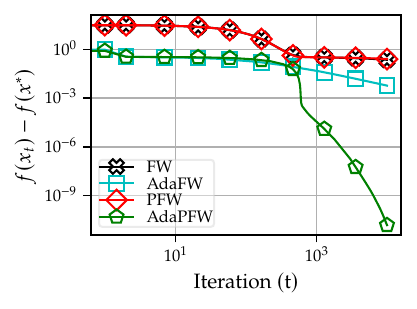}
       \qquad
       \includegraphics[width=.45\linewidth,
       alt={Local smoothness estimate of adaptive Frank–Wolfe
         variants, oscillating around the value
         \(100\).}]{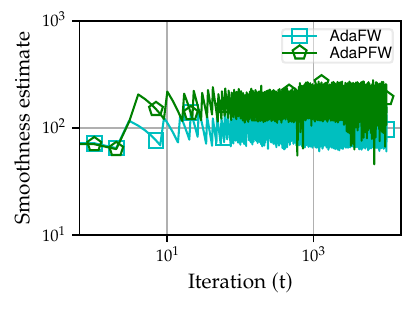}
  \caption{\label{fig:adaFW}Performance in primal gap
    for Frank–Wolfe algorithms
    with and without adaptive step size rules (left figure)
    minimizing a strongly convex and smooth quadratic function over
    the \(200\)-dimensional unit $\ell_1$-ball,
    see Example~\ref{ex:ada} for details.
    The algorithms that use the adaptive step size rule
    (Algorithm~\ref{alg:AdaptiveStepSize}) have the prefix
    ``Ada'',
    the remaining algorithms simply use the short step rule.
    The algorithms depicted are the vanilla Frank–Wolfe algorithm
    (FW, Algorithm~\ref{fw}), and the Pairwise Frank–Wolfe algorithm
    (PFW, Algorithm~\ref{pairwise}).
    The right figure shows the evolution of the local smoothness
    estimate $\widetilde{L}$ used for adaptive step size.
    Note how this estimate is much smaller than
    the smoothness constant
    $L \approx \num{10000}$,
    leading to faster convergence for the adaptive variants.
    Note that we have
    used hollow markers to differentiate between the line plots that overlap.
    \vspace{-1ex}
  }
\end{figure}%
  the primal gap evolution of the vanilla Frank–Wolfe algorithm (Frank–Wolfe),
   and the Pairwise Frank–Wolfe algorithm (PFW), and their variants
  with the
  \emph{adaptive step size strategy} (denoted by
  the \lq{}Ada\rq{} prefix in the graphs)
  via Algorithm~\ref{alg:AdaptiveStepSize}.  The feasible region
  is the unit \index{l1-ball@\(\ell_{1}\)-ball}\(\ell_{1}\)-ball and the objective
    function is a quadratic $f(x) =
    \frac{1}{2} \norm[2]{Q x}^{2} + \innp{b}{x}$ where
    $Q\in \mathbb{R}^{n \times n}$ and
    $b\in \mathbb{R}^{n}$ with $n = 200$. The entries
    of $Q$ and $b$ are selected uniformly
    random from the interval between $0$ and $1$.
    This results in an objective function with
    $L \approx \num{10043}$ and $\mu \approx \num{0.0005}$. Figure~\ref{fig:adaFW} (right) shows
     the evolution of the smoothness
     estimate $\widetilde{L}$ used by the
     adaptive algorithms. All
     adaptive variants use $\tau =
2$, $\eta = 0.9$ and a starting
     smoothness estimate of $0.01$.
     Note that the smoothness estimate
     used by the adaptive variants hovers around $150$, and is much
     smaller than the $L$-smoothness
     parameter of the function, which
     results in larger step sizes and greater progress.

\end{example}

In fact the adaptive step size strategy has two more properties
that make it extremely powerful.

\begin{remark}[Adaptivity to local smoothness]
  One significant advantage in comparison to the short step rule is that the adaptive step size strategy can adapt to \emph{local} smoothness depending on the trajectory of the algorithm. This can be observed in the example above, where the estimated smoothness constant is much smaller then the worst-case smoothness ($L \approx \num{10000}$). The result is a significantly improved real-world performance.
\end{remark}

\begin{remark}[Multiplicative vs. additive accuracy]
  The second advantage of the adaptive step size strategy in
  comparison to line searches is that in actual computations we rarely
  perform a line search but rather an approximate one or a
  backtracking line search with finite precision. Once the function
  change is below that additive threshold, a line-search based
  algorithm cannot further progress, so that high-precision solutions
  can be rarely obtained. In contrast, the adaptive step size
  strategy estimates the local smoothness and errors in the estimation
  of $L$ lead to a multiplicative error
  (rather than an additive error) in progress,
  so that the optimal solution can still be approached,
  just a little slower.
\end{remark}

\begin{remark}[Implementation notes for Algorithm~\ref{alg:AdaptiveStepSize}] 
  In practice, the test in Line~\ref{SufficientDecrease} of
  Algorithm~\ref{alg:AdaptiveStepSize} sometimes fails
  due to numerical
  inaccuracy, leading to rapid increase of the smoothness estimate
  \(M\).
  As a consequence the resulting steps have length close to $0$
  and optimization would stall.

  As a countermeasure to numerical errors,
  one should relax the required progress,
  trading theoretical optimality of the short step estimation for error tolerance.
  A good way is to include a factor \(\alpha \leq 1\) for the step size:
  \begin{equation}
  	f\bigl(x + \gamma (v - x)\bigr) - f(x)
    \leq \alpha \gamma \innp{\nabla f(x)}{v - x}
    + \frac{\alpha^2 \gamma^{2} M}{2} \norm{x - v}^{2}
    .
  \end{equation}
 For example the \texttt{FrankWolfe.jl} package uses $\alpha = 0.5$ as
 default, which works well in practice.
 The case $\alpha = 1$ is the case without the relaxation,
 exactly as it appears in Algorithm~\ref{alg:AdaptiveStepSize}.
 Note that the factor \(\alpha\) is not included in the returned step
 size, as there are no numerical issues there.
 
 Moreover, note that this relaxation affects the convergence speed only by a multiplicative factor of about $1/\alpha^2$. 
\end{remark}

\subsection{Lazification}
\label{sec:lazification}

Depending on the feasible region of interest,
the cost of the linear minimization oracle,
while still much smaller than the cost of projections,
might be huge.
This is for example the case if the feasible region $\mathcal{X}$
is the convex hull of solutions to
a hard combinatorial optimization problem,
such as the \emph{traveling salesman problem},
where a single solution to the linear programming problem can easily
take minutes to hours for moderately-sized instances.
In such cases it is highly desirable to reduce the cost of
the linear minimization oracle in
conditional gradient algorithms.
Facing such expensive linear minimization
oracles, the questions are
\begin{enumerate*}[itemjoin={{, }}, itemjoin*={{, and }}]
\item do we really have to call the (expensive) LMO in each
  iteration?
\item is there a way to reuse information from previous calls?
\end{enumerate*}

The currently known weakest relaxation of an LMO that answers the two questions above in the affirmative while maintaining the same convergence rates (up to constant factors) as the more expensive LMO was proposed by
\citet{pok17lazy}: a \emph{Weak Separation Oracle}
(Oracle~\ref{ora:LPsep}), which searches only for a solution that is \emph{good enough},
i.e., one that has an objective value better than a given threshold.
In particular, this oracle model is weaker than the approximate LMO model in \citet{jaggi13fw},
and even weaker than approximation with multiplicative error,
as it requires a (realistic) additive improvement goal \(\phi\) only rather than some approximately optimal solution.
Upon presentation with an improvement goal $\phi$ and the current iterate, the weak separation oracle will either return a positive answer: an improving vertex $y$
realizing the target improvement \(\phi\) over the current iterate (w.r.t.~the provided linear function)
within a constant multiplicative error \(K\),
or a negative answer that the improvement goal \(\phi\) is impossible (which then also serves as an approximate optimality certificate as we will see).

\begin{oracle}
  \caption{Weak Separation LP Oracle for \(\mathcal{X}\) (LPsep)}
  \label{ora:LPsep}
  \begin{algorithmic}
    \REQUIRE Linear objective \(c\), point \(x\in\mathcal{X}\),
      accuracy \(K\geq 1\),
      target objective \(\phi> 0\)
    \ENSURE
      \(\begin{cases*}
        y \in \mathcal{X}, & with \(\innp{c}{x-y} > \phi / K\)
        \quad\hfill
        (positive answer)
        \\
        \FALSE, & if \(\innp{c}{x-z} \leq \phi\) for all
        \(z \in \mathcal{X}\)
        \quad\hfill
        (negative answer)
      \end{cases*}\)
\end{algorithmic}
\end{oracle}

The weak separation oracle model enables two important speed-ups compared to a single LMO
call: caching and early termination.  \emph{Caching} searches first a
small cache of vertices for an improving vertex meeting the improvement requirement, falling back to the more
costly linear minimization oracle in case of a cache miss.  The cache is
maintained by some heuristic, like recently used or previously computed vertices in a context where
the oracle is used repeatedly.
In conditional gradient algorithms natural
implementations are
\begin{enumerate*}[after=., itemjoin={{, }}, itemjoin*={{, or }}]
\item the list of previously computed vertices
\item the active set of vertices for those variants that maintain
  active sets
\end{enumerate*}
\emph{Early termination} reuses an
LMO implementation, which typically would compute an
(approximately) optimal solution going through many vertices in the
process but would terminate the search early once a vertex satisfying
the improvement condition is found. Another important reason to use
lazification is in order to obtain much sparser solutions in the
number of vertices (or atoms) used to express the iterate as a convex
combination. This is promoted by reusing vertices (or atoms), via the
cache whenever possible. For completeness the following remark is in
order:

\begin{remark}[Linear Optimization realizes Weak Separation]
	\label{rem:linoptweaksep}
	The Weak-Separation Oracle can be realized with a single call to a Linear Optimization Oracle.
\end{remark}

As mentioned before, replacing the linear minimization oracle with the weak separation
oracle in conditional gradient algorithms provides similar
convergence rates (up to small constant factors)
in the number of oracle calls, where the cheapness of the
oracle calls to the weak separation oracle should compensate for the overhead of maintaining the cache and loss in constant factors etc. 

Here we discuss only the lazified vanilla Frank–Wolfe algorithm in
detail, leading to the \emph{Parameter-free Lazy Conditional
  Gradient algorithm} (Algorithm~\ref{alg:ParamFreeLCG}), as well as
showing how the same approach carries over to
other conditional gradient
algorithm variants exemplified with the Away-step Frank–Wolfe
algorithm.  The choice of \(\phi_{0}\) affects the convergence rate
mainly
in the initial part of the algorithm:
a choice that is too large
leads to many initial negative oracle answers,
while a choice that is too small can lead to slow progress in the initial iterations, due to unnecessarily weak improving directions.
The condition \(\phi_{0} \geq h_{0} / 2\) is included in the algorithm
only to simplify the convergence proof by omitting
inefficient initial iterations.
While this condition is not directly verifiable
in practice as \(h_{0}\) might not be known,
the original paper ensured it via setting
\(\phi_{0} = \min_{x} \innp{\nabla f(x_{0})}{x - x_{0}}
\mathbin{/} 2\), i.e., to be half of the Frank–Wolfe gap,
requiring a single, initial linear minimization oracle call.

\begin{algorithm}[t]
\caption{Parameter-free Lazy Conditional Gradient (LCG) \citep{pok17lazy}}
\label{alg:ParamFreeLCG}
\begin{algorithmic}[1]
\REQUIRE Start atom $x_0\in\mathcal{X}$, accuracy $K\geq 1$, initial
  goal \(\phi_{0} > 0\)
\ENSURE Iterates $x_1, \dotsc \in \mathcal{X}$
\FOR{$t=0$ \TO \dots}
  \STATE $v_t \leftarrow \text{LPsep}_{\mathcal{X}}(\nabla f(x_t), x_t,
    \phi_{t}, K)$
\IF{$v_t = \FALSE$}
\STATE $x_{t+1} \leftarrow x_t$
\STATE $\phi_{t+1} \leftarrow \phi_{t} / 2$
\ELSE
\STATE $\gamma_{t} \leftarrow \argmin_{0 \leq \gamma \leq 1} f(x_t + \gamma(v_t- x_t))$ or adaptive step size strategy (Algorithm~\ref{alg:AdaptiveStepSize})
\STATE $x_{t+1} \leftarrow x_t + \gamma_t(v_t - x_t)$
\STATE $\phi_{t+1} \gets \phi_{t}$
\ENDIF
\ENDFOR
\end{algorithmic}
\end{algorithm}

The \(\phi_{t}\) are the thresholds to be used in the Weak Separation
Oracle, which are intended to be kept within a constant factor of the
Frank–Wolfe gap \(g_{t}\) to ensure progress comparable to the
(non-lazy) vanilla Frank–Wolfe
algorithm.  A negative answer certifies that $\phi_t \geq g_t \geq h_t$, so
that the controlled decrease of \(\phi_{t}\) by any constant factor between $0$
and $1$ (the choice of $1/2$ is arbitrary)
ensures we remain within a fixed
constant factor of the Frank–Wolfe gap and at the same time
we keep $\phi_t$ large
to find descent directions with sufficient progress via the Weak Separation
Oracle.

\begin{theorem}
  \label{lazyconvergence}
  For an \(L\)-smooth convex objective function \(f\)
  over a convex domain $\mathcal{X}$ with diameter \(D\),
  if \(h_{0} / 2 \leq \phi_{0} \leq L D^{2} / 2\)
  then
  Algorithm~\ref{alg:ParamFreeLCG} achieves
  primal gap at most \(\varepsilon\),
  i.e., $f(x_t) - f(x^*) \leq \varepsilon$
  in at most
  \(\mathcal{O}(K^{2}L D^{2} / \varepsilon)\),
  or more precisely the following number of
  gradient computations and weak separation oracle calls:
  \begin{equation}
    \label{eq:lazyFW}
    \max\left\{
      \ceil*{\log_{2} \frac{\phi_0}{\varepsilon}}
      , 0\right\}
    + \ceil*{\frac{8 K^{2} L D^{2}}{\varepsilon}}
    .
  \end{equation}
\end{theorem}
Note that up to a constant factor
the last term \(\mathcal{O} (L D^{2} / \varepsilon)\) is
the rate for the vanilla Frank–Wolfe algorithm
(Theorem~\ref{fw_sub}).
The first term \(\log (\phi_{0} / \varepsilon)\)
is the number of negative oracle calls,
until \(\phi_{t}\) drops below \(\varepsilon\),
upon which the next negative oracle answer ensures
\(h_{t} \leq \varepsilon\).

The condition \(h_{0} / 2 \leq \phi_{0} \leq L D^{2} / 2\)
is satisfiable by two linear oracle calls without knowledge of
parameters like \(L\) or \(D\), similar to Remark~\ref{rem:initial-bound}: 
With the cost of at most one linear oracle call find an arbitrary \(x_{-1} \in \mathcal{X}\). Then choose
\(x_{0} \defeq \argmin_{x \in \mathcal{X}} \innp{\nabla
  f(x_{-1})}{x}\)
and \(\phi_{0} \defeq g(x_{0}) \mathbin{/} 2\)
and clearly $h_0/2 \leq g(x_0) / 2$.
Now, note that for any \(y \in \mathcal{X}\),
one has
\(\innp{\nabla f(x_{-1})}{x_{0}} \leq \innp{\nabla f(x_{-1})}{y}\)
by the choice of \(x_{0}\).
Rearranging and using \(L\)-Lipschitzness of gradients lead to
\begin{equation*}
  \begin{split}
   \innp{\nabla f(x_{0})}{x_{0} - y}
   &
   \leq \innp{\nabla f(x_{0}) - \nabla f(x_{-1})}{x_{0} - y}
   \\
   &
   \leq \dualnorm{\nabla f(x_{0}) - \nabla f(x_{-1})} \cdot
   \norm{x_{0} - y}
   \\
   &
   \leq L \norm{x_{0} - x_{-1}} \cdot \norm{x_{0} - y}
   \leq L D^{2}.
  \end{split}
\end{equation*}
Thus \(g(x_{0}) \leq L D^{2}\) justifying the choice of \(\phi_{0}\) above.
However, even for an arbitrary choice of \(\phi_{0}\) and \(x_{0}\)
the algorithm converges with a similar rate
(even if \(h_{0} > L D^{2}\)) at the cost of a longer initial burn-in phase; in the case that $\phi_0 > L D^{2} / 2$ is chosen too large we have linear convergence (at least) until $\phi_t \leq L D^{2} / 2$.

\begin{proof}
First note that the primal gap \(h_{t}\) is obviously non-increasing.
The main claim is that \(\phi_{t} \geq h_{t} / 2\)
for every iteration \(t\),
which we prove by induction on \(t\).
For \(t = 0\) this holds by assumption.
In an iteration \(t\) when
the weak separation oracle returns a negative answer, this guarantees
\(\phi_{t} \geq g_{t}\) and hence \(\phi_{t} \geq h_{t}\).
In the next iteration, we will have
\(\phi_{t+1} = \phi_{t} / 2 \geq h_{t} / 2 = h_{t+1} / 2\).
The remaining case is an iteration \(t\)
with a positive oracle answer, where \(\phi_{t} \geq h_{t} / 2\)
by assumption.
Then simply \(\phi_{t+1} = \phi_{t} \geq h_{t} / 2 \geq h_{t+1} / 2\).

The above argument also shows that right after
\(\ceil{\log_{2} (\phi_{0} / \varepsilon)}\)
many negative oracle answer,
we have \(h_{t} \leq \phi_{t} \leq \varepsilon\).
This upper bounds the number of negative oracle answers until reaching
\(h_{t} \leq \varepsilon\).
(Actually, this upper bound only holds if \(\phi_{0} > \varepsilon\),
otherwise we have \(h_{t} \leq \varepsilon\) after a single negative
oracle answer.)

It remains to estimate the number of iterations
with positive oracle calls,
which we do by estimating the primal progress during these iterations.
When the oracle returns a positive answer in iteration \(t\),
using Lemma~\ref{lemma:progress}
we have the following per iteration progress
(estimating line search using the short step rule),
which is the same as used several times before:
\begin{equation}
  \label{eq:lazyFW-progress}
  h_{t} - h_{t+1} \geq
  \frac{\phi_{t}}{2K} \min \left\{1, \frac{\phi_{t}}{K L D^{2}}\right\}
  =
  \frac{\phi_{t}^{2}}{2 K^{2} L D^{2}}
  \geq
  \frac{h_{t}^{2}}{8 K^{2} L D^{2}}
  .
\end{equation}
Here we have used \(\phi_{t} \geq h_{t} / 2\).
Thus after \(k\) positive oracle in the final phase it holds
\(h_{t} \leq 8 K^{2} L D^{2} / (k+8)\),
as can be easily shown by induction using \(h_{0} \leq L D^{2}\).

Summing up the positive and negative oracles answers
provides the claimed upper bounds.
\end{proof}

Lazification applies also to other conditional gradient algorithms,
such as, e.g., the
Away-step Frank–Wolfe algorithm (Algorithm~\ref{away}), which we demonstrate
below for the sake of completeness. Making the search for both the away vertex
and the Frank–Wolfe vertex lazy yields the \emph{Lazy Away-step Frank–Wolfe
algorithm} (Algorithm~\ref{alg:ParamFreeLAFW}), where
Line~\ref{line:LAFW-weak-oracle} contains the weakened, lazy search. Here the cache is essentially provided by the active set that is maintained. The other
parts of the algorithm are essentially the same, so we have again omitted the
details of the update of the active set for readability.  In practice, one has to go through the whole active set anyway in search of the away vertex, hence searching for
\emph{both, an away vertex and a Frank–Wolfe vertex} to implement lazification comes at no extra cost, and the fallback to a
slower algorithm if no suitable vertex has been found is the same as before; we make this strategy more precise in Remark~\ref{rem:LPSepViaLP-typical} for the important case where the lazification wraps the LMO.

\begin{algorithm}[t]
  \caption{Lazy Away-step Frank–Wolfe (Lazy AFW) \citep{pok17lazy}}
  \label{alg:ParamFreeLAFW}
  \begin{algorithmic}[1]
    \REQUIRE Start atom \(x_{0} \in P\), accuracy \(K\geq 1\),
      initial goal \(\phi_{0} \geq h_{0} / 2\)
      \\ 
      \COMMENT{E.g., \(\phi_{0} \defeq
    \max_{x \in \mathcal{X}} \innp{\nabla f(x_{0})}{x_{0} - x}
    \mathbin{/} 2\)} 
    \ENSURE Iterates \(x_{1}, \dotsc \in P\)
    \STATE \(\mathcal{S}_{0} \leftarrow \{x_{0}\}\),
      \(\lambda_{x_{0}, 0} \leftarrow 1\)
    \FOR{\(t=0\) \TO \dots}
      \STATE \label{line:LAFW-weak-oracle}
        Find \(v_{t} \in \mathcal{S}_{t}\) with
        \(\innp{\nabla f(x_{t})}{v_{t} - x_{t}} \geq \phi_{t} / K\)
        or
        \(v_{t} \in P\) with
        \(\innp{\nabla f(x_{t})}{x_{t} - v_{t}} \geq \phi_{t} / K\)
        unless
        \(\innp{\nabla f(x_{t})}{v - x_{t}} \leq \phi_{t}\)
        for all \(v \in \mathcal{S}_{t}\)
        and
        \(\innp{\nabla f(x_{t})}{x_{t} - v} \leq \phi_{t}\)
        for all \(v \in P\).
      \IF{\(v_{t}\) found}
        \IF[away step]{\(\innp{\nabla f(x_{t})}{v_{t} - x_{t}}
              \geq \phi_{t} / K\)
              and \(v_{t} \in \mathcal{S}_{t}\)}
          \STATE $d_{t} = x_{t} - v_{t}$,
            $\gamma_{\max} =
            \frac{\lambda_{v_{t}, t}}{1 - \lambda_{v_{t}, t}}$
        \ELSE[Frank–Wolfe step]
          \STATE $d_{t} = v_t - x_t$, $\gamma_{\max} = 1$
        \ENDIF
        \STATE \(\gamma_{t} \leftarrow \min \left\{
            \frac{\innp{\nabla f(x_{t})}{d_{t}}}{L \norm{d_{t}}^{2}},
    \gamma_{\max} \right\}\) or adaptive step size strategy (Algorithm~\ref{alg:AdaptiveStepSize})
        \STATE
          \(x_{t+1} \leftarrow x_{t} - \gamma_{t} d_{t}\)
        \STATE \(\phi_{t+1} \leftarrow \phi_{t}\)
        \STATE Find convex combination of \(x_{t+1}\):
          set \(\mathcal{S}_{t+1}\) and \(\lambda_{v, t+1}\)
          for \(v \in \mathcal{S}_{t+1}\).
          \COMMENT{see Algorithm~\ref{away}}
      \ELSE[No \(v_{t}\) found; update dual estimate]
        \STATE $x_{t+1} \leftarrow x_t$
        \STATE $\phi_{t+1} \leftarrow \phi_{t} / 2$
      \ENDIF
    \ENDFOR
  \end{algorithmic}
\end{algorithm}

Algorithm~\ref{alg:ParamFreeLAFW} has a similar convergence rate as the original Away-step
Frank–Wolfe algorithm. However the lazification of the Away-step
Frank–Wolfe algorithm significantly improves sparsity of the iterates
in actual computations. We include a sketch of the convergence proof.

\begin{theorem}
\label{th:lazyAFWconvergence}
  Let $f$ be an $L$-smooth and $\mu$-strongly convex function
  over a polytope \(P\) with diameter \(D\)
  and pyramidal width \(\delta\). Then the Lazy AFW algorithm (Algorithm~\ref{alg:ParamFreeLAFW}) ensures
  \(h_{t} \leq \varepsilon\) for the primal gap
  after \(\mathcal{O}(\log (\phi_{0} / \varepsilon) +
  K^{2} L D^{2} / (\mu \delta^{2}) \log (h_{0} / \varepsilon))\),
  or more precisely the following number of
  gradient computations and weak separation oracle calls:
  \begin{equation}
    \label{eq:lazyAFW}
    \max\left\{
      \ceil*{\log_{2} \frac{\phi_0}{\varepsilon}}
      , 0\right\}
    + 1
    +
    2
    \cdot
    \ceil*{\frac{4 K^{2} L D^{2}}{\mu \delta^{2}}
    \max
    \left\{
      \ln \frac{h_{0}}{\varepsilon},
      0
    \right\}}
    .
  \end{equation}
\end{theorem}

\begin{proof}[Proof sketch]
The proof is analogous to that of Theorem~\ref{lazyconvergence},
so we present only the details that differ.
The main difference is in estimating the number of positive
oracle answers after the first negative oracle call,
reusing ideas from Theorem~\ref{th:AFWConvergence}.
For example, the number of drop steps is estimated the same way
as at most half the number of steps with a positive oracle answer,
adding a factor of \(2\).
In the remaining part of the proof we consider only non-drop steps.
Via Geometric Strong Convexity
(Lemma~\ref{lem:geoSC})
we have
\(2 \phi_{t} \geq \innp{\nabla f(x_{t})}{v^{\text{A}} - v^{\text{FW}}}
\geq \sqrt{2 \mu \delta^{2} \cdot h_{t}}\)
in addition to \(\phi_{t} \geq h_{t} / 2\)
after the first negative oracle call,
resulting for non-drop steps
in a progress bound of the form (cf., Equation~\eqref{eq:AFWstep})
\begin{equation}
  h_{t} - h_{t+1}
  \geq
  \min \left\{
    \frac{\phi_{t}}{2 K},
    \frac{\phi_{t}^{2}}{2 K^{2} L D^{2}}
  \right\}
  \geq
  \min \left\{
    1,
    \frac{\mu \delta^{2}}{K L D^{2}}
  \right\}
  \frac{h_{t}}{4K}
  =
  \frac{\mu \delta^{2}}{4 K^{2} L D^{2}}
  \cdot
  h_{t}
  .
\end{equation}
This results in an upper bound of
\((4 K^{2} L D^{2} / (\mu \delta^{2})) \ln (h_{0} / \varepsilon)\) iterations
with a positive oracle answer, which are not drop steps,
happening after the first negative oracle answer.

In conclusion, we obtain an overall bound that is similar to
Theorem~\ref{lazyconvergence},
with the last two terms replaced,
and an extra factor of \(2\) (which accounts for drop steps).
\end{proof}

The bound simplifies when \(\phi_{0}\) and \(x_{0}\)
are chosen as in Theorem~\ref{lazyconvergence},
at the additional cost of two linear minimizations.

\begin{remark}[Typical implementation of LPSep via LMO]
	\label{rem:LPSepViaLP-typical}
Typically the weak separation oracle is realized via wrapping the LMO in a
smart way. This results in a few extra possibilities for improvements as we
detail now.  We demonstrate this for $\text{LPsep}_{\mathcal{X}}(\nabla f(x_t),
x_t, \phi_{t}, K)$ in Algorithm~\ref{alg:ParamFreeLCG} but the same applies for
Algorithm~\ref{alg:ParamFreeLAFW} with the only difference that we can combine
the search for a lazy solution directly with the search for an away vertex.
Given an LMO and a list of previously visited vertices $\mathcal L$ we can
implement  $\text{LPsep}_{\mathcal{X}}(\nabla f(x_t), x_t, \phi_{t}, K)$ as
follows:
\begin{enumerate*}
\item
  Check whether any vertex $x \in \mathcal L$ satisfies
  $\innp{\nabla f(x_t)}{x_t-x} > \phi / K$. If so return such an $x$.
\item  Otherwise call
the LMO (potentially with early termination) to solve $x = \argmin_{v \in
\mathcal{X}} \innp{\nabla f(x_t)}{v}$. If $x$ satisfies  $\innp{\nabla
f(x_t)}{x_t-x} > \phi / K$ return $x$.
\item  Otherwise return \textbf{false} and
$g(x_t) = \innp{\nabla f(x_t)}{x_t-x}$, which is the Frank–Wolfe gap
at $x_t$ (negative call).
\end{enumerate*}
In most reasonable implementations one would
choose $K=1$.

With this implementation, we obtain the Frank–Wolfe gap in the case of a
negative call and we can use this information to more effectively update
$\phi_t$ as $\phi_{t+1} \leftarrow \min\{\phi_{t} / 2, g(x_t)\}$. This improved
update reduces unnecessary and potentially expensive LMO calls that are guaranteed to lead to negative answers and no progress: we have to scale $\phi_t$ down to
the Frank–Wolfe gap anyways before we can obtain a positive call with
actual progress again. Note that typically calls to the weak
separation oracle with
negative answer are much more expensive than those with positive
answers: the positive one are usually obtained via the cache or an
early-terminated LMO call, whereas the negative ones require an
optimal solution, i.e., a full LMO call in order to ensure
$\innp{\nabla f(x_t)}{x_t-x} \leq \phi_{t}$
for all $x \in \mathcal X$.

Care has to be taken to keep \(\mathcal{L}\) small,
so that search over \(\mathcal{L}\) remains cheaper than an LMO call.
E.g., in the extreme case that the feasible region is a polytope $P$
and \(\mathcal{L}\) is the set of all its vertices,
searching over \(\mathcal{L}\) is the same and hence as hard a problem
as linear optimization over \(P\).

We refer the interested reader to \citet{pok17lazy,BPZ2017jour,pok18bcg} for
methodological considerations and \citet{CPSZ2021,BCP2021} for implementation
related aspects. 
\end{remark}

\begin{remark}[Trade-off lazification vs.~vanilla method]
	While having the same worst-case convergence rates, as the lazified algorithm does not solve the LMO subproblem exactly but only uses sufficiently good solutions, the obtained descent directions might provide less progress per iteration compared to those that would have been employed by the vanilla method. This is a consequence of the smoothness inequality and that the primal progress (using the short step rule or line search) is essentially lower bounded proportional to $\innp{\nabla f
(x_t)}{x_t - v_t}^2$. Moreover, the lazified version usually has to maintain a cache of previously visited vertices over which it has to search for sufficiently good vertices. Sometimes this comes for free as, e.g., for the lazy AFW algorithm as the active set has to be maintained anyways at other times this creates an overhead. On the other hand, the execution of the weak separation oracle can often be several orders of magnitude faster than the LMO. As such in actual computations a natural trade-off arises: less per-iteration progress vs.~faster iterations in wall-clock time. In particular if the LMO is somewhat expensive, dominating the iteration cost of the algorithm, lazification usually provides excellent performance. 
\end{remark}

\begin{example}[A numerical run of the lazy variants] \label{example:lazy_comparison}
  \begin{figure}[b]
  \centering
  \includegraphics[width=.45\linewidth, alt={Graph of
    primal gap in iteration
    for lazy and non-lazy Frank–Wolfe variants, behaving similarly
    with the lazy variants performing slightly better in the end.}]{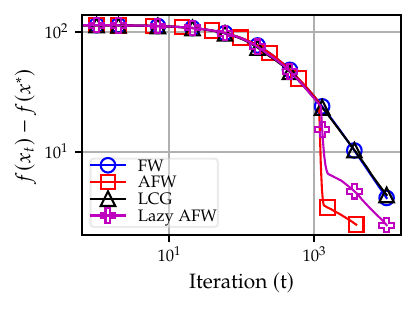}
  \qquad
  \includegraphics[width=.45\linewidth, alt={Graph of primal gap in
    wall-clock time for lazy and non-lazy Frank–Wolfe algorithms, with
    the lazy ones being faster.}]{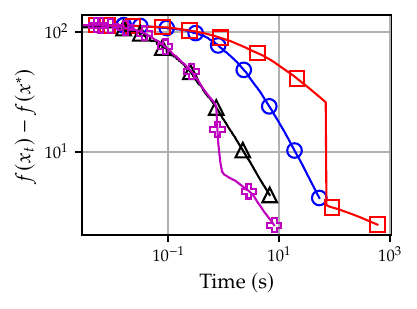}
  \caption{\label{fig:lazy}Primal gap evolution
    for minimizing a \(900\)-dimensional quadratic function
    over the spectrahedron, see Example~\ref{example:lazy_comparison}
    for details.
    We compare
    the performance of the vanilla Frank–Wolfe algorithm
    (FW, Algorithm~\ref{fw}), the Away-step Frank–Wolfe algorithm
    (AFW, Algorithm~\ref{away}),
    the Lazy Conditional Gradient algorithm
    (LCG, Algorithm~\ref{alg:ParamFreeLCG}), and the Lazy Away-step
    Frank–Wolfe algorithm
    (Lazy AFW, Algorithm~\ref{alg:ParamFreeLAFW}).
    Note that in spite of Lazy AFW converging slower than AFW in
    number of iterations (as expected), it outperforms it in wall-clock
    time.}
\end{figure}%
  In Figure~\ref{fig:lazy}
  we compare the primal gap evolution of the
  vanilla Frank–Wolfe algorithm (FW) and the Away-step Frank–Wolfe (AFW)
  algorithm with their lazy variants (LCG and Lazy AFW).
  For the separation oracle in the
  algorithms we first search among the vertices in the
  active set for an appropriate vertex. If no such vertex is
  found, the algorithms fall back to performing a linear
  minimization oracle call, and checking if the conditions in Oracle~\ref{ora:LPsep} are satisfied.
  Note that this requires maintaining an active set for the LCG iterates. The feasible region
  used is the set of all matrices in $\mathbb{R}^{d \times d}$
  of nuclear norm less than \(1\), i.e., the \myindex{spectrahedron},
  and the objective
    function is a quadratic $f(x) =
    \frac{1}{2} \norm[2]{Q x}^{2} + \innp{b}{x}$ where
    $Q\in \mathbb{R}^{d^2 \times d^2}$ and
    $b\in \mathbb{R}^{d^2}$ with $d = 30$. The entries
    of $Q$ and $b$ are selected uniformly
     at random from the interval between $0$ and $1$.
    The images in Figure~\ref{fig:lazy} show
    the evolution of the primal gap in iterations and time.
    Note that the results in time are implementation dependent,
    however they help to better
     understand the benefits and trade-offs of lazification.  As we can see,
     the Lazy AFW algorithm sometimes make less primal
     gap progress per iteration than the AFW algorithm, however, the
     Lazy AFW iterations are often cheaper than the AFW iterations, which
     makes the Lazy AFW converge faster in wall-clock time than the AFW
     algorithm. The cheaper computational cost of the Lazy AFW iterations is
     due to the fact that the Lazy AFW algorithm combines
     search for an away vertex with the search for a
     sufficiently good vertex that guarantees enough primal progress,
     so that the Lazy AFW has smaller additional cost that for LCG.
     On the other hand, in this instance
     progress in iteration is essentially the same for the lazy
     algorithms as their non-lazy counterparts,
     demonstrating that the advantage of lazy variants is cheaper
     oracle calls compared to linear minimization.

\end{example}

\subsection{Conditional Gradient Sliding algorithm}
\label{cgs}

In this section we will use the Euclidean norm.
Recall that the vanilla Frank–Wolfe algorithm requires
\(\mathcal{O}(1 / \varepsilon)\) gradient evaluations and
\(\mathcal{O}(1 / \varepsilon)\) linear optimizations
for obtaining a primal gap at most \(\varepsilon\)
for convex smooth objective functions.
In this section, we present the \emph{Conditional Gradient Sliding algorithm
(CGS)} \citep{lan2016conditional},
which reduces the required gradient evaluations to the minimal number
(see Table~\ref{tab:FOO-complexity} in Section~\ref{sec:oracle-complexity}),
without increasing the number of linear optimizations. In order to achieve this, CGS (Algorithm~\ref{cgs-smooth}) is essentially an implementation of 
Nesterov's accelerated projected gradient descent algorithm
\citep{nesterov27method},
using the vanilla Frank–Wolfe algorithm for performing approximate
projections back into the feasible region.  The key point is that the
latter does not require additional gradient evaluations of the
original functions: it is simply minimizing the quadratic arising from
projection in Euclidean norm. The subproblems in Line~\ref{cgs:sub} in
Algorithm~\ref{cgs-smooth} provide the approximate solution (with
accuracy $\eta_t$ in iteration $t$) to the following projection
problem
(note that the arguments of \(\argmin\) differ by a constant term
\(\norm[2]{\nabla f(x_{t})}^{2}\)
and a constant scaling factor \(\eta_{t} / 2\))
\begin{equation*}
  \argmin_{x \in \mathcal{X}} \innp{\nabla f(w_{t})}{x}
  + \frac{\eta_{t}}{2} \norm[2]{x - x_{t}}^{2}
  = \argmin_{x\in\mathcal{X}}
  \norm*[2]{x - x_{t} + \frac{1}{\eta_{t}} \nabla f(w_{t})}^{2}.
\end{equation*}
For completeness, the simplified conditional gradient subroutine that
arises is given in Algorithm~\ref{alg:CG}.

\begin{algorithm}[t]
\caption{Conditional Gradient Sliding (CGS) \citep{lan2016conditional}}
\label{cgs-smooth}
\begin{algorithmic}[1]
  \REQUIRE Start point $x_0\in\mathcal{X}$,
    step sizes $0 \leq \gamma_t \leq 1$,
    learning rates $\eta_t>0$,
    accuracies $\beta_t \geq 0$
  \ENSURE Iterates \(x_{1}, \dotsc \in \mathcal{X} \)
\STATE$y_0\leftarrow x_0$
\FOR{$t=0$ \TO \dots}
\STATE$w_t\leftarrow(1-\gamma_t)y_t+\gamma_tx_t$
\STATE\label{cgs:sub}
  $x_{t+1} \gets \operatorname{CG}(\nabla f(w_{t}),
  x_{t}, \eta_{t}, \beta_{t})$
\STATE$y_{t+1}\leftarrow y_t+\gamma_t(x_{t+1}-y_t)$
\ENDFOR
\end{algorithmic}
\end{algorithm}

\begin{algorithm}[t]
  \caption[]{\(\text{CG}(g_{0},u_{0},\eta,\beta)\)}
  \label{alg:CG}
  \begin{algorithmic}[1]
    \FOR{\(k = 0\) \TO \dots}
      \STATE \(v_{k} \leftarrow
        \argmin_{v\in\mathcal{X}} \innp{g_{k}}{v}\)
      \IF{\(\innp{g_{k}}{u_{k} - v_{k}} \leq \beta\)}
        \RETURN \(u_{k}\)
      \ENDIF
      \STATE\label{cgs:short}
        \(\alpha_{k} \leftarrow \min \left\{
          \frac{\innp{g_{k}}{u_{k} - v_{k}}}
          {\eta \norm{u_{k} - v_{k}}^{2}}
          ,
          1
  \right\}\) \COMMENT{step size according to short step rule}
	\STATE \(u_{k + 1} \leftarrow u_{k} + \alpha_{k}(v_{k} - u_{k})\) \COMMENT{update of primal solution}
      \STATE \(g_{k + 1} \leftarrow g_{0} + \eta(u_{k + 1} - u_{0})\) \COMMENT{gradient of the projection problem}
    \ENDFOR
  \end{algorithmic}
\end{algorithm}

\begin{theorem}
 \label{th:cgs1}
 Let $\mathcal{X}$ be a compact convex set with diameter $D$ and
 $f\colon \mathcal{X} \to \mathbb{R}$ be an $L$-smooth convex
 function in the Euclidean norm.
 Then CGS (Algorithm~\ref{cgs-smooth}) with
 $\gamma_t=3/(t+3)$, $\eta_t=3L/(t+2)$, and
 $\beta_t=LD^2/((t+1)(t+2))$ satisfies for all iterates $t \geq 1$,
 \begin{equation*}
  f(y_t)-f(x^*)
  \leq\frac{15LD^2}{2(t+1)(t+2)}.
 \end{equation*}
 Equivalently, the primal gap is at most \(\varepsilon\)
 after at most
  \(\mathcal{O}(\sqrt{L D^{2}/ \varepsilon})\)
  gradient computations and
  \(\mathcal{O}(L D^{2} / \varepsilon)\)
  linear minimizations.
\end{theorem}

Theorem~\ref{th:cgs1} shows that under the proposed setting of parameters, CGS
achieves the optimal LMO and FOO complexity
(Section~\ref{sec:oracle-complexity}).  However, it is
possible to improve the FOO complexity's dependence on $D$.  Below,
Theorem~\ref{th:cgs2} shows that by fixing the maximum number of iterations
ahead of time, one can replace the dependence on $D$ with a dependence on the
initial distance to the set of minimizers; observe the dependence of $\beta_t$
on $T$. In the general setting it is implicitly assumed that
$T=+\infty$. This improvement, now relying on the initial distance to
the minimizers, immediately allows for restart schemes as used in
Section~\ref{sec:second-order-sliding}.

\begin{theorem}
  \label{th:cgs2}
  Let $\mathcal{X}$ be a compact convex set with diameter $D$ and
  $f\colon \mathcal{X} \to \mathbb{R}$ be an $L$-smooth convex
  function.
  Let $D_{0} \geq \max_{x^* \in \argmin_{\mathcal{X}} f} \norm{x_{0}-x^*}$
  and fix the maximum number of iterations $T \geq 0$.
  Consider CGS (Algorithm~\ref{cgs-smooth}) with $\gamma_{t}=2/(t+2)$,
  $\eta_{t}=2L/(t+1)$, and $\beta_{t}=2LD_{0}^{2}/(T(t+1))$.
  Then
  \begin{equation*}
  f(y_T)-f(x^*)
  \leq\frac{6LD_0^2}{T(T+1)}.
  \end{equation*}
  Equivalently, the primal gap is at most \(\varepsilon\)
  after at most
  \(\mathcal{O}(L D_{0}^{2} / \varepsilon)\)
  gradient computations
  and
  \(\mathcal{O}\left( L D^{2} / \varepsilon
  + \sqrt{L D_{0}^{2} / \varepsilon} \right)\)
  linear minimizations.
\end{theorem}

In the absence of a better estimate for $D_0$ in Theorem~\ref{th:cgs2}, we can
always choose $D_0 = D$.  The achieved  FOO and LMO complexities in
Theorem~\ref{th:cgs2} are optimal for smooth convex minimization,
see Section~\ref{sec:oracle-complexity}.

See \citet{nonconvexCGS2017} for convergence rates for non-convex
objective functions.
\index{convergence for non-convex objective}

\begin{example}[CGS performance comparison for smooth convex
  functions]
  \label{example:CGS_cvx}
 \begin{figure}[t]
  \centering
  \begin{minipage}[b]{0.45\textwidth}
  \centering
  \includegraphics[width=\linewidth, alt={Graph of primal gap in the
    number of first-order oracle calls; Conditional Gradient Sliding
    algorithm slightly worse in the beginning but eventually slightly
    outperforming vanilla Frank–Wolfe algorithms.}]{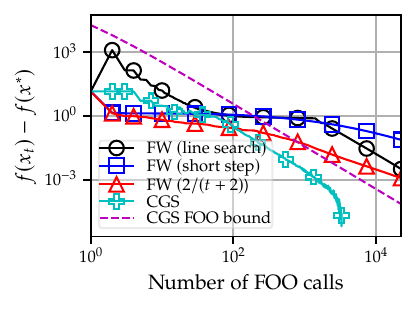}
    \subcaption{Primal gap convergence in FOO calls}
    \label{fig:CGS_cvx_FOO}
  \end{minipage}
  \quad
  \begin{minipage}[b]{0.45\textwidth}
    \centering
     \includegraphics[width=\linewidth, alt={Graph of primal gap in
       linear minimizations; Conditional Gradient Sliding
       eventually slightly
       outperforming vanilla Frank–Wolfe algorithms.}]{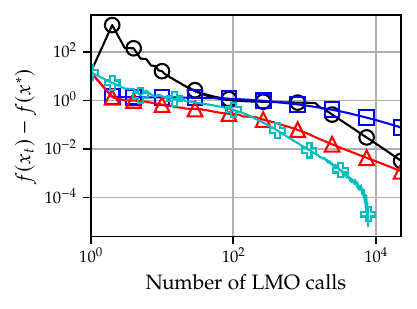}
    \subcaption{Primal gap convergence in LMO calls}
    \label{fig:CGS_cvx_LMO}
  \end{minipage}

  \caption{Primal gap convergence of a smooth convex function in
  oracle calls for the CGS algorithm (Algorithm~\ref{cgs-smooth}) for 
  and the vanilla FW algorithm (Algorithm~\ref{fw})
   with three different step size strategies. The problem optimizes
   a smooth convex but not strongly convex function over the space
   of positive semidefinite matrices of unit nuclear norm.
   Note that for
   this problem, in contrast to the FW algorithms which require 
   $\mathcal{O}(1/\varepsilon)$ calls to the FOO oracle to reach a
   primal gap at most \(\varepsilon\), the CGS algorithm
   requires $\mathcal{O}(1/\sqrt{\varepsilon})$ FOO calls.
   The dashed pink line is the theoretical bound on FOO calls from
   Theorem~\ref{th:cgs1}.}
 \label{fig:CGS_cvx}
\end{figure}
In Figure~\ref{fig:CGS_cvx} we compare the performance of the CGS algorithm
 (Algorithm~\ref{cgs-smooth}) for smooth convex functions and the FW
 algorithm (Algorithm~\ref{fw}), in FOO and LMO calls required
  to reach a primal gap at most $\varepsilon$, when minimizing
   a smooth convex function $f$ over the space of positive semi-definite
    matrices of unit trace in $\R^{10\times 10}$. We run the FW 
    algorithm with three different step size strategies, namely, 
    the $\gamma_t = 2/(2+t)$ step size, the short step rule, and 
    line search. Note that the CGS algorithm, does not require exact 
    line searches for $f$, but requires knowledge of the smoothness
    constant of
    the function and the diameter of the feasible region, or some 
    $D_{0} \geq \max_{x^* \in \argmin_{\mathcal{X}} f} \norm{x_{0}-x^*}$.

The objective function $f$ is generated by taking a random matrix
$Q$ in $\R^{n\times n}$ with entries drawn uniformly between $0$
and $1$ and $n=100$, and computing the spectral decomposition of 
$Q^{\top} Q = \sum_{i=1}^n \lambda_i u_i u_i^{\top}$ with
\(\lambda_{1} \geq \lambda_{2} \geq \dots \geq \lambda_{n} \geq 0\).
We then change the smallest eigenvalue of \(Q^{\top} Q\)
to zero, i.e,
compute $M^{\top} M= \sum_{i=1}^{n-1} \lambda_i u_i u_i^{\top}$, which is 
a positive semi-definite matrix.
(This equation does not determine \(M\) uniquely,
but it is not necessary to define the objective function,
as it depends only through \(M^{\top} M\) on \(M\).)
We also
construct a vector $b$ whose entries are drawn uniformly between $0$ and $1$, and 
we set $f(x) = \frac{1}{2} \norm{Mx}_2^{2} + \innp{b}{x}$.
The resulting objective function
is \(2500\)-smooth and convex.

We have included the FOO oracle complexity from Theorem~\ref{th:cgs1}
as a dashed line
on the left of Figure~\ref{fig:CGS_cvx}, to highlight
the difference between the algorithms' FOO complexities.
While FOO oracle complexity of CGS is
$\mathcal{O}(1/\sqrt{\varepsilon})$
for a primal gap of \(\varepsilon\),
that of the other algorithms shown on the image is
$\mathcal{O}(1/\varepsilon)$.
\end{example}

\subsubsection{CGS for strongly convex objectives}

In order to also match the optimal FOO complexity when $f$ is
additionally strongly convex, \citet{lan2016conditional} propose the
restart scheme presented in Algorithm~\ref{cgs-sc}, which is quite natural once one has obtained Theorem~\ref{th:cgs2} as mentioned earlier (see, e.g., \citet{pokutta2020restarting} for a discussion on restart schemes). It consists of
iteratively running CGS for a fixed number of iterations $T$, and
halving the accuracy schedule in-between. Note that $\gamma_{t}$,
$\eta_{t}$, and $\beta_{t}$ denote sequences for $t=0$ to $t=T-1$.

\begin{algorithm}[h]
\caption{Conditional Gradient Sliding (CGS) for strongly convex objectives \citep{lan2016conditional}}
\label{cgs-sc}
\begin{algorithmic}[1]
  \REQUIRE Start point $w_0\in\mathcal{X}$,
    estimate $\phi_{0} \geq f(w_0) - f(x^{*})$,
    smoothness constant $L>0$
    and strong convexity constant $\mu>0$
  \ENSURE Iterates \(x_{1}, \dotsc \in \mathcal{X}\)
\FOR{\(s = 0\) \TO \dots}
  \STATE Let $w_{s+1}$ be the last point output by the CGS algorithm
    (Algorithm~\ref{cgs-smooth}) when run for $T = \left\lceil
      2\sqrt{6L/\mu}\right\rceil$ iterations with parameters
    $x_{0} = w_{s}$, $\eta_{t} = 2L/(t+1)$, $\alpha_{t} = 2/(t+2)$ and
    $\beta_{t}=\frac{8L\phi_{0}s^{-2}}{\mu T(t+1)}$.
\ENDFOR
\end{algorithmic}
\end{algorithm}

\begin{theorem}
  \label{th:cgs3}
  Let $\mathcal{X}$
  be a compact convex set with diameter $D>0$
  and $f\colon\mathbb{R}^n\rightarrow\mathbb{R}$
  be an $L$-smooth $\mu$-strongly convex function
  in the Euclidean norm.
  Then the CGS algorithm for strongly convex functions
  (Algorithm~\ref{cgs-sc}) satisfies for all iteration $s\geq0$,
 \begin{equation*}
  f(w_s)-f(x^*)
  \leq\frac{\phi_{0}}{2^s}.
 \end{equation*}
 Equivalently, the primal gap is at most \(\varepsilon\)
 after at most
 \(\mathcal{O}(\sqrt{L / \mu} \log(\phi_{0} / \varepsilon))\)
 gradient computations and
 \(\mathcal{O}(L D^{2} / \varepsilon +
 \sqrt{L / \mu} \log (\phi_{0} / \varepsilon))\)
 linear minimizations.
\end{theorem}
While we have seen better convergence rates
for special convex domains,
like polytopes (Section~\ref{sec:line-conv-gener})
and strongly convex sets (Section~\ref{sec:impr-conv-strongly}),
the theorem provides a tradeoff for arbitrary convex domains,
drastically reducing the number of first-order oracle calls
in exchange for extra linear minimization,
compared to the vanilla Frank–Wolfe algorithm
(Theorem~\ref{fw_sub}).

\begin{example}[CGS performance comparison for the strongly convex case]\label{example:CGS_str_cvx}
  In Figure~\ref{fig:CGS_str_cvx}
  we compare the performance of the CGS algorithm (Algorithm~\ref{cgs-sc})
and the FW algorithm (Algorithm~\ref{fw}) for three different step size 
strategies, in FOO and LMO calls required to reach a
primal gap at most \(\varepsilon\), when minimizing a smooth,
strongly convex function $f$
over the \(500\)-dimensional unit \index{l1-ball@\(\ell_{1}\)-ball}\(\ell_{1}\)-ball.
As in Example~\ref{example:CGS_cvx}, the step sizes strategies used for 
FW are the $\gamma_t = 2/(2+t)$ step size, the short step rule, and line 
search.

\begin{figure}
  \centering
  \begin{minipage}[b]{0.45\textwidth}
    \includegraphics[width=\linewidth, alt={Graph of primal in
      first-oracle calls, Conditional Gradient Sliding eventually
      outperforming vanilla Frank–Wolfe algorithms.}]{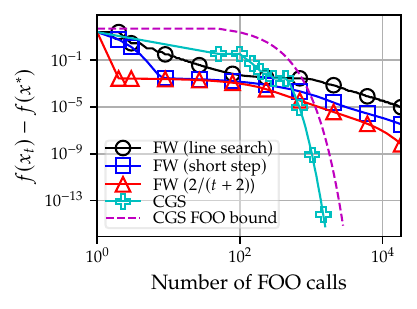}
    \subcaption{Primal gap convergence in FOO calls}
    \label{fig:CGS_str_cvx_FOO}
  \end{minipage}
  \qquad
  \begin{minipage}[b]{0.45\textwidth}
     \includegraphics[width=\linewidth, alt={Graph of primal in
       linear minimizations, Conditional Gradient Sliding behaving
       similarly tham vanilla Frank–Wolfe algorithms.}]{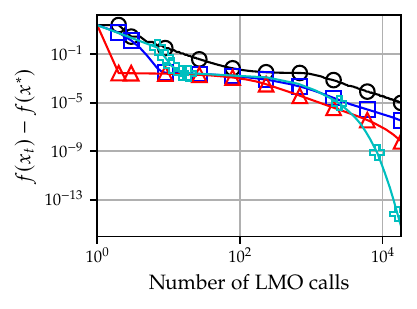}
    \subcaption{Primal gap convergence in LMO calls}
    \label{fig:CGS_str_cvx_LMO}
  \end{minipage}

  \caption{Primal gap convergence for smooth
    strongly convex optimization in
    oracle calls for the CGS algorithm (Algorithm~\ref{cgs-sc})
    and the vanilla FW algorithm
  (Algorithm~\ref{fw}) with three different step size strategies. 
  The optimum is in the interior of the feasible region,
  so the FW algorithm with short step rule and line search
  converges linearly.
    The dashed pink line is the theoretical bound on FOO calls
    from Theorem~\ref{th:cgs3}.}
  \label{fig:CGS_str_cvx}
\end{figure}

The objective function is generated by generating a random orthonormal 
basis for $\R^{500}$, selecting a minimum eigenvalue of $\mu = 1$, and 
a maximum eigenvalue of $L = 100$, a set of $498$ remaining eigenvalues
 drawn uniformly between $\mu$ and $L$, and using these ingredients to 
 build a positive definite matrix over $\R^{500}$. The unconstrained
 minimizer of the $f$ is placed in the interior of the unit
 $\ell_1$-ball, and so we expect to obtain linear convergence for the FW
   algorithm with the short step rule and line search (see Section~\ref{sec:linConvInterior}).

As we can observe from the figure shown on the left of 
Figure~\ref{fig:CGS_str_cvx}, the number of FOO calls needed 
to reach a primal gap at most \(\varepsilon\) is
$\mathcal{O}(\sqrt{L/\mu}\log 1/\varepsilon)$ for the CGS algorithm
 for strongly convex functions, whereas the FW algorithm requires
  $\mathcal{O}(L/\mu\left(D/r\right)^2\log 1/\varepsilon)$ calls,
  where $r$ is the radius of the largest ball around $x^*$ that
   is contained in $\mathcal{X}$. We have added a dashed line 
   depicting the bound on the number of FOO calls from 
   Theorem~\ref{th:cgs3} to aid in the comparison of the 
   algorithms. Note that the FOO oracle complexity of the 
   CGS algorithm has a better dependence on the condition 
   number $L/\mu$, and does not have the dimension dependent 
   term $D/r$.  Using the AFW algorithm
   (Algorithm~\ref{away}) with the short step rule 
   or line search to solve this problem, we obtain
   very similar results to those of the FW algorithm 
   with short step or line search, respectively, in FOO and LMO calls,
   but with a a higher computational
     cost, due to the need to maintain an active set at 
     each iteration. Note that the FW algorithm does not
      converge linearly in general for this class of 
      problems (smooth and strongly convex), while the
       AFW algorithm converges linearly if the feasible
        region is a polytope. On the other hand, the CGS
         algorithm (Algorithm~\ref{cgs-sc}) requires a 
         linear number of FOO calls and a sublinear number
          of LMO calls for any compact convex feasible region.
\end{example}

\subsection{Sharpness a.k.a.~Hölder Error Bound (HEB)}
\label{sec:adaptive_rates}

Sharpness is a weakening of strong convexity,
introduced in \citet{GradDominance},
leading to convergence rates
between \(\mathcal{O}(1/\varepsilon)\) and
\(\mathcal{O}(\log 1/\varepsilon)\).
The weakening comes in three aspects
\begin{enumerate*}[after=., itemjoin={{, }}, itemjoin*={{, and }}]
\item allowing a non-unique minimum to the objective function
\item generalizing the exponent \(2\)
\item local to the neighborhood of the minima (only)
\end{enumerate*}
Sharpness is also known as the
\emph{Hölder(-ian) error bound} \citep{hoffman52holder}, the \emph{sharp minima
condition} \citep{polyak79} (more generally Polyak-Łojasiewicz condition), or
the \emph{strict minima condition} \citep{nem85opt}.  Although it is not as
well studied as strong convexity, it has been widely used to obtain faster
rates than \(\mathcal{O}(1 / \varepsilon)\),
up to linear convergence (see, e.g.,
\citet{alex17sharp,xu18heb,kerdreux2018restarting,combettes19bmp,UniformConvexFW_2020} for recent
developments).
Similarly, Hölder smoothness generalizes smoothness by using
a different exponent than $2$ and can be combined with sharpness to capture a wide variety of convergence regimes. Here we will confine ourselves to the standard smooth case though and we refer the reader to \citet{kerdreux2018restarting}
for details and convergence results. 

Sharpness basically captures how fast the
function increases around its optimal solutions,
see Figure~\ref{fig:sharpness} for simple examples.
Recall that we denote by \(\Omega^*\)
the set of optimal solutions to \(\min_{x \in \mathcal X} f(x)\); if the feasible region $\mathcal X$ is not clear from the context we write \(\Omega_{\mathcal X}^*\).

\begin{definition}[Sharpness a.k.a. Hölder Error Bound condition]
\label{def:sharpness}
  Let
  \(\mathcal X\) be a compact convex set. A convex function \(f\) is
  \emph{\((c,\theta)\)-sharp (over $\mathcal X$)} with \(0 < c < \infty\),
  \(0 < \theta \leq  1\)
  if for all \(x \in \mathcal{X}\) and $x^* \in \Omega_{\mathcal X}^*$ we have
  \begin{equation}
	  c\cdot\bigl(f(x) - f(x^*)\bigr)^\theta \geq \min_{y \in \Omega_{\mathcal X}^*} \norm{x-y}.
  \end{equation}
\end{definition}

\begin{figure}
\centering
\includegraphics[width=.45\linewidth, alt=]{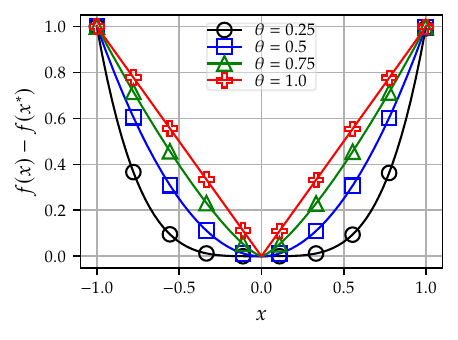}
  \caption{Graph of \((1, \theta)\)-sharp function
    $f(x) = \abs{x}^{1/\theta}$ for various \(\theta\)}
 \label{fig:sharpness}
\end{figure}

The restriction \(\theta \leq 1\) is justified by the easy observation that a
non-constant
convex function cannot be \((c, \theta)\)-sharp for \(\theta > 1\). Moreover,
it is straightforward that a non-constant \(L\)-smooth and \((c, \theta)\)-sharp function must
satisfy $\theta\leq 1/2$ \citep{nem85opt}
(and \(L \geq 2 / c^{2}\) for \(\theta = 1/2\), which will be used
in Corollary~\ref{cor:AFWconvergence-sharp}).
Often sharpness is required only in
a neighborhood of \(\Omega^{*}\), resulting in convergence rates improving only
after an initial burn-in phase, see
\citet{kerdreux2018restarting,combettes19bmp}.  For simplicity of exposition we
assume sharpness over the whole feasible region \(\mathcal{X}\).

Every \(\mu\)-strongly convex function is clearly \((\sqrt{2/\mu},
1/2)\)-sharp however there are \((c, 1/2)\)-sharp functions
which are not strongly convex, see the following example.
These appear in, e.g., signal processing and machine learning,
see, e.g., \citet{garber2015faster,garber2020sparseFW} and follow-up
work by these authors as well as
\citet{kerdreux2018restarting,combettes19bmp} and references contained
therein.

\begin{example}[Distance to a closed convex set inside the domain]
  \label{ex:sharp-distance}
  An easy example of a sharp function is distance to a
  non-empty closed convex set \(C\)
  \begin{equation*}
    f_{C}(x) \defeq \min_{y \in C} \norm{x - y}
    .
  \end{equation*}
  Obviously, \(\Omega^{*} = C\), and
  \(f_{C}\) is \((1,1)\)-sharp over any convex set \(\mathcal{X}\)
  containing \(C\), as
  the defining inequality of
  sharpness holds with equality for all \(x \in \mathbb{R}^{n}\).
\end{example}

We turn now to polytope domains, for which
we recall a special case of the Hoffman bound
from \citet{hoffman52holder}.
\begin{lemma}[Hoffman bound]
  Let \(P\) be a polytope and \(A\) a linear function on \(P\)
  (with values in an arbitrary vector space).
  Then for some \(c > 0\) for any \(x, y \in P\) one has
  \begin{equation}
    \label{eq:Hoffman-polytope}
    \norm{A x - A y} \geq c \distance{x}{\{z \in P \mid A z = A y\}}
    .
  \end{equation}
  In particular, \(c\) depends only on \(P\) and \(A\).
\end{lemma}
This is the key to the following example.

\begin{example}[Composition of a strongly convex function and a linear
  one over a polytope]
  \label{ex:sharp-polytope}
  A common class of \((c, 1/2)\)-sharp functions
  consist of functions of the
  form \(f(x) = g(Ax)\),
  a composition of a \(\mu\)-strongly convex function \(g\) and
  a linear function \(A\) over a polytope \(\mathcal{X}\).
  This class includes functions, which have several minima and hence
  are not strongly convex, this is the case for non-injective \(A\).
  In other words, these are functions like \(f(x, y) = x^{2}\),
  which are strongly convex in some coordinates
  (\(x\) in the example),
  and are independent of the rest of coordinates
  (\(y\) in the example).

  For completeness, we reproduce the verification of sharpness
  following \citet{garber2019logarithmic,beck2017linearly}.
  Besides strong convexity, we use
  \(\nabla f(x) = A^{\top} \nabla g(A x)\),
  i.e., \(\innp{\nabla g(Ax)}{Ay} = \innp{\nabla f(x)}{y}\),
  and also that \(\innp{\nabla f(x^{*})}{x - x^{*}} \geq 0\)
  for \(x \in \mathcal{X}\) by minimality.
  \begin{equation}
    \label{eq:strongly-convex-by linear}
   \begin{split}
    f(x) - f(x^{*}) 
    &
    = g(Ax) - g(Ax^{*})
    \\
    &
    \geq
    \innp{\nabla g(Ax^{*})}{Ax - Ax^{*}}
    + \mu \frac{\norm{Ax - Ax^{*}}^{2}}{2}
    \\
    &
    =
    \innp{{\nabla f(x^{*})}}{x - x^{*}}
    +
    \mu \frac{\norm{Ax - Ax^{*}}^{2}}{2}
    \geq
    \mu \frac{\norm{Ax - Ax^{*}}^{2}}{2}
    .
   \end{split}
  \end{equation}
  This is very close to strong convexity, the only difference is
  having the distance between \(Ax\) and \(Ax^{*}\)
  instead of \(x\) and \(x^{*}\).
  Nonetheless we conclude that the minimal set of \(f\)
  is \(\Omega^{*} = \{z \in P \mid A z = A x^{*}\}\).
  Thus we continue by the Hoffman bound
  (Equation~\eqref{eq:strongly-convex-by linear}),
  for some \(c > 0\) depending only on \(f\) and \(\mathcal{X}\)
  \begin{equation*}
    f(x) - f(x^{*})
    \geq
    \mu \frac{\norm{Ax - Ax^{*}}^{2}}{2}
    \geq
    \frac{\mu c^{2}}{2}
    \distance{x}{\Omega^{*}}^{2}
    .
  \end{equation*}
  Thus \(f\) is \((\sqrt{2 / (\mu c^{2})}, 1/2)\)-sharp
  over \(\mathcal{X}\).

  When the polytope \(\mathcal{X}\) is bounded,
  there is a slight generalization
  with an additional linear summand:
  the function \(f(x) = g(Ax) + \innp{b}{x}\) is \(1/2\)-sharp
  (with still \(g\) strongly convex, \(A\) linear).
  See \citet[Lemma~2.5]{beck2017linearly} for a proof, we only note
  boundedness is used to control the additional linear summand,
  e.g., \(f(x) = \innp{b}{x}\) is obviously not \((c, \theta)\)-sharp
  for \(\theta < 1\)
  over any polytope where \(f\) is unbounded above.
\end{example}

Obviously, for any sharp function on \(\mathbb{R}^{n}\),
its restriction to a convex set containing \(\Omega^{*}\),
the set of all its minima in \(\mathbb{R}^{n}\),
is sharp with the same parameters.
However, sharpness needs no longer be preserved
when restricting to convex sets not containing \(\Omega^{*}\).
This makes it much harder to
construct sharp functions over a bounded domain.
We present examples highlighting this subtle fact,
showing in particular that
several assumptions in the constructions above are not superfluous.

\begin{example}[Difficulties with sharpness over bounded domains]
  \label{ex:non-sharp}
  As a warm-up, we consider the function commonly used for
  constructing local, infinitely differentiable functions:
  \(h \colon \mathbb{R} \to \mathbb{R}\)
  via \(h(x) = e^{- 1/x^{2}}\) (extending continuously to the point
  \(0\), i.e., \(h(0) = 0\)).
  This function has a unique minimum at \(0\) with value \(h(0) = 0\),
  and it is not sharp:
  \(\lim_{x \to 0} (h(x) - h(0))^{\theta} / \abs{x} = \infty\)
  for all \(\theta > 0\).
  However the function is convex in a neighborhood of \(0\),
  i.e., for \(- \varepsilon \leq x \leq \varepsilon\)
  with a suitable \(\varepsilon > 0\).

  We use the function \(h\) to construct other examples of non-sharp
  functions under the Euclidean norm.
  Consider
  the closed convex planar set
  \(\smash{\mathcal{X} \defeq \{(x, y) \mid
    - \varepsilon \leq x \leq \varepsilon},
    \allowbreak
h(x) \leq y \leq h(\varepsilon) \}\)
  with the standard Euclidean norm.
  Consider the quadratic function \(f(x, y) = y^{2}\),
  which is the composition of the strongly convex \(1\)-dimensional
  function \(g \colon \mathbb{R} \to \mathbb{R}\)
  and the linear function \(A \colon \mathbb{R}^{2} \to \mathbb{R}\)
  defined via \(g(t)=t^{2}\) and \(A(x, y) = y\).
  We show that it is not sharp over \(\mathcal{X}\),
  a compact convex set, which is not a polytope,
  in contrast to Example~\ref{ex:sharp-polytope}.
  Restricted to \(\mathcal{X}\),
  it clearly has a unique minimum, namely, \(x = y = 0\).
  However, \(f\) is not sharp,
  as for all \(\theta > 0\) we have
  \(\lim_{x \to 0} (f(x, h(x)) - f(0, 0))^{\theta} /
  \norm[2]{(x, h(x)) - (0, 0)}
  = \lim_{x \to 0} h(x)^{2 \theta} /
  \sqrt{x^{2} + h(x)^{2}}
  \geq
  \lim_{x \to 0} h(x)^{2 \theta} / (\sqrt{2} \abs{x}) = + \infty
  \).
  Similarly, the linear function \(A\) is not sharp either
  on \(\mathcal{X}\),
  which is the distance function
  to the interval \smash{\(\mathcal{Y} \defeq \{ (x, 0) \mid -\varepsilon
  \leq x \leq \varepsilon\}\)},
  in contrast to Example~\ref{ex:sharp-distance}.

  A variant of this example is the distance function \(d\) to
  \(\mathcal{X}\) on \(\mathcal{Y}\),
  which has its unique minimum also at \((0, 0)\),
  and satisfies \(d(0,0) = 0 \leq d(x, 0) \leq
  \norm[2]{(x, 0) - (x, h(x))} = h(x)\)
  for \(- \varepsilon \leq x \leq \varepsilon\),
  again showing that \(d\) is not sharp because \(h\) is not sharp.
  In this example the feasible region is even a polytope.
\end{example}

Interestingly, in the unconstrained case ($\mathcal X = \mathbb{R}^n$)
a property (which is beyond the scope of the discussion here) closely related to sharpness
almost always holds in finite dimensional spaces.
As a consequence, any real-valued
lower semicontinuous, convex, subanalytic function \(f\)
defined on a neighborhood \(U\) of a compact convex set
\(\mathcal{X}\) is sharp, provided
its minimum set \(\Omega^{*}_{U}\) on \(U\)
is non-empty and contained in \(\mathcal{X}\).
The Huber loss is an example of such a function,
which is an average distance to a set of points
based on a smoothened version of the absolute value function,
see \citet{combettes19bmp}.
Note that the functions in Example~\ref{ex:non-sharp} violate
the condition \(\Omega^{*}_{U} \subseteq \mathcal{X}\).

\begin{lemma}[\citet{bolte07lojo}]
\label{lem:bolte}
Let $f \colon \mathbb{R}^{n} \to \mathbb{R}\cup\{+\infty\}$
be a lower semicontinuous, convex, and subanalytic function with a
global minimum.
Then for
any bounded set $\mathcal{K}\subset\mathbb{R}^n$, there exists $0 <
\theta < 1$ and $c>0$ such that for all $x\in\mathcal{K}$,
\begin{equation*}
  \distance*{x}{\argmin_{\mathbb{R}^{n}} f}
  \leq c \left( f(x) - \min_{\mathbb{R}^{n}}f \right)^{\theta}.
\end{equation*}
\end{lemma}

In general there is no easy way to estimate sharpness parameters,
which are required for most gradient descent style algorithms in order
to exploit sharpness.  However via restarts
\citep[see][]{alex17sharp,kerdreux2018restarting},
i.e., trying out various values for sharpness parameters in a smart way,
gradient descent style algorithms can be made adaptive to these
parameters in order to achieve these improved convergence rates.
This is unnecessary for conditional gradient algorithms as they usually do not explicitly depend on these parameters and in fact are adaptive to those parameters using line search or the short step rule
\citep[see][]{xu18heb,kerdreux2018restarting,combettes19bmp}, i.e., conditional gradient algorithms usually automatically adapt to the sharpness parameters without explicit knowledge of them.

The first-order oracle complexity
for minimization of smooth, convex, \((c, \theta)\)-sharp functions
over closed convex sets is
$\Omega(\min\{n,1/\varepsilon^{1/2-\theta}\})$ for \(\theta < 1/2\)
and $\Omega(\min\{n,\ln(1/\varepsilon)\})$ for \(\theta = 1/2\), where \(n\) denotes the linear
dimension of the ambient space
(i.e., \(\mathcal{X} \subseteq \mathbb{R}^{n}\)) \citep{nem85opt};
clearly the bounds are of interest only in the so-called large-scale
regime where $n$ is assumed to be very large (otherwise other methods,
such as e.g., conjugate gradients may exist with complexity
$\mathcal{O}(n)$). For conditional gradient algorithms
\citep{kerdreux2018restarting,combettes19bmp},
these complexity bounds are achieved for \(\theta = 1/2\),
and within a factor of \(2\) in the exponent \(1/2 - \theta\)
for \(\theta < 1/2\).
Rather than reproducing these results,
we will present only the key estimation enabled by sharpness,
and as an illustrative example,
derive sharpness-induced convergence rates for the
Away-step Frank–Wolfe algorithm (see Algorithm~\ref{away}).

\begin{lemma}[Primal gap bound from sharpness]
  \label{lem:hebprimalgap}
  Let \(f\) be a \((c, \theta)\)-sharp convex function.
  Then for all \(x \in \mathcal{X}\)
  there is an optimal point \(x^{*} \in \Omega^{*}\) depending on \(x\)
  \begin{equation}
    \label{eq:hebprimalgap}
    \frac{1}{c}\bigl(f(x) - f(x^*)\bigr)^{1-\theta} \leq \frac{\innp{\nabla
        f(x)}{x-x^*}}{\norm{x-x^*}}.
  \end{equation}
 \begin{proof}
 This simply follows from convexity and sharpness:
 \begin{equation*}
   \frac{\innp{\nabla f(x)}{x - x^*}}{\norm{x - x^*}}
   \geq
   \frac{f(x) - f(x^{*})}{c \bigl(f(x) - f(x^{*})\bigr)^{\theta}}
   =
   \frac{1}{c} \bigl(f(x) - f(x^{*})\bigr)^{1 - \theta}
   .
   \qedhere
 \end{equation*}
\end{proof}
\end{lemma}

Turning to the Away-step Frank–Wolfe algorithm,
we follow the proof of Theorem~\ref{th:AFWConvergence}. Our
starting point is the following straightforward replacement of
Lemma~\ref{lem:geoSC} by combining Lemma~\ref{lem:hebprimalgap}
with the scaling inequality from Lemma~\ref{lemma:pyraScaling}. Note
that all other modifications work the same way: we simply combine the
scaling inequality for the respective setting with Lemma~\ref{lem:hebprimalgap}.

\begin{lemma}[Geometric sharpness\index{geometric sharpness}]
  \label{lem:geoSharp}
  Let \(P\) be a polytope with pyramidal
  width \(\delta > 0\) and let \(f\) be a \((c, \theta)\)-sharp
  convex function over \(P\).
  Recall from Section~\ref{sec:line-conv-gener}
  the Frank–Wolfe vertex
  \(v^{\text{FW}} = \argmin_{v \in P} \innp{\nabla f(x)}{v}\)
  and the away vertex \(v^{\text{A}}
  = \argmax_{v \in \mathcal{S}} \innp{\nabla f(x)}{v}\) with $S
  \subseteq \vertex{P}$, so that $x \in \conv{S}$.
  Then
  \begin{equation}
    \label{eq:geoSharp}
 \frac{1}{c}\bigl(f(x) - f(x^*)\bigr)^{1-\theta}
    \leq
    \frac{\innp{\nabla f(x)}{v^{\text{A}}
        - v^{\text{FW}}}}
    {\delta}.
\end{equation}

\begin{proof}
  The statement follows from combining Equation~\eqref{eq:pyraScale}
  with Lemma~\ref{lem:hebprimalgap} for \(\psi = \nabla f(x)\).
  We have
  \begin{equation*}
 \frac{1}{c}\bigl(f(x) - f(x^*)\bigr)^{1-\theta}
    \leq
    \frac{\innp{\nabla f(x)}{x - x^{*}}}
    {\norm{x - x^{*}}}
    \leq
    \frac{\innp{\nabla f(x)}{v^{\text{A}}
        - v^{\text{FW}}}}
    {\delta}.
    \qedhere
  \end{equation*}
\end{proof}
\end{lemma}

We are ready to generalize the convergence rate of
Away-step Frank–Wolfe algorithm (Algorithm~\ref{away})
to sharp objective functions.
As mentioned earlier, \citet{beck2017linearly} proved linear
convergence for functions of the form
\(f(x) = g(Ax) + \innp{b}{x}\) with \(g\) strongly convex and \(A\) linear).

\begin{corollary}
  \label{cor:AFWconvergence-sharp} Let $f$ be
  a \((c, \theta)\)-sharp, \(L\)-smooth function \(f\). Then
  the Away-step Frank–Wolfe algorithm (Algorithm~\ref{away})
  has the following convergence rates in primal gap after
  \(t\) iterations over a polytope $P$ with diameter \(D\)
  and pyramidal width \(\delta\).
  For \(\theta < 1/2\) the convergence rate is
  \begin{align}
    \label{eq:AFWconvergence-sharp}
    h_{t}
    &
    \leq
    \begin{cases}
      \frac{L D^{2}}{2^{\lceil (t-1) / 2 \rceil}}
      & 1 \leq t \leq t_{0}, \\[1ex]
      \frac{(c^{2} L D^{2} / \delta^{2})^{1 / (1 - 2 \theta)}}
      {\bigl(1 + (1/2 - \theta)
        \lceil (t - t_{0}) / 2 \rceil\bigr)^{1 / (1 - 2 \theta)}}
      = \mathcal{O}\bigl((2 / t)^{1 \mathbin{/} (1 - 2 \theta)}\bigr)
      & t \geq t_{0}
      ,
    \end{cases}
    \\[2ex]
    \text{where } t_{0}
    &
    \defeq
    \max
    \left\{
      2 \left\lfloor
        \log_{2}
        \left(
          \frac{L D^{2}}{\bigl(c^{2} L D^{2} / \delta^{2}\bigr)^{1 \mathbin{/}
              (1 - 2 \theta)}}
        \right)
      \right\rfloor
      + 1,
      1
    \right\}
    .
    \intertext{For \(\theta = 1/2\) the convergence rate is}
    \label{eq:AFWconvergence-sharp1/2}
    h_{t}
    &
    \leq
    \left(
      1 - \frac{\delta^{2}}{2 c^{2} L D^{2}}
    \right)^{\lceil (t-1)/2 \rceil}
    \frac{L D^{2}}{2}.
  \end{align}
  Equivalently, for a primal gap error of at most \(\varepsilon > 0\)
  the algorithm performs at most the following number of linear
  optimizations:
  \begin{equation}
    \begin{cases}
    \mathcal{O} \left(
      \frac{c^{2} L D^{2}}
      {\delta^{2} (1/2 - \theta) \varepsilon^{1 - 2 \theta}}
    \right)
    &
    \theta < 1/2,\
    \varepsilon \leq (c^{2} L D^{2} / \delta^{2})^{1 / (1 - 2 \theta)}
    ,
    \\[1ex]
    1 + \frac{4 c^{2} L D^{2}}{\delta^{2}}
    \ln \frac{L D^{2}}{2 \varepsilon}
    &
    \theta = 1/2.
\end{cases}
  \end{equation}

\end{corollary}
\begin{proof}
The proof largely follows that of Theorem~\ref{th:AFWConvergence},
so we point out only the difference.
Below the first inequality is similar to \eqref{eq:AFWstep},
but then we apply Lemma~\ref{lem:geoSharp}, for which
recall that \(d_{t} = v_{t}^{\text{FW}} - v_{t}^{\text{A}}\)
is the pairwise direction.
\begin{equation}
  \label{eq:AFWstepsharp}
 \begin{split}
  h_{t} - h_{t+1}
  &
  \geq
  \min \left\{\gamma_{t, \max} \frac{\innp{\nabla f(x_{t})}{-d_{t}}}{2},
    \frac{\innp{\nabla f(x_{t})}{-d_{t}}^{2}}{2 L \norm{d_{t}}^{2}}
  \right\}
  \\
  &
  \geq
  \min
  \left\{
    \gamma_{t, \max} \frac{h_{t}}{2},
    \frac{\delta^{2} h_{t}^{2(1-\theta)}}{2c^2 L D^{2}}
  \right\}
  = \min\left\{
    \frac{\gamma_{t, \max}}{2}, \frac{\delta^{2} h_{t}^{1-2\theta}}{2c^2 L D^{2}}
  \right\}
  \cdot h_{t}.
 \end{split}
\end{equation}
Therefore,
\begin{equation}
  \label{contraction:afwsharp}
  h_{t+1}
  \leq
  \left( 1 -
  \min \left\{
    \frac{\gamma_{t, \max}}{2},
    \frac{\delta^{2} h_{t}^{1-2\theta}}{2c^2 L D^{2}}
  \right\}
  \right)
  h_{t}
  .
\end{equation}
From this the convergence rate follow by applying
Lemma~\ref{lem:ConvergenceRate}.  The lemma is applied to the sequence
\(h_{t}\) with iterations resulting in drop steps omitted,
so that the step inequalities hold with \(\gamma_{t, \max}\) replaced
by \(1\).
The omission of drop steps is compensated by adjustment of indices,
replacing \(t\) with \(\lceil (t-1) / 2 \rceil\) 
in the exponents exploiting the fact that
at most half of the steps can be drop steps up to any iteration~\(t\).
For \(\theta = 1/2\), we additionally use \(L \geq 2 / c^{2}\)
and \(\delta \leq D\) to simplify the constant
\(\min\{1/2, \delta^{2} / (2 c^{2} L D^{2})\}
= \delta^{2} / (2 c^{2} L D^{2})\)
in the recursion.
\end{proof}

\begin{remark}[Connection to gradient dominated property]
	\label{rem:gradientDom}
  Closely related to sharpness is the
  \emph{gradient dominated property},
  a special case of which was already introduced in
  Definition~\ref{def:grad-dominate}:
  a function \(f\)
  is \emph{$(c,\theta)$-gradient dominated} for some \(c, \theta > 0\)
  if for all \(x\) in the domain of \(f\)
  \begin{equation}
  \label{eq:gradDom}
  \frac{1}{c}\bigl(f(x) - f(x^{*})\bigr)^{1-\theta} \leq \dualnorm{\nabla f(x)},
\end{equation}
where \(x^{*}\) is the minimizer of \(f\) as usual.
In particular, \emph{$c$-gradient dominated} is the case
where $\theta = 1/2$.
Gradient dominance, similar to sharpness, is a local property
and can be considered a natural weakening of strong convexity.
As already mentioned,
most convergence proofs that involve strong convexity almost
immediately carry over to the gradient dominated case.

At least for smooth functions with minima lying in the relative
interior of the domain, gradient dominance is equivalent to
sharpness by the following lemma; see \citet{bolte17}, which appeared online in 2015, for more on this.

\end{remark}
\begin{lemma}[Gradient dominated versus sharp]\index{gradient
    dominated function}
  \label{lem:gradient-dominated}
  Let \(\mathcal{X} \subseteq \mathbb{R}^{n}\)
  be a compact convex set,
  \(f \colon \mathcal{X} \to \mathbb{R}\)
  be a differentiable convex function,
  and \(0 < \theta < 1\) be a number.
  If \(f\) is \((c, \theta)\)-sharp for some \(c > 0\)
  then it is
  \((c, \theta)\)-gradient dominated.
  Conversely, if \(f\) is smooth,
  \((c, \theta)\)-gradient dominated for some \(c > 0\),
  and its set \(\Omega^{*}\) of minima
  lies in the relative interior of \(\mathcal{X}\),
  then \(f\) is \((c', \theta)\)-sharp for some \(c' > 0\).
\begin{proof}
First assume \(f\) is \((c, \theta)\)-sharp.
By Lemma~\ref{lem:hebprimalgap}, for any \(x \in \mathcal{X}\)
there is an \(x^{*} \in \Omega^{*}\) such that
\begin{equation}
  \frac{1}{c}\bigl(f(x) - f(x^*)\bigr)^{1-\theta}
  \leq
  \frac{\innp{\nabla f(x)}{x - x^*}}{\norm{x - x^*}}
  \leq
  \dualnorm{\nabla f(x)}
  ,
\end{equation}
showing that \(f\) is \((c, \theta)\)-gradient dominated.

Conversely, assume \(f\) is smooth,
\((c, \theta)\)-gradient dominated
and its minimum set \(\Omega^{*}\)
lies in the relative interior of \(\mathcal{X}\),
Let \(M \defeq \min_{\partial \mathcal{X}} f\) be the minimum value of
\(f\) on the boundary of \(\mathcal{X}\), then \(U \defeq \{x \in
\mathcal{X} \mid f(x) < M\}\) is an open neighborhood of
\(\Omega^{*}\) in the relative interior of \(\mathcal{X}\).

As all norms are equivalent, we assume without loss of generality
that the norm \(\norm{\cdot}\) is the Euclidean norm and
\(\innp{\cdot}{\cdot}\) is the standard scalar product.
Let \(x \in U\) be arbitrary.
Consider the maximal solution to the differential equation
\(y(0) = x\), \(y'(t) = - \nabla f(y(t))\), which
uniquely exists by the Picard–Lindelöf theorem using Lipschitz
continuity of \(\nabla f\)
(which follows from convexity and
smoothness of \(f\), see Lemma~\ref{lem:smooth-Lipschitz}).
This choice ensures \(\innp{\nabla f(y(t))}{y'(t)} = -
\dualnorm{\nabla f(y(t))} \cdot \norm{y'(t)}\).
In particular,
\((f \circ y)'(t) = \innp{\nabla f(y(t))}{y'(t)} \leq 0\),
i.e., \(f \circ y\) is monotonically decreasing, and therefore \(y(t) \in
U\) for \(t \geq 0\).
The domain of \(y\) is an interval \(\{t \mid t_{-} < t < t_{+}\}\).
By standard arguments, all accumulation points of \(y\) at \(t_{+}\)
lie in \(\Omega^{*}\).  Let \(x^{*}\) be one such accumulation point.
Then
\begin{equation*}
 \begin{split}
  \bigl(f(x) - f(x^{*})\bigr)^{\theta}
  &
  =
  \theta \int_{t=0}^{t_{+}} \Bigl(f(x) - f\bigl(y(t)\bigr)\Bigr)^{\theta - 1}
  \innp{- \nabla f\bigl(y(t)\bigr)}{y'(t)} \mathrm{d} t
  \\
  &
  =
  \theta \int_{t=0}^{t_{+}} \Bigl(f(x) - f\bigl(y(t)\bigr)\Bigr)^{\theta - 1}
  \dualnorm{\nabla f\bigl(y(t)\bigr)} \cdot \norm{y'(t)} \mathrm{d} t
  \\
  &
  \geq
  \frac{\theta}{c} \int_{t=0}^{t_{+}} \norm{y'(t)} \mathrm{d} t
  \geq
  \frac{\theta}{c} \norm{x - x^{*}}
  .
 \end{split}
\end{equation*}
The first inequality follows from \(f\) being gradient dominated,
while the second inequality states that the length of the curve
\(\{y(t)\}, 0 \leq t \leq t_{+}\) is at least the distance between its
start point and the accumulation point \(x^{*}\) at the other end of
the curve.

Thus \(f\) is \((c/\theta, \theta)\)-sharp on \(U\).
It follows that \(f\) is \((c', \theta)\)-sharp on \(\mathcal{X}\)
for some \(c' > 0\) by standard arguments, as
\((f(x) - f(y))^{\theta} / \norm{x-y}\)
has a positive lower bound for \(x \in \mathcal{X} \setminus U\)
and \(y \in \Omega^{*}\), on a compact set
where it is continuous and positive.
\end{proof}
\end{lemma}

\subsection{Frank–Wolfe with Nearest Extreme Point oracle}
\label{sec:near-extreme-point}
The core of conditional gradients algorithms is repeatedly
optimizing \emph{only linear} functions instead of more complicated
functions, such as, e.g., quadratic functions which would allow us to express projections.
The \emph{Nearest Extreme Point (NEP) oracle} of \citet{garber2021NEP-FW}
interpolates between linear and quadratic objectives:
it optimizes a quadratic function but \emph{only over the extreme points} of the
feasible region, where the quadratic part is the Euclidean norm
squared, i.e., a rather simple quadratic.  More precisely, the \emph{NEP oracle} solves the problem
\begin{equation}
	\label{eq:nep}
	\argmin_{v \in \vertex{\mathcal{X}}} \innp{c}{v}
	+ \lambda \norm[2]{v - x}^{2}
\end{equation}
for any \(c\), \(\lambda\), and \(x\).
In other words, the NEP oracle
returns the nearest extreme point to \(x - c / (2 \lambda)\) (in $\ell_{2}$-norm).

The NEP oracle is particularly useful
when all the extreme points of the
feasible region lie on an \(\ell_{2}\)-ball, since over
\(\ell_{2}\)-balls the objective
function minimized by the NEP oracle is actually a linear function,
so that the cost of the oracle becomes that of linear minimization.
Examples of such feasible regions include
spectrahedra, trace-norm balls, and \(0/1\)-polytopes.
(Recall that \(0/1\)-polytopes are the polytopes,
with all vertices that are \(0/1\)-vectors,
which lie on an \(\ell_{2}\)-ball with center
\(\allOne/2\), the vector all whose entries are \(1/2\)).
As the step size is also a parameter
to the NEP oracle, it needs to be fixed \emph{before} the call to the
NEP oracle, ruling out any form of line search.
Hence the step size is chosen in the usual function-agnostic manner.

\begin{algorithm}
	\caption{\label{alg:NEP-FW}%
		Frank–Wolfe with Nearest Extreme Point Oracle (NEP FW) \citep{garber2021NEP-FW}}
	\begin{algorithmic}[1]
    \REQUIRE Start vertex \(x_{0} \in \vertex{\mathcal{X}}\),
		objective function \(f\),
    \ENSURE Iterates \(x_{1}, x_{2}, \dotsc \in \mathcal{X}\)
		\FOR{\(t=0\) \TO \dots}
		\STATE \(\gamma_{t} \gets 2 / (t+2)\)
		\STATE\(v_{t} \gets \argmin_{v \in \vertex{\mathcal{X}}}
		\innp{\nabla f(x_{t})}{v} \label{alg:NEP:step}
		+ L \gamma_{t} \norm[2]{x_t - v}^{2} \mathbin{/} 2\)
		\STATE\(x_{t+1} \gets x_{t} + \gamma_{t} (v_{t} - x_{t})\)
		\ENDFOR
	\end{algorithmic}
\end{algorithm}

The improvement in the convergence rate of Algorithm~\ref{alg:NEP-FW} compared to the vanilla
Frank–Wolfe algorithm is
that it depends mainly on the neighborhood of the optimal solutions
and only to a lesser extent on global parameters.
It is unknown whether the same improvement already holds for the
vanilla Frank–Wolfe algorithm.

Concretely, the diameter \(D\) of the feasible region is replaced by
the diameter \(D_{0}\) of, roughly
speaking, the minimal face containing the optimal solutions.
Actually, the dependence on the diameter \(D\) is lessened
but not eliminated.
The remaining dependence on \(D\) is justified by
a variant of Example~\ref{example:lowerbound} providing a lower bound
\(\Omega(\sqrt{L D^{2} / \varepsilon})\)
on the number of linear minimizations
\citep[see][Theorem~2]{garber2021NEP-FW}
for a maximum primal gap \(\varepsilon > 0\) of a smooth objective
function with \(D_{0} \ll D\).
(This lower bound has also the caveat that the counterexample depends
on \(\varepsilon\), and even
Algorithm~\ref{alg:NEP-FW} converges linearly with a better
step size rule at least for strongly convex functions
\citep[see][Theorem~3]{garber2021NEP-FW}.)

Our presentation differs only in minor ways from
\citet{garber2021NEP-FW}: to simplify the algorithm,
we do not enforce monotonicity of the function value,
and we generalize the convergence rate
to arbitrary sharp objective functions; see \citet{garber2021NEP-FW} for the case of (only) smooth functions, which follows in an analogous fashion using the modified NEP progress lemma (Lemma ~\ref{lem:NEP-FW-step}).
We present only the convergence rate for an agnostic step size rule
based on \citet[Theorem~1]{garber2021NEP-FW}.

\begin{theorem}
	\label{thm:NEP-FW}
  Let \(f\) be a \((c, \theta)\)-sharp, \(L\)-smooth function
  in the Euclidean norm over a convex set \(\mathcal{X}\)
  of diameter at most \(D\).
  Let \(S^{*}\) be a set of extreme points of \(\mathcal{X}\)
  whose convex hull contains all the minimizers of \(f\) over
  \(\mathcal{X}\).
  Let \(D_{0}\) denote the diameter of \(S^{*}\).
  Then the Nearest Extreme Point
	Frank–Wolfe algorithm (Algorithm~\ref{alg:NEP-FW})
  with step size rule \(\gamma_{t} = 2 / (t+2)\)
	has convergence rate for all \(t \geq 1\):
	\begin{equation}
		\label{eq:NEP-FW}
		h_{t}
		\leq
    \frac{L D^{2}}{t (t+1)} + \frac{2 L D_{0}^{2}}{t + 2}
		+ \frac{2^{2 \theta + 1} L^{2 \theta + 1} c^{2} D^{4 \theta}}{t (t+1)}
		\cdot
		\begin{cases}
			\frac{(t + 2)^{1 - 2 \theta}}{1 - 2 \theta}
			& 0 < \theta < 1/2 \\
      \ln (t + 2) & \theta = 1/2
		\end{cases}
		.
	\end{equation}
	Equivalently, for a primal gap error at most \(\varepsilon\)
	the algorithm calls the NEP oracle at most the following number of
	times:
	\begin{equation}
		\label{eq:NEP-FW-cost}
		\begin{cases}
			\mathcal{O}
			\left(
			\max \left\{
      \frac{L D_{0}^{2}}{\varepsilon},
			\sqrt{\frac{L D^{2}}{\varepsilon}},
			\sqrt[2 \theta + 1]{\frac{(2L)^{2 \theta + 1}
					c^{2} D^{4 \theta}}%
				{(1 - 2 \theta) \varepsilon}}
			\right\}
			\right)
      & 0 < \theta < 1/2 ,\\[3ex]
			\mathcal{O}
			\left(
			\max \left\{
      \frac{L D_{0}^{2}}{\varepsilon},
			\sqrt{\frac{L D^{2}}{\varepsilon}},
      \frac{2L c D}{\sqrt{\varepsilon}}
        \sqrt[4]{\ln \left(
          \frac{(2L c D)^{2}}{\varepsilon}
          \right)}
			\right\}
			\right)
      & \theta = 1/2
      .
		\end{cases}
	\end{equation}
\end{theorem}

The convergence rate is based on an improved progress bound
compared to \eqref{eq:FW-step-progress},
relying on the NEP oracle. The first inequality is not contained in
the original reference \citet[Lemma~1]{garber2021NEP-FW} but follows
directly and allows us to remove the technical assumption of
monotonicity.

\begin{lemma}[NEP step progress]	\label{lem:NEP-FW-step}
  Under the conditions of Theorem \ref{thm:NEP-FW}, we have
	\begin{align}
		\label{eq:NEP-FW-old-step}
		h_{t+1} &\leq (1 - \gamma_{t}) h_{t}
		+ \gamma_{t}^{2} \cdot
		\frac{L D^{2}}{2},
		\\
		\label{eq:NEP-FW-step}
		h_{t+1} &\leq (1 - \gamma_{t}) h_{t}
		+ \gamma_{t}^{2} \cdot
    \frac{L (\distance{x_{t}}{\Omega^{*}}^{2} + D_{0}^{2})}{2}
		.
	\end{align}
\end{lemma}

Recall that \(\Omega^{*}\) is the set of minimizers of \(f\) over
\(\mathcal{X}\).

\begin{proof}
	We first prove Equation~\eqref{eq:NEP-FW-old-step} as a warm-up.
	\begin{equation}
		\begin{split}
			h_{t+1} - h_{t}
			&
			\leq
			\gamma_{t} \innp{\nabla f(x_{t})}{v_{t} - x_{t}}
			+ \gamma_{t}^{2} \frac{L \norm[2]{v_{t} - x_{t}}^{2}}{2}
			\\
			&
			\leq
			\min_{v \in \vertex{\mathcal{X}}}
			\gamma_{t} \innp{\nabla f(x_{t})}{v - x_{t}}
			+ \gamma_{t}^{2} \frac{L D^{2}}{2}
			\\
			&
			\leq
			\gamma_{t} \innp{\nabla f(x_{t})}{x^{*} - x_{t}}
			+ \gamma_{t}^{2} \frac{L D^{2}}{2}
			\\
			&
			\leq
			- \gamma_{t} h_{t}
			+ \gamma_{t}^{2} \frac{L D^{2}}{2}
			.
		\end{split}
	\end{equation}
	The first inequality is the usual smoothness inequality.
	The second inequality
	uses the minimality of \(v_{t}\)
	and \(\norm[2]{v - x_{t}} \leq D\)
	for all extreme points \(v\); i.e., here we use the guarantee provided by the NEP oracle.
	Thus now the minimum of a \emph{linear} function is taken,
	which the third inequality upper bounds
	by a function value taken at the point \(x^{*}\)
	in the convex hull.
	The fourth inequality follows from convexity.
	
	The proof of Equation~\eqref{eq:NEP-FW-step} follows a similar
	pattern.
  Recall that \(S^{*}\) is a set of extreme points,
  whose convex hull has diameter \(D_{0}\)
  and contains \(\Omega^{*}\).
	For all \(x^{*} \in \Omega^{*}\) we have
	\begin{equation}
		\begin{split}
			h_{t+1} - h_{t}
			&
			\leq
			\gamma_{t} \innp{\nabla f(x_{t})}{v_{t} - x_{t}}
			+ \gamma_{t}^{2} \frac{L \norm[2]{v_{t} - x_{t}}^{2}}{2}
			\\
			&
			\leq
			\min_{v \in S^{*}}
			\gamma_{t} \innp{\nabla f(x_{t})}{v - x_{t}}
			+ \gamma_{t}^{2} \frac{L \norm[2]{v - x_{t}}^{2}}{2}
			\\
			&
			\leq
			\min_{v \in S^{*}}
			\gamma_{t} \innp{\nabla f(x_{t})}{v - x_{t}}
			+ \gamma_{t}^{2}
			\frac{L \bigl(\norm[2]{v - x_{t}}^{2} - \norm[2]{v - x^{*}}^{2}\bigr)}{2}
      + \gamma_{t}^{2} \frac{L D_{0}^{2}}{2}
			\\
			&
			\leq
			\gamma_{t} \innp{\nabla f(x_{t})}{x^{*} - x_{t}}
			+ \gamma_{t}^{2} \frac{L \norm[2]{x^{*} - x_{t}}^{2}}{2}
      + \gamma_{t}^{2} \frac{L D_{0}^{2}}{2}
			\\
			&
			\leq
			- \gamma_{t} h_{t}
      + \gamma_{t}^{2} \frac{L \bigl(\norm[2]{x^{*} - x_{t}}^{2} + D_{0}^{2}\bigr)}{2}
			.
		\end{split}
	\end{equation}
	The first inequality is the usual smoothness inequality.
	The second inequality
	uses the minimality of \(v_{t}\)
	to restrict the implicit minimum to the set \(S^{*}\)
	of vertices; again here we use the guarantee provided by the NEP oracle.
  The third inequality uses \(\norm[2]{v - x^{*}} \leq D_{0}\)
  for all \(v \in S^{*}\).  Note that
  \begin{equation*}
    \norm[2]{v - x_{t}}^{2} - \norm[2]{v - x^{*}}^{2}
    = \norm[2]{x_{t}}^{2} - \norm[2]{x^{*}}^{2} - 2 \innp{v}{x_{t} - x^{*}}
  \end{equation*}
	is linear in \(v\), a property specific to the Euclidean norm.
	The fourth inequality upper bounds the minimum
	by a function value taken at the point \(x^{*}\)
	in the convex hull,
	relying on that the minimum of a \emph{linear} function is taken.
	The fifth inequality follows from convexity. Taking minimum over \(x^{*} \in \Omega^{*}\) proves the claim.
\end{proof}

We are ready to prove Theorem~\ref{thm:NEP-FW}, which is straightforward once the improved progress lemma (Lemma~\ref{lem:NEP-FW-step}) is established and follows the usual template.

\begin{proof}[Proof of Theorem~\ref{thm:NEP-FW}]
	Recall that the progress equation \eqref{eq:NEP-FW-old-step}
	implies that  \(h_{1} \leq L D^{2} / 2\) and  \(h_{t} \leq 2 L D^{2} / (t+2)\) for \(t \geq 1\)
	 (see Theorem~\ref{fw_sub} and Remark~\ref{rem:initial-bound}), which we will use as an
	initial crude upper bound.
	By sharpness of \(f\) (see Section~\ref{sec:adaptive_rates}) we thus obtain
	\begin{equation}
		\label{eq:1}
    \distance{x_{t}}{\Omega^{*}}
		\leq
		c h_{t}^{\theta}
		\leq
		c
		\left(
		\frac{2 L D^{2}}{t + 2}
		\right)^{\theta}
		.
	\end{equation}
	For a tighter convergence rate,
	we combine this with the progress equation \eqref{eq:NEP-FW-step},
	using the exact value \(\gamma_{t} = 2 / (t+2)\)
	of the step size:
	\begin{equation}
		\begin{split}
			(t+1) (t+2) h_{t+1}
			&
			\leq
			t (t+1) h_{t}
      + 2 L (\distance{x_{t}}{\Omega^{*}}^{2} + D_{0}^{2})
			\cdot
			\frac{t + 1}{t + 2}
			\\
			&
			\leq
			t (t+1) h_{t}
			+ 2 L \left(
			c^{2} \left(
			\frac{2 L D^{2}}{t + 2}
			\right)^{2 \theta}
      + D_{0}^{2}
			\right)
			.
		\end{split}
	\end{equation}
	Summing up for $\tau = 1, \dots, t-1$, we obtain a telescopic sum and with
	standard estimations:
	\begin{equation}
		\begin{split}
			t (t+1) h_{t}
			&
			\leq
			2 h_{1}
			+ 2 L \left(
			c^{2} \sum_{\tau =1}^{t-1} \left(
			\frac{2 L D^{2}}{\tau + 2}
			\right)^{2 \theta}
      + (t - 1) D_{0}^{2}
			\right)
			\\
			&
			\leq
      L D^{2} + 2 (t - 1) L D_{0}^{2}
			+ 2^{2 \theta + 1} L^{2 \theta + 1} c^{2} D^{4 \theta}
			\sum_{\tau =1}^{t-1} \frac{1}{(\tau + 2)^{2 \theta}}
			\\
			&
			\leq
      L D^{2} + \frac{2 t (t + 1) L D_{0}^{2}}{t + 2}
      +
      \begin{multlined}[t]
        2^{2 \theta + 1} L^{2 \theta + 1} c^{2} D^{4 \theta}
        \\
			\cdot
			\begin{cases}
				\frac{(t + 2)^{1 - 2 \theta} - 3^{1 - 2 \theta}}{1 - 2 \theta}
				&\text{for } 0 < \theta < 1/2, \\
     \ln (t + 2) - \ln 3 &   \text{for }  \theta = 1/2.
			\end{cases}
    \end{multlined}
   \end{split}
	\end{equation}
	Now the claimed convergence rate follows by dropping negative terms.
\end{proof}

\begin{example}[Comparison with the vanilla Frank–Wolfe algorithm]
  \label{example_NEP}
  In Figure~\ref{fig:NEP}
\begin{figure}
  \centering
  \includegraphics[width=.45\linewidth, alt={Graph of primal gap in
    iteration over the hypercube, with Nearest Extreme Point
    Frank–Wolfe algorithm consistently faster than vanilla Frank–Wolfe
    algorithm.}]{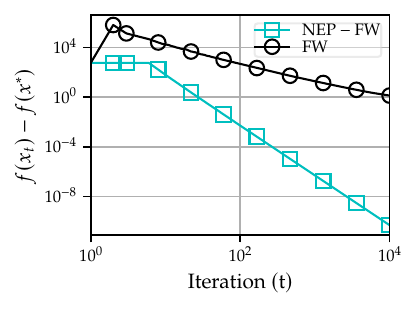}
  \qquad
  \includegraphics[width=.45\linewidth, alt={Graph of primal gap in
    iteration over the hypercube, with Nearest Extreme Point
    Frank–Wolfe algorithm oscillating widely but otherwise converging
    similarly to vanilla Frank–Wolfe
    algorithm.}]{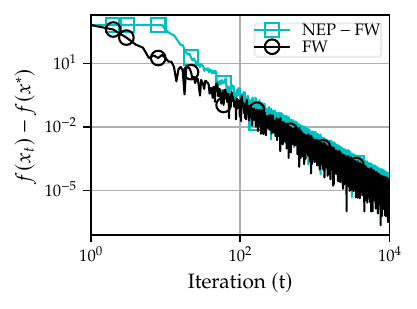}
  \caption{\label{fig:NEP}
    Performance of Nearest Extreme Point Frank–Wolfe algorithm
    (NEP-FW)
    (Algorithm~\ref{alg:NEP-FW})
    over the hypercube (left) and on the probability simplex (right)
    in a \(1200\)-dimensional Euclidean space.
    for a quadratic objective function
    where the optimal solution lies on a low-dimensional face.
    For the hypercube low-dimensional faces
    have significantly smaller diameter
    than the feasible region, while for the probability simplex
    all faces of dimension at least \(1\) have the same diameter.
    This is reflected in observed convergence:
    NEP-FW is faster than the vanilla Frank–Wolfe algorithm for the
    hypercube but convergence is similar for the probability simplex.}
\end{figure}%
  we compare the performance of the Nearest Extreme Point
  Frank–Wolfe algorithm (Algorithm~\ref{alg:NEP-FW})
  and the vanilla Frank–Wolfe algorithm (Algorithm~\ref{fw}),
  both with the $\gamma_t = 2/(t+2)$ step size rule
  for a \myindex{regression problem} over two different feasible regions,
  namely the \(0/1\)-\myindex{hypercube} and
  the \myindex{probability simplex}.
  Recall that for both of these feasible regions we can compute the
  NEP oracle with a simple call to the LMO, as
  \begin{equation}
   \begin{split}
	\argmin_{v \in \vertex{\mathcal{X}}} \innp{c}{v}
	+ \lambda \norm[2]{v - x}^{2} 
  &
  = \argmin_{v \in \vertex{\mathcal{X}}} \innp{c}{v}
	+ \lambda\left(\norm{v}^2 - 2 \innp{v}{x} \right)
  \\
  &
  = \argmin_{v \in \vertex{\mathcal{X}}} \innp{c}{v}
	+ \lambda   \innp{v}{\allOne - 2x}.
   \end{split}
  \end{equation}
  Where the first equality simply 
  follows from writing out the norm term and ignoring the resulting 
  terms that only involve $x$,
  and the last inequality follows from the fact that both 
  of these polytopes are \(0/1\)-polytopes, and so for any 
  $v \in \vertex{\mathcal{X}}$ we have $\norm{v}^2 = \innp{v}{\allOne}$.
   The objective function 
  being minimized is
\begin{equation*}
	\min_{x \in \mathcal{X}} \norm[2]{y  - Ax}^{2},
\end{equation*}   
where $A\in \R^{m \times n}$ with $n = 1200$ and $m = 600$, 
and $y\in \R^m$. The entries of $A$ are selected at random 
from the standard normal distribution, which results in $L \approx 7000$. 
For both feasible regions we select the value of $y$ so that $x^*$ 
lies in a low-dimensional face of the feasible region.
For the hypercube, we
replicate the setup described in \citet{garber2021NEP-FW}, 
that is, we choose $x^*$ by selecting a vertex at random,
and then setting the first five of its
entries to $0.5$; we choose this point as $x^*$ and 
set $y = A x^*$. Note that in this case the face on 
which $x^*$ lies has dimension $5$ and $D^* = \sqrt{5}$, whereas 
$D = \sqrt{1200}$. In this case we also 
select the initial point $x_0$ for both algorithms 
in the same way as described in \citet{garber2021NEP-FW}, 
that is, we start off by setting $x_0 = x^*$, and 
then we set the first five entries of $x_0$ to 
zero, and the sixth entry we set to $0$ if it 
is $1$, or to $1$ if it is zero. The reason why we 
choose \(x_{0}\) so close to \(x^{*}\) is so that 
$\distance{x_{0}}{\Omega^{*}}^{2}$ is small
(it has a value of $2.25$), as this will
mean that \(D^{2} \gg \distance{x_{0}}{\Omega^{*}}^{2} + D_{0}^{2}\),
and so by comparing Equation~\eqref{eq:NEP-FW-step} and 
Equation~\eqref{eq:NEP-FW-old-step} we expect to get greater 
primal progress from a NEP-FW step than from a FW step right from
the first iteration. Note that this particular starting point was 
chosen to highlight and illustrate the advantage of the NEP-FW algorithm over the FW 
algorithm, and does not portray the typical numerical behaviour 
of the algorithm in most cases.

For the problem over the probability simplex, we 
choose $x^*$ by selecting the first five entries 
of the vector uniformly at random between $0$ 
and $1$, and then dividing these numbers by the 
total sum of the entries. We select the starting 
point by setting a single random element of a 
vector to $1$, and the remaining entries to $0$. 
As in the problem over the hypercube, we have 
that $x^*$ lies in a face of dimension $5$, but 
in this case we have that $D^* =  D = \sqrt{2}$, 
and so we should not expect to see an improvement 
in convergence rate.
Indeed, Figure~\ref{fig:NEP} confirms the expectation:
convergence of the Nearest Extreme Point Frank–Wolfe algorithm
compared to the vanilla Frank–Wolfe algorithm
is faster for the hypercube but not significantly different for the
probability simplex.

As a final remark, in the initial iterations for the hypercube,
we observe that Frank–Wolfe algorithm with Nearest Extreme Point
(Algorithm~\ref{alg:NEP-FW}) does not make any progress at all.
In fact the iterates are all the same because the step size is too
large,
as the term $L\gamma_t \norm{x_t - v_t}/2$
dominates the NEP oracle in Line~\ref{alg:NEP:step}.
This indicates that an adaptive strategy at least for the initial step
size probably performs better, but this is beyond the scope of this survey.
\end{example}

\section{Further conditional gradient variants}
\label{sec:further-FW}

In this section, we present further variants of conditional gradients,
which are more involved or less prominent than the preceding ones.

\subsection{Boosted Frank–Wolfe algorithm}
\label{sec:boostfw}
We will now present another method for improving the convergence of
conditional gradient algorithms by better aligning their descent
directions with the negative of the gradient. We have seen in the
previous section that the Away-step Frank–Wolfe algorithm
overcomes the zigzagging issue of the vanilla Frank–Wolfe algorithm
by allowing to move away from vertices, and the same idea also
led to the Pairwise-Step Frank–Wolfe algorithm.  To this end,
they maintain a convex decomposition of the iterates $x_t$
into vertices, which can become very memory intense. Furthermore,
they are limited to directions given by vertices of the feasible
region.

In this section,
the norm is the Euclidean norm with the standard salar
product.
A rule-of-thumb for first-order methods is to follow the direction of
steepest descent. In this vein, the Boosted Frank–Wolfe algorithm
(BoostFW) \citep{combettes20boostfw} revisits the zigzagging issue and
\emph{chases} the steepest descent direction by descending in a
direction better aligned with the negative gradient, while remaining
projection-free. This is achieved via a matching pursuit-style
procedure, performing a sequence of alignments to build a descent
direction $g_t$ with the following properties:

\begin{enumerate}
 \item Better aligned with $-\nabla f(x_t)$ than the FW descent direction $v_t-x_t$, i.e.,
   \begin{equation*}
     \frac{\innp{-\nabla f(x_{t})}{g_{t}}}{\norm{\nabla f(x_{t})}
       \norm{g_{t}}}
     \geq
     \frac{\innp{-\nabla f(x_{t})}{v_{t} - x_{t}}}%
     {\norm{\nabla f(x_{t})} \norm{v_{t} - x_{t}}};
   \end{equation*}
 \item\label{gt:proj} Allowing a projection-free update of $x_t$, i.e.,
 \begin{equation*}
  x_t+\gamma g_t\in\mathcal{X}\quad\text{for all }\gamma\in\left[0,1\right],
 \end{equation*}
 or, equivalently by convexity, $\left[x_t,x_t+g_t\right]\subseteq\mathcal{X}$.
\end{enumerate}
An illustration is available in Figure~\ref{fig:boostfw}.

\begin{figure}[b]
  \footnotesize
  \centering
    \subcaptionbox{\label{fig:Boost-r0}
      We compute the Frank–Wolfe vertex
      $v_{0} \in \argmax_{v\in\mathcal{X}} \innp{r_{0}}{v}$
      for the direction $r_0$, and then decompose $r_0$ into
      the parallel component
      $\lambda_0u_0$ where $u_0 =  v_0-x_t$ and
      $\lambda_{0} =
   \frac{\innp{r_{0}}{u_0}}{\norm{u_0}^{2}}$, and the 
   residual $r_1 =  r_0-\lambda_0u_0$.}[.4\linewidth]{%
 \begin{tikzpicture}[scale=1.2]
   \draw
   (-1.2, 1) coordinate (a) --
   (-0.2, 2) coordinate (b) --
   (2, 0.8)  node[point, label=right:$v_0$](c){} --
   (1.3, -0.8) coordinate (d) --
   (-0.8, -0.5) coordinate (e) --
   (-1.2, 1) coordinate (f) --
   cycle;

   \node[point, label=below:$x_t$, alias=l0s, alias=r1bs, alias=l1s]
   (x) at (0.6, 0.6) {};
   \draw[vector] (x.center) -- +(0.9, 0.9)
   coordinate[label=above:{$-\nabla f(x_t)=r_0$}, alias=r1e] (n);
  \coordinate[alias=l0e] (r1s) at (1.608, 0.744);
  \coordinate[alias=r2e] (r1be) at (0.492, 1.356);
  \coordinate[alias=l1e] (r2s) at (0.247, 1.216);

  \draw[alignment line][red] (l0e) -- (c);

  \draw[vector, red] (r1s) -- (r1e) node [midway, right] {$r_1$};
  \draw[vector, red] (l0s) -- (l0e)
  node [below, near end] {$\lambda_0u_0$};
 \end{tikzpicture}
}
\hfil
\subcaptionbox{We repeat \subref{fig:Boost-r0} to decompose $r_1$
  as \(r_{1} = \lambda_{1} u_{1} + r_{2}\),
  and then \(r_{2}\) and so on.}[.4\linewidth]{%
  \begin{tikzpicture}[scale=1.2]
   \draw
   (-1.2, 1) coordinate (a) --
   (-0.2, 2) node[point, label=above:$v_1$] (b) {} --
   (2, 0.8)  node[point, label=right:$v_0$](c){} --
   (1.3, -0.8) coordinate (d) --
   (-0.8, -0.5) coordinate (e) --
   (-1.2, 1) coordinate (f) --
   cycle;

   \node[point, label=below:$x_t$, alias=l0s, alias=r1bs, alias=l1s]
   (x) at (0.6, 0.6) {};
  \coordinate[alias=r1e] (n) at (1.5, 1.5);
  \coordinate[alias=l0e] (r1s) at (1.608, 0.744);
  \coordinate[alias=r2e] (r1be) at (0.492, 1.356);
  \coordinate[alias=l1e] (r2s) at (0.247, 1.216);

  \draw[alignment line][cyan] (l1e) -- (b);
  \draw[alignment line][red] (l0e) -- (c);

  \draw[vector, red] (l0s) -- (l0e)
  node [below, near end] {$\lambda_0u_0$};
  \draw[vector, red] (r1bs) -- (r1be) node [midway, right] {$r_1$};
  \draw[vector, cyan] (r2s) -- (r2e) node [midway, above] {$r_2$};
  \draw[vector, cyan] (l1s) -- (l1e) node [midway, left] {$\lambda_1u_1$};
 \end{tikzpicture}
 }

 \subcaptionbox{We set $d_{K_{t}} = \lambda_0u_0 + \lambda_1u_1 +
   \dots + \lambda_{K_{t} - 1} u_{K_{t} - 1}$,
   with \(K_{t}\) denoting the number of decompositions
   (in the figure \(K_{t} = 2\)).}%
 [.4\linewidth]{%
 \begin{tikzpicture}[scale=1.2]
   \draw
   (-1.2, 1) coordinate (a) --
   (-0.2, 2) node[point, label=above:$v_1$] (b) {} --
   (2, 0.8)  node[point, label=right:$v_0$](c){} --
   (1.3, -0.8) coordinate (d) --
   (-0.8, -0.5) coordinate (e) --
   (-1.2, 1) coordinate (f) --
   cycle;

   \node[point, label=below:$x_t$, alias=l0s, alias=r1bs, alias=l1s]
   (x) at (0.6, 0.6) {};
  \coordinate[alias=r1e] (n) at (1.5, 1.5);
  \coordinate[alias=l0e] (r1s) at (1.608, 0.744);
  \coordinate[alias=r2e] (r1be) at (0.492, 1.356);
  \coordinate[alias=l1e] (r2s) at (0.247, 1.216);

  \draw[vector, green] (x) -- ($(l0e) + (l1e) - (x)$)
  coordinate[label=above:{$d_{K_{t}}$}] (d2) {};

  \draw[alignment line][cyan] (l1e) -- (b);
  \draw[alignment line][red] (l0e) -- (c);
  \draw[alignment line][red] (l0e) -- (d2);
  \draw[alignment line][cyan] (l1e) -- (d2);

  \draw[vector, red] (l0s) -- (l0e)
  node [below, near end] {$\lambda_0u_0$};
  \draw[vector, cyan] (l1s) -- (l1e) node [midway, left] {$\lambda_1u_1$};
 \end{tikzpicture}
}
\hfil
\subcaptionbox{We scale $d_{K_{t}}$ to obtain
  $g_t = d_{K_{t}} / (\lambda_0+ \dots + \lambda_{K_{t}})$.
  Note that $x_t+d_{K_{t}}$ need not be in \(\mathcal{X}\)
  but $[x_t,x_t+g_t]\subseteq\mathcal{X}$.}[.4\linewidth]{%
  \begin{tikzpicture}[scale=1.2]
    \draw
    (-1.2, 1) coordinate (a) --
    (-0.2, 2) coordinate (b) --
    (2, 0.8)  node[point, label=right:$v_0$](c){} --
    (1.3, -0.8) coordinate (d) --
    (-0.8, -0.5) coordinate (e) --
    (-1.2, 1) coordinate (f) --
    cycle;

    \node[point, label=below:$x_t$, alias=l0s, alias=r1bs, alias=l1s]
    (x) at (0.6, 0.6) {};
    \draw (x) -- +(0.9, 0.9)
    coordinate[label=right:{$r_{0} = -\nabla f(x_t)$}, alias=r1e] (n);

  \coordinate[alias=l0e] (r1s) at (1.608, 0.744);
  \coordinate[alias=r2e] (r1be) at (0.492, 1.356);
  \coordinate[alias=l1e] (r2s) at (0.247, 1.216);

  \draw[alignment line, red] (x) -- (c);

  \draw[vector] (x) -- (n);
  \draw[vector, green] (x) -- ($(l0e) + (l1e) - (x)$)
  node [above] {$d_{K_{t}}$};
  \draw[vector, blue] (x) -- +(0.482, 0.572) node [left] {$g_t$};
\end{tikzpicture}
}
 \caption{The boosting procedure, which approximates a direction
   \(r_{0}\) at a point \(x_{t}\) in a compact convex set \(\mathcal{X}\)
   with a direction \(g_{t}\) with \(x_{t} + g_{t} \in \mathcal{X}\).
   It improves upon the Frank–Wolfe direction $v_0-x_t$
   by replacing the extreme point \(v_{0}\)
   with a convex combination of several extreme points.
   In iteratively optimizing an objective function \(f\)
   over \(\mathcal{X}\),
   one chooses \(r_{0} = - \nabla f(x_{t})\)
   at an iterate \(x_{t}\),
   and the next iterate \(x_{t+1}\) is selected from the segment
   $[x_t,x_t+g_t]$ ensuring $x_{t+1}\in\mathcal{X}$.
   \vspace{-1ex}
 }
 \label{fig:boostfw}
\end{figure}
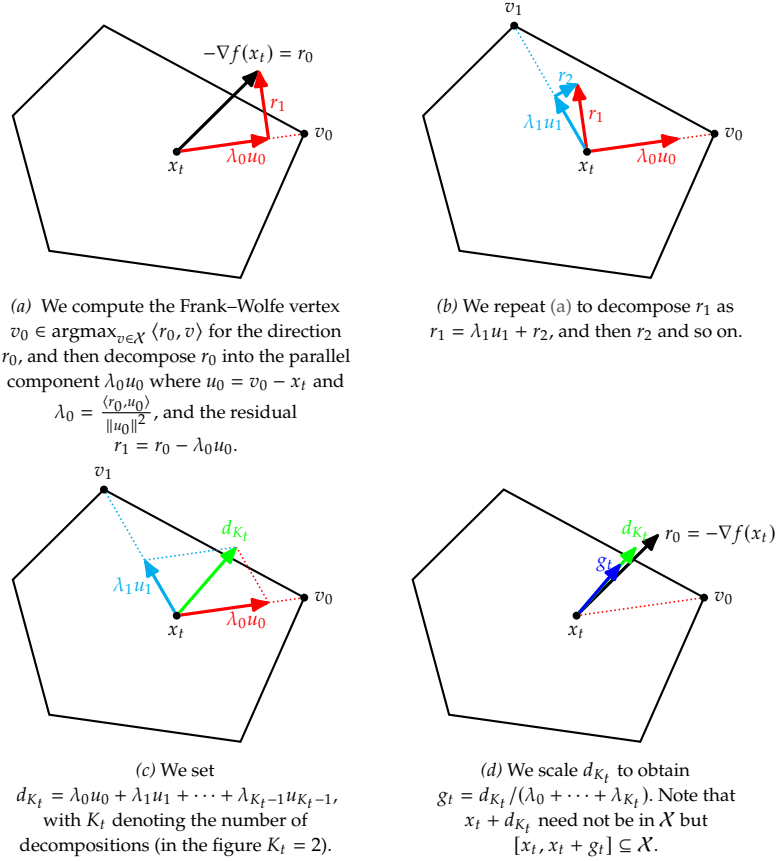

\subsubsection{Boosting via gradient pursuit}

BoostFW is presented in Algorithm~\ref{boostfw}, where the boosting
procedure illustrated in Figure~\ref{fig:boostfw} takes place in
Lines~\ref{boostfw:d0}--\ref{boostfw:end}.  It replaces the linear
minimization oracle in FW.  The cosine
measures the alignment between a target direction $d\neq0$
and its estimate $\hat{d}$:
\begin{equation*}
  \cos (d,\hat{d}) \defeq
 \begin{cases}
   \frac{\innp{d}{\hat{d}}}{\norm{d} \norm{\hat{d}}}
   &\text{if } \hat{d} \neq 0,\\
  -1 &\text{if } \hat{d} = 0.
 \end{cases}
\end{equation*}

\begin{algorithm}[t]
\caption{Boosted Frank–Wolfe (BoostFW) \citep{combettes20boostfw}}
\label{boostfw}
\begin{algorithmic}[1]
  \REQUIRE Start point $x_{0}\in\mathcal{X}$,
    maximum number of rounds $K > 0$,
    alignment improvement tolerance $0 < \delta < 1$, step sizes
    $0 \leq \gamma_{t} \leq 1$
  \ENSURE Iterates \(x_{1}\), \(x_{2}\), \dots
\FOR{$t=0$ \TO \dots}
\STATE$d_0\leftarrow0$\label{boostfw:d0}
\STATE$\Lambda_t\leftarrow0$
\STATE$K_{t} \gets K$
\FOR{$k=0$ \TO $K-1$}
\STATE\label{boostfw:r}
  $r_{k} \gets - \nabla f(x_{t}) - d_{k}$ \COMMENT{$k$-th residual}
\STATE\label{boostfw:v}
  $v_{k} \gets \argmax_{v \in \mathcal{X}} \innp{r_{k}}{v}$
  \COMMENT{FW oracle}
\STATE\label{boostfw:extra}
  $u_{k} \gets \argmax_{u \in \{v_{k} - x_{t}, -d_{k} / \norm{d_{k}}\}}
  \innp{r_{k}}{u}$
  \algorithmicif\ \(d_{k} \neq 0\) \algorithmicelse\ $v_k-x_t$
\STATE\label{boostfw:lbd}
  $\lambda_{k} \gets \frac{\innp{r_{k}}{u_{k}}}{\norm{u_{k}}^{2}}$
\STATE\label{boostfw:dprime}
  $d_{k+1} \gets d_{k} + \lambda_{k}u_{k}$
\IF{$\cos (-\nabla f(x_{t}),d_{k+1})
      - \cos (-\nabla f(x_{t}),d_{k}) \geq \delta$}
    \label{boostfw:criterion}
  \STATE\label{boostfw:Lbd}
    $\Lambda_{t} \gets
    \begin{cases}
      \Lambda_{t} + \lambda_{k} & \text{if } u_{k}=v_{k}-x_{t} \\
      \Lambda_{t} (1-\lambda_{k} / \norm{d_{k}}) & \text{if }
      u_{k} = - d_{k} / \norm{d_{k}}
    \end{cases}$
\ELSE
\STATE \(K_{t} \gets k\)
\BREAK \label{boostfw:break} \COMMENT{exit $k$-loop}
\ENDIF
\ENDFOR
\label{boostfw:kt}
\STATE\label{boostfw:end}
  $g_{t} \gets d_{K_{t}} / \Lambda_{t}$
  \COMMENT{normalization, \(K_{t}\) is total number of rounds}
\STATE$x_{t+1}\leftarrow x_t+\gamma_tg_t$
\ENDFOR
\end{algorithmic}
\end{algorithm}

The choice $\delta = 10^{-3}$ as a default value works well in most setups.
The role of the hyperparameter $K$ is only to cap the number of
pursuit rounds per iteration when the FW oracle is particularly
expensive.
In the toy example of Figures~\ref{fig:zigzag}-\ref{fig:pairwise}, BoostFW converges in exactly $1$ iteration and $2$ oracle calls. Proposition~\ref{boostfw:prop} shows that the
boosting procedure is well defined.

\begin{proposition}
 \label{boostfw:prop}
  Let \(\mathcal{X}\) be a compact convex set.
 Suppose $x_{t} \in \mathcal{X}$ for a nonnegative integer~\(t\).
 Let \(K_{t}\) denote
 the number of alignment rounds (i.e., iterations of the inner loop
 starting at Line~\ref{boostfw:r}) used for computing
 \(g_{t}\)
 (i.e., the value of \(K_{t}\) at the end of the algorithm).
 Then
 \begin{enumerate}
  \item\label{proof:prop1} $d_{1}$ is defined and $K_t\geq1$,
  \item\label{proof:prop2} $\lambda_0, \dotsc, \lambda_{K_t-1} \geq 0$,
  \item\label{proof:prop3}
    $d_k\in\operatorname{cone}(\mathcal{X}-x_t)$
    for all $0 \leq k \leq K_t$,
  \item\label{boostfw:prop:gx} $x_t+g_t\in\mathcal{X}$ and $x_{t+1}\in\mathcal{X}$,
  \item\label{boostfw:prop:eta}
    $\cos(-\nabla f(x_t),g_t) \geq \cos(-\nabla f(x_t),v_t - x_t)
    + (K_t - 1) \delta$ and $\cos (-\nabla f(x_t), v_t - x_t)
    \geq 0$,
     where $v_t \in \argmin_{v\in\mathcal{X}}
    \innp{\nabla f(x_t)}{v}$ is a Frank–Wolfe vertex.
 \end{enumerate}
\end{proposition}

We are ready to provide the main convergence theorem; here we assume
that the function is gradient dominated
(see Remark~\ref{rem:gradientDom} for the definition)
and we reuse $c$ as a parameter
as a strongly convex function is also gradient dominated.

\begin{theorem}
  \label{th:boostfw1}
  Let \(\mathcal{X}\) be a compact convex set with diameter \(D\),
  $f \colon \mathcal{X} \to \mathbb{R}$ be an $L$-smooth,
 convex, $c$-gradient dominated function. Consider BoostFW
 (Algorithm~\ref{boostfw}) with
 $x_{0}\in\argmin_{v\in\mathcal{X}}\innp{\nabla f(y)}{v}$, where
 $y\in\mathcal{X}$, and
 the short step rule
 $\gamma_{t} = \min\{\innp{-\nabla
   f(x_{t})}{g_{t}} \mathbin{/}
 (L\norm{g_{t}}^{2}),1\}$ or
 $\gamma_{t} = \argmin_{0 \leq \gamma \leq 1} f(x_{t}+\gamma g_{t})$.
 Then for all $t \geq 0$,
 \begin{equation*}
   f(x_{t}) - f(x^{*})
   \leq
   \frac{LD^{2}}{2} \prod_{s=0}^{t-1}
   \left(
     1-\eta_{s}^{2} \frac{c}{L}
   \right)^{\allOne_{\{\gamma_{s} < 1\}}}
   \left(
     1 - \frac{\norm{g_{s}}}{2\norm{v_{s}-x_{s}}}
   \right)^{\allOne_{\{\gamma_{s} = 1\}}}
 \end{equation*}
 where $v_{s} \in \argmin_{v\in\vertex{\mathcal{X}}}
 \innp{\nabla f(x_{s})}{v}$ for all $s \geq 0$.
\end{theorem}

Theorem~\ref{th:boostfw1} provides a quantitative estimation of the
convergence of BoostFW. In order to obtain a fully explicit rate, we
can observe that in practice, $\gamma_{t}<1$
at almost every iteration, a phenomenon similar to that in AFW or PFW,
and that $K_{t} > 1$ at almost every iteration,
which simply means that
at least two rounds of alignment are performed
in almost every iterations.
Thus,
$\size{\{0 \leq s \leq t-1\mid\gamma_{s}<1,K_{s}>1\}}\approx t$
for every $t$ is consistently observed.
In Theorem~\ref{th:boostfw2}, a weaker assumption
$\size{\{0 \leq s \leq t-1\mid\gamma_{s}<1,K_{s}>1\}}\geq\omega t$
is used, where $\omega>0$,
and a linear convergence rate is established.
For completeness, note that if $K_{t} = 1$ for every $t$
then BoostFW reduces to FW, and the convergence rate would be
$\mathcal{O}(1/t)$.

\begin{theorem}
  \label{th:boostfw2}
  Let \(\mathcal{X}\) be a compact convex set with diameter \(D\)
  and pyramidal width \(\delta\),
  $f \colon \mathcal{X} \to \mathbb{R}$ be an $L$-smooth,
  convex, $c$-gradient dominated function.
  Consider BoostFW (Algorithm~\ref{boostfw}) with
  $x_{0}\in\argmin_{v\in\mathcal{X}} \innp{\nabla f(y)}{v}$,
  where $y\in\mathcal{X}$, and
  $\gamma_{t} = \min\{\innp{-\nabla f(x_{t})}{g_{t}}
  \mathbin{/}
  (L\norm{g_{t}}^{2}),1\}$ or
  $\gamma_{t} = \argmin_{0 \leq \gamma \leq 1} f(x_{t}+\gamma g_{t})$.
  Assume that $\size{\{0 \leq s \leq t-1\mid\gamma_{s}<1,K_{s}>1\}}
  \geq \omega t$ for all $t \geq 0$, where $\omega>0$.
  Then for all $t \geq 0$,
  (here \(\delta\) is alignment improvement tolerance parameter
  of the algorithm and not pyramidal width)
  \begin{equation*}
    f(x_{t}) - f(x^{*})
  \leq\frac{LD^2}{2}\exp\left(-\delta^2\frac{c}{L}\omega t\right).
  \end{equation*}
  Equivalently, the primal gap is at most $\varepsilon$ after at most
  the following number of linear minimizations:
 \begin{equation*}
 \begin{cases}
   \mathcal{O}\left(
     \frac{L D^{2} \min\{K, 1/\delta\}}{\varepsilon}
   \right)
   &\text{in the worst-case scenario,}
   \\[1ex]
   \mathcal{O}\left(
     \min\left\{K, \frac{1}{\delta}\right\}
     \frac{1}{\omega \delta^{2}} \frac{L}{c}
     \log \left( \frac{1}{\varepsilon} \right)
   \right)
   &\text{in the practical scenario with \(K \geq 2\)}.
  \end{cases}
 \end{equation*}
\end{theorem}
Note that for \(K=1\) BoostFW reduces to FW,
justifying $K\geq2$ in the practical scenario.

Although BoostFW may perform multiple linear minimizations per
iteration, \citet{combettes20boostfw} show in a wide range of
computational experiments that the progress obtained overcomes this
cost.
The boosting approach can also be used to speedup any  conditional gradient algorithms, resulting in
boosted variants of AFW, PFW, DI-PFW, BCG, etc.; see
\citet{combettes20boostfw} for details.

\subsection{Blended Conditional Gradient algorithm}
\label{sec:bcg}

The Fully-Corrective Frank–Wolfe algorithm (Algorithm~\ref{fcfw})
adds steps fully optimizing the objective function
over the convex hull of the active set,
however without considering computational costs.
The Blended Conditional Gradient algorithm
\citep{pok18bcg} (BCG, Algorithm~\ref{bcg})
has been designed with computational costs in mind,
to optimize over the convex hull of the active set only
to the extent it is cheaper than continuing with Frank–Wolfe steps.
This follows the common pattern of finding a balance between
the cost and value of accuracy for optimizing subproblems.

The main new idea is to optimize over the convex hull of the active
set in small cheap steps, and use the number of the cheap steps
as computational cost.
The prime example of such a step is a single
gradient descent step without projection over a simplex
whose vertices are the active set.
As usual, we hide the implementation of the steps behind an oracle,
the \emph{Simplex Descent oracle} (Oracle~\ref{sido}),
and formulate only a progress requirement: the decrease in function
value is quadratic in the strong Frank–Wolfe gap over the active set.
A gradient descent based implementation of Simplex Descent oracle
is provided in Algorithm~\ref{sigd} at the end of this section but
many other implementations are conceivable and possible, e.g.,
accelerated projected gradient descent steps.

\begin{oracle}
\caption{Simplex Descent oracle $\text{SiDO}(x,\mathcal{S})$}
\label{sido}
\begin{algorithmic}
  \REQUIRE Finite set $\mathcal{S}$,
    point $x \in \conv{\mathcal{S}}$
  \ENSURE Finite set $\mathcal{S}'\subseteq\mathcal{S}$ and point
    $x' \in \conv{\mathcal{S}'}$
    satisfying either $f(x')\leq f(x)$ and
    $\mathcal{S}'\neq\mathcal{S}$ (drop step), or
    $f(x) - f(x') \geq \max_{u,v \in \mathcal{S}} \innp{\nabla
      f(x)}{u-v}^2 \mathbin{/} (L \size{\mathcal{S}})$ (descent step)
\end{algorithmic}
\end{oracle}

Recall that a quadratic progress in the strong Frank–Wolfe gap is the key to
linear convergence for strongly convex functions;
combined with Lemma~\ref{lem:geoSC} this was the key to convergence
for the Away-step and Pairwise Frank–Wolfe algorithms
(Algorithms~\ref{away} and \ref{pairwise}).
Admittedly, the strong Frank–Wolfe gap over the active set can be much
smaller than over the whole polytope, in which case the progress
guarantee for simplex descent is weak.  However, this is also a sign
that there is not much to gain continuing optimizing over the active set, i.e., it is more
promising to extend the active set with Frank–Wolfe steps, which is precisely when we will switch back to Frank–Wolfe steps.

This leads to an algorithm very similar to
the Away-step Frank–Wolfe algorithm,
which via a simple heuristic of comparing (strong) Frank–Wolfe gaps,
chooses between a Frank–Wolfe step and another kind of step
in every iteration; this time a simplex descent step instead of an
away step.
To avoid linear optimization in simplex descent steps iterations,
the Blended Conditional Gradient algorithm  (Algorithm~\ref{bcg})
is actually a modification of
the Lazy Away-step Frank–Wolfe algorithm
(Algorithm~\ref{alg:ParamFreeLCG}),
where a cheap estimate of the strong Frank–Wolfe gap is always available
without linear optimization. As such also all remarks regarding lazification from Section~\ref{sec:lazification} also apply here.

\begin{algorithm}
\caption{Blended Conditional Gradient (BCG) \citep{pok18bcg}}
\label{bcg}
\begin{algorithmic}[1]
  \REQUIRE Start atom $x_0\in\mathcal{X}$, accuracy $\kappa\geq1$
  \ENSURE Iterates $x_1, \dotsc \in \mathcal{X}$
\STATE\label{bcg_phi0}
  \(\phi_{0} \leftarrow \max_{v\in\mathcal{X}}
  \innp{\nabla f(x_{0})}{x_{0} - v} \mathbin{/} 2\)
\STATE$\mathcal{S}_0\leftarrow\{x_0\}$
\FOR{$t=0$ \TO \dots}
\STATE$v_t^{\text{FW-}\mathcal{S}}\leftarrow\argmin_{v\in\mathcal{S}_t}\innp{\nabla f(x_t)}{v}$
\STATE$v_t^\text{A}\leftarrow\argmax_{v\in\mathcal{S}_t}\innp{\nabla f(x_t)}{v}$
\IF{$\innp{\nabla f(x_t)}{v_t^\text{A}-v_t^{\text{FW-}\mathcal{S}}}\geq\phi_t$}\label{bcg_criterion}
\STATE$x_{t+1},\mathcal{S}_{t+1}\leftarrow\text{SiDO}(x_t,\mathcal{S}_t)$\COMMENT{SiDO gradient step}\label{bcg_sido}
\STATE$\phi_{t+1}\leftarrow\phi_t$
\ELSE
  \STATE$v_t\leftarrow\text{LPsep}_{\mathcal{X}}(\nabla f(x_t),x_t,
    \phi_t,\kappa)$\label{bcg_weak}
    \COMMENT{\(\text{LPsep}\) is Oracle~\ref{ora:LPsep}}
\IF{$v_t=\FALSE$}\label{bcg_false}
\STATE$x_{t+1}\leftarrow x_t$
\STATE$\mathcal{S}_{t+1}\leftarrow\mathcal{S}_t$
\STATE$\phi_{t+1}\leftarrow\phi_t/2$\COMMENT{Frank–Wolfe gap step}\label{bcg_dual}
\ELSE
  \STATE$x_{t+1}\leftarrow\argmin_{\left[x_t,v_t\right]}f$
    \COMMENT{Frank–Wolfe step}\label{bcg_fw}
\STATE$\mathcal{S}_{t+1}\leftarrow\mathcal{S}_t\cup\{v_t\}$
\STATE$\phi_{t+1}\leftarrow\phi_t$
\ENDIF
\ENDIF
\STATE\label{line:BCG-minimize-S}
  Optional:
  \(\mathcal{S}_{t+1} \leftarrow
  \argmin \{ \size{\mathcal{S}} :
  \mathcal{S} \subseteq \mathcal{S}_{t+1},\
  x_{t+1} \in \conv{\mathcal{S}} \}\)
\ENDFOR
\end{algorithmic}
\end{algorithm}

Finally we note that the size of active set should be constrained
because it affects both the cost and guarantee of simplex descent.
Recall that by Carathéodory's theorem any point of \(P\)
is a convex combination of at most \(\dim P + 1\) vertices,
and therefore we require an \(\mathcal{O}(\dim P)\) bound
for the convergence results; note however that in actual computations enforcing such a bound explicitly is usually not required as the iterations tend to be much sparser than the bound.

\begin{theorem}
 \label{th:bcg}
 Let $f$ be an $L$-smooth convex function over a polytope \(P\) with
 diameter~\(D\).  Provided that all the
 active sets \(\mathcal{S}_{t}\) contain \(\mathcal{O}(\dim P)\)
 vertices
 (e.g., the optional step in Line~\ref{line:BCG-minimize-S} ensures
 this),
 the Blended Conditional Gradient algorithm (Algorithm~\ref{bcg})
 satisfies $f(x_t)-f(x^*) \leq \varepsilon$
 after at most the following number of gradient computations
 and total number of weak separations and simplex descents together:
 \begin{equation}
   \mathcal{O}\left( \frac{L D^{2} \dim P}{\varepsilon} \right)
   .
 \end{equation}
 Furthermore, if $f$ is \(\mu\)-strongly convex
  then $f(x_t)-f(x^*) \leq \varepsilon$
 after at most the following number of gradient computations
 and total number of weak separations and simplex descents together,
 denoting the pyramidal width of \(P\) by \(\delta\):
 \begin{equation}
   \mathcal{O}\left(
     \frac{L D^{2}}{\mu \delta^{2}}
     \dim P
     \log \frac{1}{\varepsilon}
   \right)
   .
 \end{equation}
\end{theorem}

\begin{example}[Congestion Balancing in Traffic Networks]

We study a \myindex{congestion balancing}
problem over a traffic network, 
where we have a series 
of source nodes and sink nodes, and we 
are interested in finding a feasible flow that 
minimizes the cost over the 
flows on the edges \citep[see][Chapter~14]{ahuja1988network}. 
The objective function used is a convex quadratic, as this 
type of objective function has the role of balancing congestion 
over the network edges
\citep[see, e.g.,][]{diakonikolas2018fair}.

The weight in the objective function associated 
with each edge is chosen uniformly at random with 
$L/\mu = 100$. The network used is the \myindex{flow polytope}
\texttt{road\_paths\_01\_DC\_a}, from the suite of benchmark problems 
considered in \citet{kovacs2015minimum} (as in 
\citet{lan2017conditional} or \citet{pok18bcg}), which 
leads to a problem of dimension \num{29682}.
Unfortunately,
as this flow polytope is not a $0/1$ polytope, we 
cannot use the DI-PFW algorithm, or its boosted variant. 
FW algorithms are 
particularly suited for this problem since solving a 
linear optimization problem 
over the flow polytope is equivalent to finding 
the shortest path in a network, 
and there are efficient algorithms to do this. 
In the left column of Figure~\ref{fig:BCG_boosted_comparison} we present a 
comparison of the FW, PFW, AFW and BCG algorithms when solving 
this problem. Additionally, we 
also test the performance of the boosted version of FW with several 
values $\delta$. 

For the BoostFW algorithms we tested different 
values of the $\delta$ parameter, namely $10^{-15}$, $10^{-11}$, 
$10^{-7}$, $10^{-6}$, $10^{-5}$, $10^{-4}$and $10^{-3}$.
For all these boosted algorithms we use $K = \infty$. Note that 
the BoostFW balances the number of LMO and FOO oracle calls required by the 
algorithm through the $\delta$ and the $K$ parameter. A 
value of $K = 1$ makes the BoostFW algorithm 
reduce to the FW algorithm. On the other hand, a value 
of $K = \infty$ with various values of $\delta$ leads to a 
 range of LMO to FOO ratios. During the tuning of the $\delta$ 
 parameter one can observe that the smaller the value of $\delta$, the higher the value 
of this ratio. This makes intuitive sense, 
as small values of $\delta$ means that we want the 
direction of movement of the BoostFW algorithm to be well 
aligned with the gradient, and a finer alignment is achieved through 
repeated calls to the LMO for a single FOO. 
We report 
the data from the BoostFW algorithm with 
$\delta = 10^{-5}$ as it achieved the best
performance for primal gap in wall-clock time.
All the algorithms are run until
the total wall-clock time reaches \num{36000} seconds.
\end{example}

\begin{example}[Logistic regression] \label{example_log_reg_BCG_boosted}
  In the right column of Figure~\ref{fig:BCG_boosted_comparison}
  \begin{figure}[b]
  \footnotesize
  \begin{tabular}{cc}
    \emph{Traffic congestion} & \emph{Logistic regression} \\[\smallskipamount]
    \includegraphics[width=.45\linewidth]{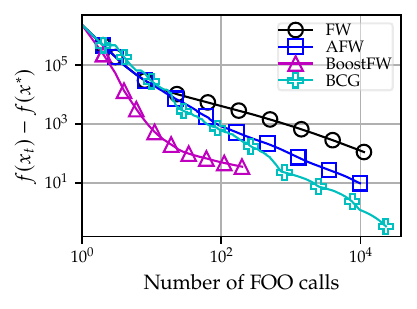}
  &
  \includegraphics[width=.45\linewidth]{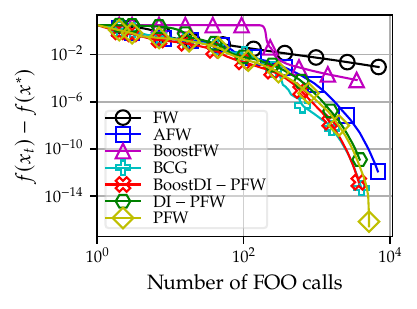}
  \\
  \includegraphics[width=.45\linewidth]{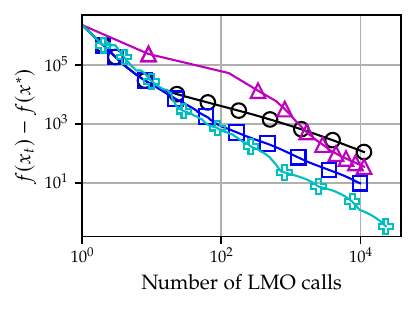}
  &
  \includegraphics[width=.45\linewidth]{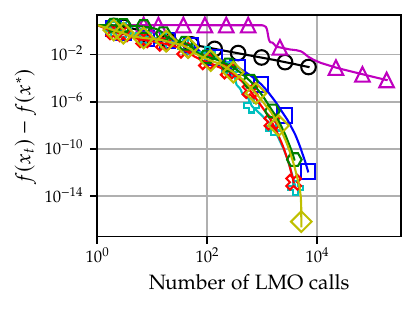}
  \\
  \includegraphics[width=.45\linewidth]{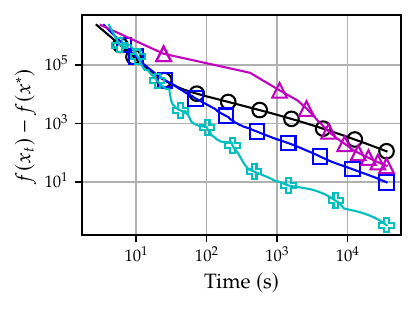}
  &
  \includegraphics[width=.45\linewidth]{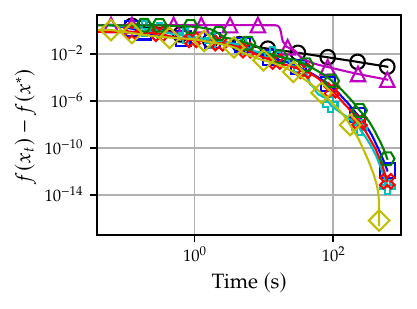}
  \end{tabular}

  \caption{\label{fig:BCG_boosted_comparison}
    Numerical comparison: Primal gap of a traffic
    congestion balancing problem (minimizing a convex quadratic
    function over a flow polytope),
    and a logistic regression problem over
    the \texttt{gisette} dataset.}
\end{figure}%
  we compare the BoostFW and BCG algorithms, along with several of the FW variants, 
  when applied to a sparse logistic \myindex{regression problem}, where
  as in the last example, we use the $\ell_1$-norm to enforce sparsity
  in the solutions. The problem can be phrased as:
  \begin{equation}
  \min_{\norm{x}_1 \leq \tau} \sum_{i = 1}^m \log \left(1 + e^{-y_i \innp{x}{z_i}} \right) 
  \end{equation}
  for some $\tau > 0$. The data samples $\{y_i, z_i\}_{i=1}^m$ with $y_i\in\{-1,1\}$ and $z_i\in \R^n$ are taken from the \texttt{Gisette} dataset \citep{guyon2004result}. This dataset represents handwritten digits encoded in every $z_i$ (along with a large number of distractor features, to make the learning process harder), and whether the actual handwritten digit is either a $4$, or a $9$, encoded in $y_i$.

We have expressed the problem over the \index{l1-ball@\(\ell_{1}\)-ball}\(\ell_{1}\)-ball
in dimension $n$ as
an equivalent 
problem over a scaled \myindex{probability simplex} of dimension $2n$
(see for example \citet{jaggi2013equivalence}
or Example~\ref{Example:sparse_recovery}). Solving the problem over the probability simplex has 
the additional advantage that we can easily compute away-steps or pairwise-steps without having to explicitly store 
an active set. Moreover, it allows us to use the DI-PFW algorithm (Algorithm~\ref{DIPFW}) from Section~\ref{sec:decomposition-invariant}, as well as its boosted variant, as these algorithms can only be applied to $0/1$ polytopes of special structure.

The total number of samples is $m=2000$, and the dimension of the
problem is $n = 5000$, although note that the actual number of
\emph{useful} features is much lower than~$5000$. For all the examples
we use the adaptive step size strategy of \citet{pedregosa2018step}, as
there is no closed-form solution for line search. For the BoostFW 
and BoostDI-PFW algorithm 
we use $\delta = 10^{-9}$ and $K = \infty$. and All the algorithms 
are run until the total wall-clock time reaches $600$ seconds. The 
objective function under consideration is smooth and convex, and 
the feasible region is a polytope. Despite the objective 
function not being strongly convex, the AFW and the BCG 
algorithms converge linearly in primal gap, as one can observe 
when looking at the primal gap convergence vs. number of FOO 
calls. This linear convergence ultimately translates into 
an advantage in wall-clock
time of these two variants over the vanilla FW
algorithm and the BoostFW algorithm.

In this case, the PFW algorithm outperforms all the other algorithms 
for primal gap in wall-clock time, closely followed by the BCG
and the BoostDI-PFW algorithms. 
\end{example}

\subsubsection{An implementation of the simplex descent oracle}
\label{sido_app}

Here we provide a realization
of the simplex descent oracle (Oracle~\ref{sido}) in
Algorithm~\ref{sigd}.  Let $\mathcal{S}=\{v_1, \dotsc, v_k\}$
denote the input vertex set.
Recall that $\allOne \defeq (1 \dots 1) \in \mathbb{R}^{k}$
is the vector with all entries \(1\) and \(e_{1}, \dotsc, e_{k} \in
\mathbb{R}^{k}\) are the coordinate vectors.

\begin{algorithm}[t]
\caption[]{Simplex descent $\text{SiDO}(x,\mathcal{S})$}
\label{sigd}
\begin{algorithmic}[1]
\REQUIRE Finite set $\mathcal{S} = \{v_1, \dotsc, v_k\}$, point
  $x \in \conv{\mathcal{S}}$ with convex decomposition
  $x=\sum_{i=1}^{k} \lambda_{i} v_{i}$
\ENSURE Finite set $\mathcal{S}'\subseteq\mathcal{S}$ and point
  $x' \in \conv{\mathcal{S}'}$ satisfying either $f(x')\leq f(x)$ and
  $\mathcal{S}'\neq\mathcal{S}$ (drop step), or
  $f(x)-f(x')\geq\max_{1 \leq i, j \leq k}\innp{\nabla
    f(x)}{v_i-v_j}^2 \mathbin{/} 4L$ (descent step)
\STATE\label{sido_g}
  $g \gets (\innp{\nabla f(x)}{v_{1}}, \dotsc, \innp{\nabla f(x)}{v_{k}})$
\STATE\label{sido_d}
  $d \gets g - \frac{\innp{g}{\allOne}}{k} \allOne$
\IF{$d=0$}
\STATE$x'\leftarrow v_1$
\STATE$\mathcal{S}'\leftarrow\{v_1\}$
\ELSE
\STATE$\gamma\leftarrow\max\{\gamma\geq0\mid\lambda-\gamma d\geq0\}$
\STATE$y\leftarrow x-\gamma\sum_{i=1}^kd_iv_i$\label{sido_y}
\IF{$f(y)\leq f(x)$}
\STATE$x'\leftarrow y$\COMMENT{drop step}\label{sido_drop}
\STATE Choose $\mathcal{S}' \subsetneq \mathcal{S}$ such that
  $x' \in \conv{\mathcal{S}'}$
\ELSE
\STATE$x'\leftarrow\argmin_{\left[x,y\right]}f$\COMMENT{descent step}
\STATE$\mathcal{S}'\leftarrow\mathcal{S}$
\ENDIF
\ENDIF
\end{algorithmic}
\end{algorithm}

Algorithm~\ref{sigd}
essentially operates on the probability simplex $\Delta_{k}$
in the \(\ell_{2}\)-norm
by a coordinate transformation treating the coefficients of linear
decomposition as coordinates in the space of \(\Delta_{k}\).
Concretely, the transformation is provided by
the linear map \(V \colon \mathbb{R}^{k} \to \mathbb{R}^{n}\)
with $V e_{i} \defeq v_{i}$, i.e.,
the objective function on \(\Delta_{k}\) is
the convex and \(L \norm{V}\)-smooth \(f_{\mathcal{S}}\)
defined via
$f_{\mathcal{S}}(\tau) \defeq f(V \tau) = f (
  \sum_{i=1}^{k} \tau_{i} v_{i} )$.

Now the algorithm essentially makes a gradient descent step
with line search,
truncating the step if it would go outside the simplex
to go only to the boundary (Line~\ref{sido_y}).
In this context \(d\) is the gradient at \(\lambda\) in the affine space of
the feasible region \(\Delta_{k}\), computed as a projection of
the gradient \(g = \nabla f_{\mathcal{S}}(\lambda) =
V^{\top} \nabla f(V \lambda)\)
in the vectors space \(\mathbb{R}^{k}\) containing \(\Delta_{k}\).

The following lemma shows that Algorithm~\ref{sigd} is a correct
realization of the simplex descent oracle,
with the choice $h=f_{\mathcal{S}}$ since
\begin{multline}
  \innp{\nabla f_{\mathcal{S}}(\lambda)}{e_i-e_j}
  =
  \innp{V^{\top} \nabla f(x)}{e_i-e_j}
  =
  \innp{\nabla f(x)}{Ve_i-Ve_j}
  =
  \innp{\nabla f(x)}{v_i-v_j}
  ,
  \\
  1 \leq i, j \leq k.
\end{multline}

\begin{lemma}
  Let $h \colon \Delta_{k} \to \mathbb{R}$ be $L$-smooth.
  Let $\lambda \in \Delta_{k}$, $d \defeq \nabla h(\lambda) -
  (\innp{\nabla h(\lambda)}{\allOne} / k) \allOne$,
  $\gamma \defeq \max\{\gamma \geq 0 \mid \lambda - \gamma d \geq
  0\}$, and $\lambda' \defeq \argmin_{[\lambda, \lambda - \gamma
    d]}h$.
  Then either $h(\lambda - \gamma d) \leq h(\lambda)$ or
 \begin{equation}
   h(\lambda)-h(\lambda')\geq\max_{1 \leq i, j \leq k}
   \frac{\innp{\nabla h(\lambda)}{e_i-e_j}^2}{4L}.
 \end{equation}
\end{lemma}

\subsection{Nonsmooth objectives and composite convex optimization}
\label{nonsmooth}

In this section, we consider the composite convex optimization problem
\begin{equation}
 \min_{x\in\mathcal{X}}h(x)+g(Ax)\label{npb1},
\end{equation}
where $\mathcal{X}\subset\mathbb{R}^n$
is a compact convex set, $h\colon\mathbb{R}^n\rightarrow\mathbb{R}$
is an $L_{h}$-smooth convex function,
$g\colon\mathbb{R}^m\rightarrow\mathbb{R}$
is a convex $G_g$-Lipschitz continuous function,
and $A \colon \mathbb{R}^{n} \to \mathbb{R}^{m}$ is a linear map.
Examples of such problems include regularization problems with
composite penalties,
simultaneously sparse and low-rank matrix recovery,
and sparse PCA \citep[Sections~4.1, 5, and 6]{argyriou14}.
We restrict to the Euclidean norm \(\norm[2]{\cdot}\) in this section
and to the simplest algorithm.
Other algorithms use techniques out of scope of this monograph,
e.g., Lagrangian methods \citep{CGaugmentedLagrangian2020,CGLangrangian2019}.

The idea of \citet{argyriou14}
\citep[later rediscovered in][]{yurtsever18}
is to solve Problem~\eqref{npb1}
by \myindex{smoothing} $g$ via its \emph{\myindex{Moreau envelope}}
\citep{moreau65} defined as
\begin{align*}
  g_{\beta}(x)
  &\defeq
  \min_{u \in \mathbb{R}^m} \left(
    g(u) + \frac{1}{2 \beta} \norm[2]{u-x}^{2}
  \right),
  \intertext{where $\beta > 0$ is a smoothing parameter.
    We will also need the \myindex{proximity operator} \citep{moreau62}
    providing the minimizer}
  \prox_{g}(x)
  &\defeq
  \argmin_{u \in \mathbb{R}^{m}} \left(
    g(u) + \frac{1}{2} \norm[2]{u-x}^{2}
  \right)
  .
\end{align*}

The Moreau envelope $g_{\beta}$
is an $1/\beta$-smooth convex function
approximating \(g\) \citep{bc17} with error
$g_{\beta} \leq g \leq g_{\beta} + \beta G_{g}^{2} / 2$ \citep{argyriou14} and its gradient has the form
\begin{equation*}
	\nabla g_{\beta}(x) = \frac{1}{\beta}\left(x-\prox\nolimits_{\beta g}(x)\right).
\end{equation*}
The core idea is to solve the smoothed problem
$\min_{x\in\mathcal{X}} h(x) + g_{\beta}(Ax)$
instead of the original one.
For a rough estimate of achievable convergence rate,
let us consider first the simplest case: use a fixed \(\beta\)
and the vanilla Frank–Wolfe algorithm (Algorithm~\ref{fw}).
We immediately have a primal gap error
\(\mathcal{O}((L_{h} + 1 / \beta) D^{2} / t
+ \beta \norm{A} G_{g}^{2})\)
after \(t\) linear minimization oracle calls
(cf.~Theorem~\ref{fw_sub}).
In other words,
the algorithm achieves a primal gap error of at most
\(\varepsilon > 0\)
after
\(\mathcal{O}(L_{h} D^{2} / \varepsilon
+ \norm{A} G_{g}^{2} D^{2} / \varepsilon^{2})\)
linear minimizations with \(\beta = \Theta(1 / \sqrt{\varepsilon})\).

By varying the smoothing parameter \(\beta\),
one avoids the need of a prespecified accuracy \(\varepsilon\),
leading to
the Hybrid Conditional Gradient-Smoothing algorithm (HCGS),
presented in Algorithm~\ref{hcgs:fw},
which is essentially the vanilla Frank–Wolfe algorithm
using the smooth approximation as outlined above.
The computational cost of the prox-operator depends on the specific
complexity of the function $g$ and can be cheap or quite expensive.

\begin{algorithm}[H]
\caption{Hybrid Conditional Gradient-Smoothing (HCGS) \citep{argyriou14}}
\label{hcgs:fw}
\begin{algorithmic}[1]
  \REQUIRE Start point $x_0\in\mathcal{X}$,
    smoothing parameters $\beta_t>0$,
    step sizes $0 \leq \gamma_{t} \leq 1$
  \ENSURE Iterates \(x_{1}, \dotsc \in \mathcal{X}\)
\FOR{$t=0$ \TO \dots}
  \STATE\label{hcgs:fw:grad}
    $z_t \gets \nabla h(x_t) + \frac{1}{\beta_t} A^{\top} \left(
      A x_t - \prox_{\beta_t g}(Ax_t)
    \right)$
    \COMMENT{\(z_{t} = \nabla (h + g_{\beta_{t}} \circ A) (x_{t})\)}
  \STATE\label{hcgs:lmo}
    \(v_{t} \gets \argmin_{v\in\mathcal{X}}
    \innp{z_{t}}{v}\)
\STATE$x_{t+1}\leftarrow x_t+\gamma_t(v_t-x_t)$
\ENDFOR
\end{algorithmic}
\end{algorithm}

Experiments on simultaneously sparse and low-rank matrix recovery and sparse PCA are presented in \citet[Sections~5 and~6]{argyriou14}, comparing HCGS to the Generalized Forward-Backward algorithm (GFB) \citep{raguet13} and the Incremental Proximal Descent algorithm (IPD) \citep{bertsekas12}.

\begin{theorem}
 Let $h \colon \mathbb{R}^n \to \mathbb{R}$ be an $L_h$-smooth convex
 function, $g \colon \mathbb{R}^m \to \mathbb{R}$ be a convex
 $G_g$-Lipschitz continuous function,
 and $\mathcal{X}\subset\mathbb{R}^n$ be a compact convex set with
 diameter $D>0$.
 Set $\beta>0$, $\beta_t=\beta/\sqrt{t+1}$, and $\gamma_t=2/(t+2)$.
 Then the iterates of HCGS (Algorithm~\ref{hcgs:fw}) satisfy
 for all \(t \geq 1\),
 \begin{equation*}
   h(x_{t}) + g(Ax_{t}) - \min_{x\in\mathcal{X}}\bigl(h(x) + g(Ax)\bigr)
   \leq
   \frac{2L_{h} D^{2}}{t+1}
   + \frac{2\norm{A}^{2} D^{2}}{\beta \sqrt{t+1}}
   + \frac{\beta G_{g}^{2} \sqrt{t+2}}{t}
   .
 \end{equation*}
 With $\beta = \sqrt{2} \norm{A} D \mathbin{/} G_{g}$,
 for all \(t \geq 1\)
 \begin{equation*}
   h(x_{t}) + g(Ax_{t}) - \min_{x \in \mathcal{X}} \bigl(h(x) + g(Ax)\bigr)
   \leq
   \frac{2 L_{h} D^{2}}{t+1} + \frac{4 \norm{A} G_{g} D}{\sqrt{t+1}}
   .
 \end{equation*}
 Equivalently,
 still with $\beta = \sqrt{2} \norm{A} D \mathbin{/} G_{g}$,
 the algorithm achieves a primal gap additive error
 at most \(\varepsilon > 0\) after at most
 \(\mathcal{O}(L_{h} D^{2} / \varepsilon +
 \norm{A}^{2} G_{g}^{2} D^{2} / \varepsilon^{2})\)
 linear minimizations.
\end{theorem}

In particular, for \(G\)-Lipschitz-continuous functions,
i.e., for $h=0$ and \(A\) the identity map,
HCGS performs
\(\mathcal{O}(G^{2} D^{2} / \varepsilon^{2})\)
linear minimizations,
which is tight (see Example~\ref{example:lowerbound} or
\citet{lan2013complexity}), as long as the problem is allowed to
depend on \(\varepsilon\).

The approach from above also naturally generalizes to the more general setting of Hilbert spaces and we refer the reader to \citet{argyriou14} for details. Another approach to smoothing the nonsmooth part $g\circ A$ is via the ball-conjugate \citep{pierucci14}. Also, recall that Frank–Wolfe algorithms approximate the objective function
locally with a linear function, and
the role of smoothness is to ensure that the approximation is good in
a large enough neighborhood.
Based on this insight, the recent work \citet{cheung18}
chooses the best approximation in an explicitly selected
small neighborhood.
It would be interesting to combine it with \citet{ravi19},
which introduces a smoothness concept for
non-differentiable functions, using gradients from a small neighborhood.

\subsubsection{Connection with subgradient methods}
\label{sec:equiv-mirror-descent}

The mirror descent algorithm \citep{nemirovsky1983problem}
is a generalization of many subgradient methods \citep{shor85subgrad}.
Here we recall an interpretation from \citet{bach15dualfw}
of the mirror descent algorithm
on the important class of composite objective functions
considered so far
as a generalized Frank–Wolfe algorithm
\citep{bredies08gencg,Mine1981},
whose convergence has been studied in
\citet{bach15dualfw,pena19fwmd},
which has been also considered as
a discretized version of a continuous algorithm in a
unification of first-order methods \citep{FirstOrderUnified2019}.
Here we restrict to demonstrating only the interpretation
of mirror descent as a conditional gradient algorithm.

We consider the minimization problem
\begin{equation}
  \label{primalpb}
  \min_{x \in \mathbb{R}^{n}} h(x) + g(Ax),
\end{equation}
where $h \colon \mathbb{R}^{n} \to \mathbb{R} \cup \{+\infty\}$
is a lower semi-continuous, strongly convex,
essentially smooth function, $g \colon \mathbb{R}^{m} \to \mathbb{R}$
is a convex Lipschitz-continuous function, and
$A \colon \mathbb{R}^{n} \to \mathbb{R}^{m}$ is a linear function.
Let \smash{\(\dom h \defeq h^{-1}(\mathbb{R})\)} denote the set of places
where \(h\) has a finite value.
Essential smoothness of $h$ means that
\(h\) is differentiable on \(\interior(\dom h) \neq \emptyset\), and that
$\nabla h(x_{t}) \to +\infty$ when $x_t\to x\in\partial\dom h$.
This ensures that $\nabla h \colon \interior(\dom h) \to \mathbb{R}^{n}$
is a bijection \citep[Corollary~26.3.1]{rocky70convex},
which is central for the mirror descent algorithm.
Lipschitz-continuity of $g$ ensures that $\mathcal{X}=\dom g^{*}$
is bounded \citep[Corollary~13.3.3]{rocky70convex},
which will be the feasible region for the Frank–Wolfe algorithm.
Here and below \(g^{*}\) is the convex conjugate of \(g\),
i.e., \(g^{*}(y) \defeq \max_{x \in \mathbb{R}^{n}} \innp{y}{x} - g(x)\).
The dual maximization problem is
\begin{equation}
  \label{dualpb}
  \max_{y \in \mathcal{X}} -h^{*}(-A^{\top} y) - g^{*}(y).
\end{equation}
Under these assumptions, \citet{bach15dualfw} showed that applying
the mirror descent algorithm to the primal problem is 
equivalent to applying the generalized Frank–Wolfe algorithm
\citep{bredies08gencg,Mine1981} to the dual problem.

Let us describe first the mirror descent algorithm.
The Bregman divergence \(D_{h}\) of \(h\) is
\(D_{h}(x_{1}, x_{2}) \defeq
h(x_{1}) - h(x_{2}) - \innp{\nabla h(x_{2})}{x_{1}-x_{2}}\)
for \(x_{1} \in\dom h\) and \(x_{2}\interior(\dom h)\)
\citep{bregman67}.
To minimize an objective function \(f\),
the mirror descent algorithm generates its iterates \(x_{t}\)
via the recursion
\begin{equation}
  \label{mda}
  x_{t+1} = \argmin_{x \in \mathbb{R}^n}
  \gamma_{t} \innp{\nabla f(x_{t})}{x} + D_{h}(x, x_{t}),
\end{equation}
which has the unique solution
\begin{equation}
  x_{t+1} = (\nabla h)^{-1}
  \bigl(
    \nabla h(x_{t}) - \gamma_{t} \nabla f(x_{t})
  \bigr)
  .\label{mdupdate}
\end{equation}
Note that for \(h \defeq \norm[2]{\cdot}^2 / 2\),
this is the subgradient descent algorithm
\(x_{t+1} = x_{t} - \gamma_{t} \nabla f(x_{t})\).

The mirror descent algorithm
for our objective \(f = h + g \circ A\)
is presented in Algorithm~\ref{mirror},
with the recursive formula deliberately broken up
to highlight similarity with the
upcoming generalization of Frank–Wolfe algorithm.
Even though \(g\) need not be differentiable,
for simplicity of notation, we write \(\nabla g(y)\) for some
chosen subgradient of \(g\) at \(y\).

\begin{algorithm}[t]
\caption[]{Mirror Descent \citep{nemirovsky1983problem}}
\label{mirror}
\begin{algorithmic}[1]
  \REQUIRE Start point $x_{0} \in \mathbb{R}^{n}$, step sizes
    \(0 \leq \gamma_{t} \leq 1\) for \(t \geq 0\)
  \ENSURE Iterates $x_1, \dotsc \in \mathbb{R}^n$
    for solving \(\argmin_{x \in \mathbb{R}^{n}} h(x) + g(Ax)\)
  \FOR{$t=0$ \TO \dots}
    \STATE \label{mirrorsubgd}
      \(v_{t} \gets \nabla g(A x_{t})\)
    \STATE \label{mdaupdate}
      \(x_{t+1} \gets (\nabla h)^{-1}
      \left(
        (1 - \gamma_{t}) \nabla h(x_{t})
        - \gamma_{t} A^{\top} v_{t}
      \right)\)
  \ENDFOR
\end{algorithmic}
\end{algorithm}

The interpretation as a Frank–Wolfe algorithm
is provided in Algorithm~\ref{gencgmirror},
which arises by rewriting the iterative formulae
in an alternative form,
using \(\nabla g(z) = \argmax_{y \in \mathcal{X}} \innp{y}{z} - g^{*}(y)\)
and \(\nabla h^{*} = (\nabla h)^{-1}\).
The iterates of the Frank–Wolfe algorithm
are new quantities \(y_{t}\)
chosen such that \(\nabla h(x_{t}) = - A^{\top} y_{t}\),
so that Line~\ref{mdaupdate} reduces
to the Frank–Wolfe update
\(y_{t+1} = (1 - \gamma_{t}) y_{t} + \gamma_{t} v_{t}\)
(after omitting the \(A^{\top}\)).
The objective of the Frank–Wolfe algorithm
\smash{\(\argmin_{y \in \mathcal{X}} h^{*}(- A^{\top} y) + g^{*}(y)\)},
uses the convex conjugate \(g^{*}\) and \(h^{*}\) of \(g\) and \(h\),
respectively, and
its feasible region \smash{\(\mathcal{X} \defeq \dom(g^{*})\)}
is the domain of \(g^{*}\), which is bounded.
The \(v_{t}\) becomes the Frank–Wolfe vertex, but note that
for computing \(v_{t}\)
the algorithm replaces only one summand of the objective function,
namely, \(h^{*}(- A^{\top} \cdot)\),
with a linear approximation \(\innp{\cdot}{- A x_{t}}\),
leaving the other summand unmodified;
this is where the Frank–Wolfe algorithm is generalized.
For explicit convergence rates for Algorithm~\ref{gencgmirror}
we refer to \citet{khue2021regularized}.

\begin{algorithm}[t]
  \caption{Generalized Conditional Gradients
    \citep{bredies08gencg,Mine1981}}
\label{gencgmirror}
\begin{algorithmic}[1]
  \REQUIRE Start point $y_0\in\mathcal{X}$, step sizes
    \(0 \leq \gamma_{t} \leq 1\) for \(t \geq 0\)
  \ENSURE Iterates $y_1, \dotsc \in \mathcal{X}$
    for solving
    \(\argmin_{y \in \mathcal{X}} h^{*}(- A^{\top} y)) + g^{*}(y)\)
    \COMMENT{\(\mathcal{X} = \dom g^{*}\)}
\FOR{$t=0$ \TO \dots}
  \STATE \label{gradh}
    \(x_{t} \gets \nabla h^{*}(-A^{\top} y_{t})\)
    \COMMENT{\(A^{\top} y_{t} = - \nabla h(x_{t})\)}
  \STATE \label{gengdcgmirror}
    \(v_{t} \gets \argmin_{v \in \mathcal{X}}
    \innp{v}{- A x_{t}} + g^{*}(v)\)
    \COMMENT{\(- A x_{t} = \nabla (h^{*} \circ -A^{\top}) (y_{t})\)}
\STATE$y_{t+1}\leftarrow(1-\gamma_t)y_t+\gamma_tv_t$\label{gencg}
\ENDFOR
\end{algorithmic}
\end{algorithm}

\begin{proposition}
  Let $h \colon \mathbb{R}^{n} \to \mathbb{R} \cup \{+\infty\}$
  be a lower semi-continuous, strongly convex,
  essentially smooth function,
  and $g \colon \mathbb{R}^{m} \to \mathbb{R}$
  be a convex Lipschitz-continuous function, and
  $A \colon \mathbb{R}^{m} \to \mathbb{R}^{n}$ a linear map.
  Then the Generalized Conditional Gradient algorithm
  (Algorithm~\ref{gencgmirror}) on~\eqref{dualpb} starting at $y_{0}
  \in \dom g^{*}$
  is equivalent to mirror descent (Algorithm~\ref{mirror}) on~\eqref{primalpb}
  starting at
  point $x_{0} \defeq \nabla h^{*} (- A^{\top} y_{0})$:
  the iterates \(x_{t}\) (and \(v_{t}\))
  generated by both algorithms are the same.
\end{proposition}

\subsection{In-Face directions}

The Frank–Wolfe algorithm (Algorithm~\ref{fw}) is invariant
under affine transformations of $\mathcal{X}$. While the AFW algorithm (Algorithm~\ref{away}) is also affine invariant,
 it is not trajectory-independent, as the steps taken by
 the algorithm at iteration $t$ depends on the active set
\(\mathcal{S}_t\), and the latter depends on the 
trajectory that the algorithm followed in 
earlier iterations.

\begin{figure}[b]

\bigskip

\centering
  \includegraphics[width=0.4\linewidth]{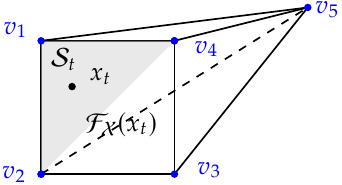}
  \caption{\emph{Trajectory dependence of the AFW algorithm:} We depict the vertices that define the polytope $\mathcal{X}$ in blue, that is, $\mathcal{X} = \conv{\{v_1, v_2, v_3, v_4, v_5 \}}$. Note that the smallest face containing $x_t$ is $\face_{\mathcal{X}} (x_t) = \conv{\{v_1, v_2, v_3, v_4 \}}$. The region shaded in grey indicates the convex hull of the active set $\mathcal{S}_t=\{v_1, v_2, v_4
\}$, i.e. $\conv{\mathcal{S}_t}$. Note that the iterate $x_t$ can be written as a convex combination of the elements in $\mathcal{S}_t$, or as a convex combination of the vertices in $\face_{\mathcal{X}} (x_t)$, as such, the away-step in the AFW algorithm will depend on how we write $x_t$ as a convex decomposition. This means that the step taken by the AFW algorithm a
at iteration $t$ depends on the vertices that have been picked up in previous iterations, 
making it trajectory dependent.
\vspace{-1ex}
}
    \label{ActiveSetFacialRepresentation}
\end{figure}

Let $\face_{\mathcal{X}} (x_t)$ denote the smallest face of
$\mathcal{X}$ containing $x_t$.  Clearly $\conv{\mathcal{S}_t}
\subseteq \face_{\mathcal{X}} (x_t)$, and in many cases there can potentially be multiple
active sets such that $x_t \in \mathcal{S}_t$. In the example shown in
Figure~\ref{ActiveSetFacialRepresentation} the point $x_t$ can be
represented as the convex combination of $\mathcal{S}_t=\{v_1, v_2, v_4
\}$ (as shown in the image in light gray), or as a convex
combination of the vertices of $\face_{\mathcal{X}} (x_t)$, that is, $ \vertex {\face_{\mathcal{X}} (x_t)} =\{v_1, v_2, v_3, v_4 \}$, as
such the away step taken by the AFW algorithm (Algorithm~\ref{away})
 will depend on $\mathcal{S}_t$.
The Extended Frank–Wolfe Method with
"In-face directions"
(Algorithm~\ref{extendedInFace}) from \citet{freund2017extended}
attempts to generalize the notion of away steps from \citet{gm86}, giving
greater control over the sparsity of the generated solutions
through two preference parameters, namely $\alpha_1$ and $\alpha_2$.
For example, small values of $\alpha_1$ and $\alpha_2$
in Algorithm~\ref{extendedInFace} promote sparsity over reduction in function value,
which is especially beneficial in matrix completion problems,
where sparsity translates to low rank of the solution as a matrix, which can be important, e.g., when solving matrix completion problems.

\begin{algorithm}
  \caption{In-Face Extended Frank–Wolfe
    \citep{freund2017extended}}
\label{extendedInFace}
\begin{algorithmic}[1]
  \REQUIRE Start atom $x_0\in \mathcal{X}$,
    initial lower bound $B_{-1} \leq f(x^*)$, and
    parameters $0 \leq \alpha_1 \leq \alpha_2 \leq 1$
  \ENSURE Iterates $x_1, \dotsc \in \mathcal{X}$
\FOR{$t=0$ \TO \dots}
\STATE$B_{t} \leftarrow B_{t-1}$
\STATE \label{line:extfw_inface}
  Choose $d_t$ with $x_t - d_t \in \affine
  \left(\face_{\mathcal{X}}(x_t) \right)$ and
  $\innp{\nabla f(x_t)}{d_t} \geq 0$.
\STATE $\gamma_{\max} \leftarrow \argmax_{\gamma}
  \{ \gamma \mid x_t - \gamma d_t \in \face_{\mathcal{X}}(x_t) \}$
  \label{maxstepInFace}
\STATE $x_t^{B} \leftarrow x_t - \gamma_{\max} d_t$
\STATE \label{linesearchInFace}
  $x_t^{A} \leftarrow x_t - \tilde{\gamma} d_t$
  where $\tilde{\gamma} \in [0, \gamma_{\max}]$
\IF{$\frac{2LD^2}{2LD^2 + \alpha_1\left(f(x_t) - B_t\right)} \left(f(x_t) - B_t\right) \geq  f(x_t^{B}) - B_t$} \label{alpha1InFace}
\STATE$x_{t+1} \leftarrow x_t^{B}$ \COMMENT{Go to lower dimensional face.} \label{performmaxstepInFace}
\ELSIF{$\frac{2LD^2}{2LD^2 + \alpha_2\left(f(x_t) - B_t\right)} \left(f(x_t) - B_t\right) \geq  f(x_t^{A}) - B_t$} \label{alpha2InFace}
\STATE$x_{t+1} \leftarrow x_t^{A}$ \COMMENT{Stay in current face.} \label{performlinesearchInFace}
\ELSE
\STATE$v_t^\text{FW}\leftarrow\argmin_{v\in \mathcal{X}}\innp{\nabla
    f(x_t)}{v}$\label{FWStepInFace}
  \COMMENT{Frank–Wolfe step}
\STATE$x_{t+1} \leftarrow x_t + \hat{\gamma}(v_t - x_t)$ with $\hat{\gamma} \in [0,1]$
\STATE$B_t \leftarrow \max \{
  B_{t-1}, f(x_{t}) + \innp{\nabla f(x_{t})}{v_{t} - x_{t}}
  \}$
\ENDIF
\ENDFOR
\end{algorithmic}
\end{algorithm}

The In-Face Extended Frank–Wolfe Method  algorithm  (Algorithm~\ref{extendedInFace}) assumes that it is possible to work with the minimal
face $\face_{\mathcal{X}} (x_t)$.
This is the case for example when minimizing over the
spectrahedron, but it is usually not possible in general, and is a
limiting assumption for this algorithm. Note that the quantity $B_t$ maintained during 
the run of the algorithm is a lower bound in $f(x^*)$, and used throughout the algorithm
to decide which type of steps to perform, either steps that remain in $\face_{\mathcal{X}}(x_t)$,
or steps that can potentially leave $\face_{\mathcal{X}}(x_t)$.

At each iteration the algorithm
chooses between three improving strategies
with a preference to solutions in lower dimensional faces.
The most preferred strategy
is moving to a subface of the minimal face
(Line~\ref{performmaxstepInFace}),
followed by staying within the minimal face
(Line~\ref{performlinesearchInFace});
these two together generalize away steps.
The third one is a fallback strategy:
a standard Frank–Wolfe step
(with no restrictions on the solution).
The preference is expressed by requiring less progress
from more preferred strategies.
The parameters $\alpha_1$ and $\alpha_2$
quantify the degree of preference of the strategies in 
the conditions in Line~\ref{alpha1InFace} and 
Line~\ref{alpha2InFace}. Note that $B_t\leq f(x^*)$, 
and so $f(x_t) - B_t \geq f(x_t) - f(x^*)\geq 0$, which means 
that the multiplicative factors $2LD^2/(2LD^2 + \alpha_1\left(f(x_t) - B_t\right))$ 
and $2LD^2/(2LD^2 + \alpha_2\left(f(x_t) - B_t\right))$ are both smaller than 
$1$ (while the former is smaller than the latter as $\alpha_2 \geq \alpha_1$). 
This means that if $\alpha_1$ is large, in order to move to $x_t^{B}$ in Line~\ref{performmaxstepInFace}
we will require that the reduction 
in function value for moving to $x_t^{B}$ be large (as the factor $2LD^2/(2LD^2 + \alpha_1\left(f(x_t) - B_t\right))$ will be small), compared to staying in $\face_{\mathcal{X}}(x_t)$. Similarly
if $\alpha_1$ is small, we will not require a large reduction in function 
value for moving to $x_t^{B}$, compared to staying in $\face_{\mathcal{X}}(x_t)$.
At a high-level if $\alpha_1$
is large we prioritize decreasing the function value over iterations, whereas if $\alpha_1$ 
is small we prioritize sparsity. Similar comments can be made regarding the condition in 
Line~\ref{alpha2InFace}.

The first two strategies also search for improving solutions in only one
direction \(d_{t}\) like a Frank–Wolfe step.
There are two recommendations for choosing $d_t$
in Line~\ref{line:extfw_inface}:
the analogue of away steps over $\face_{\mathcal{X}}(x_t)$,
namely $d_t = v - x_t$,
where
$v = \argmax_{v\in\face_{\mathcal{X}}(x_t)} \innp{\nabla f(x_t)}{v}$,
and a Frank–Wolfe step on $\face_{\mathcal{X}}(x_t)$,
that is $d_t = u - x_t$, where
$u = \argmin_{v\in\face_{\mathcal{X}}(x_t)} \innp{\nabla f(x_t)}{v}$.
 Or, as \citet{freund2017extended} suggest, one could even attempt to find the 
minimum of $f$ over $\face_{\mathcal{X}}(x_t)$,
and use $d_t = w - x_t$,
where $w = \argmin_{v\in\face_{\mathcal{X}}(x_t)} f(v)$. The three aforementioned 
strategies all require access to $\face_{\mathcal{X}} (x_t)$. 

The convergence rate in \citet{freund2017extended}
for Algorithm~\ref{extendedInFace}
is roughly of the form \(\mathcal{O}(L D^{2} / t)\)
after \(t\) iterations.
The overall effect of $\alpha_1$ and $\alpha_2$
on sparsity and convergence rate is unclear.
Unlike the AFW algorithm,
due to the algorithmic setup,
the algorithm may omit the LMO call
over the whole feasible region in some iterations. When working with the minimal face directly is too expensive,
a good compromise with similar control over sparsity versus function value decrease is the Blended
 Conditional Gradient algorithm in Section~\ref{sec:bcg}
 that works with a specific 
 decomposition and does not 
rely on access to $\face_{\mathcal{X}} (x_{t})$.

\chapter{Conditional gradients in the large-scale setting}
\label{cha:FW-large}

We will now focus on conditional gradient methods for large-scale settings,
where access to the exact value of the objective function or its gradient is
computationally prohibitive. We start by considering stochastic gradients
(Section~\ref{sec:FW_Sto}) where, instead of exact gradients, (ostensibly
cheaper) gradient estimators with some random noise are used.  This is typical
for \emph{stochastic optimization}, in which the objective function that is the
expected value of some random function (Section~\ref{sec:FW_Sto}).  We also
study a special case of this problem where the objective function is defined as
the sum of a finite number of functions, also known as \emph{finite-sum
optimization}.

In \emph{distributed optimization} (Section~\ref{sec:distr-cond-grad}),
\emph{communication} is a new scarce resource required for function access.  In
particular, the components of the objective function are distributed over a set
of nodes, which thus have to share information to compute, e.g., the gradient
of the objective function.  Hence, one needs to design a set of algorithms that
can be implemented in a decentralized manner.  We discuss conditional gradient
algorithms that exploit only local and neighboring information to find a
suitable descent direction.

In Section~\ref{sec:online-cg} we examine conditional gradient methods for
\emph{online optimization}, which adds a time dimension to the problem: we
optimize objective functions, which will be only revealed in the future.  We
study a class of dynamic conditional gradient variants that exploit information
from prior objective functions.

The various large-scale settings can apply simultaneously
to the same problem, but this is out of scope of this survey.
We mention only \citet{WanTuZhang2020,WangWeiZhang2022}
for distributed online optimization.

\section{Conditional gradients for stochastic optimization}
\label{sec:FW_Sto}

In this section, we focus on the extension of Frank–Wolfe methods to
stochastic optimization problems. Here
the goal is to minimize an objective function which
is defined through an expectation over a set of functions.
The main challenge arises from the fact that typically
the cost of computing the objective value or its gradient
is much more expensive than computing the same quantities
for the individual underlying functions.
Stochastic optimization problems appear naturally in several areas, including machine learning
\citep[see, e.g.,][]{bottou2010large},
adaptive filters \citep{haykin2008adaptive}, and portfolio selection
\citep{shapiro2009lectures}. 

We consider a function
$\tilde{F} \colon\mathcal{X}\times \mathcal{Z} \to \mathbb{R}$
together with a random variable $Z\in \mathcal{Z}$.
We assume that the function $\tilde{F}$ is smooth and convex
in its first variable \(x \in \mathcal{X}\) unless stated otherwise.
The goal is to minimize the expectation of \(\tilde{F}(\cdot, Z)\),
which is formally defined as the objective function \(f\):
\begin{equation}\label{eq:stochastic_problem}
  \min_{x\in\mathcal{X}} f(x)
  \quad
  \text{where}
  \quad
  f(x) \defeq \E{\tilde{F}(x,Z)}.
\end{equation}
In the above expression, the expectation is over the randomness of $Z$. 

\begin{example}
  \label{ex:ML-supervised}
  A classic example of a stochastic optimization problem is a
  \myindex{supervised learning} problem defined by
  an expected risk minimization.
  In this setting the data points to learn from are
  pairs $(z,y)$ of
  an input or feature vector \(z\) and an output or label \(y\).
  The data are encoded by a random variable \(Z\) over the data
  points.
  The goal is to find a map $x$ on the inputs $z$ predict the outputs $y$.
  The most simple case is a linear model where $x$ is a linear
  function,
  which can be formulated as
  \begin{equation*}
    \min f(x)=\E({(z,y)\sim Z}){\ell (\innp{x}{z}, y)}
    ,
  \end{equation*}
  where the loss function $\ell \colon \mathbb{R}^{n} \to  \mathbb{R}$
  measures the gap between
  the predicted label $\innp{x}{z}$ and the true label $y$.
  Note that this is a stochastic optimization problem
  in the form of \eqref{eq:stochastic_problem} with
  $\tilde{F}(x, (z, y)) \defeq \ell (\innp{x}{z}, y)$.
  A common choice for the loss function \(\ell\) is the quadratic loss
  which leads to the following problem
  \begin{equation*}
    \min f(x)=\E({(z,y) \sim Z}){\norm[2]{\innp{x}{z} - y}^{2}}
    .
  \end{equation*}
\end{example}

We assume that the distribution of \(Z\) is independent of \(x\),
which is called the \emph{oblivious setting}. A more general version of the above problem is the
\emph{non-oblivious} setting in which the underlying distribution of
$Z$ depends on the variable $x$ and may change during the optimization
procedure.
Non-oblivious stochastic optimization appears naturallly in several problem classes, such as, e.g., multilinear extension of a discrete
submodular function, Maximum a Posteriori (MAP) inference in
determinantal point processes, and policy gradients in reinforcement
learning. In this section, we shall not consider the
\emph{non-oblivious} setting, but we should highlight that most
algorithms for the oblivious setting work equally well in the
non-oblivious setting. 

While we will assume that the expected function \(f\) is smooth,
with the exception of
Sections~\ref{sec:stoc-FW-grad-diff} and~\ref{sec:zeroth-order-stoch},
the individual functions \(\tilde{F}(\cdot, z)\)
need not be smooth.
All algorithms presented in this section aim at avoiding the 
expensive task of computing the exact gradient \(\nabla f\)
of the objective function
by utilizing some sort of a \emph{\myindex{gradient estimator}}
\(\widetilde{\nabla} f\). 
This results in the basic template shown in Template~\ref{sfw:tmp}
for stochastic Frank–Wolfe algorithms.
The idea of using a gradient estimator instead of the
true gradient has been extensively investigated for projected gradient decent and resulted in the popular stochastic gradient descent methodology. The idea of using gradient estimation was first proposed in the context of Markov chains by \citet{stoch51}.

\begin{template}[h]
\caption{Stochastic Conditional Gradient}
\label{sfw:tmp}
\begin{algorithmic}[1]
  \REQUIRE Start point $x_{0}\in\mathcal{X}$,
    step-sizes $0 \leq \gamma_{t} \leq 1$.
\FOR{$t=0$ \TO \dots}
  \STATE Compute gradient estimator $\widetilde{\nabla}f(x_t)$
  \STATE$v_{t} \gets \argmin_{v \in \mathcal{X}}
    \innp{\widetilde{\nabla}f(x_{t})}{v}$
\STATE$x_{t+1}\leftarrow x_t+\gamma_t(v_t-x_t)$
\ENDFOR
\end{algorithmic}
\end{template}

The  simplest gradient estimator is
the \myindex{stochastic gradient} \( \nabla \tilde{F}(x, z)\)
for a realization \(z\) of the random variable \(Z\),
which is known as
the \emph{Stochastic First-Order Oracle} (SFOO),
shown in Oracle~\ref{ora:SFOO}.

\begin{oracle}[H]
  \caption{\myindex{Stochastic First-Order Oracle} for \(f\) (SFOO)}
  \label{ora:SFOO}
  \begin{algorithmic}
    \REQUIRE Point \(x \in \mathcal X\)
    \ENSURE \(\widetilde{\nabla} f(x) \gets \nabla \tilde{F}(x,z)\),
      where \(z\) is an independent realization of \(Z\)
\end{algorithmic}
\end{oracle}

In order to quantify the accuracy of the stochastic first-order oracle,
we make the following assumption throughout this section,
which is standard in the stochastic optimization literature.
Throughout this section, the norm will be always the Euclidean norm,
i.e., the \(\ell_{2}\)-norm.
\begin{assumption}\label{ass:variance_stochastic}
  There is a \(\sigma^{2} > 0\) such that
  the variance of the stochastic gradient samples is bounded
  by $\sigma^{2}$, i.e., for all $x\in \mathcal{X}$, i.e.,
  $\E{\norm{\nabla \tilde{F}(x,Z) - \nabla f(x)}^2} \leq \sigma^2$.
\end{assumption}
\begin{remark}[Frank–Wolfe vertices become biased]
  The main difficulties of stochastic conditional gradients algorithms
  arise from the fact that Frank–Wolfe vertex of even an unbiased
  gradient estimator is usually a biased estimator
  of the Frank–Wolfe vertex.
  As a simple example, let \(Z\) be a uniformly distributed random
  variable over \(\{-1, +1\}\), i.e.,
  \(\Prob{Z=-1} = \Prob{Z=+1} = 1/2\).
  Let us consider \(\smash{f(x) \defeq \E{(x - x^{*} + Z)^{2}}}
  = (x - x^{*})^{2} + 1\)
  over the interval \([0, 1]\) for a constant \(0 < x^{*} < 1\).
  Now at any \(0 \leq x \leq 1\)
  the Frank–Wolfe vertex of the stochastic gradient
  \(\widetilde{\nabla} f(x) = 2 (x - x^{*} + Z)\),
  is uniformly distributed over \(\{0, 1\}\),
  and hence holds no information about \(x^{*}\).

  This simple example highlights that
  the core of the Frank–Wolfe algorithm, linear minimization,
  is sensitive to accuracy, and introduces bias:
  \(\E{\min_{v \in \mathcal{X}} \innp{\tilde{\nabla}}{v}}
  \neq \min_{v \in \mathcal{X}} \innp{\E{\tilde{\nabla}}}{v}\),
  where \(\tilde{\nabla}\) is a random vector.
  Stochastic Frank–Wolfe algorithms counteract this
  by increasing accuracy of gradient estimators.
  This is a drawback compared to stochastic projected gradient descent
  methods.
\end{remark}

We shall use \(\E(t){\cdot}\), \(\Var(t){\cdot}\)
to denote the conditional expectation and variance
given the run of the algorithm till
the end of computation of iterate \(x_{t}\).
The following lemma quantifies how the accuracy of the
gradient estimate \(\widetilde{\nabla} f(x_{t})\) affects the 
progress made in a single step.

\begin{lemma}
  \label{lem:sfw}
  If $f$ is an \(L\)-smooth convex function on
  a compact convex set $\mathcal{X}$ with diameter at most \(D\),
  then for all $t\geq 0$
  the iterates of Template~\ref{sfw:tmp} satisfy the following
 \begin{align}
   \label{eq:SFW-step}
   &
   \begin{multlined}
   \E{f(x_{t+1}) - f(x^*)}
   \\[-1ex]
   \leq
  (1 - \gamma_{t}) \E{f(x_{t}) - f(x^*)}
  + \gamma_{t}
  \smash[t]{\sqrt{\E{\norm{\widetilde{\nabla}f(x_{t}) - \nabla f(x_{t})}^{2}}}} D
  + \frac{L D^{2}}{2} \gamma_{t}^{2}
  ,
  \end{multlined}
  \\
  &
  \begin{multlined}
  \E{f(x_{t+1}) - f(x_{t})}
  \\[-1ex]
  \leq
  - \gamma_{t} \E{g(x_{t})}
  + \gamma_{t}
  \smash[t]{\sqrt{\E{\norm{\widetilde{\nabla}f(x_{t}) - \nabla f(x_{t})}^{2}}}} D
  + \frac{L D^{2}}{2} \gamma_{t}^{2},
  \end{multlined}
\end{align}
where $  g(x) = \max_{v \in \mathcal{X}} \innp{\nabla f(x)}{x-v}$ is
the Frank–Wolfe gap (Definition~\ref{FrankWolfeGap}). The second inequality does not require convexity of \(f\) as long as it is smooth. Moreover, these inequalities also hold for conditional expectation, i.e., 
by replacing \(\E{\cdot}\)
with \(\E(t){\cdot}\) on the left-hand side,
and with the identity function on the right-hand side.
\end{lemma}

Note that the key difference from the non-stochastic version in Lemma~\ref{lemma:progress} is the extra term
measuring the accuracy of stochastic gradients:
\[
\gamma_{t}
\sqrt{\E{\norm{\widetilde{\nabla}f(x_{t}) - \nabla f(x_{t})}^{2}}} D.
\]
Controlling this term is key to obtaining good convergence rates.

\begin{proof}[Proof of Lemma~\ref{lem:sfw}]
The first inequality easily follows from the second inequality,
as 
\(f(x_{t}) - f(x^{*}) \leq g(x_{t})\) for convex \(f\), see Equation~\eqref{eq:gaps}.
Hence, we prove only the second one.
The key is to estimate the error due to the approximation of the
gradient, using the optimality of \(v_{t}\):
\begin{equation*}
 \begin{split}
  g(x_{t}) + \innp{\nabla f(x_{t})}{v_{t} - x_{t}}
  &
  =
  \max_{v \in \mathcal{X}} \innp{\nabla f(x_{t})}{v_t - v}
  \\
  &
  \leq
  \max_{v \in \mathcal{X}} \innp{\widetilde{\nabla} f(x_{t})}{v_t - v}
  + \max_{v \in \mathcal{X}}
  \innp{\nabla f(x_{t}) - \widetilde{\nabla} f(x_{t})}{v_t - v}
  \\
  &
  \leq
  \norm{\nabla f(x_{t}) - \widetilde{\nabla} f(x_{t})} D
  ,
 \end{split}
\end{equation*}
where the last inequality follows by using
the fact that $v_{t} = \argmin_{v \in \mathcal{X}}
\innp{\widetilde{\nabla}f(x_{t})}{v}$ as well as the Cauchy-Shwarz inequality.
We combine this with the standard progress estimation, using
smoothness of $f$ (see Lemma~\ref{lemma:progress}):
\begin{equation*}
 \begin{split}
  f(x_{t+1}) - f(x_{t})
  &
  \leq \gamma_{t} \innp{\nabla f(x_{t})}{v_{t} - x_{t}}
  + \frac{L}{2} \gamma_{t}^2\norm{v_{t} - x_{t}}^{2}
  \\
  &
  \leq
  - \gamma_{t} g(x_{t})
  + \gamma_{t}
  \norm{\widetilde{\nabla}f(x_{t}) - \nabla f(x_{t})} D
  + \frac{L D^{2}}{2}\gamma_{t}^2
  .
 \end{split}
\end{equation*}
The proof concludes by taking expectations, and using
the following application of Jensen's inequality
to bound the  expectation of the error of the stochastic gradient:
 \begin{equation*}
  \E{\norm{\widetilde{\nabla}f(x_t)-\nabla f(x_t)}}
  \leq \sqrt{\E{\norm{\widetilde{\nabla}f(x_t)-\nabla f(x_t)}^2}}
  .
  \qedhere
 \end{equation*}
\end{proof} 

We include here an application appearing in many proofs below
for the agnostic step size rule \(\gamma_{t} = 2 / (t+2)\).
\begin{lemma}
  \label{lem:stochastic-rate}
  Let $f$ be an \(L\)-smooth convex function on
  a compact convex set $\mathcal{X}$, with diameter~\(D\).
  In Template~\ref{sfw:tmp} let the step sizes be
  \(\gamma_{t} = 2 / (t+2)\).
  Assume that the stochastic gradients satisfy
  \begin{equation}
    \label{eq:stochastic-gradient-rate}
    \sqrt{\E{\norm{\widetilde{\nabla}f(x_{t})-\nabla f(x_{t})}^{2}}}
    \leq
    G \gamma_{t}
    .
  \end{equation}
  Then the iterates \(x_{t}\) satisfy
  \begin{equation}
    \label{eq:stochastic-rate}
    \E{f(x_{t}) - f(x^{*})} \leq \frac{4 G D + 2 L D^{2}}{t + 2}
    .
  \end{equation}
  (Here as an exception.
  \(G\) is not necessarily an upper bound on the size of
  gradient of \(f\).)
\end{lemma}
\begin{proof}
Plugging \eqref{eq:stochastic-gradient-rate} into Lemma~\ref{lem:sfw}
we obtain
\begin{equation}
  \label{recstochastic}
  \E{f(x_{t+1})-f(x^*)}
  \leq
  (1 - \gamma_{t}) \E{f(x_{t}) - f(x^{*})}
  + \frac{2 G D +  L D^{2}}{2} \gamma_{t}^{2}
  .
\end{equation}
We conclude the proof by induction.
Let \(Q \defeq 4 G D + 2 L D^{2}\),
so that the claim becomes \(\E{f(x_{t}) - f(x^{*})} \leq Q / (t+2)\).
As $\gamma_0 = 1$, we have that $\E{f(x_{1})-f(x^*)}
 \leq Q / 4$ by application of Equation~\eqref{recstochastic}.
Hence \eqref{sfw-cv} holds for $t=1$.
Suppose \eqref{sfw-cv} holds for some \(t\).
By \eqref{recstochastic}, we have
\begin{equation*}
  \begin{split}
   \E{f(x_{t+1}) - f(x^{*})}
   &
   \leq
   \left(
     1 - \frac{2}{t+2}
   \right)
   \frac{Q}{t + 2}
   + \frac{Q}{(t+2)^{2}}
   \\
   &
   =
   \frac{Q (t + 1)}{(t + 2)^{2}}
   \\
   &
   \leq\frac{Q}{t + 3}
\end{split}
\end{equation*}
by $(t+1)(t+3)\leq(t+2)^2$.
\end{proof}

The accuracy of the gradient estimation plays a key role in how fast stochastic optimization methods converge. 
This explains why estimators built from a few number of stochastic gradients
but with significantly smaller variance
have received a lot of attention over the past few years,
primarily for stochastic gradient descent methods.
Some of these estimators will be mentioned in later sections.

A simple way to obtain more accurate estimates of the gradient
is by computing the average over a finite number of independent samples.
More precisely, we can construct a simple unbiased gradient
estimator with improved accuracy by drawing independent 
samples $z_1, \dotsc, z_b$ of $Z$ and computing:
\begin{equation}
  \label{eq:def_stochastic_gradient}
  \widetilde{\nabla}f(x) =
  \frac{1}{b} \sum_{j=1}^{b} \nabla \tilde{F} (x, z_{j}).
\end{equation}
It is easy to verify that the variance of such estimator is a factor $b$ less compared to the single simple estimate for the case $b=1$. 
This technique is used to increase the accuracy of the estimators as the
algorithm progresses.

We close this section by remarking that the stochastic problem in
\eqref{eq:stochastic_problem} appears ubiquitously in many domains,
however, an important special case of it is when the randomness has a
finite support.

\begin{remark}[The case of \myindex{finite-sum minimization}]
  \label{remark-finite-sum}
The constrained finite-sum problem
\begin{equation}
  \label{pb:finite}
  \min_{x\in\mathcal{X}} \frac{1}{m} \sum_{i=1}^{m} f_{i}(x)
\end{equation}
is a special case of Problem~\eqref{eq:stochastic_problem}
where $\mathcal{X}$ is a compact convex set and
$f_{1}, \dotsc, f_{m} \colon \mathcal{X} \to \mathbb{R}$
are smooth, possibly non-convex functions,
via the choices \smash{\(\tilde{F}(x,z) \defeq f_{z}(x)\)}
and \(Z\) a uniform random variable over \(\{1, 2, \dotsc, m\}\).
This setting is quite
ubiquitous in machine learning, as many learning problems are
formulated via empirical risk minimization, which is of this form.
As such, all results of the more general
Problem~\eqref{eq:stochastic_problem} also apply to
finite-sum minimization,
hence we will only explicitly mention results on finite-sum
minimization, where the finite-sum structure leads to improvement.
\end{remark}

\subsection{Stochastic first-order oracle complexity}
\label{sec:stoch-FOO-complexity}

Here we recall some of the known complexity lower bounds for stochastic
first-order methods.
For a maximum expected primal gap error \(\varepsilon > 0\)
the lower bounds are summarized in Table~\ref{tab:stoch-FOO-complexity}.
Note that here the only accuracy constraint on stochastic gradients
is that they are bounded, while most convergence results in this chapter
use a weaker assumption that
only the variance of the stochastic gradients are bounded.

\begin{table}
  \caption{Minimum number of stochastic first-order oracle calls
    to optimize a convex objective function
    up to an expected primal gap
    of at most \(\varepsilon > 0\), where the worst-case
    examples depend on \(\varepsilon\).
    The feasible region is \(n\)-dimensional and
    contains an \(\ell_{\infty}\)-ball of radius \(r>0\).
    The stochastic gradients are assumed to be unbiased
    and bounded by \(G\) in the dual of the \(\ell_{p}\)-norm
    \citep[Theorems~1 and~2]{Agarwal2009}.}
  \label{tab:stoch-FOO-complexity}

  \small
  \centering
  \begin{tabularx}{\linewidth}{@{}>{\hangindent=1em\hangafter=1}Xll@{}} 
    \toprule
    Objective function & Stochastic FOO calls & Restrictions\\
    \midrule
    \(G\)-Lipschitz in \(\ell_{p}\)-norm &
    \(\Omega\bigl((G r n^{\max \{1/2, 1/p\}} / \varepsilon)^{2}\bigr)\) \\
    \(G\)-Lipschitz in \(\ell_{p}\)-norm,
    \(\mu\)-strongly convex in \(\ell_{2}\)-norm &
    \(\Omega\bigl(G^{2} n^{2/p - 1} / (\mu^{2} \varepsilon)\bigr)\)
    & \(1 \leq p < 2\) \\
    \(G\)-Lipschitz in \(\ell_{p}\)-norm,
    \(\mu\)-strongly convex in \(\ell_{2}\)-norm &
    \(\Omega\bigl(G^{2} / (\mu^{2} \varepsilon)\bigr)\)
    & \(p=\infty\) \\
    \bottomrule
  \end{tabularx}
\end{table}

The authors are not aware of a lower bound on
the number of required stochastic oracle calls for
finding a point with a Frank–Wolfe gap of $\varepsilon$
for constrained problems,
but for completeness, we would like to state that
the analogous version of the above goal for unconstrained
optimization problems, which is obtaining a solution \(x\)
with \(\dualnorm{\nabla f(x)} \leq \varepsilon\)
requires at least
\(\Omega(\sqrt{L h_{0}} / \varepsilon
+ (\sigma^{2} / \varepsilon^{2}) \log (L h_{0} / \varepsilon^{2}))\)
stochastic first-order oracle calls
in the worst-case depending on \(\varepsilon\)
\citep[Theorem~2]{foster19b}, where \(L\) is the objective function smoothness parameter and \(\sigma^{2}\) is an upper bound on the stochastic oracle of variance.
Here as usual \(h_{0}\) is the primal gap error of the starting
point.

\subsection{Stochastic Frank–Wolfe algorithm}
\label{sec:stoch-FW}

As mentioned above, the most natural extension of the FW algorithm for
stochastic optimization is obtained by replacing the gradient in the
update with its stochastic approximation, as suggested in
\citet{svrf16}.  We formally present this algorithm in
Algorithm~\ref{sfw}.
Note that the batch sizes $b_{t}$ are input parameters of this algorithm,
to allow for gradient estimators that increase in accuracy
during the run of the algorithm. In the next theorem we formally state the convergence rate of this algorithm. 

\begin{algorithm}[h!]
\caption{Stochastic Frank–Wolfe (SFW) \citep{svrf16}}
\label{sfw}
\begin{algorithmic}[1]
  \REQUIRE Arbitrary start point $x_0\in\mathcal{X}$,
    batch sizes $b_{t}$, step sizes $0 \leq \gamma_{t} \leq 1$
    for \(t \geq 0\)
  \ENSURE Iterates $x_1, \dotsc \in \mathcal{X}$
\FOR{$t=0$ \TO \dots}
\STATE Draw $\{z_1,\dotsc,z_{b_t}\}$, that is, $b_t$ independent realizations of $Z$.
\STATE
  \label{gradest}
  \(\widetilde{\nabla}f(x_{t})
  \leftarrow
  \frac{1}{b_{t}} \sum_{j=1}^{b_{t}} \nabla \tilde{F}(x_t,z_j) \)
\STATE$v_t\leftarrow\argmin_{v\in\mathcal{X}}\innp{\widetilde{\nabla}f(x_t)}{v}$
\STATE$x_{t+1}\leftarrow x_t+\gamma_t(v_t-x_t)$
\ENDFOR
\end{algorithmic}
\end{algorithm}

\begin{theorem}
  \label{th:sfw:expectation}
  Let \(f\) be a convex, \(L\)-smooth function
  over a convex compact set \(\mathcal{X}\) with diameter at most \(D\).
  If the stochastic gradient samples of \(f\) have
  variance of at most \(\sigma^{2}\),
  then the Stochastic Frank–Wolfe algorithm (Algorithm~\ref{sfw}) with
  batch sizes $b_{t} \defeq (t+2)^{2} / \alpha$
  for a fixed \(\alpha > 0\)
  and step sizes $\gamma_{t} \defeq 2/(t+2)$,
  ensures for all \(t \geq 1\)
 \begin{equation}
   \label{sfw-cv}
   \E{f(x_{t})} - f(x^{*})
   \leq
   \frac{2 (\sqrt{\alpha} \sigma D +  L D^{2})}{t+2}
   .
 \end{equation}
  Equivalently, to achieve an expected primal gap
  \(\E{f(x_{t})} - f(x^{*}) \leq \varepsilon\),
  for \(\varepsilon > 0\),
  the Stochastic Frank–Wolfe algorithm (Algorithm~\ref{sfw})
  computes $\mathcal{O}((\sqrt{\alpha} \sigma D +  L D^{2})^{3} /
  (\alpha \varepsilon^{3}))$
  stochastic gradients and
  performs $\mathcal{O}((\sqrt{\alpha} \sigma D +  L D^{2})
  / \varepsilon)$ linear minimizations.
\end{theorem}

For the sake of comparison, we note that the projected stochastic gradient descent
uses \(\mathcal{O}(G^{2} D^{2} / \varepsilon^{2})\)
stochastic gradients for an expected primal gap at of most \(\varepsilon
> 0\), where \(G^{2}\) is an upper bound on the second moment
\(\E{\norm{\nabla \tilde{F}(x, z)}^{2}} \leq G^{2}\)
of the stochastic gradients, see \citet[Eq.~(2.18)]{Nemirovski}

We should also add that in the original paper \citet{svrf16}, the
parameter $\alpha$ is selected as
\(\alpha = \left(L D / \sigma \right)^{2}\),
which is obtained by minimizing the upper bound of the
linear minimizations performed.
However, here
we have presented a slightly more general form of this result
to highlight that
a parameter-free choice like \(\alpha = 1\)
has the same convergence bound
up to a  constant factor that depends on some parameters of the objective function.
\begin{proof}
We will use Lemma~\ref{lem:stochastic-rate} to prove the bound.
In order to do so we
need to bound $\norm{\widetilde{\nabla}f(x_t)-\nabla f(x_t)}$.
Since the samples are drawn i.i.d., the variance of
$\widetilde{\nabla}f(x_t)$ is bounded from above by
$ \E{\norm{\widetilde{\nabla}f(x_t)-\nabla f(x_t)}^2}
 \leq \frac{\sigma^{2}}{b_{t}}$. 
Thus using the values of \(b_{t}\) and \(\gamma_{t}\) provided in the statement of the theorem we obtain
\begin{equation*}
  \sqrt{\E{\norm{\widetilde{\nabla}f(x_{t}) - \nabla f(x_{t})}^{2}}}
  \leq
  \frac{\sigma}{\sqrt{b_{t}}}
  =
  \frac{\sqrt{\alpha} \sigma}{t + 2}
  =
  \frac{\sqrt{\alpha} \sigma}{2} \gamma_{t}
  .
\end{equation*}
The claim follows by Lemma~\ref{lem:stochastic-rate}
with the choice \(G \defeq \sqrt{\alpha} \sigma / 2\).
\end{proof}

We now turn our attention to the case where $f$ is not convex.
In the deterministic setting (Section~\ref{sec:nonconvex}),
the minimal Frank–Wolfe gap \(g(x_{t})\) of all iterates \(x_{t}\)
is used as an accuracy measure, with the implicit understanding that
an iterate realizing the minimal gap is chosen as the final solution.
In the stochastic setting,
together with the exact gradients \(\nabla f(x_{t})\) of the iterates,
the Frank–Wolfe gaps are expensive to compute, too,
requiring a different choice of final solution.
We examine the common choice of a uniformly random selection
from the generated iterates, thus the expected Frank–Wolfe gap
becomes the average of the expected Frank–Wolfe gaps
of the individual iterates: \(\sum_{\tau=0}^{t-1} \E{g(x_{\tau})} / t\).

Theorem~\ref{th:sfw:ncvx}\ref{th:sfw:ncvx:i} follows from \citet{reddi2016stochastic} and assumes a fixed time horizon. Alternatively, Theorem~\ref{th:sfw:ncvx}\ref{th:sfw:ncvx:ii} presents a slightly slower convergence rate as done in \citet{combettes20adasfw} but does not require a fixed time horizon.

\begin{theorem}
 \label{th:sfw:ncvx}
 Let $f$ be an $L$-smooth (and not necessarily convex) function on a compact convex set  $\mathcal{X}$ with diameter
 at most \(D\).
 Consider SFW (Algorithm~\ref{sfw})
 with stochastic gradients of variance at most \(\sigma^{2}\).
 \begin{enumerate}
 \item\label{th:sfw:ncvx:i}
   If $b_{t} = (\sigma/(LD))^2T$ and $\gamma_{t} = 1 / \sqrt{T}$,
   for all \(0 \leq t \leq T-1\), then
   \begin{equation*}
     \frac{1}{T} \sum_{t=0}^{T-1} \E{g(x_{t})}
     \leq\frac{\bigl(f(x_0) - f(x^{*})\bigr) + (3/2) L D^{2}}{\sqrt{T}}.
   \end{equation*}
   Equivalently,
   for an expected Frank–Wolfe gap at most \(\varepsilon > 0\),
   the Stochastic Frank–Wolfe algorithm (Algorithm~\ref{sfw}) computes
   $\mathcal{O}(\sigma^{2} L^{2} D^{4}/\varepsilon^{4})$
   stochastic gradients and performs
   $\mathcal{O}(L^{2}D^{4}/\varepsilon^{2})$
   linear minimizations over the feasible region,
   running \(T = L^{2} D^{4} / \varepsilon^{2}\)
   iterations of the algorithm,
   and choosing one of the iterates
   \(x_{0}\), \dots, \(x_{T}\)
   uniformly randomly as the final solution.
 \item\label{th:sfw:ncvx:ii}
   If $b_{t} = (\sigma/(LD))^{2}(t+1)$
   and $\gamma_{t} = 1/(t+1)^{1/2+\nu}$
   for a parameter \(0 < \nu < 1/2\),
   then for all \(t \geq 1\)
   \begin{equation*}
     \frac{1}{t} \sum_{\tau=0}^{t-1} \E{g(x_{\tau})}
     \leq
     \frac{\bigl(f(x_{0}) - f(x^{*})\bigr)
       + (3/2) L D^{2} \zeta(1+\nu)}{t^{1/2-\nu}}
     ,
   \end{equation*}
   where $\zeta(\nu) \defeq
 \sum_{s=0}^{+\infty}1/(s+1)^{\nu}\in\mathbb{R}_+$ is the Riemann zeta
 function.
 \end{enumerate}
\end{theorem}

\begin{remark}
  In Theorem~\ref{th:sfw:ncvx}\ref{th:sfw:ncvx:ii},
  the upper bound is minimized approximately for
  \(\nu \approx 1 / \ln (t+1)\) using
  the approximation \(\zeta(1 + \nu) \approx 1 / \nu\),
  suggesting a choice \(\nu = 2 / \ln (L D^{2}  / \varepsilon)\)
  for a maximum gap \(\varepsilon > 0\).
  The approximation \(\zeta(1 + \nu) \approx 1 / \nu\)
  is also useful for estimating the choice of a fixed value of
  \(\nu\),
  e.g., the somewhat arbitrary choice
  $\nu = 0.05$ provides
 $\E{g(X_t)}=\mathcal{O}(1/t^{0.45})$ for all \(t\) and
 $\zeta(1+\nu) = \zeta(1.05)\approx20.6$.
\end{remark}

\begin{proof}
By Lemma~\ref{lem:sfw}, we obtain for \(t \geq 0\)
 \begin{equation*}
  \gamma_t\E{g(x_t)}
  \leq\E{f(x_t)}-\E{f(x_{t+1})}+\gamma_t\frac{\sigma}{\sqrt{b_t}}D+\frac{LD^2}{2}\gamma_t^2.
 \end{equation*}
 \begin{enumerate}
 \item When $\gamma_t = 1/\sqrt{T}$ and $b_t = (\sigma/(LD))^{2}T $,
   we obtain
 \begin{equation*}
  \frac{1}{\sqrt{T}}\E{g(x_t)}
  \leq\E{f(x_t)}-\E{f(x_{t+1})}+\frac{3LD^2}{2T},
 \end{equation*}
 so, by telescoping for $0 \leq t\leq T-1$,
 \begin{equation*}
  \frac{1}{\sqrt{T}}\sum_{t=0}^{T-1}\E{g(x_t)}
  \leq\E{f(x_0)}-\E{f(x_T)}+\frac{3LD^2}{2}
  \leq f(x_0) - f(x^{*})+\frac{3LD^2}{2}
  ,
 \end{equation*}
 \item When $\gamma_{t} = 1/(t+1)^{1/2+\nu}$ and
   $b_{t} = (\sigma/(LD))^{2} (t+1)$, we obtain
 \begin{equation*}
  \frac{1}{(t+1)^{1/2+\nu}}\E{g(x_t)}
  \leq\E{f(x_t)}-\E{f(x_{t+1})}+\frac{3LD^2}{2(t+1)^{1+\nu}},
 \end{equation*}
 so, by telescoping for $0 \leq \tau\leq t$,
 \begin{equation*}
  \begin{split}
   \frac{1}{(t + 1)^{1/2+\nu}} \sum_{\tau=0}^{t} \E{g(x_{\tau})}
   & \leq
   \sum_{\tau=0}^{t} \frac{1}{(\tau+1)^{1/2+\nu}} \E{g(x_{\tau})}
   \\
  & \leq\E{f(x_{0})}-\E{f(x_{t+1})}+\frac{3LD^{2} \zeta(1+\nu)}{2}\\
  &\leq f(x_{0})-f(x^{*})+\frac{3LD^{2} \zeta(1+\nu)}{2}
  .
  \qedhere
  \end{split}
 \end{equation*}
 \end{enumerate}
\end{proof}

\subsection{Stochastic Away Frank–Wolfe algorithm}
\label{sec:ASFW}

Here we recall stochastic variants from \citet{ASFW-PSFW2016} of two of the most important advanced conditional gradient algorithms, namely,
the Away-step Frank–Wolfe algorithm (Algorithm~\ref{away}) and
the Pairwise Frank–Wolfe algorithm (Algorithm~\ref{pairwise}).
The idea is the same as for the Stochastic Frank–Wolfe algorithm
(Algorithm~\ref{sfw}):
Do exactly the same as the deterministic variant except using
stochastic gradients instead of exact gradients.
Here this is done even more consequently, using the short step rule.
These result in the Away-step Stochastic Frank–Wolfe algorithm
(ASFW, Algorithm~\ref{alg:ASFW}) and the Pairwise Stochastic Frank–Wolfe
algorithm (PSFW), of which we present only the first one; the second one follows similarly.

\begin{algorithm}[H]
  \caption{Away-step Stochastic Frank–Wolfe (ASFW)
		\citep{ASFW-PSFW2016}}
	\label{alg:ASFW}
	\begin{algorithmic}[1]
    \REQUIRE Start atom \(x_{0} \in \vertex{P}\)
    \ENSURE Iterates $x_{1}, \dotsc \in P$
		\STATE$\mathcal{S}_{0} \gets \{x_{0}\}$
		\STATE$\lambda_{x_{0}, 0} \gets 1$
		\FOR{$t=0$ \TO \dots}
		\STATE Draw $\{z_{1}, \dotsc ,z_{b_{t}}\}$, that is,
		$b_{t}$ independent realizations of $Z$.
		\STATE
		\label{line:ASFW-gradest}
		\(\widetilde{\nabla} f(x_{t}) \gets
		\frac{1}{b_{t}} \sum_{j=1}^{b_{t}} \nabla \tilde{F}(x_{t}, z_{j})\)
		\STATE \label{line:ASFW-FWVertex}
		\(v_{t}^{\text{FW}} \gets
		\argmin_{v \in P} \innp{\widetilde{\nabla} f(x_{t})}{v}\)
		\STATE \label{line:ASFW-AwayVertex}
		$v_{t}^{\text{A}} \gets
		\argmax_{v \in \mathcal{S}_{t}} \innp{\widetilde{\nabla} f(x_{t})}{v}$
		\IF[Frank–Wolfe step]
		{\(\innp{\widetilde{\nabla} f(x_{t})}{x_{t} - v_{t}^{\text{FW}}}
			\geq
			\innp{\widetilde{\nabla} f(x_{t})}{v_{t}^{\text{A}} - x_{t}}\)}
		\label{line:ASFW-FWstep}
		\STATE \(d_{t} \gets x_{t} - v_{t}^{\text{FW}}\),
		\(\gamma_{t, \max} \gets 1\)
		\ELSE[away step]
		\STATE \label{line:ASFW-away-step}
		\(d_{t} \gets v_{t}^{\text{A}} - x_{t}\),
		\(\gamma_{t,\max} \gets \frac{\lambda_{v_{t}^{\text{A}}, t}}{1 -
			\lambda_{v_{t}^{\text{A}}, t}}\)
		\ENDIF
		\STATE \(\gamma_{t} \gets \min \left\{
		\frac{\innp{\widetilde{\nabla} f(x_{t})}{d_{t}}}{L
			\norm{d_{t}}^{2}},
		\gamma_{t, \max} \right\}\)
		\STATE \(x_{t+1} \gets x_{t} - \gamma_{t} d_{t}\)
		\STATE Update coefficients and active set:
		choose the \(\lambda_{v, t+1}\) and \(\mathcal{S}_{t+1}\)
		as in Algorithm~\ref{away}.
		\ENDFOR
	\end{algorithmic}
\end{algorithm}

Convergence results follow similarly to the deterministic
algorithms and the stochastic results above, see \citet{ASFW-PSFW2016}
for details.
The rates presented here correct typos in the final simplification
in \citet[Theorem~1, Corollary~2]{ASFW-PSFW2016}
(avoiding the false claim \(\log (1 - \rho) / \log (1 - \beta) < 1\)
for \(0 < \beta < \rho < 1\)).
We point out only that the maximal rate of possibly non-progressing
iterations (drop steps and swap steps) add an additional factor to the
exponent in the number of stochastic gradient oracle calls.
Also note that linear convergence rates in expectation automatically
lead to high-probability linear convergence rates:
Let \(0 < \alpha < \rho \leq 1\) be constants and let \(h_{t}\) be
nonnegative random numbers with \(\E{h_{t}} \leq (1 - \rho)^{t}\)
for all large enough \(t\).
Then one has \(\sum_{t=0}^{\infty} \E{h_{t} / (1 - \alpha)^{t}}
\leq \sum_{t=0}^{\infty} (1 - \rho)^{t} / (1 - \alpha)^{t} < \infty\),
and in particular the sum
\(\sum_{t=0}^{\infty} h_{t} / (1 - \alpha)^{t}\)
is almost surely convergent, and hence
\(\lim_{t \to \infty} h_{t} / (1 - \alpha)^{t} = 0\)
almost surely.
Also note that while the batch size increases exponentially
in the number of iterations
to control the variance of the gradient estimator,
the increase is polynomial in the inverse of the primal gap bound.

\begin{theorem}
	\label{thm:ASFW}
	Let the stochastic function \(\tilde{F}(\cdot, z)\)
	be \(G\)-Lipschitz, \(L\)-smooth,
  and \(\mu\)-strongly convex in its first argument,
	with \(\abs{\tilde{F}(\cdot, z)} \leq M\)
	over a polytope \(P\) of dimension \(n\)
	and diameter at most \(D\).
	Let \(f\) be the expectation of \(\tilde{F}(\cdot, z)\) in \(z\) and choose
	\[\rho \defeq \min \left\{
	\frac{1}{2},
	\frac{\mu \delta^{2}}{16 L D^{2}}
	\right\}
  .
  \]
	With the choice
	\(b_{t} = \left\lceil 1 / (1 - \rho)^{2 t + 4} \right\rceil\)
	we have for all \(t \geq 1\)
	that
  the Away-step Frank–Wolfe algorithm (Algorithm~\ref{away})
	converges as
	\begin{equation}
		\label{eq:ASFW}
    \E{f(x_{t})} - f(x^{*})
		\leq
		\mathcal{O}_{M, G, D, n, \rho}
		\left(
		t^{3/2} (1 - \rho) ^{t / 2}
		\right)
		,
	\end{equation}
  where the big O notation hides a constant depending on
	\(M\), \(G\), \(D\), \(\rho\) and \(n\).
	Equivalently,
	for an expected primal gap at most \(\varepsilon > 0\),
	the algorithm makes
	\(\mathcal{O}(\log (1 / \varepsilon))\)
  linear minimizations
	and
	\(\mathcal{O}(\log^{6} (1 / \varepsilon) / \varepsilon^{4})\)
	stochastic oracle calls. As a consequence,
  for any \(0 < \alpha < \rho\), for all but finitely many \(t\)
	one has almost surely
	\(f(x_{t}) - f(x^{*}) \leq (1 - \alpha)^{t / 2}\).
\end{theorem}

\subsection{Momentum Stochastic Frank–Wolfe algorithm}
\label{sec:momentum-FW}

As mentioned earlier, the Stochastic Frank–Wolfe algorithm requires an
increasing batch size to ensure convergence,
i.e., as the algorithm progresses,
it computes an increasing number of new stochastic gradients before
making any progress even under favorable circumstances.
In this section, we present a variant of the Stochastic Frank–Wolfe
method which immediately incorporates
\emph{every new stochastic gradient} at the cost of an extra linear
minimization.
Allowing the algorithm to make steady progress when circumstances
permit is an important soft heuristic for fast
convergence in practice.

Accuracy of the stochastic gradient estimator
is maintained by using the new stochastic gradient only
to modify a little the previous gradient approximation,
akin to momentum steps.
This is called the momentum or averaging technique,
which is not specific to conditional gradients,
has been widely studied in the stochastic optimization literature
\citep{ruszczynski1980feasible,ruszczynski2008merit,yang2016parallel,mokhtari2017large},
and has led to more robust algorithms than using only fresh stochastic gradients for gradient approximation.

The main advantage of this approach is reducing the size of active batch, i.e.,
the number of independent realizations of $Z$ and subsequently stochastic
gradient realizations $\tilde{F}(x_t,z)$ evaluated at the current iterate $x_t$
for computing the new iterate~$x_{t+1}$. Requiring a small active batch size
that is not growing over time is an important feature for a stochastic method
from a practical point of view. This feature allows us to incorporate the new
obtained information (realizations of $Z$) as fast as possible and update our
decision variable instantaneously, rather than accumulating a large set of
information before each adjustment of our decision variable.  This property is
indeed crucial in several streaming applications, in which the sample data that
is used in the objective function of our optimization  problem arrive
sequentially over time. In this case, we might only be able to compute a
certain number of stochastic gradients per iteration, and may not be able to
freely increase the number of stochastic gradient calls per iteration, as we
need to come up with a new decision variable within a fixed amount of time. 

Concretely, the momentum method uses the following gradient
estimation, which mixes a stochastic gradient $\nabla \tilde{F}(x_t,z)$
with the gradient estimation $\widetilde{\nabla}f(x_{t-1})$ from the previous iteration:
\begin{equation}\label{eq:grad_averaging}
\widetilde{\nabla}f(x_{t}) = (1-\rho_t) \widetilde{\nabla}f(x_{t - 1}) + \rho_t \nabla \tilde{F}(x_t,z_t),
\end{equation}
where \(\rho_{0} = 1\) to remove the need of a good initial estimator
\(\widetilde{\nabla} f(x_{0})\),
and the momentum parameter satisfies $0\leq \rho_t \leq 1$.
Note that $\widetilde{\nabla}f(x_{t})$ is not an unbiased estimator of
the gradient $\nabla f(x_t)$ as the previous estimator
\(\widetilde{\nabla} f(x_{t-1})\) induces a bias.
The estimator trades bias for lower variance via the momentum parameter
\(\rho_{t}\).

The Momentum Stochastic Frank–Wolfe algorithm (Momentum SFW)
\citep{mokhtari2020stochastic}, presented in
Algorithm~\ref{eq:grad_averaging},
is obtained from the Stochastic Frank–Wolfe algorithm (Algorithm~\ref{sfw})
by replacing the gradient estimator
with the momentum estimator above.

\begin{algorithm}[h]
\caption{Momentum Stochastic Frank–Wolfe (Momentum SFW) \citep{mokhtari2020stochastic}}\label{algo_MSFW}
  \begin{algorithmic}[1]
    \REQUIRE Start point $x_0\in\mathcal{X}$,
      step sizes \(\gamma_{t}\), and momentum terms $\rho_{t}$
      for \(t \geq 0\) with \(\rho_{0} = 1\)
    \ENSURE Iterates $x_1, \dotsc \in \mathcal{X}$
    \FOR{$t=0$ \TO \dots}
      \STATE Draw $z$, an independent realization of $Z$.
   \STATE $\widetilde{\nabla}f(x_{t}) \leftarrow (1-\rho_t) \widetilde{\nabla}f(x_{t - 1}) + \rho_t \nabla \tilde{F}(x_t,z)$
     \COMMENT{\(\widetilde{\nabla} f (x_{0})
       = \nabla \tilde{F}(x_{0}, z)\)}
 \STATE$v_t\leftarrow\argmin_{v\in\mathcal{X}}\innp{\widetilde{\nabla}f(x_{t})}{v}$
   \STATE $x_{t+1} \leftarrow (1-\gamma_{t}) x_{t} + \gamma_{t} v_t$
\ENDFOR
\end{algorithmic}\end{algorithm}

The accuracy of the gradient estimator
in Equation~\eqref{eq:grad_averaging} is
measured via the second moment, which unlike the variance 
(which is the second \emph{central} moment)
includes the bias of the estimator.
In the following lemma we derive an upper bound on the  second moment
conditioned on the run of the algorithm
until the end of computing the current iterate \(x_{t}\).
This bound is based on the following simple but useful chain of
inequalities: a triangle inequality followed by smoothness:
\begin{equation*}
  \begin{split}
\norm{\widetilde{\nabla} f (x_{t-1}) - \nabla f(x_{t})} 
&\leq \norm{\widetilde{\nabla} f (x_{t-1}) - \nabla f(x_{t-1})}
+ \norm{\nabla f(x_{t-1}) - \nabla f(x_{t})} \\
&\leq \norm{\widetilde{\nabla} f (x_{t-1}) - \nabla f(x_{t-1})}
+ L D \gamma_{t-1},
\end{split}
\end{equation*}
This leads to the accuracy estimate
\begin{equation}
  \label{eq:bound_on_gradient_error}
 \begin{split}
  \E(t){\norm{\widetilde{\nabla} f (x_{t}) - \nabla f (x_{t})}^{2}}
  &
  =
  (1 - \rho_{t})^{2}
  \norm{\widetilde{\nabla} f (x_{t-1}) - \nabla f(x_{t})}^{2}
  +
  \rho_{t}^{2} \Var(t){\nabla \tilde{F} (x_{t}, z)}
  \\
  &
  \leq
  (1 - \rho_{t})^{2} \left(
    \norm{\widetilde{\nabla} f (x_{t-1}) - \nabla f(x_{t-1})}
    + L D \gamma_{t-1}
  \right)^{2}
  +
  \rho_{t}^{2} \sigma^{2}
  .
 \end{split}
\end{equation}
Under \(\gamma_{t} = \Theta(1/t)\)
the best achievable bound on the second moment from the above calculation is
\(\E{\norm{\widetilde{\nabla} f (x_{t}) - \nabla f (x_{t})}^{2}}
= \mathcal{O}(t^{- 2/3})\)
with the choice \(\rho_{t} = \Theta(t^{-2/3})\).
To simplify analysing \eqref{eq:bound_on_gradient_error},
one can derive the following simpler estimate to the second moment,
useful when \(\rho_{t}\) is not too small,
see \citet[Lemma~1]{mokhtari2017large}.

\begin{lemma}\label{lemma:bound_on_gradient_error_10}
  Let \(f\) be an \(L\)-smooth function over a compact convex set
  \(\mathcal{X}\) of diameter at most \(D\).
  If the stochastic gradient samples of \(f\)
  have variance at most \(\sigma^{2}\),
  then the squared gradient errors
  \smash{$\norm{\nabla f(x_t) - \widetilde{\nabla}f(x_{t})}^{2}$}
  for the iterates generated by Momentum SFW
  (Algorithm~\ref{algo_MSFW}) satisfy
  \begin{multline}
    \label{eq:bound_on_gradient_error_10}
    \E(t){\norm{\widetilde{\nabla} f (x_{t}) - \nabla f (x_{t})}^{2}}
   \\ 
    \leq
    \left(1 - \frac{\rho_{t}}{2}\right) 
      \norm{\widetilde{\nabla} f (x_{t-1}) - \nabla f(x_{t-1})}^2
      +\frac{2 L^2 D^2 \gamma_{t-1}^2}{\rho_t}
    +
    \rho_{t}^{2} \sigma^{2}
    .
  \end{multline}
\end{lemma}

Combining the lemma with the proof technique for
Theorem~\ref{th:sfw:expectation} leads to the following convergence
rate, where parameter choices have been optimized for the proof,
while being function-agnostic,
i.e., deliberately not depending on problem parameters
like the variance bound \(\sigma\) or the smoothness \(L\).
For a detailed proof, see \citet{mokhtari2020stochastic}.

\begin{theorem} \label{thm:rate_convex}
  For an \(L\)-smooth convex function \(f\) over a compact convex set $\mathcal{X}$ with
  diameter at most \(D\),
  and stochastic gradients with variance at most \(\sigma^{2}\),
  the Momentum Stochastic Frank–Wolfe algorithm
  (Algorithm~\ref{algo_MSFW}) with $\gamma_{t} = 2 / (t+7)$
  and $\rho_{t} = 4 / (t+8)^{2/3}$, ensures for all $t\geq 0$,
  \begin{equation}
    \label{eq:rate}
  \E{f(x_{t})} -f(x^*) \leq \frac{Q}{(t+9)^{1/3}},
\end{equation}
where 
  \begin{equation*}
    Q \defeq \max\left\{9^{1/3}\bigl(f(x_{0}) -f(x^{*})\bigr) \ ,
      \frac{LD^{2}}{2} + 2 D \max \left\{
        3 \norm{\nabla f(x_{0})},
        \left( 16 \sigma^{2} + 2 L^{2} D^{2} \right)^{1/2}
      \right\}
    \right\}.
  \end{equation*}
  Equivalently,
  for an expected primal gap at most $\varepsilon>0$,
  the Momentum Stochastic Frank–Wolfe algorithm
  (Algorithm~\ref{algo_MSFW})
  computes
  $\mathcal{O}((D \max\{ \norm{\nabla f(x_{0})}, \sigma, L D \}
  / \varepsilon)^{3})$ stochastic gradients
  and
  linear minimizations.
\end{theorem}

\subsection{Estimating the difference of gradients}
\label{sec:stoc-FW-grad-diff}

In this section, we consider estimators for the \emph{difference}
of gradients, e.g., at two consecutive iterates,
which is expected to vary much less than a single gradient. 
This estimator is the basis of the algorithms that we will describe shortly. 
The simplest unbiased gradient estimator of
\(\nabla f(x_{t}) - \nabla f(x_{t-1})\) is
\(\nabla \tilde{F} (x_{t}, z) - \nabla \tilde{F} (x_{t - 1}, z)\), 
which deliberately uses the \emph{same} sample \(z\)
for both stochastic gradients.
In this section,
we assume that \(\tilde{F}\) is \(L\)-smooth in its first argument.
Thus the relation
\(\norm{\nabla \tilde{F} (x_{t}, z)
	- \nabla \tilde{F} (x_{t - 1}, z)}
\leq L \norm{x_{t} - x_{t-1}}\)
already shows that the estimator is reasonably small,
which is even stronger than having small variance.
Here the assumption is that as the algorithm converges,
the difference \(\norm{x_{t} - x_{t-1}}\) between successive iterates
converges to \(0\).
This argument already highlights
that in order to obtain 
a highly accurate difference estimator 
it suffices to have access to a simple unbiased stochastic gradient without any additional
accuracy requirements.
While beyond the scope of this survey, we note
that the above estimator is not appropriate for
the non-oblivious setting,
when the distribution of \(Z\)
used for sampling the stochastic gradient \(\nabla \tilde{F}(x, Z)\)
depends on \(x\); for this case \citet{hassani2019stochastic} provide
a drop-in replacement estimator
based on a second-order stochastic oracle.

From the difference estimator we then obtain the gradient estimator
\begin{equation}
  \label{eq:diff_update}
	\widetilde{\nabla} f (x_{t}) = \widetilde{\nabla} f (x_{t-1})
	+ \nabla \tilde{F} (x_{t}, z) - \nabla \tilde{F} (x_{t - 1}, z)
	,
\end{equation}
which has the drawback of accumulating variance,
due to the law of total variance
\begin{multline}
  \label{eq:variance_diff_update}
	\Var{\widetilde{\nabla}f(x_{t}) - \nabla f(x_{t})}
  \\
	= \Var{\widetilde{\nabla}f(x_{t-1}) - \nabla f(x_{t-1})}
	+ \E{\Var(t){\nabla \tilde{F} (x_{t}, z)
			- \nabla \tilde{F} (x_{t - 1}, z)}}
	.
\end{multline}

Therefore, it is necessary from time to time to use a highly accurate
gradient estimator instead, to prevent the variance from accumulating
indefinitely; this approach is sometimes referred to as \emph{check-pointing}. We denote by \(s_{0}, s_{1}, \dotsc\) the iterations when
this is performed.

\subsubsection{The SPIDER Frank–Wolfe algorithm}

The Stochastic Path-Integrated Differential Estimator (SPIDER)
algorithm from \citet{fang2018spider}
is a stochastic gradient descent algorithm that uses this estimator
and achieves optimal stochastic optimization complexity
of \(\mathcal{O}(1 / \varepsilon^{2})\) stochastic oracle calls
for an expected primal gap error of \(\varepsilon > 0\)
(see Section~\ref{sec:stoch-FOO-complexity}).
Its
conditional gradients analogue is
the SPIDER Frank–Wolfe algorithm (Algorithm~\ref{algo_FW++})
\citep{yurtsever2019conditional}.
A variant of the algorithm appears in \citet{shen2019complexities}
using exact gradients in Line~\ref{line:SPIDER:AccurateEstimation},
i.e., \(\widetilde{\nabla} f(x_{t}) \gets \nabla f(x_{t})\).
For a variant using second-order methods, see \citet{hassani2019stochastic}.

\begin{algorithm}[t]
  \caption{SPIDER Frank–Wolfe (SPIDER FW)
    \citep{yurtsever2019conditional}}
	\label{algo_FW++}
	\begin{algorithmic}[1]
    \REQUIRE Start point \(x_{0} \in \mathcal{X}\),
      batch size $b_{t}$, step size $\gamma_{t}$
    \ENSURE Iterates $x_{1}, \dotsc \in \mathcal{X}$
    \FOR{\(k = 0\) \TO \dots}
      \FOR{$t = s_{k}$ \TO \(s_{k+1} - 1\)}
        \STATE Draw $\{z_1,\dotsc,z_{b_t}\}$, that is,
          $b_{t}$ independent realizations of $Z$.
        \IF{\(t = s_{k}\)}
          \STATE \label{line:SPIDER:AccurateEstimation}
            \(\widetilde{\nabla} f(x_{t}) \gets
            \frac{1}{b_{t}} \sum_{j=1}^{b_{t}}
            \nabla \tilde{F} (x_t,z_j)\)
        \ELSE
          \STATE
            \(\widetilde{\nabla} f(x_{t}) \gets
            \widetilde{\nabla} f(x_{t-1})
            + \frac{1}{b_{t}} \sum_{j=1}^{b_{t}} \left(
              \nabla \tilde{F} (x_{t}, z_{j})
              - \nabla \tilde{F} (x_{t - 1}, z_{j})
            \right)\)
        \ENDIF
        \STATE \(v_{t} \gets
          \argmin_{v \in \mathcal{X}}
          \innp{\widetilde{\nabla} f (x_{t})}{v}\)
        \STATE \(x_{t+1} \gets x_{t} + \gamma_{t} (v_{t} - x_{t})\)
      \ENDFOR
    \ENDFOR
  \end{algorithmic}
\end{algorithm}

\begin{theorem}	\label{thm_SPIDER_convex}
  Let the stochastic function \(\tilde{F}(\cdot, z)\)
  be \(L\)-smooth in its first argument,
  over a compact convex set \(\mathcal{X}\) with diameter at most \(D\).
  Let \(f\) be the expectation of \(\tilde{F}(\cdot, z)\) in \(z\)
  with stochastic gradients having variance at most \(\sigma^{2}\).
  Then
  the SPIDER Frank–Wolfe algorithm (Algorithm~\ref{algo_FW++})
  with $\gamma_{t} = 2 / (t + 2)$,
  \(s_{k} = 2^{k} - 1\),
  $b_{t} = 6 (t + 1)$ for \(t \neq s_{k}\) and
  $b_{t} = (\sigma^{2} (t+1)^{2}) / (L^{2} D^{2})$ for \(t = s_{k}\),
  satisfies for all \(t \geq 1\):
	\begin{equation*}
    \E{f(x_t)} - f(x^*)
    \leq \frac{(2 \sqrt{5} + 4) L D^{2}}{t+2}.
	\end{equation*}
  Equivalently,
  to reach an expected primal gap of at most $\varepsilon > 0$,
  the algorithm computes
  \(\mathcal{O}(((L D^{2})^{2} + \sigma^{2} D^{2}) / \varepsilon^{2})\)
  stochastic gradients
  and
  \(\mathcal{O}(L D^{2} / \varepsilon)\)
  linear minimizations.
\end{theorem}
\begin{proof}
We again estimate the error of the gradient estimator.
Due to smoothness of \(\widetilde{F}\),
we have for any iteration \(t\) not of the form \(s_{k}\) that
\(\Var(t){\nabla \tilde{F} (x_{t}, z)
  - \nabla \tilde{F} (x_{t - 1}, z)} \leq (L \norm{x_{t} - x_{t-1}})^{2}
\leq (L D \gamma_{t-1})^{2}\).
Combining this with~\eqref{eq:variance_diff_update},
we obtain for \(t \neq s_{k}\)
\begin{equation}
 \begin{split}
  \E{\norm{\widetilde{\nabla}f(x_{t}) - \nabla f(x_{t})}^{2}}
  &
  \leq
  \E{\norm{\widetilde{\nabla}f(x_{t-1}) - \nabla f(x_{t-1})}^{2}}
  + \frac{(L D \gamma_{t-1})^{2}}{b_{t}}
  \\
  &
  =
  \E{\norm{\widetilde{\nabla}f(x_{t-1}) - \nabla f(x_{t-1})}^{2}}
  + \frac{2 L^{2} D^{2}}{3 (t+1)^{3}}
  .
 \end{split}
\end{equation}
Let us consider an iteration \(t\) with \(s_{k} \leq t < s_{k+1}\).
Using the recursion above we obtain
{\allowdisplaybreaks
\begin{align*}
  \E{\norm{\widetilde{\nabla}f(x_{t}) - \nabla f(x_{t})}^{2}}
  &
  \leq
  \E{\norm{\widetilde{\nabla}f(x_{s_{k}}) - \nabla f(x_{s_{k}})}^{2}}
  +
  \sum_{\tau=s_{k} + 1}^{t} \frac{2 L^{2} D^{2}}{3 (\tau + 1)^{3}}
  \\
  &
  \leq
  \frac{\sigma^{2}}{b_{s_{k}}}
  + \sum_{\tau=s_{k} + 1}^{t}
  \left(
    \frac{L^{2} D^{2}}{3 \tau^{2}}
    - \frac{L^{2} D^{2}}{3 (\tau + 1)^{2}}
  \right)
  \\
  &
  \leq
  \frac{L^{2} D^{2}}{(s_{k} + 1)^{2}}
  + \frac{L^{2} D^{2}}{3 (s_{k} + 1)^{2}}
  - \frac{L^{2} D^{2}}{3 (t + 1)^{2}}
  \\
  &
  \leq
  \frac{L^{2} D^{2}}{((t + 2) / 2)^{2}}
  + \frac{L^{2} D^{2}}{3 ((t + 2) / 2)^{2}}
  - \frac{L^{2} D^{2}}{3 (t + 2)^{2}}
  =
  \frac{5 L^{2} D^{2}}{(t + 2)^{2}}
  .
\end{align*}
}
Now Lemma~\ref{lem:stochastic-rate} with \(G = \sqrt{5} L D / 2\)
proves the claim.
\end{proof}

\begin{remark}
  For the finite-sum optimization problem in
  Equation~\eqref{pb:finite}, one can show that
  the SPIDER SFW algorithm
  reaches an expected primal gap of at most $\varepsilon > 0$,
after computing
  \(\mathcal{O}(m \log (\frac{L D^{2}}{\varepsilon}) + \frac{ L^2D^{4} }{\varepsilon^{2}})\)
  stochastic gradients
  and
  \(\mathcal{O}(\frac{L D^{2}}{\varepsilon})\)
  linear minimizations.
  See \citet{yurtsever2019conditional} for more details.
\end{remark}

\subsubsection{Stochastic Variance-Reduced Frank–Wolfe algorithm}

The Stochastic Variance-Reduced Frank–Wolfe algorithm (SVRF,
Algorithm~\ref{svrf}) \citep{svrf16} is the conditional gradient
counterpart of SVRG in
\citet{svrg13}, which mainly differs from the SPIDER algorithm by the 
fact that it estimates the difference of the gradient
to the previous \emph{highly accurate} gradient, as opposed
to the last (possibly less accurate) gradient estimate.
For convenience, we use exact gradients as opposed to
highly accurate
gradients, which are typically only realistic
in the finite-sum setting.
As a consequence of this simplification,
the convergence result below needs
no assumption on the variance of stochastic gradients
\(\nabla \tilde{F} (x, Z)\),
highlighting (once more) that the assumption is unnecessary
to obtain low-variance estimators for difference of gradients.

\begin{algorithm}[h]
\caption{Stochastic Variance-Reduced Frank–Wolfe (SVRF) \citep{svrf16}}
\label{svrf}
\begin{algorithmic}[1]
\REQUIRE Start point $x_0\in\mathcal{X}$, snapshot times $s_k<s_{k+1}$
  with $s_0=0$, batch sizes $b_{t} \in \mathbb{Z}_{+}$
  step-sizes $ \gamma_{t} $
\ENSURE Iterates \(x_{1}, \dotsc\)
\FOR{\(k=0\) \TO \dots}
  \FOR{$t=s_{k}$ \TO $s_{k+1} - 1$}
    \IF{$t = s_{k}$}
      \STATE \label{svrf:snap}
        $\widetilde{\nabla}f(x_{t}) \gets \nabla f(x_{t})$
    \ELSE
      \STATE Draw $b_t$ independent realizations of \(Z\), namely,
        \(z_{1}, \dotsc, z_{b_{t}}\).
      \STATE\label{svrf:est}
        $\widetilde{\nabla}f(x_{t}) \gets\nabla f(x_{s_{k}})
        + \frac{1}{b_{t}} \sum_{j=1}^{b_{t}}
        (\nabla \tilde{F}(x_{t}, z_{j})
        - \nabla \tilde{F}(x_{s_{k}}, z_{j}))$
    \ENDIF
    \STATE$v_{t} \gets \argmin_{v\in\mathcal{X}}
      \innp{\widetilde{\nabla}f(x_{t})}{v}$
\STATE$x_{t+1}\leftarrow x_t+\gamma_t(v_t-x_t)$
\ENDFOR
\ENDFOR
\end{algorithmic}
\end{algorithm}

We follow \citet[Lemma~D.6 and Theorem~3.5]{combettes20adasfw}
for presenting a convergence rate which differs from the original one in
\citet{svrf16}.
More precisely, we use the common, function-agnostic step size
\(\gamma_{t} = 2 / (t + 2)\) instead of
the original \(\gamma_{t} = 2 / (t + 2 - s_{k})\),
for more consistency with other conditional gradients algorithms.
This makes the convergence smoother by avoiding large step sizes in late
iterations. The core part of the convergence bound is
that variance of the gradient estimator is upper bounded by the
\emph{primal gap} (of both the current iterate and the last iterate with
known exact gradient),
ensuring decrease of variance even without increasing the batch size.
The bound is formulated in the following lemma, which is
based on Lemma~\ref{lem:smooth-grad-diff},
requiring that the objective function is smooth even outside the
feasible region.

\begin{lemma}
  \label{lem:svrf}
  Assume that \(\tilde{F}\) is \(L\)-smooth in its first argument
  on the whole vector space \(\mathbb{R}^{n}\)
  containing a closed convex feasible region \(\mathcal{X}\)
  (which need not be bounded).
  Then the iterates of SVRF (Algorithm~\ref{svrf})
  for all $t\geq 0$ with \(s_{k} \leq t < s_{k+1}\) satisfy
 \begin{equation*}
  \E{\norm{\widetilde{\nabla}f(x_t)-\nabla f(x_t)}^2}
  \leq
  \frac{4L}{b_{t}}
  \left(
    \E{f(x_{t}) - f(x^{*})} + \E{f(x_{s_{k}}) - f(x^{*})}
  \right)
  .
 \end{equation*}
\end{lemma}

As usual, the variance bound leads to an overall convergence rate.

\begin{theorem}
  \label{th:svrf}
  Assume \(\tilde{F}\) is \(L\)-smooth in its first argument
  on the whole vector space \(\mathbb{R}^{n}\)
  containing the compact convex domain \(\mathcal{X}\),
  which has diameter at most \(D\).
  Consider SVRF (Algorithm~\ref{svrf}) with
  $s_{k} = 2^{k}-1$, $b_{t} = 48(t+2)$,
  and $\gamma_{t} = 2 / (t+2)$.
  Then for all $t\geq 1$,
 \begin{equation}
  \E{f(x_t)}-f(x^*)
  \leq\frac{4LD^2}{t+2}.\label{th:svrf:res}
 \end{equation}
 Equivalently, for an expected primal gap
 \(\E{f(x_{t}) - f(x^{*})} \leq \varepsilon\)
 the algorithm computes
 \(\mathcal{O}((L D^{2} / \varepsilon)^{2})\)
 stochastic gradients,
 \(\mathcal{O}(\log (L D^{2} / \varepsilon))\)
 exact gradients
 and performs
 \(\mathcal{O}(L D^{2} / \varepsilon)\)
 linear minimizations.
\end{theorem}

\begin{proof}
We proceed by induction using Lemma~\ref{lem:stochastic-rate}
with \(G = L D / 2\):
Suppose that the claim holds for all \(1 \leq t'\leq t\).
We prove Equation~\eqref{eq:stochastic-gradient-rate}
so that Lemma~\ref{lem:stochastic-rate} provide the claim for \(t+1\).
There exist \(k \geq 0\) such that
\(s_{k} \leq t \leq s_{k+1} - 1\),
in particular \(t + 2 \leq 2(s_{k} + 2)\)
using \(s_{k+1} = 2 s_{k} + 1\).
For \(t > s_{k}\) (and hence \(k > 0\))
we have by induction and Lemma~\ref{lem:svrf}
\begin{equation*}
  \begin{split}
 \E{\norm{\widetilde{\nabla}f(x_t)-\nabla f(x_t)}^2}
 &\leq\frac{4L}{b_t}\left(\E{f(x_t)
     - f(x^{*})}
   + \E{f(x_{s_{k}}) - f(x^{*})}
 \right)\\
 & \leq\frac{4L}{b_t}\left(\frac{4LD^2}{t+2}+\frac{4LD^2}{s_k+2}\right)
  \leq\frac{4L}{b_t}\left(\frac{4LD^2}{t+2}+\frac{8LD^2}{t+2}\right)\\
 &=\frac{48L^2D^2}{b_t(t+2)}
 =\left(\frac{LD}{t+2}\right)^2.
 \end{split}
\end{equation*}
Thus \(\E{\norm{\widetilde{\nabla}f(x_{t}) - \nabla f(x_{t})}^{2}}
\leq (LD / (t + 2))^{2}\),
which clearly also holds for \(t = s_{k}\)
(as then \(\widetilde{\nabla}f(x_{t}) = \nabla f(x_{t})\)).
This proves Equation~\eqref{eq:stochastic-gradient-rate}
with \(G = L D / 2\), as claimed.
\end{proof}

\begin{remark}
For the finite-sum optimization problem in \eqref{pb:finite}, one can show that the SVRF algorithm reaches an expected primal gap of at most $\varepsilon > 0$, 
after computing
  \(\mathcal{O}(m \log (\frac{L D^{2}}{\varepsilon}) + \frac{ L^2D^{4} }{\varepsilon^{2}})\)
  stochastic gradients
  and
  \(\mathcal{O}(\frac{L D^{2}}{\varepsilon})\)
  linear minimizations. See \citet{svrf16} for more details.
\end{remark}

\subsubsection{One-Sample Stochastic Frank–Wolfe algorithm}

In this section, we provide the momentum version of the
SPIDER Frank–Wolfe algorithm (Algorithm~\ref{algo_FW++}),
to use only a single stochastic gradient per iteration
under the same convergence rate.
The resulting algorithm is the
One-Sample Stochastic Frank–Wolfe algorithm (1-SFW)
\citep{zhang2020one} (Algorithm~\ref{algo_1SFW}).
The algorithm is the conditional gradient equivalent
of the stochastic gradient variant 
Stochastic Recursive Momentum (STORM) \citep{cutkosky2019momentum}.
For more details see \citet{zhang2020one},
including generalizing the non-oblivious SPIDER algorithm, where
for the stochastic gradient \(\tilde{F}(x, Z)\),
the probability distribution of $Z$ depends on the variable $x$, too.

\begin{algorithm}[t]
  \caption{One-Sample Stochastic Frank–Wolfe (1-SFW) \citep{zhang2020one}}\label{algo_1SFW}
	\begin{algorithmic}[1]
    \REQUIRE Start points $x_0 = x_{-1} \in \mathcal{X}$,
      number of iterations $T$,
      step sizes \(\gamma_{t}\),
      and momentum terms $\rho_{t}$.
      \ENSURE Iterates $x_1, \dotsc,x_T\in \mathcal{X}$
			\FOR{$t=0$ \TO \dots}
			\STATE Draw $z$, an independent realization of $Z$.
			\STATE $\widetilde{\nabla}f(x_t) \leftarrow (1-\rho_t) (\widetilde{\nabla}f(x_{t-1})  -\nabla \tilde{F}(x_{t-1},z) )+ \nabla \tilde{F}(x_t,z) $ \label{alg:1SFW:grad_estimation}
			\STATE$v_t\leftarrow\argmin_{v\in\mathcal{X}}\innp{\widetilde{\nabla}f(x_t)}{v}$
			\STATE $x_{t+1} \gets x_{t} + \gamma_{t} (v_{t} - x_{t})$
      \ENDFOR
\end{algorithmic}\end{algorithm}

The 1-SFW algorithm builds on the Momentum SFW (Algorithm~\ref{algo_MSFW}) gradient estimator, but slightly differs form that to remove the bias in the gradient estimation direction of Momentum SFW. More precisely, in 1-SFW we add a correction term to the estimator of Momentum SFW, similar to the idea in the 
SPIDER estimator, to obtain the following \emph{unbiased} gradient estimator
\begin{equation}
  \label{eq:grad_averaging_2}
 \begin{split}
  \widetilde{\nabla}f(x_t) & \defeq
  (1 - \rho_{t}) \left(
    \widetilde{\nabla}f(x_{t-1}) +
    \nabla \tilde{F}(x_{t}, z) - \nabla \tilde{F}(x_{t-1}, z)
  \right)
  + \rho_{t} \nabla \tilde{F}(x_{t}, z) \\
  & = (1-\rho_t) (\widetilde{\nabla}f(x_{t-1})  -\nabla \tilde{F}(x_{t-1},z) )+ \nabla \tilde{F}(x_t,z),
 \end{split}
\end{equation}
where \(\rho_{t}\) is the momentum parameter which combines
the previous gradient estimator with the current stochastic gradient.
The following bound on the accuracy of the gradient estimator
combines the estimations for
SPIDER and Momentum SFW:
\begin{equation}
 \begin{split}
  \MoveEqLeft
  \E(t){\norm{\widetilde{\nabla} f (x_{t}) - \nabla f (x_{t})}^{2}}
  \\
  &
  =
  (1 - \rho_{t})^{2}
  \norm{\widetilde{\nabla} f (x_{t-1}) - \nabla f(x_{t-1})}^{2}
  +
  \Var(t){(1 - \rho_{t}) \nabla \tilde{F} (x_{t-1}, z)
    + \nabla \tilde{F} (x_{t}, z)}
  \\
  &
  \leq
  \begin{aligned}[t]
  (1 - \rho_{t})^{2}
  \norm{\widetilde{\nabla} f (x_{t-1}) - \nabla f(x_{t-1})}^{2}
    &
  +
  2 (1 - \rho_{t})^{2} \Var(t){\nabla \tilde{F} (x_{t-1}, z)
    - \nabla \tilde{F} (x_{t}, z)}
    \\
    &
  + 2 \Var(t){\nabla \tilde{F} (x_{t}, z)}
  \end{aligned}
  \\
  &
  \leq
  (1 - \rho_{t})^{2}
  \norm{\widetilde{\nabla} f (x_{t-1}) - \nabla f(x_{t-1})}^{2}
  +
  2 (1 - \rho_{t})^{2} (L D \gamma_{t-1})^{2}
  + 2 \rho_{t}^{2} \sigma^{2}
  .
 \end{split}
\end{equation}
Compared to Equation~\eqref{eq:bound_on_gradient_error_10},
the missing summand
\(\norm{\widetilde{\nabla} f (x_{t-1}) - \nabla f(x_{t-1})}
\cdot L D \gamma_{t-1}\)
is the major difference,
improving the upper bound on gradient estimator, namely, to
\(\E{\norm{\widetilde{\nabla} f (x_{t}) - \nabla f (x_{t})}^{2}}
= \mathcal{O}(1/t)\)
for the optimal choice of parameters
\(\rho_{t} = \Theta(1/t)\) and \(\gamma_{t} = \Theta(1/t)\).
This leads to the following convergence rates with a parameter-free
choice.

\begin{theorem} \label{thm:rate_convex_1_sfw}
  Let the stochastic function \(\tilde{F}(\cdot, z)\)
  be \(G\)-Lipschitz and \(L\)-smooth in its first argument,
  over a compact convex set \(\mathcal{X}\) of diameter at most \(D\). Further, 
  let \(f\) be the expectation of \(\tilde{F}(\cdot, z)\) in \(z\),
  and let the stochastic gradients of \(f\) have bounded variance
  at most $\sigma^{2}$.
 \begin{enumerate}
 \item \citep[Theorem~4 (1)]{zhang2020one}
  If \(f\) is convex,
  then the One-Sample Stochastic Frank–Wolfe algorithm
  (Algorithm~\ref{algo_1SFW})
  with $\gamma_{t} = 1 / (t + 1)$ and $\rho_{t} = 1/t$
  ensures that the output $x_{t}$ satisfies
  \begin{equation}\label{eq:rate_1_sfw}
    \E{f(x_{t})} -f(x^*) \leq \mathcal{O}\left(\frac{1}{\sqrt{t}} \right).
  \end{equation}
  Equivalently, the expected primal gap is at most \(\varepsilon > 0\)
  after \(\mathcal{O}(1
  / \varepsilon^{2})\) stochastic gradients
  and
  linear minimizations.
 \item \citep[Theorem~4 (2)]{zhang2020one}
  If \(f\) is not necessarily convex,
  then the output of One-Sample Stochastic Frank–Wolfe algorithm
  (Algorithm~\ref{algo_1SFW})
  with $\gamma_t=1/T$ and $\rho_t=t^{-2/3}$
  satisfies after \(T\) iterations:
  \begin{equation}
    \label{eq:rate_1_sfw_ncvx}
    \frac{1}{T} \sum_{t=1}^{T} \E{g(x_{t}) }
    \leq
    \mathcal{O}\left(\frac{1}{T^{1/3}}\right).
  \end{equation}
  Equivalently, to achieve an expected primal gap
  at most \(\varepsilon > 0\) for a fixed \(\varepsilon\),
  the algorithm performs
  \(\mathcal{O}(1/ \varepsilon^{3})\) stochastic gradients
  and linear minimizations.
 \end{enumerate}
\end{theorem}

\subsection{Stochastic Conditional Gradient Sliding algorithm}
\label{sec:SCGS}

The conditional gradient sliding (CGS) algorithm presented in Section~\ref{cgs} can be 
extended quite naturally to the stochastic setting by substituting the
exact gradient in the algorithm with an unbiased estimator of the
gradient using SFOO calls, this results
in the Stochastic Conditional Gradient Sliding algorithm (SCGS)
\citep{lan2016conditional} (Algorithm~\ref{scgs}).
Also, see \citet{nonconvexCGS2017} for convergence rates for non-convex
objective functions.
\index{convergence for non-convex objective}

\begin{algorithm}[h]
\caption{Stochastic Conditional Gradient Sliding (SCGS) \citep{lan2016conditional}}
\label{scgs}
\begin{algorithmic}[1]
  \REQUIRE Start point $x_{0}\in\mathcal{X}$,
    maximum number of iterations $T \in \mathbb{Z}_{+}$,
    step-sizes $0 \leq \gamma_{t} \leq 1$,
    learning rates $\eta_{t}>0$,
    accuracies $\beta_{t} \geq 0$,
    batch sizes $b_{t} \in \mathbb{Z}_{+}$
  \ENSURE Iterates $x_{1}, \dotsc, x_{T} \in \mathcal{X}$
\STATE$y_0\leftarrow x_0$
\FOR{$t=0$ \TO \dots}
\STATE$w_{t} \gets (1 - \gamma_{t}) x_{t} + \gamma_{t} y_{t}$
\STATE Draw $z_{1}, \dotsc, z_{b_{t}}$, that is,
  $b_{t}$ independent realizations of $Z$.
\STATE$\widetilde{\nabla}f(w_t) \gets
  \frac{1}{b_t} \sum_{i=1}^{b_t} \widetilde{\nabla}F(w_t, z_i)$
\STATE$y_{t+1} \gets
  \operatorname{CG}(\widetilde{\nabla}f(w_{t}), y_{t},
  \eta_{t}, \beta_{t})$
  \COMMENT{See Algorithm~\ref{alg:CG} for \(\operatorname{CG}\).}
\STATE$x_{t+1} \gets x_{t} + \gamma_{t} (y_{t+1} - x_t)$
\ENDFOR
\end{algorithmic}
\end{algorithm}
	
\begin{theorem}
  Let \(f\) be a convex \(L\)-smooth function
  over a compact convex set $\mathcal{X}$ of diameter at most \(D\).
  If stochastic gradient samples have variance at most \(\sigma^{2}\),
  then the Stochastic Conditional Gradient Sliding algorithm
  (Algorithm~\ref{scgs})
  with $\gamma_{t} = 3/(t+3)$, $\eta_{t} = 4 L / (t+3)$,
  $\beta_{t} = L D^{2} / ((t+1)(t+2))$,
  and $b_{t} = \sigma^{2} (t+3)^{3} / (LD)^{2}$
  ensures for all $1\leq t\leq T$:
 \begin{equation*}
   \E{f(x_{t})} - f(x^{*}) \leq \frac{7.5 L D^{2}}{(t+1)(t+2)}.
 \end{equation*}
 Thus, for a primal gap error of at most $\varepsilon>0$
 in expectation, the Stochastic Conditional Gradient
 Sliding (SCGS) algorithm (Algorithm~\ref{scgs}) computes
 \smash[b]{$\mathcal{O}(\sqrt{\frac{L D^{2}}{\varepsilon}}
 +\frac{\sigma^2D^2}{\varepsilon^2})$}
 stochastic gradients and performs
 $\mathcal{O}(\frac{LD^2}{\varepsilon})$
 linear minimizations over $\mathcal{X}$.
\end{theorem}

We refer the interested reader to \citet{lan2016conditional} for a
large deviation analysis of SCGS and an application of CGS to solving
nonsmooth saddle point problems. CGS and SCGS can also be lazified as
shown in \citet{lan2017conditional}.

\subsection{Frank–Wolfe with adaptive gradients}
\label{sec:adasfw}

Inspired by the Adaptive Gradient algorithm (AdaGrad)
\citep{duchi11,mcmahan10}, a method setting entry-wise step sizes that
automatically adjust to the geometry of the problem,
\citet{combettes20adasfw} show that a projection-free variant is also
very performant. AdaGrad is presented in Algorithm~\ref{adagrad}.
In Line~\ref{adagrad:matrix} an offset hyperparameter $\delta$
is added to ensure that the matrix $H_{t}$ is non-singular.
We denote by $\norm[H]{x}=\sqrt{\innp{x}{Hx}}$
the norm induced by a symmetric positive definite matrix $H$.

\begin{algorithm}[h]
\caption{Adaptive Gradient (AdaGrad) \citep{duchi11,mcmahan10}}
\label{adagrad}
\begin{algorithmic}[1]
  \REQUIRE Start point $x_0\in\mathcal{X}$, offset $\delta>0$,
    learning rate $\eta>0$
  \ENSURE Iterates $x_1, \dotsc \in \mathcal{X}$
\FOR{$t=0$ \TO \dots}
\STATE Update the gradient estimator $\widetilde{\nabla}f(x_t)$.
\STATE \label{adagrad:matrix}
  $H_t\leftarrow\operatorname{diag} \left(
    \delta +
    \sqrt{\sum_{s=0}^{t} \widetilde{\nabla} f(x_{s})_{i}^{2}}
    : i=1, \dots, n
  \right)$
\STATE\label{adagrad:new}
  $x_{t+1} \gets \argmin_{x \in \mathcal{X}}
  \eta \innp{\widetilde{\nabla}f(x_{t})}{x}
  + \frac{1}{2} \norm[H_{t}]{x - x_{t}}^{2}$
\ENDFOR
\end{algorithmic}
\end{algorithm}

Thanks to the adaptive step sizes, AdaGrad is particularly
efficient for learning models with sparse features, e.g., in natural language processing (NLP) tasks. 
However, it needs to solve a quadratic subproblem at each iteration
(Line~\ref{adagrad:new}), which can become quite expensive overall and
inefficient for solving the constrained optimization
problem~\eqref{pb:finite}. 
This may also be why the successful applications of AdaGrad are on unconstrained optimization problems.
In order to make AdaGrad more efficient on constrained optimization
problems, \citet{combettes20adasfw} proposes to solve these subproblems
\emph{very inaccurately}, via a small and fixed number of
iterations of the Frank–Wolfe algorithm.
This contrasts CGS
(Section~\ref{cgs}), where a quadratic subproblem is solved to a high
accuracy. The method is presented in Template~\ref{adafw} via a
generic template, where the gradient can be estimated as in SFW, SVRF,
or CSFW and the matrix $H_t$ can be updated as in AdaGrad.

\begin{template}[h]
\caption{Frank–Wolfe with adaptive gradients \citep{combettes20adasfw}}
\label{adafw}
\begin{algorithmic}[1]
  \REQUIRE Start point $x_{0}\in\mathcal{X}$,
    batch sizes \(b_{t}\),
    bounds $0 < \lambda_{t}^{-} \leq
    \lambda_{t+1}^{-} \leq \lambda_{t+1}^{+}
    \leq \lambda_{t}^{+}$,
    number of inner iterations $K$,
    learning rates $\eta_{t} > 0$,
    step size bounds $0 \leq \gamma_{t} \leq 1$
  \ENSURE Iterates $x_1, \dotsc \in \mathcal{X}$
  \FOR{$t=0$ \TO \dots}
\STATE Update the gradient estimator $\widetilde{\nabla}f(x_t)$.
\STATE Update the diagonal matrix $H_t$ and clip its entries to $[\lambda_t^-,\lambda_t^+]$
\STATE$y_0^{(t)}\leftarrow x_t$\label{sub:start}
\FOR{$k=0$ \TO $K-1$}
  \STATE
    \(d_{t} \gets \widetilde{\nabla}f(x_{t})
    + \frac{1}{\eta_{t}} H_{t} (y_{k}^{(t)} - x_{t})\)
  \STATE \(v_{k}^{(t)} \gets
    \argmin_{v \in \mathcal{X}} \innp{d_{t}}{v}\)
  \STATE\label{adafw:gamma}
    \(\gamma_{k}^{(t)} \gets \min\left\{
      \eta_{t} \frac{\innp{d_{t}}{y_{k}^{(t)} - v_{k}^{(t)}}}{%
        \norm[H_{t}]{y_{k}^{(t)} - v_{k}^{(t)}}^{2}},
      \gamma_{t}
    \right\}\)
\STATE$y_{k+1}^{(t)}\leftarrow y_k^{(t)}+\gamma_k^{(t)}(v_k^{(t)}-y_k^{(t)})$
\ENDFOR\label{sub:end}
\STATE$x_{t+1}\leftarrow y_K^{(t)}$

\ENDFOR
\end{algorithmic}
\end{template}

Lines~\ref{sub:start}--\ref{sub:end} apply
$K$ iterations of the Frank–Wolfe algorithm on
\begin{equation}
  \label{adasfw:sub}
  \min_{x\in\mathcal{X}} \left\{
    Q_{t}(x) \defeq
    f(x_{t}) + \innp{\widetilde{\nabla}f(x_{t})}{x - x_{t}}
    + \frac{1}{2 \eta_{t}} \norm[H_{t}]{x - x_{t}}^{2}
  \right\}
  ,
\end{equation}
which is equivalent to the AdaGrad subproblem with a time-varying
learning rate $\eta_t$.
The iterates are denoted by $y_0^{(t)},\dotsc,y_K^{(t)}$,
starting from $x_t=y_0^{(t)}$ and ending at $x_{t+1}=y_K^{(t)}$.
The step size in Line~\ref{adafw:gamma} is optimal in the sense that
$\gamma_k^{(t)} = \argmin_{\gamma \in [0, \gamma_t]}
Q_t(y_k^{(t)} + \gamma(v_k^{(t)}-y_k^{(t)}))$. We report in Theorems~\ref{th:adasfw}
and~\ref{th:ncvx} the convergence rates of the method with gradients
estimated as in SFW (Algorithm~\ref{sfw}).
The matrix $H_t$
can be any diagonal matrix with positive entries bounded over time.
Note that this level of generality comes at a price, 
as the upper bound on the convergence rate is worse 
than that of SFW for example, in order to account for all possible choices of $H_t$.
In practice, the authors suggest to set $H_t$ as in AdaGrad and demonstrate good performance. Note that better choices for $H_t$ may also be found and are permitted by the convergence analysis.

\begin{theorem}
  \label{th:adasfw}
  Let \(f\) be an \(L\)-smooth convex stochastic function
  over a compact convex set \(\mathcal{X}\)
  with diameter at most \(D\).
  Consider Template~\ref{adafw} with
  $b_{t} = (KG(t+2)/(LD))^{2}$,
  \smash{$\tilde{\nabla}f(x_t)=(1/b_t)\sum_{j=1}^{b_t}\tilde{\nabla}F(x_t,z_j)$} where $z_1,\ldots,z_{b_t}$ are $b_t$ independent realizations of $Z$,
  $\eta_{t} = \lambda_{t}^{-}/L$,
  and $\gamma_{t} = 2/(t+2)$,
  and let $\kappa \defeq \lambda_{0}^{+}/\lambda_{0}^{-}$.
  Then for all $1 \leq t \leq T$,
 \begin{equation*}
  \E{f(x_t)} - f(x^{*})
  \leq\frac{2LD^2(\kappa+1+1/K)}{t+1}.
 \end{equation*}
  Equivalently,
  to reach an expected primal gap of at most $\varepsilon > 0$,
  the algorithm computes
$ \mathcal{O}(K^2G^2LD^4\kappa^3/\varepsilon^3)$
  stochastic gradients
  and
 $\mathcal{O}(K LD^2\kappa/\varepsilon)$
  linear minimizations.
\end{theorem}

For non-convex objectives, convergence is measured via the Frank–Wolfe
gap (see Definition~\ref{FrankWolfeGap}), as done in, e.g.,
\citet{lj16nonconvex,reddi2016stochastic}. It measures the convergence
to a stationary point of $f$ over $\mathcal{X}$.

\begin{theorem}
  \label{th:ncvx}
  Let \(f\) be a  (not necessarily convex) smooth stochastic function
  \index{convergence for non-convex objective}
  over a compact convex set \(\mathcal{X}\) with diameter at most \(D\).
  Consider Template~\ref{adafw} with
  $b_t = (KG/(LD))^{2} (t+1)$,
  $\tilde{\nabla}f(x_t)=(1/b_t)\smash{\sum_{j=1}^{b_t}}\tilde{\nabla}F(x_t,z_j)$ where $z_1,\ldots,z_{b_t}$ are $b_t$ independent realizations of $Z$,
  $\eta_t = \lambda_t^-/L$,
  and $\gamma_t = 1/(t+1)^{1/2+\nu}$
  where $0 < \nu < 1/2$,
  and let \smash{$\kappa \defeq \lambda_0^+/\lambda_0^-$}.
  Then for all $0 \leq t \leq T-1$,
 \begin{equation*}
 \frac{1}{t+1} \sum_{\tau = 0}^{t} \E{g(x_{\tau})}
  \leq\frac{\bigl(f(x_0)-f(x^{*})\bigr)
    + LD^2 (\kappa/2+1+1/K) \zeta(1+\nu)}{(t+1)^{1/2-\nu}},
 \end{equation*}
 where $\zeta(\nu)
\defeq \sum_{s=0}^{+\infty}1/(s+1)^{\nu}$ is the
 Riemann zeta function.
 Alternatively, if the time horizon $T$ is fixed, then with $b_t = (KG/(LD))^2T$ and $\gamma_t = 1/\sqrt{T}$,
 \begin{equation*}
   \frac{1}{T} \sum_{t = 0}^{T-1} \E{g(x_{t})}
  \leq\frac{\bigl(f(x_0) - f(x^{*})\bigr) + LD^2 (\kappa/2+1+1/K)}{\sqrt{T}}.
 \end{equation*}
\end{theorem}

In practice, it is not often necessary to put bounds
$\lambda_t^-,\lambda_t^+$ and we can set $\gamma_t = 1$ and
$K\sim5$. This value of $K$ is small enough so that the complexity of
an iteration remains cheap while information from the adaptive matrix $H_t$,
given via subproblem~\eqref{adasfw:sub}, is still leveraged.  The
learning rate $\eta_t$ can be set  to a constant value and tuned in
the range $\{10^{i/2}\mid i\in\mathbb{Z}\}$.

\subsection{Zeroth-Order Stochastic Conditional Gradient algorithms}
\label{sec:zeroth-order-stoch}

In many applications we do not even have access to stochastic
first-order oracles (SFOO), which makes the task of estimating the
gradient at a point particularly challenging, and thus tackling
Problem~\eqref{eq:stochastic_problem} even more difficult than in
previous sections.  Such is the case in simulation based modelling, or
when designing black-box attacks for deep neural networks.
This motivates the zeroth-order setting where only function evaluations
are available.
We consider here the stochastic case where we
only have access to \emph{Stochastic Zeroth-Order Oracle} (SZOO).
As conditional gradient algorithms require the gradient, in the
zeroth-order setting they use an estimator,
so the deterministic setting is not much simpler than the stochastic
one.

\begin{oracle}[H]
  \caption{Stochastic Zeroth-Order Oracle for \(f\) (SZOO)}
  \label{ora:SZOO}
  \begin{algorithmic}
    \REQUIRE Point \(x \in \mathcal X\)
    \ENSURE \( \tilde{F}(x,z)\),
      where \(z\) is an independent realization of \(Z\)
\end{algorithmic}
\end{oracle}
One of the main ideas in \citet{zerocg18} is the following gradient
estimator under the standard Euclidean scalar product
using a series \(u_1, \dotsc, u_b\) of independent standard normal
random variables,
and \(b\) independent samples from the probability distribution of $Z$, denoted by \(z_1, \dotsc, z_b\)
(similar estimators also appear in other contexts,
like Equation~\eqref{eq:point_estimate} for online methods)
\begin{equation}
  \label{eq:gradient-normal-estimate}
  \widetilde{\nabla} f(x) = \frac{1}{b}\sum_{i=i}^b\frac{\tilde{F}(x + \eta u_i,z_i) - \tilde{F}(x,z_i)}{\eta} u_i,
\end{equation}
for some $\eta > 0$. Note that one can also sample the \(u_{i}\)
from other distributions,
for example choosing $b$ samples uniformly from the unit sphere, as suggested in
\citet{polyak1987introduction}, but we focus here on the
standard normal distribution.
The accuracy of the estimator in
Equation~\eqref{eq:gradient-normal-estimate} is controlled by
the distance \(\eta\) and the batch size \(b\):
in general a smaller \(\eta\) and a larger \(b\) increases accuracy.

Using estimator \eqref{eq:gradient-normal-estimate}
in Template~\ref{sfw:tmp} yields the
Zeroth-Order Stochastic Conditional Gradient method (ZSCG)
for the stochastic optimization problem
(Equation~\eqref{eq:stochastic_problem}),
thus differing from  the Stochastic Frank–Wolfe algorithm
(Algorithm~\ref{sfw}) only in the gradient estimator.

\begin{algorithm}
\caption{Zeroth-Order Stochastic Conditional Gradient
  Method (ZSCG) \citep[Algorithm~1]{zerocg18}}
\label{zscg}
\begin{algorithmic}[1]
  \REQUIRE Arbitrary start point $x_0\in\mathcal{X}$,
    batch sizes $b_{t}$, parameter $\eta_t$, step sizes $0 \leq \gamma_{t} \leq 1$
    for \(t \geq 0\)
  \ENSURE Iterates $x_1, \dotsc \in \mathcal{X}$
\FOR{$t=0$ \TO \dots}
\STATE Draw $\{z_1,\dotsc,z_{b_t}\}$ and $\{u_1,\dotsc,u_{b_t}\}$, that is, $b_t$ independent realizations of $Z$ and of the standard normal distribution, respectively.
\STATE \label{gradestz_zeroth_order} 
  \(\widetilde{\nabla} f(x_{t}) \gets
  \frac{1}{b_{t}}\sum_{i=i}^{b_{t}}
  \frac{\tilde{F}(x_{t} + \eta_{t} u_{i}, z_{i})
    - \tilde{F}(x,z_{i})}{\eta_{t}} u_{i}\)
\STATE$v_t\leftarrow\argmin_{v\in\mathcal{X}}\innp{\widetilde{\nabla}f(x_t)}{v}$
\STATE$x_{t+1}\leftarrow x_t+\gamma_t(v_t-x_t)$
\ENDFOR
\end{algorithmic}
\end{algorithm}

Convergence rates are based on the accuracy bound
\(\E{\norm{\widetilde{\nabla} f(x) - \nabla f(x)}^2}
\leq 4 (n + 5) G^{2}/b + 1.5 (n + 3)^{3} \eta^{2} L^{2}\)
of the gradient estimator by \citet[Lemma~2.1]{zerocg18},
where $L$ is the smoothness of $\tilde{F}$
and \(\E(z){\norm{\nabla \widetilde{F}(x,z)}^{2}} \leq G^{2}\).
(Here and in the theorem below we simplify the original assumptions
of bounded expected value and variance to bounded second moment.)
By Lemma~\ref{lem:stochastic-rate}, this readily provides the
following convergence rate \citep[cf.][Theorem~2.1]{zerocg18}.
\begin{theorem}
  Let the stochastic function \(\tilde{F}(\cdot, z)\)
  be convex and \(L\)-smooth in its first argument,
  over a compact convex set \(\mathcal{X} \subseteq \mathbb{R}^{n}\)
  with diameter at most \(D\).
  Let \(f\) be the expectation of \(\tilde{F}(\cdot, z)\) in \(z\)
  with stochastic gradients having second moment at most \(G^{2}\).
  Then the Zeroth-Order Stochastic Conditional Gradient method (ZSCG)
  (Algorithm~\ref{zscg})
  with
  \(\gamma_{t} = 2 / (t+2)\),
  \(b_{t} = (n + 5) (t + 2)^{2}\)
  and \(\eta_{t} = D / ((n + 3)^{3/2} (t+2))\)
  ensures for all \(t \geq 1\)
  \begin{equation*}
    \E{f(x_{t})} - f(x^{*})
    \leq
    \frac{\sqrt{16  G^{2} D^{2} + 6 L^{2} D^{4}} + 2 L D^{2}}{t + 2}
    \leq
    \frac{4 G D + (2 + \sqrt{6}) L D^{2}}{t + 2}
    .
  \end{equation*}
  Equivalently,
  for an expected primal gap at most \(\varepsilon > 0\),
  the Zeroth-Order Stochastic Conditional Gradient method (ZSCG)
  computes \(\mathcal{O}(n (G D + L D^{2})^{3} / \varepsilon^{3})\)
  stochastic function values (SZOO calls)
  and \(\mathcal{O}(G D + L D^{2} / \varepsilon)\)
  linear minimizations.
\end{theorem}
Note that the number of stochastic function evaluations contains
the dimension \(n\) as an additional factor
compared to the number of stochastic gradients in other algorithms,
which is necessary due to the lower bound
\(\Omega(n / \varepsilon)\) on the number of function evaluations
by \citet[Proposition~1]{duchi15zero}.
The intuitive explanation is that a function value is a single number,
while a gradient is an \(n\)-dimensional vector,
so that computing a gradient is equivalent to computing \(n\) numbers
(the entries of a vector).
A stronger lower bound \(\Omega(n^{2} / \varepsilon^{2})\)
holds for algorithms evaluating each stochastic function at most one
point,
see \citet[Theorem~7]{derivative-free-stochastic2013}.

Zeroth-order gradient estimation can be used in other algorithms, too,
with the expectation to substitute
each stochastic gradient computation with
\(n\) stochastic function evaluation
in convergence rate.
A few algorithms and convergence rates we have collected
in Table~\ref{table:zero}.

\begin{table}
 \caption{Complexities of zeroth-order
   stochastic conditional gradient algorithms to achieve
   an expected primal gap (Frank–Wolfe gap in the non-convex case)
   at most $\varepsilon>0$
   on an \(n\)-dimensional feasible region.
   \index{convergence for non-convex objective}
   See \citet{zerocg18} for ZSCG and ZSAG,
   and \citet{zeroFW2020} for Acc-SZOFW.}
 \label{table:zero}

 \centering
 \setlength{\tabcolsep}{2pt}
 \begin{tabular*}{\linewidth}{@{\extracolsep{\fill}}lllcc@{}}
  \toprule
  Algorithm&Base algorithm&
  Setting&Linear minimizations
  &Function evaluations\\
  \midrule
  \multirow{2}{*}{ZSCG} & \multirow{2}{*}{SFW} & non-convex
  & $\mathcal{O}\bigl(\frac{1}{\varepsilon^2}\bigr)$
  & $\mathcal{O}\bigl(\frac{n}{\varepsilon^4}\bigr)$ \\
  &&convex
  &$\mathcal{O}\bigl(\frac{1}{\varepsilon}\bigr)$
  &$\mathcal{O}\bigl(\frac{n}{\varepsilon^3}\bigr)$ \\
  \addlinespace
  ZSAG & SCGS & convex
  &$\mathcal{O}\bigl(\frac{1}{\varepsilon}\bigr)$
  &$\mathcal{O}\bigl(\frac{n}{\varepsilon^2}\bigr)$ \\
  \addlinespace
  Acc-SZOFW & SPIDER FW & non-convex
  &\(\mathcal{O}\bigl(\frac{1}{\varepsilon^{3}}\bigr)\)
  &\(\mathcal{O}\bigl(\frac{n}{\varepsilon^{3}}\bigr)\) \\
  \bottomrule
 \end{tabular*}
\end{table}

\subsection{Summary}

Table~\ref{table:sfw} summarizes the complexities of the main convex
Stochastic Frank–Wolfe algorithms discussed in this section.
All the guarantees presented in this section require that
the objective function is convex, smooth.
Except for SVRF,
stochastic gradient samples are required to have a bounded variance.

\begin{table}
 \caption{Comparison of the complexities of the stochastic variants of Frank–Wolfe
   algorithms to achieve at most $\varepsilon>0$ primal gap
   in expectation when the objective function is convex and smooth, over a compact convex set.
   Convergence rates in the lower half of the table require that
   the stochastic function samples are smooth, too.
   A batch is a collection of consecutive stochastic gradients,
   which the algorithm incorporates together at the same time
   (this is subjective, we are not aware of a precise definition).
   The number of batches is omitted from the table
   as it is the same as the number of linear minimizations,
   except for SCGS, which has \(\mathcal{O}(1 / \sqrt{\varepsilon})\) batches.
   The Ada- versions of SFW and SVRF yield the same bounds in
   primal accuracy $\varepsilon$.}
 \label{table:sfw}

\centering
 \setlength{\tabcolsep}{2pt}
 \begin{tabular*}{\linewidth}{@{\extracolsep{\fill}}lccc@{}}
  \toprule
  Algorithm & Size of batch \(t\) &
  \makecell[t]{Linear\\ minimizations}
  & \makecell[t]{Stochastic\\ gradients}  \\
  \midrule
  SFW (Algorithm~\ref{sfw}) & $\Theta(t^2)$
  & $\mathcal{O}(1/\varepsilon)$
  &$\mathcal{O}(1/\varepsilon^3)$
  \\
  Momentum SFW  (Algorithm~\ref{algo_MSFW})
  & $\Theta(1)$
   & $\mathcal{O}(1/\varepsilon^3)$
  &$\mathcal{O}(1/\varepsilon^3)$
  \\
  SCGS (Algorithm~\ref{scgs}) & $\Theta(t^3)$
  & $\mathcal{O}(1/\varepsilon)$
  &$\mathcal{O}(1/\varepsilon^2)$
  \\
  \addlinespace
  SPIDER FW (Algorithm~\ref{algo_FW++}) & $\Theta(t)$
  & $\mathcal{O}(1/\varepsilon)$
  &$\mathcal{O}(1/\varepsilon^2)$
  \\
  SVRF (Algorithm~\ref{svrf}) & $\Theta(t)$
  & $\mathcal{O}(1/\varepsilon)$
  & $\mathcal{O}(1/\varepsilon^2)$
 \\
  1-SFW (Algorithm~\ref{algo_1SFW}) & $\Theta(1)$
    & $\mathcal{O}(1/\varepsilon^2)$
  &$\mathcal{O}(1/\varepsilon^2)$
  \\
  \bottomrule
 \end{tabular*}
\end{table}

The following example gives a brief comparison of the algorithms presented in this section when applied to specific instances of  Problem~\eqref{eq:stochastic_problem}.

\begin{example}[Performance comparison of stochastic conditional
  gradient algorithms]
  \label{ex:stochastic-normal-noise}
Consider minimizing an $L$-smooth function defined as $f(x)
=\E(z){\tilde{F}(x, z)}$, where
\begin{equation*}
  \tilde{F}(x, z) \defeq
  \frac{1}{2} x^{\top} \bigl(A + \diag{z}\bigr) x + (b + z)^{\top} x,
\end{equation*}
$A\in \mathbb{R}^{n \times n}$, $z\in \mathbb{R}^{n}$ and $z\sim \mathcal{N}(0,s^{2} I^n) $ for $s>0$. The matrix $A$ was obtained by first generating an orthonormal basis 
$u_{1}, \dotsc, u_{n}$
in $\mathbb{R}^n$ and a set of $n$ uniformly
distributed values $\{ \lambda_1, \dotsc, \lambda_n \}$
 between $\mu=0$ and $L=100$ and setting 
 $A = \sum_{i=1}^{n} \lambda_{i} u_{i} u_{i}^{\top}$. Note that 
 we generate an orthonormal basis by drawing a random 
 sample from the Haar distribution, which is the only 
 uniform distribution on the special orthogonal group in 
 dimension $n$.
  We set
 $b = - A y$,
 where $y$ is uniformly randomly sampled from
 the \myindex{hypercube} \([-10,+10]^{n}\).
Our goal is to solve the problem $\min_{x\in\mathcal{X}}
 \E(z){\tilde{F}(x, z)}$,
  where $\mathcal{X} = \left\{ x\in \mathbb{R}^{n} \mid
    \norm{x}_{\infty}\leq \tau \right\}$ is a hypercube and $n =100$.
 Note that if
$\tau \geq 10$, then $y = \argmin_{x\in\mathcal{X}}
 \E(z){\tilde{F}(x, z)}$ as the unconstrained minimizer of 
 $f$, given by $y$, is contained in $\mathcal{X}$. For this problem, the
  variance of the stochastic gradients sampled for $z\sim
  \mathcal{N}(0,s^{2} I^n) $
   is given by
 \begin{equation*}
  \begin{split}
   \E{\norm{\nabla \tilde{F}(x,z) - \nabla F(x)}^{2}} 
   &
   =
   \E{\norm{\diag{z}x + z}^{2}}
   \\
   &
   = \sum_{i=1}^{n} (x_{i} + 1)^{2} \E{z_{i}^{2}}
   \\
   & 
   = \sum_{i=1}^{n} (x_{i} + 1)^{2} s^{2}
   \leq n (\tau + 1)^{2} s^{2}.
  \end{split}
 \end{equation*}
 Therefore the variance of the
 stochastic gradients is bounded by $n (\tau + 1)^{2} s^{2}$. We have 
 used this bound when 
 running the SFW, SPIDER FW, and 
 SCGS algorithms, as these algorithms require knowledge of 
 a bound on $\sigma^{2}$ in order to appropriately 
 set the values of the batch size. 
 Figure~\ref{fig:stochastic_expectation_minimization} shows the
evolution of the primal gap in the number of stochastic
gradient oracle calls, and the number of linear minimization oracle
call for $\tau = 100$, and for two different values of $s$, namely
$s = 10$ and $s = 100$ (which 
leads to $\sigma = \num{10100}$ and $\sigma= \num{101000}$, respectively). In this
case, as $\tau = 100$, the minimum is in the strict
interior of the feasible region. 
The algorithms were run until a maximum time of \num{14400}
seconds had been reached. Figure~\ref{fig:stochastic_expectation_minimization_2} 
shows these same metrics when a value of $\tau = 5$ is used for 
$\sigma = \num{10100}$ and $\sigma= \num{101000}$.  In the second example, the
minimum was on a lower dimensional face of the polytope, which we 
verified by computing a high accuracy solution to the problem (which 
was also used to obtain an approximate value for $f(x^*)$, which
is also used for the primal gap values in
Figure~\ref{fig:stochastic_expectation_minimization_2}).

\begin{figure}
\centering
\small
\begin{tabular}{cc}
\(\sigma = \num{10100}, \tau = 100 \) & \(\sigma= \num{101000}, \tau = 100\) \\[\smallskipamount]
\includegraphics[width=.45\linewidth]{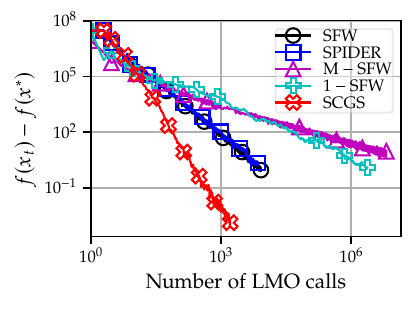}
&
\includegraphics[width=.45\linewidth]{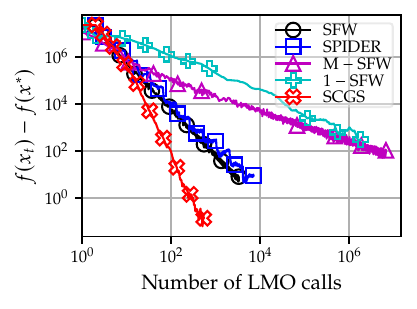}
\\
\includegraphics[width=.45\linewidth]{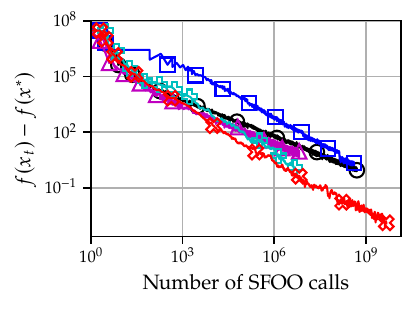}
&
\includegraphics[width=.45\linewidth]{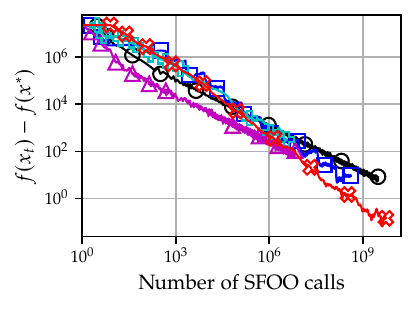}
\\
\includegraphics[width=.45\linewidth]{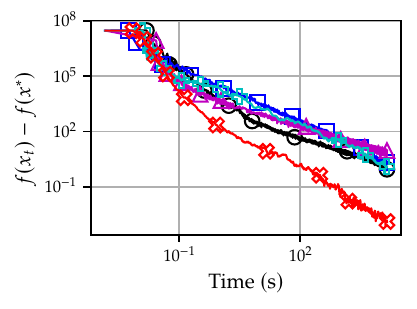}
&
\includegraphics[width=.45\linewidth]{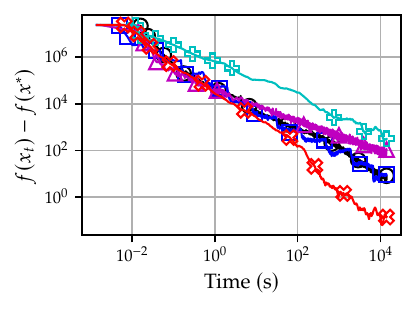}
\end{tabular}

  \caption{Performance of stochastic conditional algorithms
    for minimizing a quadratic function
    over the \(100\)-dimensional \(-\tau /\tau\)-hypercube
    with minimum value in the interior,
    see Example~\ref{ex:stochastic-normal-noise} for details.
    See Figure~\ref{fig:stochastic_expectation_minimization_2}
    for the case when the optimum lies on the boundary.
    The number \(\sigma\) is an upper bound  on the square root of the
    variance of the stochastic oracle.
    Momentum SFW algorithm is abbreviated to M-SFW.
       }
  \label{fig:stochastic_expectation_minimization}
\end{figure}

\begin{figure}
\centering
\begin{tabular}{cc}
\(\sigma = \num{10100}, \tau = 5 \) & \(\sigma= \num{101000}, \tau = 5\) \\[\smallskipamount]
\includegraphics[width=.45\linewidth]{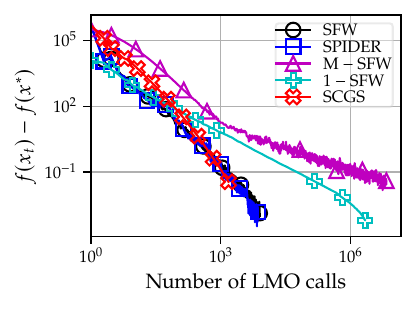}
&
\includegraphics[width=.45\linewidth]{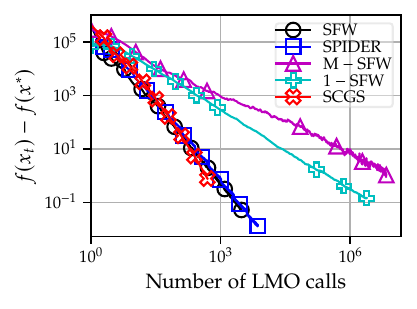}
\\
\includegraphics[width=.45\linewidth]{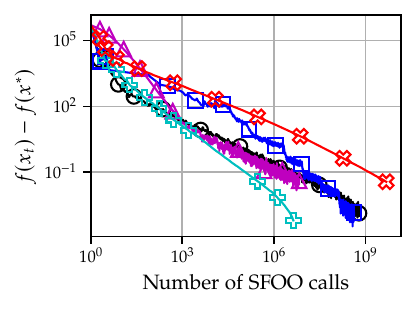}
&
\includegraphics[width=.45\linewidth]{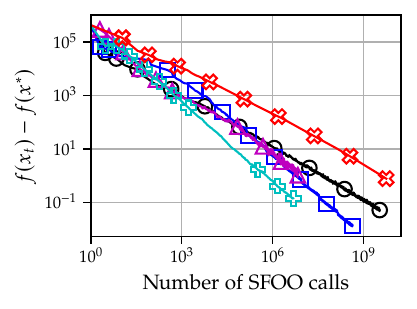}
\\
\includegraphics[width=.45\linewidth]{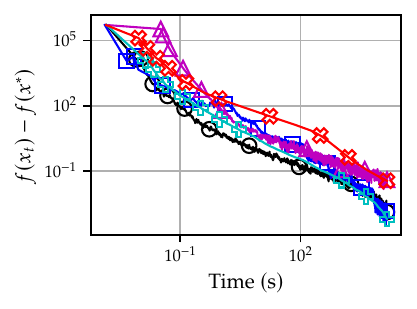}
&
\includegraphics[width=.45\linewidth]{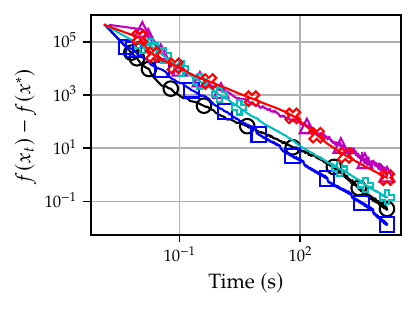}
\end{tabular}

  \caption{Performance of stochastic conditional gradient algorithms
    for minimizing a quadratic function over the \(-\tau /\tau\)-hypercube
    with variance of the stochastic oracle being at most
    \(\sigma^{2}\).
    Unlike Figure~\ref{fig:stochastic_expectation_minimization},
    the optimum $x^*$ is contained in a proper face of the polytope.
    Momentum SFW algorithm is abbreviated to M-SFW. 
    In the left column the SFOO calls have
    lower variance than in the right column.}
  \label{fig:stochastic_expectation_minimization_2}
\end{figure}

We see in both Figures~\ref{fig:stochastic_expectation_minimization}
and~\ref{fig:stochastic_expectation_minimization_2} that the M-SFW
and the 1-SFW algorithms perform the highest number of LMO oracle 
calls per SFOO oracle call. The batch size for these algorithms is 
$1$ and $2$, respectively. Note that the convergence guarantees shown 
in this section state that $\mathcal{O}(1/\varepsilon^3)$ and
$\mathcal{O}(1/\varepsilon^2)$ oracle calls are required to reach an
expected primal gap at most $\varepsilon$, whereas SFW,
SPIDER FW and SCGS algorithms require $\mathcal{O}(1/\varepsilon^2)$ calls.
For the two values $\sigma$ shown 
in Figure~\ref{fig:stochastic_expectation_minimization} the 
SCGS algorithm outperforms the other algorithms by several orders of 
magnitude for primal gap convergence in wall-clock time. This
algorithm is also the one that performs the fewest number of LMO calls 
to reach a given primal gap tolerance, by a wide margin, despite the 
same complexity bounds in LMO calls for the three
algorithms. The improved performance 
of the SCGS algorithm could very well be due to the fact that $x^*$ is in 
the strict interior of the feasible region. In the examples 
in Figure~\ref{fig:stochastic_expectation_minimization} we 
also see that the SFW and the
SPIDER FW algorithms achieve very similar performance for primal gap
convergence in LMO calls. For both these algorithms, this performance
is worse than the one achieved by the SCGS algorithm, but better than the performance 
of the 1-SFW and M-SFW algorithms. If we focus on the 
number of SFOO calls, we see that the
slope of the primal gap in the
number of SFOO calls is very similar for the M-SFW and the SFW algorithms, 
in this case, the convergence guarantees state that we need 
$\mathcal{O}(1/\varepsilon^3)$ SFOO calls to reach an $\varepsilon$-optimal
solution for both algorithms. Similarly the slope of the primal 
gap in
the number of SFOO calls  is very similar for 
the SPIDER, 1-SFW and SCGS algorithms. In this case, the theory states that 
$\mathcal{O}(1/\varepsilon^2)$ SFOO oracle calls are required to reach an
$\varepsilon$-optimal solution.

In Figure~\ref{fig:stochastic_expectation_minimization_2} we can see that in this case, 
in which $x^*$ is in a lower dimensional face of the polytope, the SFW, SPIDER FW and 
SCGS all require a very similar number of LMO calls to achieve a target primal gap 
accuracy, and we don't observe the improved convergence of the SCGS 
algorithm seen in Figure~\ref{fig:stochastic_expectation_minimization}. We can also see that clearly the difference 
in performance between the M-SFW and 1-SFW algorithms in LMO oracle
calls.
While
the former requires in general $\mathcal{O}(1/\varepsilon^3)$ LMO calls to reach some accuracy
$\varepsilon$,
the latter requires $\mathcal{O}(1/\varepsilon^2)$ LMO calls. In this case,
the SPIDER FW, SFW and 1-SFW are the algorithms that perform the best
in wall-clock
time, with similar performance for the three algorithms.
\end{example}

\pagebreak

\begin{example}[Performance of stochastic conditional gradients
  algorithms for \myindex{finite-sum minimization}]
  \label{ex:stoch-regression}
  \looseness=1
  We consider the sparse logistic \myindex{regression problem}
  from Example~\ref{example_log_reg_BCG_boosted}:
  \begin{equation}
  \min_{\norm{x}_1 \leq \tau} \frac{1}{m} \sum_{i = 1}^m f_i(x) =  \min_{\norm{x}_1 \leq \tau} \frac{1}{m} \sum_{i = 1}^m \log \left(1 + e^{-y_i \innp{x}{z_i}} \right),
  \end{equation}
  where $m$ is the number of samples and $x\in\R^{n}$.
  Unlike Example~\ref{example_log_reg_BCG_boosted}
  we consider the case with the number \(m\) of data samples being
  large,
  and hence computing the gradient of the objective function
  being very expensive.
  We use the same datasets and the same regularization as in
  \citet{negiar20}, namely the binary \texttt{rcv1} dataset \citep{CC01a}
  with $\tau = 100$
and the \texttt{breast-cancer} dataset \citep{mangasarian1990cancer} with $\tau = 5$. For the former we have 
$m = \num{697641}$ and $n = \num{47236}$ and for the latter we have $m = 683$ and $n = 10$. Note that the latter
dataset is not particularly large, however it is often used when benchmarking the stochastic 
versions of the Frank–Wolfe algorithm in the literature, which is why it has been included in this
experimental section.
\par

\begin{figure}
\small
\centering
\begin{tabular}{cc}
\texttt{rcv1} & \texttt{breast-cancer} \\[\bigskipamount]
\includegraphics[width=.45\linewidth]{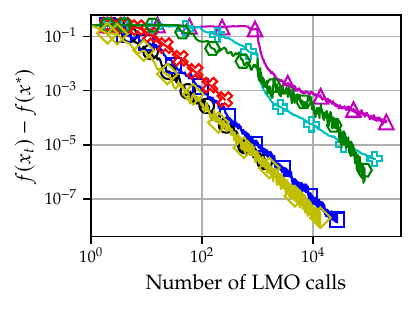}
&
\includegraphics[width=.45\linewidth]{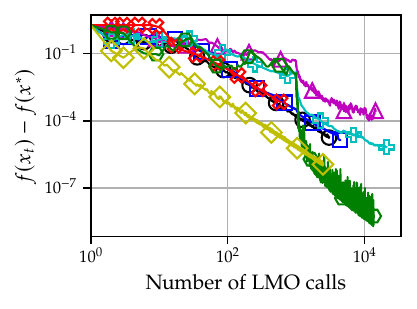}
\\
\includegraphics[width=.45\linewidth]{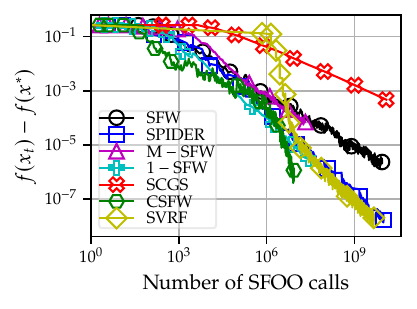}
&
\includegraphics[width=.45\linewidth]{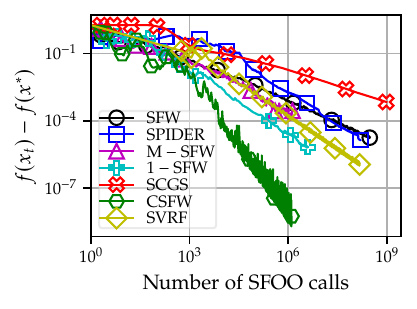}
\\
\includegraphics[width=.45\linewidth]{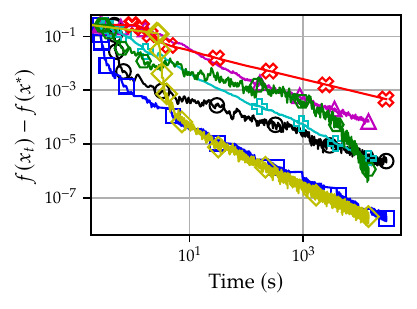}
&
\includegraphics[width=.45\linewidth]{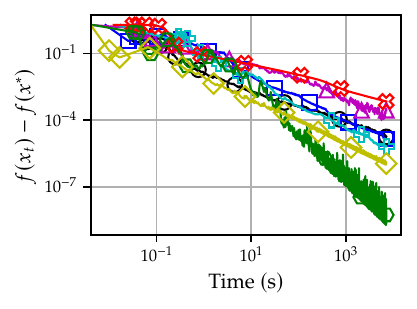}
\end{tabular}

\caption{Comparison of stochastic conditional gradients algorithms
for logistic regression over large datasets,
see Example~\ref{ex:stoch-regression} for details.
The name of the datasets are given over each column.
Momentum SFW algorithm is abbreviated to M-SFW. Note 
that we also compare the performance of the algorithms presented 
in this section with the \emph{Constant batch-size Stochastic 
Frank–Wolfe} (CSFW) algorithm from \citet{negiar20}.}
\label{fig:stochastic_finite_sum_minimization}
\end{figure}

Note that we also compare the performance of the algorithms presented 
in this section with the \emph{Constant batch-size Stochastic 
Frank–Wolfe} (CSFW) algorithm from \citet{negiar20}, which is a
stochastic FW algorithm that can only be applied if the objective 
function is a finite-sum that is linearly separable over the 
data samples. The images shown in Figure~\ref{fig:stochastic_finite_sum_minimization} depict the primal gap 
in the number of SFOO and LMO calls and wall-clock time for all the
algorithms presented in Table~\ref{table:sfw}. Note that in order to compute a FOO call 
in the SVRF algorithm, we need to loop through the $m$ samples in the dataset. In the 
graphs, we account for each individual FOO that is performed by the SVRF algorithm 
by increasing the number of SFOO calls by $m$. Regarding the $\sigma$ 
parameter used by the SFW, SPIDER and SCGS algorithms, we estimate it by
randomly sampling a set of points $\{w_1, \dotsc, w_r \}$ uniformly inside the scaled
$\ell_1$-ball
and computing $\max_{ 1\leq j \leq r}\sum_{i = 1}^m  \frac{1}{m}\norm{\nabla f_i(w_j) - \nabla f(w_j)}^2$. For the smoothness parameter $L$, we estimate it by 
randomly sampling a set of points 
$\{w_1, \dotsc, w_{r} \}$ uniformly inside the scaled
$\ell_1$-ball, and a set of vectors $\{v_1, \dotsc, v_{r} \}$ whose
$n$ components are randomly distributed between $-0.5$ and $0.5$ and
computing $\max_{ 1\leq j \leq r} \norm{\nabla f(w_j) - \nabla f(w_{j}
+ \nu v_{j})}/(\nu\norm{v_{j}})$ for $\nu = 10^{-8}$. This results in
a value of $\sigma \approx 0.3$ and $L \approx 0.0005$ for the
\texttt{breast-cancer} example and $\sigma \approx 3$ and $L \approx
1.5$ for the \texttt{rcv1} example. Using other choices for $\sigma$
and $L$ can
substantially affect the performance of the SFW, SCGS and SPIDER algorithms.

We have
set $\alpha = \left(LD/\sigma \right)^2$ in the batch size of 
the SFW algorithm (see Theorem~\ref{th:sfw:expectation}), as this was the 
original batch size suggested in \citet{svrf16}. We remark that 
the M-SFW, 1-SFW, SVRF, and CSFW do not require knowledge, or tuning, of the value of 
$L$, $\sigma$ and $D$ for setting the batch size of these algorithms. However, note 
that the 
dependence of the SFW, SCGS and the
SPIDER algorithms on some of these constants
  in the original papers is often not necessary,
  as parameters are often chosen to optimize theoretical convergence
rates, or to simplify the proof.

The SVRF algorithm requires computing exact gradients at certain iterations, 
which in this case requires going through all the samples in the finite-sum objective function. As
we account for these exact gradients by counting $m$ SFOO calls, we can clearly see the 
cost of this exact gradient computation in the first iteration on the primal gap convergence 
in SFOO calls for the \texttt{rcv1} example.
In this case the first iteration
requires computing the equivalent of $m=\num{697641}$ SFOO calls.
\end{example}

\begin{example}[Performance of stochastic conditional gradient algorithms
  for non-convex finite-sum minimization]
  \label{ex:CIFAR10}
In this example we compare the performance of several stochastic 
conditional gradient algorithms (and their adaptive versions)
for training a
convolutional neural network, which is a non-convex
finite-sum optimization problem. We reproduce the experiments in
\citet{combettes20adasfw} (and we borrow the 
code in \citet{pokutta2020deep} which uses TensorFlow)
to classify a series of images
in the CIFAR-10 \citep{cifar10}
dataset.
The images have size $32 \times 32$ pixels,
and are from $10$ different classes, $6000$ images per class,
i.e., there are $\num{60000}$ images in total.
The training set and the testing set contain
\num{50000} and \num{10000} images,
respectively. The convolutional neural network used in the classification task 
has an initial layer with ReLU activation and $3\times 3$ 
convolutional kernels layers with $32$ channels, followed by 
a $2 \times 2$ max-pooling layer. This is followed by another 
layer with ReLU activation, $3\times 3$ convolutional kernels 
and $64$ channels, which is followed by a $2 \times 2$ 
max-pooling layer. Finally, we have a fully connected layer
with ReLU activation, $64$ channels, and a $10$-channel output 
softmax layer.
As in \citet{combettes20adasfw} each layer is constrained to an
$\ell_{\infty}$-ball
with $\ell_{2}$-diameter equal to $200$ times the expected $\ell_{2}$-norm of the Glorot uniform
initialized values \citep{glorot2010understanding}.

The algorithms used in this example are the SFW algorithm
(and its adaptive counterpart AdaSFW, as well as a 
variant of AdaSFW with momentum, inspired by Adam and AMSGrad, 
named AdamSFW), the SVRF and the SPIDER FW algorithms.  Additionally we
also compare the performance against the ORGFW algorithm \citep{xie2019efficient}, 
and two projection based algorithms frequently used in the training of deep learning 
networks, namely AdaGrad \citep{duchi11, mcmahan10} and AMSGrad \citep{reddi18}. 
For full transparency, we note that the FW algorithms tested 
in this experiment do not follow 
the exact implementation presented in this section. 
We utilize the hyperparemeters in \citet{combettes20adasfw}, 
which were obtained using grid search, and as in the aforementioned 
paper all the algorithms use a constant batch size of $100$,
opposed
to a $\Omega(t^2)$ batch size for SFW, or a $\Omega(t)$ batch size for SPIDER FW 
or SVRF. In addition to this, all the algorithms utilize a step size of the 
form $\gamma_t = a / t^b$, where the parameters $a>0$ and $b\geq 0$ are tuned 
on a per-algorithm basis using a grid search. A value of $b = 0$ leads to a 
constant step size of $a$, which is what was used for all the 
algorithms except ORGFW. Additional parameters 
that were tuned were the momentum terms for ORGFW, which where also of the 
form $1 / t^{\rho}$ with $\rho = 2/3$ and the algorithmic parameters pertaining to 
AdaSFW, AdamSFW and AMSGrad, such as the maximum number of inner step sizes $K$ to 
perform for AdaSFW and AdamSFW, which were set to $10$ and~$5$ respectively, and the
$\beta_1$ and $\beta_2$ parameters from AdamSFW and AMSGrad, which are the initial decay 
rates used when estimating the first and second moments of the gradient, and which were 
set to $\beta_1= 0.9$ and $\beta_2= 0.999$ (see the papers referenced above for 
more details regarding these parameters). Finally, as was mentioned in
the preceding sections, SVRF and SPIDER FW periodically 
compute gradients to high accuracy, at 
iterations numbered $2^k - 1$ with $k \in \N$. 
However in the experiments presented in this example, 
we compute exact gradients, 
as opposed to high accuracy gradients, once every $100$ 
iterations, which we dub an 
epoch. 

In Figure~\ref{fig:CIFAR10} we show the training error and the testing
accuracy in the number of epochs (where each epoch
consists of $100$~iterations), and the training time.
Each algorithm was run $10$~times, and we indicate with a
solid line the mean training loss and testing accuracy, as well as
plus and minus one standard deviation with shaded regions for each
algorithm. The experiments were run on two Nvidia Tesla V100~GPUs with
32\,GB of memory.

\begin{figure}
\centering
\begin{tabular}{cc}
\includegraphics[width=.45\linewidth]{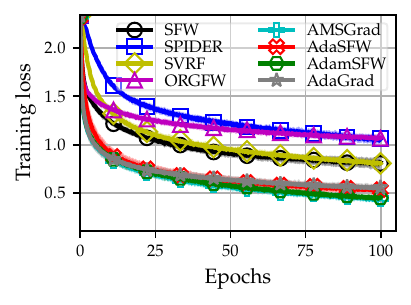}
&
\includegraphics[width=.45\linewidth]{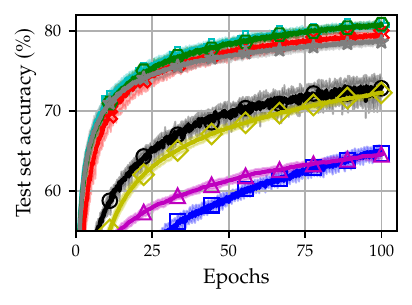}
\\
\includegraphics[width=.45\linewidth]{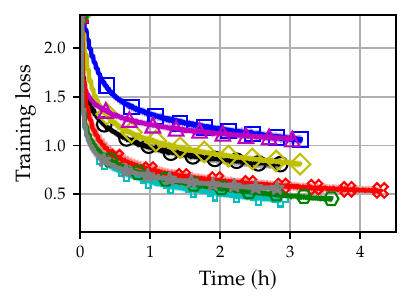}
&
\includegraphics[width=.45\linewidth]{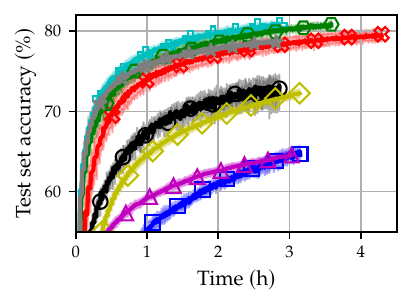}
\end{tabular}

\caption{Comparison of stochastic conditional gradients algorithms
for the non-convex finite-sum problem of training a neural
network,
see Example~\ref{ex:CIFAR10} for details.
An epoch consists of $100$ consecutive iterations. The
solid lines are the mean of $10$ runs for each algorithm, while
the shaded regions denote plus and minus one standard 
deviation.  Note that
both AMSGrad and AdaGrad are projection-based algorithms, while 
the others are not.}
\label{fig:CIFAR10}
\end{figure} 

For the parameters used in this experiment, and given the modifications 
made to the algorithms detailed above,
for training loss and testing accuracy in the number of epochs, the
AdamSFW, AMSGrad, AdaGrad, and AdaSFW
 algorithms are the best 
performing, followed by the SFW and SVRF algorithms, and the 
SPIDER and the ORGFW algorithms.  For training loss and testing
accuracy in wall-clock
time the conclusions are very similar, the 
AMSGrad, AdamSFW, AdaGrad and AdaSFW are the best 
performing, 
followed by the SFW and SVRF algorithms, and followed lastly by 
the SPIDER and ORGFW algorithms.
\end{example}

\section{Online conditional gradient methods}
\label{sec:online-cg}

The online learning setting models a dynamic optimization process in which data
becomes available in a sequential manner
and it is desirable to start optimizing before all data is disclosed,
e.g., because intermediate results are needed,
such as in non-parametric regression or portfolio management \citep{cesa2006prediction}.
This view of
optimization is particularly fruitful in complex settings where precise models
are not available and optimization and learning should be combined.  In
particular, in online optimization, a large or complex optimization problem is
broken down into a sequence of smaller optimization problems, each of which
must be solved with limited information. At the same time, the outcome
associated with each subproblem might be unknown
beforehand and may very well be adversarially chosen. Therefore, the metric
upon which we measure the success of an online algorithm also changes from
previous chapters, namely, primal gap, and we will instead rely on a game
theoretic notion of regret, defined more formally below.
Recall that primal gap is the difference of the function value
to the absolute optimal
choice (with full information and unlimited computational resources).
In contrast, regret is difference of the accumulated functions values, resulted from  solving each subproblem, to the best fixed solution in hindsight.
In fact, in this section, we will be interested in the trade-off between  regret
and the computational cost of solving each subproblem.

The general framework of online optimization is typically formulated as a repeated game 
between a player and an adversary.
In round $t$, the player first chooses an
action $x_t$ from the set 
of all possible actions $\mathcal{X}$; the adversary then responds by
revealing at least the loss for that action
from the loss function
$f_t \colon \mathcal{X} \rightarrow \mathbb{R}$ of the round.
The player's goal is to minimize the total losses she receives,
while the adversary's goal is to maximize the losses
she assigns to the player's action.
The standard benchmark for success is known as \emph{\myindex{regret}},
denoted by $\mathcal{R}_{T}$, which is the difference between the
player's accumulated loss and that of the best fixed action in
hindsight, formally defined as
\begin{equation} \label{eq:definition_regret}
  \mathcal{R}_{T} = \sum_{t=1}^{T} f_{t}(x_{t}) -
  \inf_{x \in \mathcal{X}} \sum_{t=1}^{T} f_{t}(x),
\end{equation}
where $T$ is the number of rounds, also known as the horizon. Throughout this section, we assume that the horizon is known to the online learning algorithm. This is almost without a loss of generality as there is a simple doubling trick that converts an online algorithm with the knowledge of the horizon to an online algorithm without the knowledge of the horizon with almost the same regret. 
Several variants of this general learning setting have been
studied. In particular, when $\mathcal{X}$ is a discrete set of actions, the game is
called either \emph{Prediction From Experts} (PFE), if the player is
allowed knowledge of the complete function $f_t$ in each round \citep{cesa2006prediction}, or
\emph{Multi-Armed Bandit} (MAB), if only $f_t(x_t)$, the loss of the
action performed by the player, is revealed in each round \citep{bubeck2012regret}.

In this section, we  consider the continuous analogs of PFE and MAB,
namely, \emph{Online Convex Optimization (OCO)}, specified as follows \citep{hazan2016introduction, shalev2011online}.
We assume that $ \mathcal{X} $, the action set of the player, is a compact
convex subset of $ \mathbb{R}^n $. We let $ \mathcal{F} $ be a family of differentiable convex 
functions from $ 
\mathcal{X} $ 
to $ \mathbb{R} $ from which the adversary selects the loss
function $f_t$ in round $t$. We mainly consider two variants of OCO, namely,  full-information and bandit feedback. In the full information setting (similar to PFE)
we assume that the player observes the loss function $ f_t $
after playing his action at every round.
In contrast, in the bandit setting (similar to MAB), we assume that
the only information observed by the player is $ f_t(x_t) $ which is
the loss value of the action \(x_{t}\) chosen at round $t$.
Among other feedback types, we mention only one
that differs from the full information setting
in that the player can access the loss functions only through
a stochastic oracle (Oracle~\ref{ora:SFOO}),
instead of an exact first-order oracle (Oracle~\ref{ora:FO}),
for which see \citet{chen2018projection,xie2019efficient}
on conditional gradient algorithms.

We shall confine ourselves to algorithms,
where the computational cost of a round is roughly comparable to that of a single
linear minimization (like a Frank–Wolfe step),
even at the price of suboptimal regret
bounds compared to the information-theoretic setting where computatinal power is not accounted for. The trade-off between information-theoretic bounds and efficiently achievable regret is indeed an active research area \citep{chen2019minimax}.

\subsection{Full information feedback}
\label{sec:online-PFE}

Many  online learning methods are based on the
Follow The Leader (FTL) algorithm,
where in each round $t$, we find the vector $x_t$ that has minimal
loss in our previous rounds (this is indeed equivalent to running
Empirical Risk Minimization in statistical learning theory). However,
it is easy to see that this rule of prediction may not be very stable
and the vectors $x_t$ might change drastically from round to round,
which is heuristically a sign of being too sensitive to recent noise.
In fact, even for linear loss functions, the regret can be made proportional to $T$.  One way to stabilize the behavior of FTL is by adding a regularizer that leads to an optimal regret bound for convex functions. We use the  Regularized Follow The Leader (RFTL) as a template to develop an online conditional gradient method.

\subsubsection{Regularized Follow The Leader}
\label{sec:RFTL}

A general way to stabilize an optimization algorithm
is to regularize the objective function: approximate it with
a better-behaving objective function in the hope
that removal of pathological local behaviour compensates for
the approximation error. An example of this approach has already been presented
in Section~\ref{nonsmooth}.
We now employ another regularization technique by adding to
the objective function of the Follow The Leader
algorithm
a strongly convex and smooth regularizer
$R \colon \mathcal{X}\rightarrow \mathbb{R}$
e.g.,
$R(x) = \norm[2]{x}^2$
(which is \(2\)-strongly convex).
Applied to the Follow The Leader algorithm,
this yields the
\emph{Regularized Follow The Leader (RFTL)} algorithm outlined in
Algorithm~\ref{algo:rftl}\footnote{Note that here we use the terms
  \(x_{t}^{*}\) as the iterates of RFTL so that later in this section we
  can refer to them without causing any confusion.}.

\begin{algorithm}
	\centering
  \caption[]{Regularized Follow The Leader (RFTL)}
  \label{algo:rftl}
	\begin{algorithmic}[1]
\REQUIRE Horizon $ T $, constraint set $ \mathcal{X} $
    \ENSURE $x_{0}^{*}, x_{1}^{*}, \dots, x_{T-1}^{*}$
    \STATE $x_{0}^{*} \leftarrow \argmin_{x\in \mathcal{X}} R(x)$
    \FOR{$t=1$ \TO \(T\)}
    \STATE Play $x_{t-1}^*$ and obtain the loss function $f_t$
    \STATE $x_{t}^* \leftarrow \argmin_{x\in\mathcal{X}} \left(
        \sum_{\tau=1}^{t}f_\tau(x) + R(x)
      \right)$
		\ENDFOR
	\end{algorithmic}    
\end{algorithm}

The following key lemma regarding the performance of RFTL relates the
regret to the cumulative difference of
the loss of playing $x_{t-1}^{*}$ (following the leader)
and playing $x_{t}^{*}$ (being the leader)
\citep[Lemma~2.3]{shalev2011online}; this is very much reminiscient of
the proximal analysis of mirror descent (Algorithm~\ref{mirror}).
The lemma does not require any convexity assumptions
on the set \(\mathcal{X}\) or the functions.
\begin{lemma}
  The output of RFTL (Algorithm~\ref{algo:rftl})
  satisfies for any $x\in \mathcal{X}$
  \begin{equation}
    \sum_{t=1}^{T} [f_{t}(x_{t-1}^{*}) - f_{t}(x)]
    \leq R(x) - R(x_{0}^{*})
    + \sum_{t=1}^{T} [f_{t}(x_{t-1}^{*}) - f_{t}(x_{t}^{*})]
    .
  \end{equation}
\end{lemma}

We shall use the lemma in the following simplified form,
combining the sums on both sides,
so that regret is directly compared to the regularizer
(even though the regret here is not exactly that of RFTL)
\begin{lemma}\label{lem:rftl}
  The output of RFTL (Algorithm~\ref{algo:rftl})
  satisfies for any $x\in \mathcal{X}$
  \begin{equation}
    \sum_{t=1}^{T} [f_{t}(x_{t}^{*}) - f_{t}(x)]
    \leq R(x) - R(x_{0}^{*})
    .
  \end{equation}
\end{lemma}

An immediate consequence of Lemma~\ref{lem:rftl} for the
squared-\(\ell_{2}\)-norm regularization and $G$-Lipschitz convex functions is
the following corollary \citep[Lemma~2.12]{shalev2011online}.
(We slightly improved a constant factor of \(\eta T G^{2}\).)
\begin{corollary}
  \label{cor:rftl}
  Let $f_1, \dots, f_T$ be a sequence of $G$-Lipschitz convex
  functions over a compact convex set \(\mathcal{X}\).
  Then, the regret of RFTL (Algorithm~\ref{algo:rftl})
  with the squared-\(\ell_{2}\)-norm
  regularizer $R(x) \defeq \frac{1}{\eta}\norm[2]{x}^2$ is bounded by
\begin{equation}
  \sum_{t=1}^{T} [f_{t}(x_{t-1}^{*}) - f_{t}(x)]
  \leq
  \frac{1}{\eta}\norm[2]{x}^{2}
  + \frac{1}{2}
  \eta \sum_{t=1}^{T} \norm[2]{\nabla f_{t}(x_{t-1}^{*})}^{2}
  \leq  \frac{1}{\eta}\norm[2]{x}^{2}
  + \frac{1}{2} \eta T G^{2}.
\end{equation}
\end{corollary}
Note that if it is known that $\norm[2]{x}\leq D$ for all $x\in X$,
then by choosing \smash{$\eta = \frac{D}{G\sqrt{2T}}$}
we obtain $\ccalR_T = \mathcal{O}(DG\sqrt{T})$.
The role of regularization is to smooth out local irregularities
of the loss functions, like in Section~\ref{nonsmooth}.
Here this is achieved by adding
a \(2/\eta\)-strongly convex regularizer.
This is manifested in the following version of Lemma~\ref{lem:SCprimal},
adapted to the online setting, from which Corollary~\ref{cor:rftl}
follows using Lemma~\ref{lem:rftl},
see \citet[Theorem 5.2]{hazan2016introduction}.
(Again we have improved the constant factor on the right-hand side.)
\begin{lemma}
  \label{lem:rftl2}
  Let $f_1, \dots, f_T$
  be a sequence of differentiable convex functions.
  Moreover, let the regularizer be
  $R(x) \defeq \frac{1}{\eta}\norm[2]{x}^2$.
  Then, for any $1\leq t \leq T$
  the iterates $x_t^*$
  and the gradients in the RFTL algorithm
  (Algorithm~\ref{algo:rftl})
  are related as follows:
\begin{equation}
  f_{t}(x_{t-1}^{*}) - f_{t}(x_{t}^{*})
  \leq
  \innp{\nabla f_{t}(x_{t-1}^{*})}{x_{t-1}^{*} - x_{t}^{*}}
  \leq
  \frac{1}{2} \eta \norm[2]{\nabla f_{t}(x_{t-1}^{*})}^{2}.
\end{equation}
\end{lemma}

\begin{proof}
Recall that \(x_{t-1}^{*}\)
is the minimum of
\(F_{t-1}(x) = \sum_{\tau =1}^{t-1} f_{\tau}(x) + R(x)\).
 Hence
\begin{equation}
  \innp{\nabla F_{t}(x_{t-1}^{*})
    - \nabla f_{t}(x_{t-1}^{*})}{x_{t-1}^{*}-x_{t}^{*}}
  =
  \innp{\nabla F_{t-1}(x_{t-1}^{*})}{x_{t-1}^{*}-x_{t}^{*}}
  \leq
  0
\end{equation}
Rearranging and using that
\(F_{t}\) is \((2/\eta)\)-strongly convex
(due to the summand \(R(x)\))
with minimum \(x_{t}^{*}\)
\begin{equation}
  \innp{\nabla f_{t}(x_{t-1}^{*})}{x_{t-1}^{*}-x_{t}^{*}}
  \geq
  \innp{\nabla F_{t}(x_{t-1}^{*})}{x_{t-1}^{*}-x_{t}^{*}}
  \geq
  \frac{2}{\eta} \norm[2]{x_{t-1}^{*} - x_{t}^{*}}^{2}
  .
\end{equation}
Therefore
\begin{equation*}
  \innp{\nabla f_{t}(x_{t-1}^{*})}{x_{t-1}^{*}-x_{t}^{*}}
  \leq
  \frac{\eta}{2} 
  \frac{\innp{\nabla f_{t}(x_{t-1}^{*})}{x_{t-1}^{*}-x_{t}^{*}}^{2}}%
  {\norm[2]{x_{t-1}^{*} - x_{t}^{*}}^{2}}
  \leq
  \frac{\eta}{2} \norm[2]{\nabla f_{t}(x_{t-1}^{*})}^{2}
  .
  \qedhere
\end{equation*}
\end{proof}

\subsubsection{Online Conditional Gradient algorithm}

A disadvantage of the
Regularized Follow the Leader algorithm (Algorithm~\ref{algo:rftl})
is the need to solve a possibly hard convex minimization problem
in every round.
We will now show how to use the conditional gradient idea
to control the computation cost.
As a start, we replace the loss functions with a linear approximation
to simplifying gradient computations.
The other technique relies on the total loss function changing
minimally in late rounds, so that the optimum does not move much,
therefore even a single conditional gradient step
provides a good approximation of the new optimum.
These ideas result in the Online Conditional Gradient algorithm
(Algorithm~\ref{alg:conditional_online});
an earlier approach, which is quite different in
presentation, can be found in \citet{hazan12}.

\begin{algorithm}[htb]
	\begin{algorithmic}[1]
    \REQUIRE Horizon $ T $, constraint set $ \mathcal{X} $,
      feasible point $x_1 \in \mathcal{X} $, parameter $\eta$,
      parameters \(\gamma_{t}\) for \(1 \leq t \leq T\)
		\ENSURE $x_2, \dots, x_T $
    \FOR{\(t=1\) \TO \(T\)}
    \STATE Play $ x_t $ and observe the loss function $ f_t$.
    \STATE $ F_{t}(x) \gets  \eta \sum_{\tau=1}^{t} \innp{\nabla
    f_\tau(x_\tau)}{
    x} +\norm[2]{x - x_1}^2 $
    \STATE $ v_{t}\gets \argmin_{x\in \mathcal{X}} \innp{\nabla
		F_t(x_t)}{x} $
    \STATE $ x_{t+1} \gets (1-\gamma_t) x_t +\gamma_t v_t $
		\ENDFOR
	\end{algorithmic}
	\caption{Online Conditional Gradient \citep{hazan2016introduction}}\label{alg:conditional_online}
\end{algorithm}

Note that the steps of Algorithm~\ref{alg:conditional_online} and
Algorithm~\ref{algo:rftl} performed in each round are almost identical except the use of gradients (this will also be useful  for the case with bandit feedback that we will cover later).
For now observe that the algorithm works even with limited feedback,
when only first-order information of the loss function \(f_{t}\)
at the played action \(x_{t}\) is revealed, and the following regret
bound depends only on this first-order information besides
convexity.
In particular, the loss functions need not be smooth.
\begin{theorem} \label{thm:OCG}
  Let $f_1, \dots, f_T$ be a sequence of $G$-Lipschitz and
  differentiable convex functions over a compact convex set
  $\mathcal{X}$
  with diameter $D$ in the Euclidean norm.
  Let $\eta=\frac{D}{GT^{3/4}}$ and
  $\gamma_t=\min\left\{1,\frac{2}{t^{1/2}}\right\}$.
  Then, the regret of Algorithm~\ref{alg:conditional_online}
  is bounded by
\begin{equation}
\ccalR_T=\mathcal{O}(DGT^{3/4}).
\end{equation}	
Here \(\mathcal{O}\) hides an absolute constant.
\end{theorem}	

\begin{proof} We follow the elegant proof given in \citet[Theorem~7.2]{hazan2016introduction}. To do so, we first prove a few auxiliary lemmas. 
Let us define \[x_{t}^* = \argmin_{x\in\mathcal{X}}F_t(x)\] and \[x^* = \argmin_{x\in\mathcal{X}} \sum_{t=1}^T f_t(x).\] Let us also denote \[h_t(x) = F_t(x) - F_t(x_t^*).\] We first need to relate  the iterates $x_t$ to the behavior of $h_t$.

\begin{lemma}
  \label{lem:OCO-FW-step}
		Under the conditions of Theorem~\ref{thm:OCG}, we have:
\begin{equation}
h_t(x_{t+1})\leq (1-\gamma_t)h_t(x_t)+ \frac{D^2}{2}\gamma_t^2.
\end{equation}
\end{lemma}

The proof is identical to that of
Equation~\eqref{eq:FW-step-progress}, and hence omitted. The following lemma bounds the convergence rate of the iterates \(x_{t}\)
despite the objective function continually changing
and performing only a single Frank–Wolfe step per round.
This obviously requires that the objective function changes slowly,
which is enforced by the condition 
\(\eta = \mathcal{O}(\gamma_{t}^{3/2})\).
This condition is a major limitation for the overall regret bound,
while the other conditions of the lemma are only technical in nature.
The lemma is deliberately formulated to be also useful for bandit
feedback later.

\begin{lemma}\label{lem:ht}
  Assume \(0 \leq \gamma_{t} - \gamma_{t+1} \leq \gamma_{t}^{2} / 2\)
  and \(\eta G / D \leq (\gamma_{t} / 2)^{3/2}\)
  for all \(1 \leq t \leq T\).
  Then, for any \(1 \leq t \leq T\), we have
  \begin{equation}
    h_{t-1}(x_{t}) \leq 2D^{2} \gamma_{t}
    \quad\text{and}\quad
    h_{t}( x_{t}) \leq 2D^{2} \gamma_{t}
    .
  \end{equation}
  As a consequence,
  \begin{equation}
    \label{eq:24}
    \norm[2]{x_{t} - x_{t}^{*}}
    \leq
    \sqrt{2} D \sqrt{\gamma_{t}}
    \quad\text{and}\quad
    \norm[2]{x_{t} - x_{t-1}^{*}}
    \leq
    \sqrt{2} D \sqrt{\gamma_{t}}
    .
  \end{equation}
\end{lemma}
\begin{proof}
We prove the upper bounds on \(h_{t-1}(x_{t})\) and \(h_{t}(x_{t})\)
simultaneously by induction.
The case \(t=1\) is obvious as \(h_{0}(x_{1}) = 0\) and
\(h_{1}(x_{1}) = \eta \innp{\nabla f_{1}(x_{1})}{x_{1} - x_{1}^{*}}
- \norm[2]{x_{1}^{*} - x_{1}}^{2}
\leq \eta G D \leq D^{2} (\gamma_{1} / 2)^{3/2} < D^{2} \gamma_{1}\).
Now we turn to the induction step.

By Lemma~\ref{lem:OCO-FW-step}, we have
\begin{equation*}
  h_{t}(x_{t+1})
  \leq
  (1-\gamma_{t}) h_{t}(x_{t}) + D^{2} \gamma_{t}^{2} / 2
  \leq
  2 D^{2} \gamma_{t} - 3 D^{2} \gamma_{t}^{2} / 2
  ,
\end{equation*}
from which the claim
\(h_{t}(x_{t+1}) \leq 2 D^{2} \gamma_{t+1}\)
readily follows (cf.~Equation~\eqref{eq:lem-ht-1}).

For the other claim, to estimate \(h_{t+1}(x_{t+1})\),
observe that
by the definition of $ h_t $ and $ F_t $ and that $
x_t^* $ is the minimizer of $ F_t $, we obtain
\begin{equation}
	\begin{split}
   h_{t+1} (x_{t+1})
   & = F_{t+1}(x_{t+1}) - F_{t+1}(x_{t+1}^{*})
   \\
   & = F_t(x_{t+1}) - F_t(x_{t+1}^*) + \eta \innp{ \nabla f_{t+1}(x_{t+1})}{x_{t+1}-
		x_{t+1}^*}\\
    &\le  F_t(x_{t+1}) - F_t(x_{t}^*) + \eta \innp{ \nabla f_{t+1}(x_{t+1})}{x_{t+1}-
		x_{t+1}^*}\\
    &= h_t(x_{t+1}) + \eta \innp{\nabla f_{t+1}(x_{t+1})}{x_{t+1}-
		x_{t+1}^*} \\
  &\le  h_t(x_{t+1})
  + \eta \norm[2]{\nabla f_{t+1}(x_{t+1})} \cdot\norm[2]{x_{t+1}-
		x_{t+1}^*},
	\end{split}
\end{equation}
where we used \(F_{t+1}(x) =
F_{t}(x) + \eta \innp{\nabla f_{t+1}(x_{t+1})}{x}\).
Notice that $ F_t $ is $ 2 $-strongly convex and that $ 
x_t^* $ is the minimizer of $ F_t $. Therefore,
$\norm[2]{x-x_t^*}^2 \leq F_t(x)-F_t(x_t^*)$.
Combining this with the above estimate on \(h_{t}(x_{t+1})\)
we obtain:
\begin{equation*}
 \begin{split}
  h_{t+1}(x_{t+1})
  &
  \leq
  2 D^{2} \gamma_{t} - 3 D^{2} \gamma_{t}^{2} / 2
  + \eta \norm[2]{
  \nabla f_{t+1}(x_{t+1})
	}\sqrt{F_{t+1}(x_{t+1})-F_{t+1}(x_{t+1}^*)}\\
  &
  \leq
  2 D^{2} \gamma_{t} - 3 D^{2} \gamma_{t}^{2} / 2
  + \eta G \sqrt{h_{t+1}(x_{t+1})}
  .
 \end{split}
\end{equation*}
This is a quadratic inequality in \(\sqrt{h_{t+1}(x_{t+1})}\).
As the constant term
\(2 D^{2} \gamma_{t} - 3 D^{2} \gamma_{t}^{2} / 2\)
is positive, any positive number satisfying the inequality in the
\emph{opposite} direction is an upper bound to the solution
\(\sqrt{h_{t+1}(x_{t+1})}\).
In particular, for the claim
\(h_{t+1}(x_{t+1}) \leq 2 D^{2} \gamma_{t+1}\)
it is enough to prove
\begin{equation}
  \label{eq:lem-ht-1}
  2 D^{2} \gamma_{t+1} \geq
  2 D^{2} \gamma_{t} - 3 D^{2} \gamma_{t}^{2} / 2
  + \eta G \sqrt{2 D^{2} \gamma_{t+1}}
  ,
\end{equation}
which by rearranging is equivalent to
\begin{equation*}
  \underbrace{\frac{\gamma_{t} - \gamma_{t+1}}{\gamma_{t}^{2}}}%
  _{{} \leq 1/2}
  +
  \frac{1}{4}
  \underbrace{\frac{\eta G/D}{(\gamma_{t} / 2)^{3/2}}}_{{} \leq 1}
  \cdot \underbrace{\sqrt{\frac{\gamma_{t+1}}{\gamma_{t}}}}_{{} \leq 1}
  \leq
  \frac{3}{4}
  .
\end{equation*}
This clearly holds by the assumptions
\(\gamma_{t} - \gamma_{t+1} \leq \gamma_{t}^{2} / 2\),
\(\eta G / D \leq (\gamma_{t} / 2)^{3/2}\)
and
\(\gamma_{t+1} \leq \gamma_{t}\),
finishing the inductive proof.

For the other claims recall that $F_{t}$ is $2$-strongly convex,
hence using the already proven \(h_{t}(x_{t}) \leq 2 D^{2} \gamma_{t}\)
\[
  \norm[2]{x_{t} - x_{t}^{*}} \leq \sqrt{F_{t}(x_{t})-F_{t}(x_{t}^{*})}
  = \sqrt{h_{t}(x_{t})}
  \leq \sqrt{2} D \sqrt{\gamma_{t}}
  .
\]
The proof of
\(\norm[2]{x_{t} - x_{t-1}^{*}} \leq \sqrt{2} G D \sqrt{\gamma_{t}}\)
is similar.
\end{proof}
Now, we are equipped with enough lemmas to bound the
regret\footnote{We should never regret to have too many lemmas if they
  make the argument more clear.}.
Our starting point is Lemma~\ref{lem:rftl}
for the loss functions \(\innp{\eta \nabla f_{t}(x_{t})}{x}\)
in \(x\) in place of \(f_{t}(x)\)
and the regularizer \(R(x) = \norm[2]{x - x_{1}^{*}}^{2}\):
\begin{equation*}
  \eta \sum_{t=1}^{T} \innp{\nabla f_{t}(x_{t})}{x_{t}^{*} - x^{*}}
  \leq
  \norm[2]{x^{*} - x_{1}}^{2} - \norm[2]{x_{0}^{*} - x_{1}}^{2}
  \leq D^{2}
  ,
\end{equation*}
using \(x_{0}^{*} = x_{1}\). To bound the regret of Algorithm~\ref{alg:conditional_online},
we combine this with convexity
and \(G\)-Lipschitzness of the \(f_{t}\),
followed by Lemma~\ref{lem:ht}
(for which we need \(T \geq 4\), and the theorem is obvious
for \(T < 4\)):
\begin{equation*}
 \begin{split}
  \ccalR_T &= \sum_{t=1}^T [f_t(x_t) - f_t(x^*)]\\
  &
  \leq \sum_{t=1}^{T} \innp{\nabla f_{t}(x_{t})}{x_{t} - x^{*}}
  \\
  &
  =
  \sum_{t=1}^{T} \innp{\nabla f_{t}(x_{t})}{x_{t} - x_{t}^{*}}
  + \sum_{t=1}^{T} \innp{\nabla f_{t}(x_{t})}{x_{t}^{*} -  x^{*}}
  \\
  &
  \leq
  \sum_{t=1}^{T} \norm[2]{\nabla f_{t}(x_{t})}
  \norm[2]{x_{t} - x_{t}^{*}}
  + \sum_{t=1}^{T} \innp{\nabla f_{t}(x_{t})}{x_{t}^{*} -  x^{*}}
  \\
  &
  \leq
  \sqrt{2} G D \sum_{t=1}^{T} \sqrt{\gamma_{t}}
  + \frac{D^{2}}{\eta}
  =
  \mathcal{O}\left(GDT^{3/4}\right)
  .
 \end{split}
\end{equation*}
Note that the parameters \(\eta\) and \(\gamma_{t}\)
were chosen to optimize the regret bound on the last line
(up to a constant factor) while satisfying the conditions of
Lemma~\ref{lem:ht}.
\end{proof}

\begin{remark}[Cost of optimal regret]
  The same proof can be used to show
  that \(\mathcal{O}(T^{2})\) linear optimizations are enough for
  an information-theoretically optimal
  \(\mathcal{O}(\sqrt{T})\) regret bound.
  Here is a sketch: Let us modify Algorithm~\ref{alg:conditional_online}
  to choose \(x_{t+1}\)
  with \(h_{t}(x_{t+1}) \leq D^{2} \gamma_{t}\).
  With \(\gamma_{t} = 1/t\) and \(\eta = \Theta(D / (G \sqrt{T}))\),
  this is achieved, e.g., by the vanilla Frank–Wolfe algorithm
  with \(\mathcal{O}(t D^{2})\) linear optimizations in round \(t\).
  The proof above provides an \(\mathcal{O}(G D \sqrt{T})\)
  regret (the condition on \(\eta\) in Lemma~\ref{lem:ht}
  being relaxed to \(\eta G / D = \mathcal{O}(\gamma_{t}^{2})\)).
\end{remark}

While the best known regret is still \(\mathcal{O}(T^{3/4})\)
with a constant number of linear optimizations on average per round,
with extra assumptions better regret bounds are known.
For smooth functions, \citet{hazan2020faster}
provides \(\mathcal{O}(T^{2/3})\) regret
via randomizing regularization (i.e., perturbation). Also for strongly convex functions,
an \(\mathcal{O}(T^{2/3})\) regret is shown in
\citet{garber2020projfree-learn,wan21_proj_online}.
When the feasible region is also strongly convex,
\citet{wan21_proj_online} provides an \(\mathcal{O}(\sqrt{T}))\)
regret.
Strong convexity is employed by
using quadratic lower bounding functions,
as loss function estimates in the algorithm,
instead of linear functions.

\begin{example}[Performance of the Online Conditional Gradient algorithm (Algorithm~\ref{alg:conditional_online})]
  \label{ex:online-matrix}
In order to showcase the performance of the Online Conditional Gradient (OCG) algorithm 
(Algorithm~\ref{alg:conditional_online}) we draw from a popular online 
recommendation problem, where one wants to recommend products to users
according to their preference, which is approximated from feedback
to previous recommendations.
This is formalized as follows.
We have a collection of $n$ users and $m$ products, and
an unknown matrix $Y \in \R^{n \times m}$ that encodes the affinity
that each user has for each product. That is, the element $Y_{i,j}$ is
the numerical value of the preference of user $i$ towards product $j$.
At each round,
the player outputs a preference matrix $X\in \R^{n \times m}$
(on which recommendation to the next user is based), and an
adversary selects a user \(i\) and product $j$,
along with the real preference for the product by the user, that is,
$Y_{i,j}$ (assuming for simplicity it can be computed from user
feedback).
The loss incurred by the player is defined as $(Y_{i, j} - X_{i, j})^2$.

\begin{figure}[b]
  \centering
  \includegraphics[width=.45\linewidth]{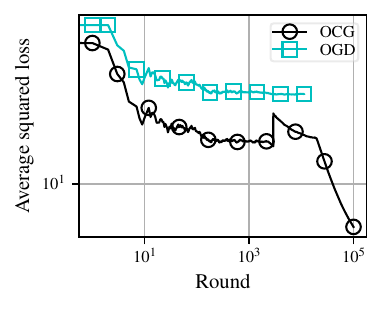}
  \hfill
     \includegraphics[width=.45\linewidth]{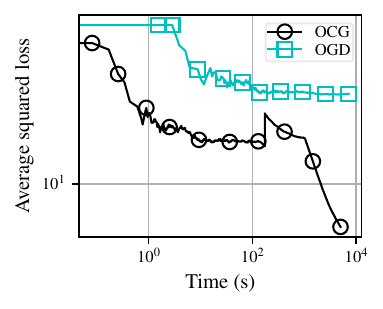}
  \caption{Average squared loss (not regret) for \myindex{matrix completion}
    of an \(1700 \times \num{100000}\)~matrix with \qty{6}{\percent} of known entries
    in rounds (left) and
    wall-clock time (right) for the OCG algorithm
    (Algorithm~\ref{alg:conditional_online}) and the Online
    Gradient Descent (OGD) algorithm \citep{hazan2016introduction},
    see Example~\ref{ex:online-matrix} for details.
    Both algorithms are run
    until a maximum time of $7200$ seconds is reached (see image on the right). 
    As each round of the 
    OGD algorithm is computationally more expensive
    than a round of the
    OCG algorithm, we see that the
    OGD algorithm completes fewer rounds during this 
    experiment (see image on the left). Note that solving a linear 
    optimization problem over the nuclear norm 
    ball requires computing only a single singular vector pair, while computing a 
    projection onto the nuclear norm ball requires computing a 
    full singular-value decomposition.
    \vspace{-1ex}
  }
  \label{fig:OCO}
\end{figure}

In practice, a low-rank matrix \(X\) allows fast computation of
recommendations, so one searches for a low-rank approximation of \(Y\).
Unfortunately, introducing a low-rank matrix approximation constraint on
our learning task makes the problem non-convex, so as a proxy
a natural convex relaxation is used instead, namely an upper bound on the
nuclear norm
 (in a similar fashion as \(\ell_{1}\)-norm constraints can be viewed as convex
 relaxations of constraints on the number of non-zero entries).
 Thus the feasible region will be a nuclear norm ball,
 also called \myindex{spectrahedron}.
 Note that,
 as observed in \citet{fazel2002matrix}, the convex envelope of the set 
 of matrices in $\R^{n \times m}$ of rank at most $k$ and with maximum singular 
 value less than or equal to $M$ is given by the set of matrices of nuclear norm 
 and maximum singular value less than or equal to $kM$ and $M$, respectively.
 We will use the \texttt{MovieLens 100K} database \citep{harper2015movielens} for this
 example, which contains 
 \num{100000} film ratings. The total number of films is $1700$, and the total number of users
 is $1000$, which means that only approximately
 \qty{6}{\percent} of the entries of $Y$ are known.
 At each round $t$ we will draw uniformly at random without replacement
 from these ratings.
 The feasible region we will be considering is the set of matrices in
 $\R^{n \times m}$ with nuclear norm less than $100$, which will serve as 
 a relaxation for the set of matrices in $\R^{n \times m}$ with rank less than~$100$.
 We will compare the performance of the Online Conditional Gradient
 algorithm (Algorithm~\ref{alg:conditional_online}) to that of Online Gradient Descent (OGD) 
 \citep[see][]{hazan2016introduction}. The latter algorithm requires computing projections onto 
 the feasible region, for which we compute a full singular value
 decomposition, while the former requires only linear minimization,
 for which we compute only the top pair
of singular vectors (see the discussion in
Chapter~\ref{cha:introduction}).

In a similar fashion as in \citet{hazan12} we measure the quality of
the solution found by each of the algorithms at round $T$ by computing
the \emph{average squared loss} up to round $T$, defined as
$\frac{1}{T}\sum_{t=0}^T f_t(X_t)$.  Note that this is equal to
$\frac{1}{T}(\mathcal{R}_{T} + \inf_{x \in \mathcal{X}}
\sum_{t=1}^{T} f_{t}(x))$ by
Equation~\eqref{eq:definition_regret}.
The images in Figure~\ref{fig:OCO} show the average squared loss
in rounds and wall-clock time for the OCG and OGD algorithms
where for the former we have used $\eta = 10^5$, which is an approximate
value for the parameter in Theorem~\ref{thm:OCG}.
As in the experiments in \citet{hazan12}
we see that OCG outperforms OGD both in
rounds and in term of wall-clock time.

\end{example}

\subsection{Bandit feedback}
As stated earlier, in the bandit setting the only information that the
player observes in each round $t$ is $f_t(x_t)$.
Since $f_t$ is not revealed,
a natural idea is
to estimate its gradient from its single observed function value
for adapting Algorithm~\ref{alg:conditional_online} to this setting.
To do so, \citet{chen19-aistats} proposed
the \myindex{smoothing} and one-point gradient estimates.
Using such an estimate instead of the true gradient in
Algorithm~\ref{alg:conditional_online}
leads to Algorithm~\ref{alg:conditional_bandit}.

The idea behind the estimate is to
first construct a smoothed version of
loss functions, whose gradients are easier to estimate.
It is expected that the extra regret due to the estimation is small. Given a function $f$,  its $ \delta $-smoothed version is defined as
\begin{equation}
	\label{eq:smoothed function}
  \hat{f}_\delta(x) = \E(v \sim B^n){f(x + \delta v)},
\end{equation} 
where $v \sim B^n$ is drawn uniformly at 
random from the $ n $-dimensional unit ball $B^n$.  Here, $ \delta $ controls the 
radius of the ball that the function $ f $ is averaged over. Since $ \hat{f}_\delta $ is a smoothed 
version of $f$, it  inherits analytical properties from $ f $. Lemma~\ref{lem:smoothing} 
formalizes this idea, see \citet[Lemma 2.6]{hazan2016introduction}.
Note also that if \(f\) is \(G\)-Lipschitz,
then so is \(\hat{f}_{\delta}\).

\begin{lemma}
  \label{lem:smoothing}
  Let \(\mathcal{X} \subseteq \mathbb{R}^{n}\) be a compact convex
  set,
  and $f \colon \mathcal{X} \to \mathbb{R}$ be a convex, $ G
	$-Lipschitz continuous function and let $ \mathcal{X}_0\subseteq \mathcal{X} $ be such that 
  for all $ x\in \mathcal{X}_0$, $v \in B^n $,
  we have $ x + \delta v \in \mathcal{X}$.
  Let
	$\hat{f}_\delta$ be the $\delta$-smoothed function defined above. Then  $\hat{f}_\delta$ 
  is also convex, and $ \norm[\infty]{
  \hat{f}_\delta-f} \le \delta G $ on $ \mathcal{X}_0 $.
\end{lemma}
The function $ \hat{f}_\delta $ is an approximation of $ f $,
with the advantage of admitting
one-point gradient estimates of $\hat{f}_\delta$
based on only a \emph{single} value of $f$
as shown in the  following lemma,
see \citet{flaxman2004online} and \citet[Lemma 6.4]{hazan2016introduction}.
\begin{lemma}
  \label{lem:smooth_estimate}
	Let $\delta > 0$ be any fixed positive real number and $\hat{f}_\delta$ be the 
	$\delta$-smoothed version of  function $ f $. The following equation holds
	\begin{equation}
    \label{eq:point_estimate}
	\nabla 
  \hat{f}_\delta (x) =
  \E(u \sim S^{n-1}){\frac{n}{\delta}f(x + \delta u) u},
	\end{equation}
  where $u \sim S^{n-1}$ is drawn uniformly at
  random from the surface \(S^{n-1}\)
  of the $ n $-dimensional unit ball $B^n$.
\end{lemma}
This lemma shows that the gradient estimator
$ \widetilde{\nabla} f_{t}(x_{t}) $
is an unbiased estimator of the gradient of
an approximation \(\hat{f}_{t, \delta}\)
to the loss function \(f_{t}\), as intended.
The price of using a one-point gradient estimate is that the
estimator is of the size of the function value instead of the
gradient, compared to gradient estimates using several points like
Equation~\eqref{eq:gradient-normal-estimate}.

\begin{algorithm}[htb]
	\begin{algorithmic}[1]
    \REQUIRE horizon $ T $,
      constraint set $ \mathcal{X} \subseteq \mathcal{R}^{n}$,
      approximation parameters \(0 < \alpha < 1\), \(\delta > 0\)
      with \((1 - \alpha) \mathcal{X} + \delta B^{n}
      \subseteq \mathcal{X}\)
		\ENSURE $ y_1, y_2, \dots, y_T $
		\STATE $ x_1 \in (1-\alpha) \mathcal{X} $
    \FOR{\(t=1\) \TO \(T\)}
    \STATE $ y_t\gets x_t+\delta u_t $, where $ u_t \sim S^{n-1} $
		\STATE Play $ y_t $ and observe $ f_t(y_t) $ 	
    \STATE $ \widetilde{\nabla} f_{t}(x_{t})\gets \frac{n}{\delta} f_t(y_t)  u_t$
    \STATE $ F_{t}(x) \gets \eta
      \sum_{\tau=1}^{t} \innp{\widetilde{\nabla} f_{\tau}(x_{\tau})}
    {x} + \norm[2]{x - x_1}^2 $ \label{ln:F_t}
		\STATE $ v_{t}\gets \argmin_{x\in (1-\alpha) \mathcal{X}} \innp{\nabla 
		F_t(x_t)}{x} $	\label{line:linear_optimization}	
		\STATE $ x_{t+1}\gets (1-\gamma_t) x_t +\gamma_t v_t $
		\ENDFOR
	\end{algorithmic}
	\caption{Projection-Free Bandit Convex 
		Optimization \citep{chen19-aistats}}\label{alg:conditional_bandit}
\end{algorithm}

Note that instead of the regret minimizer candidate $ x_t $,
the player  plays a random point $ y_t $
close to $ x_t $
for obtaining a good gradient estimator.
Thus some regret is intentionally incurred
to learn more about the loss functions,
which is a typical feature of bandit algorithms,
sometimes called balancing of exploitation
and exploration.

\begin{theorem}\label{thm:regret_bound}
  Let $f_1, \dots, f_T$ be a sequence of $G$-Lipschitz convex
  functions over a compact convex set $\mathcal{X}$ with diameter $D$.
  Without loss of generality, we assume  that the
  constraint set $ \mathcal{X} $ contains
  a ball of radius $ r $ centered at the
  origin (this is always achievable by shifting the constraint set
  as long as it
  has a non-empty interior), i.e., $ r B^n\subseteq \mathcal{X}$.
  Furthermore, we assume that the convex functions are all uniformly bounded over $\mathcal{X}$, namely,  $
  \norm[\infty]{f_t} \leq M $
  on $ \mathcal{X} $.
  Let $ \eta =
  \frac{c D}{2 \sqrt{2} n M} T^{-4/5}$,
  $ \gamma_t = t^{-2/5} $,  $ \delta = c
  T^{-1/5} $, and $ \alpha=\delta/r < 1$
  in Algorithm~\ref{alg:conditional_bandit},
  where $c>0$ is a  constant.
  Then all iterates remain
  inside $\mathcal{X}$, i.e., $ y_t \in \mathcal{X}$
  for all $1\le t\le T$.
  Moreover,
	the expected regret $ \E{\mathcal{R}_{T}}$ up to horizon $ T $  is at most
  \begin{equation}
    \E{\mathcal{R}_{T}}
    \leq
    \mathcal{O} \left(
      \frac{n M D}{c} + D G + \frac{c G D}{r}
    \right)
    T^{4/5}
    .
  \end{equation}
  Here \(\mathcal{O}\) hides an absolute constant.
\end{theorem}

\begin{proof}
Note that the choice of \(\alpha\) ensures
\(\delta B^{n} = \alpha r B^{n} \subseteq \alpha\mathcal{X}\),
hence
 \((1 - \alpha) \mathcal{X} + \delta B^{n} \subseteq \mathcal{X}\).

First we show that $ y_t\in \mathcal{X}$.
Note that by induction and convexity,
we have \(x_{t}, v_{t} \in (1 - \alpha) \mathcal{X}\).
Thus \(y_{t} = x_{t} + \delta u_{t} \in (1 - \alpha) \mathcal{X} +
\delta B^{n} \subseteq \mathcal{X}\).
	
Now, let us bound the expected regret but at first only
against points $z \in (1 - \alpha) \mathcal{X}$.
	\begin{equation}
		\label{eq:regret decomposition}
    \sum_{t=1}^{T} \E{f_t(y_t)-f_t(z)} =
    \sum_{t=1}^{T} \E{f_t(y_t)-f_t(x_t)}
		+
    \sum_{t=1}^{T} \E{f_t(x_t)-f_t(z)},
	\end{equation}
  and as $ f_t $ is $ G $-Lipschitz
	\begin{equation}
		\label{eq:regret1}
		\sum_{t=1}^{T}\mathbb{E}[f_t(y_t)-f_t(x_t)] 
    \leq \delta T G,
	\end{equation}
	we only need to obtain an upper bound of 
  the second term on the right hand side of~\eqref{eq:regret
    decomposition}.  By Lemma~\ref{lem:smoothing},
  \begin{equation*}
   \begin{split}
    \sum_{t=1}^{T} & \E{f_t(x_t)-f_t(z)}
      \\
      &
      = \E{ \sum_{t=1}^{T}(
			\hat{f}_{t,\delta}(x_t)-\hat{f}_{t,\delta}(z )) 
			+\sum_{t=1}^T (f_t(x_t)-\hat{f}_{t,\delta}(x_t))
			-\sum_{t=1}^T 
      (f_t(z)-\hat{f}_{t,\delta}(z)) } 
      \\
      &
      \stackrel{(a)}{\leq} \E{ \sum_{t=1}^{T}(
			\hat{f}_{t,\delta}(x_t) 
      -\hat{f}_{t,\delta}(z ))} + 2\delta GT 
      \\
      &
      \stackrel{(b)}{\leq}
			\sum_{t=1}^{T} 
      \E{ \innp{\nabla
			\hat{f}_{t,\delta}(x_t)}{ 
      x_t-z}}
    + 2 \delta GT.
   \end{split}
  \end{equation*}
	Inequality $ (a) $ is due to 
	Lemma~\ref{lem:smoothing}. We used the convexity of  $\hat{f}_{t,\delta}$ in $ 
  (b) $. Let $ x_t^* \defeq \argmin_{x\in (1-\alpha)\mathcal{X}} F_t(x) $. We
  add \eqref{eq:regret1},
	split $  \innp{\nabla 
  \hat{f}_{t,\delta}(x_t)}{x_t-z}$ into $  \innp{\nabla \hat{f}_{t,\delta}(x_t)}{x_{t-1}^{*}-z}+   \innp{\nabla \hat{f}_{t,\delta}(x_t)}{
  x_t-  x_{t-1}^{*}} $ and thus obtain
	\begin{multline}	\label{eq:bound_regret}
      \sum_{t=1}^{T} \E{f_t(y_t)-f_t(z)}
      \\
      \le  \sum_{t=1}^{T} \E{ \innp{\nabla
			\hat{f}_{t,\delta}(x_t)}{ 
      x_{t-1}^{*} - z}} + \sum_{t=1}^{T} \mathbb{E}[ \innp{\nabla
			\hat{f}_{t,\delta}(x_t)}{ 
      x_t-  x_{t-1}^{*}}] + 3 \delta G T
		\end{multline}	
As \(\widetilde{\nabla} f_{t}(x_{t})\) is an unbiased estimator of
\(\nabla \hat{f}_{t, \delta}\)	given all data before round \(t\),
Equation~\eqref{eq:bound_regret} can be expressed as
	\begin{multline}	\label{eq:bound_regret2}
    \sum_{t=1}^{T} \E{f_t(y_t)-f_t(z)}
    \\
    \le \sum_{t=1}^{T} \E{\innp{\widetilde{\nabla} f_{t}(x_{t})}{
      x_{t-1}^{*} -
      z}}
  + \sum_{t=1}^{T} \E{ \innp{\nabla \hat{f}_{t,\delta}(x_t)}{
      x_t-  x_{t-1}^{*}}} + 3 \delta G T.
	\end{multline}	
  To bound $ \sum_{t=1}^{T}\innp{ \widetilde{\nabla} f_{t}(x_{t})}{x_{t-1}^{*} - z}$ we resort to Corollary~\ref{cor:rftl}:
\begin{equation}
  \sum_{t=1}^{T} \innp{\widetilde{\nabla} f_{t}(x_{t})}{x_{t-1}^{*} - z}
  \le D^2 / \eta + \frac{1}{2} \eta n^2 M^2
T/\delta^2.
\end{equation}
To bound
\(\innp{\nabla \hat{f}_{t,\delta}(x_t)}{x_{t} -  x_{t-1}^{*}}\),
we apply Lemma~\ref{lem:ht}:
\begin{equation}
\innp{\nabla \hat{f}_{t,\delta}(x_t)}{
  x_{t} -  x_{t-1}^{*}}
  \leq
  \norm[2]{\nabla \hat{f}_{t,\delta}(x_{t})}
  \norm[2]{x_{t} -  x_{t-1}^{*}}
  \leq \sqrt{2} G D \sqrt{\gamma_{t}}
  .
\end{equation}
Therefore, \eqref{eq:bound_regret2} becomes
\begin{equation}
  \label{eq:regret_in_shrunk}
  \sum_{t=1}^{T} \mathbb{E}[f_t(y_t)-f_t(z)]
  \le  D^2/\eta + \frac{1}{2}\eta n^2 M^2
  T/\delta^2 + \sqrt{2} D G \sum_{t=1}^{T}
  \sqrt{\gamma_{t}}
  + 3 \delta G T
  .
\end{equation}
  Let $ x^* \defeq \argmin_{x\in \mathcal{X}} \sum_{t=1}^{T} f_t(x) $
  and note that \(\norm[2]{ x^* - (1-\alpha)x^* } \le \alpha D\).
  Plugging $ z = (1 - \alpha) x^* $ into \eqref{eq:regret_in_shrunk},
  we obtain
  \begin{equation*}
   \begin{split}
  \E{\mathcal{R}_{T}}
  &
  =
  \sum_{t=1}^{T} \E{f_t(y_t)-f_t(x^*)}
  =
  \sum_{t=1}^{T}
  \E{f_t(y_t)-f_t((1 - \alpha) x^*)+f_t((1 - \alpha) x^*)-f_t(x^*)} \\
  &
  \leq
  D^{2}/\eta + \frac{1}{2} \eta n^{2} M^{2} T / \delta^{2}
  + \sqrt{2} D G \sum_{t=1}^{T} \sqrt{\gamma_{t}}
  + 3 \delta G T
  + \alpha D G T
  \\
  &
  \stackrel{(a)}{\le}
  \frac{2 \sqrt{2} n M D}{c} T^{4/5}
  + \frac{n M D}{4 \sqrt{2} c} T^{3/5}
  +
  \frac{5}{2 \sqrt{2}}
  DGT^{4/5} +
  (3 + D/r) c G T^{4/5}\\
  &
  =
  \frac{n M D}{4 \sqrt{2} c} T^{3/5}
  + \left(
    \frac{2 \sqrt{2} n M D}{c}
    +
    \frac{5}{2 \sqrt{2}} D G + (3 + D/r) c G
  \right)
  T^{4/5}
  ,
 \end{split}
\end{equation*}
  where we used $ \sum_{t=1}^{T} t^{-1/5}\leq
  \frac{5}{4}T^{4/5}$ and $ \alpha=\delta/r$ in (a).
We note that the parameters have been chosen to optimize
the regret bound (up to a constant factor) while satisfying
Lemma~\ref{lem:ht}.
Note that we use \(n M / \delta\) instead of \(G\)
for the gradient bound in the lemma, as we need to bound
the \(\widetilde{\nabla} f_{t}(x_{t})\).
\end{proof}
If you think that $\mathcal{O}(T^{4/5})$
does not seem like the tightest regret bound one can achieve in the
bandit setting, you are actually right.
The bottleneck in the proof above is the large size of the gradient
estimators involving a factor $\delta^{-1}$.
To circumvent the issue, \citet{garber20a} recently proposed to
update the base action \(x_{t}\) less often, only after every \(K\)
rounds,
reducing the variance of the average of
the most recent \(K\) gradient estimators
by making them independent.
This leads to Algorithm~\ref{alg:block-bandit},
where in addition online optimization is given up,
and instead the total loss estimator is optimized
afresh in every \(K\) rounds,
hoping that rareness of optimization counterbalances its cost.
Indeed, \citet{garber20a} shows the regret bound to be
$\mathcal{O}(T^{3/4})$ with only \(\mathcal{O}(T)\)
linear minimizations in total.
However, it is still an open problem whether a better regret bound for
projection-free convex bandit optimization is achievable.

\begin{algorithm}
  \caption{Block Bandit Conditional Gradient Method \citep{garber20a}}
  \label{alg:block-bandit}
  \begin{algorithmic}[1]
    \REQUIRE
      horizon \(T\),
      constraint set \(\mathcal{X} \subseteq \mathbb{R}^{n}\),
      block size \(K\),
      tolerance \(\varepsilon > 0\),
      approximation parameters \(0 < \alpha < 1\), \(\delta > 0\)
      with \((1 - \alpha) \mathcal{X} + \delta B^{n}
      \subseteq \mathcal{X}\)
    \ENSURE \(y_{1}, y_{2}, \dots\)
    \STATE \(x_{1} \in (1 - \alpha) \mathcal{X}\)
    \FOR{\(m = 1\) \TO \(T/K\)}
      \FOR{\(t = (m - 1) K + 1\) \TO \(m K\)}
        \STATE \(y_{t} \gets x_{m} + \delta u_{t}\)
          where \(u_{t} \sim S^{n-1}\)
        \STATE Play \(y_{t}\) and observe \(f_{t}(y_{t})\).
      \ENDFOR
      \STATE \(\widetilde{\nabla}_{m} \gets
        \sum_{t = (m - 1) K + 1}^{m K}
        \frac{n}{\delta} f_{t}(y_{t}) u_{t}\)
      \STATE
        \(F_{m}(x) \gets \eta
        \sum_{\tau=1}^{m}
        \innp{\widetilde{\nabla}_{\tau}}{x}
        + \norm[2]{x - x_{0}}^{2}\)
      \STATE\label{line:block-bandit-optimize}
        \(x_{m+1} \in (1 - \alpha) \mathcal{X}\)
        with
        \(\max_{x \in (1 - \alpha) \mathcal{X}}
        F_{m}(x_{m+1}) - F_{m}(x) \leq \varepsilon\)
    \ENDFOR
  \end{algorithmic}
\end{algorithm}

Compared to \citet{garber20a},
we provide a slightly improved regret bound below.
The most important improvement is that the number of linear
minimizations is \emph{always} \(\mathcal{O}(T)\)
instead of only in expectation,
and with an \emph{absolute} constant factor,
i.e., independent of any problem parameters.
\begin{theorem}
  \label{thm:block-bandit-rate}
  Let $f_1, \dots, f_T$ be a sequence of
  uniformly bounded (\(\norm[\infty]{f_t} \leq M\)),
  $G$-Lipschitz convex functions over $\mathcal{X}$ with diameter $D$.
  Assume that the constraint set \(\mathcal{X}\) contains
  a ball of radius $r$ centered at the origin. Let \(\eta = \frac{c D}{n M} T^{-3/4}\),
  \(\delta = c T^{-1/4}\),
  \(\varepsilon = D^{2} T^{-1/2}\),
  \(K = T^{1/2}\)
  and \(\alpha=\delta/r < 1\),
  where \(c>0\) is a constant.
  Then the Block Bandit Conditional Gradient Method
  (Algorithm~\ref{alg:block-bandit})
  after \(T\) rounds has regret
  \begin{equation}
    \label{eq:block-bandit-pure-rate}
    \E{\mathcal{R}_{T}}
    \leq
    \left(
      3 c G + \frac{c G D}{r} + G D
      + \frac{3 D n M}{2 c}
      + \frac{c D G^{2}}{2 n M}
    \right)
    T^{3/4}
    .
  \end{equation}
  Using the vanilla Frank–Wolfe algorithm
  in Line~\ref{line:block-bandit-optimize}
  (with stopping criterion that the Frank–Wolfe gap is
  at most \(\varepsilon\)),
  the algorithm performs at most \(8 T\) linear minimizations.
\end{theorem}

Note that the number of required linear minimization calls differs
from \citet{garber20a} in removing the expectation and the
dependence on function parameters due to an improved estimation in the
proof.

For strongly convex loss functions,
a variant of Algorithm~\ref{alg:block-bandit},
analogous to the full information setting mentioned above,
achieves \(\mathcal{O}(T^{2/3} \ln T)\) regret,
slightly worse than the regret for the full information case,
see \citet{garber2020projfree-learn}.

\begin{proof}
The proof of the regret bound is
similar to Theorem~\ref{thm:regret_bound},
hence we point out only the differences.
As before, let
\(x_{m-1}^{*} = \argmin_{(1 - \alpha) \mathcal{X}} F_{m}\)
and \(z \in (1 - \alpha) \mathcal{X}\) be arbitrary.
We obtain the following analogue of Equation~\eqref{eq:bound_regret2}:
\begin{multline}
  \label{eq:block-bandit-regret2}
  \sum_{t=1}^{T} \E{f_{t}(y_{t}) - f_{t}(z)}
  \\ 
  \leq
  \sum_{m=1}^{T/K} \E{\innp{\widetilde{\nabla}_{m}}{x_{m-1}^{*} - z}}
  +
  \sum_{m=1}^{T/K} \sum_{t = (m - 1) K + 1}^{m K}
  \E{\innp{\nabla \hat{f}_{t, \delta}(x_{m})}{x_{m} - x_{m-1}^{*}}}
  + 3 \delta G T
  .
\end{multline}
The main difference is in the estimation of the summands here.
Again, we bound the first sum using Corollary~\ref{cor:rftl}:
\begin{equation}
  \sum_{m=1}^{T/K} \innp{\widetilde{\nabla}_{m}}{x_{m-1}^{*} - z}
  \leq
  D^2 / \eta
  + \frac{1}{2}
  \eta \sum_{m=1}^{T/K} \norm[2]{\widetilde{\nabla}_{m}}^{2}
  .
\end{equation}
We let \(\E(m){\cdot}\) denote expectation conditioned on the run of
algorithm before outer iteration \(m\).
Using that \(\widetilde{\nabla}_{m}\) is a sum of independent
estimators given the run before iteration \(m\),
we obtain
\begin{equation}
 \begin{split}
  \E(m){\norm[2]{\widetilde{\nabla}_{m}}^{2}}
  &
  =
  \Var(m){\widetilde{\nabla}_{m}}
  +
  \norm*[2]{\E(m){\widetilde{\nabla}_{m}}}^{2}
  \\
  &
  =
  \sum_{t = (m - 1) K + 1}^{m K}
  \Var(m){\frac{n}{\delta} f_{t}(y_{t}) u_{t}}
  +
  \norm*[2]{\sum_{t = (m - 1) K + 1}^{m K}
    \nabla \hat{f}_{t, \delta}(x_{m})}^{2}
  \\
  &
  \leq
  \frac{K n^{2} M^{2}}{\delta^{2}}
  + K^{2} G^{2}
  .
 \end{split}
\end{equation}
Thus
\begin{equation}
  \sum_{m=1}^{T/K} \E{\innp{\widetilde{\nabla}_{m}}{x_{m-1}^{*} - z}}
  \leq
  D^2 / \eta
  + \frac{1}{2} \eta \left(
    \frac{n^{2} M^{2} T}{\delta^{2}} + K G^{2} T
  \right)
  .
\end{equation}
To bound
\(\innp{\nabla \hat{f}_{t, \delta}(x_{m})}{x_{m} - x_{m-1}^{*}}\),
note that \(\norm[2]{x_{m} - x_{m-1}^{*}} \leq
\smash[b]{\sqrt{F_{m}(x_{m}) - F_{m}(x_{m-1}^{*})}} 
\allowbreak
\leq \sqrt{\varepsilon}\)
by \(2\)-strong convexity of \(F_{m}\).
Hence
\begin{equation}
  \innp{\nabla \hat{f}_{t, \delta}(x_{m})}{x_{m} - x_{m-1}^{*}}
  \leq
  \norm[2]{\nabla \hat{f}_{t,\delta}(x_{m})}
  \norm[2]{x_{m} -  x_{m-1}^{*}}
  \leq G \sqrt{\varepsilon}
  .
\end{equation}
The rest of the proof of the regret bound
is identical to Theorem~\ref{thm:regret_bound},
yielding
\begin{equation}
  \E{\mathcal{R}_{T}}
  \leq
  \left(
    3 \delta G + \alpha D G + G \sqrt{\varepsilon}
    + \frac{\eta n^{2} M^{2}}{2 \delta^{2}}
    + \frac{1}{2} \eta K G^{2}
  \right)
  T
  + \frac{D^{2}}{\eta}
  .
\end{equation}
%
Finally, the number of linear minimizations is
at most \(8 D ^{2} T / (K \varepsilon)\)
by Theorem~\ref{fw_sub} (using the Frank–Wolfe gap bound).
The claims of the theorem follow by substituting the explicit values
of parameters, chosen to optimize regret up to a constant factor.
The \(K\) and \(\varepsilon\) were chosen to approximately minimize
the number of linear optimizations, while not significantly worsening
the regret bound.
\end{proof}

\section{Distributed conditional gradient methods}
\label{sec:distr-cond-grad}

This section focuses on developing conditional gradient frameworks in
network-structured settings.
Here, the fundamental new resource constraint is \emph{communication}:
data on the problem is distributed among many nodes,
who necessarily have to cooperate for solving the underlying optimization  problem
(see Figure~\ref{fig:networks}).
We shall consider only the case where only a summand \(f_{i}\)
of the objective function \(f\)
is known to each node \(i\), i.e., the problem is
similar to the stochastic setting
disregarding resource constraints
(Equation~\eqref{eq:stochastic_problem}):
\begin{equation}\label{objective_distributed_problem}
  \min_{x\in \mathcal{X}} f(x),
  \quad
  \text{where}
  \quad
  f(x) = \frac{1}{m}\sum_{i=1}^{m} f_i(x)
  .
\end{equation}
Communication cost might depend 
on the amount of information exchanged (e.g., function values,
iterates and gradient estimates),
the distance needed for information to travel,
or on the content of messages,
to account for, e.g., privacy concerns.
Also, nodes might communicate to only a limited set of other nodes,
forming a \emph{communication graph}.

The challenges vary fundamentally based on the networking
model, which we will describe in Section~\ref{networking_scenarios},
and how the problem is distributed.
We will confine ourselves to a basic algorithm, which
will be presented in Section~\ref{distributed_algorithms}.

\subsection{Networking models}
\label{networking_scenarios}

In general, there are three main networking models for distributed
optimization depending on whether or not a centralized unit is present
and how its role is defined.

\begin{description}
\item[The decentralized model.]
  The decentralized model is the most general one:
  it is the model we described above in which a communication graph
  encodes which nodes can send messages to each other
  (see Figure~\ref{fig:networks}).
Decentralized optimization methods have become core aspects of many
domains such as Internet of Things (IoT) \citep{abu2013data},
multi-robot systems \citep{tanner2005towards}, (wireless) sensor
networks \citep{rabbat2004distributed,Schizas2008-1}, remote
sensing \citep{ma2015remote}, and large-scale learning
\citep{assran2018stochastic, koloskova2019decentralized,
  BoydEtalADMM11}.
These methods address important challenges of modern technology such
as privacy and distributed computation, and data processing
\citep{yang2017survey}. 

\begin{figure}[t]
  \footnotesize
  \definecolor{node}{rgb}{0.01568628 0.2 1}
  \hfil
  \begin{tikzpicture}[x=.2pt, y=.2pt,
    baseline=(current bounding box.center),
    node/.style={circle, fill=node, draw, line width=0.33pt,
      inner sep=0pt,
      minimum size=4.5pt},scale=.9]
    \node foreach \p/\n in {(34.325, 352)/5,
      (73, 185.716)/m,
      (225.307, 341.922)/9,
      (112.605, 314.019)/7,
      (201.616, 507.971)/1,
      (466.463, 235.8)/6,
      (368.907, 367.033)/4,
      (259.855, 119.983)/8,
      (453.998, 92.383)/i,
      (632.912, 348.001)/3,
      (466.4635, 507.972)/2}
    [node, at={\p}] (\n) {}
    foreach \n in {1,2,3,4}
      {node also [label=below:\(\n\), label=above:\(f_{\n}\)] (\n)}
    node also [label=below:\(i\), label=above left:\(f_{i}\)] (i)
    node also [label=below:\(m\), label=left:\(f_{m}\)] (m);
    \path graph{
      (1) -- {(2), (4)} -- (3) -- {(i) -- (8), (6) -- (4)} -- (9)
      -- (5) -- {(1), (m)} -- (7), (6) -- (i)};
    \fill[black] foreach \p in {
        (347.444, 258.969), (322.029, 231.372), (290.55,  201.224)}
      {circle[at={\p}, radius=6.67]};
  \end{tikzpicture}
  \hfil
  \begin{tikzpicture}[baseline=(current bounding box.center)]
    \tikzset{mobile phone/.pic={
        \begin{scope}[x=.2/3pt, y=.2/3pt]
          \draw[line width=1pt, rounded corners=8/3]
          (-100,284) rectangle (100,-51);
          \fill (-79,0) rectangle (79,234);
          \filldraw (0,-23) circle[radius=10];
          \draw[draw=white, rounded corners=.3, ultra thin]
          (-7,-16) rectangle (7,-29);
        \end{scope}}}
    \newcommand{\nodeimage}[1]{%
      \includegraphics[height=1.2em]{#1}\hspace{1ex}}
    \newcommand{\nodeimages}[1]{\left\{\vcenter{\hbox{%
            \forcsvlist{\nodeimage}{#1}%
            \unskip}}\right\}}
    \matrix[matrix of math nodes] (nodes) {
      |(node-1)| \tikz[baseline] \pic{mobile phone}; &
      |(node-2)| \tikz[baseline] \pic{mobile phone}; & \dots &
      |(node-m)| \tikz[baseline] \pic{mobile phone}; \\
      |[label=below:f_{1}]| \nodeimages{hill2,yellow-flower,bird}
      &|[label=below:f_{2}]| \nodeimages{plane-dark,sunset, plane}
      &&|[label=below:f_{m}]| \nodeimages{hill1,hill3,orange-flower}
      \\};

    \node(coordinator) [fill=cyan, above=of nodes, shape=cloud,
    cloud ignores aspect, cloud puffs=10, cloud puff arc=90]
    {coordinator};

    \graph{(coordinator) <->[shorten <=1ex] {(node-1), (node-2), (node-m)}};
  \end{tikzpicture}
  \hfil
  \caption{Communication networks for distributed optimization
    of \((1/m) \sum_{i=1}^{m} f_{i}\), where each~\(f_{i}\) is only
    accessible to node \(i\).
    Left: decentralized network with arbitrary communication links
    between nodes.
    Right: federated network, where nodes only communicate with a
    coordinator node; here mobile devices with a cloud server.}
\label{fig:networks}
\end{figure}
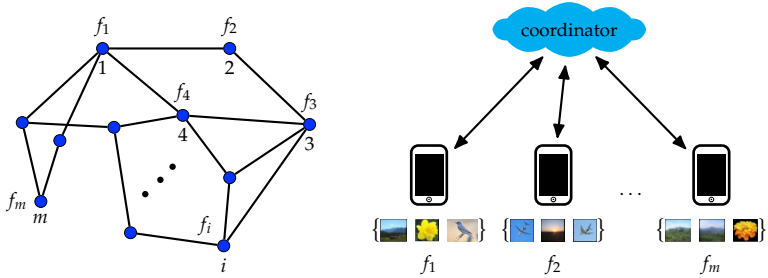

\item[The federated model and the master-worker model.]
  In the federated setting,
  all nodes communicate only to a central coordinator,
  however with constraints on information content still in place,
  e.g., due to privacy
  \citep{mcmahan2016communication, konevcny2015federated}.
  The intended role of the central coordinator is solely
  to facilitate communication between the other nodes (leaves). 
  It has no information on the objective function
  and has limited computational resources compared to the leaves.
  In other
  words, the communication graph in the federated model is a star
  graph with the central coordinator in the center and the computing
  units as the leaves, where the leaves perform the main computation (see Figure~\ref{fig:networks}).
  We are not aware of any use of conditional gradients
  for the federated model.

  The federated model can be considered as a special case of the
  master-worker model.  Similar to the federated model, the
  master-worker model also considers leaf nodes, called \emph{workers},
  that communicate only to a central coordinator, the \emph{master}.
  However, here, the master node can be much more powerful:
  it can have huge computational
  resources, with no restriction on information content exchanged between the workers and the master. 
  The main
  application of the master-worker model is in massive-scale
  distributed computing \citep{bellet2014distributed,
    tran2015distributed, wang2016parallel, moharrer2019distributing}.
\end{description}

Throughout this section, we will mainly focus on the decentralized setting. We will give pointers to the literature on conditional gradient methods in the other networking models at the end of the section. 

\newpage

\subsection{The decentralized setting}
\label{distributed_algorithms}

Decentralized optimization is a
relatively mature area with a vast number of primal and dual
algorithms. More recently, decentralized non-convex optimization
has been extensively studied.
In this section,
we will describe one of the main
conditional gradient methods for the decentralized setting
with similar building blocks as other methods.
For simplicity, we assume
that the local functions $f_i$ are convex.

Recall that in the decentralized setting
the goal is to minimize \(f = (1/m) \sum_{i=1}^{m} f_{i}\)
where node \(i\) has access only to the summand \(f_{i}\).
A naive Frank–Wolfe algorithm is to have the nodes jointly doing a Frank–Wolfe
algorithm by sharing the gradients of the local functions \(f_{i}\)
in every iteration, which in its simplest form involves communicating
a huge amount of data.
To reduce communication,
we employ the local information aggregation, a standard technique
in decentralized optimization, which here roughly means that
instead of sending gradients of its own function \(f_{i}\),
every node \(i\) sends the average of the gradient of \(f_{i}\)
with other gradients received in preceding rounds from other nodes.
This necessitates to take a weighted average of these averages
later on to obtain the gradient of \(f\).
While aggregation reduces the amount of communication,
it does not reduce communication paths: some information still needs
to travel the diameter (maximum length of shortest paths)
of the communication graph, as long as we insist on every node
obtaining the true gradient of \(f\).
Instead the protocol requires communicating only information
already available to neighbours at each iteration, so that
communication delay is minimal, at the price of nodes having inexact
and possibly differing gradient estimators of \(f\),
which is compensated by stochastic methods and
synchronizing the iterates between the nodes.
In the end, compared to traditional consensus optimization methods
which only exchange nodes' local iterates
\citep{nedic2009distributed,nedic2010constrained},
this method also exchanges local gradients, due to
conditional gradient methods requiring accurate gradient estimates.
Indeed, there are settings for which exchanging only local iterates
provides arbitrarily bad solutions \citep{mokhtari2018decentralized}.

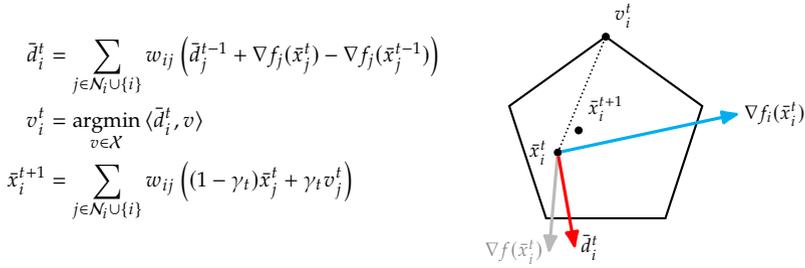
\begin{figure}
\footnotesize
    \centering
  \begin{math}\displaystyle
    \begin{aligned}
      \bar{d}_{i}^{t} &=
      \sum_{j\in \ccalN_i \cup\{i\}} w_{ij}
      \left(
        \bar{d}_{j}^{t-1}
        + \nabla f_{j}(\bar{x}_{j}^{t})
        - \nabla f_{j}(\bar{x}_{j}^{t-1})
      \right)
      \\
      v _{i}^{t} &= \argmin_{ v \in \mathcal{X}} \innp{\bar{d}_{i}^{t}}{v}
      \\
      \bar{x}_{i}^{t+1} &=
      \sum_{j \in \ccalN_{i}\cup\{i\}} w_{ij}
      \left(
         (1 - \gamma_{t}) \bar{x}_{j}^{t} + \gamma_{t} v_{j}^{t}
       \right)
    \end{aligned}
  \end{math}
  \hfil
  \begin{tikzpicture}[x=.18pt, y=.18pt,scale=.7,
    baseline=(current bounding box.center)]
      \draw
      (90:302) node(v-FW)[point, label=above right:\(v_{i}^{t}\)]{}
      \foreach \x in {1, 2, 3, 4}
      {-- (90 + \x*72:302)}
      -- cycle;

        \path
        (198:150)
        node(xi)[point, label={left:\(\bar{x}_{i}^{t}\)}]{}
        edge[vector, draw=black!27]
        node[left, at end, text=black!44.94]
        {\(\nabla f(\bar{x}_{i}^{t})\)}
        +(-95:300)
        edge[alignment line]
        node[right=16, pos=0.17, point,
        label={above right:\(\bar{x}_{i}^{t+1}\)}]{}
        (v-FW)
        edge[vector, draw=red]
        node[right, at end] {\(\bar{d}_{i}^{t}\)}
        +(-80:290)
        edge[vector, draw=cyan]
        node[right, at end] {\(\nabla f_{i}(\bar{x}_{i}^{t})\)}
        (10:400)
        ;
    \end{tikzpicture}
  \caption{A single step of Decentralized Frank–Wolfe algorithm
    (Algorithm~\ref{algo_DFW})
    at node $i$.
    Like for stochastic algorithms, an estimator \(\bar{d}_{i}^{t}\)
    is used instead of the true gradient, which here is the
    distributed analogue of the estimator used in the
    One-Sample Stochastic Frank–Wolfe algorithm
    (Algorithm~\ref{algo_1SFW}).
    In particular, the \(w_{ij}\) are constant weights
    for incorporating data from neighbors,
    optimized for the actual communication network.
    Communication is implicit in this formulation:
    node \(i\) shares its estimators only with the set
    \(\mathcal{N}_{i}\) of its neighbours,
    Unlike other conditional gradient algorithms,
    the next iterate \(\bar{x}_{i}^{t+1}\)
    incorporates updates from other nodes besides line search,
    and hence may not lie on the
    segment between the old iterate \(\bar{x}_{i}^{t}\)
    and the Frank–Wolfe vertex \(v_{i}^{t}\).  }
  \label{DFW}
\end{figure}

These ideas with a gradient estimator similar to that of
One-Sample Stochastic Frank–Wolfe algorithm
(Algorithm~\ref{algo_1SFW})
lead to Algorithm \ref{algo_DFW}
(see also Figure~\ref{DFW} for summary and illustration),
where \(\ccalN_{i}\) is the set of neighbors of node \(i\),
the \(x_{i}^{t}\) are the iterates generated by node \(i\),
and \(d_{i}^{t}\) is the estimate of
the gradient \(\nabla f(x_{i}^{t})\) at node \(i\).
The \(x_{i}^{t}\) and \(d_{i}^{t}\) are sent to neighboring nodes
and the \(w_{ij}\) are carefully chosen weights
for incorporating neighbors' data.
The algorithm need not start at a shared point between the
nodes as presented here.
We do this only to simplify the convergence result,
Theorem \ref{theorem:main_theorem} below.
The nodes can start at arbitrary, independent points with
similar convergence.

\begin{algorithm}[htbp]
  \caption{Decentralized Frank–Wolfe (DeFW) at node $i$
    \citep{wai2017decentralized}}
  \label{algo_DFW}
\begin{algorithmic}[1] 
  \REQUIRE Weights $w_{ij}$ for $j\in\ccalN_{i}\cup\{i\}$,
    common initial point \(\bar{x}^{0} \in \mathbb{R}^{n}\),
    step sizes $\gamma_{t}$
   \ENSURE Iterates $\bar{x}^{1}, \dotsc \in \mathcal{X}$
   \STATE $\bar{x}_{i}^{0} \gets \bar{x}^{0}$
   \STATE \(d_{i}^{0} \gets \nabla f_{i}(\bar{x}_{i}^{0})\)
   \FOR {\(t=0\) \TO \dots}
   \STATE Send $ d_i^{t}$ to neighboring nodes ${j\in \ccalN_i }$
    \STATE $\bar{d}_i^{t} \gets \sum_{j\in \ccalN_i\cup\{i\}}  w_{ij} d _{j}^{t}$
    \STATE $ v _i^t \gets
      \argmin_{ v \in \mathcal{X}} \innp{\bar{d}_i^{t}}{v}$
       \label{line:distr-descent-update}
     \STATE $x _i^{t+1} \gets (1 - \gamma_{t}) \bar{x}_i^{t}
       + \gamma_{t} v _i^t$
       \label{line:distr-variable-update}
    \STATE Send $ x _i^{t+1}$ to neighboring nodes $j\in \ccalN_{i}$
    \STATE $\bar{x}_{i}^{t+1} \gets
      \sum_{j\in \ccalN_i\cup\{i\}}  w_{ij} x _{j}^{t+1}$
      \label{line:distr-variable-mixing}
    \STATE   \label{line:distr-gradient-update}
      $d_{i}^{t+1} \gets \bar{d}_{i}^{t}
      + \nabla f_{i}(\bar{x}_{i}^{t+1})
      -  \nabla f_{i}(\bar{x}_{i}^{t})$
\ENDFOR
\end{algorithmic}\end{algorithm}

\subsubsection{Efficiency of aggregation}

We now describe a standard condition on the weights \(w_{ij}\)
used for the update rules
from \citet{yuan2016convergence}
to ensure efficient local aggregation,
which is essentially a weighted form of the graph being an expander.
Let $W$ be the \(m \times m\) matrix with entries \(w_{ij}\).
We use $\norm{W}$ to denote the spectral norm of matrix $W$,
which is also the operator norm in the \(\ell_{2}\)-norm.
Recall that $\allOne\in \mathbb{R}^{m}$
is the vector with all entries \(1\).

\begin{assumption}\label{ass:weights}
  The weights \(w_{ij}\) that nodes assign to each other
  form a symmetric doubly stochastic matrix \(W\)
  with a positive spectral gap \(1 - \beta\), i.e.,
  \begin{equation}\label{eqn_conditions_on_weights}
    W^{\top}=W \geq 0, \quad
    W\allOne=\allOne, \quad
    \beta \defeq \norm{W - \allOne \allOne^{\top} / m} < 1
    .
  \end{equation}
  Furthermore,
  the weights between distinct unconnected nodes is \(0\),
  i.e., \(w_{ij} = 0\) for \(j \notin \ccalN_{i}\) and \(j \neq i\).
\end{assumption}

A doubly stochastic matrix is a nonnegative matrix
where the sum of the entries in every row and column is \(1\).
For \(W\), this means \(w_{ij} \geq 0\), $w_{ij}=w_{ji}$,
and $\sum_{j=1}^{m} w_{ij}=1$ for all
$i$, so that aggregation is a convex combination.
For the condition on the spectral gap,
recall that the largest eigenvalue of any doubly stochastic \(W\)
in absolute value is \(1\) with eigenvector \(\allOne\),
so the condition states that
the eigenvalues of \(W\) satisfy
$1 = \lambda_{1}(W) > \lambda_{2}(W)
\geq \dots  \geq  \lambda_{m}(W) > -1$,
and the spectral gap is \(1 - \max\{\abs{\lambda_{2}(W)},
\abs{\lambda_{m}(W)}\}\).
As the matrix \(W - \allOne \allOne^{\top} / m\) has eigenvalues
\(0, \lambda_{2}(W), \dotsc, \lambda_{m}(W) \),
one has
\(\norm{W - \allOne \allOne^{\top} / m} =
\max\{\abs{\lambda_{2}(W)}, \abs{\lambda_{m}(W)}\}\),
and thus the spectral gap is
\(1 - \norm{W - \allOne \allOne^{\top} / m} = 1 - \beta\).

For the existence of weights satisfying these conditions,
it is enough that the communication graph is connected,
which is clearly a necessary condition for any distributed algorithm
to converge.
For any connected graph, finding a weight matrix $W$
with maximal spectral gap is a convex optimization problem,
solvable via semidefinite programming \citep{boyd2004fastest}.
We emphasize that $W$ is not a problem parameter, but a conscious choice
prior to running DeFW.

For distributed gradient descent methods, the
spectral gap appears in both the upper and lower bounds 
on the convergence rate \citep{duchi2012dual}.
As a rule of thumb, a larger spectral gap $1-\beta$
yields faster convergence.

\subsubsection{Convergence analysis}
We now study the convergence properties of DeFW, following
\citet[Section~3]{wai2017decentralized} with some simplifications.
As usual, we denote by \(D\) the diameter of the feasible region
\(\mathcal{X}\).

The main idea of the convergence proof is to treat the algorithm
as a stochastic Frank–Wolfe algorithm, treating the \(x_{i}^{t}\)
as approximations of the true iterates, which we define as averages
of the local iterates:
\begin{equation} \label{average_x}
\bar{ x }^{t}= \frac{1}{m}\sum_{i=1}^{m}  x _i^t.
\end{equation}
Obviously, all the local iterates \(x_{i}^{t}\) and their average
\(\bar{x}^{t}\) lie in the feasible region.
The following lemma shows that the \(\bar{x}^{t}\)
satisfy an update rule like for conditional gradients,
as local aggregation has no effect on the average.

\begin{lemma}\label{lemma:ar_in_avg_bound}
  For Algorithm DeFW (Algorithm~\ref{algo_DFW}),
  the Euclidean distance between two consecutive average vectors
  (defined in Equation~\eqref{average_x})
  is upper bounded by
  \begin{equation}
    \label{eq:var_in_avg_bound}
    \norm{\bar{x}^{t+1} - \bar{x}^{t}} \leq D\gamma_{t}.
  \end{equation}
\end{lemma}
\begin{proof}
By averaging both sides of the update in
Line~\ref{line:distr-variable-update} over the nodes in the network
and using \eqref{eqn_conditions_on_weights} and the fact that and
$w_{ij}=0$ if $i$ and $j$ are not neighbors, we can write
\begin{equation}
  \label{proof_avg_x}
 \begin{split}
  \bar{x}^{t+1}
  =
  \frac{1}{m}\sum_{i=1}^{m} x_{i}^{t+1}
  &
  =
  \frac{1}{m} \sum_{i=1}^{m} (1 - \gamma_{t})
  \sum_{j\in \ccalN_{i}\cup\{i\}}  w_{ij} x_{j}^{t}
  +
  \gamma_{t} \frac{1}{m}\sum_{i=1}^{m} v_{i}^{t}
  \\
  &
  =
  \frac{1 - \gamma_{t}}{m} \sum_{i=1}^{m}\sum_{j=1}^{m}  w_{ij} x_{j}^{t}
  +
  \gamma_{t} \frac{1}{m}\sum_{i=1}^{m} v_{i}^{t}
  \\
  &
  =
  \frac{1 - \gamma_{t}}{m} \sum_{j=1}^{m}  x_{j}^{t}
  +
  \gamma_{t} \frac{1}{m}\sum_{i=1}^{m} v_{i}^{t}
  \\
  &
  =
  (1 - \gamma_{t}) \bar{x}^{t}
  +
  \gamma_{t} \frac{1}{m}\sum_{i=1}^{m} v_{i}^{t}
  \\
  &
  =
  \bar{x}^{t}
  +
  \gamma_{t} \frac{1}{m}\sum_{i=1}^{m} (v_{i}^{t} - \bar{x}^{t})
  ,
\end{split}
\end{equation}
where the third equality holds since $W^{\top} \allOne =
W \allOne = \allOne$.
Since $v_i^t$ and $\bar{x}^t$ belong to the convex set $\mathcal{X}$,
their distance is bounded by the diameter of \(\mathcal{X}\), i.e.,
$\norm{v_i^t - \bar{x}^{t}} \leq D$.
This inequality and the expression in \eqref{proof_avg_x} yield
\begin{equation}
  \norm{\bar{x}^{t+1} - \bar{x}^{t}} \leq  D \gamma_{t},
\end{equation}
and the claim in \eqref{eq:var_in_avg_bound} follows.
\end{proof}

As motivated by the stochastic approach,
we now bound the analogue of variance,
with the spectral gap of \(W\) providing variance reduction.
We start with the iterates.

\begin{lemma}\label{lemma:eq:bound_on_dif_from_avg}
  For Algorithm DeFW (Algorithm~\ref{algo_DFW}),
the Euclidean distance between the local iterates $ \bar{x} _i^t$
and their average $\bar{ x }^{t}$
(defined in Equation~\eqref{average_x})
is upper bounded by
\begin{equation}
  \frac{1}{m}
  \sum_{i=1}^{m} \norm{\bar{x}_i^t-\bar{ x }^t}
  \leq
  \sqrt{\frac{\sum_{i=1}^{m} \norm{\bar{x}_i^t-\bar{ x }^t}^{2}}{m}}
  \leq
  D
  \sum_{s=0}^{t-1} \gamma_{s} \prod_{\tau=s+1}^{t-1} (1 - \gamma_{\tau})
  \cdot \beta^{t-s}
  .
\end{equation}
In particular, with step sizes \(\gamma_{\tau} = 2 / (\tau + 2)\)
for all \(\tau\),
\begin{equation}
  \label{eq:bound_on_dif_from_avg}
  \frac{1}{m}
  \sum_{i=1}^{m} \norm{\bar{x}_i^t-\bar{ x }^t}
  \leq
  \sqrt{\frac{\sum_{i=1}^{m} \norm{\bar{x}_i^t-\bar{ x }^t}^{2}}{m}}
  \leq
  \frac{2 \beta}{1 - \beta} \cdot \frac{D}{t + 1}
  .
\end{equation}
\end{lemma}

\begin{proof}
The first inequality is the power mean inequality between the
arithmetic mean and the quadratic mean, so we only need to prove the
second one.
The third one follows as a special case by a simple calculation:
\begin{equation}
 \begin{split}
  \sqrt{
    \frac{\sum_{i=1}^{m} \norm{\bar{x}_{i}^{t} - \bar{x}^{t}}^{2}}{m}}
  &
  \leq
  D
  \sum_{s=0}^{t-1} \frac{2 (s + 1)}{t (t + 1)}
  \cdot \beta^{t-s}
  \\
  &
  =
  \frac{2 D}{t (t + 1)}
  \frac{\beta^{t+2} - (t+1) \beta^{2} + t \beta}{(1 - \beta)^{2}}
  \\
  &
  \leq
  \frac{2 D \beta}{(1 - \beta) (t + 1)}
  .
 \end{split}
\end{equation}
To efficiently handle the local iterates \(x_{i}^{t}\),
we make them the columns of a single matrix
\smash{$\bar{X}^{t} \defeq [\bar{x}_1^{t};\dots;\bar{x}_m^{t}] \in \mathbb{R}^{m \times n}$}.
Similarly, we make the Frank–Wolfe vertices the columns of a matrix
\smash{$V^{t} \defeq[v_1^{t};\dots;v_m^{t}]^{\top} \in \mathbb{R}^{m \times n}$}.
The main advantage of this notation is that the mixing of local
iterates in Line~\ref{line:distr-variable-mixing} of the algorithm becomes simply
multiplication with the matrix \(W\),
thus the third equation in Figure~\ref{DFW} becomes
\begin{equation}
  \bar{X}^{t+1} = (1 - \gamma_{t}) \bar{X}^{t} W
  + \gamma_{t} V^t W
  ,
\end{equation}
which yields (c.f., Equation~\eqref{eq:FW-combined-point})
\begin{equation}
  \label{proof_3_100}
  \bar{X}^{t} =
  \prod_{s=0}^{t-1} (1 - \gamma_{s}) \cdot \bar{X}^{0} W^{t}
  + \sum_{s=0}^{t-1} \gamma_{s}
  \prod_{\tau =s}^{t-1} (1 - \gamma_{\tau}) \cdot V^{s} W^{t-s}.
\end{equation}
Note that \(\bar{x}^{t} = \bar{X}^{t} \allOne / m\)
and (as all the \(x_{i}^{0}\) are equal)
\(\bar{X}^{0} = \bar{x}^{0} \allOne^{\top}\),
hence
\begin{equation}
  \bar{X}^{t} - \bar{x}^{t} \allOne^{\top} =
  \bar{X}^{t}
  \left(
    I  - \frac{\allOne \allOne^{\top}}{m}
  \right)
  =
  \sum_{s=0}^{t-1} \gamma_{s} \prod_{\tau=s+1}^{t-1} (1 - \gamma_{\tau})
  \cdot
  V^{s}
  \left(
    W^{t-s} - \frac{\allOne \allOne^{\top}}{m}
  \right)
  .
\end{equation}
Therefore, letting \(\norm[F]{\cdot}\) denote the Frobenius norm,
i.e., the entrywise \(\ell_{2}\)-norm for matrices,
\begin{equation}
  \label{proof_3_500}
 \begin{split}
  \sqrt{\sum_{i=1}^{m} \norm*{\bar{x}_i^t-\bar{ x }^t}^{2}}
  &
  =
  \norm[F]{\bar{X}^{t} - \bar{x}^{t} \allOne^{\top}}
  \\
  &
  \stackrel{(a)}{\leq}
  \sum_{s=0}^{t-1} \gamma_{s} \prod_{\tau=s+1}^{t-1} (1 - \gamma_{\tau})
  \cdot
  \norm[F]{V^{s}}
  \norm*{W^{t-s}-\frac{\allOne\allOne^{\top}}{m}}
  \\
  &
  \stackrel{(b)}{\leq}
  \sqrt{m}D
  \sum_{s=0}^{t-1} \gamma_{s} \prod_{\tau=s+1}^{t-1} (1 - \gamma_{\tau})
  \cdot
  \norm*{W^{t-s} -
    \frac{\allOne \allOne^{\top}}{m}}
  ,
 \end{split}
\end{equation}
where (a) follows since the spectral norm is the matrix norm
induced by \(\ell_{2}\)-norms.  Also, (b)
follows since $\norm{v_i^t} \leq D$ and therefore
$\norm[F]{V^{t}} \leq \sqrt{m}D$.

Recall that the largest eigenvalue of $W$ is \(1\)
and its eigenspace is one-dimensional generated by \(\allOne\),
hence \(W^{s} - \allOne \allOne^{\top} / m\)
has eigenvalues \(0, \lambda_{2}(W)^{s}, \dotsc,
\lambda_{m}(W)^{s}\).
Thus \(\norm{W^{s} - \allOne \allOne^{\top} / m} = \beta^{s}\).
Applying this substitution into the right hand side of \eqref{proof_3_500} yields
\begin{equation*}\label{proof_3_600}
  \sqrt{\sum_{i=1}^{m} \norm{\bar{x}_i^t-\bar{ x }^t}^{2}}
  \leq
  \sqrt{m} D
  \sum_{s=0}^{t-1} \gamma_{s} \prod_{\tau=s+1}^{t-1} (1 - \gamma_{\tau})
  \cdot \beta^{t - s}
  .
  \qedhere
\end{equation*}
\end{proof}

Similarly to the previous lemma,
we now show that
the individual local gradient approximation vectors $ d_i^t$ are close
to the gradient \(\nabla f(\bar{x}^{t})\) at the average of the global
objective function.
Here, we simplify the original proof by
skipping an intermediate approximation, namely,
the average of the \(d_{i}^{t}\).

\begin{lemma}\label{lemma:bound_on_gradient_consensus_error}
  If the \(f_{i}\) are \(L\)-smooth
  and \(\norm{\nabla f_{i}(\bar{x}^{0})} \leq G\),
  then
  Algorithm DeFW (Algorithm~\ref{algo_DFW}) satisfies
  with step sizes \(\gamma_{\tau} = 2 / (\tau + 2)\) for all \(\tau\),
  \begin{equation}
    \frac{1}{m}
    \sum_{i=1}^{m} \norm{d_{i}^{t} - \nabla f(\bar{x}^{t})}
    \leq
    \sqrt{\frac{\sum_{i=1}^{m} \norm{d_{i}^{t} -
          \nabla f(\bar{x}^{t})}^{2}}{m}}
    \leq
    G \beta^{t}
    +
    \frac{2 (1 + \beta)}{(1 - \beta)^{3}} \cdot \frac{L D}{t + 1}
    .
  \end{equation}
\end{lemma}

\begin{proof}
Similarly to the previous lemma, we define matrices
with columns the local iterates and gradient estimators:
$\bar{X}^{t} \defeq [\bar{x}_1^{t}; \dots;
\bar{x}_m^{t}] \in \mathbb{R}^{n \times m}$
and
$\mathbb{D}^{t} \defeq [d_{1}^{t};\dots;d_{m}^{t}]
\in \mathbb{R}^{n \times m}$,
with rows the \(x_{i}^{t}\) and \(d_{i}^{t}\) respectively.
We also define
\(\nabla F(X) \defeq [\nabla f_{1}(x_{1}); \dots ;
\nabla f_{m}(x_{m})]\),
where the \(x_{i}\) are the rows of \(X\).
With these definitions
\(\nabla F\) is \(L\)-Lipschitz continuous in the Frobenius
norm \(\norm[F]{\cdot}\) and
the update rule for the \(d_{i}^{t}\) becomes
\begin{equation}
  \label{eq:gradient-consensus-update}
  \mathbb{D}^{t+1} = \mathbb{D}^{t} W
  + \nabla F(\bar{X}^{t+1})
  - \nabla F(\bar{X}^{t})
  .
\end{equation}
In particular, as \(W\) is doubly stochastic
(hence \(W \allOne = \allOne\)),
we obtain a simpler formula for the sum of rows:
\begin{equation}
  \mathbb{D}^{t+1} \allOne = \mathbb{D}^{t} \allOne
  + \nabla F(\bar{X}^{t+1}) \allOne
  - \nabla F(\bar{X}^{t}) \allOne
  .
\end{equation}
Thus \(\mathbb{D}^{t} \allOne =
\nabla F(\bar{X}^{t}) \allOne\)
(i.e., \(\sum_{i=1}^{m} d_{i}^{t} =
\sum_{i=1}^{m} \nabla f_{i}(\bar{x}_{i}^{t})\))
follows by induction.
We use this to modify \eqref{eq:gradient-consensus-update} to
get rid of the largest eigenvalue of \(W\),
which is the key for controlling accuracy,
as we have already seen:
\begin{equation}
  \mathbb{D}^{t+1} =
  \mathbb{D}^{t}
  \left(
    W - \frac{\allOne \allOne^{\top}}{m}
  \right)
  + \nabla F(\bar{X}^{t+1})
  -
  \nabla F(\bar{X}^{t})
  \left(
    I - \frac{\allOne \allOne^{\top}}{m}
  \right)
  .
\end{equation}
We know approximate every term with a gradient taken at
\(\bar{x}^{t+1}\) or \(\bar{x}^{t}\).
Using double stochasticity of \(W\) again, this leads to
\begin{equation}
 \begin{split}
  \mathbb{D}^{t+1}
  - \nabla f(\bar{x}^{t+1}) \allOne^{\top}
  &
  =
  \mathbb{D}^{t+1}
  -
  \nabla F\bigl(\bar{x}^{t+1} \allOne^{\top}\bigr)
  \frac{\allOne \allOne^{\top}}{m}
  \\
  &
  =
  \begin{multlined}[t]
    \bigl(\mathbb{D}^{t} - \nabla f(\bar{x}^{t}) \allOne^{\top}\bigr)
     \left(
      W - \frac{\allOne \allOne^{\top}}{m}
    \right)
    + \left(
      \nabla F(\bar{X}^{t+1})
      - \nabla F\bigl(\bar{x}^{t+1} \allOne^{\top}\bigr)
    \right)
    \\
    +
    \left(\nabla F\bigl(\bar{x}^{t+1} \allOne^{\top}\bigr)
    - \nabla F(\bar{X}^{t})\right)
    \left(
      I - \frac{\allOne \allOne^{\top}}{m}
    \right)
    .
  \end{multlined}
 \end{split}
\end{equation}
Estimating the norms on both sides,
and using \(L\)-Lipschitz continuity of \(\nabla F\),
we obtain
\begin{equation}
 \begin{split}
  \MoveEqLeft
  \norm[F]{\mathbb{D}^{t+1} - \nabla f(\bar{x}^{t+1}) \allOne^{\top}}
  \\
  &
  \leq
  \begin{aligned}[t]
  \beta
  \norm[F]{\mathbb{D}^{t} - \nabla f(\bar{x}^{t}) \allOne^{\top}}
    &
    + \norm[F]{
      \nabla F(\bar{X}^{t+1})
      - \nabla F(\bar{x}^{t+1} \allOne^{\top})}
    \\
    &
    + \norm[F]
    {\nabla F(\bar{x}^{t+1} \allOne^{\top})
      - \nabla F(\bar{X}^{t})}
  \end{aligned}
  \\
  &
  \leq
  \beta
  \norm[F]{\mathbb{D}^{t} - \nabla f(\bar{x}^{t}) \allOne^{\top}}
  + L \norm[F]{\bar{X}^{t+1} - \bar{x}^{t+1} \allOne^{\top}}
  + L \norm[F]{\bar{x}^{t+1} \allOne^{\top} - \bar{X}^{t}}
  \\
  &
  \leq
  \begin{aligned}[t]
  \beta
  \norm[F]{\mathbb{D}^{t} - \nabla f(\bar{x}^{t}) \allOne^{\top}}
  + L \norm[F]{\bar{X}^{t+1} - \bar{x}^{t+1} \allOne^{\top}}
  &
  + L \norm[F]{\bar{x}^{t+1} \allOne^{\top} - \bar{x}^{t} \allOne^{\top}}
  \\
  & 
  + L \norm[F]{\bar{x}^{t} \allOne^{\top} - \bar{X}^{t}}
  \end{aligned}
  \\
  &
  \leq
  \beta
  \norm[F]{\mathbb{D}^{t} - \nabla f(\bar{x}^{t}) \allOne^{\top}}
  + \frac{2 \beta \sqrt{m} L D}{(1 - \beta) (t + 2)}
  + \frac{2 \sqrt{m} L D}{t + 2}
  + \frac{2 \beta \sqrt{m} L D}{(1 - \beta) (t + 1)}
  \\
  &
  \leq
  \beta
  \norm[F]{\mathbb{D}^{t} - \nabla f(\bar{x}^{t}) \allOne^{\top}}
  + \frac{2 (1 + \beta) \sqrt{m} L D}{(1 - \beta) (t + 1)}
  ,
 \end{split}
\end{equation}
where we have used
Lemmas~\ref{lemma:ar_in_avg_bound} and
\ref{lemma:eq:bound_on_dif_from_avg}
to estimate the norms for step sizes \(\gamma_{t} = 2 / (t+2)\).
By induction, this provides
\begin{equation}
 \begin{split}
  \sqrt{\sum_{i=1}^{m} \norm{d_{i}^{t} - \nabla f(\bar{x}^{t})}^{2}}
  &
  =
  \norm[F]{\mathbb{D}^{t} - \nabla f(\bar{x}^{t}) \allOne^{\top}}
  \\[-2ex]
  &
  \leq
  \beta^{t}
  \norm[F]{\mathbb{D}^{0} - \nabla f(\bar{x}^{0}) \allOne^{\top}}
  +
  \sum_{s=0}^{t-1}
  \frac{2 (1 + \beta) \sqrt{m} L D}{(1 - \beta) (s + 1)}
  \cdot \beta^{t-s}
  \\
  &
  \leq
  \beta^{t}
  \norm*[F]{
    \mathbb{D}^{0}
    \left(
      I - \frac{\allOne \allOne^{\top}}{m}
    \right)}
  +
  \frac{2 (1 + \beta) \sqrt{m} L D}{(1 - \beta) (t + 1)}
  \sum_{s=0}^{t-1}
  (t - s + 1) \cdot \beta^{t-s}
  \\
  &
  \leq
  \beta^{t} \sqrt{m} G
  +
  \frac{2 (1 + \beta) \sqrt{m} L D}{(1 - \beta) (t + 1)}
  \sum_{s=0}^{\infty}
  (s + 1) \cdot \beta^{s}
  \\
  &
  =
  \beta^{t} \sqrt{m} G
  +
  \frac{2 (1 + \beta) \sqrt{m} L D}{(1 - \beta)^{3} (t + 1)}
  .
  \qedhere
 \end{split}
\end{equation}
\end{proof}

With the accuracy bound for gradient estimators in the previous lemma,
we are ready for the convergence result of the decentralised
algorithm.
Note that each iteration of Algorithm~\ref{algo_DFW} consists of two rounds of communication: In the first round, the nodes send their estimate of the gradient, $d_i^t$, to their neighbors, and in the second round they send their local iterates, $x_i^{t+1}$, to their neighbors. From this perspective, the communication complexity of DeFW per round is twice the complexity of typical decentralized methods (which require communicating only the iterates at each round).

\begin{theorem}\label{theorem:main_theorem}
  Let \(\mathcal{X}\) be a compact convex set
  with diameter at most \(D\)
  and the \(f_{i} \colon \mathcal{X} \to \mathbb{R}\) be
  convex \(L\)-smooth functions.
  Let \(x^{*}\) be a minimizer of
  \smash{\(f \defeq (1/m) \sum_{i=1}^{m} f_{i}\)}.
  Let the gradients at the initial point be bounded:
  \(\norm{\nabla f_{i}(\bar{x}^{0})} \leq G\).
  Let \(\beta < 1\) be the magnitude of the second largest eigenvalue
  of the weight matrix \(W\) for the communication graph.
  Then the distributed algorithm DeFW (Algorithm~\ref{algo_DFW})
  computes solutions \(x_{i}^{t}\) after \(t\) iterations
  satisfying for all \(t \geq 1\),
  \begin{equation}
    \label{local_node_bound}
    \frac{1}{m} \sum_{i=1}^{m} \bigl(f(x_{i}^{t}) - f(x^{*})\bigr)
    \leq
    \left(
      1 + \frac{2 (1 + \beta)}{(1 - \beta)^{3}}
    \right)
    \frac{2 L D^{2}}{t + 1}
    +
    \frac{2 G D}{(1 - \beta)^{2} t (t + 1)}
    .
  \end{equation}
  Equivalently, to obtain solutions with an average primal gap at most
  \(\varepsilon > 0\),
  the algorithm makes
  \(\mathcal{O}(L D^{2} / ((1 - \beta)^{3} \varepsilon) +
  G D / ((1 - \beta)^{2} \sqrt{\varepsilon}))\)
  iterations.
  Every node performs one
  linear minimization in every iteration,
  and sends its gradient estimate and local iterate
  to all its neighbours in every iteration.
\end{theorem}

The convergence rate on the average obviously implies
an \(\mathcal{O}(m / \varepsilon)\)-convergence in primal gap
for \emph{each node}.
It is worth mentioning that like classical results in
decentralized optimization,
the additional terms compared to the non-distributed setting
vanish faster for communication graphs with larger spectral gap $1-\beta$.

\begin{proof}
The proof is similar to the convergence proofs for
stochastic conditional gradients.
We start by the distributed version of
the stochastic progress inequality \eqref{eq:SFW-step}.
Using \(L\)-smoothness of the global objective function \(f\)
and minimality of the Frank–Wolfe vertices \(v_{i}^{t}\),
we obtain
\begin{equation}
  \label{eq:DFW-step}
 \begin{split}
  \MoveEqLeft
  f(x_{i}^{t+1}) - f(\bar{x}^{t})
  \leq
  \gamma_{t}
  \innp{\nabla f(\bar{x}^{t})}{v_i^t - \bar{x}^{t}}
  + \gamma_{t}^{2} \frac{L D^{2}}{2}
  \\
  &
  =
  \gamma_{t}
  \innp{\nabla f(\bar{x}^{t})}{x^{*} - \bar{x}^{t}}
  +
  \gamma_{t}
  \innp{\nabla f(\bar{x}^{t}) - \bar{d}_{i}^{t}}{v_{i}^{t} - x^{*}}
  + \gamma_{t} \innp{\bar{d}_{i}^{t}}{v_{i}^{t} - x^{*}}
  + \gamma_{t}^{2} \frac{L D^{2}}{2}
  \\
  &
  \leq
  \gamma_{t} (f(x^{*}) - f(\bar{x}^{t}))
  +
  \gamma_{t}
  \norm{\nabla f(\bar{x}^{t}) - \bar{d}_{i}^{t}} D
  + \gamma_{t}^{2} \frac{L D^{2}}{2}
  .
 \end{split}
\end{equation}
We use Lemma~\ref{lemma:bound_on_gradient_consensus_error}
to bound the error in the gradient estimate,
after taking average over the nodes.
Recall that \(\bar{x}^{t} = (1/m) \sum_{i=1}^{m} x_{i}^{t}\),
and hence by Jensen's inequality,
\(f(\bar{x}^{t}) \leq (1/m) \sum_{i=1}^{m} f(x_{i}^{t})\).

\begin{equation}
 \begin{split}
  \MoveEqLeft
  \frac{1}{m} \sum_{i=1}^{m}
  \bigl(f(x_{i}^{t+1}) - f(x^{*})\bigr)
  - \frac{1 - \gamma_{t}}{m} \sum_{i=1}^{m}
  \bigl(f(x_{i}^{t}) - f(x^{*})\bigr)
  \\
  &
  \leq
  \frac{1}{m} \sum_{i=1}^{m}
  \bigl(f(x_{i}^{t+1}) - f(x^{*})\bigr)
  - (1 - \gamma_{t}) \bigl(f(\bar{x}^{t}) - f(x^{*})\bigr)
  \\
  &
  \leq
  \frac{\gamma_{t}}{m} \sum_{i=1}^{m}
  \norm{\nabla f(\bar{x}^{t}) - \bar{d}_{i}^{t}} D
  + \gamma_{t}^{2} \frac{L D^{2}}{2}
  \\
  &
  \leq
  D \gamma_{t}
  \left(
    G \beta^{t}
    +
    \frac{2 L D (1 + \beta)}{(1 - \beta)^{3} (t + 1)}
  \right)
  + \gamma_{t}^{2} \frac{L D^{2}}{2}
  \\
  &
  \leq
  \frac{2 G D \beta^{t}}{t + 2}
  +
  \left(
    1 + \frac{2 (1 + \beta)}{(1 - \beta)^{3}}
  \right)
  \frac{2 L D^{2}}{(t+1) (t + 2)}
  .
 \end{split}
\end{equation}
We rearrange to more easily sum up the inequalities.
\begin{multline}
  \frac{(t + 1) (t + 2)}{m} \sum_{i=1}^{m}
  \bigl(f(x_{i}^{t+1}) - f(x^{*})\bigr)
  - \frac{t (t + 1)}{m} \sum_{i=1}^{m}
  \bigl(f(x_{i}^{t}) - f(x^{*})\bigr)
  \\
  \leq
  2 G D (t + 1) \beta^{t}
  +
  \left(
    1 + \frac{2 (1 + \beta)}{(1 - \beta)^{3}}
  \right)
  2 L D^{2}
  .
\end{multline}
Summing up leads to a telescope sum,
and we obtain
\begin{equation}
 \begin{split}
  \frac{t (t + 1)}{m} \sum_{i=1}^{m}
  \bigl(f(x_{i}^{t}) - f(x^{*})\bigr)
  &
  \leq
  2 G D
  \sum_{s=0}^{t-1} (s + 1) \beta^{s}
  +
  \left(
    1 + \frac{2 (1 + \beta)}{(1 - \beta)^{3}}
  \right)
  2 L D^{2} t
  \\
  &
  \leq
  \frac{2 G D}{(1 - \beta)^{2}}
  +
  \left(
    1 + \frac{2 (1 + \beta)}{(1 - \beta)^{3}}
  \right)
  2 L D^{2} t
  .
  \qedhere
 \end{split}
\end{equation}
\end{proof}

\subsection{Notes}
\label{sec:decentral-notes}
Decentralized optimization has a long history as it has been developed
across several communities.  For decentralized convex optimization,
many primal and dual algorithms have been proposed in the late
1980s \citep{bertsekas1989parallel}. Among primal methods,
decentralized (sub)gradient descent is perhaps the most well-known
algorithm, which is a mix of local gradient descent and successive
averaging \citep{nedic2009distributed,yuan2016convergence}.  It can
also be interpreted as a penalty method that encourages agreement
among neighboring nodes. This latter interpretation has been exploited
to solve the penalized objective function using accelerated gradient
descent \citep{jakovetic2014fast,qu2017accelerated}, Newton's method
\citep{mokhtari2017network,bajovic2017newton}, or quasi-Newton
algorithms \citep{eisen2017decentralized}. The methods that operate in
the dual domain consider a constraint that enforces equality between
nodes' variables and solve the problem by ascending on the dual
function to find optimal Lagrange multipliers. A short list of dual
methods consist of
the alternating directions method of multipliers (ADMM)
\citep{Schizas2008-1,BoydEtalADMM11}, dual ascent
algorithm \citep{rabbat2005generalized}, and augmented Lagrangian
methods
\citep{jakovetic2015linear,chatzipanagiotis2015convergence,mokhtari2016dsa}. 

More recently, algorithms for decentralized non-convex optimization
have been developed in several works
\citep{zhang2017projection,wai2017decentralized, mokhtari2018decentralized, di2016next,sun2016distributed,hajinezhad2016nestt,tatarenko2017non}.
Like the non-distributed case, these converge only
to a stationary point,
but not necessarily to the optimum.

Finally, in this chapter we have mainly focused on the decentralized
setting (see Section~\ref{networking_scenarios}).  The master-worker
model has been studied in \citet{wang2016parallel}.  A MapReduce
implementation of conditional gradient methods is developed in
\citet{moharrer2019distributing}.  Moreover, there is still no work that
extends the conditional gradient methods to the federated setting with
theoretical guarantees.  This is an interesting open problem.

\chapter{Miscellaneous}
\label{cha:misc}

\section{Related methods}
\label{sec:related}

In this chapter, we present methods related to conditional gradients,
which substantially differ from its core idea,
such as, e.g., matching
pursuit and the continuous greedy method for submodular
maximization. We also look at slightly more involved algorithms that, e.g., achieve acceleration but where the underlying algorithmic philosophy is slightly modified.

\subsection{From conditional gradients to matching pursuit}

We will now consider extensions of the Frank–Wolfe algorithm
to the \emph{unconstrained} setting.
The goal is to minimize a
convex function $f\colon\mathbb{R}^n\to\mathbb{R}$
over the linear subspace \(\Span{\mathcal{D}} \subseteq
\mathbb{R}^{n}\) spanned by a given set $\mathcal{D}$,
instead of a bounded convex set:
\begin{equation}
  \label{spanpb}
  \min_{x\in\Span{\mathcal{D}}}f(x).
\end{equation}
We call $\mathcal{D}$ a \emph{dictionary} and its elements are 
referred to as \emph{atoms}.
This problem is feasible if $f$ is coercive, i.e.,
$\lim_{\norm{x}\rightarrow+\infty}f(x)=+\infty$.
Since $f$ is convex,
being coercive is actually a mild assumption.

\subsubsection{Generalized Matching Pursuit}

The natural extension of the Frank–Wolfe algorithm to
Problem~\eqref{spanpb} is Algorithm~\ref{gmp}, where the update does
no longer need to be a convex combination as the problem is
unconstrained.
In fact, it can be seen as a natural generalization of the Matching
Pursuit algorithm \citep{mallat93mp} (outlined in Algorithm~\ref{mp}),
which aims at recovering a sparse representation of a given signal
$y\in\mathbb{R}^n$ via the dictionary $\mathcal{D}$, by solving 
$\min_{x\in\Span{\mathcal{D}}} \norm[2]{y-x}^{2}$, i.e., we want to find a point $x$ that can be written as a sparse linear combination of elements from $\mathcal D$ and approximates $y$ well. The trade-off for this class of problems is typically sparsity vs. closeness of approximation (also referred to as \emph{recovery error}).

\begin{algorithm}[h]
\caption{Matching Pursuit \citep{mallat93mp}}
\label{mp}
\begin{algorithmic}[1]
\REQUIRE Start atom $x_0\in\mathcal{D}$, another vector \(y\)
\ENSURE Iterates $x_{1}, \dotsc \in\Span{\mathcal{D}}$
\STATE$\mathcal{S}_0\leftarrow\{x_0\}$
\FOR{$t=0$ \TO \dots}
\STATE$v_t\leftarrow\argmax_{v\in\mathcal{D}} \abs{\innp{y-x_t}{v}}$
\STATE$\mathcal{S}_{t+1}\leftarrow\mathcal{S}_t\cup\{v_t\}$
\STATE$\gamma_{t} \gets \argmin_{\gamma \in \mathbb{R}}
  \norm[2]{y - x_{t} - \gamma v_{t}}^{2}$
\STATE$x_{t+1} \gets x_{t} + \gamma_{t} v_{t}$
\ENDFOR
\end{algorithmic}
\end{algorithm}

The sparsity of the solution in Matching Pursuit (see Algorithm~\ref{mp}) is obtained by design, similarly to what happens in the Frank–Wolfe algorithm:
only one atom is added in each iteration, chosen as the atom
most correlated with the residual $y-x_t$, which is the gradient
of the objective at $x_t$.
With this observation, a unifying view on the Frank–Wolfe algorithms
and Matching Pursuit algorithms was proposed in
\citet{locatello17mpfw} and led to the \emph{Generalized Matching
Pursuit (GMP)} and \emph{Generalized Orthogonal Matching Pursuit
(GOMP)} algorithms (see Algorithm~\ref{gmp}): GMP and GOMP proceed
very similarly to the Frank–Wolfe algorithm (Algorithm~\ref{fw}), and
are the analog to the Frank–Wolfe algorithm (in the case of GMP) with
line segment search and the Fully-Corrective Frank–Wolfe algorithms
(in the case of GOMP). The subproblem in Line~\ref{gomp_step} can be
implemented using a sequence of projected gradient steps (over the set
$\Span{\mathcal{S}_{t+1}}$); note that projection onto a linear
subspace can be done efficiently using basic linear algebraic
operations as long as \(\mathcal{S}_{t+1}\) is small.
In practice, GMP converges faster in wall-clock time while GOMP offers
(much) sparser iterates at the expense of a much more expensive
subproblem, which however can be solved efficiently in some special
cases; this is very much in line with what we would expect from their
conditional gradient counterparts.

\begin{algorithm}[h]
\caption{Generalized (Orthogonal) Matching Pursuit \citep{locatello17mpfw}}
\label{gmp}
\begin{algorithmic}[1]
  \REQUIRE Start atom $x_{0} \in \mathcal{D}$
  \ENSURE Iterates $x_{1}, \dotsc \in \Span{\mathcal{D}}$
\STATE$\mathcal{S}_0\leftarrow\{x_0\}$
\FOR{$t=0$ \TO \dots}
  \STATE\label{gmpquery}
    $v_t \gets \argmin_{v \in \mathcal{D}} \innp{\nabla f(x_t)}{v}$
\STATE$\mathcal{S}_{t+1}\leftarrow\mathcal{S}_t\cup\{v_t\}$ \\
\textbf{Option 1: GMP variant}
\STATE \hspace{\algorithmicindent} $\gamma_t\leftarrow\argmin_{\gamma \in \mathbb{R}}f(x_t+\gamma v_t)$
\STATE\hspace{\algorithmicindent} $x_{t+1}\leftarrow x_t+\gamma_t v_t$  \\
\textbf{Option 2: GOMP variant}
\STATE \hspace{\algorithmicindent} $x_{t+1}\leftarrow\argmin_{x \in \Span{\mathcal{S}_{t+1}}}f(x)$
  \label{gomp_step}
\ENDFOR
\end{algorithmic}
\end{algorithm}

\pagebreak

GMP and GOMP do not require any assumption on the geometry of the
dictionary~$\mathcal{D}$, and in particular do not require incoherence or restricted isometry
properties (RIP) \citep{candes05rip}, which are usually assumed in the
literature on matching pursuit algorithms in order to obtain
convergence guarantees.
The convergence rates of
GMP and GOMP from \citet[Theorems~13]{locatello17mpfw}
are given in the following theorem.
\begin{theorem}
  \label{th:gmp}
 Let $\mathcal{D} \subseteq \mathbb{R}^{n}$
 be a bounded set with diameter \(D\).
 Let \(f \colon \mathbb{R}^{n} \to \mathbb{R}\) be
 a \(\mu\)-strongly convex, \(L\)-smooth function,
 and \(x^{*}\) a minimum point of \(f\) over \(\Span{\mathcal{D}}\),
 the linear space spanned by \(\mathcal{D}\).
 Then Generalized Matching Pursuit and
 Generalized Orthogonal Matching Pursuit
 algorithms (Algorithm~\ref{gmp}) satisfy
 \begin{equation}
   f(x_t) - f(x^{*})
   \leq \left(1 - \frac{\mu \delta^{2}}{L D^{2}} \right)^{t}
   \left( f(x_0) - f(x^{*}) \right)
 \end{equation}
 for all \(t \geq 0\),
 where \(\delta\) is the pyramidal width of
 (the convex hull of) \(\mathcal{D}\)
 (see Definition~\ref{def:pyramidal}).
 Equivalently, for a primal gap \(\varepsilon > 0\),
 the algorithm makes
 \(\mathcal{O}(\frac{L D^{2}}{\mu \delta^{2}} \log
 ((f(x_{0}) - f(x^{*})) / \varepsilon))\)
 linear minimizations, line searches (for the GMP variant)
 and minimizations over a convex hull (for the GOMP variant).
\end{theorem}

\subsubsection{Blended Matching Pursuit}

Based on
Blended Conditional Gradient algorithm
(Algorithm~\ref{bcg} in Section~\ref{sec:bcg}),
\citet{combettes19bmp} has further developed
Generalized Matching Pursuit just discussed
into \emph{Blended
  Matching Pursuit algorithm (BMP)} (Algorithm~\ref{bmp}),
which possesses the main
properties of GMP (speed) and GOMP (sparsity).
The idea is to consider a GOMP step
is as a sequence of projected gradient steps (PG), BMP mixes
GMP steps 
with
PG steps. 
For projections we use the standard scalar product here,
notice that the projection operation simply
computes the component of $\nabla f(x_t)$
parallel to $\Span{\mathcal{S}_t}$ and can be done efficiently.
The blending of steps proceeds similarly to BCG
(Section~\ref{sec:bcg}) and is controlled via
the Frank–Wolfe gap estimate $\phi_{t}$,
estimating
$\max_{v\in-\mathcal{D}\cup\mathcal{D}}\innp{\nabla f(x_t)}{v}$,
the unconstrained analogous of the Frank–Wolfe gap. It is no longer an
upper bound on the primal gap but it is directly related to the
progress achievable by a PG or a GMP step.  This is why the decision as
to which step to take is based on it (Line~\ref{criterion}).

As we are considering unconstrained optimization,
we will need a slight reformulation of the weak-separation oracle
(Oracle~\ref{ora:LPsep}):

\begin{oracle}[H]
  \caption{Weak-separation $\operatorname{LPsep}_{\mathcal{D}}(c,\phi,K)$}
  \label{bmp-lpsep}
  \begin{algorithmic}
    \REQUIRE Linear objective \(c\),
      accuracy \(K \geq 1\),
      target objective value \(\phi \leq 0\)
    \ENSURE
      \(\begin{cases*}
        v \in \mathcal{D}, & with \(\innp{c}{v} \leq \phi / K\)
        \quad\hfill
        (positive answer)
        \\
        \FALSE, & if \(\innp{c}{z} \geq \phi\) for all
        \(z \in \mathcal{D}\)
        \quad\hfill
        (negative answer)
      \end{cases*}\)
\end{algorithmic}
\end{oracle}

The BMP algorithm is presented in Algorithm~\ref{bmp}. 
It introduces the parameter $\eta>0$ (Line~\ref{criterion}) in order
to adjust for the desired output of the algorithm: setting $\eta$ to a
higher value favors PG steps and therefore induces more sparsity in
the iterates, however trading for speed of convergence.
Lines~\ref{proj}--\ref{projend} compute a PG step,
Lines~\ref{full_step}--\ref{fullend} compute a GMP step, and
Lines~\ref{stationary_step}--\ref{tau} update the Frank–Wolfe gap estimate
$\phi_{t}$.

\begin{algorithm}[t]
\caption{Blended Matching Pursuit (BMP) \citep{combettes19bmp}}
\label{bmp}
\begin{algorithmic}[1]
\REQUIRE Convex function $f \colon \mathbb{R}^{n} \to \mathbb{R}$,
dictionary $\mathcal{D}\subset\mathbb{R}^{n}$,
initial atom $x_0\in\mathcal{D}$, accuracy parameters $K\geq1$ and $\eta>0$,
scaling parameter $\tau>1$
\ENSURE Iterates $x_1, \dotsc \in \Span{\mathcal{D}}$
\STATE$\mathcal{S}_0\leftarrow\{x_0\}$
\STATE$\phi_0 \gets
  \max_{v\in-\mathcal{D} \cup \mathcal{D}} \innp{\nabla f(x_0)}{v}
  \mathbin{/} \tau$
\label{phi0}
\FOR{$t=0$ \TO \dots}\label{for}
\STATE$v_t^{\mathcal{S}}
\leftarrow\argmin_{v\in-\mathcal{S}_t\cup\mathcal{S}_t}\innp{\nabla f(x_t)}{v}$\label{bmp_atom}
\IF{$\innp{\nabla f(x_t)}{v_t^{\mathcal{S}}} \leq - \phi_t / \eta$}
    \label{criterion}
\STATE$\widetilde{\nabla}f(x_t)\leftarrow
\proj_{\Span{\mathcal{S}_t}}(\nabla f(x_t))$
\label{proj}
    \STATE$x_{t+1}\leftarrow
    \argmin_{x_t+\mathbb{R}\widetilde{\nabla}f(x_t)}f$
    \COMMENT{PG step}\\
    \label{constrained_step}
    \STATE$\mathcal{S}_{t+1}\leftarrow\mathcal{S}_t$
    \STATE$\phi_{t+1}\leftarrow\phi_t$\label{projend}
\ELSE
    \STATE$v_t\leftarrow\text{LPsep}_{-\mathcal{D} \cup \mathcal{D}} 
    (\nabla f(x_t),\phi_t,K)$
    \label{call}
    \IF{$v_t=\FALSE$}
    \STATE$x_{t+1}\leftarrow x_t$
    \label{stationary_step}
    \STATE$\mathcal{S}_{t+1}\leftarrow\mathcal{S}_t$
    \STATE$\phi_{t+1}\leftarrow\phi_t/\tau$\COMMENT{Frank–Wolfe gap estimate update}\\\label{tau}
    \ELSE
    \STATE$x_{t+1}\leftarrow
    \argmin_{x_t+\mathbb{R}v_t}f$
    \COMMENT{GMP step}
    \label{full_step}
    \STATE$\mathcal{S}_{t+1}\leftarrow\mathcal{S}_t\cup\{v_t\}$
    \STATE$\phi_{t+1}\leftarrow\phi_t$\label{fullend}
    \ENDIF
  \ENDIF
  \STATE\label{line:actiev-opt}
    Optional: \(\mathcal{S}_{t+1} \gets
    \argmin \{\size{\mathcal{S}} \mid \mathcal{S} \subseteq
    \mathcal{S}_{t+1}, x_{t+1} \in \Span{\mathcal{S}}\}\)
\ENDFOR
\end{algorithmic}
\end{algorithm}

The key convergence result of BMP makes use of the notion of sharpness
(Definition~\ref{def:sharpness}),
a localized generalization of strong convexity
(see Section~\ref{sec:adaptive_rates}).
In a similar fashion,
the notions of smoothness and strong convexity also extend as follows.
The function $f$ is \emph{$L$-smooth of order $\ell>1$}
where $L>0$ is a positive number if for all $x,y\in\mathbb{R}^n$
\begin{equation*}
  f(y) \leq f(x) + \innp{\nabla f(x)}{y-x}
  + \frac{L}{\ell} \norm{y-x}^{\ell}
  .
\end{equation*}
We will not need it but for completeness we recall that
$f$ is \emph{$\mu$-strongly convex of order $s>1$},
where \(\mu > 0\) is a positive number,
if for all $x,y\in\mathbb{R}^n$
\begin{equation*}
  f(y) \geq f(x) + \innp{\nabla f(x)}{y - x} +
  \frac{\mu}{s} \norm{y-x}^{s}.
\end{equation*}
Note that if $f$ is smooth of order $\ell>1$
and sharp of order $0 < \theta < 1$,
then $\ell\theta\leq1$.
Theorems~\ref{th:bmpsmooth} and~\ref{th:bmpsharp} give the main
convergence results of BMP.
In particular, Theorem~\ref{th:bmpsharp} states
linear convergence rates even in the case that $f$ is not strongly convex
for sufficiently sharp functions.

\begin{theorem}[Smooth convex case]
  \label{th:bmpsmooth}
  Let $\mathcal{D}\subset\mathbb{R}^{n}$ be a dictionary such that
  \(0\) lies in the relative interior of
  \(\conv{-\mathcal{D} \cup \mathcal{D}}\)
and let $f \colon \mathbb{R}^{n} \to \mathbb{R}$ be \(L\)-smooth of order
$\ell>1$, convex, and coercive. Then the Blended Matching Pursuit
algorithm (Algorithm~\ref{bmp}) ensures that
$f(x_{t}) - f(x^{*}) \leq \varepsilon$ for all $t\geq t_{0}$ where
\begin{equation*}
  t_{0} = \mathcal{O}\left(
    \log
    \left(
      \frac{\phi_{0}}{\varepsilon}
    \right)
    +
    \left( \frac{1}{\varepsilon} \right)^{1/(\ell-1)}
  \right)
  .
\end{equation*}
Here the big O notation hides a factor independent of \(\varepsilon\)
but depending on all the input to the algorithm.
\end{theorem}

\begin{theorem}[Smooth convex sharp case]
  \label{th:bmpsharp}
  Let $\mathcal{D}\subset\mathbb{R}^{n}$ be a dictionary such that
  \(0\) lies in the relative interior of
  \(\conv{-\mathcal{D} \cup \mathcal{D}}\),
  and let \(r > 0\) be the distance of \(0\)
  from the boundary of the convex
  hull of \(-\mathcal{D} \cup \mathcal{D}\).
  Let $f \colon \mathbb{R}^{n} \to \mathbb{R}$
be $L$-smooth
of order $\ell>1$,
convex and coercive.
Furthermore, let \(f\) be \((c, \theta)\)-sharp
over \(\Span{\mathcal{D}}\).
Then the Blended Matching Pursuit algorithm (Algorithm~\ref{bmp})
ensures that $f(x_t) -f(x^{*}) \leq\varepsilon$
for all $t\geq t_{0}$ where
\begin{equation*}
  t_{0} = \max\{K, \eta\}^{\ell / (\ell-1)}
  \begin{cases}
    \mathcal{O}\left(
      \left(\frac{\tau c}{r} \right)^{1/(1-\theta)}
      L^{1/(\ell-1)} D^{\ell / (\ell - 1)}
   \log_{\tau} \left(
     \frac{c \phi_0}{r \varepsilon^{1-\theta}}
   \right)
 \right) & \text{if } \ell \theta = 1 \\
 \mathcal{O}\left(
   \log_{\tau} \left(
     \frac{c \phi_0}{r \varepsilon^{1-\theta}}
   \right)
   +
   \left(
     \tau^{(1 + \ell - 2 \ell \theta) / (1 - \theta)}
     \frac{c^{\ell} L D^{\ell}}{r^{\ell} \varepsilon^{(1 - \ell \theta)}}
    \right)^{1/(\ell-1)}
 \right) & \text{if } \ell \theta < 1.
\end{cases}
\end{equation*}
Here the big O-notation hides a constant depending only on \(\ell\).
\end{theorem}

A common situation is when $f$ is smooth and strongly convex,
which occurs in most ($\ell_2$-regularized) machine learning problems.
Corollary~\ref{cor:bmpstrong} shows that BMP
converges linearly in this setting. Note that this is a special case
of Theorem~\ref{th:bmpsharp} with
\(c = 1 / \sqrt{\mu}\), \(\ell = 2\) and \(\theta = 1/2\),
since strong convexity
implies sharpness.

\begin{corollary}[Smooth strongly convex case]
  \label{cor:bmpstrong}
  Let $\mathcal{D}\subset\mathbb{R}^{n}$ be a dictionary such that
  \(0\) lies in the relative interior of
  \(\conv{-\mathcal{D} \cup \mathcal{D}}\),
  and let \(r > 0\) be the distance of \(0\)
  from the boundary of the convex
  hull of \(-\mathcal{D} \cup \mathcal{D}\).
  Let $f \colon \mathbb{R}^{n} \to \mathbb{R}$
be $L$-smooth
and $\mu$-strongly convex.
Then the Blended Matching Pursuit algorithm (Algorithm~\ref{bmp})
ensures that $f(x_t)-f(x^*)\leq\varepsilon$
for all $t\geq t_{0}$ where
\begin{equation*}
  t_{0} = \mathcal{O} \left(
    \max\{K, \eta\}^{2} \cdot
    \frac{\tau^{2}}{\ln \tau} \cdot
    \frac{L D^{2}}{\mu r^{2}} \ln \frac{\phi_0}{r \sqrt{\mu \varepsilon}}
  \right)
  .
\end{equation*}
Here the big O-notation hides a constant depending only on \(\ell\).
\end{corollary}

\subsection{Second-order Conditional Gradient Sliding algorithm}
\label{sec:second-order-sliding}

\looseness=1
Now we provide a glimpse into conditional gradient methods,
which besides first-order information also have access to
second-order information of the objective function \(f\):
the Hessian \(\nabla^{2} f(x)\) for any \(x\)
in the feasible region \(\mathcal{X}\).
Namely, we present a second-order version
of Conditional Gradient Sliding (Algorithm~\ref{cgs-smooth}),
employing conditional gradients to optimize second-order
Taylor approximations of \(f\),
which are built using the Hessian instead of the Euclidean norm.
Note that the use of second-order information with conditional gradient
methods was first explored 
in \citet{gonccalves2017newton},
which proposed an 
algorithm consisting of unconstrained Newton 
steps, followed 
by a projections onto the feasible region using the 
vanilla Frank–Wolfe
algorithm (Algorithm~\ref{fw}).
Here we present a different algorithm, dubbed the
Second-Order Conditional Gradient Sliding algorithm
(SOCGS, Algorithm~\ref{socgs})
from \citet{carderera2020second}
that uses the Away-step Frank–Wolfe algorithm instead.
The algorithm works with inexact second-order information,
but for simplicity we assume that
exact second-order information is available.

Minimizing a second-order Taylor approximation of \(f\) at every
iteration
is the classical second-order analogue of
the Projected Gradient Descent algorithm,
called the Projected Newton algorithm \citep{polyak66cg},
to solve a convex constrained minimization problem:
\begin{equation}
  \label{Eq:VarMetricStep}
 \begin{split}
  x_{k+1} &\in \argmin_{x \in \mathcal{X}} f(x_k) +
    \innp{\nabla f(x_k)}{x - x_k} + \frac{1}{2} 
    \norm[\nabla^2 f(x_k)]{x - x_{k}}^{2}
  \\
  &
  = \argmin_{x \in \mathcal{X}} \norm*[\nabla^2 f(x_k)]{x -
    (x_k - \left[\nabla^2 f(x_k) \right]^{-1} \nabla f(x_k) )}^{2}
  .
 \end{split}
\end{equation}
We use the suggestive notation $\norm[H]{x}^{2} \defeq \innp{H x}{x}$.
Note that $\norm[H]{\cdot}$ is a norm only for a
positive definite $H$, and that the Hessian \(\nabla^{2} f(x)\)
is positive semi-definite for convex functions \(f\).
For strongly convex \(f\), the Hessian \(\nabla^{2} f(x)\)
is positive definite.
The second form of the minimization problem
is only valid for a non-degenerate Hessian, and
provides an interpretation
of every iteration as going into a direction and projecting back to
the feasible region \(\mathcal{X}\)
similar to Projected Gradient Descent.
The main appeal of the Projected Newton algorithm is its local
quadratic convergence in distance to the optimum. We use the term \emph{local 
convergence} to refer to the convergence rate 
achieved when the algorithm is given a starting point close enough 
to a minimizer. Alternatively, we use the term \emph{global convergence} 
to refer to the global convergence rate that applies regardless of the 
starting point provided to the algorithm. 
Unfortunately, 
the algorithm as defined in Equation~\eqref{Eq:VarMetricStep} does
not converge globally, an issue often remedied by the use 
of line search.

\begin{algorithm}[h!]
\caption{Second-Order Conditional Gradient Sliding (SOCGS) \citep{carderera2020second}}
\label{socgs}
\begin{algorithmic}[1]
  \REQUIRE Start point $x_0\in\mathcal{X}$
  \ENSURE Iterates \(x_{1}, \dotsc \in \mathcal{X}\)
\STATE $x_0 \leftarrow \argmin_{v \in  \mathcal{X}} \innp{\nabla f\left(x \right)}{v}$, $\mathcal{S}_0 \leftarrow \{ x_0 \}$, $\lambda_0(x_0) \leftarrow 1$
\STATE $x_0^{\text{AFW}} \leftarrow x_0$, $\mathcal{S}_0^{\text{AFW}} \leftarrow \mathcal{S}_0$, $\lambda_0^{\text{AFW}}(x_0) \leftarrow 1$
\FOR{$t=0$ \TO \dots}
  \STATE\label{algLine:independentAFW}
    $x^{\text{AFW}}_{t + 1}, \mathcal{S}^{\text{AFW}}_{t + 1},
    \lambda^{\text{AFW}}_{t + 1} \leftarrow
    \operatorname{AFW}\left(f, x^{\text{AFW}}_{t},
      \mathcal{S}^{\text{AFW}}_{t}, \lambda^{\text{AFW}}_{t} \right)$
    \COMMENT{\(\operatorname{AFW}\) is Algorithm~\ref{AFW:SOCGS}.}
  \STATE\label{algLine:buildQuadratic}
    \(\hat{f}_{t}(x) \gets \innp{\nabla f(x_{t})}{x - x_{t}}
    + \frac{1}{2} \norm[\nabla^2 f(x_t)]{x - x_{t}}^{2}\)

\STATE \label{AlgLine:AccuracyParameter}
  $\varepsilon_t \leftarrow 
    \operatorname{LBO}(x_{t})^4/\operatorname{LBO}(x_{0})^3$
  \COMMENT{\(0 < \operatorname{LBO}(x_{t}) \leq f(x_{t}) - f(x^{*})\)}
\STATE \label{algLine:TransferPoint}
  $\tilde{x}^0_{t + 1} \leftarrow x_t$,
  $\tilde{\mathcal{S}}^0_{t + 1} \leftarrow \mathcal{S}_t$,
  $\tilde{\lambda}^0_{t + 1} \leftarrow \lambda_t$,
  $k \leftarrow 0$
\WHILE{$\max\limits_{v \in \mathcal{X}} \innp{\nabla \hat{f}_t (
        \tilde{x}^k_{t + 1} )}{\tilde{x}^k_{t + 1} - v}
      \geq \varepsilon_t$}
  \STATE $\tilde{x}^{k+1}_{t + 1}, \tilde{\mathcal{S}}^{k+1}_{t + 1},
    \tilde{\lambda}^{k+1}_{t + 1} \leftarrow
    \operatorname{AFW}\left( \hat{f}_t, \tilde{x}^{k}_{t + 1},
      \tilde{\mathcal{S}}^{k}_{t + 1}, \tilde{\lambda}^{k}_{t + 1}
    \right)$
\STATE $k \leftarrow k + 1$
\ENDWHILE
\STATE $\tilde{x}_{t + 1} \leftarrow \tilde{x}^k_{t + 1} $, $\tilde{\mathcal{S}}_{t + 1} \leftarrow \tilde{\mathcal{S}}^k_{t + 1}$, $\tilde{\lambda}_{t + 1} \leftarrow \tilde{\lambda}^k_{t + 1}$ \label{algLine:approximate_minimizer_SOCGS}
\IF{$f\left(\tilde{x}_{t + 1} \right) \leq f(x^{\textup{AFW}}_{t + 1} )$}
\STATE $x_{t+1} \leftarrow \tilde{x}_{t + 1} $, $\mathcal{S}_{t+1} \leftarrow \tilde{\mathcal{S}}_{t + 1}$, $\lambda_{t+1} \leftarrow \tilde{\lambda}_{t + 1}$
\ELSE
\STATE $x_{t+1} \leftarrow x^{\text{AFW}}_{t + 1} $, $\mathcal{S}_{t+1} \leftarrow \mathcal{S}^{\text{AFW}}_{t + 1}$, $\lambda_{t+1} \leftarrow \lambda^{\text{AFW}}_{t + 1}$
\ENDIF
\ENDFOR
\end{algorithmic}
\end{algorithm}

\begin{algorithm}
  \caption[]{Frank–Wolfe Away Step AFW($f, x, \mathcal{S}, \lambda)$)
    (a single iteration of Algorithm~\ref{away})}
  \label{AFW:SOCGS}
  \begin{algorithmic}[1]
    \REQUIRE Convex function \(f\), feasible point \(x\),
      convex combination \(\mathcal{S}, \lambda\) of \(x\)
      as vertices
    \ENSURE Feasible point \(x'\) as
      convex combination \(\mathcal{S}', \lambda'\)
\STATE $v \leftarrow \argmin_{v \in  \mathcal{X}} \innp{\nabla f\left(x \right)}{v}$, $a\leftarrow \argmax_{v \in  \mathcal{S}} \innp{\nabla f\left(x \right)}{v}$
\IF{$\innp{\nabla f(x)}{x - v} \geq \innp{\nabla f(x)}{a - x}$}
\STATE $d \leftarrow x - v$, $\gamma_{\max} \leftarrow 1$
\ELSE
\STATE $d \leftarrow a - x$, $\gamma_{\max} \leftarrow \lambda_a/\left( 1 - \lambda_a\right)$
\ENDIF
\STATE $\gamma \leftarrow \argmin_{\gamma\in [0,\gamma_{\max}]} f\left(x + \gamma d \right)$
\STATE $x' \leftarrow x + \gamma d$
\STATE Update the active set $\mathcal{S}$ and barycentric coordinates $\lambda$ of $x'$ to $\mathcal{S}'$ and $\lambda'$.  
\RETURN $x', \mathcal{S}', \lambda'$
  \end{algorithmic}

\bigskip

\end{algorithm}

The Second-Order Conditional Gradient Sliding
(SOCGS, Algorithm~\ref{socgs})
at each iteration computes independently two possible candidates
for the next iterate
using the Away-step Frank–Wolfe (AFW) algorithm (Algorithm~\ref{away}):
one via a single AFW step in Line~\ref{algLine:independentAFW},
and one via an inexact Projected \myindex{Newton step}.
The algorithm chooses the candidate with
the smaller function value.
The AFW step is taken to ensure global
convergence of the algorithm, as the Projected \myindex{Newton algorithm}
does not converge globally in general.

For the Projected Newton step,
a target Frank–Wolfe gap $\varepsilon_t$ is necessary to
terminate the underlying conditional gradient algorithm,
which as of now
depends on a positive lower bound $\operatorname{LBO}\left( x_t\right)$
on the primal gap.
To the best of our knowledge,
all Newton methods which approximate the projected Newton step
use such a lower bound to control approximation accuracy
for achieving a quadratic convergence,
regardless of whether they are employing conditional gradients.

The difficulty with the use of the lower bound
$\operatorname{LBO}\left( x_t\right)$ lies in the fact that a loose
bound on the primal gap will result in solving the quadratic problems
to an accuracy that is higher than necessary to ensure  quadratic
convergence  to the optimum. This will increase the number
of LMO calls performed by the algorithm (but not the number of FOO
calls). In certain problems one can know a-priori the value of
$f(x^*)$, see for example the approximate Carathéodory problem
(Section~\ref{sec:caratheodory}), in which one is given a point $x$ that 
belongs to a convex set $\mathcal{X}$, and wants to find a convex 
combination of a set of extreme of $\mathcal{X}$ that approximate the 
original point $x$ in a certain $\ell_p$-norm. In this case we know that
$f(x^*) = 0$.

The global linear convergence of the algorithm in primal gap
is driven by the independent AFW steps, while the local quadratic
convergence of the algorithm in primal gap in the number of iterations
is driven by the inexact Projected Newton steps.
This local convergence kicks in
after a finite number of iterations
independent of the target accuracy $\varepsilon$
(see the next theorem).

\begin{theorem} \label{th:socgs:primal_gap_convergence}
  Assume the function $f$
  is strongly convex and smooth, and the feasible region $\mathcal{X}$
  is a polytope. If the strict complementarity assumption is
  satisfied, that is, $\innp{\nabla f(x^*)}{x - x^*} = 0$ if and only
  if \(x\) is contained in the minimal face \(\mathcal{F}(x^*)\)
  containing \(x^{*}\),
  and the LBO oracle returns $\operatorname{LBO}(x) = f(x) - f(x^*)$ for
  all $x\in\mathcal{X}$, then after a finite burn-in phase independent of the
  target accuracy, the SOCGS algorithm achieves primal gap at most
  \(\varepsilon\) in
  $\mathcal{O}(\log \log (1/\varepsilon))$
  first-order and second-order oracle calls and
  $\mathcal{O}(\log (1/\varepsilon) \log \log (1/\varepsilon))$
  linear minimization oracle calls.
\end{theorem}

Other algorithms have been developed that deal with a more general
class of functions and feasible regions, and require milder
assumptions, albeit with weaker convergence guarantees. For example,
the \emph{Newton Conditional Gradient} algorithm \citep{liu2020newton}
uses the standard FW algorithm to approximately minimize a
second-order approximation to a self-concordant function $f$ over a general
convex compact set $\mathcal{X}$.
The algorithm requires exact second-order information,
but does not require the function to be smooth and
strongly convex. After a finite number of steps independent of the target 
primal gap accuracy, the algorithm
achieves an $\varepsilon$-optimal solution with $\mathcal{O}(\log
(1 / \varepsilon))$ first-order and exact second-order oracle
calls and $\mathcal{O}(\varepsilon^{-1-o(1)})$ linear optimization
oracle calls.

\subsection{Acceleration}
\label{sec:acceleration}

Recall that for the Gradient Descent algorithm
to achieve a primal gap at most \(\varepsilon\),
the required number of first-order oracle calls is
$\mathcal{O}(1 / \varepsilon)$
for smooth functions
and
$\mathcal{O}(\frac{L}{\mu} \log (1 / \varepsilon))$
for $L$-smooth convex and $\mu$-strongly convex functions.
The algorithm has a so-called \emph{accelerated} version,
with better convergence rates, as the name implies,
namely,
the required number of first-order oracle calls is
$\mathcal{O}(1 / \sqrt{\varepsilon})$
for smooth functions
and
$\mathcal{O}(\sqrt{L / \mu} \log(1/\varepsilon))$
for $L$-smooth convex and $\mu$-strongly convex functions.
A general-purpose acceleration method,
which we do not pursue here,
is \emph{Catalyst} \citep{lin2015universal}, which
accelerates any linearly-convergent first-order method
including, e.g., the Away-step
Frank–Wolfe algorithm (Algorithm~\ref{away}).
We are not aware of any generic comparison of accelerated versions of
gradient descent methods and conditional gradient algorithms, as these algorithms are usually used in very different settings.
Nonetheless, in this section, we present an accelerated
Frank–Wolfe algorithm,
\emph{Locally Accelerated Conditional Gradient} (LaCG)
(Algorithm~\ref{algo:LaCG}) from \citet{diakonikolas2020locally},
actually blending the 
Away-step Frank–Wolfe algorithm (Algorithm~\ref{away})
with accelerated Projected Gradient Descent,
with an asymptotical rate
$\mathcal{O}(\sqrt{L / \mu} \log ((L - \mu) / \varepsilon))$ for
minimizing an $L$-smooth and $\mu$-strongly convex function
over a polytope $P$.
Here \(\innp{\cdot}{\cdot}\) is the standard scalar product.

A simple but crucial insight that directly follows from \citet{gm86}
is that there is a critical radius $r$
depending on \(x^{*}\)
such that for all feasible points \(x\) with $\norm{x - x^* } \leq r$,
the optimum \(x^{*}\) is contained in the convex hull of every set
\(\mathcal{S}\) of vertices whose convex hull contains \(x\).
Thus the main idea of the algorithm is to
find a feasible point inside the critical radius
using a conditional gradient method
providing a linear decomposition of the iterates
into a small number of vertices,
such as Away-step Frank–Wolfe algorithm or
Pairwise Frank–Wolfe algorithm (Algorithm~\ref{pairwise}),
and switch to an accelerated gradient descent method
optimizing over the convex hull of the vertices in the linear decomposition.
This relies on projection onto convex hull of a set of points,
which is expected to be much cheaper for small sets than the
projection onto the whole feasible region.

The main difficulty is recognizing the right time to switch between
the two methods, as the critical radius $r$ is unknown.
To overcome this, both methods are run in parallel, and
a carefully tuned restart scheme is employed for the accelerated method,
balancing between capturing changes in the active set
and progress through acceleration.
This restart scheme requires knowledge of $\mu$ and $L$; a later variant, the 
\emph{Parameter-Free Locally Accelerated Conditional Gradient}
algorithm \citep{carderera2021parameter}, achieves the same
complexity as LaCG up to poly-log factors without knowledge of these
function parameters.

\begin{algorithm}
\caption{Locally Accelerated Conditional Gradient (LaCG) \citep{diakonikolas2020locally}}
\label{algo:LaCG}
\begin{algorithmic}[1]
  \REQUIRE Polytope \(P\),
    \(\mu\)-strongly convex \(L\)-smooth function \(f\),
    start point \(x_{0} \in P\)
  \ENSURE Iterates \(x_{1}, \dotsc \in P\)
  \STATE
    $\mathcal{S}_0^{\operatorname{AFW}} \leftarrow \{x_0\}$,
    $\lambda_0^{\operatorname{AFW}} \leftarrow [1]$
\STATE \(y_0 \leftarrow x_{0}\),
  \(x_{0}^{\operatorname{AFW}} \gets x_{0}\),
  \(\hat{x}_0 \leftarrow x_{0}\), \(w_0 \leftarrow x_0\),
  $z_0 \leftarrow -\nabla f(y_0) + L y_0$
\STATE
  $\mathcal{C}_1 \leftarrow \conv{\mathcal{S}_{0}^{\operatorname{AFW}}}$
  \COMMENT{\(\mathcal{C}_{t}\) is convex hull used for acceleration.}
\STATE \(A_0 \leftarrow 1\),
  $\theta \leftarrow \sqrt{\frac{\mu}{2L}}$
\STATE $H \leftarrow \frac{2}{\theta}\ln(1/\theta^2 - 1)$
  \COMMENT{restart frequency}
\STATE $r_f \leftarrow \FALSE$
  \COMMENT{restart flag}
\STATE $r_c \leftarrow 0$ \COMMENT{number of iterations since last restart}
\FOR{$t=1$ \TO \dots}
\STATE $x_t^{\operatorname{AFW}}, \, \mathcal{S}_t^{\operatorname{AFW}}, \,
  \lambda_t^{\operatorname{AFW}} \leftarrow \operatorname{AFW}(f,
  x_{t-1}^{\operatorname{AFW}}, \, \mathcal{S}_{t-1}^{\operatorname{AFW}}, \,
  \lambda_{t-1}^{\operatorname{AFW}})$ \label{algo:LaCG:AFW}
  \COMMENT{\(\operatorname{AFW}\) is Algorithm~\ref{AFW:SOCGS}.}
\IF{$r_f$ \AND $r_c \geq H$}
  \STATE $y_{t} \gets
    \argmin_{x \in \{x_{t}^{\operatorname{AFW}}, \hat{x}_{t-1}\}} f(x)$
\STATE $\mathcal{C}_{t+1} \leftarrow \conv{\mathcal{S}_t^{\operatorname{AFW}}}$ 
\STATE \(A_t \leftarrow 1\),
  $z_t \leftarrow -\nabla f(y_t) + Ly_t$
\STATE $\hat{x}_t \leftarrow \argmin_{u \in \mathcal{C}_{t+1}}
  \{-\innp{z_t}{u} + \frac{L}{2}\norm[2]{u}^2\}$
\STATE \(w_t \leftarrow \hat{x}_{t}\)
\STATE $r_c \leftarrow 0$, $r_f \leftarrow \FALSE$ 
\ELSE
\STATE $A_t \leftarrow A_{t-1}/(1-\theta)$,
\STATE $\hat{x}_t,\, z_t,\, w_t \leftarrow \text{\(\mu\)-AGD+}(x_{t-1},
  z_{t-1}, w_{t-1}, \mu, L - \mu, \theta, A_t, \mathcal{C}_t)$
\IF{$\mathcal{S}_t^{\operatorname{AFW}} \setminus
    \mathcal{S}_{t-1}^{\operatorname{AFW}} \neq \emptyset$}
\STATE $r_f \leftarrow \TRUE$ 
\ENDIF
\IF{\NOT $r_f$} 
\STATE $\mathcal{C}_{t+1} \leftarrow
  \conv{\mathcal{S}_{t}^{\operatorname{AFW}}}$ 
\ELSE
\STATE $\mathcal{C}_{t+1} \leftarrow \mathcal{C}_t$
\ENDIF
\ENDIF
\STATE \label{algo:selectbest}
  $x_t \leftarrow
  \argmin_{x \in \{x_t^{\operatorname{AFW}}, \hat{x}_t, \hat{x}_{t-1}\}} f(x)$
\STATE $r_c \leftarrow r_c + 1$ 
\ENDFOR
\end{algorithmic}
\end{algorithm}

\begin{algorithm}
  \caption[]{Accelerated Step $\mu$-AGD+($x, z, w, \mu, \mu_0, \theta, A,
    \mathcal{C}$) \citep{diakonikolas2020locally}}
  \label{algo:acc-stepAppx}
  \begin{algorithmic}[1]
    \REQUIRE Feasible points \(x\), \(z\), \(w\);
      positive numbers \(\mu\), \(\mu_{0}\), \(\theta\), \(A\);
      convex set \(C\)
    \ENSURE
      Feasible points \(x'\), \(z'\), \(w'\)
    \STATE $y \gets \frac{1}{1 + \theta} x + \frac{\theta}{1 + \theta} w$
    \STATE $z' \gets z
      - \theta A \nabla f(y) + \mu \theta A y$
    \STATE $w' \gets \argmin_{u \in \mathcal{C}}
      \{- \innp{z'}{u} + \frac{\mu A + \mu_0}{2} \norm{u}^{2}\}$
    \STATE $x' \gets (1 - \theta) x + \theta w'$
    \RETURN $x', z', w'$
  \end{algorithmic}
\end{algorithm}

We recall the convergence rate from \citet{diakonikolas2020locally}.

\begin{theorem}
  \label{th:LaCGConvergence}
 Suppose $P$ is a polytope and
 $f$ is an $L$-smooth and $\mu$-strongly convex function over \(P\).
 Let $x_t$ be the solution output by the LaCG algorithm
 (Algorithm~\ref{algo:LaCG}) and $x_0$ be the initial point.
Then
\(f(x_t) - f(x^*) \leq \varepsilon\)
after at most the following number of first-order oracle calls, linear
optimizations and quadratic minimizations over the probability simplex
(equivalent to minimizing over the convex hull of $\mathcal{C}_t$):
 \begin{equation}
   \min \left\{
     \frac{8 L D^{2}}{\mu \delta^{2}}
       \log \left(
         \frac{f(x_0)
         - f(x^*)}{\varepsilon} \right),
     K_{0}
     + 2 \sqrt{\frac{2L}{\mu}} \log \left(
       \frac{(L - \mu)^{2} r^{2}}{2 \mu \varepsilon}
     \right)
   \right\},
 \end{equation}
 where $D$ and $\delta$ are the
diameter and the pyramidal width of the polytope $P$, 
respectively, $r$ is the critical radius, i.e., the largest distance
for which for every point \(x\) with $\norm{x - x^* } \leq r$,
the optimum \(x^{*}\) is contained in the convex hull of every set
\(\mathcal{S}\) of vertices whose convex hull contains \(x\).
Finally, $K_{0} = \frac{8 L D^{2}}{\mu \delta^{2}}
 \log \left( \frac{f(x_0) - f(x^*)}{\mu r^2} \right)$.
\end{theorem}

For projection,
Algorithm~\ref{algo:LaCG} solves a
quadratic subproblem over the convex hull $\mathcal{C}_{t}$ of
an active set
at each iteration, which can be efficiently solved when the set
has a small number of elements (for a brief discussion on sparsity 
see Section~\ref{sec:whyFW})
Smallness of the set can be accomplished with sparsity
enhancing techniques.

It should be noted that the algorithm accesses the
feasible region through the linear optimization oracle only, i.e., it
is projection-free in the traditional sense, while internally using
projections over convex hulls of few vertices.  Finally, the
experiments in \citet{diakonikolas2020locally} showed that the
algorithm can potentially achieve significantly better performance
than Away-step Frank–Wolfe algorithm and Pairwise Frank–Wolfe
algorithm in wall-clock time in problem instances in which the
active set does not grow too large in cardinality (as this is ultimately 
related to how costly it is to project onto the convex hull
$\mathcal{C}_t$). While the theoretical guarantee only ensures
acceleration upon hitting the critical radius, empirically LaCG
outperformed Away-step Frank–Wolfe algorithm and Pairwise Frank–Wolfe
algorithm also in regimes where the iterates had not yet reached the
critical radius $r$.

\begin{example}[Sparse signal recovery\index{sparse signal recovery}]
  \label{Example:sparse_recovery}
  In this example we compare the performance of several of the algorithms 
  presented in earlier sections for a sparse signal recovery
  problem.  In sparse signal recovery, the goal is to recover a
  $z\in \R^n$ with few non-zero entries from noisy measurements.
  We consider the case, where the input is the result of measurements:
  $y = Az + w$, where $A\in\R^{m\times n}$ is known and $w\in \R^{m}$ is an
  unknown noise vector from a normal distribution
  with a diagonal covariance matrix $\sigma^2 I^m$ and zero mean. In order
  to recover the vector we use an elastic-net regularised formulation
  \citep{zou2005regularization},
  relaxing the feasible region to a convex set,
  the \index{l1-ball@\(\ell_{1}\)-ball}\(\ell_{1}\)-ball,
  i.e, solve
  \begin{equation*}
    \min_{\norm[1]{x} \leq \tau} \norm[2]{y  - Ax}^{2} + \alpha \norm[2]{x}^{2}
  \end{equation*}
  for some $\tau > 0$ and $\alpha > 0$. More specifically, we set $m = 200$ and $n = 500$ and select the 
  entries of $A$ uniformly at random between $0$ and $1$.
  The variance parameter of noise is set to $\sigma = 0.05$.
We run the experiments for two different levels of sparsity for 
the vector $z$, namely, in the first experiment we set \qty{5}{\percent}
of the entries of $z$ to be non-zero, and in the second one
we set \qty{25}{\percent} of the
entries to be non-zero. The non-zero elements of $z$ are drawn uniformly
at random between $-0.5$ and $0.5$.
The resulting condition number in both cases 
is approximately $10^5$. Using cross-validation, we select $\tau = 5$ and $\tau = 30$ 
for the experiment where \qty{5}{\percent} and \qty{25}{\percent} of the elements are non-zero, respectively. In
both cases we choose $\alpha=1$.

\begin{figure}
  \footnotesize
  \centering
  \begin{tabular}{cc}
    Sparsity: \qty{5}{\percent} & Sparsity: \qty{25}{\percent} \\[\smallskipamount]
    \includegraphics[width=.45\linewidth, alt={Plot of primal gap in
      the number of linear minimizations for various algorithms.
      After a moderate start, all except the vanilla Frank–Wolfe
      algorithm accelerate decreasing in roughly the similar
      manner.}]{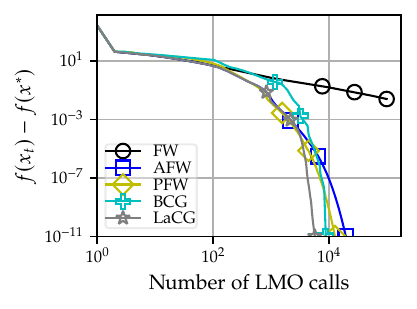}
  &
  \includegraphics[width=.45\linewidth, alt={Plot of primal gap in
    the number of linear minimizations for various algorithms.
    After a long moderate start, all except the vanilla Frank–Wolfe
    algorithm and the Away-step Frank–Wolfe algorithm
    start decreasing sharply with a small
    difference.}]{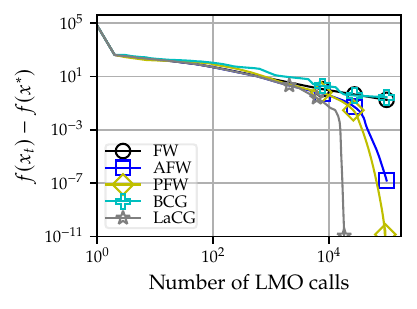}
  \\
  \includegraphics[width=.45\linewidth, alt={Plot of primal gap in
    time for various algorithms.
    The vanilla Frank–Wolfe algorithm proceeds very slowly,
    the others decrease sharply
    roughly the same way.}]{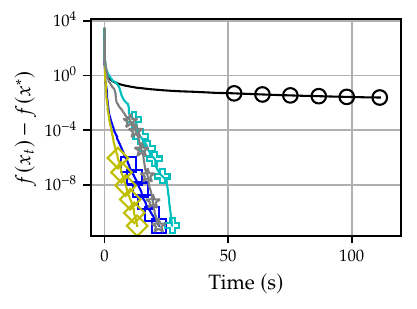}
  &
  \includegraphics[width=.45\linewidth, alt={Plot of primal gap in
    time for various algorithms.
    The vanilla Frank–Wolfe algorithm and Away-step Frank–Wolfe
    algorithm proceeds very slowly,
    the others decrease sharply.}]{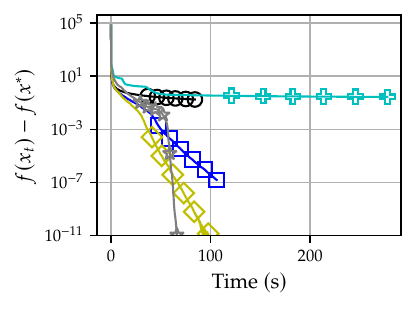}
  \end{tabular}

  \caption{\label{fig:LaCG} \emph{Sparse signal recovery:}
    Primal gap for sparse signal recovery, i.e., minimizing a
    quadratic function over an \(500\)-dimensional \(\ell_{1}\)-ball,
    where the
    optimal solution is known to have a low percentage of non-zero
    entries.  This percentage is called sparsity here.
    The graphs for primal gap in the number of first-order oracle
    look visually the same as in the first row (for number of LMO
    calls), and hence are omitted.}
\end{figure}  

\pagebreak

Note that in the feasible region,
the \(n\)-dimensional \index{l1-ball@\(\ell_{1}\)-ball}\(\ell_{1}\)-ball
has only a small number of vertices, namely \(2n\),
and it is therefore especially advantageous for algorithms using active
sets.
For this reason we don't consider algorithms aiming to reduce active
set usage, like
Decomposition-invariant Pairwise Frank–Wolfe (DI-PFW)
(Algorithm~\ref{DIPFW}).

All the algorithms presented in the comparison use line search,
using the following closed form for the optimal step size:
\(\gamma = -\innp{\nabla f(x)}{d}/(2 \norm[2]{A d}^{2}
+ 2 \alpha \norm[2]{d}^{2})\). In the left and the right 
columns of Figure~\ref{fig:LaCG}
we see that after the finite burn-in phase, once the convex hull
of the
active set of LaCG contains the optimum, acceleration kicks in 
for this algorithm, and we see a faster convergence 
rate in the number of LMO and FOO calls. This happens approximately
after \num{8000} FOO and LMO oracle calls on the left column, and \num{20000}
FOO and LMO oracle
calls on the right column. However, at each iteration the LaCG algorithm 
requires computing 
a projection onto the convex hull of the active set, an operation 
which can be costly (note that LaCG does not require access to a projection 
oracle for $\mathcal{X}$), and which can potentially wash out the 
advantage of this faster convergence rate in the number of iterations for
convergence rate in wall-clock time. This is
precisely what happens in the experiment shown in the left-column of 
Figure~\ref{fig:LaCG} where we 
see that LaCG converges slower than Pairwise Frank–Wolfe algorithm
in wall-clock time,
despite the fact that 
it requires fewer LMO's and FOO's to reach a given primal gap value. 
\end{example}

\section{Applications of Frank–Wolfe algorithms}\label{applications}

In this section, we present several applications which
showcase the advantages of the Frank–Wolfe algorithm, the most
important of which are
its projection-free nature and its sparsity properties.

\subsection{Submodular optimization} \label{app:submodular}

Before we introduce submodular optimization problems, some preliminary definitions are needed.

\begin{definition}[Set Function] \label{def:SetFunction} Given a set
  $\mathcal{V}$,
  a \emph{set function} $F \colon 2^{\mathcal{V}} \to \mathbb{R}$
  is a real-valued function
  defined on the  power set \(2^{\mathcal{V}}\) of \(\mathcal{V}\),
  i.e., the set of all subsets of \(\mathcal{V}\).
\end{definition}

Examples of set functions include the cardinality function, which
returns the cardinality of a subset $S$, or the cut function for an
undirected graph $G = (\mathcal{V},E)$, which when given a subset of
vertices $S$ returns the number of edges $E$ that connect a vertex in
$S$ to a vertex in $\mathcal{V} \setminus S$.

\begin{definition}[Submodularity and Monotonicity] \label{def:SubmodularFunction}
  A set function $F \colon 2^{\mathcal{V}} \to \mathbb{R}$ is
  \emph{submodular} if for all subsets $\mathcal{A}$ and $\mathcal{B}$
  of \(\mathcal{V}\)
  \begin{equation*}
    F(A \cap B) + F(A \cup B)
    \leq F(A) + F(B)
    .
  \end{equation*}
  Further, the set function \(F \colon 2^{\mathcal{V}} \to \mathbb{R}\) is
  \emph{monotone}
  if $F(A) \leq F(B)$ for every $ A\subseteq  B \subseteq  \mathcal{V}$.
\end{definition}

For example, the cardinality of a set is a monotone and submodular
function, whereas the cut function is submodular but non-monotone.  We
refer the interested reader to
\citet{tohidi2020submodularity, buchbinder2018submodular} for more
details regarding the theory and applications of submodularity.

A central problem in submodular optimization, is to maximize a
monotone submodular function \(F\)
over a collection \(\mathcal{I}\)
of some subsets of \(\mathcal{V}\):
\begin{equation} \label{discrete_problem}
  \max_{S\in \mathcal{I}} F(S).
\end{equation}
Submodular functions arise naturally in many
areas, such as the study of graphs, matroids, covering problems, and facility location problems.
These functions have been extensively studied in operations research and combinatorial optimization.
More recently, submodular functions have proven to be key concepts in other areas such as machine
learning, algorithmic game theory, and social sciences.  Submodular set functions can be maximized approximately up to constant factors \citep{nemhauser1978analysis} and can be  minimized exactly
\citep{lovasz1983submodular} in polynomial time with efficient
combinatorial optimization algorithms.

One of the main approaches in optimizing submodular functions has been through appropriate continuous relaxations. 
Such approaches and their resulting algorithms are of interest in this survey
as conditional gradient methods have been crucial in developing efficient
methodologies for submodular optimization problems. In this section,
we will illustrate a conditional gradient method for maximizing
monotone submodular functions subject to a matroid constraint (known
as the \emph{continuous greedy} algorithm). However, we remark that
there are several instances of submodular maximization and
minimization problems that benefit enormously from conditional
gradient methods, such as submodular minimization \citep{fujishige1980lexicographically, chakrabarty2014provable, bach2013learning}
and (non-monotone) submodular maximization subject to general matroid constraints \citep{calinescu2011maximizing}.

We consider submodular maximization when \(\mathcal{I}\)
is the set of independent subsets of a matroid over \(\mathcal{V}\).
An illustrative example is when \(\mathcal{V}\) is a set of vectors in
a vector space and \(\mathcal{I}\) is the set of
linearly independent subsets of \(\mathcal{V}\).
The most widely used method to solve the above \emph{discrete} problem
is through an appropriate continuous extension.
First we extend the objective function.

\begin{definition} Given a set function $F \colon 2^{\mathcal{V}} \to
  \mathbb{R}$, its \emph{multilinear extension} is the function
  $f \colon [0,1]^{\size{\mathcal{V}}} \to \mathbb{R}$ defined as
\begin{equation}  \label{multilinear}
f(x) = \sum_{S \subseteq \mathcal{V}} F(S) \prod_{i\in S} x_i \prod_{i \notin S} (1-x_i)
\end{equation} 
for all $x = (x_1, \dotsc, x_{\size{\mathcal{V}}}) \in [0, 1]^{\size{\mathcal{V}}}$.
In other words, $f(x)$ is the expected value of of $F$ over
sets wherein each element $i$ is included with probability $x_i$ independently.
\end{definition}

The extension of maximizing a set function \(F\)
over a matroid \(\mathcal{I}\)
is maximizing the multilinear extension \(f\) of \(F\)
over the \emph{matroid base polytope} \(P\) of $\mathcal{I}$,
defined as
\[\mathcal{P} \defeq \left\{
    x \in \mathbb{R}_{+}^{\mathcal{V}}
    \,\middle|\,
    x(\mathcal{V}) = r(\mathcal{V}),\
    \forall S \subseteq \mathcal{V} \colon
    x(S) \leq r(S)
  \right\},
\]
where $r(\cdot)$ is the matroid's rank function. 
(When \(\mathcal{V}\) is a set of vectors in a vector space, and
\(\mathcal{I}\) is the set of linearly independent subsets
of \(\mathcal{V}\), then \(r(S)\) is the dimension of the subspace
generated by \(S\).)
A subset \(S \subseteq \mathcal{V}\) is represented by its characteristic
vector \(x^{S}\) with coordinates \(x^{S}_{i} = 1\) for \(i \in S\)
and \(x^{S}_{i} = 0\) for \(i \notin S\).
Note that \(f(x^{S}) = F(S)\). Indeed, one can show
\begin{equation} \label{continous-discrete}
\max_{x \in \mathcal{P}} f(x) =   \max_{S\in \mathcal{I}} F(S).
\end{equation}
Also, feasible solutions of the continuous problem (i.e., maximizing $f$ over $\mathcal{P}$) can be efficiently
rounded to a feasible discrete solution in $\mathcal{I}$
without loss in objective value
(i.e., a solution to the original discrete problem),
via pipage or swap rounding
\citep{pipage2004}.
Here we focus only on solving the continuous problem,
which can be done very efficiently using a conditional gradient method
with strong theoretical guarantees. Indeed, the conditional gradient
method has a pleasing connection with greedy-type methods and hence it
is called the Continuous Greedy Algorithm (outlined in
Algorithm~\ref{alg:continuous_greedy}) in the literature of
submodular optimization, which is the vanilla Frank–Wolfe algorithm
(Algorithm~\ref{fw}) starting at the origin (i.e., the empty set)
with the notable difference that instead of a linear combination,
the Frank–Wolfe vertex is simply added to the previous iterate
in Line~\ref{line:submodular-update}.
This update rule is more aligned with the greedy nature of the
algorithm.
A related change is the use of a fixed step size,
$1/T$, where $T$ is the number of iterations,
providing equal weight to all increments
to heuristically increase the overall function value the most.
This step size also ensures that all the
iterates are inside the polytope $\mathcal{P}$ (proving this fact is
an easy exercise!).

\begin{algorithm}[t]
\caption{The Continuous Greedy Algorithm \label{alg:continuous_greedy}}
\begin{algorithmic}[1]
  \REQUIRE
    Start atom $x_0 = (0, \dotsc, 0)$, number of iterations $T$
  \ENSURE Iterates $x_{1}, \dotsc \in \mathcal{P}$
  \FOR{$t=1$ \TO \(T\)}
\STATE$v_t\leftarrow\argmin_{v\in\mathcal{P}}\innp{- \nabla f(x_t)}{v}$ \label{MonotoneLMO}
\STATE$x_{t+1} \leftarrow x_{t} + \frac{1}{T}v_t$
  \label{line:submodular-update}
\ENDFOR
\end{algorithmic}
\end{algorithm}
The convergence guarantee for Algorithm~\ref{alg:continuous_greedy} is
provided in Theorem~\ref{th:submodularmonotone} from \citet{bian2017guaranteed}.
\begin{theorem}
  \label{th:submodularmonotone}
  Let \(f\) be the multilinear extension of
  a monotone submodular set function
  over the ground set of a matroid
  with value \(0\) on the empty set.
  Given a fixed number $T$ of linear minimizations,
  Algorithm~\ref{alg:continuous_greedy} computes a solution \(x_{T}\)
  in the base polytope of the matroid with
  \begin{equation}
    f(x_T) \geq \left(1 - \frac{1}{e}\right)f(x^*) - \frac{C_F}{T}
    ,
  \end{equation}
  where \(x^{*}\) is a maximizer of \(f\) and
  $C_F$ is a constant that depends on the set function $F$.
\end{theorem}
The above theorem guarantees that the continuous greedy algorithm is
able to find an $(1-1/e)$-optimal solution
as the number of iterations $T$ grows large (the rate is $1/T$).
It is worth remarking that the approximation factor $(1-1/e)$
is tight, i.e. maximizing submodular functions beyond the $(1-1/e)$
approximation factor is known to be NP-hard
\citep{nemhauser1978analysis}.

\looseness=1
The continuous greedy algorithm was first developed in
\citet{calinescu2011maximizing}.
It is the first algorithm that obtains a
$(1-1/e)$-optimal solution\footnote{This is a solution whose function value is at least a fraction $(1-1/e)$ of the optimal value.} in expectation for maximizing monotone submodular
functions subject to a general matroid constraint. Unfortunately, the
original complexity of the continuous greedy algorithm is quite large:
it scales with the size $n \defeq \size{\mathcal{V}}$ of the ground
set~$\mathcal{V}$, at least $\Omega(n^8)$ evaluations of the set function.
This is mainly because
efficiently computing the multi-linear function, given in
\eqref{multilinear}, is a non-trivial task. First of all, note that an
exact computation of the multilinear extension requires to compute all
the function values $f(S)$, for any $S \in \mathcal{V}$, and as a
result it requires $2^n$ function evaluations.
One can of course
compute the multilinear extension in an approximate manner, by viewing
it as an expectation of a stochastic process, and by using
$\mathcal{O}((\log n)/\varepsilon^2)$ function computations
it is possible to
compute the multilinear extension with $\varepsilon$-accuracy. This
approximation was used in the original paper
\citet{calinescu2011maximizing}. Later on, in
\citet{badanidiyuru2014fast}, a more efficient method was proposed to reduce the
complexity of finding an $(1-1/e-\varepsilon)$-optimal solution to
$\mathcal{O}(n^2/\varepsilon^4)$. Another way to efficiently solve the submodular
maximization problem is via the stochastic continuous greedy algorithm
developed in \citet{mokhtari2020stochastic, hassani2019stochastic}.
These methods are based on the continuous greedy
algorithm, but views the continuous problem \eqref{continous-discrete}
as a stochastic optimization problem, and uses the ideas developed in
Section~\ref{sec:momentum-FW}, to obtain a complexity of~$\mathcal{O}(n^{5/2}/\varepsilon^3)$.

\subsection{Structural SVMs and the Block-Coordinate Frank–Wolfe algorithm}
\label{sec:structural_SVMs}

Classification is an
important task in machine learning, 
in which one attempts to correctly assign 
labels to a series of data points.
With just two labels (binary classification),
a simple model one can use to classify points
into two classes is by using an affine hyperplane, 
defined by a linear equality, with which 
one classifies points on one side of the hyperplane 
as belonging to the first class, and the points 
on the other side as belonging to the second class. 
If there exists a hyperplane to correctly
classify all the points in this way, the data is called
linearly separable (see the image on the left in Figure~\ref{fig:SVM_margins}).
If there are multiple hyperplanes correctly classifying all the points, we would like to select
the one that is most robust to noise. That is, if we perturb our points 
with noise we still want the same hyperplane to classify them
correctly.
This leads to the notion of maximum-margin hyperplane,
which is the hyperplane
that correctly classifies all the data points, and maximizes the 
minimum distance to the data points.  This simple model is
known as the \emph{Support Vector Machine} (SVM) model
\citep{vapnik1964class}.

This model has been extended, to create
non-linear classifiers \citep{boser1992training}, and to deal with 
data that might not be perfectly separable (see the 
right-most image in Figure~\ref{fig:SVM_margins}), using what is known as 
the \emph{soft-margin} formulation of SVMs \citep{cortes1995support}. In 
this formulation, we allow some of the data points to be incorrectly 
classified, but we penalize these incorrect classifications. 
By increasing the penalization for the incorrectly 
classified points, the optimal hyperplane will classify
a larger number of data points correctly with a smaller margin.
By decreasing the
penalization, the optimal hyperplane will have a larger margin
but classifies a smaller
number of data points correctly. See \citet{shalev2014understanding} and the 
references therein for an in-depth treatment of SVM.

\begin{figure*}[!tbp]
  \centering
  \begin{minipage}[b]{0.35\textwidth}
    \includegraphics[width=\textwidth, alt={A solid straight line
      separating orange squares and blue circles, with dashed
      lines indicating the empty area
      around the line.}]{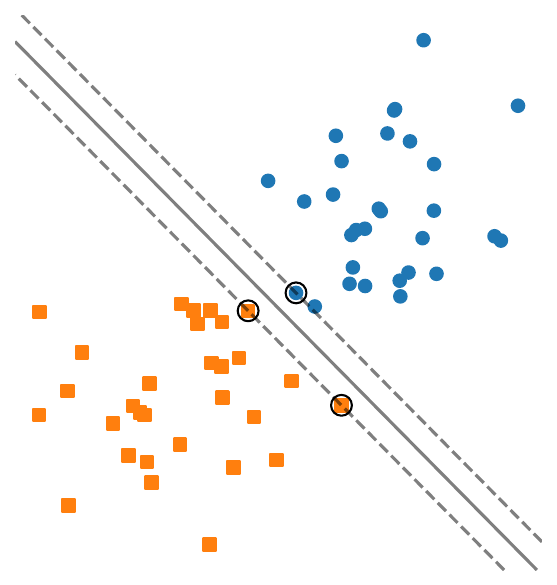}
  \end{minipage}
  \begin{minipage}[b]{0.35\textwidth}
    \includegraphics[width=\textwidth, alt={A solid line separating
      most orange squares from blue circles
      with dashed lines indicating an area around the line
      containing few of the circles and
      squares.}]{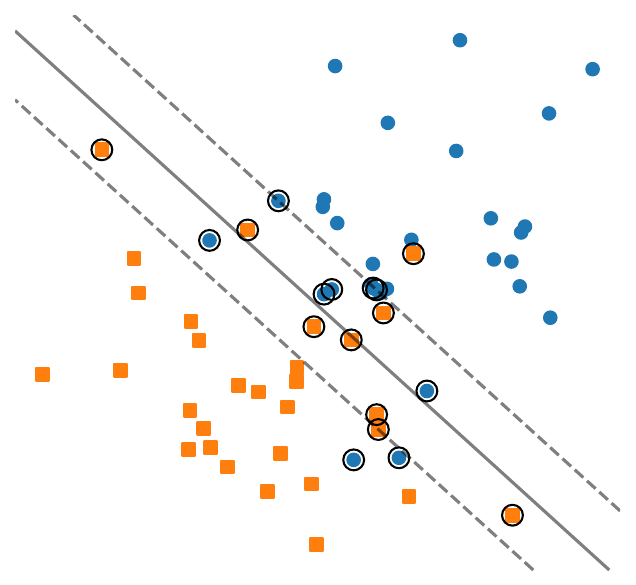}
  \end{minipage}
        \caption{\label{fig:SVM_margins}
    \emph{Hard-margin versus soft-margin SVM:}
    Binary classification example: separating
    orange squares and blue circles with a straight line in the plane.
    On the left the data is linearly separable, and the solid line is
    the one separating the squares and circles with maximum margin,
    i.e., maximizing the minimum distance to the points.
    The circled points are the closest to the separating line,
    i.e., these points are the support vectors.
    On the right, there is no separating straight line,
    so the dataset is a good candidate for soft-margin SVM, in which
    the goal is to find a hyperplane that roughly separates the data, and allows some orange squares to be classified as blue circles, and vice-versa.
    A possible good separator is depicted as a solid line.
    The circled points are the support vectors, i.e., the points
    which are either misclassified or lying closer than the margin
    to the separating line.}
\end{figure*}

We focus in this section on \emph{Structural SVMs}, which are a 
generalization of SVMs in which the classification has some structure,
arising from a graph or some other combinatorial
object \citep{tsochantaridis2005large, taskar2004max}.
We use the Euclidean norm with the standard scalar product.
More formally, given a
data point $z \in \mathcal{Z}$ its classification should be
from a set $\mathcal{Y}(z)$ depending on the data point itself.
A given feature mapping~$\phi \colon \mathcal{Z} \times \bigcup_{i=1}^{n }\mathcal{Y}(z_{i})
\to \R^d$ maps data points with their possible classifications
to a vector space, and the goal is to find a linear classifier \(w\),
so that a data point \(z\) is classified as
$\argmax_{y \in \mathcal{Y}(z)} \innp{w}{\phi(z,y)}$.
To formulate the optimality criterion for \(w\),
we follow the formulation of \citet{tsochantaridis2005large}
and \citet{lacoste2013block}. Given a training dataset 
$\{ z_i, y_i \}_{i=1}^n$,
i.e., data points \(z_{i}\) of class \(y_{i}\),
we find $w$ by solving
\begin{equation}
  \label{applications:equation:SVM_primal}
\min_{\substack{\xi_i \geq 0 \ \forall i  \\  \innp{w}{\phi(z_i, y_i)}
    \geq \innp{w}{\phi(z_i, y)} + L(y_i, y) -\xi_i \ \forall i \forall
    y \in \mathcal{Y}(z_i)}} \frac{\lambda}{2} \norm[2]{w}^{2}
+ \frac{1}{n}\sum_{i = 1}^n \xi_{i},
\end{equation}
where $L(y_i, y)$ is a loss function returning the penalty
for classifying a data point as $y$
when its correct classification is $y_i$ (both penalizing incorrect
classifications and serving as a soft-margin criterion).
This serves as a 
generalization of the standard zero-one loss, which penalizes all incorrect 
classifications equally. Note that Problem~\eqref{applications:equation:SVM_primal} has $m=\sum_{i=1}^n\abs{\mathcal{Y}(z_i)}$ linear constraints in $w$ and $\xi$, which makes 
computing a projection, or solving a linear program over the feasible region, 
computationally prohibitive, because of the potentially huge size of
any $\mathcal{Y}(z_i)$.
The unboundedness of the problem is also a disadvantage.
These all suggest that a reformulation of
Problem~\eqref{applications:equation:SVM_primal} should be used
instead to eliminate these problems,
and the Lagrange dual turns out to be satisfactory:
\begin{multline}
  \label{applications:equation:SVM_dual}
  \min_{\substack{ \sum_{y\in \mathcal{Y}(z_i)}\alpha_i(y) =1 \
      \forall i \\  \alpha  \geq 0}}
  \frac{\lambda}{2}
    \norm*[2]{\sum_{i=1}^{n} \sum_{y \in \mathcal{Y}(z_{i})}
      \alpha_{i}(y) \cdot
      \frac{\phi(z_i, y_i) - \phi(z_i, y)}{\lambda n}}^{2}
      \\
    + \sum_{i=1}^{n} \sum_{y \in \mathcal{Y}(z_{i})}
    \frac{L(y_{i}, y)}{n} \cdot \alpha_{i}(y),
\end{multline}
where $\alpha_i(y)$ denotes the dual variable for the
constraint \(\innp{w}{\phi(z_i, y_i)}
\geq \innp{w}{\phi(z_i, y)} + L(y_i, y) -\xi_i\)
on data point \(z_{i}\) and the classification
$y \in \mathcal{Y}(z_i)$.
The primal solution is obtained from the dual solution via
\(w = \sum_{i=1}^{n} \sum_{y \in \mathcal{Y}(z_{i})}
\alpha_{i}(y) \cdot
\frac{\phi(z_i, y_i) - \phi(z_i, y)}{\lambda n}\).

Note that now
Problem~\eqref{applications:equation:SVM_dual} has $n$ constraints,
and the feasible region is expressed
as the Cartesian product of $n$ probability simplices,
i.e., each \(\alpha_{i}\) lies in a probability simplex.
However, there are \(m\)
dual variables, and this large dimension prevents the use
of algorithms that maintain dense representations of the iterates, 
such as projection-based algorithms for solving the dual 
problem in Problem~\eqref{applications:equation:SVM_dual}. On the 
other hand, we 
can efficiently use the Frank–Wolfe algorithm, as we can maintain its
iterates as convex combinations of vertices of the 
feasible region, which only have $n$ non-zero 
elements. Moreover, computing an LMO over the feasible region, expressed 
as a Cartesian product is trivial. 

The algorithm presented in Algorithm~\ref{fw_SVM} \citep{lacoste2013block} solves the dual problem presented in 
Problem~\eqref{applications:equation:SVM_dual} using the vanilla 
Frank–Wolfe algorithm with line search. The algorithm is termed \emph{Primal-Dual}
because even though it is conceptually solving the dual problem shown
in Problem~\eqref{applications:equation:SVM_dual},
it maintains only the primal solution
to avoid having to deal with the $m$ dual variables.
For example to compute the gradient note that for some matrices \(A\) and \(b\),
the objective function is
\(f(\alpha) = \norm[2]{A \alpha}^{2} + \innp{b}{\alpha}\)
and $w = A \alpha$.  The gradient of $f$ is
$\nabla f(\alpha) = A^{\top} A\alpha + b = A^{\top} w + b$,
and therefore needs only the primal solution \(w\).
Line~\ref{fw_SVM:LMO} minimizes the gradient over the probability
simplex of \(\alpha_{i}\) by choosing the smallest coordinate of the
gradient, making use of the fact mentioned above that the
feasible region of the dual problem is the Cartesian product of
$n$ probability simplices.

\begin{algorithm}[t]
\caption{Primal-Dual Frank–Wolfe algorithm for structural SVMs \citep{lacoste2013block}}
\label{fw_SVM}
\begin{algorithmic}[1]
  \REQUIRE Start point $w_0 =0$, $l_0 =0$
  \ENSURE Iterates $w_0, \dotsc$
\FOR{$t=0$ \TO \dots}
\FOR{$i = 1$ \TO $n$}
  \STATE\label{fw_SVM:LMO}
    $y_i^* \gets \argmin_{y \in \mathcal{Y}(z_i)}
    L(y_i, y) + \innp{w_t}{\phi(z_i, y_i) - \phi(z_i, y)}$
\ENDFOR
\STATE $v_t \gets
  \frac{1}{\lambda n} \sum_{i=1}^{n} \phi(z_i, y_i) - \phi(z_i, y_i^*)$
\STATE $l \leftarrow \frac{1}{\lambda n} \sum_{i=1}^n L(y_i, y_i^*)$
\STATE $\gamma_t \gets \min \left\{1,
    \frac{\lambda \innp{w_t - v_t}{w_t} + l_t - l}{
      \lambda \norm[2]{w_t - v_t}^2} \right\}$
\STATE $w_{t+1} \leftarrow w_t + \gamma_t (v_t - w_t)$
\STATE $l_{t+1}\leftarrow l_t + \gamma_t (l - l_t)$
\ENDFOR
\end{algorithmic}
\end{algorithm}

More generally, in \citet{lacoste2013block}, a \emph{block-coordinate} 
generalization of the Frank–Wolfe algorithm is presented, which applies
to any problem of the form
\begin{equation*}
\min_{x \in \mathcal{M}_1 \times \dotsm \times \mathcal{M}_n} f(x),
\end{equation*}
where $\mathcal{M}_i$ is a compact convex set
for all $1 \leq i \leq n$,
and $f$ is a convex function.
For example, for the structural SVM dual formulation above
in Equation~\eqref{applications:equation:SVM_dual} we have
$\mathcal{M}_i = \{\alpha _{i} \in\R^{\abs{\mathcal{Y}(z_i)}}  \mid
\sum_{y\in \mathcal{Y}(z_i)}\alpha_{i}(y) = 1, \alpha_{i}(y) \geq 0\}$,
so $\mathcal{M}_i$ is a probability simplex of dimension
$\abs{\mathcal{Y}(z_i)}$. The main idea of the algorithm, shown in 
Algorithm~\ref{block_coordinate_fw}, is to pick one compact convex 
set $\mathcal{M}_i$ uniformly at random at each iteration, and to compute the 
update for that iteration taking into account only that set.
Algorithm~\ref{fw_SVM} above does not use this stochastic update,
nevertheless the stochastic update can be used
for structural SVMs, too.
Sometimes additional constraints are added, as in
\citet{FWsplitLagrangian2018},
which we omit for simplicity.

We use $[x]_{\mathcal{M}_i}$ to denote the set of coordinates of $x$
associated with $\mathcal{M}_i$, and $[\nabla f(x)]_{\mathcal{M}_i}$ to denote
the gradient of $f$ restricted to the coordinates in
$\mathcal{M}_i$.

\begin{algorithm}[t]
\caption{Block-Coordinate Frank–Wolfe \citep{lacoste2013block}}
\label{block_coordinate_fw}
\begin{algorithmic}[1]
  \REQUIRE Start point $x_{0} \in
    \mathcal{M}_{1} \times \dotsm \times \mathcal{M}_{n}$
  \ENSURE Iterates $x_{1}, \dotsc \in
    \mathcal{M}_{1} \times \dotsm \times \mathcal{M}_{n}$
\FOR{$t=0$ \TO \dots}
\STATE Pick $i$ uniformly at random from $[1,n]$.
\STATE $v_t \gets \argmin_{v \in \mathcal{M}_i} \innp{v}{
  [\nabla f(x_{t})]_{\mathcal{M}_i}}$
\STATE $\gamma_t \leftarrow \frac{2n}{t + 2n}$
\STATE $[x_{t+1}]_{\mathcal{M}_{j}} \leftarrow
  [x_{t}]_{\mathcal{M}_{j}}$
  for all \(j \neq i\).
\STATE $[x_{t+1}]_{\mathcal{M}_i} \gets
  (1 - \gamma_t) [x_{t}]_{\mathcal{M}_i} + \gamma_t v_t$
\ENDFOR
\end{algorithmic}
\end{algorithm}

Furthermore, one can readily obtain $\mathcal{O}(1/t)$ primal gap and Frank–Wolfe gap convergence guarantees
in expectation for these algorithms, shown in Theorem~\ref{block_fw_theorem}, using a notion 
derived from the curvature defined in Section~\ref{sec:affine-invariance}.
With a different problem-specific constant factor,
the same convergence holds for variants
with a simple, cyclic, deterministic choice of \(\mathcal{M}_{i}\)
at each iteration, and using the short step rule, see \citet{beck2015block}.

\begin{theorem}
  \label{block_fw_theorem}
  Let \(f\) be a smooth convex function over
  \(\mathcal{M}_{1} \times \dotsm \times \mathcal{M}_{n}\),
  a product of compact convex sets $\mathcal{M}_i$.
  Then Algorithm~\ref{block_coordinate_fw} (or
  a variant with line search) satisfies
  \begin{equation*}
    \E{f(x_t)} - f(x^{*}) \leq
    \frac{2n}{t + 2n}\left( C_f^{\otimes} + f(x_{0}) - f(x^{*}) \right),
  \end{equation*}
  where $C_f^{\otimes} = \sum_{i=1}^n C_f^{\mathcal{M}_i}$, and $C_f^{\mathcal{M}_i}$ 
  is the curvature of \(f\) restricted to the compact convex set
  $\mathcal{M}_i$ (more precisely to \(\{[y]_{1}\} \times \dotsm
  \times
  \{[y]_{i-1}\} \times \mathcal{M}_{i} \times \{[y]_{i+1}\} \times
  \dotsm \{[y]_{n}\}\) for any \(y\)).
  Moreover, we also have that
  \begin{equation*}
    \E{\min_{0\leq \tau \leq t}g(x_{\tau})}
    \leq \frac{6n}{t + 1} \left( C_f^{\otimes} + f(x_{0}) - f(x^{*}) \right).
  \end{equation*}
\end{theorem}

\subsection{The approximate Carathéodory problem}
\label{sec:caratheodory}

Carathéodory's theorem is a fundamental result in convex analysis.
It states that
any point \(x\) of an \(n\)-dimensional convex set \(\mathcal{X}\)
can be written as a convex combination of
at most $n+1$ extreme points of $\mathcal{X}$.
This theorem provides the name for the
\emph{approximate Carathéodory problem}
(illustrated in Figure~\ref{fig:Caratheodory}):
given a polytope $\mathcal{X}$ and a point $u\in\mathcal{X}$,
find a point $x\in\mathcal{X}$ which is the convex combination of
a small number of distinct vertices
$v_1, \dotsc, v_k$ of $\mathcal{X}$
such that
$\norm{x - u}\leq \epsilon$.
This approximation in the $\ell_p$-norm
has applications in \citet{barman15cara,barman2018approximating}
for computing approximate Nash equilibria and for finding the densest
subgraph of a graph. The $\ell_{2}$-ball case
has been used to speed up Frank–Wolfe methods in \citet{garber2013playing}
and to find a sparse solution that minimizes the loss of a linear
predictor in \citet{shalev2010trading}.

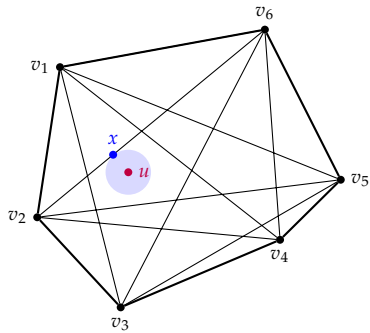
\begin{figure}[t]
\footnotesize
\centering
\begin{tikzpicture}
  \node foreach \p/\dir/\n in {
    (0.3, 3) /left/1,
    (0, 1) /left/2,
    (1.1, -0.2) /below/3,
    (3.2, 0.7) /below/4,
    (4, 1.5) /right/5,
    (3, 3.5) /above/6}
  [point, at={\p}, label=\dir:{\(v_{\n}\)}] (v\n) {};
  \draw (v1) -- (v2) -- (v3) -- (v4) -- (v5) -- (v6) -- (v1);

  \draw[thin] foreach \n/\m in {1/3, 1/4, 1/5, 1/6, 2/4, 2/5, 2/6,
    3/5, 3/6, 4/6}
  {(v\n) -- (v\m)};

  \fill[fill=blue, fill opacity=0.15] (1.2, 1.6) circle (0.3)
  coordinate (u);
  \node[point, purple, label={[purple]right:\(u\)}] at (u) {};

  \node[point, blue, label={[blue]above:\(x\)}] (x) at (1, 5.5/3) {};
 \end{tikzpicture}
 \caption{\label{fig:Caratheodory}%
   Approximate Carathéodory problem in two dimensions:
   approximating a point \(u\) of a polygon with convex combinations
   of at most two vertices, i.e., points on an edge or a diagonal.
   The point \(x\) is the closest approximation in the Euclidean norm.}
\end{figure}

For $\ell_{2}$-balls, it turns out that one can use the Perceptron 
algorithm \citep{rosenblatt1958perceptron} to obtain an approximate solution 
to the Carathéodory problem \citep{novikoff1963convergence}.
This was later extended to the $\ell_{p}$-norm in \citet{barman15cara},
$p\geq2$,
with an algorithm that involves finding the
exact solution to the convex decomposition, interpreting the convex
coefficients as probability distributions and then stochastically
sampling from them to find a smaller number of points $v_1,\ldots,v_k$,
showing that there exists a point $x$ which is the convex combination
of $k=\mathcal{O}(pD_p^2/\epsilon^2)$ vertices and such that
$\norm[p]{x-u}\leq\epsilon$, where $D_p$ denotes
the diameter of $\mathcal{X}$ in the \(\ell_{p}\)-norm.
The method is expensive as we first need to solve
the exact problem. More recently, \citet{mirrokni2017tight} showed that
a computationally efficient algorithm based on mirror descent
achieved the same result. Furthermore, they proved
that $\Omega(pD_p^2/\epsilon^2)$ is a lower bound 
on the number of vertices required to reach $\epsilon$-accuracy in the 
\(\ell_{p}\)-norm.

As noted in \citet{mirrokni2017tight}, the mirror descent algorithm
solution can be turned into a Frank–Wolfe algorithm
solution.
The Frank–Wolfe algorithm is particularly well suited for the
approximate Carathéodory problem, as it converges to $u$ by
selecting only one vertex at each iteration when applied to the
problem $\min_{x\in\mathcal{X}} \norm[p]{x-u}^{2}$
\citep{combettes19cara}.
We state the convergence result in Theorem~\ref{th:approxcar}.

\begin{theorem}
  \label{th:approxcar}
  Let $u \in \mathcal{X}$ and $p\geq2$. Then running the
  Frank–Wolfe algorithm (Algorithm~\ref{fw}) starting from a vertex
  $x_0\in\mathcal{X}$ 
  with $\gamma_t = 2/(t+2)$ for minimizing $f(x) = \norm[p]{x - u}^2$
  over the polytope $\mathcal{X}$ outputs an iterate $x_T$ after $T=\mathcal{O}(p D_p^2 / \epsilon^2)$ iterations such that
 \begin{equation*}
   \norm*[p]{x_T-u} \leq \epsilon.
 \end{equation*}
 The iterate $x_T$ is the convex combination of at most $T+1$ 
 distinct vertices of $\mathcal{X}$.
\end{theorem}

Further, improved or new bounds can be derived \citep{combettes19cara} 
by applying the convergence rates of FW when $\mathcal{X}$ is strongly 
convex or when $u\in\interior(\mathcal{X})$
(see Table~\ref{tab:fw-faster}), or by applying 
variants of FW, such as NEP-FW (Algorithm~\ref{alg:NEP-FW}) when $u$ 
lies in the convex hull of a subset of extreme points of $\mathcal{X}$
with small diameter, or HCGS (Algorithm~\ref{hcgs:fw})  on
$\min_{x\in\mathcal{X}}\norm[p]{x-u}$ when
\(1 \leq p \leq 2\) or \(p = + \infty\),
whose objective is nonsmooth
but Lipschitz-continuous.

\begin{example}[Approximate Carathéodory problem for
  \myindex{traffic routing} in servers]
  \label{ex:approx_Caratheodory}
  We consider a classical example of traffic routing
  where every device can communicate with only one device at the same
  time, like in old telephone networks, where switches are set to pair
  devices for communication, but similar situation also occurs in
  modern systems, like data centers.
  Thus a routing is a matching between devices,
  which changes from time to time to allow single devices
  communicating with multiple other devices.
  However, change of routing causes some downtime,
  so it is desirable to minimize the number of changes.
  See \citet{chen2013osa, farrington2010helios} for more on routing changes,
  including a more comprehensive setup.

  We consider the simplest problem: Find a routing for
  \(n\) senders
  needing to send some amount of data to \(n\) receivers,
  i.e., matchings between senders and receivers and the amount of time
  a matching should be used for communication.
  This amounts to a positive linear decomposition of
  the matrix \(M\) of data to be sent
  (i.e., sender \(i\) needs to send \(M_{i,j}\) data to sender \(j\))
  into \(0/1\)-characteristic matrices of the matchings.

  As common in the literature, we restrict to
  doubly stochastic matrices \(M\) in~$\R^{n\times n}$,
  i.e., that every sender and receiver needs to send the same amount
  of data altogether.
  Recall from Example~\ref{ex:Birkhoff}, that a
  matrix is doubly stochastic if its entries are nonnegative
  and its columns and rows sum up to $1$.
  The problem is then to decompose a given doubly stochastic matrix
  as a convex combination of permutation matrices,
  where a sparse decomposition, i.e., one with few permutation
  matrices is highly desirable.
  This is an active topic of research
  \citep[see][]{kulkarni2017minimum},
  beneficial for the performance
  of adaptive server networks
  \citep{liu2015scheduling, porter2013integrating},
  as hinted above.

Thus we solve the Carathéodory problem over the \myindex{Birkhoff polytope}.
By the Birkhoff--von Neumann theorem,
any doubly stochastic matrix is
the convex combination of at most
$(n -1)^2 + 1$ permutation matrices, however,
in many cases there are decompositions with much
fewer permutation matrices.
Finding the minimum number of permutation matrices
needed for a convex decomposition
is NP-hard \citep{dufosse2016notes}, and so this motivates 
attempting to find a convex decomposition that approximates 
the original matrix in a sparse manner.

In order to test the performance of different Frank–Wolfe variants, we
uniformly sample a point \(\lambda\)
from the \myindex{probability simplex} in $\R^m$,
and we randomly generate $m = 30$ permutation matrices in $\R^{n \times n}$, 
denoted by $\{v_1, \dotsc, v_m \}$.
We set $u = \sum_{i=1}^m \lambda_i v_i$,
and we want to solve $\min_{x \in \mathcal{X}} \norm[2]{u - x}^2$.

We test the performance of several Frank–Wolfe variants that
maintain an active set (this rules out variants like the 
DI-PFW algorithm), more concretely, we compare the performance of
the Parameter-free Lazy Conditional Gradient algorithm (LCG, 
Algorithm~\ref{alg:ParamFreeLCG}), the Lazy Away-step 
Frank–Wolfe algorithm (Lazy AFW, Algorithm~\ref{alg:ParamFreeLAFW}),
the Fully-Corrective Frank–Wolfe algorithm (FCFW, Algorithm~\ref{fcfw}),
the Frank–Wolfe with Nearest Extreme Point Oracle (NEP-FW, 
Algorithm~\ref{alg:NEP-FW}), and the Blended Conditional Gradient algorithm
(BCG, Algorithm~\ref{bcg}). We use the lazified variants of the CG/FW and the AFW 
algorithms as they generally obtain sparser solutions than their non-lazified 
counterparts.
		 
The results are shown in Figure~\ref{fig:approxCaratheodory} where we show the primal gap 
when solving the approximate Carathéodory problem,
in this case $\norm{u - x}^2_2$, in the wall-clock time elapsed,
and the number of vertices in the convex decomposition of the
current iterate. We can observe that given a fixed number of vertices, the FCFW algorithm has a smaller 
primal gap, or reconstruction error, than the other algorithms in the comparison. However, this 
comes at a cost, as at each iteration the FCFW algorithm has to minimize the objective function 
over the convex hull of the current active set, and this can be computationally expensive. The BCG 
algorithm performs well in this task,
as it achieves similar performance to the FCFW algorithm in
primal gap versus number of vertices used,
but runs faster in wall-clock time than the aforementioned algorithm.
\end{example}

\subsection{Video co-localization}

The success of many computer vision (CV) applications involving neural networks 
depends on the quality and the quantity 
of the labelled data being used to train a given architecture. Manually expanding a dataset
by annotating and labelling images is a cumbersome task, which is why there has been a 
lot of interest in finding efficient ways to label weakly annotated
data (i.e., data with some possibly usable extra information (the
annotation), where the extra information might be unreliable, be in a
free format, or has other usability issues). One such source
of data are the videos found on video sharing sites
like YouTube or Vimeo.  For example, one
could be interested in retrieving images of cats from the web to train a neural network. A 
simple search on YouTube for cat videos reveals a large number of videos that include the word 
``cat'' in the video title, so these videos potentially contain many frames which contain useful 
cat images. The problem of automatically 
annotating images of common objects (in this case cats) 
given the frames of a set of videos is called \emph{video co-localization}.

We briefly outline at a high level some of the key steps from the realm of CV that 
allow us to frame this as a tractable optimization problem, as discussed in \citet{joulin2014efficient}. We consider 
a given video to be formed by an ordered series of images. As a first step, for 
each image of every video, we generate a series of $m$ candidate bounding boxes that contain objects 
of any kind with well-defined boundaries \citep{alexe2012measuring} (as opposed to amorphous background 
elements). Note that by doing so we might be generating bounding 
boxes around all kinds of objects like cats, cows, telephone poles, or cars. If we have a total of $n$ images this means that we have a total of $nm$ candidate bounding boxes.

We now want to select one bounding box from each 
image that contains an object that is common to all the images. In order to do so, we stack 
the images from each video, from the beginning to the end of each video, and we concatenate 
the images from different videos together.
We now consider a directed graph structure in
which the nodes are the bounding boxes in each image. We number these nodes with indices ranging from $1$ to $nm$. We 
use $\mathcal{I}_i$ to denote the set of indices of the nodes obtained 
from the $i^{th}$ image, with $i\in [1,n]$. The indices of the images refer to 
the order in which they are stacked, that is, the $i^{th}$ image occurs before the 
$(i+1)^{st}$ image in the video.

For each image in the stack, 
we compute a similarity measure between each bounding box 
and all the bounding boxes in the next image in the stack. We denote this similarity 
measure between the node \(i\) and the node \(j\)
by $s_{i,j} = s_{j,i}$, with $i \in \mathcal{I}_k$ and $j \in \mathcal{I}_{k + 1}$ with $k \in [1, n-1]$. If these two nodes are not in temporally adjacent frames, we assign 
a value of zero for the similarity measure, that is $s_{i,j} = 0$ if $i \in \mathcal{I}_k$ and $j \notin \mathcal{I}_{k + 1} \cap \mathcal{I}_{k - 1}$. This similarity 
measure is based on temporal consistency 
between the bounding boxes, as we do not expect objects 
in a box to change drastically in size, position, and form between 
one frame and the next. If the $s_{i,j}$ is above 
a given threshold for some $i \in \mathcal{I}_{k}$ and $j \in \mathcal{I}_{k+1}$, 
we connect the two nodes with a directed edge that points from node
$i$ to node $j$. For
our simplified example with cats, we have concatenated the frames of two videos 
together, and we are considering only the first four images of the stacked images, with two bounding boxes each. The 
resulting graph structure, after computing which nodes to connect based on the temporal 
similarity measure, is depicted in 
Figure~\ref{fig:applications:video-colocalization-schematic}.

\begin{figure}
  \centering
  \includegraphics[width=0.95\textwidth, alt={Four images of a video
    with bounding boxes of two objects.
    The first three image depicts a scene near the outside wall of a
    house, and one of the objects identified is a moving cat.
    The other object is a nearby pole and the gutter on the wall,
    mistakenly identified as the same object.
    The fourth image depicts a room in presumably the same house,
    where the cat is identified, but the other object is a curtain
    holder on the floor.}]{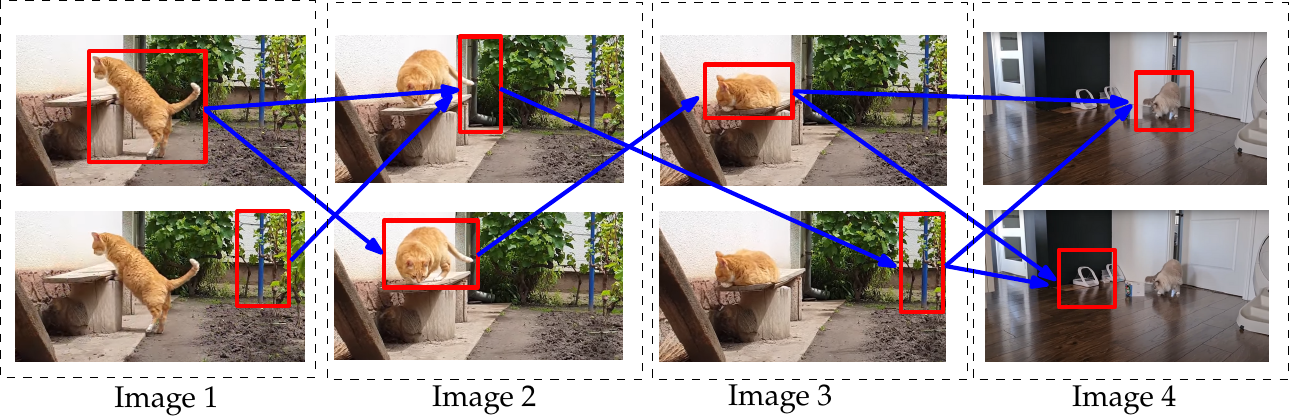}

  \caption{The directed graph built from the  series of images in a
    video
    as a help structure for video colocalization: identifying objects
    in the video.
    The nodes in the graph are the bounding boxes of objects in the
    images.
    Two nodes are connected if they
  are in temporally adjacent images/frames, and the bounding boxes satisfy a similarity measure 
  criterion.
  These are heuristics for connecting bounding boxes of the
  same object accross the images.
  In this very simple example, the algorithm creates a
  series of bounding boxes for each frame and selects two bounding boxes to keep. 
  The algorithm might not
 create bounding boxes over the same objects in consecutive frames, which is why 
 in the first frame the algorithm creates a bounding box around the cat and around the pole, 
 and in the second frame the algorithm creates a bounding box over the cat and the gutter 
 shown in the image.
  } \label{fig:applications:video-colocalization-schematic}
\end{figure}

Now that we have built a directed graph structure with the bounding
boxes,
we try to select
one (and only one) bounding box in each image so that some similarity
criteria are maximized, or respectively some
measure of dissimilarity is minimized. This 
would amount to finding a valid path from the first image in the first video, to the last image in the 
last video that minimizes some dissimilarity criteria.
In order to mathematically frame the problem,
we define a variable $z_{i}$ that has a value of $1$
if node $i$ has been selected for output, and $0$ otherwise, and we
define a variable $y_{i,j}$ that 
has a value of $1$
if both node $i$ and the node $j$ have been selected for output.
We use $p(k)$ and $c(k)$ to
denote the indices of the parent nodes and the child nodes of
node \(k\).
With this structure in mind, we write the following integer optimization problem:
{\allowdisplaybreaks
\begin{align}
    \min_{z,y} \quad &   \norm{U z}^2 - \innp{z}{\lambda}\\
    \text{s.t.} \quad &
    z_i \in \{0,1\}, y_{i,j} \in \{0,1\} \quad \forall i,j, \\
    & \sum_{j \in \mathcal{I}_i} z_j = 1 \quad \forall i, \\
    & \sum_{i \in p(j)} y_{i,j} = \sum_{i \in c(j)} y_{i,j} \quad \forall j \in \mathcal{I}_k, k\in [2, n-1], \\
    & y_{i,j}=z_iz_j \quad \forall i,j. \label{eq:video-colocalization}
\end{align}}%
The convex quadratic objective function $\norm{U z}^2 - \innp{z}{\lambda}$ in this problem measures the dissimilarity criteria for a 
given set of bounding boxes using temporal and spatial similarity metrics. In 
Figure~\ref{fig:applications:video-colocalization-schematic} this function would measure 
the dissimilarity between the images in the bounding 
box $z_1$ and $z_3$ and $z_4$, between $z_2$ and $z_3$ and so forth. We refer the interested 
reader to \citet{joulin2014efficient} and \citet{tang2014co} for the full details of the objective function. 
Due to the difficulty of solving this integer programming problem over a non-convex set, we instead focus on solving a convex relaxation of the problem presented above, and we 
minimize the objective function over the 
convex hull of the non-convex set. The convex hull of the set of constraints 
shown in Equation~\eqref{eq:video-colocalization} is
the \myindex{flow polytope}.
Recalling that the graph has \(nm\) vertices, and
denoting the number of edges of the graph) by $E$
(i.e., the number of entries of $y$),
the complexity of solving the problem in
Equation~\eqref{eq:video-colocalization}, or computing a projection,
with an interior point method is
$\mathcal{O}( (nm)^{3} E+ E^{2} )$.
On the other hand we can solve a linear programming problem with the shortest path problem 
with complexity $\mathcal{O}(nm + E)$,
which motivates the use of Frank–Wolfe algorithms
when the number of images and bounding boxes processed is large.
In \citet{joulin2014efficient}, the authors note that the proposed
Frank–Wolfe method for solving the relaxed problem over the convex hull of the non-convex
set shown in Equation~\eqref{eq:video-colocalization}
performs better than several existing algorithms, like the one in \citet{prest2012learning}, 
for various classes of objects in a well-known video-colocalization dataset known as the YouTube-Objects dataset. However, for simple, non-deformable objects, the algorithm in
\citet{prest2012learning} performed better than the proposed
Frank–Wolfe-based methodology. The authors in
\citet{joulin2014efficient} also remark that some algorithms, such as
the
one presented in \citet{papazoglou2013fast} outperform the proposed 
Frank–Wolfe algorithms for all classes in this dataset due to the
relatively small size of the
dataset, and the fact that the differences among objects in a given class 
are very large (for example, the different images of aeroplanes in different videos are extremely 
different, making learning a common aeroplane model difficult). 

\begin{figure}
  \footnotesize
  \centering
  \begin{tabular}{cc}
    Approximate Carathéodory & Video Co-localization \\[\smallskipamount]
    \includegraphics[width=.45\linewidth, alt={Plot of distance and
      time for various algorithms, behaving roughly similarly: sharp
      initial decline, moderate further decline.  BCG decreases more
      rapidly than other at the very end.}]{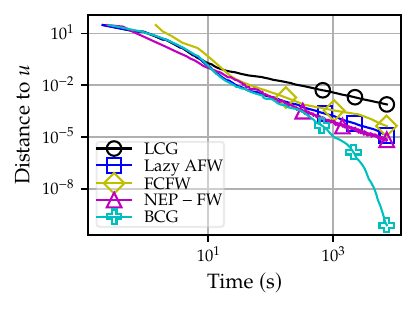}
  &
  \includegraphics[width=.45\linewidth, alt={Plot of primal gap in
    first-order oracle calls, algorithms having various not easily
    comparable behavior. }]{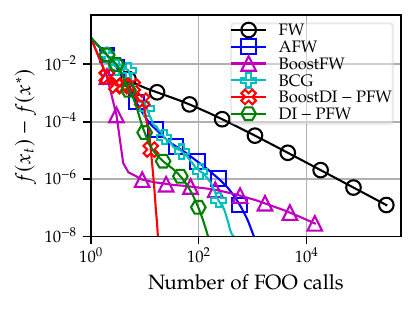}
  \\
  \includegraphics[width=.45\linewidth, alt={Plot of sparsity in time,
  all algorithms steadily increasing roughly the same way.}]{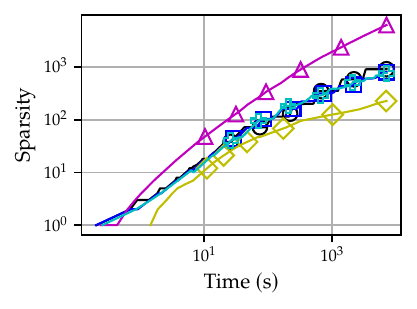}
  &
  \includegraphics[width=.45\linewidth, alt={Primal gap in the number of
  linear minimizations, a slightly deformed variant of the plot in the
  number of first-order oracle
  calls.})]{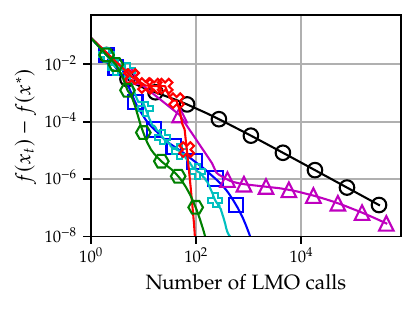}
  \\
  \includegraphics[width=.45\linewidth, alt={Plot of distance in
    sparsity, slightly different to the plot distance in
    time.}]{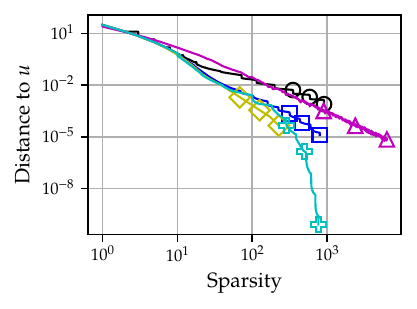}
  &
  \includegraphics[width=.45\linewidth, alt={Plot of primal gap in
    time, interpolating the plots in first-oracle calls and linear
    minimizations.}]{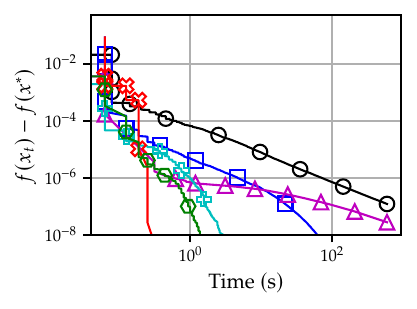}
  \end{tabular}

  \caption{\label{fig:approxCaratheodory}
    \emph{Approximate Carathéodory problem for network traffic routing (left) and video co-localization (right):}
    The routing problem formally asks for
    finding a sparse convex combination of
    permutation matrices (the vertices)
    close to a given doubly stochastic matrix \(u\).
    Here we present on the left column the
     evolution of distance and sparsity (number of vertices
    in the linear decomposition of current iterate \(x_{t}\))
    in wall-clock time and compared to each other.
    Distance is the square of the entrywise $\ell_2$-norm. In the images
    on the column on the right we depict the primal gap convergence for the video
    co-localization problem with the \texttt{aeroplane} dataset
    presented in \citet{lacoste15}. The function being minimized is a
    convex quadratic function, and the feasible region
    is the flow polytope.}
\end{figure}

In the right-side column of Figure~\ref{fig:approxCaratheodory} 
we present a comparison of Frank–Wolfe variants when solving a video
co-localization
problem. The algorithms presented in the comparison are the vanilla FW algorithm 
(as well as its boosted variant), the AFW algorithm, the DI-PFW algorithm (as well 
as its boosted variant), and the BCG 
algorithm. The images and the data 
used for this example come from the \texttt{aeroplane} image dataset, used in \citet{lacoste15} and 
\citet{joulin2014efficient}, which leads to a problem with $660$ variables. 

For the BoostedFW and the BoostedDI-PFW algorithms we set $\delta = 10^{-15}$, 
and $K = \infty$. For these parameter values we can see in
Figure~\ref{fig:approxCaratheodory} that the BoostDI-PFW
algorithm outperforms all other algorithms; it is
interesting to observe that this is also true for
the number of LMO calls, for which the boosting procedure 
is particularly greedy by nature.
All the algorithms are run until 
the total wall-clock time reaches $600$ seconds, or we reach a
solution with a primal gap below $10^{-8}$.

Given that the feasible region 
is a polytope, more concretely a $0/1$ polytope, and 
that the objective function is smooth and strongly convex
with a condition number of $L/\mu \approx 30$, we 
expect the AFW, the BCG, and the DI-PFW algorithms 
to converge linearly in primal gap, in contrast to 
the FW algorithm, which converges at sublinear rate.
This matches with what we observe in the figure that
shows the converge in primal gap vs. FOO calls. The
BoostedFW, on the hand, can converge in the worst-case
at a sublinear rate, like the vanilla FW algorithm,
which is what we observe in this experiment. 
However, in experiments it often converges at a 
linear rate for this class of problems (see 
Theorem~\ref{th:boostfw2}.
Finally, the BoostedDI-PFW algorithm exhibits
a linear convergence rate for this problem instance, 
requiring fewer FOO and LMO calls than any of the other 
algorithms to achieve the final primal gap tolerance.
For performance in wall-clock time, we remark
that both the DI-PFW and the BoostedDI-PFW algorithms 
do not require maintenance of an explicit active set, 
which the BCG and the AFW algorithms require. Operations 
that involve an active set often scale as 
$\mathcal{O}(\size{S} n)$, where $\size{S}$ is 
the cardinality of the active set, and $n$ is 
the dimension of the problem. If the active set 
grows too large, these operations can quickly 
become computationally expensive. In the figure 
that shows the primal gap convergence rate in
wall-clock time we can clearly see the difference
between the linearly convergent algorithms that 
do not require an active set (BCG and AFW), and 
the algorithms that do not require an active 
set (DI-PFW and BoostedDI-PFW). Moreover, we 
can also see the difference between the AFW 
and the BCG algorithms. Not only does the BCG 
algorithm converge faster than the AFW algorithm 
in iteration count, but as it maintains
an active set of smaller cardinality, it outperforms 
AFW significantly in wall-clock time. It is
important to note however that DI-PFW and 
BoostedDI-PFW can only be applied to feasible regions 
that admit a special structure and not to 
arbitrary polytopes or compact convex sets 
(see Section~\ref{sec:decomposition-invariant} 
for a discussion)

\subsection{The coreset problem}
\label{sec:coreset}

Many modern applications in machine 
learning, statistics, and computational 
geometry rely on using vast amounts of 
data to train a model, or to compute some estimator or metric. 
Solving these problems approximately (or 
exactly, if this is possible at all)
typically has complexity that depends on the
number of data points used. In many cases, only a 
small number of data points are relevant or suffice to 
contain most of the information. This often means 
that solving the original problem over this 
smaller dataset yields approximately the same
result as using the original dataset, while
being cheaper to compute.

For simplicity of exposition,
we consider a similar, simply formulated geometric problem here:
the \emph{Minimum Enclosing Ball}\index{Minimum Enclosing Ball problem}
(MEB) problem, from computational geometry, 
which is to find the ball $B(c_{\mathcal{S}},r_{\mathcal{S}})$
of minimum  radius $r_{\mathcal{S}}>0$ containing 
a given set $\mathcal{S}$ of $m$ points in $\R^n$, 
where $c_{\mathcal{S}}$ and $r_{\mathcal{S}}$ denote the 
center and the radius of the MEB of $\mathcal{S}$, respectively.
Computing the MEB of $\mathcal{S}$ can be done exactly with a randomized 
algorithm with expected complexity $\mathcal{O}(m (n+2)!)$ \citep{welzl1991smallest}.
Note that the factorial complexity
in the dimension $n$ makes the algorithm impractical for large dimensions.
One question that arises is: if we find the MEB of $\mathcal{S}'$,
a subset of $\mathcal{S}$, will it be a good approximation (where 
we have yet to define what ``good'' is)
to the MEB of $\mathcal{S}$? That is, can we solve 
a cheaper problem with fewer points to obtain a
``good-enough'' solution?

This is the intuition behind the \emph{coreset} 
problem.
Many different
definitions of an $\epsilon$-coreset for the MEB 
problem have arisen in the literature. At first, 
an $\epsilon$-coreset with $\epsilon > 0$ 
of $\mathcal{S}$ for the MEB problem 
was defined as a set $\mathcal{S}'\subseteq \mathcal{S}$ 
such that MEB \(B(c_{\mathcal{S}'},r_{\mathcal{S}'})\) of
\(\mathcal{S}'\)
scaled by \(1 + \epsilon\), i.e., the ball \(B(c_{\mathcal{S}'},
(1 + \epsilon) r_{\mathcal{S}'})\)
contains
the MEB of $\mathcal{S}$ \citep{clarkson03}. However, in this section
we focus on the broader definition from
\citet{clarkson08}, recalled as
Definition~\ref{applications:def:coreset} below
(which is closely related to the one in \citet{yildirim2008two}). 
A subset $\mathcal{S}'\subseteq \mathcal{S}$ is an
\emph{$\epsilon$-coreset} for $\mathcal{S}$ if $r_{\mathcal{S}}^2 \leq
r_{\mathcal{S}'}^2 /(1 - 2\epsilon)$. Note that since 
$c_{\mathcal{S}} \in \conv{\mathcal{S}}$ by Carathéodory's
theorem there always exists a $0$-coreset of at most $n + 1$
points for the MEB problem, where $n$ is the dimension
of the space containing \(S\) (i.e., \(S \subseteq \mathbb{R}^{n}\)).

The concept of coreset problem can be applied to a wide 
variety of scenarios, which all share a common form, and 
which are suited to the application of
Frank–Wolfe algorithms \citep{clarkson08} (this expands and
sharpens the analysis of the coreset problem \citep{clarkson03}).
Many of these problems, like MEB, can
be phrased as a concave 
maximization problem: 
\begin{equation}
  \label{coreset:primal}
  \max_{x \in \Delta_{m}} f(x),
\end{equation}
where $f \colon \mathbb{R}^m \to \mathbb{R}$
is a concave differentiable function and
$\Delta_{m} =\{x\in\mathbb{R}^m\mid
x^\top\allOne=1,x\geq0\}$ is the \myindex{probability simplex}.
For example, given a set of points $\mathcal{S} = \{
a_1, \dotsc, a_m\}$ in $\R^n$, the MEB problem can be phrased as
\begin{align}\label{MEB:primal}
    \min_{c\in\mathbb{R}^n,r > 0} \quad &   r^2\\
    \text{s.t.} \quad &
    \norm[2]{a_{i} - c}^{2} \leq r^{2} \quad \forall 1 \leq i \leq m.
\end{align}
Taking the Lagrange dual leads to
\begin{equation}
  \label{MEB:dual}
  \max_{x \in \Delta_{m}} \underbrace{\sum_{i = 1}^{m} x_{i} \norm[2]{a_{i}}^{2} -
  \norm*[2]{\sum_{i=1}^{m} x_{i} a_{i}}^{2}}_{f(x)}.
\end{equation}
Thus the dual problem of the MEB problem has the form shown in 
Equation~\eqref{coreset:primal}. The primal problem 
satisfies Slater's condition (take any $c\in \conv{\mathcal{S}}$ 
and a large enough~$r$), and thus strong duality holds, 
therefore let
$x^* \in \argmax_{x \in \Delta_{m}} f(x)$, then one recovers the MEB
for $\mathcal{S}$ via $c_{\mathcal{S}} =
\sum_{i=1}^m x_i^* a_i$ and 
$r_{\mathcal{S}}^2=f(x^*)$. Note that Equation~\eqref{coreset:primal} 
is a maximization problem, and the primal gap and 
Frank–Wolfe gap at $x$ are given by $f(x^*) - f(x)$,
and $g(x) = \max_{v\in \Delta_m} \innp{\nabla f(x)}{v - x} = \max_{1\leq i \leq m} \nabla f(x)_i - \innp{\nabla f(x)}{x}$, respectively. 

For any $\mathcal{N}\subseteq\{1, 2, \dotsc, m\}$, let
$\Delta^{\mathcal{N}} \defeq \conv{e_{i}\mid i \in \mathcal{N}}$, 
where the $e_i$ are the coordinate vectors, and
let $x^{\mathcal{N}}\in\argmax_{x\in \Delta^{\mathcal{N}}}f(x)$ for any
$\mathcal{N}\subseteq\{1, 2, \dotsc, m\}$. 
Before formally defining a coreset,
we first recall a local version of the curvature $C$
presented in
Section~\ref{sec:affine-invariance}:
 \begin{equation*}
  C_f^* \defeq
  \sup_{\substack{z \in \Delta_m \\
      \alpha \in \mathbb{R} \\
      y = x^* + \alpha (z - x^*) \in \Delta_m}}
  \frac{1}{\alpha^{2}}
  \bigl(
    f(x^*) - f(y) + \innp{\nabla f(x^*)}{y - x^*}
  \bigr)
  \geq
  0.
\end{equation*}
With this in mind:
\begin{definition} \label{applications:def:coreset}
 Given a concave function $f(x)$, and $\epsilon>0$, an 
 \emph{$\epsilon$-coreset} for Problem~\eqref{coreset:primal} 
 is a subset $\mathcal{N}\subseteq\{1, 2, \dotsc, m\}$ 
 satisfying $f(x^{*}) -  f(x^{\mathcal{N}}) \leq 2\epsilon C_f^*$.
\end{definition}

Looking at Definition~\ref{applications:def:coreset}, one can quickly see that
returning an $0$-coreset is trivial, simply return \(\mathcal{N} = \{1, 2, \dotsc, m\}\). 
However, this solution is of no use, as in most problems 
the whole point of a coreset is finding a 
set \(\mathcal{N}\) of cardinality much smaller than $m$ such that computing 
$x^{\mathcal{N}}\in\argmax_{x\in \Delta^{\mathcal{N}}}f(x)$ is much cheaper 
than computing $x^{*}$ by solving Problem~\eqref{coreset:primal}, while ensuring that 
the solution found using \(\mathcal{N}\) satisfies 
$f(x^{*}) - f(x^{\mathcal{N}}) \leq 2\epsilon C_f^*$.

\begin{remark}
The definition of a $\epsilon$-coreset in 
Definition~\ref{applications:def:coreset} is different 
than the one presented in \citet{clarkson08}, where an 
$\epsilon$-coreset is defined as a subset $\mathcal{N}$ 
satisfying $g(x^{\mathcal{N}}) \leq 2\epsilon C_f^*$, where $g(x)$
denotes the Frank–Wolfe gap at $x$ for Problem~\eqref{coreset:primal}. It is important
to state that both definitions qualitatively try to characterize what are good 
subsets $\mathcal{N}$ that will provide a ``good-enough'' solution to the original 
problem. However, there are two reasons why we use 
Definition~\ref{applications:def:coreset}, as opposed to the definition in 
\citet{clarkson08}.

First, note that 
a point $x$ encoded by a subset $\mathcal{N}$, such 
that \smash{$f(x^{*}) -  f(x) \leq 2\epsilon C_f^*$}, encodes an
$\epsilon$-coreset according to 
Definition~\ref{applications:def:coreset}, 
as $f(x) \leq f(x^{\mathcal{N}})$.
From a practical point of view, if 
we find a point such that $g\left(x\right) \leq 2\epsilon C_f^*$, this point 
certifiably encodes an $\epsilon$-coreset according to 
Definition~\ref{applications:def:coreset}, as 
$f(x^{*}) - f(x) \leq g(x)$, and we do not
even need to find $x^{\mathcal{N}}\in\argmax_{x\in \Delta^{\mathcal{N}}}f(x)$ in order to 
show that $\mathcal{N}$ is an $\epsilon$-coreset according to Definition~\ref{applications:def:coreset}. However, according to the definition of 
\citet{clarkson08}, a point that satisfies $g\left(x\right) \leq 2\epsilon C_f^*$ 
is not an 
$\epsilon$-coreset, as there is no way to 
certify that $g(x^{\mathcal{N}}) \leq 2\epsilon C_f^*$ other than
solving $x^{\mathcal{N}}\in\argmax_{x\in \Delta^{\mathcal{N}}}f(x)$ and computing its 
Frank–Wolfe gap. This is due to the
fact that even though \smash{$g(x)\leq 2\epsilon C_f^*$}, it is possible that 
$g(x^{\mathcal{N}})> 2\epsilon C_f^*$, despite 
the fact that the $\mathcal{N}$ determined by $x$ such that 
$g(x) \leq 2\epsilon C_f^*$ would allow us to obtain
a good solution to the problem in Equation~\eqref{coreset:primal}.

Secondly, using Definition~\ref{applications:def:coreset} allows us to 
analyse in a more direct fashion the algorithms presented later 
on in this section, as most of the convergence rates presented in this survey 
are primal gap convergence rates, and quantifying the number of iterations 
$t$ until \smash{$f(x^{*}) -  f(x_t) \leq 2\epsilon C_f^*$} becomes
straightforward (as opposed to quantifying the number of iterations until 
\smash{$g(x^{\mathcal{N}_t}) \leq 2\epsilon C_f^*$}).
\end{remark}

\begin{remark}
Returning to the case of the 
MEB problem, if we are given a set of points $\mathcal{S} = \{
a_1 , \dotsc, a_m \}$, and we note that, $f(x^*) =
r^2_{\mathcal{S}}$, then an $\epsilon$-coreset encodes a set 
of points $\mathcal{S}' = \{ a_i  \mid i \in \mathcal{N} \}$ 
through $\mathcal{N}$, such that $r^2_{\mathcal{S}} - 
r^2_{\mathcal{S}'} \leq 2\epsilon C_f^*$. That is, a coreset gives us 
a smaller set of points that give us a solution that is 
$\epsilon$-``good-enough'' compared to the solution to
the original MEB problem. To see this more clearly, 
note that for the case of the MEB, 
problem, we have that
\begin{equation}
  C_f^* =
  \sup_{\substack{z \in \Delta_m \\
      \alpha \in \mathbb{R} \\
      y = x^* + \alpha (z - x^*) \in \Delta_m}}
  \frac{1}{\alpha^{2}}
  \norm*[2]{\sum_{i = 1}^m a_i y_i - c_{\mathcal{S}}}^{2}
  = r_{\mathcal{S}}^2.
\end{equation}
This means 
that solving the MEB problem over a \smash{$2\epsilon C_f^*$}-coreset 
yields a \smash{$r^2_{\mathcal{S}'}$} such that~$r^2_{\mathcal{S}} \leq r^2_{\mathcal{S}'}/(1 - 2\epsilon)$.
\end{remark}

Due to the fact that the Frank–Wolfe algorithm
accesses at most one vertex of the probability simplex 
at each iteration, it has been successfully applied 
to the coreset problem (see Algorithm~\ref{fw:coreset}, which 
is nothing but the vanilla FW algorithm with line search applied 
to the maximization problem in Problem~\eqref{coreset:primal}, and 
run until the Frank–Wolfe gap is below $\epsilon$), and
has been used to characterize the 
cardinality of the set $\mathcal{N}$ that can allow us to 
reach a given $\epsilon$-coreset for a given problem. 

\begin{algorithm}
\caption{Frank–Wolfe for the coreset problem \citep{clarkson08}}
\label{fw:coreset}
\begin{algorithmic}[1]
  \REQUIRE Initial point $x_0 = \argmax_{1 \leq i \leq m}f(e_i)$,
    index set $\mathcal{N}_0=\{ x_0 \}$, and accuracy $\epsilon>0$
  \ENSURE Subset $\mathcal{N}_t\subseteq\{1, 2, \dotsc, m\}$ and
    solution $x_t\in\mathcal{S}^{\mathcal{N}_t}$
\FOR{$t=0$ \TO \dots}
  \STATE $i_t \gets \argmax_{1 \leq i \leq m} [\nabla f(x_t)]_i$
  \STATE $\gamma_t \gets \argmax_{\gamma \in \R}
    f(x_t + \gamma (e_{i_t} - x_t ))$
  \STATE $x_{t+1} \gets x_t + \gamma_t (e_{i_t} - x_t)$
\STATE$\mathcal{N}_{t+1}\leftarrow \mathcal{N}_{t} \cup \{i_t\}$
\STATE$t\leftarrow t+1$
\ENDFOR
\end{algorithmic}
\end{algorithm}

\begin{theorem}
 \label{th:coreset}
Given a concave function $f$, the vanilla Frank–Wolfe algorithm with line search (Algorithm~\ref{fw:coreset})
returns a $2\epsilon C/C_f^*$-coreset
in $t = \mathcal{O}(1/\epsilon)$ iterations,
resulting in a coreset $\mathcal{N}_{t}$ of cardinality at most $t$.
\end{theorem}

Thus the coreset problem attempts to balance
the cardinality of \(\mathcal{N}_{t}\)
with the accuracy one can expect to achieve with \(\mathcal{N}_{t}\). The way in which the balance
tips towards one side or the other depends on the application at hand. As a 
rule of thumb, if computing 
$x^{\mathcal{N}_{t}}$ depends very
unfavourably on the cardinality of $\mathcal{N}_{t}$, one will
settle for coarser, or less accurate solutions, 
to the coreset problem. If on the other hand, computing $x^{\mathcal{N}_{t}}$ is
expensive, but 
manageable, we will tip towards finer, or more accurate solutions to the coreset 
problem. We will  
see in Example~\ref{example:MEB}, that for the MEB problem in $\R^3$, one can often 
find coreset of points of cardinality around $10$ that allows us to compute an MEB 
solution that is extremely close to the solution to the MEB problem
using around
\num{500000} points.

To obtain coresets of smaller cardinality \citet{clarkson08} 
essentially proposed to use the Fully-Corrective 
Frank–Wolfe algorithm. The
convergence rate of this algorithm for the coreset problem is the same 
as the one shown for the vanilla Frank–Wolfe algorithm with line search. One can
also blend the steps taken by the FCFW algorithm with 
those taken by the AFW algorithm obtaining an Away-step/\allowbreak Fully-Corrective 
Frank–Wolfe blend for the coreset problem \citep{clarkson08opti}.

\begin{example}[Coreset for the MEB problem] \label{example:MEB}
We showcase the performance of the different Frank–Wolfe variants for
the coreset problem with an MEB application. As stated before, at a
high level, this problem aims to find the ball with the smallest
volume that encloses a given set of points. We use two different sets
of points in $\R^3$ for this application, which are commonly used in
computer graphics, namely, the \emph{Stanford bunny}
\citep{turk1994zippered}, and the \emph{Stanford dragon}
\citep{curless1996volumetric}, which have \num{35947} and \num{566098}
points respectively,
shown in Figure~\ref{fig:MEB_bunny_dragon}.
For example, the MEB problem has been used for fast
collision detection, in which one has to detect in a time-critical manner if a collision between 
two objects has occurred, or is about to occur 
\citep{hubbard1996approximating}. In this case, each object (typically 
represented as a large set of vertices) is approximated by a 
collection of a few balls that contain 
these vertices. These balls are used to approximately detect 
in a fast manner if a collision between 
these two objects has occurred or is about to occur, 
as it is typically 
computationally easier to check if a small set of balls intersects with 
another small set of balls, than to check if two polyhedra with a large 
number of vertices intersect.

\begin{figure}
\centering
\includegraphics[width=.4\linewidth, alt={A sitting rabbit in
  gray drawing.}]{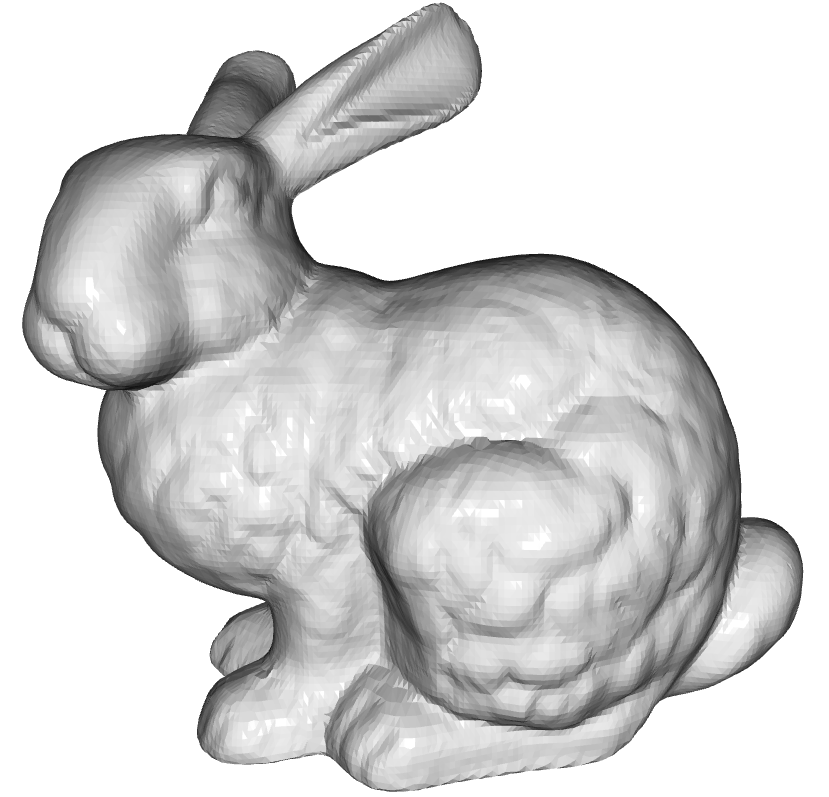}
\qquad
\includegraphics[width=.4\linewidth, alt={A snake-shaped dragon in
  gray drawing.}]{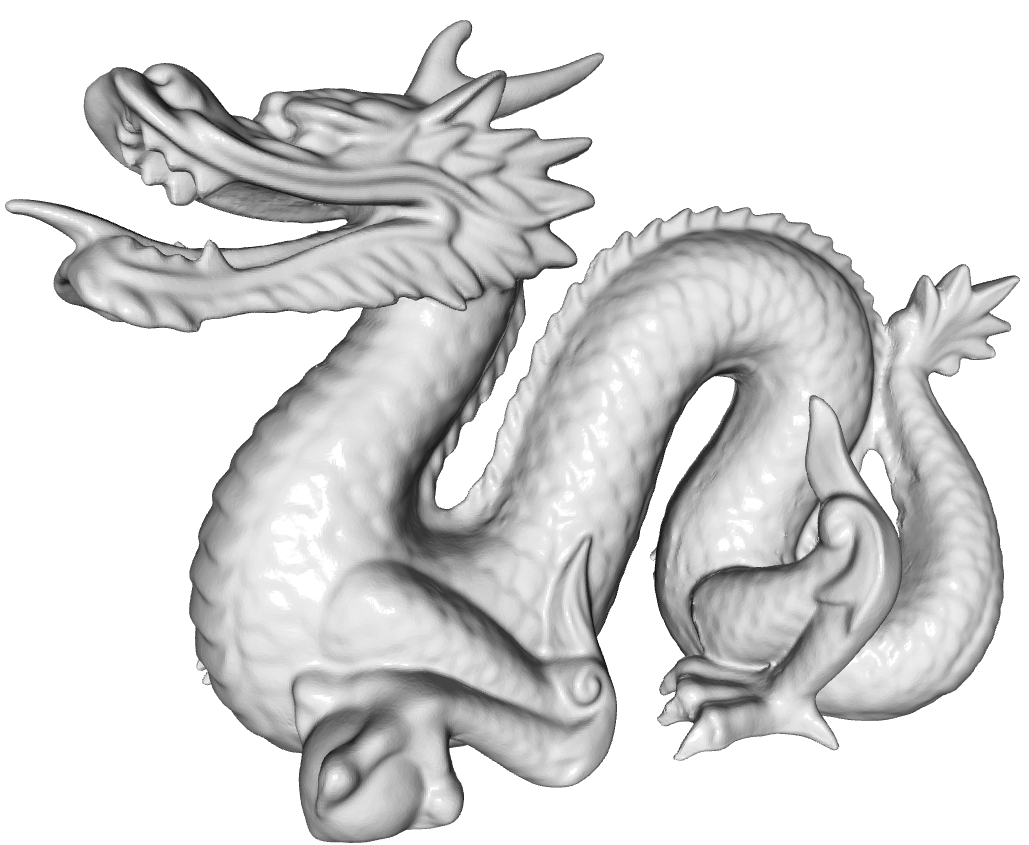}

\caption{\label{fig:MEB_bunny_dragon}
          Set of points for the coreset MEB (minimal enclosing ball)
          experiment, here forming bodies in \(3\) dimensions:
          the Stanford bunny (\num{35947} points) on the left,
          and the Stanford dragon (\num{566098}) on the right
          from the
          \href{https://graphics.stanford.edu/data/3Dscanrep/}{3D Scanning Repository}
          of Stanford Computer Graphics Laboratory.}
\end{figure}

Computing the MEB of a series of $m$ points in $\R^n$
can be done exactly with Welzl's algorithm (with expected complexity
$\mathcal{O}(m (n+1)(n+1)!)$ \citep{welzl1991smallest}).
For the Stanford bunny (with $m = \num{35947}$ points),
it takes approximately $630$ seconds to compute the MEB
using all the data points (which both have a MEB or radius
approximately $0.1$ meters),
using the implementation of Welzl's package in the
\href{https://github.com/marmakoide/miniball}{\texttt{miniball}}
package.  For the Stanford dragon, we were not able to run
the aforementioned implementation of Welzl's algorithm
as it exceeded the memory limit of the machine being used. Note that
we have not used faster (and more efficient) packages based on
Cython/Python bindings to Welzl's algorithm in C++ to provide a fair
comparison to the rest of the code, which is based exclusively in
Python, and have therefore used a MEB implementation fully written in
Python. For full transparency, using the
\href{https://github.com/hirsch-lab/cyminiball}{\texttt{cyminiball}}
package, with Cython/Python bindings to C++, one can solve the MEB
problem for the Stanford dragon in $0.5$
seconds with a low memory footprint.

\begin{figure}
  \footnotesize
  \centering
  \begin{tabular}{cc}
    \emph{Stanford bunny} & \emph{Stanford dragon} \\[\smallskipamount]
    \includegraphics[width=.45\linewidth, alt={Plot of estimated
      accuracy in iterations, FCFW sharply decreasing, FW and AFW
      making less but steady progress in a roughly similar
      manner.}]{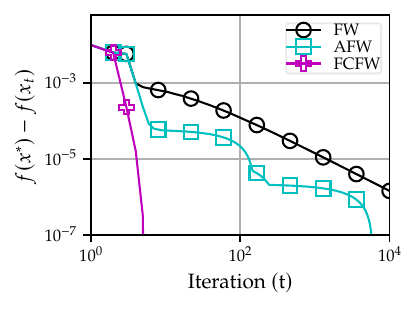}
  &
  \includegraphics[width=.45\linewidth, alt={Plot of estimated
    accuracy in iterations,
    similar the previous such
    plot.}]{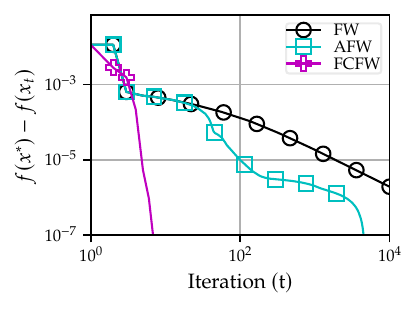}
  \\
  \includegraphics[width=.45\linewidth, alt={Plot of true accuracy in
    iterations, FCFW and FW both sharply decreasing, AFW starts
    sharply decreasing much
    later.}]{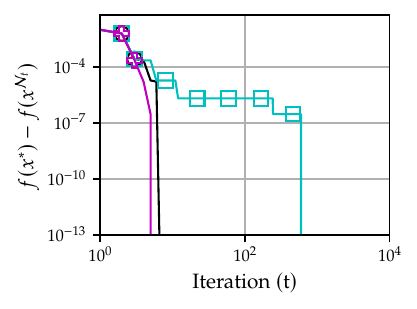}
  &
  \includegraphics[width=.45\linewidth, alt={Plot of true accuracy in
    iterations, all three algorithms sharply decreasing after a
    varying amount of initial
    phase.}]{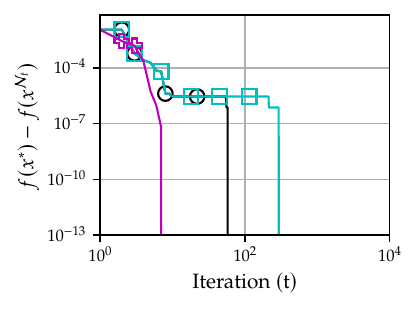}
  \\
  \includegraphics[width=.45\linewidth, alt={Plot of coreset size in
    iterations, after an initial sharp increase, FCFW and FW remain
    constant though at slightly different sizes.
    AFW fluctuates at a much small size after an initial sharp
    increase.}]{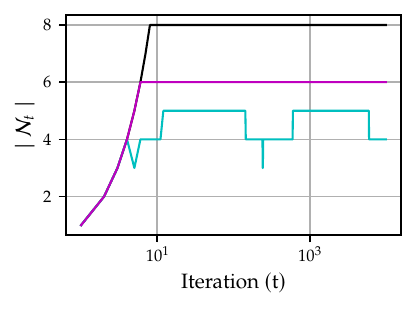}
  &
  \includegraphics[width=.45\linewidth, alt={Plot of coreset size in
    iterations, after a roughly identical sharp initial increase, and
    short constant phase, the algorithms start to behave differently.
    FCFW remains constant, FW slowly transitions to a larger constant,
    while AFW starts decreasing and fluctuation around a smaller
    size.}]{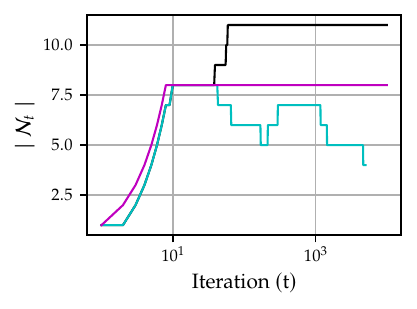}
  \end{tabular}

  \caption{\emph{Coreset MEB (minimal enclosing ball):}
    Performance of Frank–Wolfe algorithms for finding coresets
    of the Stanford animals (see Figure~\ref{fig:MEB_bunny_dragon}).
    The figures show the per iteration progress:
    the first row depicts an upper bound
    \(f(x^{*}) - f(x_{t})\)
    on the accuracy of
    the current coreset \(\mathcal{N}_{t}\)
    obtained from the current iterate \(x_{t}\)
    with minimal computation,
    the second row depicts the true accuracy
    \(f(x^{*}) - f(x^{\mathcal{N}_{t}})\)
    of the coreset,
    and the third row depicts the size \(\size{\mathcal{N}_{t}}\)
    of the current coreset.
    (The objective function \(f\) is the concave function
    from Equation~\eqref{MEB:dual}.)
  The second shows FCFW quickly identifying a very good coreset
  in a few iterations, at the cost of these iterations being more computationally 
  expensive than for the other algorithms.}
  \label{fig:coresetMEB}
\end{figure}

\looseness=1
In the left column of Figure~\ref{fig:coresetMEB} we see the
performance of the FW algorithm (Algorithm~\ref{fw:coreset}), the AFW
algorithm (Algorithm~\ref{away}), and the FCFW algorithm
(Algorithm~\ref{fcfw}) on the coreset problem for the MEB setting
using the data from the Stanford bunny. It is important to remark that
all the algorithms identify a coreset in less than $10$
iterations with which one can obtain a high accuracy solution to the
MEB problem with the full dataset.
The cardinality of this coreset is $8$ for the FW algorithm,
around 3--5 for the AFW algorithm,
and $6$ for the FCFW algorithm.
The FCFW algorithm identified the appropriate coreset,
and obtained a high accuracy solution to the MEB problem
with Welzl's algorithm in less than a second,
compared to the $630$ seconds it took to compute the MEB
with the full set of \num{35947}~points.

In the right column of Figure~\ref{fig:coresetMEB} we see the
performance of the algorithms mentioned in the previous paragraph
using the data from the Stanford dragon. Again, all the algorithms
identify a coreset in less than $10$ iterations
obtaining a high accuracy solution to the
MEB problem with the full dataset (which has \num{566098}
points in total).
\end{example}

\subsection{Adversarial attacks}

Given a (trained) model, adversarial
machine learning consists of finding input data points for which the model
predicts deceptive labels, the end goal being to check the robustness of 
the model.
In other words, a model is an algorithm with a data point as input and
a label as output.
In this context, the role of the \emph{adversary} is to find a data
point $x$ that is $\epsilon$-close to a data point $\hat{x}$ with label $\hat{y}$
and for which the model predicts a label $y\neq\hat{y}$ (Figure~\ref{fig:adv}).
More precisely, the adversary can specify the label $y\neq\hat{y}$ and search for 
a data point $x$ such that $\norm[p]{x-\hat{x}}\leq\epsilon$.
The distance between $x$ and $\hat{x}$ is usually measured
in the \(\ell_{p}\)-norm, $p\geq1$, the \(\ell_{\infty}\)-norm being particularly often
used in the literature on adversarial learning. The problem can be
formulated as follows:
\begin{equation*}
  \min_{\norm[p]{x - \hat{x}} \leq \epsilon} \ell(x,y),
\end{equation*}
where $\ell(x,y)$ is a classification loss function.

\begin{figure}[b]
 \centering
 \footnotesize
 \begin{tabular}{ccc}
   \includegraphics[height=7em, alt={A close shot of a daisy flower
     in a grassy field.}]{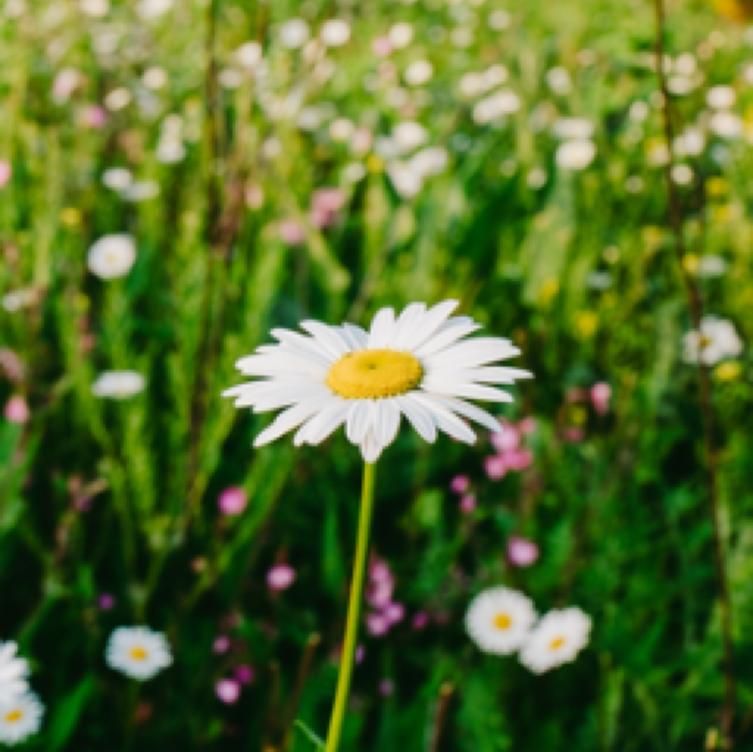}
   &
   \includegraphics[height=7em, alt={Image with randomly colored
     pixels}]{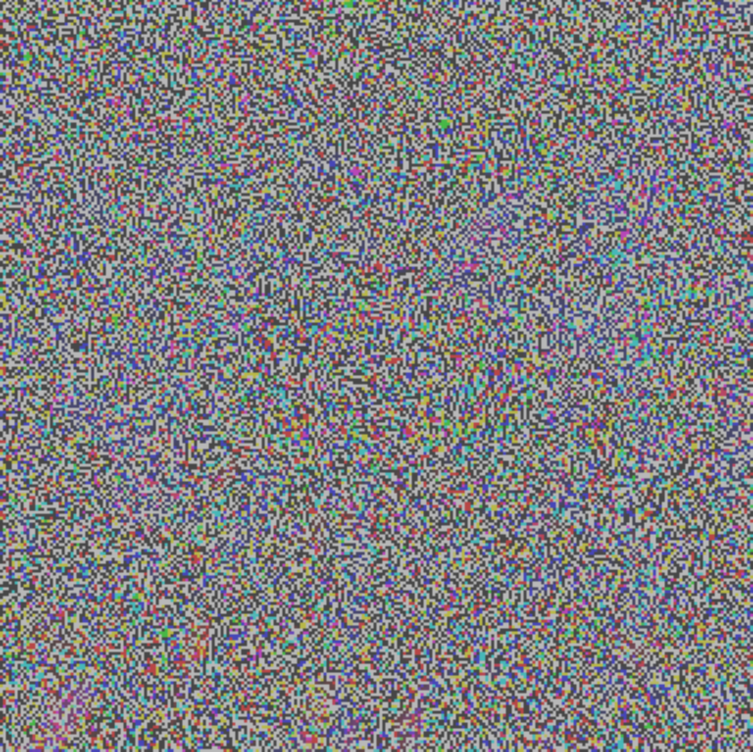}
   &
   \includegraphics[height=7em, alt={The previous image of daisy with
     the random image added as noise, indistinguishable to the daisy
     image to the human eye.}]{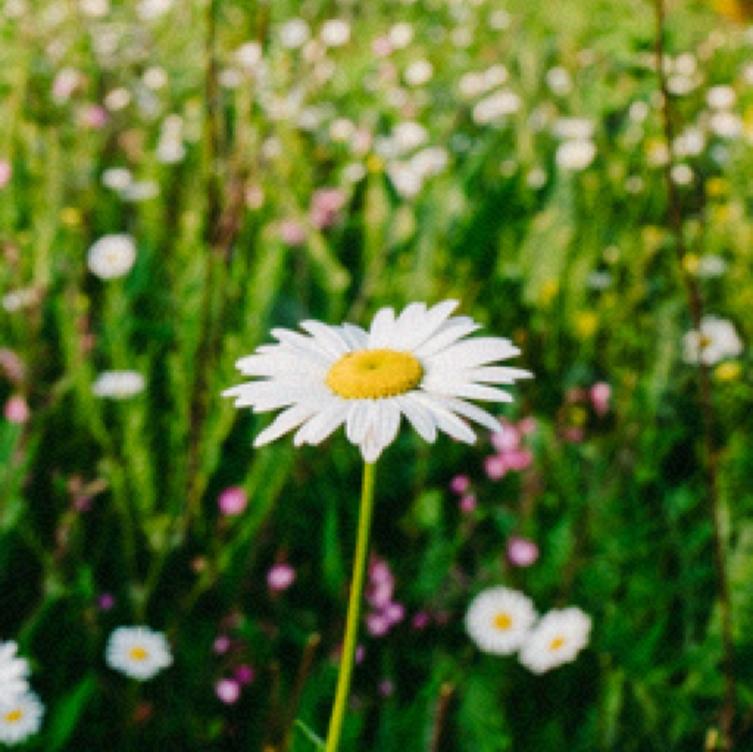}
   \\[.5em]
    original image: daisy & perturbation \(\epsilon = 0.03\) & adversarial image: quill pen 
  \end{tabular}
 \caption{\emph{Schematic of an adversarial attack in computer vision}:
   An image classifying algorithm, here a trained neural network, is
   fooled by minor differences undetectable to the human eye:
   it correctly classifies the left image as ``daisy'',
   while it classifies the right image as ``quill pen''.
   The right image has been obtained by adding some low yet carefully
   selected noise to the left image, which humans don't perceive.
   While this example is benign and funny, implications
 can be much more harmful: what if a self-driving car classifies
 a stop sign for a speed limit, simply because there was a tiny stain
 on it?
 The example here is from the NIPS2017 dataset.
 \vspace{-1ex}
}
 \label{fig:adv}
\end{figure}

There are two settings to be
considered: the \emph{white-box} setting, where the adversary has
complete access to the model, e.g., for a neural network model, access to the architecture and the weights, and the
\emph{black-box} setting, where the adversary has access only to the
inputs and outputs of the model. 
In practice, the difference between white-box and black-box attacks is
that the adversary cannot perform backpropagation to compute the
gradient of the loss function $\ell$.
Thus, \citet{chen20adv} use the classical Frank–Wolfe variant with
momentum in the white-box setting and a zeroth-order variant with
momentum in the black-box setting.
The motivation for using the Frank–Wolfe algorithm here is that the
iterates of projection-based attacks will likely be on the boundary of
the \index{lp-ball@\(\ell_{p}\)-ball}\(\ell_{p}\)-ball very fast,
while the iterates of Frank–Wolfe tend to remain in the
interior of the ball.
Thus, the \emph{distortion} $\norm[p]{x-\hat{x}}$
will likely be lower, hence better, by using the Frank–Wolfe
algorithm. Furthermore, linear minimizations over $\ell_p$-balls
have closed-form solutions while projections can be very expensive
when $p \neq 1, 2, +\infty$, as already stated in Table~\ref{tab:lmo}.

In both settings, the Frank–Wolfe algorithm
converges at a rate $\min_{1 \leq t\leq T} g_{t} =
\mathcal{O}(1/\sqrt{T})$ and $\mathbb{E}[\min_{1 \leq t\leq T} g_{t}] =
\mathcal{O}(1/\sqrt{T})$,
where $T$ is the total number of iterations (fixed in advance).
These have been compared to FGSM
\citep{goodfellow15}, PGD \citep{madry18}, and MI-FGSM \citep{dong18}
in the white-box setting, and to NES-PGD \citep{ilyas18nes} and bandit
attack \citep{ilyas18bandit} in the black-box setting, where the goal
is to attack the MNIST \citep{lecun98} and ImageNet \citep{deng09}
datasets.
Computational experiments show that among these methods,
the Frank–Wolfe approach yields the
best attack success rates (together with PGD and MI-FGSM in the
white-box setting), and with the fastest convergence and the lowest
distortion.

\emph{Adversarial training}
has the goal of training a neural network
with defense against adversarial attacks,
e.g, by using a robust optimization formulation.
\citet{madry18} analyze distortions in the \(\ell_{2}\)-norm
and \(\ell_{\infty}\)-norm caused by training via projected gradient
descent
and demonstrate strong resistance to
gradient obfuscation.
Adversarial training using Frank–Wolfe was then proposed in
\citet{tsiligkaridis20}.

Another defense is the use of \emph{interpretable} neural networks,
which provide brief explanation for their predictions.
The intent of the explanation is independent verification of the
predictions, e.g., by humans.
Defense against adversarial attacks is a by-product, as the goal
usually is to detect wrong prediction for naturally occurring
real-world data.
See \citet{interpretableNN2022} for the use of conditional gradients
in interpretable neural networks.

\subsection{Optimal experimental design}
\label{sec:optimal-exp-design}

In physics, biology, and medicine many experiments are
expensive or time consuming, so that minimizing the number of
experimental runs or sample sizes is desirable.
Therefore there is a need for optimal designs,
i.e., designs of experiments
optimal for some statistical criterion.

In this section, we consider a simple example,
where we want to learn $x^{*} \in \mathbb{R}^{d}$,
but we can only measure some linear function of it
that has been contaminated by Gaussian noise,
i.e., we can observe $y = \innp{a}{x^{*}} + \omega$
where $\omega$ is a normally distributed random variable
with zero mean and variance $\sigma^2$.
We further assume that the measurement vector \(a\)
can be chosen only among finitely many
values: \(a_{1}, \dotsc, a_{n}\),
which generate the vector space \(\mathbb{R}^{d}\).
The goal is to find the numbers of times 
\(m_{1}, \dotsc, m_{n}\) we must repeat 
the measurement for each
\(a_{1}, \dotsc, a_{n}\) in order 
to satisfy some statistical criterion. We assume that we can
perform a very large, fixed total number
\(m \defeq \sum_{i=1}^{n} m_{i}\) of measurements.

\pagebreak

Let
\(V([x]_{1}, \dotsc, [x]_{n}) \defeq
\sum_{i=1}^{n} [x]_{i} a_{i} a_{i}^{\top}\).
As is well known, the least-squares unbiased estimator for \(x^{*}\) is
\begin{equation*}
  \hat{x}
  \defeq
  \sum_{j=1}^{m}
  y_{j}
  V(m_{1}, \dotsc, m_{n})^{-1}
  a_{i_{j}}
  ,
\end{equation*}
where \(a_{i_{j}}\) is the measurement vector chosen for 
the \(j\)-th measurement and \(y_{j}\) is a random realization 
of \(\innp{a_{i_{j}}}{x^{*}} + \omega\).
The estimator \(x^{*}\) is normally distributed with
mean $x^*$ and
covariance matrix $\sigma^2 V(m_{1}, \dotsc, m_{n})^{-1}$.

There are many optimization criteria for optimal experiment design,
of which we consider only \emph{D-optimal design},
which maximizes the information content about \(x^{*}\)
of the estimate \(\hat{x}\), as measured by the differential Shannon
entropy. This boils down to maximizing \(\det V(m_{1}, \dotsc, m_{n})\).
For ease of exposition we
focus on the \emph{continuous (approximate) design of experiments}
where the \(m_{i}\) can take any positive real value,
assuming the total number of experiments \(m\) is large enough
so that rounding introduces minimal error.
Reparametrizing with \([x]_{i} \defeq m_{i} /m\),
the goal becomes that of maximizing
\(\det V(m_{1}, \dotsc, m_{n}) = m^{d} \det V([x]_{1}, \dotsc, [x]_{n})
= m^{d} V(x)\)
with \(x \defeq ([x]_{1}, \dotsc, [x]_{n})\) lying in
the \myindex{probability simplex}
\(\Delta_{n} \defeq \{ x \mid \sum_{i=1}^{n} [x]_{i} = 1, x \geq
0\}\).
We write \([x]_{i}\) for the coordinate \(i\) of \(x\) to distinguish
it from the iterate \(x_{i}\) of the algorithm below and we 
choose an equivalent objective function, which is convex,
so that the optimization problem becomes
\begin{equation}\label{eq:D-optimal-design}
  \min_{x \in \Delta_{n}} \bigl(- \ln \det V(x)\bigr)
  \qquad
  f(x) \defeq - \ln \det V(x)
  .
\end{equation}
Note that the objective function shown in 
Problem~\eqref{eq:D-optimal-design} has a Lipschitz 
continuous gradient for points on the relative 
interior of $\Delta_{n}$. As a side remark, the dual of the problem shown in 
Equation~\eqref{eq:D-optimal-design} has a geometric interpretation:
it is the \emph{minimum volume enclosing ellipsoid} problem,
in which we want to find the ellipsoid
with the smallest volume that contains $a_{1}, \dotsc, a_{n}$ \citep{silvey1972discussion}.
See \citet[Chapter~21]{ls20} for connections to bandit algorithms.

Recall that \(\ln \det V\) is a concave function in \(V\)
for \(V\) positive definite, and its gradient 
is given via
\begin{equation*}
  \innp{\nabla \ln \det V}{X} = \trace{V^{-1} X}
  .
\end{equation*}
In particular, this provides the scalar product of the 
gradient of the objective
function with the vertices of \(\Delta_{n}\), i.e.,
with the coordinate vectors \(e_{1}, \dotsc, e_{n}\):
\begin{equation*}
 \begin{split}
  \innp{\nabla f(x)}{e_{i}}
  =
  \frac{\partial f(x)}{\partial [x]_{i}}
  &
  =
  - \trace{V(x)^{-1} \frac{\partial V(x)}{\partial [x]_{i}}}
  = - \trace{V(x)^{-1} a_{i} a_{i}^{\top}}
  \\
  &
  = - a_{i}^{\top} V(x)^{-1} a_{i}
  = - \norm[V(x)^{-1}]{a_{i}}^{2}
 \end{split}
\end{equation*}

Now the vanilla Frank–Wolfe algorithm (Algorithm~\ref{fw})
with exact line search
for the D-optimal design Problem~\eqref{eq:D-optimal-design}
becomes Algorithm~\ref{fw-gopt}
\citep[which was originally proposed by][]{fedorov72},
with Line~\ref{fw-gopt:FW_vertex} computing the index of the
Frank–Wolfe vertex.

\begin{algorithm}[b]
\caption{Frank–Wolfe for D-optimal design \citep{fedorov72}}
\label{fw-gopt}
\begin{algorithmic}[1]
  \REQUIRE Start point $x\in\Delta_n$,
    measurement vectors \(a_{1}, \dotsc, a_{n}\)
  \ENSURE Iterates \(x_{1}, \dotsc\)
  \FOR{$t=0$ \TO \dots}
    \STATE\label{fw-gopt:FW_vertex}
      \(i_{t} \gets \argmax_{1 \leq i \leq n} \norm[V(x_t)^{-1}]{a_{i}}^{2}\)
    \STATE\label{fw-gopt:step}
      \(\gamma_{t} \gets
      \frac{(1/n) \norm[V(x_{t})^{-1}]{a_{i_{t}}}^{2} - 1}{%
        \norm[V(x_{t})^{-1}]{a_{i_{t}}}^{2} - 1}\)
    \STATE
      \(x_{t+1} \gets (1 - \gamma_{t}) x_{t} + \gamma_{t} e_{i_{t}}\)
      \COMMENT{\(e_{i}\) is coordinate vector \(i\).}
  \ENDFOR
\end{algorithmic}
\end{algorithm}

Line~\ref{fw-gopt:step} is actually the closed form solution
to line search.
We include a proof for completeness.
We compute the derivative of the objective function 
in \(\gamma\) when moving
from \(x_{t}\) to \(x_{t} + \gamma (e_{i_{t}} - x_{t})\) via
simple manipulations and the Sherman--Morrison formula:
\begin{spreadlines}{2ex}
\begin{equation*}
 \begin{split}
  \MoveEqLeft
  \frac{\partial f\bigl(x_{t} + \gamma (e_{i_{t}} - x_{t})\bigr)}{\partial \gamma}
  \\ 
  &
   =
   - \trace{ V( x_{t} + \gamma (e_{i_{t}} - x_{t}))^{-1}
     \frac{\partial V\bigl(x_{t} + \gamma( e_{i_{t}} - x_{t})\bigr)}
     {\partial \gamma}}
   \\
   &
   =
   - \trace{ [V(x_{t}) (1 - \gamma)
     + \gamma a_{i_{t}} a_{i_{t}}^{\top}]^{-1}
     \bigl(a_{i_{t}} a_{i_{t}}^{\top}  - V(x_{t})\bigr)}
   \\
   &
   =
   - \trace{\left(
       \frac{V(x_{t})^{-1}}{(1 - \gamma)}
       - \frac{\gamma}{(1 - \gamma)^{2}}
       \frac{V(x_{t})^{-1} a_{i_{t}} a_{i_{t}}^{\top} V(x_{t})^{-1}}
       {1 + \gamma a_{i_{t}}^{\top} V(x_{t})^{-1} a_{i_{t}} /(1-\gamma)}
     \right)
     \bigl(a_{i_{t}} a_{i_{t}}^{\top} - V(x_{t})\bigr)}
   \\
   &
   =
   \frac{1}{1 - \gamma} \left(
     n - \norm[V(x_{t})^{-1}]{a_{i_{t}}}^{2}
     + \frac{\gamma \norm[V(x_{t})^{-1}]{a_{i_{t}}}^{2}
       (\norm[V(x_{t})^{-1}]{a_{i_{t}}}^{2} - 1)}{1 - \gamma
       + \gamma \norm[V(x_{t})^{-1}]{a_{i_{t}}}^{2}}
   \right).
 \end{split}
\end{equation*}
\end{spreadlines}%
The minimal function value occurs at the zero of the derivative,
i.e., at
\(\gamma = [(1/n) \norm[V(x_{t})^{-1}]{a_{i_{t}}}^{2} - 1] /
(\norm[V(x_{t})^{-1}]{a_{i_{t}}}^{2} - 1)\).

The fact that there is a closed form expression for the 
line search already shows that this algorithm has the 
advantage of having to perform cheaper basic operations 
than other algorithms, due to the fact that the algorithm 
only moves towards vertices of $\Delta_{n}$.
Note that no closed form expression is known for the line search
between two arbitrary points on the simplex for \(f\),
which is a disadvantage, e.g., for projected gradient descent.
The following recursive formulae provide further
optimization possibilities:
implementing each matrix operation (like computing the determinant 
and the inverse of a matrix)
in Algorithm~\ref{fw-gopt} can be done using only
\(\mathcal{O}(d^{2})\) arithmetic operations.
However, in the first iteration we still have to compute $V(x_0)^{-1}$ 
and $\det V(x_0)$ with \(\mathcal{O}(d^{3})\) arithmetic operations.
The recursive nature of the formulae accumulates computational errors
across iterations, so periodically computing the quantities directly
is still advisable in practice.
\begin{align*}
  \det V(x_{t+1})
  &
  =
  (1 - \gamma_{t} + \gamma_{t} \norm[V(x_t)^{-1}]{a_{i_{t}}}^{2})
  (1 - \gamma_{t})^{d - 1} \det V(x_{t})
  ,
  \\
  V(x_{t+1})^{-1}
  &
  =
  \frac{1}{1 - \gamma_t} \left(
    V(x_{t})^{-1}
    - \frac{\gamma_{t}
      V(x_{t})^{-1} a_{i_{t}} a_{i_{t}}^{\top} V(x_{t})^{-1}}%
    {1 - \gamma_{t} + \gamma_{t} \norm[V(x_{t})^{-1}]{a_{i_{t}}}^{2}}
  \right)
  .
  \intertext{For the gradient of \(f\), the following formulae
    needs only \(\mathcal{O}(d)\) arithmetic operations
    (essentially the scalar product with \(a_{i}\))
    per coordinate:}
  \innp{\nabla f(x_{t+1})}{e_{i}}
  &
  =
  \frac{1}{1 - \gamma_{t}} \left(
    \innp{\nabla f(x_{t})}{e_{i}}
      - \frac{\gamma_{t} (a_{i}^{\top} V(x_t)^{-1}a_{i_{t}})^2}{
        1 - \gamma_{t} + \gamma_{t} \norm[V(x_{t})^{-1}]{a_{i_{t}}}^{2}}
  \right)
  .
\end{align*}
Finally, we note that Algorithm~\ref{fw-gopt} has an away-step version
relying on the Away-step Frank–Wolfe algorithm (Algorithm~\ref{away})
instead of the vanilla Frank–Wolfe algorithm. Note that as the away 
steps of the aforementioned algorithm move away from the vertices 
of $\Delta_{n}$, the optimizations shown above also apply, 
e.g., there is a similar closed-form solution for the line search
for away steps. Also, for both algorithms, if $x_{0}= \allOne /n$ with 
$n> 1$ and the non-zero vectors \(a_{1}, \dotsc, a_{n}\) generate 
the vector space \(\mathbb{R}^{d}\), then we will always have that
$\gamma_t < 1$, and so no drop-steps will be taken in the Away-step
Frank–Wolfe algorithm in Algorithm~\ref{fw-gopt}.
The convergence rate of both Frank–Wolfe variants of
Algorithm~\ref{fw-gopt} is presented in
Theorem~\ref{theorem:DoptDesign} from
\citet[Theorems~3.9 and~3.11]{todd2016minimum}.

\begin{theorem}
  \label{theorem:DoptDesign}
  Let the measurement vectors
  \(a_{1}, \dotsc, a_{m} \in \mathbb{R}^{d}\)
  be of full rank, i.e., the vector space spanned by them
  be the whole \(\mathbb{R}^{d}\).
  Let $x_{0}= \allOne /n$, and let $x^*=\argmax_{x\in\Delta_n} \det V(x)$, 
  then Algorithm~\ref{fw-gopt}
  finds a continuous solution to the D-optimal experiment design
  such that $\det V(x_t) \geq e^{-\epsilon} \det V(x^*)$ for
   $t \geq 4d (\log \log n + 3/2) + 28d/\epsilon$.
  Similarly, the AFW algorithm applied to the D-optimal experiment
  design achieves a solution such that 
  $\det V(x_t) \geq e^{-\epsilon} \det V(x^*)$ for 
  $t\geq 4d (\log \log n + 3/2) + 56 d/\epsilon$. 
\end{theorem}
This upper bounds the number of arithmetic operations
independent of the \(a_{i}\).
Nonetheless, the actual values of the \(a_{i}\)
have a great effect
on the achievable design optimality \(V(x^{*})\)
in practice.

With the optimizations shown above computing one 
iteration of Algorithm~\ref{fw-gopt} 
requires $\mathcal{O}(n + d^2)$ arithmetic 
operations. On the other hand, 
algorithms that do not move towards 
(or away from) vertices of $\Delta_{n}$ require, for example, 
$\mathcal{O}(d^2n+d^3)$ arithmetic operations when computing an FOO in general. The results from Theorem~\ref{theorem:DoptDesign} state that both FW and AFW algorithms globally
converge at a $\mathcal{O}(1/\epsilon)$ rate. Note however, that after a finite number of iterations 
independent of the target accuracy the AFW algorithm converges at a
 $\mathcal{O}(\log 1/\epsilon)$ rate
 \citep[see][Theorem~3.16]{todd2016minimum}. This is often termed the \emph{local} convergence 
of the algorithm.
In fact, even though the function $f$ is not
globally smooth, it is smooth and strongly convex over the level set 
$\mathcal{X}_0 = \{ x \in \Delta_{n} \mid f(x)\leq f(x_0) \}$ for any
$x_0$ in the domain of $f$.
This makes the AFW algorithm with the short step rule, or
line search \emph{globally} linearly convergent in primal gap, as
$f$ is self-concordant, and the results from \citet{carderera2021simple} apply.

\begin{example}[Frank–Wolfe variants for D-optimal experiment design]
  \label{example:D-optimal}
We generate a matrix $M\in \R^{d\times d}$ with entries drawn uniformly at random between
$0$ and $1$.
We then draw the measurement vectors \(a_{1}, \dotsc, a_{n}\)
as \(n\) samples
from a multivariate Gaussian distribution with zero mean
and covariance $M^{\top}M$.

\begin{figure}
  \footnotesize
  \centering
  \begin{tabular}{cc}
    \(d = 100\), \(n = 1000\)  & \(d = 5000\), \(n = \num{10000}\) \\[\smallskipamount]
    \includegraphics[width=.45\linewidth, alt={Plot of differential
      Shannon information in iterations.
      The vanilla Frank–Wolfe algorithm is slowly decreasing,
      the Away-step Frank–Wolfe algorithms decreases sharply after an
      initial slow
      phase.}]{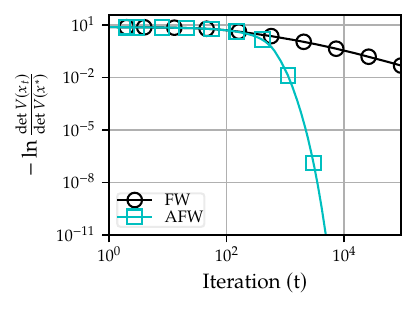}
  &
  \includegraphics[width=.45\linewidth, alt={Plot of differential
    Shannon information in iterations.
    Like the previous image but the initial slow phase is
    much longer.}]{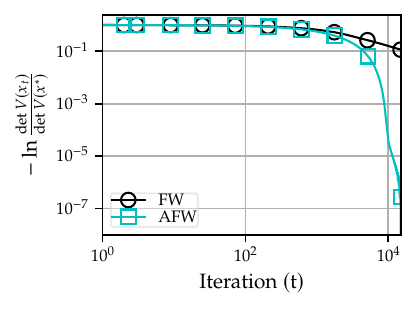}
  \\
  \includegraphics[width=.45\linewidth, alt={Plot of differential
    Shannon information in time.
    A distorted version of the plot with short initial slow phase in
    iterations.}]{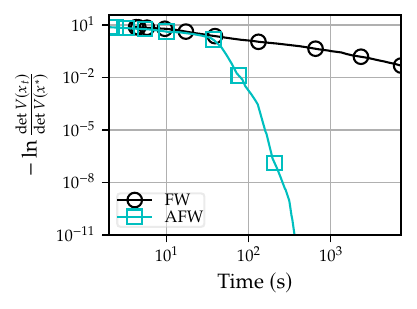}
  &
  \includegraphics[width=.45\linewidth, alt={Plot of differential
    Shannon information in time.
    A distorted version of the plot with long initial slow phase in
    iterations.}]{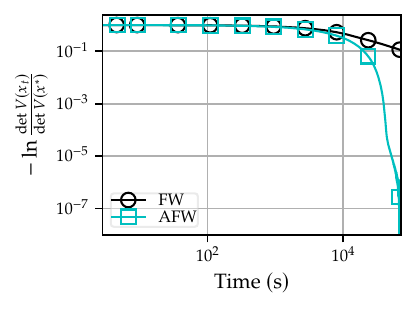}
  \end{tabular}

  \caption{\emph{D-optimal experiment design:}
    Primal gap convergence of differential Shanon information
    in the number of
    iterations and wall-clock time for
    two different D-optimal experiment design problems of dimension \(d\)
    with \(n\) measurement vectors
    sampled from a zero-mean normal distribution.}
  \label{fig:D-optimal_experiment}
\end{figure}

As we can see in Figure~\ref{fig:D-optimal_experiment}, the AFW algorithm 
outperforms the FW algorithm in both the number of iterations and wall-clock
time.
Both
algorithms are implemented using the optimized matrix operations
detailed above.
\end{example}

\backmatter
\chapter*{Funding information and grants}
\addcontentsline{toc}{chapter}{Funding information and grants}
\label{sec:acknowledgements}

Cyrille Combettes acknowledges the support of the AI Interdisciplinary
Institute ANITI funding, through the French ``Investments for the Future --
PIA3'' program under the grant agreement ANR-19-PI3A0004, Air Force Office of
Scientific Research, Air Force Material Command, USAF, under grant number
FA9550-19-1-7026, and ANR MaSDOL 19-CE23-0017-0. This work was partially
supported by the DFG Cluster of Excellence MATH+ (EXC-2046/1, project id
390685689) funded by the Deutsche Forschungsgemeinschaft (DFG).

Hamed Hassani acknowledges the support by the NSF CAREER Award (1943064), AFOSR Young Investigator Award, AI Institute for Learning-Enabled Optimization at Scale (TILOS), and the Institute for CORE Emerging Methods in Data Science (EnCORE). 

Amin Karbasi acknowledges the support of NSF (IIS-1845032), ONR
(N00014- 19-1-2406), and the AI Institute for Learning-Enabled
Optimization at Scale (TILOS). He also would like to thank Nila and
Liam for their constant distractions. He blames them for all the typos
he made \smiley.

The research of Aryan Mokhtari is supported in part by NSF Grants 2007668 and 2127697, an ARO Early Career Program Award W911NF2110226, the Machine Learning Lab (MLL) at UT Austin, the NSF AI Institute for Foundations of Machine Learning (IFML), and the NSF AI Institute for Future Edge Networks and Distributed Intelligence (AI-EDGE).

Sebastian Pokutta acknowledges support by NSF CAREER Award CMMI-1452463 and the DFG Cluster of Excellence MATH+ (EXC-2046/1, project id 390685689) funded by the Deutsche Forschungsgemeinschaft (DFG).

\printbibliography[heading=bibintoc]
\printindex
\end{document}